\documentclass[12pt,rqno, draft]{amsart}
\usepackage{amsthm, amssymb, amsfonts}
\usepackage[dvips]{graphicx}
\usepackage[T5]{fontenc}
\title[THE HIT PROBLEM FOR THE POLYNOMIAL ALGEBRA]{THE HIT PROBLEM FOR  THE POLYNOMIAL\\ ALGEBRA OF FOUR VARIABLES}

\author{Nguy\~\ecircumflex n Sum}

\theoremstyle{plain}
\newtheorem{thm}{Theorem}[section]
\newtheorem{prop}[thm]{Proposition}
\newtheorem{lem}[thm]{Lemma}
\newtheorem{con}[thm]{Conjecture}
\theoremstyle{definition}
\newtheorem{defn}[thm]{Definition}
\newtheorem{rem}[thm]{Remark}
\newtheorem{nota}[thm]{Notation}

\theoremstyle{plain}
\newtheorem{thms}{Theorem}[subsection]
\newtheorem{props}[thms]{Proposition}
\newtheorem{lems}[thms]{Lemma}
\theoremstyle{definition}

\newtheorem{rems}[thms]{Remark}

\def\vdvh{\vrule height 10pt depth 4pt width 0.5pt}
\def\vdq{\vdvh\ }
\def\DD{D\kern-.7em\raise0.4ex\hbox{\char '55}\kern.33em}

\begin{document} 

\setlength{\baselineskip}{14pt}

\begin{abstract} We study the problem of determining a minimal set of generators for the polynomial algebra $\mathbb F_2[x_1,x_2,\ldots,x_k]$ as a module over the mod-2 Steenrod algebra $\mathcal{A}$. In this paper, we give an explicit answer in terms of the monomials for $k=4$.  \end{abstract}

\footnotetext[3]{\ This version is a revision of a preprint of Quy Nh\ohorn n University, Vi\^{\d e}t Nam, 2007.}

\maketitle


\bigskip
\section{Introduction}\label{1} 

Let $E^k$ be an elementary abelian 2-group of rank $k$. Denote by $BE^k$ the classifying space of $E^k$. It may be thought of as the product of $k$ copies of real projective space $\mathbb RP^\infty$.  We have  
$$P_k:= H^*(BE^k) = \mathbb F_2[x_1,x_2,\ldots ,x_k],$$ 
a polynomial algebra on  $k$ generators $x_1, x_2, \ldots , x_k$, each of degree 1. Here the cohomology is taken with coefficients in the prime field $\mathbb F_2$ of two elements. 

It is well known that $P_k$ is a module over the mod-2 Steenrod algebra $\mathcal{A}$. The action of $\mathcal A$ on $P_k$ is determined by the formula
$$Sq^i(x_j) = \begin{cases} x_j, &i=0,\\ x_j^2, &i=1,\\ 0, &\text{otherwise,}\end{cases}$$
and the Cartan formula 
$$Sq^n(fg) = \sum_{i=0}^nSq^i(f)Sq^{n-i}(g),$$
for $f, g \in P_k$ (see Steenrod \cite{st}).

Many authors  study the {\it hit problem} of determination of the minimal set  of generators for $P_k$ as a module over the Steenrod algebra, or equivalently, a basis of $\mathbb F_2 \underset {\mathcal{A}} \otimes P_k$.   This problem has first been studied by Peterson \cite{pe}, Wood \cite{wo}, Singer \cite {si1}, Priddy \cite{pr}, Carlisle-Wood \cite{cw}, who show its relationship to several classical problems in homotopy theory.

Peterson conjectured in \cite{pe} that as a module over the Steenrod algebra $\mathcal{A}$, $P_k$ is generated by monomials in degrees $n$ that satisfy $\alpha(n+k) \leqslant k$, where $\alpha(n)$ denotes the number of ones in dyadic expansion of $n$, and proved it for $k\leqslant 2$. The conjecture was established in general by Wood \cite {wo}.  This is a useful tool for determining $\mathcal{A}$-generators for $P_k$. 

The tensor product $\mathbb F_2 \underset {\mathcal{A}} \otimes P_k$ has explicitly been calculated  by Peterson \cite{pe} for $k=2$ and  Kameko for $k=3$ in his thesis \cite{ka} and in ``generic degrees'' for all $k$ by Nam \cite{na}. Singer \cite{si2}, Silverman \cite{sl}, H\uhorn ng-Nam \cite{hn1, hn2} have showed that many  polynomials in $P_k$ are in $\mathcal{A}^+.P_k$, where $\mathcal{A}^+$ denotes the augmentation ideal in $\mathcal A$.

Recently, many authors showed their interest in the study of $\mathbb F_2 \underset {\mathcal{A}} \otimes P_k$ in conjunction with the transfer, which was defined by Singer \cite{si1}. The transfer is a homomorphism 
$$Tr_k:\text{Tor}_k^{\mathcal{A}}(\mathbb F_2,\mathbb F_2)\longrightarrow (\mathbb F_2 \underset {\mathcal{A}} \otimes P_k)^{GL_k(\mathbb F_2)}.$$ 
Here Tor$_k^{\mathcal{A}}(\mathbb F_2,\mathbb F_2)$  is isomorphic to Ext$_{\mathcal{A}}^k(\mathbb F_2,\mathbb F_2)$, the $E_2$ term of the Adams spectral sequence of spheres and $GL_k(\mathbb F_2)$ is the general linear group  which acts on $\mathbb F_2 \underset {\mathcal{A}} \otimes P_k$ in the usual manner. It was shown that  the transfer is an isomorphism for $k=1,2$ by Singer in \cite{si1} and for $k=3$ by Boardman in \cite{bo}. Bruner-H\`a-H\uhorn ng \cite{br}, H\uhorn ng \cite{hu} and  Ha \cite{ha} have studied the transfer for $k=4, 5$.

One of important tools in Kameko's computation  of  $\mathcal{A}$-generators for $P_3$ is the squaring operation
$$Sq^0_*: (\mathbb F_2 \underset {\mathcal{A}} \otimes P_k)_n \longrightarrow (\mathbb F_2 \underset {\mathcal{A}} \otimes P_k)_{\frac{n-k}2},$$
which is determined for all $n \geqslant k$ such that $n-k$ is even. Kameko showed in \cite{ka} that if $\beta(n)=k$ then $Sq^0_*$ is an isomorphism of $\mathbb F_2$-vector spaces, where 
$$\beta(n) = \min\{m\in \mathbb Z \ ; \ \alpha(n+m) \leqslant m\}.$$ 
From this and Wood's theorem, the hit problem is reduced to the cases $\beta(n) <k$. 
Hence, for $k=4$, it suffices to consider six cases:
\begin{align*} 
1)\  n&=2^{s+1}-3,\\
2)\  n&=2^{s+1}-2,\\
3)\  n&=2^{s+1}-1,\\
4)\  n &=2^{s+t+1}+2^{s+1}-3,\\
5)\  n&=2^{s+t}+2^s-2,\\
6)\  n &= 2^{s+t+u}+2^{s+t}+2^s-3,
\end{align*}
where $s,t,u$ are the positive integers.

\smallskip
In this paper, we explicitly determine $\mathbb F_2 \underset {\mathcal{A}} \otimes P_k$ for $k=4$. We have

\eject
\begin{thm}\label{1.1} Let $n $ be an   arbitrary positive integer with $\beta(n) < 4$. The dimension of the $\mathbb F_2$-vector space $(\mathbb F_2 \underset {\mathcal A} \otimes P_4)_n$ is given by the following table

\smallskip{\rm
\centerline{\begin{tabular}{clcccc}
$n$ &\vdq $s=1$ & $s=2$ & $s=3$ & $s=4$  & $s\geqslant 5$\cr
\hline
$2^{s+1}-3$ &\vdq\ \ 4 & 15 & 35 & 45 & 45  \cr
$2^{s+1}-2$&\vdq \ \ 6 & 24 & 50 & 70 & 80  \cr
$2^{s+1}-1$ &\vdq\ 14 & 35 & 75 &  89 & 85 \cr 
$2^{s+2}+2^{s+1}-3$ &\vdq\ 46 & 94 & 105  &105&105\cr
$2^{s+3}+2^{s+1}-3$&\vdq\ 87 & 135 & 150  &150&150\cr
$2^{s+4}+2^{s+1}-3$ &\vdq 136& 180 & 195 &195&195\cr 
$2^{s+t+1}+2^{s+1}-3, t\geqslant 4$ &\vdq 150& 195 & 210 & 210 & 210\cr
 $2^{s+1}+2^s-2$ &\vdq\ 21 & 70 & 116 & 164 & 175 \cr
 $2^{s+2}+2^s-2$&\vdq \ 55 & 126 & 192 &240 & 255  \cr
 $2^{s+3}+2^s-2$ &\vdq \ $73$& 165 & 241 & 285 &  300  \cr 
 $2^{s+4}+2^s-2$ &\vdq \ 95& 179 & 255 & 300&  315  \cr
 $2^{s+5}+2^s-2$ &\vdq 115& 175 & 255 & 300 & 315 \cr
$2^{s+t}+2^s-2, t\geqslant 6$ &\vdq 125& 175 & 255 & 300 & 315  \cr
 $2^{s+2}+2^{s+1}+2^s-3$ &\vdq\  64&120  &120 &120&120\cr
 $2^{s+3}+2^{s+2}+2^s-3$  &\vdq 155 & 210 &210&210&210 \cr
 $2^{s+t+1}+2^{s+t}+2^s-3, t\geqslant 3$  &\vdq 140&210  &210&210&210 \cr
 $2^{s+3}+2^{s+1}+2^s-3$ &\vdq 140& 225 &225&225&225 \cr
 $2^{s+u+1}+2^{s+1}+2^s-3, u\geqslant 3$ &\vdq  120& 210 &210 &210&210\cr
 $2^{s+u+2}+2^{s+2}+2^s-3, u\geqslant 2$&\vdq 225 & 315 &315&315&315 \cr
 $2^{s+t+u}+2^{s+t}+2^s-3, u\geqslant 2, t\geqslant 3$  &\vdq 210 &315& 315 &315 &315. \cr
\end{tabular}}}
\end{thm}

From this theorem, we see that Kameko's conjecture
$$\sup_n\dim (\mathbb F_2 \underset {\mathcal A} \otimes P_k)_n = \prod_{1\leqslant i\leqslant k}(2^i-1)$$
is true for $k=4$. Furthermore, we have the following conjecture.

\begin{con}Let $n = 2^{s_1+s_2+\ldots+s_{k-1}} + 2^{s_1+s_2+\ldots+s_{k-2}} + \ldots + 2^{s_{1}}-k+1$ with $s_j$ the positive integer, $ j= 1,2,\ldots, k-1$. If $s_j \geqslant 2$ for all $j$ then
$$\dim (\mathbb F_2 \underset {\mathcal A} \otimes P_k)_n = \prod_{1\leqslant i\leqslant k}(2^i-1).$$
\end{con}

This paper is organized as follows.

1. Introduction.

2. Preliminaries.

3. The indecomposables of $P_4$ in degree $2^{s+1}-3$.

4. The indecomposables of $P_4$ in degree $2^{s+1}-2$.

5. The indecomposables of $P_4$ in degree $2^{s+1}-1$.

6. The indecomposables of $P_4$ in degree $2^{s+t+1}+2^{s+1}-3$.

7. The indecomposables of $P_4$ in degree $2^{s+t}+2^{s}-2$.

8. The indecomposables of $P_4$ in degree $2^{s+t+u}+2^{s+t}+2^{s}-3$.

9. Final remark.

\smallskip
 In Section \ref{2}, we recall  some results on the admissible monomials in $P_k$ and Kameko's squaring operation. We compute $\mathbb F_2 \underset {\mathcal{A}} \otimes P_4$ in the next sections by describing explicitly a basis of it in terms of the monomials. In each case, we order the monomials in a basis by the product of powers of variables and then by the lexicographical order on the set of $\sigma$-sequences of monomials (see Section \ref{2}).

Finally, we give in Section \ref{9} some matrices from which we can describe the matrices associated with  admissible monomials in $P_4$.

\medskip\noindent
{\bf Acknowledgements.}
It is a pleasure for me to thank Prof. Nguy\~\ecircumflex n H. V. H\uhorn ng for his valuable suggestions and encouragement. My thanks also go to all colleagues at Department of Mathematics, University of Quy Nh\ohorn n for many helpful conversations. 

\section{Preliminaries}\label{2}

In this section, we recall some results in Kameko \cite{ka}, Wood \cite{wo}, Singer \cite{si2} on the admissible monomials and the hit monomials in $P_k$. Recall that an element in $P_k$ is called hit if it belongs to $\mathcal{A}^+.P_k$.

Let $x=x_1^{a_1}x_2^{a_2}\ldots x_k^{a_k}$ in $P_k$. Following Kameko \cite{ka},  we assign a matrix $(\varepsilon_{ij}(x))$ to  monomial $x$, where $\varepsilon_{ij}(x) = \alpha_{i-1}(a_j),$ the $(i-1)$-th coefficient in dyadic expansion of $a_j$ for $ i= 1,2,\ldots $ and $ j=1,2,\ldots , k .$   

From now on, if we say $M$ is a matrix then we assume that its entries belong to $\{0,1\}$ and the number of non-zero entries is finite. If $M=(\varepsilon_{ij})$ is a matrix then the monomial $x=x_1^{a_1}x_2^{a_2}\ldots x_k^{a_k}$ corresponding to $M$ is determined by
$$a_j = \sum_{i\geqslant 1}2^{i-1}\varepsilon_{ij},\ j=1,2,\ldots , k.$$

Define two sequences associated with monomial $x$ by
\begin{align*} \tau(x)&=(\tau_1(x);\tau_2(x);\ldots ; \tau_i(x); \ldots),\\
\sigma(x) &= (a_1;a_2;\ldots ;a_k),\end{align*}
where
$$\tau_i(x) = \sum_{j=1}^k\varepsilon_{ij}(x),\ i=1,2,\ldots .$$
We call  $\tau(x)$ the $\tau$-sequence and $\sigma(x)$ the $\sigma$-sequence associated with monomial $x$.

We identify a finite sequence $(\xi_1;\xi_2;\ldots ; \xi_m)$ with  $(\xi_1;\xi_2;\ldots ; \xi_m;0;\ldots)$. 

\begin{defn}\label{2.1} 
Let $x, y$ be the monomials in $P_k$. We say that $x <y$ if and only if one of the following holds

1. $\tau (x) < \tau(y)$;

2. $\tau (x) = \tau(y)$ and $\sigma(x) < \sigma(y).$

\noindent Here the order on the set of sequences of non-negative integers is the lexicographical one.
\end{defn}

\begin{defn}\label{2.2} 
A monomial $x$ is said to be inadmissible if there exists the monomials $y_1,y_2,\ldots, y_r$ such that
$$x=y_1+y_2+\ldots + y_r \ \text{mod }\mathcal{A}^+.P_k \ \text{and } y_j<x,\ j=1,2,\ldots , r.$$
A monomial $x$ is said to be admissible if it is not inadmissible.
\end{defn}

Obviously, the set of all admissible monomials in $P_k$ is a minimal set of $\mathcal{A}$-generators of $P_k$.

\begin{defn}\label{2.3} 
Let $M$ be an $s\times k$-matrix and $x$ the monomial corresponding to $M$. The matrix $M$ is said to be strictly inadmissible if and only if there exists the monomials $y_1,y_2,\ldots, y_r$ such that
$$x=y_1+y_2+\ldots + y_r + \sum_{0<i<2^s}\gamma_iSq^i(z_i), \ \text{and } y_j<x,\ j=1,2,\ldots , r,$$
where $\gamma_i \in \mathbb F_2,$ and $z_i \in P_k$ with $ 0 < i < 2^s$.
\end{defn}

For a sequence of non-negative integers $\tau = (\tau_1;\tau_2;\ldots ; \tau_m )$, we set
$$\mathcal{L}_k(\tau) = \text{Span}\Big\{x=x_1^{a_1}x_2^{a_2}\ldots x_k^{a_k}\in P_k \ ; \deg x = \sum_{i=1}^m2^{i-1}\tau_i ,\ \tau(x) < \tau\Big\}.$$

If $x$ is a monomial in $P_k$ and $y$ is a monomial in $ \mathcal{L}_k(\tau(x))$ then $y<x$. The space $\mathcal{L}_k(\tau)$ is an $\Sigma_k$-submodule of $P_k$, where $\Sigma_k$ denote the symmetric group on $k$ letters.

Suppose $M$ is an $s\times k$-matrix and $x$ is the monomial corresponding to $M$. Observe that if there exists the monomials $y_1,y_2,\ldots, y_r$ such that
$$x=y_1+y_2+\ldots + y_r + \sum_{0<i<2^s}\gamma_iSq^i(z_i) \quad \text{mod }\mathcal L_k(\tau(x)), $$
where $\gamma_i \in \mathbb F_2, z_i \in P_k$ with $ 0 < i < 2^s$ and $ y_j<x,\ j=1,2,\ldots , r,$ then the matrix $M$ is strictly inadmissible.

\begin{nota}\label{2.3b} Let $x$ be a monomial and $\Delta = (\delta_{ij})$ be an $s\times k$-matrix. If there is some non-negative integer $r$ such that $\delta_{ij} = \varepsilon_{i+r,j}(x)$ for $1\leqslant i \leqslant s$ and $1\leqslant j \leqslant k$ then we write $\Delta \triangleright x$.
\end{nota}

The following theorem is one of our main tools.

\begin{thm}[Kameko \cite{ka}]\label{2.4}  
 Let $x$ be a monomial and $\Delta = (\delta_{ij})$ be an $s\times k$-matrix.  If the matrix $\Delta$ is strictly inadmissible and $\Delta \triangleright x$ then $x$ is inadmissible.
\end{thm}

We need the following.

\begin{lem}\label{2.5} 
 Let $M$ be an $2\times k$-matrix and $x$ the monomial corresponding to $M$. If either $\tau_1(x)=0, \tau_2(x)>0$ or $\tau_1(x)<k, \tau_2(x)=k$ then $M$ is strictly inadmissible.\end{lem}

\begin{proof} If $\tau_1(x) = 0$ and $\tau_2(x) >0$ then $x=Sq^1(y)$ for some monomial $y$ in $P_k$. Hence, the lemma is true. If $0<\tau_1(x)=m<k$ and $\tau_2(x)=k$ then there is a ring homomorphism $f:P_k\to P_k$ induced by a permutation of $\{x_1,x_2,\ldots ,x_k\}$ such that $x=f(x_1^3\ldots x_m^3x_{m+1}^2x_{m+2}^2\ldots x_k^2)$. It is easy to see that $f$ is an $\mathcal  A$-homomorphism. It sends monomials to monomials and preserves the associated $\tau$-sequences. Obviously, we have
\begin{align*} x_1^3\ldots x_m^3x_{m+1}^2\ldots x_k^2&= Sq^1(x_1^3\ldots x_m^3x_{m+1}x_{m+2}^2\ldots x_k^2)\\ &\qquad + \sum_{i=1}^mx_1^3\ldots x_i^4\ldots x_m^3x_{m+1}x_{m+2}^2\ldots x_k^2.\end{align*}
Since 
\begin{align*} \tau(x_1^3\ldots x_i^4\ldots x_m^3x_{m+1}x_{m+2}^2\ldots x_k^2)&=(m;k-2;1)\\ &<(m;k)=\tau(x_1^3\ldots x_m^3x_{m+1}^2\ldots x_k^2),\end{align*}
for $i=1,2,\ldots ,m$, the lemma holds.  
\end{proof}
\medskip
From Theorem \ref{2.4} and Lemma \ref{2.5}, we easily obtain

\begin{prop}\label{2.6} Let $x$ be an admissible monomial in $P_k$. Then we have

1. If there is an index $i_0$ such that $\tau_{i_0}(x)=0$ then $\tau_{i}(x)=0$ for all $i>i_0$.

2. If there is an index $i_0$ such that $\tau_{i_0}(x)<k$ then $\tau_{i}(x)<k$ for all $i>i_0$.
\end{prop}

For latter use, we set 
\begin{align*} 
Q_k &=\text{Span}\{x=x_1^{a_1}x_2^{a_2}\ldots x_k^{a_k} \ ; \ a_1a_2\ldots a_k=0\},
\\ R_k &= \text{Span}\{x=x_1^{a_1}x_2^{a_2}\ldots x_k^{a_k} \ ; \ a_1a_2\ldots a_k>0\}. 
\end{align*}

It is easy to see that $Q_k$ and $R_k$ are the $\mathcal{A}$-submodules of $P_k$. Furethermore, we have the following.

\begin{prop}\label{2.7} We have a direct summand decomposition of the $\mathbb F_2$-vector spaces
$$\mathbb  F_2\underset{\mathcal{A}}  \otimes P_k = (\mathbb F_2\underset{\mathcal{A}}  \otimes  Q_k) \oplus  (\mathbb F_2\underset{\mathcal{A}}  \otimes R_k).$$
\end{prop}

 From this proposition,  if we know the set of all admissible monomials of $P_{k-1}$ then we can easily determine a basis of $\mathbb F_2\underset{\mathcal{A}} \otimes Q_k$. So, to determine $\mathbb F_2\underset{\mathcal{A}} \otimes P_k$, we need only to determine a basis of $\mathbb F_2\underset{\mathcal{A}} \otimes R_k$.

\smallskip
Now, we recall a result  in Kameko \cite{ka} on the squaring operation.

\begin{defn}\label{2.8}  Define the homomorphisms $ \phi , \overline{Sq^0_*}: P_k \to P_k$ by
\begin{align*} \phi(x) &= x_1x_2\ldots x_kx^2,\\
\overline{Sq^0_*}(x) &= \begin{cases} y, &\text{if } x=\phi(y),\\ 0, &\text{otherwise},\end{cases}
\end{align*}
for any monomial $x$ in $P_k$.
\end{defn}

Note that $\overline{Sq^0_*}$  commutes with the doubling map on $\mathcal{A}$. That is $Sq^{t}\overline{Sq^0_*} = \overline{Sq^0_*} Sq^{2t}$. Hence, $\overline{Sq^0_*}$ induces a homomorphism of $\mathbb F_2$-vector spaces
$$Sq^0_*: (\mathbb F_2 \underset {\mathcal{A}} \otimes P_k)_n \longrightarrow (\mathbb F_2 \underset {\mathcal{A}} \otimes P_k)_{\frac{n-k}2},$$
for all $n\geqslant k$ such that $n-k$ is even.

In general, $Sq^{2t}\phi  \ne \phi  Sq^t$.  However, 
in one particular situation, we have the following.

\begin{thm}[Kameko \cite{ka}] \label{2.9} If $\beta(n)=k$ then $Sq^0_*$ is an isomorphism of vector spaces and $\phi $ induces an inverse of it.
\end{thm}

Next, we recall some results of Wood \cite{wo} and Singer \cite{si2} on the hit monomials in $P_k$. 

\begin{thm}[Wood \cite{wo}] \label{2.10} Let $x$ be a monomial of degree $n$. If $\beta(n) > \tau_1(x)$ then $x$ is hit.
\end{thm}

\begin{defn}\label{2.11}  A monomial $z=x_1^{b_1}x_2^{b_2}\ldots x_k^{b_k}$ is called the spike if $b_j=2^{s_j}-1$ for $s_j$ a non-negative integer and $j=1,2, \ldots , k$. If $z$ is the spike with $s_1>s_2>\ldots >s_{r-1}\geqslant s_r>0$ and $s_j=0$ for $j>r$ then it is called the minimal spike.
\end{defn}

The following is a main tool in our computation for $P_4$.

\begin{thm}[Singer \cite{si2}]\label{2.12} Suppose $x \in P_k$ is a monomial of degree $n$, where $\alpha(n+k)\leqslant k$. Let $z$ be the minimal spike of degree $n$. If $\tau(x) < \tau(z)$ then $x$ is hit.
\end{thm}

From this theorem, we see that if $z$ is a minimal spike then $\mathcal L_k(\tau(z)) \subset \mathcal A^+.P_k$.

\smallskip
We end this section by defining some homomorphisms of $\mathcal{A}$-modules from $P_4$ to $P_3$ and some endomorphisms of $P_4$ which will be used in the next sections.

Let $V_k$ denote the $\mathbb F_2$-vector subspace of $P_k$ generated by $x_1,x_2,\ldots ,x_k$. If $\widetilde{\varphi}: V_k\to V_{k'}$ is a homomorphism of $\mathbb F_2$-vector spaces then there exists uniquely  ring homomorphism  $\overline{\varphi}: P_k\to P_{k'}$ such that $\overline{\varphi}(x_i)= \widetilde{\varphi}(x_i)$ for $i=1,2,\ldots, k$. This homomorphism is also a homomorphism of $\mathcal{A}$-modules. So, it induces a homomorphism of $\mathbb F_2$-vector spaces $\varphi : \mathbb F_2\underset{\mathcal{A}} \otimes P_k \to \mathbb F_2\underset{\mathcal{A}} \otimes P_{k'}$.

Note that if  $\overline{\varphi}:P_k \to P_k$ is induced by a permutation of $\{x_1,x_2,\ldots ,x_k\}$ then it sends monomials to monomials and preserves the associated $\tau$-sequences.

In particular, consider the $\mathcal{A}$-homomorphisms $\overline{f}_i, \overline{g}_j, 1\leqslant i\leqslant 6, 1\leqslant j \leqslant 4$ and $\overline{h}$ from $P_4$ to $P_3$ defined by the following table

\smallskip
\centerline{\begin{tabular}{clccc}\offinterlineskip
$x$ &\vdq \qquad $x_1$ &  $x_2$ & $x_3$ & $x_4$ \cr
\hline
$\overline{f}_1(x)$ &\vdq\qquad $ x_1$ &  $x_1$ & $x_2$ & $x_3$ \cr
$\overline{f}_2(x)$ &\vdq\qquad $x_1$ &  $x_2$ & $x_1$ & $x_3$ \cr
$\overline{f}_3(x)$ &\vdq\qquad $x_1$ &  $x_2$ & $x_3$ & $x_1$ \cr
$\overline{f}_4(x)$ &\vdq\qquad $x_1$ &  $x_2$ & $x_2$ & $x_3$ \cr
$\overline{f}_5(x)$ &\vdq\qquad $x_1$ &  $x_2$ & $x_3$ & $x_2$ \cr
$\overline{f}_6(x)$ &\vdq\qquad $x_1$ &  $x_2$ & $x_3$ & $x_3$ \cr
$\overline{g}_1(x)$ &\vdq\quad $x_1+x_2$ & $x_1$ & $x_2$ & $x_3$ \cr
$\overline{g}_2(x)$ &\vdq\quad $x_1+x_2$ & $x_1$ & $x_3$ & $x_2$ \cr
$\overline{g}_3(x)$ &\vdq\quad $x_1+x_2$ & $x_3$ & $x_1$ & $x_2$ \cr
$\overline{g}_4(x)$ &\vdq\qquad $x_3$ & $x_1+x_2$ & $x_1$ & $x_2$ \cr
$\overline{h}(x)$ &\vdq $x_1+x_2+x_3$ & $x_1$ & $x_2$ & $x_3.$ \cr
\end{tabular}}

\medskip
These homomorphisms induce the homomorphisms $f_i,g_j$ and $h$ of $\mathbb F_2$-vector spaces  from $\mathbb F_2\underset {\mathcal{A}} \otimes P_4$ to $\mathbb F_2\underset {\mathcal{A}} \otimes P_3$. 

We use the above homomorphisms and the results in Kameko \cite{ka} to prove a certain subset of $\mathbb F_2\underset {\mathcal{A}} \otimes P_4$ is linearly independent. More precisely, to prove a subset of $\mathbb F_2\underset {\mathcal{A}} \otimes P_4$ is linearly independent, we consider a linear relation of the elements in this subset with coefficients in $\mathbb F_2$. By using Theorem \ref{2.12}, we compute the images of this linear relation under the action of the  homomorphisms $f_i, g_j, h$. From the relations in $\mathbb F_2\underset {\mathcal{A}} \otimes P_3$ and the results in \cite{ka}, we can show that all coefficients in the linear relation are zero.

\smallskip
Next, consider the  endomorphisms $\overline{\varphi}_i, i=1,2,3,4,$ of $P_4$ defined by the following table

\smallskip
\centerline{\begin{tabular}{clccc}\offinterlineskip
$x$ &\vdq\ \quad $x_1$ &  $x_2$ & $x_3$ & $x_4$ \cr
\hline
$\overline{\varphi}_1(x)$ &\vdq\quad\ $ x_2$ &  $x_1$ & $x_3$ & $x_4$ \cr
$\overline{\varphi}_2(x)$ &\vdq\quad\ $x_1$ &  $x_3$ & $x_2$ & $x_4$ \cr
$\overline{\varphi}_3(x)$ &\vdq\quad\ $x_1$ &  $x_2$ & $x_4$ & $x_3$ \cr
$\overline{\varphi}_4(x)$ &\vdq $x_1+x_2$ &  $x_2$ & $x_3$ & $x_4$. \cr
\end{tabular}}

\medskip
These endomorphisms induce the isomorphisms $\varphi_i, i = 1,2,3,4,$ from $\mathbb F_2$-vector space   $\mathbb F_2\underset {\mathcal{A}} \otimes P_4$ to itself.

Note that the general linear group $GL_4(\mathbb F_2)$ is generated by $\overline{\varphi}_i, i=1,2,3,4$.

\smallskip
For simplicity, from now on, we denote the monomial $x=x_1^{a_1}x_2^{a_2}\ldots x_k^{a_k}$ in $P_k$ by $(a_1,a_2,\ldots,a_k)$ and denote by $[x] = [a_1,a_2,\ldots,a_k]$ the class in $\mathbb F_2\underset{\mathcal{A}} \otimes P_k$ represented by the monomial  $(a_1,a_2,\ldots,a_k)$.

Suppose $I$ is a finite subset of non-negative integers and $\gamma_i \in \mathbb F_2$ for $ i\in I$. Denote by
$$\gamma_I = \sum_{i\in I}\gamma_i.$$

\section{The indecomposables of $P_4$ in degree $2^{s+1}-3$}\label{3}

The purpose of this section is to prove Theorem \ref{1.1} for the case $n = 2^{s+1}-3$.

According to Kameko \cite{ka}, $\dim (\mathbb F_2\underset{\mathcal A}\otimes P_3)_{2^{s+1}-3} =3$ with a basis given by the classes
\begin{align*}&v_{s,1}= [2^{s-1}-1,2^{s-1}-1,2^{s}-1], \quad v_{s,2}=  [2^{s-1}-1,2^{s}-1,2^{s-1}-1],\\
 &v_{s,3}= [2^{s}-1,2^{s-1}-1,2^{s-1}-1].
\end{align*}
So, we easily obtain

\begin{prop}\label{3.5} \

1. $(\mathbb F_2\underset{\mathcal A}\otimes P_4)_{1}$ is  an $\mathbb F_2$-vector space of dimension 4 with a basis consisting of all the classes represented by the following monomials:
$$a_{1,1} = (0,0,0,1),\ a_{1,2} = (0,0,1,0),\ a_{1,3} = (0,1,0,0),\ a_{1,4} = (1,0,0,0).$$

2. For $s \geqslant 2$, $(\mathbb F_2\underset{\mathcal A}\otimes Q_4)_{2^{s+1}-3}$ is  an $\mathbb F_2$-vector space of dimension 12 with a basis consisting of all the classes represented by the following monomials:

\medskip
\centerline{\begin{tabular}{ll}\offinterlineskip
$a_{s,1} = (0,2^{s-1} - 1,2^{s-1} - 1,2^s - 1),$& $a_{s,2} = (0,2^{s-1} - 1,2^s - 1,2^{s-1} - 1),$\cr 
$a_{s,3} = (0,2^s - 1,2^{s-1} - 1,2^{s-1} - 1),$& $a_{s,4} = (2^{s-1}-1,0,2^{s-1} - 1,2^s - 1),$\cr 
$a_{s,5} = (2^{s-1}-1,0,2^s - 1,2^{s-1} - 1),$& $a_{s,6} = (2^{s-1}-1,2^{s-1} - 1,0,2^s - 1),$\cr 
$a_{s,7} = (2^{s-1}-1,2^{s-1} - 1,2^s - 1,0),$& $a_{s,8} = (2^{s-1}-1,2^s - 1,0,2^{s-1} - 1),$\cr 
$a_{s,9} = (2^{s-1}-1,2^s - 1,2^{s-1} - 1,0),$& $a_{s,10} = (2^s-1,0,2^{s-1} - 1,2^{s-1} - 1),$\cr 
$a_{s,11} = (2^s-1,2^{s-1} - 1,0,2^{s-1} - 1),$& $a_{s,12} = (2^s-1,2^{s-1} - 1,2^{s-1} - 1,0).$\cr
\end{tabular}}
\end{prop}

By Proposition \ref{2.7}, we need only to determine $(\mathbb F_2\underset{\mathcal A}\otimes R_4)_{2^{s+1}-3}$. 

Set $\mu_1(1) = 4, \mu_1(2)=15, \mu_1(3)=35$ and $\mu_1(s) = 45$ for $s \geqslant 4$. We have

\begin{thm}\label{dlc3} For any integer $s \geqslant 2$, $(\mathbb F_2 \underset {\mathcal A} \otimes R_4)_{2^{s+1}-3}$ is an $\mathbb F_2$-vector space of dimension $\mu_1(s)-12$
with a basis consisting of all the classes represented by the following monomials:

\smallskip
For $s=2$,

\centerline{\begin{tabular}{lll}
$a_{2,13} = (1,1,1,2),$& $a_{2,14} = (1,1,2,1),$& $a_{2,15} = (1,2,1,1)$.\cr
\end{tabular}}

\smallskip
For $s\geqslant 3$,

\smallskip
\centerline{\begin{tabular}{ll}
$a_{s,13} = (1,2^{s-1} - 2,2^{s-1} - 1,2^s - 1),$& $a_{s,14} = (1,2^{s-1} - 2,2^s - 1,2^{s-1} - 1),$\cr 
$a_{s,15} = (1,2^{s-1} - 1,2^{s-1} - 2,2^s - 1),$& $a_{s,16} = (1,2^{s-1} - 1,2^s - 1,2^{s-1} - 2),$\cr 
$a_{s,17} = (1,2^s - 1,2^{s-1} - 2,2^{s-1} - 1),$& $a_{s,18} = (1,2^s - 1,2^{s-1} - 1,2^{s-1} - 2),$\cr 
$a_{s,19} = (2^{s-1}-1,1,2^{s-1} - 2,2^s - 1),$& $a_{s,20} = (2^{s-1}-1,1,2^s - 1,2^{s-1} - 2),$\cr 
$a_{s,21} = (2^{s-1}-1,2^s - 1,1,2^{s-1} - 2),$& $a_{s,22} = (2^s-1,1,2^{s-1} - 2,2^{s-1} - 1),$\cr 
$a_{s,23} = (2^s-1,1,2^{s-1} - 1,2^{s-1} - 2),$& $a_{s,24} = (2^s-1,2^{s-1} - 1,1,2^{s-1} - 2),$\cr 
$a_{s,25} = (1,2^{s-1} - 1,2^{s-1} - 1,2^s - 2),$& $a_{s,26} = (1,2^{s-1} - 1,2^s - 2,2^{s-1} - 1),$\cr 
$a_{s,27} = (1,2^s - 2,2^{s-1} - 1,2^{s-1} - 1),$& $a_{s,28} = (2^{s-1}-1,1,2^{s-1} - 1,2^s - 2),$\cr 
$a_{s,29} = (2^{s-1}-1,1,2^s - 2,2^{s-1} - 1),$& $a_{s,30} = (2^{s-1}-1,2^{s-1} - 1,1,2^s - 2).$\cr
\end{tabular}}

\smallskip
For $s=3$,
\begin{align*} 
&a_{3,31} = (3,3,5,2), \quad a_{3,32} = (3,5,2,3),\quad a_{3,33} = (3,5,3,2),\\
&a_{3,34} = (3,3,3,4),\quad a_{3,35} = (3,3,4,3).
\end{align*}

\smallskip
For $s \geqslant 4$,

\smallskip
\centerline{\begin{tabular}{ll}\offinterlineskip
$a_{s,31} = (3,2^{s-1} - 3,2^{s-1} - 2,2^s - 1),$& $a_{s,32} = (3,2^{s-1} - 3,2^s - 1,2^{s-1} - 2),$\cr 
$a_{s,33} = (3,2^s - 1,2^{s-1} - 3,2^{s-1} - 2),$& $a_{s,34} = (2^s-1,3,2^{s-1} - 3,2^{s-1} - 2),$\cr 
$a_{s,35} = (3,2^{s-1} - 3,2^{s-1} - 1,2^s - 2),$& $a_{s,36} = (3,2^{s-1} - 3,2^s - 2,2^{s-1} - 1),$\cr 
$a_{s,37} = (3,2^{s-1} - 1,2^{s-1} - 3,2^s - 2),$& $a_{s,38} = (2^{s-1}-1,3,2^{s-1} - 3,2^s - 2),$\cr 
$a_{s,39} = (3,2^{s-1} - 1,2^s - 3,2^{s-1} - 2),$& $a_{s,40} = (3,2^s - 3,2^{s-1} - 2,2^{s-1} - 1),$\cr 
$a_{s,41} = (3,2^s - 3,2^{s-1} - 1,2^{s-1} - 2),$& $a_{s,42} = (2^{s-1}-1,3,2^s - 3,2^{s-1} - 2),$\cr 
$a_{s,43} = (7,2^s - 5,2^{s-1} - 3,2^{s-1} - 2).$& \cr
\end{tabular}}

\smallskip
For $s=4$,
$$ a_{4,44} = (7,7,9,6),\quad a_{4,45} = (7,7,7,8).$$

For $s \geqslant 5$,
$$ a_{s,44} = (7,2^{s-1} - 5,2^{s-1} - 3,2^s - 2), \quad  a_{s,45} = (7,2^{s-1} - 5,2^s - 3,2^{s-1} - 2).$$
\end{thm}

We prove this theorem by proving the following propositions.

\begin{prop}\label{mdc3.1} For $s \geqslant 2$, the $\mathbb F_2$-vector space $(\mathbb F_2\underset {\mathcal A}\otimes R_4)_{2^{s+1}-3}$ is generated by the $\mu_1(s)-12$ elements listed in Theorem \ref{dlc3}.
\end{prop}

The proposition is proved by combining Theorem \ref{2.4} and the following lemmas.

\begin{lem}\label{3.1} If $x$ is an admissible monomial of degree $2^{s+1}-3$ in $P_4$ then $$\tau (x) = (\underset{\text{$s-1$ times}}{\underbrace{3;3;\ldots ; 3}};1).$$
\end{lem}

\begin{proof} It is easy to see that the lemma holds for $s=1$. Suppose $s\geqslant 2$. Obviously, $z=(2^s-1,2^{s-1}-1,2^{s-1}-1,0)$ is the minimal spike of degree $2^{s+1}-3$ in $P_4$ and $ \tau (z) = (\underset{\text{$s-1$ times}} {\underbrace {3;3;\ldots ; 3}};1)$. Since  $2^{s+1}-3$ is odd, we get either $\tau_1(x)=1$ or $\tau_1(x)=3$. If $\tau_1(x)=1$ then $\tau(x) < \tau(z)$. By Theorem \ref{2.12}, $x$ is hit. This contradicts the fact that $x$ is admissible. Hence, we have $\tau_1(x) =3$. Using Proposition \ref{2.6} and Theorem \ref{2.12}, we obtain $\tau_i(x) = 3, \ i=1,2,\ldots , s-1$. From this, it implies
$$2^{s+1}-3= \deg x =\sum_{i\geqslant 1}2^{i-1}\tau_i(x) = 3(2^{s-1}-1)+\sum_{i\geqslant s}2^{i-1}\tau_i(x).$$
The last equality implies $\tau_{s}(x)=1$ and $\tau_i(x) = 0$ for $i>s$. The lemma is proved.   
\end{proof}

By this lemma, it suffices to consider monomials whose associated $\tau$-sequences are $(\underset{\text{$s-1$ times}}{\underbrace{3;3;\ldots ; 3}};1)$.

\begin{lem}\label{3.2} The following matrices are strictly inadmissible
$$\Delta_1 = \begin{pmatrix} 1&1&0&1\\  1&1&1&0\end{pmatrix},\ \Delta_2 = \begin{pmatrix} 1&0&1&1\\ 1&1&1&0\end{pmatrix},\ \Delta_3 =  \begin{pmatrix}1&0&1&1\\  1&1&0&1\end{pmatrix}, $$ 
$$\Delta_4 =  \begin{pmatrix} 0&1&1&1\\  1&1&1&0\end{pmatrix},\   \Delta_5 =  \begin{pmatrix} 0&1&1&1\\  1&1&0&1\end{pmatrix},\   \Delta_6 =  \begin{pmatrix} 0&1&1&1\\  1&0&1&1\end{pmatrix}. $$
\end{lem}

\begin{proof} The monomials corresponding to six matrices in the lemma respectively are $$(3,3,2,1), (3,2,3,1),(3,2,1,3), (2,3,3,1), (2,3,1,3), (2,1,3,3).$$ 
If $x$ is one of these monomials then $\tau(x)=(3;3)$ and there is an $\mathcal A$-homomorphism $f: P_4\to P_4$ induced by a permutation of $\{x_1, x_2, x_3, x_4\}$ such that $x=f(3,3,2,1)$.  Furethermore, $f$ sends monomials to monomials and preserves the associated $\tau$-sequences. By a direct computation, we have
$$(3,3,2,1) = Sq^1(3,3,1,1)+(4,3,1,1)+(3,4,1,1)+(3,3,1,2).$$
It is easy to see that 
\begin{align*} \tau(4,3,1,1) &=\tau(3,4,1,1)= (3;1;1) < (3;3),\\
\tau(3,3,1,2) &= \tau(3,3,2,1) \text{ and } \sigma(3,3,1,2) < \sigma(3,3,2,1).
\end{align*} 
Hence, the lemma is proved.  
\end{proof}

\begin{lem}\label{3.3} The following matrices are strictly inadmissible
$$\Delta_7 =  \begin{pmatrix} 1&1&1&0\\  1&1&1&0\\  1&1&0&1 \end{pmatrix}, \    \Delta_8 = \begin{pmatrix} 1&1&1&0\\  1&1&1&0\\  1&0&1&1 \end{pmatrix}, \    \Delta_9 = \begin{pmatrix} 1&1&0&1\\  1&1&0&1\\  1&0&1&1 \end{pmatrix}, $$ 
$$ \Delta_{10} = \begin{pmatrix} 1&1&1&0\\  1&1&1&0\\  0&1&1&1  \end{pmatrix}, \   \Delta_{11} =  \begin{pmatrix} 1&1&0&1\\  1&1&0&1\\  0&1&1&1  \end{pmatrix}, \   \Delta_{12} =  \begin{pmatrix} 1&0&1&1\\  1&0&1&1\\  0&1&1&1  \end{pmatrix} .$$
\end{lem}

\begin{proof} The monomials corresponding to six matrices in this lemma respectively are $$(7,7,3,4), (7,3,7,4),(7,3,4,7), (3,7,7,4), (3,7,4,7), (3,4,7,7).$$ 
If $x$ is one of these monomials then $\tau(x)=(3;3;3)$ and there is an $\mathcal A$-homomorphism $g: P_4 \to P_4$ induced by a permutation of $\{x_1, x_2, x_3, x_4\}$ such that $x=g(7,7,3,4)$.  By a direct computation, we get
\begin{align*} (7,7,3,4)&=Sq^1\big((7,7,3,3) + (8,7,2,3)\big) +Sq^2(7,7,2,3)\\
&\quad +(8,7,3,3)+(7,8,3,3) +(9,7,2,3)\\
&\quad+(7,9,2,3)+(7,8,2,4) +(7,7,2,5).
\end{align*}
 So, the lemma is proved.  
\end{proof}

\begin{lem}\label{3.4} The following matrices are strictly inadmissible
 $$\Delta_{13} = \begin{pmatrix} 0&1&1&1\\  1&0&0&0\end{pmatrix}, \quad \Delta_{14} =   
 \begin{pmatrix} 1&0&1&1\\  1&0&1&1\\  0&1&0&0\end{pmatrix}, $$
$$\Delta_{15} =  
 \begin{pmatrix} 1&1&0&1\\  1&1&0&1\\  1&1&0&1\\  0&0&1&0\end{pmatrix}, \quad  \Delta_{16} = 
\begin{pmatrix} 1&1&1&0\\  1&1&1&0\\  1&1&1&0\\  1&1&1&0\\  0&0&0&1 \end{pmatrix} .$$
\end{lem}

\begin{proof} The monomials corresponding to the above matrices are 
$$(2,1,1,1), (3,4,3,3), (7,7,8,7), (15,15,15,16).$$

By a direct computation, we have
$$(2,1,1,1)= Sq^1(1,1,1,1)+(1,2,1,1)+(1,1,2,1)+(1,1,1,2).$$
Since $\tau(1,2,1,1) = \tau(1,1,2,1) = \tau(1,1,1,2) = \tau(2,1,1,1) =(3;1)$ and $\sigma(1,2,1,1) $, $\sigma(1,1,2,1), \sigma(1,1,1,2) < \sigma(2,1,1,1) $, the lemma is true.  
\begin{align*} (3,4,3,3)&=Sq^1\big((3,3,3,3)+(2,4,3,3)\big)+Sq^2(2,3,3,3)\\ &\quad+(3,3,4,3)+(3,3,3,4) +(2,5,3,3)\\ 
&\quad +(2,3,5,3)+(2,3,3,5)+(2,3,4,4). 
\end{align*}

Let $y$ be a monomial in $(P_4)_{13}$. If $y$ is a term in the right hand side of the above equality and $y \ne (2,3,4,4)$ then $\tau(y) = \tau(3,4,3,3)=(3;3;1)$ and $\sigma(y) < \sigma(3,4,3,3)$. If $y =(2,3,4,4)$ then $\tau(y) =(1;2;2) < (3;3;1)$. Hence, the lemma holds. 
\begin{align*}  &(7,7,8,7)= Sq^1\big((7,7,7,7) + (6,8,7,7) + (8,6,7,7) + (6,6,9,7)\\
&\quad +(6,6,7,9)\big) + Sq^2\big((6,7,7,7)+(7,6,7,7)+(5,8,7,7)\big)\\
&\quad  +Sq^4(5,6,7,7) +(7,7,7,8)+(6,9,7,7)+(6,7,9,7) +(6,7,7,9)\\
&\quad +(7,6,9,7)  +(7,6,7,9)+(5,10,7,7)+(5,6,11,7)+(5,6,7,11)\\
&\quad+(6,7,8,8) + (7,6,8,8) +(5,6,9,9) + (5,6,10,8) + (5,6,8,10).
\end{align*}

Using the above equality and by an argument analogous to the previous one, we see that the lemma is true.   

By a direct computation, we get
\begin{align*}
&(15,15,15,16)= Sq^1(15,15,15,15) + Sq^2\big((14,15,15,15)+ (15,14,15,15)\\ 
&\ + (15,15,14,15)\big) +Sq^4\big((13,14,15,15)+(15,13,14,15)\big)\\ 
&\ +  Sq^8(11,13,14,15)  +(14,17,15,15)+(14,15,17,15)  +(14,15,15,17)\\ 
&\ +(15,14,17,15)+(15,14,15,17) +(15,15,14,17) +(13,18,15,15)\\
&\ +(13,14,19,15) + (13,14,15,19) +(15,13,18,15) +(15,13,14,19)\\  
&\ +(11,21,14,15) +(11,13,22,15) +(11,13,14,23)\ \text{mod }\mathcal L_4(3;3;3;3;1).
\end{align*}

Let $y$ be a monomial in $(P_4)_{61}$. If $y$ is a term of the right hand side of the above equality then $\tau(y) = \tau(15,15,15,16)=(3;3;3;3;1)$ and $\sigma(y) < \sigma(15,15,15,16)$. If $y \in \mathcal L_4(3;3;3;3;1)$ then $\tau(y) < (3;3;3;3;1)$. So $y < (15,15,15,16)$  and the lemma holds.
\end{proof}

Now we prove Proposition \ref{mdc3.1} by 
using Theorem \ref{2.4}, Lemmas \ref{3.1}, \ref{3.2}, \ref{3.3}, \ref{3.4}.

\begin{proof}[Proof of Proposition \ref{mdc3.1}] 
Let $x$ be a monomial of degree $2^{s+1}-3$ in $P_4$ with $\tau_i(x) = 3, i = 1,2, \ldots , s-1, \tau_s(x) =1, \tau_i(x) = 0, i > s.$   Then we can write $x=z_iy^2$ for $1 \leqslant i \leqslant 4$, where
$$z_1 = (0,1,1,1),\ z_2 = (1,0,1,1),\ z_3 = (1,1,0,1),\ z_4 = (1,1,1,0),$$ 
and $y$ is a monomial of degree $2^s-3$ with $\tau_i(y) = 3, i = 1,2, \ldots , s-2, \tau_{s-1}(y) =1, \tau_i(y) = 0, i \geqslant s$.

We prove the proposition by showing that if  $ x \ne a_{s,j}, 1 \leqslant j \leqslant \mu_1(s)$ then there is a strictly inadmissible matrix $\Delta$ such that $\Delta \triangleright x$.

The proof proceeds by induction on $s$. The case $s=2$ is trivial.  By the inductive hypothesis and Theorem \ref{2.4}, it suffices to consider  monomials $x=z_iy^2$ with $y=a_{s,j}, 1 \leqslant j \leqslant \mu_1(s)$.

For $s=3$, a direct computation shows

\smallskip
\centerline{\begin{tabular}{llll}
$a_{3,1} = z_{1}a_{2,1}^2$,& $a_{3,2} = z_{1}a_{2,2}^2$,& $a_{3,3} = z_{1}a_{2,3}^2$,& $a_{3,4} = z_{2}a_{2,4}^2$,\cr
$a_{3,5} = z_{2}a_{2,5}^2$,&   
$a_{3,6} = z_{3}a_{2,6}^2$,& $a_{3,7} = z_{4}a_{2,7}^2$,& $a_{3,8} = z_{3}a_{2,8}^2$,\cr $a_{3,9} = z_{4}a_{2,9}^2$,& $a_{3,10} = z_{2}a_{2,10}^2$,& $a_{3,11} = z_{3}a_{2,11}^2$,& $a_{3,12} = z_{4}a_{2,12}^2$,\cr 
$a_{3,13} = z_{2}a_{2,1}^2$,& $a_{3,14} = z_{2}a_{2,2}^2$,& $a_{3,15} = z_{3}a_{2,1}^2$,& $a_{3,16} = z_{4}a_{2,2}^2$,\cr 
$a_{3,17} = z_{3}a_{2,3}^2$,& $a_{3,18} = z_{4}a_{2,3}^2$,& $a_{3,19} = z_{3}a_{2,4}^2$,& $a_{3,20} = z_{4}a_{2,5}^2$,\cr   
$a_{3,21} = z_{4}a_{2,8}^2$,& $a_{3,22} = z_{3}a_{2,10}^2$,& $a_{3,23} = z_{4}a_{2,10}^2$,& $a_{3,24} = z_{4}a_{2,11}^2$,\cr   
$a_{3,25} = z_{4}a_{2,1}^2$,& $a_{3,26} = z_{3}a_{2,2}^2$,& $a_{3,27} = z_{2}a_{2,3}^2$,& $a_{3,28} = z_{4}a_{2,4}^2$,\cr 
$a_{3,29} = z_{3}a_{2,5}^2$,& $a_{3,30} = z_{4}a_{2,6}^2$,&$a_{3,31} = z_{4}a_{2,14}^2$,& $a_{3,32} = z_{3}a_{2,15}^2$,\cr $a_{3,33} = z_{4}a_{2,15}^2$,& $a_{3,34} = z_{4}a_{2,13}^2$,& $a_{3,35} = z_{3}a_{2,14}^2$.& \cr 
\end{tabular}}

\smallskip
We have
$\Delta_1 \triangleright z_3a_{2,j}^2$,  
$\Delta_2 \triangleright z_2a_{2,j}^2$, 
$\Delta_4 \triangleright z_1a_{2,j}^2$ for $j =$ 7, 9, 12, 13;
$\Delta_3 \triangleright z_2a_{2,j}^2$,
$\Delta_5 \triangleright z_1a_{2,j}^2$ for $j =$ 6, 8, 11, 14;\ 
$\Delta_6 \triangleright z_1a_{2,j}^2$ for $j =$ 4, 5, 10, 15;\ 
$\Delta_{14} \triangleright z_{2}a_{2,15}^2$. 

From the above equalities, Lemmas \ref{3.2}, \ref{3.3}, \ref{3.4} and Theorem \ref{2.4}, we see that if $x \ne  a_{3,j}, 1 \leqslant j \leqslant 35$ then $x$ is inadmissible. So, the case $s=3$ is true.

Let $s =4$. By  a direct computation we obtain

\smallskip
\centerline{\begin{tabular}{llll}
$a_{4,1} = z_{1}a_{3,1}^2$,& $a_{4,2} = z_{1}a_{3,2}^2$,& $a_{4,3} = z_{1}a_{3,3}^2$,& $a_{4,4} = z_{2}a_{3,4}^2$,\cr   
$a_{4,5} = z_{2}a_{3,5}^2$,& $a_{4,6} = z_{3}a_{3,6}^2$,& $a_{4,7} = z_{4}a_{3,7}^2$,& $a_{4,8} = z_{3}a_{3,8}^2$,\cr $a_{4,9} = z_{4}a_{3,9}^2$,& $a_{4,10} = z_{2}a_{3,10}^2$,& $a_{4,11} = z_{3}a_{3,11}^2$,& $a_{4,12} = z_{4}a_{3,12}^2$,\cr   
$a_{4,13} = z_{2}a_{3,1}^2$,& $a_{4,14} = z_{2}a_{3,2}^2$,& $a_{4,15} = z_{3}a_{3,1}^2$,& $a_{4,16} = z_{4}a_{3,2}^2$,\cr 
$a_{4,17} = z_{3}a_{3,3}^2$,& $a_{4,18} = z_{4}a_{3,3}^2$,& $a_{4,19} = z_{3}a_{3,4}^2$,& $a_{4,20} = z_{4}a_{3,5}^2$,\cr   
$a_{4,21} = z_{4}a_{3,8}^2$,& $a_{4,22} = z_{3}a_{3,10}^2$,& $a_{4,23} = z_{4}a_{3,10}^2$,& $a_{4,24} = z_{4}a_{3,11}^2$,\cr   
$a_{4,25} = z_{4}a_{3,1}^2$,&$a_{4,26} = z_{3}a_{3,2}^2$,& $a_{4,27} = z_{2}a_{3,3}^2$,& $a_{4,28} = z_{4}a_{3,4}^2$,\cr   
$a_{4,29} = z_{3}a_{3,5}^2$,& $a_{4,30} = z_{4}a_{3,6}^2$,& $a_{4,31} = z_{3}a_{3,13}^2$,& $a_{4,32} = z_{4}a_{3,14}^2$,\cr   
$a_{4,33} = z_{4}a_{3,17}^2$,& $a_{4,34} = z_{4}a_{3,22}^2$,& $a_{4,35} = z_{4}a_{3,13}^2$,& $a_{4,36} = z_{3}a_{3,14}^2$,\cr   
\end{tabular}}
\centerline{\begin{tabular}{llll}
$a_{4,37} = z_{4}a_{3,15}^2$,& $a_{4,38} = z_{4}a_{3,19}^2$,& $a_{4,39} = z_{4}a_{3,26}^2$,& $a_{4,40} = z_{3}a_{3,27}^2$,\cr   
$a_{4,41} = z_{4}a_{3,27}^2$,& $a_{4,42} = z_{4}a_{3,29}^2$,& $a_{4,43} = z_{4}a_{3,32}^2$,& $a_{4,44} = z_{4}a_{3,35}^2$,\cr $a_{4,45} = z_{4}a_{3,34}^2$.&&&\cr
\end{tabular}}

\smallskip
We have also $\Delta_1 \triangleright z_3a_{3,j}^2,\ \Delta_{2} \triangleright z_2a_{3,j}^2,\ \Delta_{4} \triangleright z_1a_{3,j}^2$ for $j  =$ 7, 9, 12, 16, 18, 20, 21, 23, 24, 25, 28, 30, 31, 33, 34;\
$\Delta_{3} \triangleright z_2a_{3,j}^2,\ \Delta_{5} \triangleright z_1a_{3,j}^2$ for $j  =$ 6, 8, 11, 15, 17, 19, 22, 26, 29, 32, 35;\
$\Delta_{6} \triangleright z_1a_{3,j}^2$ for $j  =$ 4, 5, 10, 13, 14, 27;\ 
$\Delta_{7} \triangleright z_4a_{3,j}^2$ for $j  =$ 21, 24, 30, 31;
$\Delta_{8} \triangleright z_4a_{3,j}^2$ for $j  =$ 20, 23, 28, 33;\ 
$\Delta_{9} \triangleright z_3a_{3,j}^2$ for $j  =$ 19, 22, 29, 32;\
$\Delta_{10} \triangleright z_4a_{3,j}^2$ for $j  =$ 16, 18, 25;\ 
$\Delta_{11} \triangleright z_3a_{3,j}^2$ for $j  =$ 15, 17, 26;\
$\Delta_{12} \triangleright z_2a_{3,j}^2$ for $j  =$ 13, 14, 27;\ 
$\Delta_{15} \triangleright z_3a_{3,25}^2$.  These equalities, Lemmas \ref{3.2}, \ref{3.3}, \ref{3.4} and Theorem \ref{2.4} imply that if $x\ne a_{4,j}, 1\leqslant j \leqslant 45$ then $x$ is inadmissible. Hence, the case $s=4$ is true.

Suppose $s\geqslant 4$ and the proposition holds for $s$. By a direct computation, we have  $\Delta_{1} \triangleright z_3a_{s,j}^2$, $\Delta_{2} \triangleright z_2a_{s,j}^2$, 
$\Delta_{4} \triangleright z_1a_{s,j}^2$ for $j  =$ 7, 9, 12, 16, 18, 20, 21, 23, 24, 25, 28, 30, 32, 33, 34, 35, 37, 38, 39, 41, 42, 43, 44, 45;\
$\Delta_{3} \triangleright z_2a_{s,j}^2$, $\Delta_{5} \triangleright z_1a_{s,j}^2$ for $j  =$ 6, 8, 11, 15, 17, 19, 22, 26, 29, 31, 36, 40;\
$\Delta_{6} \triangleright z_1a_{s,j}^2$ for $j  =$ 4, 5, 10, 13, 14, 27;\
$\Delta_{7} \triangleright z_4a_{s,j}^2$ for $j  =$ 21, 24, 30, 33, 34, 37, 38, 39, 42, 43, 44;\
$\Delta_{16} \triangleright z_4a_{s,45}^2$ for $s=4$; $\Delta_{7} \triangleright z_4a_{s,45}^2$ for $s\geqslant 5$;\
$\Delta_{8} \triangleright z_4a_{s,j}^2$ for $j  =$ 20, 23, 28, 32, 35, 41;\ 
$\Delta_{9} \triangleright z_3a_{s,j}^2$ for $j  =$ 19, 22, 29, 31, 36, 40;\
$\Delta_{10} \triangleright z_4a_{s,j}^2$ for $j  =$ 16, 18, 25;\
$\Delta_{11} \triangleright z_3a_{s,j}^2$ for $j  =$ 15, 17, 26;\
$\Delta_{12} \triangleright z_2a_{s,j}^2$ for $j  =$ 13, 14, 27 and

\smallskip
\centerline{\begin{tabular}{llll}
$a_{s+1,1} = z_{1}a_{s,1}^2$,& $a_{s+1,2} = z_{1}a_{s,2}^2$,& $a_{s+1,3} = z_{1}a_{s,3}^2$,& $a_{s+1,4} = z_{2}a_{s,4}^2$,\cr   
$a_{s+1,5} = z_{2}a_{s,5}^2$,& $a_{s+1,6} = z_{3}a_{s,6}^2$,& $a_{s+1,7} = z_{4}a_{s,7}^2$,& $a_{s+1,8} = z_{3}a_{s,8}^2$,\cr   
$a_{s+1,9} = z_{4}a_{s,9}^2$,& $a_{s+1,10} = z_{2}a_{s,10}^2$,& $a_{s+1,11} = z_{3}a_{s,11}^2$,& $a_{s+1,12} = z_{4}a_{s,12}^2$,\cr   
$a_{s+1,13} = z_{2}a_{s,1}^2$,& $a_{s+1,14} = z_{2}a_{s,2}^2$,& $a_{s+1,15} = z_{3}a_{s,1}^2$,& $a_{s+1,16} = z_{4}a_{s,2}^2$,\cr   
$a_{s+1,17} = z_{3}a_{s,3}^2$,& $a_{s+1,18} = z_{4}a_{s,3}^2$,& $a_{s+1,19} = z_{3}a_{s,4}^2$,& $a_{s+1,20} = z_{4}a_{s,5}^2$,\cr   
$a_{s+1,21} = z_{4}a_{s,8}^2$,& $a_{s+1,22} = z_{3}a_{s,10}^2$,& $a_{s+1,23} = z_{4}a_{s,10}^2$,& $a_{s+1,24} = z_{4}a_{s,11}^2$,\cr   
$a_{s+1,25} = z_{4}a_{s,1}^2$,& $a_{s+1,26} = z_{3}a_{s,2}^2$,& $a_{s+1,27} = z_{2}a_{s,3}^2$,& $a_{s+1,28} = z_{4}a_{s,4}^2$,\cr   
$a_{s+1,29} = z_{3}a_{s,5}^2$,& $a_{s+1,30} = z_{4}a_{s,6}^2$,& $a_{s+1,31} = z_{3}a_{s,13}^2$,& $a_{s+1,32} = z_{4}a_{s,14}^2$,\cr   
$a_{s+1,33} = z_{4}a_{s,17}^2$,& $a_{s+1,34} = z_{4}a_{s,22}^2$,& $a_{s+1,35} = z_{4}a_{s,13}^2$,& $a_{s+1,36} = z_{3}a_{s,14}^2$,\cr   
$a_{s+1,37} = z_{4}a_{s,15}^2$,& $a_{s+1,38} = z_{4}a_{s,19}^2$,& $a_{s+1,39} = z_{4}a_{s,26}^2$,& $a_{s+1,40} = z_{3}a_{s,27}^2$,\cr   $a_{s+1,41} = z_{4}a_{s,27}^2$,& $a_{s+1,42} = z_{4}a_{s,29}^2$,& $a_{s+1,43} = z_{4}a_{s,40}^2$,& $a_{s+1,44} = z_{4}a_{s,36}^2$,\cr   $a_{s+1,45} = z_{4}a_{s,31}^2$.&&&\cr
\end{tabular}}

\smallskip
Using these equalities, Lemmas \ref{3.2}, \ref{3.3}, \ref{3.4} and Theorem \ref{2.4}, we obtain the proposition for $s+1$. The proof is completed.
\end{proof}

Now, we show that $[a_{s,j}]; 13 \leqslant j\leqslant \mu_1(s)$, are linearly independent.

\begin{prop}\label{3.6} The elements $[a_{2,13}], [a_{2,14}], [a_{2,15}]$ are linearly independent in $(\mathbb F_2\underset {\mathcal A}\otimes R_4)_5$.
\end{prop}

\begin{proof} Suppose there is a linear relation
$ \gamma_{1}[a_{2,13}]+ \gamma_{2}[a_{2,14}]+ \gamma_{3}[a_{2,15}]=0,$
where $\gamma_1, \gamma_2, \gamma_3 \in \mathbb F_2.$ Consider the homomorphisms $f_1, f_2, f_3$. Under these homomorphisms, the images of the above linear relation respectively are
$$ \gamma_3[3,1,1]=0,\   \gamma_2[3,1,1]=0,\  \gamma_1[3,1,1]=0.$$
Hence, $\gamma_1=\gamma_2=\gamma_3=0.$ The proposition is proved. 
\end{proof}

\begin{prop}\label{3.7} The  elements $[a_{3,j}], 13 \leqslant j \leqslant 35$, are linearly independent in $(\mathbb F_2\underset {\mathcal A}\otimes R_4)_{13}$.
\end{prop}

\begin{proof} Suppose there is a linear relation
\begin{equation} \sum_{j=13}^{35} \gamma_{j}[a_{3,j}] = 0,\tag{\ref{3.7}.1} 
\end{equation}
where $\gamma_j \in \mathbb F_2, 13\leqslant j \leqslant 35$. 
By a direct calculation, we see that the homomorphisms $f_i,$ for $ i\leqslant i \leqslant 6$, send the relation (\ref{3.7}.1) respectively to
\begin{align*} 
&\gamma_{13}[3,3,7] +\gamma_{14}[3,7,3] +\gamma_{27}[7,3,3] = 0,\\
&\gamma_{15}[3,3,7]  + \gamma_{17}[3,7,3]  + \gamma_{\{26, 35\}}[7,3,3] = 0,\\
&\gamma_{16}[3,3,7]  + \gamma_{18}[3,7,3]  + \gamma_{\{25, 34\}}[7,3,3] = 0,\\
&\gamma_{19}[3,3,7]  + \gamma_{\{29, 32, 35\}}[3,7,3]  + \gamma_{22}[7,3,3] = 0,\\
&\gamma_{20}[3,3,7]  + \gamma_{\{28, 33, 34\}}[3,7,3]  + \gamma_{23}[7,3,3] = 0,\\
&\gamma_{\{30, 31, 34, 35\}}[3,3,7]  + \gamma_{21}[3,7,3]  +  \gamma_{24}[7,3,3] = 0.
\end{align*}
From the above equalities, we obtain 
\begin{equation}\begin{cases} 
\gamma_j = 0, \ j = 13, 14, 15, 16, 17, 18, 19, 20, 21, 22, 23, 24, 27, \\
\gamma_{\{26, 35\}}= 
\gamma_{\{25, 34\}}=
\gamma_{\{29, 32, 35\}}=  
\gamma_{\{28, 33, 34\}} = 
\gamma_{\{30, 31, 34, 35\}}=0.
\end{cases} \tag {\ref{3.7}.2} 
\end{equation}
Substituting (\ref{3.7}.2) into the relation (\ref{3.7}.1), we have
\begin{equation}\sum_{25\leqslant j\leqslant 35, j\ne 27} \gamma_{j}[a_{3,j}] = 0.\tag {\ref{3.7}.3}\end{equation}
The homomorphisms $g_i, i=1,2,3$, send (\ref{3.7}.3) respectively to
\begin{align*} 
&\gamma_{\{26, 29, 35\}}[3,7,3] + \gamma_{32}[7,3,3] = 0,\\
&\gamma_{\{25, 28, 34\}}[3,7,3] + \gamma_{33}[7,3,3] = 0,\\
&\gamma_{\{25, 30, 34\}}[3,7,3] + \gamma_{\{26, 31, 35\}}[7,3,3] = 0.
\end{align*}
Hence, we get 
\begin{equation}\begin{cases} 
\gamma_{32}=\gamma_{33}=0,\\
\gamma_{\{26, 29, 35\}} = 
\gamma_{\{25, 28, 34\}}=0,\\
\gamma_{\{25, 30, 34\}}=
\gamma_{\{26, 31, 35\}}=0.
\end{cases} \tag {\ref{3.7}.4} 
\end{equation}
Combining (\ref{3.7}.2) and (\ref{3.7}.4), we obtain $\gamma_j = 0, \ j = 13,14,\ldots , 35.$
 The proposition follows.    
\end{proof}

\begin{prop}\label{3.8} The elements $[a_{4,j}], 13 \leqslant j \leqslant 45$, are linearly independent in $(\mathbb F_2\underset {\mathcal A}\otimes R_4)_{29}$.
\end{prop}

\begin{proof} Suppose there is a linear relation
\begin{equation}\sum_{j=13}^{45} \gamma_{j}[a_{4,j}] = 0,\tag {\ref{3.8}.1} 
\end{equation}
where $\gamma_j \in \mathbb F_2, 13 \leqslant j \leqslant 45$. 

Applying the homomorphisms $f_i,$ for $ i=1,2,\ldots,6,$ to the relation (\ref{3.8}.1), we get
\begin{align*} 
&\gamma_{13}[7,7,15] +  \gamma_{14}[7,15,7] +   \gamma_{27}[15,7,7] = 0,\\  
&\gamma_{15}[7,7,15] +  \gamma_{17}[7,15,7] +  \gamma_{26}[15,7,7]  =0,\\  &\gamma_{16}[7,7,15] +  \gamma_{18}[7,15,7] +  \gamma_{\{25, 45\}}[15,7,7] = 0,\\
&\gamma_{19}[7,7,15] +  \gamma_{29}[7,15,7] +  \gamma_{22}[15,7,7] = 0,\\
&\gamma_{20}[7,7,15] + \gamma_{\{28, 45\}}[7,15,7] +  \gamma_{23}[15,7,7] = 0,\\  
&\gamma_{\{30, 44, 45\}}[7,7,15] +  \gamma_{21}[7,15,7] +  \gamma_{24}[15,7,7] = 0.  
\end{align*}
From these equalities, we obtain 
\begin{equation}\begin{cases}
\gamma_{j}=0, j=13,14,\ldots,24, 26, 27, 29\\
\gamma_{\{25, 45\}} =   
\gamma_{\{28, 45\}} =    
\gamma_{\{30, 44, 45\}} = 0.     
\end{cases} \tag {\ref{3.8}.2} 
\end{equation}
Substituting (\ref{3.8}.2) into the relation (\ref{3.8}.1), we get
\begin{equation}\sum_{25\leqslant j\leqslant 45, j\ne 26,27,29} \gamma_{j}[a_{4,j}] = 0.\tag {\ref{3.8}.3}
\end{equation}
Apply the homomorphisms $g_i$, for $i=1,2,3,4,$ to the relation (\ref{3.8}.3), we obtain
\begin{align*} 
&\gamma_{31}[7,7,15] +  \gamma_{36}[7,15,7] +   \gamma_{40}[15,7,7] =0,\\  
&\gamma_{32}[7,7,15] +  \gamma_{\{25, 28, 35, 45\}}[7,15,7] +  \gamma_{41}[15,7,7] = 0,\\  
&\gamma_{33}[7,7,15] + \gamma_{\{25, 30, 37, 45\}}[7,15,7] +  \gamma_{\{39, 44\}}[15,7,7] = 0,\\  
&\gamma_{34}[7,7,15] + \gamma_{\{28, 30, 38, 43, 45\}}[7,15,7] +  \gamma_{\{42, 43, 44\}}[15,7,7] = 0.
\end{align*}
From these equalities, we get 
\begin{equation} \begin{cases}  \gamma_j = 0, \ j = 31, 32, 33, 34, 36, 40, 41,\\
\gamma_{\{25, 28, 35, 45\}} =  
\gamma_{\{25, 30, 37, 45\}} = 0,\\   
\gamma_{\{39, 44\}} =   
\gamma_{\{28, 30, 38, 43, 45\}} =  
\gamma_{\{42, 43, 44\}} = 0. 
\end{cases} \tag {\ref{3.8}.4}
\end{equation}
Combining the relations (\ref{3.8}.3) and (\ref{3.8}.4) gives
\begin{equation} \gamma_{25}[a_{4,25}]+\gamma_{28}[a_{4,28}] +\gamma_{30}[a_{4,30}]  +\sum_{35\leqslant j\leqslant 45, j\ne 36,40,41} \gamma_{j}[a_{4,j}] = 0.\tag {\ref{3.8}.5}
\end{equation}
 Under the homomorphism $h$, the image of the linear relation (\ref{3.8}.5) is
\begin{align*} 
\gamma_{\{25, 28, 30, 35, 37, 38, 45\}}[7,7,15] +  \gamma_{\{39, 42, 44\}}[7,15,7] +   \gamma_{43}[15,7,7] = 0.  
\end{align*}
From this, it implies
\begin{equation}  
\gamma_{\{25, 28, 30, 35, 37, 38, 45\}}=  \gamma_{\{39, 42, 44\}}=   \gamma_{43}= 0.
 \tag {\ref{3.8}.6}
\end{equation}
Combining (\ref{3.8}.2), (\ref{3.8}.4) and (\ref{3.8}.6), we obtain $\gamma_j =0$ for $13 \leqslant j \leqslant 45$. The proposition follows.  
\end{proof}

\begin{prop}\label{3.9} For any $s \geqslant 5$, the  elements $[a_{s,j}], 13 \leqslant j \leqslant 45$, are linearly independent in $(\mathbb F_2\underset {\mathcal A}\otimes R_4)_{2^{s+1}-3}$.
\end{prop}

\begin{proof}  Suppose there is a linear relation of the above elements
\begin{equation}\sum_{j=13}^{45} \gamma_{j}[a_{s,j}] = 0,\tag {\ref{3.9}.1}
\end{equation}
where $\gamma_j \in \mathbb F_2, 13 \leqslant j \leqslant 45$. 

Applying the homomorphisms $f_i, i=1,2,\ldots,6,$ to the relation (\ref{3.9}.1), we get
\begin{align*} 
&\gamma_{13}v_{s,1} + \gamma_{14}v_{s,2} +   \gamma_{27}v_{s,3} = 0,\ \  
\gamma_{15}v_{s,1} +  \gamma_{17}v_{s,2} +  \gamma_{26}v_{s,3} = 0,\\  
&\gamma_{16}v_{s,1} +  \gamma_{18}v_{s,2} +  \gamma_{25}v_{s,3} = 0,\ \  
\gamma_{19}v_{s,1} +  \gamma_{29}v_{s,2} +  \gamma_{22}v_{s,3} = 0,\\  
&\gamma_{20}v_{s,1} +  \gamma_{28}v_{s,2} +  \gamma_{23}v_{s,3} = 0,\ \  
\gamma_{30}v_{s,1} +  \gamma_{21}v_{s,2} +  \gamma_{24}v_{s,3} = 0.  \end{align*}
From this, it implies $\gamma_j=0,\ j=13, 14, \ldots , 30$. So, the relation (\ref{3.9}.1) becomes
\begin{equation}\sum_{j=31}^{45} \gamma_{j}[a_{s,j}] = 0.\tag {\ref{3.9}.2}
\end{equation}
Applying the homomorphisms $g_i, i=1,2,3,4,$ to (\ref{3.9}.2), we obtain
\begin{align*} 
&\gamma_{31}v_{s,1} +  \gamma_{36}v_{s,2} +   \gamma_{40}v_{s,3} = 0,\ \  \gamma_{32}v_{s,1} +  \gamma_{35}v_{s,2} +  \gamma_{41}v_{s,3} = 0,\\  &\gamma_{33}v_{s,1} +  \gamma_{37}v_{s,2} +  \gamma_{39}v_{s,3} = 0,\ \  \gamma_{34}v_{s,1} +  \gamma_{38}v_{s,2} +  \gamma_{42}v_{s,3} = 0.  \end{align*}
So, we obtain $\gamma_j=0,\ j=31, 32, \ldots , 42$. Hence, the relation (\ref{3.9}.2) becomes
 \begin{equation} 
\gamma_{43}[a_{s,43}] + \gamma_{44}[a_{s,44}] + \gamma_{45}[a_{s,45}] = 0.
\tag {\ref{3.9}.3}
\end{equation}
Now, apply the homomorphisms $h$ to (\ref{3.9}.3) and we get
 $$\gamma_{44}v_{s,1} + \gamma_{45}v_{s,2} +  \gamma_{43}v_{s,3} = 0.$$
From this, it implies $\gamma_{43}= \gamma_{44} =\gamma_{45} = 0.$ The proposition follows.
\end{proof}

\section{ The indecomposables of $P_4$ in degree $2^{s+1}-2$}\label{4}

According to \cite{ka},  $\dim(\mathbb F_2 \underset {\mathcal A} \otimes P_3)_2 = 3$ with a basis consisting of all the following classes:
$w_{1,1} = [1,1,0],\ w_{1,2} = [1,0,1],\ w_{1,3} = [0,1,1].$

For $s \geqslant 2$, $(\mathbb F_2 \underset {\mathcal A} \otimes P_3)_{2^{s+1}-2} $ is an $\mathbb F_2$-vector space with a basis consisting of all the following classes:

\smallskip
For $s \geqslant 2$, 

\smallskip
\centerline{\begin{tabular}{ll}
$w_{s,1} = [1,2^s-2,2^s-1],$& $w_{s,2} = [1,2^s-1,2^s-2]$,\cr  
$w_{s,3} = [2^s-1,1,2^s-2]$,&$w_{s,4} = [0,2^s-1,2^s-1]$,\cr  
$w_{s,5} = [2^s-1,0,2^s-1]$,&$w_{s,6} = [2^s-1,2^s-1,0].$\cr
\end{tabular}}
  
\smallskip
For $s\geqslant 3$,\   $w_{s,7} = [3,2^s-3,2^s-2]$. 

\smallskip
Hence, we have

\begin{prop}\label{ 4.7}\ 

1. $(\mathbb F_2 \underset {\mathcal A} \otimes P_4)_2$  is  an $\mathbb F_2$-vector space of dimension 6 with a basis consisting of all the  classes represented by the following monomials:
\begin{align*} 
&b_{1,1} = ( 0,0,1,1), \quad b_{1,2} =  ( 0,1,0,1), \quad  b_{1,3} = (0,1,1,0), \\
&b_{1,4} =(1,0,0,1), \quad b_{1,5} = (1,0,1,0), \quad b_{1,6} = (1,1,0,0).
\end{align*}

2. For $ s \geqslant 2$, $(\mathbb F_2 \underset {\mathcal A} \otimes Q_4)_{2^{s+1}-2}$  is  an $\mathbb F_2$-vector space with a basis consisting of all the classes represented by the following monomials:

\smallskip
For $s\geqslant 2$.

\smallskip
\centerline{\begin{tabular}{lll}
$b_{s,1} = (0,1,2^{s}-2,2^{s}-1),$ &$b_{s,2} = (0,1,2^{s}-1,2^{s}-2),$\cr 
$b_{s,3} = (0,2^{s}-1,1,2^{s}-2),$ &$b_{s,4} = (1,0,2^{s}-2,2^{s}-1),$\cr 
$b_{s,5} = (1,0,2^{s}-1,2^{s}-2),$ &$b_{s,6} = (1,2^{s}-2,0,2^{s}-1),$\cr 
$b_{s,7} = (1,2^{s}-2,2^{s}-1,0),$ &$b_{s,8} = (1,2^{s}-1,0,2^{s}-2),$\cr 
$b_{s,9} = (1,2^{s}-1,2^{s}-2,0),$ &$b_{s,10} = (2^s-1,0,1,2^{s}-2),$\cr 
$b_{s,11} = (2^s-1,1,0,2^{s}-2),$ &$b_{s,12} = (2^s-1,1,2^{s}-2,0),$\cr 
$b_{s,13} = (0,0,2^{s}-1,2^{s}-1),$ &$b_{s,14} = (0,2^{s}-1,0,2^{s}-1),$\cr 
$b_{s,15} = (0,2^{s}-1,2^{s}-1,0),$ &$b_{s,16} = (2^s-1,0,0,2^{s}-1),$\cr 
$b_{s,17} = (2^s-1,0,2^{s}-1,0),$ &$b_{s,18} = (2^s-1,2^{s}-1,0,0),$\cr 
\end{tabular}}

\smallskip
For $s\geqslant 3$.

\smallskip
\centerline{\begin{tabular}{lll}
$b_{s,19} = (0,3,2^{s}-3,2^{s}-2),$ &$b_{s,20} = (3,0,2^{s}-3,2^{s}-2),$\cr
$b_{s,21} = (3,2^{s}-3,0,2^{s}-2),$ &$b_{s,22} = (3,2^{s}-3,2^{s}-2,0).$\cr
\end{tabular}}
\end{prop}

By Proposition \ref{2.7}, we need only to compute $(\mathbb F_2 \underset {\mathcal A} \otimes R_4)_{2^{s+1}-2}$. 

The main result of this section is

\begin{thm}\label{dlc4}\

1. $(\mathbb F_2 \underset {\mathcal A} \otimes R_4)_6$  is  an $\mathbb F_2$-vector space of dimension 6 with a basis consisting of all the classes represented by the following monomials:

\medskip
\centerline{\begin{tabular}{lll}
$b_{2,19} = (1,1,2,2),$& $b_{2,20} = (1,2,1,2),$& $b_{2,21} = (1,1,1,3),$\cr
$b_{2,22} = (1,1,3,1),$& $b_{2,23} = (1,3,1,1),$& $b_{2,24} = (3,1,1,1).$\cr  \end{tabular}}

 \medskip
2. For any $s \geqslant 3$, $(\mathbb F_2 \underset {\mathcal A} \otimes R_4)_{2^{s+1}-2}$  is  an $\mathbb F_2$-vector space of dimension $\mu_1(s-1)+13$ with a basis consisting of all the classes represented by the following monomials:

{\rm i.} $\phi (a_{s-1,j})$, where the monomials $a_{s-1,j}; s\geqslant 3, 1\leqslant j \leqslant \mu_1(s-1),$ are determined as in Section \ref{3} and the homomorphism $\phi$ is defined in Definition \ref{2.8}. Recall that $\mu_1(2)=15, \mu_1(3)=35, \mu_1(s-1)=45, s\geqslant 5.$

{\rm ii.} The monomials $b_{s,j}, 23 \leqslant j\leqslant 35$, are determined as follows:

For $s\geqslant 3$, 

\medskip
\centerline{\begin{tabular}{lll}
$b_{s,23} = (1,1,2^{s}-2,2^{s}-2),$ &$b_{s,24} = (1,2^{s}-2,1,2^{s}-2),$\cr
$b_{s,25} = (1,2,2^{s}-4,2^{s}-1),$ &$b_{s,26} = (1,2,2^{s}-1,2^{s}-4),$\cr 
$b_{s,27} = (1,2^{s}-1,2,2^{s}-4),$ &$b_{s,28} = (2^s-1,1,2,2^{s}-4),$\cr 
$b_{s,29} = (1,2,2^{s}-3,2^{s}-2),$ &$b_{s,30} = (1,3,2^{s}-4,2^{s}-2),$\cr 
$b_{s,31} = (1,3,2^{s}-2,2^{s}-4),$ &$b_{s,32} = (3,1,2^{s}-4,2^{s}-2),$\cr 
$b_{s,33} = (3,1,2^{s}-2,2^{s}-4),$ &$b_{s,34} = (3,2^{s}-3,2,2^{s}-4).$\cr 
\end{tabular}}

\medskip
For $s = 3$, \
$  b_{3,35} = (3,3,4,4).$

For $s \geqslant 4$, \
$ b_{s,35} = (3,5,2^{s}-6,2^{s}-4).$
\end{thm}

Theorem \ref{dlc4} is proved by combining the following propositions.

\begin{prop}\label{mdc4.1}\

1. The $\mathbb{F}_2$-vector space $(\mathbb F_2 \underset {\mathcal A} \otimes R_4)_6$ is generated by the six elements $[b_{2,j}], 19 \leqslant j \leqslant 24$.

2. For any $s \geqslant 3$, the $\mathbb F_2$-vector space $(\mathbb F_2\underset {\mathcal A}\otimes R_4)_{2^{s+1}-2}$ is generated by the $\mu_1(s-1)+13$ elements listed in Theorem \ref{dlc4}.
\end{prop}

The proof of this proposition is based on some lemmas.

\begin{lem}[Kameko \cite {ka}]\label{5.1} Let $M$ be an $2\times 3$-matrix and $x$ the monomial corresponding to $M$. If $\tau_1(x) < \tau_2(x)$ then $M$ is strictly inadmissible.
\end{lem}

From the proof of Proposition \ref{mdc3.1}, Lemma \ref{2.5} and Lemma \ref{5.1}, we easily obtain

\begin{lem}\label{4.1a} Let $x$ be a monomial of degree $2^s-3$ in $P_4$ with $s \geqslant 3$. If  $x$ is inadmissible then there is a strictly inadmissible matrix $\Delta$ such that $\Delta \triangleright x$.
\end{lem}
\begin{lem}\label{4.1} Let $x$ be an admissible monomial of degree $2^{s+1}-2$ in $P_4$. We have

1. If $s=1$ then $\tau(x) = (2;0)$.

2. If $s \geqslant 2$ then $\tau (x)$ is one of the following sequences
$$(\underset{\text{$s$ times}}{\underbrace{2;2;\ldots ; 2}}),\ (4;\underset{\text{$s-2$ times}}{\underbrace{3;3;\ldots ; 3}};1).$$
\end{lem}

\begin{proof} It is clear that the lemma holds for $s=1,2.$ Suppose $s\geqslant 3$. Observe  that $z=(2^s-1,2^s-1,0,0)$ is the minimal spike of degree $2^{s+1}-2$ in $P_4$ and $ \tau (z) = (\underset{\text{$s$ times}} {\underbrace {2;2;\ldots ; 2}})$. Since $2^{s+1}-2$ is even, using Theorem \ref{2.10}, we obtain $\tau_1(x)=2$ or $\tau_1(x)=4$. Suppose $\tau_1(x) = 2$. Then, using Theorem \ref{2.12} and Proposition \ref{2.6}, we get either
$\tau(x)=(\underset{\text{$s$ times}}{\underbrace{2;2;\ldots ; 2}})\ \text{or } \tau(x) = (\underset{\text{$a$ times}}{\underbrace{2;2;\ldots ; 2}};3;\tau_{a+2}(x);\ldots),\ \text{for some $a\geqslant 1$}.$

If $\tau(x) = (\underset{\text{$a$ times}}{\underbrace{2;2;\ldots ; 2}};3;\tau_{a+2}(x);\ldots)$ then
$$ n=\deg x = 2^a + 2^{a+1}\big(2 +\sum_{i\geqslant a+2}2^{i-a-2}\tau_i(x)\big)-2.$$
From this, it implies $\alpha(n+2)\geqslant 2$ and we have a contradiction.

Suppose that $\tau_1(x) = 4$. Then $x = \phi(x')$ with $x'$ a monomial of degree $2^s-3$. Since $x$ is admissible, using Lemma \ref{4.1a} and Theorem \ref{2.4}, we see that $x'$ is also admissible. By Lemma \ref{3.1}, we have
$$\tau(x')=(\tau_2(x);\tau_3(x);\ldots ; \tau_i(x); \ldots) = (\underset{\text{$s-2$ times}}{\underbrace{3;3;\ldots ; 3}};1).$$
The lemma is proved.
\end{proof}  

\begin{lem}\label{4.2} The following matrices are strictly inadmissible
$$\Delta_{17} = \begin{pmatrix} 1&0&0&1\\ 0&1&1&0\end{pmatrix}, \quad \Delta_{18} = \begin{pmatrix} 0&1&1&0\\ 1&0&0&1\end{pmatrix},$$
$$\Delta_{19} = \begin{pmatrix} 0&1&0&1\\ 1&0&1&0\end{pmatrix}, \quad 
\Delta_{20} = \begin{pmatrix} 0&0&1&1\\ 1&1&0&0\end{pmatrix}.$$ 
\end{lem}

\begin{proof} By a direct calculation, we have
\begin{align*} (1,2,2,1) &= Sq^1(2,1,1,1) + Sq^2(1,1,1,1) + (1,2,1,2)+(1,1,2,2),\\
(2,1,1,2) &= Sq^1(1,1,1,2) + (1,2,1,2) + (1,1,2,2),\\
(2,1,2,1) &= Sq^1(1,1,2,1) + (1,2,2,1) + (1,1,2,2),\\
(2,2,1,1) &= Sq^1(1,2,1,1) + (1,2,2,1) +(1,2,1,2).
\end{align*}
The lemma is proved.  
\end{proof}

\begin{lem}\label{4.3}The following matrices are strictly inadmissible
$$\Delta_{21} = \begin{pmatrix} 0&0&1&1\\ 0&1&0&1\end{pmatrix}, \ \Delta_{22} = \begin{pmatrix} 0&0&1&1\\ 0&1&1&0\end{pmatrix}, \ 
\Delta_{23} = \begin{pmatrix} 0&1&0&1\\ 0&1&1&0\end{pmatrix},$$
$$\Delta_{24} = \begin{pmatrix} 0&0&1&1\\ 1&0&0&1\end{pmatrix}, \ \Delta_{25} = \begin{pmatrix} 0&0&1&1\\ 1&0&1&0\end{pmatrix}, \
\Delta_{26} = \begin{pmatrix} 0&1&0&1\\ 1&0&0&1\end{pmatrix},$$
$$\Delta_{27} = \begin{pmatrix} 0&1&1&0\\ 1&0&1&0\end{pmatrix}, \ \Delta_{28} = \begin{pmatrix} 0&1&0&1\\ 1&1&0&0\end{pmatrix},\ 
\Delta_{29} = \begin{pmatrix} 0&1&1&0\\ 1&1&0&0\end{pmatrix},$$ 
$$\Delta_{30} = \begin{pmatrix} 1&0&0&1\\ 1&0&1&0\end{pmatrix}, \ \Delta_{31} = \begin{pmatrix} 1&0&0&1\\ 1&1&0&0\end{pmatrix},\  
\Delta_{32} = \begin{pmatrix} 1&0&1&0\\ 1&1&0&0\end{pmatrix}.
$$
\end{lem}

\begin{proof} See  \cite{ka}.
\end{proof}

\begin{lem}\label{4.4} The following matrices are strictly inadmissible
$$\Delta_{33} = \begin{pmatrix} 1&0&1&0\\ 0&1&1&0\\ 0&1&0&1\end{pmatrix}, \ \Delta_{34} =  \begin{pmatrix} 1&0&1&0\\ 1&0&0&1\\ 0&1&0&1\end{pmatrix}, \ \Delta_{35} =  \begin{pmatrix} 1&0&1&0\\ 1&0&1&0\\ 0&1&0&1\end{pmatrix} .$$
\end{lem}

\begin{proof} A direct calculation shows
\begin{align*} 
&(1,6,3,4)= Sq^1\big((1,5,3,4) + (1,3,5,4) +(4,3,3,3)+ (1,5,4,3)\\ 
&\quad + (1,3,4,5) + (1,4,5,3)+(1,4,3,5)\big) +Sq^2\big((2,3,3,4)\\ 
&\quad+ (2,3,4,3) + (1,6,2,3) +(2,4,3,3) \big) +(1,3,6,4)+(1,6,2,5)\\
&\quad + (1,3,4,6)  + (1,4,6,3) + (1,4,3,6)\quad  \text{mod }\mathcal L_4(2;2;2),\\
&(3,4,1,6)= Sq^1\big((3,4,1,5) +(3,3,4,3)\big)+Sq^2\big((5,2,2,3)\\
&\quad + (2,3,4,3)\big) +Sq^4(3,2,2,3) + (3,2,4,5)  + (3,3,4,4)\\
&\quad + (2,5,4,3) + (2,3,4,5)\quad  \text{mod }\mathcal L_4(2;2;2),\\
&(3,4,3,4)=Sq^1\big((3,3,3,4) +(1,4,4,4)\big) + Sq^2(2,3,3,4) \\
&\quad + (2,5,3,4) + (2,3,5,4) + (3,3,4,4).
\end{align*}
The lemma follows.  
\end{proof}

\begin{lem}\label{4.5} The following matrices are strictly inadmissible
 $$\Delta_{36} = \begin{pmatrix} 0&1&0&1\\ 0&1&0&1\\ 0&0&1&1\end{pmatrix}, \ \Delta_{37} =  \begin{pmatrix} 0&1&1&0\\ 0&1&1&0\\ 0&0&1&1\end{pmatrix}, \ \Delta_{38} =  \begin{pmatrix} 0&1&1&0\\ 0&1&1&0\\ 0&1&0&1\end{pmatrix},$$ $$\Delta_{39}= \begin{pmatrix}1&0&0&1\\ 1&0&0&1\\ 0&0&1&1\end{pmatrix},\ \Delta_{40} = \begin{pmatrix} 1&0&1&0\\ 1&0&1&0\\ 0&0&1&1\end{pmatrix}, \ \Delta_{41} = \begin{pmatrix} 1&0&0&1\\ 1&0&0&1\\ 0&1&0&1\end{pmatrix},$$ $$\Delta_{42}=\begin{pmatrix} 1&0&1&0\\ 1&0&1&0\\ 0&1&1&0\end{pmatrix}, \ \Delta_{43} = \begin{pmatrix} 1&1&0&0\\ 1&1&0&0\\ 0&1&0&1\end{pmatrix}, \ \Delta_{44}=\begin{pmatrix} 1&1&0&0\\ 1&1&0&0\\ 0&1&1&0\end{pmatrix}, $$ $$\Delta_{45} = \begin{pmatrix} 1&0&1&0\\ 1&0&1&0\\ 1&0&0&1\end{pmatrix}, \ \Delta_{46} =  \begin{pmatrix} 1&1&0&0\\ 1&1&0&0\\ 1&0&0&1\end{pmatrix}, \ \Delta_{47} = \begin{pmatrix} 1&1&0&0\\ 1&1&0&0\\ 1&0&1&0\end{pmatrix}. $$
\end{lem}

\begin{proof} See  \cite{ka}.
\end{proof}

\begin{lem}\label{4.6} The following matrices are strictly inadmissible
$$\Delta_{48} =  \begin{pmatrix} 1&1&0&0\\ 0&1&1&0\\ 0&1&0&1\\ 0&0&1&1\end{pmatrix}, \ \Delta_{49} =   \begin{pmatrix} 1&1&0&0\\ 1&0&1&0\\ 1&0&0&1\\ 0&0&1&1\end{pmatrix}, \ \Delta_{50} =   \begin{pmatrix} 1&1&0&0\\ 1&1&0&0\\ 0&0&1&1\\ 0&0&1&1\end{pmatrix}, $$ 
 $$\Delta_{51} =  \begin{pmatrix} 1&1&0&0\\ 1&0&0&1\\ 0&1&0&1\\ 0&0&1&1\end{pmatrix}, \ \Delta_{52} = \begin{pmatrix} 1&1&0&0\\ 1&0&1&0\\ 0&1&1&0\\ 0&0&1&1\end{pmatrix}, \ \Delta_{53} = \begin{pmatrix} 1&1&0&0\\ 1&1&0&0\\ 1&1&0&0\\ 0&0&1&1\end{pmatrix} .$$
\end{lem}

\begin{proof} By a direct calculation, we have
\begin{align*} 
&(1,7,10,12) = Sq^1\big((1,7,9,12)+(1,9,7,12)+(4,7,7,11) + (1,8,7,13)\\
&\quad+(1,7,8,13)\big)+Sq^2\big(2,7,7,12) +(2,8,7,11)+(2,7,8,11)\big)\\
&\quad+ Sq^4\big((1,4,7,14) + (1,6,7,12)\big) +(1,6,11,12)+(1,4,11,14)\\
&\quad + (1,7,8,14)\quad  \text{mod }\mathcal L_4(2;2;2;2),\\
&(7,1,10,12) = \overline{\varphi}_1(1,7,10,12),\\
&(3,3,12,12)= Sq^1\big((3,3,11,12) + (3,8,7,11)\big)\\
&\quad  + Sq^2\big((2,3,11,12)+(2,8,7,11)\big) + Sq^4(3,4,7,12) +(2,5,11,12)\\  
&\quad + (2,3,13,12) + (2,8,7,13)\quad  \text{mod }\mathcal L_4(2;2;2;2),\\
&(3,5,8,14) = Sq^1\big((3,3,9,14) + (3,3,5,18)\big)\\
&\quad + Sq^2\big((2,3,9,14) + (5,3,6,14) \big) +Sq^4(3,3,6,14)\\
&\quad + (2,5,9,14)+(3,4,9,14)\quad  \text{mod }\mathcal L_4(2;2;2;2),\\
&(3,5,14,8) = \overline{\varphi}_3(3,5,14,8),\\
&(7,7,8,8)= Sq^1(7,7,7,8)+Sq^2(7,6,7,8) + Sq^4\big((4,7,7,8) \\
&\quad+ (5,6,7,8)\big)+ (7,6,9,8) +(4,11,7,8)+(4,7,11,8)\\
&\quad +(5,10,7,8) + (5,6,11,8)\quad \text{mod }\mathcal L_4(2;2;2;2).
\end{align*}
Here $\overline{\varphi}_1$ and  $\overline{\varphi}_3$ are defined in Section \ref{2}. 

The lemma is proved.   
\end{proof}

\begin{proof}[Proof of Proposition \ref{mdc4.1}]  Let $x$ is an admissible monomial of degree $2^{s+1}-2$ in $P_4$ for $s \geqslant 2$.  Then either $\tau_1(x) = 4$ or $\tau_1(x) = 2$. If $\tau_1(x) = 4$ then using Lemma \ref{4.1} and Lemma \ref{4.1a}, we get $x = \phi(x')$ with $x'$ an admissible monomial of degree $2^{s}-3$. By Proposition \ref{3.5} and Theorem \ref{dlc3}, $x' = a_{s-1,j}$, for some $1 \leqslant j \leqslant \mu_1(s-1)$.

Suppose $\tau_1(x) = 2$. By Lemma \ref{4.1}, $\tau_j(x) = 2, j = 1,2,\ldots , s,$ and $\tau_j(x) = 0$ for $j > s$. Then $x = b_{1,i}y^2$ with a monomial $y$ of degree $2^s-2$, $\tau_j(y) = \tau_{j+1}(x) = 2$ for $j = 1, 2, \ldots , s-1,$ and  $1 \leqslant i \leqslant 6$. Now prove the proposition by showing that if $x \ne b_{s,j},$ for all $j$ then there is a strictly inadmissible matrix $\Delta$ such that $\Delta \triangleright x$. We prove this by induction on $s$.

For $s=2$, a direct computation shows

\smallskip
\centerline{\begin{tabular}{llll}
$b_{2,1} = b_{1,2}b_{1,1}^2$,& $b_{2,2} = b_{1,3}b_{1,1}^2$,& $b_{2,3} = b_{1,3}b_{1,2}^2$,& $b_{2,4} = b_{1,4}b_{1,1}^2$,\cr   
$b_{2,5} = b_{1,5}b_{1,1}^2$,& $b_{2,6} = b_{1,4}b_{1,2}^2$,& $b_{2,7} = b_{1,5}b_{1,3}^2$,& $b_{2,8} = b_{1,6}b_{1,2}^2$,\cr   
$b_{2,9} = b_{1,6}b_{1,3}^2$,& $b_{2,10} = b_{1,5}b_{1,4}^2$,& $b_{2,11} = b_{1,6}b_{1,4}^2$,& $b_{2,12} = b_{1,6}b_{1,5}^2$,\cr   
$b_{2,13} = b_{1,1}b_{1,1}^2$,& $b_{2,14} = b_{1,2}b_{1,2}^2$,& $b_{2,15} = b_{1,3}b_{1,3}^2$,& $b_{2,16} = b_{1,4}b_{1,4}^2$,\cr   $b_{2,17} = b_{1,5}b_{1,5}^2$,& $b_{2,18} = b_{1,6}b_{1,6}^2$,& $b_{2,19} = b_{1,6}b_{1,1}^2$,& $b_{2,20} = b_{1,5}b_{1,2}^2$,\cr
$b_{2,21} = \phi(a_{1,1}),$& $b_{2,22} = \phi(a_{1,2}),$& $b_{2,23} = \phi(a_{1,3}),$& $b_{2,24} = \phi(a_{1,4}).$\cr 
\end{tabular}}

\smallskip 
We have $\Delta_{17} \triangleright b_{1,4}b_{1,3}^2$;\ $\Delta_{18} \triangleright b_{1,3}b_{1,4}^2$;\ $\Delta_{19} \triangleright b_{1,2}b_{1,5}^2$;\ $\Delta_{20} \triangleright b_{1,1}b_{1,6}^2$;\ $\Delta_{21} \triangleright b_{1,1}b_{1,2}^2$;\ $\Delta_{22} \triangleright b_{1,1}b_{1,3}^2$;\ $\Delta_{23} \triangleright b_{1,2}b_{1,3}^2$;\ $\Delta_{24} \triangleright b_{1,1}b_{1,4}^2$;\ $\Delta_{25} \triangleright b_{1,1}b_{1,5}^2$;\ $\Delta_{26} \triangleright b_{1,2}b_{1,4}^2$;\ $\Delta_{27} \triangleright b_{1,3}b_{1,5}^2$;\ $\Delta_{28} \triangleright b_{1,2}b_{1,6}^2$;\ $\Delta_{29} \triangleright b_{1,3}b_{1,6}^2$;\ $\Delta_{30} \triangleright b_{1,4}b_{1,5}^2$;\ $\Delta_{31} \triangleright b_{1,4}b_{1,6}^2$;\ $\Delta_{32} \triangleright b_{1,5}b_{1,6}^2$. From these equalities, Lemmas \ref{4.2}, \ref{4.3} and Theorem \ref{2.4}, we see that the case $s=2$ is true. 

\smallskip
For $s=3$, by a direct computation, we obtain

\smallskip
\centerline{\begin{tabular}{llll}
$b_{3,1} = b_{1,2}b_{2,13}^2$,& $b_{3,2} = b_{1,3}b_{2,13}^2$,& $b_{3,3} = b_{1,3}b_{2,14}^2$,& $b_{3,4} = b_{1,4}b_{2,13}^2$,\cr   
$b_{3,5} = b_{1,5}b_{2,13}^2$,& $b_{3,6} = b_{1,4}b_{2,14}^2$,& $b_{3,7} = b_{1,5}b_{2,15}^2$,& $b_{3,8} = b_{1,6}b_{2,14}^2$,\cr   
$b_{3,9} = b_{1,6}b_{2,15}^2$,& $b_{3,10} = b_{1,5}b_{2,16}^2$,& $b_{3,11} = b_{1,6}b_{2,16}^2$,& $b_{3,12} = b_{1,6}b_{2,17}^2$,\cr   
$b_{3,13} = b_{1,1}b_{2,13}^2$,& $b_{3,14} = b_{1,2}b_{2,14}^2$,& $b_{3,15} = b_{1,3}b_{2,15}^2$,& $b_{3,16} = b_{1,4}b_{2,16}^2$,\cr   
$b_{3,17} = b_{1,5}b_{2,17}^2$,& $b_{3,18} = b_{1,6}b_{2,18}^2$,& $b_{3,19} = b_{1,3}b_{2,1}^2$,& $b_{3,20} = b_{1,5}b_{2,4}^2$,\cr   
$b_{3,21} = b_{1,6}b_{2,6}^2$,& $b_{3,22} = b_{1,6}b_{2,7}^2$,& $b_{3,23} = b_{1,6}b_{2,13}^2$,& $b_{3,24} = b_{1,5}b_{2,14}^2$,\cr   
$b_{3,25} = b_{1,4}b_{2,1}^2$,& $b_{3,26} = b_{1,5}b_{2,2}^2$,& $b_{3,27} = b_{1,6}b_{2,3}^2$,& $b_{3,28} = b_{1,6}b_{2,10}^2$,\cr   
$b_{3,29} = b_{1,5}b_{2,1}^2$,& $b_{3,30} = b_{1,6}b_{2,1}^2$,& $b_{3,31} = b_{1,6}b_{2,2}^2$,& $b_{3,32} = b_{1,6}b_{2,4}^2$,\cr   
$b_{3,33} = b_{1,6}b_{2,5}^2$,& $b_{3,34} = b_{1,6}b_{2,20}^2$,& $b_{3,35} = b_{1,6}b_{2,19}^2$.&\cr  
\end{tabular}}

\smallskip
We have $\Delta_{21} \triangleright b_{1,1}b_{2,j}^2$ for $j  =$ 1, 14;
$\Delta_{17} \triangleright b_{1,4}b_{2,j}^2$;
$\Delta_{22} \triangleright b_{1,1}b_{2,j}^2$;
$\Delta_{23} \triangleright b_{1,2}b_{2,j}^2$ for $j  =$ 2, 3, 15;
$\Delta_{18} \triangleright b_{1,3}b_{2,j}^2$;
$\Delta_{24} \triangleright b_{1,1}b_{2,j}^2$;
$\Delta_{26} \triangleright b_{1,2}b_{2,j}^2$ for $j  =$ 4, 6, 16;
$\Delta_{19} \triangleright b_{1,2}b_{2,j}^2$;
$\Delta_{25} \triangleright b_{1,1}b_{2,j}^2$;
$\Delta_{27} \triangleright b_{1,3}b_{2,j}^2$;
$\Delta_{30} \triangleright b_{1,4}b_{2,j}^2$ for $j  =$ 5, 7, 10, 17, 20;
$\Delta_{20} \triangleright b_{1,1}b_{2,j}^2$;
$\Delta_{28} \triangleright b_{1,2}b_{2,j}^2$;
$\Delta_{29} \triangleright b_{1,3}b_{2,j}^2$;
$\Delta_{31} \triangleright b_{1,4}b_{2,j}^2$;
$\Delta_{32} \triangleright b_{1,5}b_{2,j}^2$ for $j  =$ 8, 9, 11, 12, 18, 19;
$\Delta_{33} \triangleright b_{1,5}b_{2,3}^2$;
$\Delta_{34} \triangleright b_{1,5}b_{2,6}^2$;
$\Delta_{35} \triangleright b_{1,5}b_{2,20}^2$;
$\Delta_{36} \triangleright b_{1,2}b_{2,1}^2$;
$\Delta_{37} \triangleright b_{1,3}b_{2,2}^2$;
$\Delta_{38} \triangleright b_{1,3}b_{2,3}^2$;
$\Delta_{39} \triangleright b_{1,4}b_{2,4}^2$;
$\Delta_{40} \triangleright b_{1,5}b_{2,5}^2$;
$\Delta_{41} \triangleright b_{1,4}b_{2,6}^2$;
$\Delta_{42} \triangleright b_{1,5}b_{2,7}^2$;
$\Delta_{43} \triangleright b_{1,6}b_{2,8}^2$;
$\Delta_{44} \triangleright b_{1,6}b_{2,9}^2$;
$\Delta_{45} \triangleright b_{1,5}b_{2,10}^2$;
$\Delta_{46} \triangleright b_{1,6}b_{2,11}^2$;
$\Delta_{47} \triangleright b_{1,6}b_{2,12}^2$. 

From the above equalities, Lemmas \ref{4.2}, \ref{4.3}, \ref{4.4}, \ref{4.5} and Theorem \ref{2.4}, we see that if $x\ne b_{3,j}, 1 \leqslant j \leqslant 35,$ then $x$ is inadmissible. The case $s=3$ holds. 

\smallskip
Suppose $s\geqslant 4$ and the proposition holds for $s$.  By the inductive hypothesis and Theorem \ref{2.4}, it suffices to consider monomials $x= b_{1,i}b_{s,j}^2$ with $1 \leqslant i \leqslant 6$ and $1 \leqslant j \leqslant 35$. We have

\smallskip
\centerline{\begin{tabular}{llll}
$b_{s+1,1} = b_{1,2}b_{s,13}^2$,& $b_{s+1,2} = b_{1,3}b_{s,13}^2$,& $b_{s+1,3} = b_{1,3}b_{s,14}^2$,& $b_{s+1,4} = b_{1,4}b_{s,13}^2$,\cr   
$b_{s+1,5} = b_{1,5}b_{s,13}^2$,& $b_{s+1,6} = b_{1,4}b_{s,14}^2$,& $b_{s+1,7} = b_{1,5}b_{s,15}^2$,& $b_{s+1,8} = b_{1,6}b_{s,14}^2$,\cr   
$b_{s+1,9} = b_{1,6}b_{s,15}^2$,& $b_{s+1,10} = b_{1,5}b_{s,16}^2$,& $b_{s+1,11} = b_{1,6}b_{s,16}^2$,& $b_{s+1,12} = b_{1,6}b_{s,17}^2$,\cr   
$b_{s+1,13} = b_{1,1}b_{s,13}^2$,& $b_{s+1,14} = b_{1,2}b_{s,14}^2$,& $b_{s+1,15} = b_{1,3}b_{s,15}^2$,& $b_{s+1,16} = b_{1,4}b_{s,16}^2$,\cr   
$b_{s+1,17} = b_{1,5}b_{s,17}^2$,& $b_{s+1,18} = b_{1,6}b_{s,18}^2$,& $b_{s+1,19} = b_{1,3}b_{s,1}^2$,& $b_{s+1,20} = b_{1,5}b_{s,4}^2$,\cr   
$b_{s+1,21} = b_{1,6}b_{s,6}^2$,& $b_{s+1,22} = b_{1,6}b_{s,7}^2$,& $b_{s+1,23} = b_{1,6}b_{s,13}^2$,& $b_{s+1,24} = b_{1,5}b_{s,14}^2$,\cr   
$b_{s+1,25} = b_{1,4}b_{s,1}^2$,& $b_{s+1,26} = b_{1,5}b_{s,2}^2$,& $b_{s+1,27} = b_{1,6}b_{s,3}^2$,& $b_{s+1,28} = b_{1,6}b_{s,10}^2$,\cr   
$b_{s+1,29} = b_{1,5}b_{s,1}^2$,& $b_{s+1,30} = b_{1,6}b_{s,1}^2$,& $b_{s+1,31} = b_{1,6}b_{s,2}^2$,& $b_{s+1,32} = b_{1,6}b_{s,4}^2$,\cr   
$b_{s+1,33} = b_{1,6}b_{s,5}^2$,& $b_{s+1,34} = b_{1,6}b_{s,24}^2$,& $b_{s+1,35} = b_{1,6}b_{s,29}^2$.& \cr   
\end{tabular}}

\smallskip
For $x \ne b_{s,j}, 1 \leqslant j \leqslant 35$, we have $\Delta_{21} \triangleright b_{1,1}b_{s,j}^2$ for $j  =$ 1, 14;
$\Delta_{17} \triangleright b_{1,4}b_{s,j}^2$,
$\Delta_{22} \triangleright b_{1,1}b_{s,j}^2$,
$\Delta_{23} \triangleright b_{1,2}b_{s,j}^2$ for $j  =$ 2, 3, 15, 19;
$\Delta_{18} \triangleright b_{1,3}b_{s,j}^2$,
$\Delta_{24} \triangleright b_{1,1}b_{s,j}^2$,
$\Delta_{26} \triangleright b_{1,2}b_{s,j}^2$ for $j  =$ 4, 6, 16, 25;
$\Delta_{19} \triangleright b_{1,2}b_{s,j}^2$,
$\Delta_{25} \triangleright b_{1,1}b_{s,j}^2$,
$\Delta_{27} \triangleright b_{1,3}b_{s,j}^2$,
$\Delta_{30} \triangleright b_{1,4}b_{s,j}^2$ for $j  =$ 5, 7, 10, 17, 20, 24, 26, 29;
$\Delta_{20} \triangleright b_{1,1}b_{s,j}^2$,
$\Delta_{28} \triangleright b_{1,2}b_{s,j}^2$,
$\Delta_{29} \triangleright b_{1,3}b_{s,j}^2$,
$\Delta_{31} \triangleright b_{1,4}b_{s,j}^2$,
$\Delta_{32} \triangleright b_{1,5}b_{s,j}^2$ for $j  =$ 8, 9, 11, 12, 18, 21, 22, 23, 27, 28, 30, 31, 32, 33, 34, 35;
$\Delta_{33} \triangleright b_{1,5}b_{s,j}^2$ for $j  =$ 3, 19;
$\Delta_{34} \triangleright b_{1,5}b_{s,j}^2$ for $j  =$ 6, 25;
$\Delta_{35} \triangleright b_{1,5}b_{s,j}^2$ for $j  =$ 24, 29;
$\Delta_{36} \triangleright b_{1,2}b_{s,1}^2$;
$\Delta_{37} \triangleright b_{1,3}b_{s,2}^2$;
$\Delta_{38} \triangleright b_{1,3}b_{s,j}^2$ for $j  =$ 3, 19;
$\Delta_{39} \triangleright b_{1,4}b_{s,4}^2$;
$\Delta_{40} \triangleright b_{1,5}b_{s,5}^2$;
$\Delta_{41} \triangleright b_{1,4}b_{s,j}^2$ for $j  =$ 6, 25;
$\Delta_{42} \triangleright b_{1,5}b_{s,j}^2$ for $j  =$ 7, 26;
$\Delta_{43} \triangleright b_{1,6}b_{s,j}^2$ for $j  =$ 8, 30;
$\Delta_{44} \triangleright b_{1,6}b_{s,j}^2$ for $j  =$ 9, 27, 31;
$\Delta_{45} \triangleright b_{1,5}b_{s,j}^2$ for $j  =$ 10, 20;
$\Delta_{46} \triangleright b_{1,6}b_{s,j}^2$ for $j  =$ 11, 21, 32;
$\Delta_{47} \triangleright b_{1,6}b_{s,j}^2$ for $j  =$ 12, 22, 28, 33, 34;
$\Delta_{48} \triangleright b_{1,6}b_{s,19}^2$;
$\Delta_{49} \triangleright b_{1,6}b_{s,20}^2$;
$\Delta_{50} \triangleright b_{1,6}b_{s,23}^2$;
$\Delta_{51} \triangleright b_{1,6}b_{s,25}^2$;
$\Delta_{52} \triangleright b_{1,6}b_{s,26}^2$;
$\Delta_{53} \triangleright b_{1,6}b_{3,35}^2$, for $s=3$ and $\Delta_{47} \triangleright b_{1,6}b_{s,35}^2$, for $s \geqslant 4$.

From the above equalities, Lemmas \ref{4.2}, \ref{4.3}, \ref{4.4}, \ref{4.5}, \ref{4.6},  we see that the case $s+1$ is true. The proof is completed.
\end{proof}

\smallskip
Now, we prove that the elements listed in Theorem \ref{dlc4} are linearly independent.

\begin{prop}\label{4.8} The  elements $[b_{2,i}], 19\leqslant i \leqslant 24,$ are linearly independent.
\end{prop}

\begin{proof} Suppose there is a linear relation
$$\sum_{19 \leqslant i \leqslant 24}\gamma_i[b_{2,i}] = 0,$$
with $\gamma_i \in \mathbb F_2.$ 

Since the monomials $b_{2,i}, i=21,22,23,24,$ are the spikes, we get $\gamma_{21}=\gamma_{22}=\gamma_{23}=\gamma_{24}=0$. Then, applying the homomorphisms $f_1, f_2$ to the above linear relation, we obtain
$$\gamma_{20}[3,1,2] =0, \ \ \gamma_{19}[3,1,2] = 0.$$
So, $\gamma_{19}=\gamma_{20} =0$. The proposition follows.
\end{proof}  

\begin{prop}\label{4.9} The elements $[b_{3,i}]$ for $23\leqslant i \leqslant 35$, and  $[\phi(a_{2,j})]$ for $ 1 \leqslant j \leqslant 15,$ are linearly independent.
\end{prop}

\begin{proof} 
Suppose that there is a linear relation
\begin{equation}\sum_{23\leqslant i\leqslant 35}\gamma_i[b_{3,i}] +\sum_{1\leqslant j \leqslant 15}\eta_j[\phi(a_{2,j})] = 0, \tag {\ref{4.9}.1}
\end{equation}
with $\gamma_i, \eta_j \in \mathbb F_2$. 

Apply the squaring operation $Sq^0_*$ to (\ref{4.9}.1) and we get
$$\sum_{1\leqslant j \leqslant 15}\eta_j[a_{2,j}]=0.$$
Since $\{[a_{2,j}]; 1\leqslant j \leqslant 15\}$ is a basis of $(\mathbb F_2 \underset {\mathcal A} \otimes P_4)_{5},$ we obtain $\eta_j=0, 1\leqslant j \leqslant 15$. Hence, (\ref{4.9}.1) becomes
\begin{equation}\sum_{23\leqslant i\leqslant 35}\gamma_i[b_{3,i}] =0. \tag {\ref{4.9}.2}
\end{equation}
Now, we prove $\gamma_i=0, 23\leqslant i\leqslant 35.$

Applying the homomorphisms $f_t, t=1,2,\ldots, 5,$ to the relation (\ref{4.9}.2), we obtain
\begin{align*} 
&\gamma_{25}[1,6,7] +  \gamma_{26}[1,7,6] +   \gamma_{24}[7,1,6] +  \gamma_{29}[3,5,6] = 0,\\  
&\gamma_{25}[1,6,7] +  \gamma_{27}[1,7,6] +  \gamma_{\{23, 31, 32, 35\}}[7,1,6] +  \gamma_{30}[3,5,6] = 0,\\  
&\gamma_{26}[1,6,7] +  \gamma_{27}[1,7,6] +  \gamma_{\{23, 24, 29, 30, 33, 34, 35\}}[7,1,6] +  \gamma_{31}[3,5,6] = 0,\\  
&\gamma_{25}[1,6,7] +  \gamma_{\{23, 24, 29, 30, 33, 34, 35\}}[1,7,6] +  \gamma_{28}[7,1,6] +  \gamma_{32}[3,5,6] = 0,\\  
&\gamma_{26}[1,6,7] +  \gamma_{\{23, 31, 32, 35\}}[1,7,6] +  \gamma_{28}[7,1,6] +  \gamma_{33}[3,5,6] = 0.
\end{align*}
From the above relation, we get $\gamma_i=0, i=24,\ldots,34,$ and $\gamma_{23} = \gamma_{35}$. So, the relation (\ref{4.9}.2) becomes
$$\gamma_{23}([1,1,6,6] + [3,3,4,4])=0.
$$
Now, we prove that $[1,1,6,6] + [3,3,4,4] \ne 0$. Suppose the contrary, that  the polynomial $(1,1,6,6) + (3,3,4,4)$ is hit. Then by the unstable property of the action of $\mathcal A$ on the polynomial algebra, we have
$$(1,1,6,6) + (3,3,4,4) = Sq^1(A) + Sq^2(B) + Sq^4(C),$$
for some polynomials $A \in (R_4)_{13}, B \in (R_4)_{12}, C \in (R_4)_{10}$. 
Let $Sq^2Sq^2Sq^2$ act on the both sides of the above equality. Since $Sq^2Sq^2Sq^2Sq^1 = 0$ and $Sq^2Sq^2Sq^2Sq^2 = 0$, we get
$$Sq^2Sq^2Sq^2((1,1,6,6) + (3,3,4,4)) = Sq^2Sq^2Sq^2Sq^4(C).$$
On the other hand, by a direct computation, it is not difficult to check that 
$$Sq^2Sq^2Sq^2((1,1,6,6) + (3,3,4,4)) \ne Sq^2Sq^2Sq^2Sq^4(C),$$
for all $C \in (R_4)_{10}$. This is a contradiction. Hence, $[1,1,6,6] + [3,3,4,4] \ne 0$ and $\gamma_{23} = \gamma_{35} = 0$.
The proposition is proved.     
\end{proof}

\begin{prop}\label{4.11} For $s\geqslant 4$, the $\mu_1(s-1)+13$ elements listed in Theorem \ref{dlc4} are linearly independent.
\end{prop}

\begin{proof} 
Suppose that there is a linear relation
\begin{equation}\sum_{23\leqslant i\leqslant 35}\gamma_i[b_{s,i}] +\sum_{1\leqslant j \leqslant \mu_1(s-1)}\eta_j[\phi(a_{s-1,j})] = 0, \tag {\ref{4.11}.1}
\end{equation}
with $\gamma_i, \eta_j \in \mathbb F_2$. 

Applying the squaring operation $Sq^0_*$ to (\ref{4.11}.1), we get
$$\sum_{1\leqslant j \leqslant \mu_1(s-1)}\eta_j[a_{s-1,j}]=0.$$
Since $\{[a_{s-1,j}]; 1\leqslant j \leqslant \mu_1(s-1)\}$ is a basis of $(\mathbb F_2 \underset {\mathcal A} \otimes P_4)_{2^{s}-3},$ we obtain $\eta_j=0, 1\leqslant j \leqslant \mu_1(s-1)$. Hence, (\ref{4.11}.1) becomes
\begin{equation}\sum_{23\leqslant i\leqslant 35}\gamma_i[b_{s,i}] =0. \tag {\ref{4.11}.2}
\end{equation}

Apply the homomorphisms $f_t, t=1,2,\ldots, 6,$ to the relation (\ref{4.11}.2) and we obtain
\begin{align*} 
&\gamma_{25}w_{s,1} + \gamma_{26}w_{s,2}+ \gamma_{24}w_{s,3}+ \gamma_{29}w_{s,7} = 0,\\ 
&\gamma_{25}w_{s,1} + \gamma_{27}w_{s,2}+ \gamma_{\{23, 31,32\}}w_{s,3} + \gamma_{30}w_{s,7} = 0,\\ 
&\gamma_{26}w_{s,1} + \gamma_{27}w_{s,2}  + \gamma_{\{23, 24, 29, 30, 33, 34, 35\}}w_{s,3} + \gamma_{31}w_{s,7}=0,\\
&\gamma_{25}w_{s,1} + \gamma_{\{23, 24, 29, 30, 33, 34, 35\}}w_{s,2}+ \gamma_{28}w_{s,3} + \gamma_{32}w_{s,7} =0,\\
&\gamma_{26}w_{s,1}+ \gamma_{\{23, 31, 32\}}w_{s,2}+  \gamma_{28}w_{s,3} + \gamma_{33}w_{s,7} = 0,\\
&\gamma_{24}w_{s,1}+ \gamma_{27}w_{s,2}  + \gamma_{28}w_{s,3} + \gamma_{34}w_{s,7}= 0.
\end{align*}
From these relations, we get $\gamma_i = 0, i=23,24,\ldots,35$.

The proposition follows.  
\end{proof}

\begin{rem}\label{4.10} By a direct calculation, we can easily show that 
$$[1,1,6,6]+[3,3,4,4]$$
 is an $GL_4(\mathbb F_2)$-invariant in $(\mathbb F_2\underset{\mathcal A} \otimes P_4)_{14}$.
\end{rem}

\section{The indecomposables of $P_4$ in degree $2^{s+1}-1$}\label{5}

First of all, we recall a result in \cite{ka} on the dimension of the $\mathbb F_2$-vector space $(\mathbb F_2 \underset {\mathcal A} \otimes P_3)_{2^{s+1}-1}$.

Set $\rho_1(1)=7, \rho_1(2)=10, \rho_1(3)=13$ and $\rho_1(s) = 14$ for $s\geqslant 4$.

\medskip
According to Kameko \cite{ka},   $(\mathbb F_2 \underset {\mathcal A} \otimes P_3)_{2^{s+1}-1}$ is an $\mathbb F_2$-vector space of dimension $\rho_1(s)$ with a basis consisting of all the following classes:

\medskip
For $s \geqslant  1$,

\smallskip
\centerline{\begin{tabular}{lll}
$u_{s,1} = [0,1,2^{s+1} - 2],$& $u_{s,2} = [1,0,2^{s+1} - 2],$& $u_{s,3} = [1,2^{s+1} - 2,0],$\cr 
$u_{s,4} = [0,0,2^{s+1} - 1],$& $u_{s,5} = [0,2^{s+1} - 1,0],$& $u_{s,6} = [2^{s+1} - 1,0,0].$\cr
\end{tabular}}

\medskip
For $s=1$,\  $u_{1,7} = [1,1,1].$

For $s\geqslant 2$,

\smallskip
\centerline{\begin{tabular}{ll}
$u_{s,7} = [1,2,2^{s+1} - 4],$& $u_{s,8} = [1,2^{s} - 1,2^{s} - 1],$\cr 
$u_{s,9} = [2^{s} - 1,1,2^{s} - 1],$& $u_{s,10} = [2^{s} - 1,2^{s} - 1,1].$\cr
\end{tabular}}

\medskip
For $s \geqslant 3$,

\smallskip
\centerline{\begin{tabular}{ll}
$u_{s,11} = [3,2^{s} - 3,2^{s} - 1],$& $u_{s,12} = [3,2^{s} - 1,2^{s} - 3],$\cr 
$u_{s,13} = [2^{s} - 1,3,2^{s} - 3].$& \cr
\end{tabular}}

\medskip
For $s \geqslant 4$,\ $u_{s,14} = [7,2^{s} - 5,2^{s} - 3].$ 

\medskip
Set $\rho_2(1)=14, \rho_2(2) = 26, \rho_2(3)=38$ and $\rho_2(s) = 42$ for $s\geqslant 4$. From this result  we easily obtain

\begin{prop}\label{5.11} $(\mathbb F_2 \underset {\mathcal A} \otimes Q_4)_{2^{s+1}-1}$  is  an $\mathbb F_2$-vector space of dimension $\rho_2(s)$ with a basis consisting of all the  classes represented by the following monomials:

\smallskip
For $s\geqslant 1$,

\medskip
\centerline{\begin{tabular}{ll}
$c_{s,1} = (0,0,1,2^{s+1} - 2),$& $c_{s,2} = (0,1,0,2^{s+1} - 2),$\cr 
$c_{s,3} = (0,1,2^{s+1} - 2,0),$& $c_{s,4} = (1,0,0,2^{s+1} - 2),$\cr 
$c_{s,5} = (1,0,2^{s+1} - 2,0),$& $c_{s,6} = (1,2^{s+1} - 2,0,0),$\cr 
$c_{s,7} = (0,0,0,2^{s+1} - 1),$& $c_{s,8} = (0,0,2^{s+1} - 1,0),$\cr 
$c_{s,9} = (0,2^{s+1} - 1,0,0),$& $c_{s,10} = (2^{s+1} - 1,0,0,0).$\cr
\end{tabular}}

\medskip
For $s=1$,
$$c_{1,11} = (0,1,1,1),\ c_{1,12} = (1,0,1,1),\ c_{1,13} = (1,1,0,1),\ c_{1,14} = (1,1,1,0).$$

For $s\geqslant 2$,

\medskip
\centerline{\begin{tabular}{ll}
$c_{s,11} = (0,1,2,2^{s+1} - 4),$& $c_{s,12} = (1,0,2,2^{s+1} - 4),$\cr 
$c_{s,13} = (1,2,0,2^{s+1} - 4),$& $c_{s,14} = (1,2,2^{s+1} - 4,0),$\cr 
$c_{s,15} = (0,1,2^s - 1,2^s - 1),$& $c_{s,16} = (0,2^s - 1,1,2^s - 1),$\cr 
$c_{s,17} = (0,2^s - 1,2^s - 1,1),$& $c_{s,18} = (1,0,2^s - 1,2^s - 1),$\cr 
$c_{s,19} = (1,2^s - 1,0,2^s - 1),$& $c_{s,20} = (1,2^s - 1,2^s - 1,0),$\cr 
$c_{s,21} = (2^s - 1,0,1,2^s - 1),$& $c_{s,22} = (2^s - 1,0,2^s - 1,1),$\cr 
$c_{s,23} = (2^s - 1,1,0,2^s - 1),$& $c_{s,24} = (2^s - 1,1,2^s - 1,0),$\cr 
$c_{s,25} = (2^s - 1,2^s - 1,0,1),$& $c_{s,26} = (2^s - 1,2^s - 1,1,0).$\cr
\end{tabular}}

\medskip
For $s\geqslant 3$,

\medskip
\centerline{\begin{tabular}{ll}
$c_{s,27} = (0,3,2^s - 3,2^s - 1),$& $c_{s,28} = (0,3,2^s - 1,2^s - 3),$\cr 
$c_{s,29} = (0,2^s - 1,3,2^s - 3),$& $c_{s,30} = (3,0,2^s - 3,2^s - 1),$\cr 
$c_{s,31} = (3,0,2^s - 1,2^s - 3),$& $c_{s,32} = (3,2^s - 3,0,2^s - 1),$\cr 
$c_{s,33} = (3,2^s - 3,2^s - 1,0),$& $c_{s,34} = (3,2^s - 1,0,2^s - 3),$\cr 
$c_{s,35} = (3,2^s - 1,2^s - 3,0),$& $c_{s,36} = (2^s - 1,0,3,2^s - 3),$\cr 
$c_{s,37} = (2^s - 1,3,0,2^s - 3),$& $c_{s,38} = (2^s - 1,3,2^s - 3,0).$\cr
\end{tabular}}

\medskip
For $s\geqslant 4$,

\medskip
\centerline{\begin{tabular}{ll}
$c_{s,39} = (0,7,2^s - 5,2^s - 3),$& $c_{s,40} = (7,0,2^s - 5,2^s - 3),$\cr 
$c_{s,41} = (7,2^s - 5,0,2^s - 3),$& $c_{s,42} = (7,2^s - 5,2^s - 3,0).$\cr
\end{tabular}}
\end{prop}

Observe that $(\mathbb F_2 \underset {\mathcal A} \otimes R_4)_{3}=0$. So we need only to determine the space $(\mathbb F_2 \underset {\mathcal A} \otimes R_4)_{2^{s+1}-1}$ for $s\geqslant 2$. 

Set $\rho_3(2)=9, \rho_3(3)=37, \rho_3(4) = 47$ and $\rho_3(s) = 43$ for $s\geqslant 5$. We have

\begin{thm}\label{dlc5} 
For any $s \geqslant 2$, $(\mathbb F_2 \underset {\mathcal A} \otimes R_4)_{2^{s+1}-1}$ is an $\mathbb F_2$-vector space of dimension $\rho_3(s)$ with a basis consisting of all the classes represented by the following monomials:

\smallskip
For $s\geqslant 2$,

\medskip
\centerline{\begin{tabular}{ll}
$d_{s,1} = (1,1,2^s - 2,2^s - 1),$& $d_{s,2} = (1,1,2^s - 1,2^s - 2),$\cr 
$d_{s,3} = (1,2^s - 2,1,2^s - 1),$& $d_{s,4} = (1,2^s - 2,2^s - 1,1),$\cr 
$d_{s,5} = (1,2^s - 1,1,2^s - 2),$& $d_{s,6} = (1,2^s - 1,2^s - 2,1),$\cr 
$d_{s,7} = (2^s - 1,1,1,2^s - 2),$& $d_{s,8} = (2^s - 1,1,2^s - 2,1).$\cr
\end{tabular}}

\medskip
For $s = 2$,\ $d_{2,9} = (1,2,2,2)$.

\medskip
For $s\geqslant 3$,

\medskip
\centerline{\begin{tabular}{lll}
$d_{s,9} = (1,2,2^s - 3,2^s - 1),$& $d_{s,10} = (1,2,2^s - 1,2^s - 3),$\cr 
$d_{s,11} = (1,2^s - 1,2,2^s - 3),$& $d_{s,12} = (2^s - 1,1,2,2^s - 3),$\cr 
$d_{s,13} = (1,3,2^s - 4,2^s - 1),$& $d_{s,14} = (1,3,2^s - 1,2^s - 4),$\cr 
$d_{s,15} = (1,2^s - 1,3,2^s - 4),$& $d_{s,16} = (3,1,2^s - 4,2^s - 1),$\cr 
$d_{s,17} = (3,1,2^s - 1,2^s - 4),$& $d_{s,18} = (3,2^s - 1,1,2^s - 4),$\cr 
$d_{s,19} = (2^s - 1,1,3,2^s - 4),$& $d_{s,20} = (2^s - 1,3,1,2^s - 4),$\cr 
$d_{s,21} = (1,3,2^s - 3,2^s - 2),$& $d_{s,22} = (1,3,2^s - 2,2^s - 3),$\cr 
$d_{s,23} = (1,2^s - 2,3,2^s - 3),$& $d_{s,24} = (3,1,2^s - 3,2^s - 2),$\cr 
$d_{s,25} = (3,1,2^s - 2,2^s - 3),$& $d_{s,26} = (3,2^s - 3,1,2^s - 2),$\cr 
$d_{s,27} = (3,2^s - 3,2^s - 2,1),$& $d_{s,28} = (3,2^s - 3,2,2^s - 3),$\cr 
$d_{s,29} = (3,3,2^s - 4,2^s - 3),$& $d_{s,30} = (3,3,2^s - 3,2^s - 4),$\cr 
$d_{s,31} = (3,2^s - 3,3,2^s - 4).$& \cr
\end{tabular}}

\medskip
For $s=3$,

\centerline{\begin{tabular}{lll}
$d_{3,32} = (3,4,1,7),$& $d_{3,33} = (3,4,7,1),$& $d_{3,34} = (3,7,4,1),$\cr 
$d_{3,35} = (7,3,4,1),$& $d_{3,36} = (3,4,3,5),$& $d_{3,37} = (1,2,4,8).$\cr
\end{tabular}}

\medskip

For $s\geqslant 4$,

\medskip
\centerline{\begin{tabular}{ll}
$d_{s,32} = (1,6,2^s - 5,2^s - 3),$& $d_{s,33} = (1,7,2^s - 6,2^s - 3),$\cr 
$d_{s,34} = (7,1,2^s - 6,2^s - 3),$& $d_{s,35} = (1,7,2^s - 5,2^s - 4),$\cr 
$d_{s,36} = (7,1,2^s - 5,2^s - 4),$& $d_{s,37} = (7,2^s - 5,1,2^s - 4),$\cr 
$d_{s,38} = (3,4,2^s - 5,2^s - 3),$& $d_{s,39} = (3,5,2^s - 6,2^s - 3),$\cr 
$d_{s,40} = (3,5,2^s - 5,2^s - 4),$& $d_{s,41} = (3,7,2^s - 7,2^s - 4),$\cr 
$d_{s,42} = (7,3,2^s - 7,2^s - 4),$& $d_{s,43} = (1,2,4,2^{s+1} - 8).$\cr
\end{tabular}}

\medskip
For $s=4$,
\begin{align*}&d_{4,44} = (3,7,8,13), \quad d_{4,45} = (7,3,8,13), \\
&d_{4,46} = (7,7,8,9),\ \quad d_{4,47} =  (7,7,9,8).
\end{align*}
\end{thm}

We prove Theorem \ref{dlc5} by proving some propositions. 

\begin{prop}\label{mdc5.1} For any $s \geqslant 2$, the $\mathbb F_2$-vector space $(\mathbb F_2\underset {\mathcal A}\otimes R_4)_{2^{s+1}-1}$ is generated by the $\rho_3(s)$ elements listed in Theorem \ref{dlc5}.
\end{prop}

This proposition is proved by combining the following lemmas. First, we determine the $\tau$-sequence of an admissible monomial of degree $2^{s+1}-1$ in $P_4$.

\begin{lem}\label{5.5} Let $x$ be an admissible monomial of degree $2^{s+1}-1$ in $P_4$. We have

1. If $s=1$ then either $\tau(x) = (1;1)$ or $\tau(x) = (3;0)$.

2. If $s=2$ then either $\tau(x) = (1;1;1)$ or $\tau(x) = (1;3)$ or $\tau(x) = (3;2)$.

3. If $s\geqslant 3$ then $\tau(x)$ is one of the following sequences
$$(\underset{\text{$s+1$ times}}{\underbrace{1;1;\ldots ; 1}}),\quad (3;\underset{\text{$s-1$ times}}{\underbrace{2;2;\ldots ; 2}}).$$
\end{lem}

We need the following for the proof of Lemma \ref{5.5}.

From the proof of Proposition \ref{mdc4.1}, Lemma \ref{2.5},  Lemma \ref{5.1} and Lemma \ref{4.1a}, we get the following

\begin{lem}\label{5.1a} Let $x$ be a monomial of degree $2^{s}-2$ in $P_4$ with $s \geqslant 2$. If  $x$ is inadmissible then there is a strictly inadmissible matrix $\Delta$ such that $\Delta \triangleright x$.
\end{lem} 

\begin{lem}\label{5.2} Let $M$ be an $2\times 4$-matrix and $x$ the monomial corresponding to $M$. If either $\tau(x) =(1;2)$ or $\tau(x) = (1;3),\ x\ne (1,2,2,2)$ then $M$ is strictly inadmissible.
\end{lem}

\begin{proof} If $\tau(x) = (1;2)$ the lemma follows from Lemma \ref{5.1}, since one of the columns of $M$ is zero. We prove the lemma for the case $\tau(x) = (1;3)$ and $x\ne (1,2,2,2)$. 

If one of the columns of $M$ is zero then the lemma follows from Lemma \ref{5.1}. So, we assume that all columns of $M$ are non-zero. Since $x\ne (1,2,2,2)$, $x$ is one of the monomials $(2,1,2,2), (2,2,1,2), (2,2,2,1)$. It is easy to see that
\begin{align*} (2,1,2,2) &= Sq^1(1,1,2,2) + (1,2,2,2),\\
(2,2,1,2) &= Sq^1(1,2,1,2) + (1,2,2,2),\\
(2,2,2,1) &= Sq^1(1,2,2,1) + (1,2,2,2). \end{align*}
Hence, the lemma is proved. 
\end{proof}

\begin{lem}\label{5.3} Let $\varepsilon_1, \varepsilon_2, \varepsilon_3, \varepsilon_4 \in \{0,1\}.$ If $\varepsilon_1+ \varepsilon_2+ \varepsilon_3 + \varepsilon_4 <4$ then the following matrix is strictly inadmissible

\medskip
\centerline{$M = \begin{pmatrix} \varepsilon_1& \varepsilon_2& \varepsilon_3& \varepsilon_4\\ 1&0&0&0 \\ 0&1&1&1\end{pmatrix} .$}
\end{lem}

\begin{proof} The monomial corresponding to the matrix $M$ is $x=(\varepsilon_1+2,\varepsilon_2+4,\varepsilon_3+4,\varepsilon_4+4)$.  If  $\varepsilon_1 = 0$ then
\begin{align*} x=Sq^1(1,\varepsilon_2&+4,\varepsilon_3+4,\varepsilon_4+4)+ \varepsilon_2(1,\varepsilon_2+5,\varepsilon_3+4,\varepsilon_4+4)\\ & + \varepsilon_3(1,\varepsilon_2+4,\varepsilon_3+5,\varepsilon_4+4) + \varepsilon_4(1,\varepsilon_2+4,\varepsilon_3+4,\varepsilon_4+5).\end{align*}
Hence, the lemma is true. We assume that $\varepsilon_1 = 1$.

If $\varepsilon_2 + \varepsilon_3 + \varepsilon_4=0$ then
$$x= (3,4,4,4) =Sq^1(3,3,4,4)+Sq^2(2,3,4,4)+(2,5,4,4).$$
If $\varepsilon_2 + \varepsilon_3 + \varepsilon_4=1$ then there is an $\mathcal A$-homomorphism $f:P_4\to P_4$ induced by a permutation of $\{x_1,x_2,x_3,x_4\}$ such that $x=f(3,5,4,4)$. We have
\begin{align*} (3,5,4,4) &= Sq^1\big((3,6,2,4)+(3,4,4,4)+(3,3,1,8)\big) \\ &\quad + Sq^2(5,3,2,4) + Sq^4(3,3,2,4) + (4,3,1,8)+(3,4,1,8).
\end{align*}
If $\varepsilon_2 + \varepsilon_3 + \varepsilon_4=2$ then there is an $\mathcal A$-homomorphism $f':P_4\to P_4$ induced by a permutation of $\{x_1,x_2,x_3,x_4\}$ such that $x=f'(3,5,5,4)$. We have
\begin{align*}
&(3,5,5,4) = Sq^1\big((4,5,5,2)+(5,5,5,1)+(6,3,5,2)+(3,6,5,2)\\ &\quad+(3,5,6,2)\big)+ Sq^2\big((3,5,5,2)+ (6,3,5,1) +(3,6,5,1)+(3,5,6,1)\big)\\ 
&\quad  + (3,6,6,2) + (6,4,6,1) + (8,3,5,1) + (3,8,5,1) + (3,5,8,1).
\end{align*}
The homomorphisms $f,f'$ send  monomials to  monomials and preserve the associated $\tau$-sequences. The lemma follows. 
\end{proof}

\begin{lem}\label{5.4} Let $\varepsilon_1, \varepsilon_2, \varepsilon_3, \varepsilon_4 \in \{0,1\}.$ If $\varepsilon_1+ \varepsilon_2+ \varepsilon_3 + \varepsilon_4 >0$ then the following matrix is strictly inadmissible
$$M = \begin{pmatrix}  1&0&0&0 \\ 0&1&1&1\\ \varepsilon_1& \varepsilon_2& \varepsilon_3& \varepsilon_4 \end{pmatrix} .$$
\end{lem}

\begin{proof} The monomial corresponding to the matrix $M$ is $x=(1+4\varepsilon_1,2+4\varepsilon_2,2+4\varepsilon_3,2+ 4\varepsilon_4)$.  If  $\varepsilon_1 = 1$ then
\begin{align*} x&= Sq^2(3,2+4\varepsilon_2,2+4\varepsilon_3,2+ 4\varepsilon_4) + 
(3,4\varepsilon_2+4,4\varepsilon_3+2,4\varepsilon_4+2)\\ &\quad + (3,4\varepsilon_2+2,4\varepsilon_3+4,4\varepsilon_4+2) + (3,4\varepsilon_2+2,4\varepsilon_3+2,4\varepsilon_4+4).\end{align*}
Hence, the lemma holds. Suppose that $\varepsilon_1=0$. 

If $\varepsilon_2 + \varepsilon_3 + \varepsilon_4=1$ then there is an $\mathcal A$-homomorphism $f:P_4\to P_4$ induced by a permutation of $\{x_1,x_2,x_3,x_4\}$ such that $x=f(1,6,2,2)$. A direct computation shows
$$(1,6,2,2) = Sq^1(2,6,1,1) + Sq^2(1,6,1,1) + Sq^4(1,4,1,1).$$

If $\varepsilon_2 + \varepsilon_3 + \varepsilon_4=2$ then there is an 
$\mathcal A$-homomorphism $f':P_4\to P_4$ induced by a permutation of $\{x_1,x_2,x_3,x_4\}$ such that $x=f'(1,6,6,2)$. We have
\begin{align*} (1,6,6,2) &= Sq^1\big((2,5,5,2) + (2,3,5,4)\big)\\ 
&\quad  + Sq^2\big((1,5,5,2) + (1,3,5,4)\big) + (1,4,6,4).
\end{align*}
If $\varepsilon_2 + \varepsilon_3 + \varepsilon_4=3$ then 
\begin{align*} x=(1,6,6,6) &= Sq^1\big((2,5,5,6)+(2,3,5,8)\big)\\ &\quad + Sq^2\big((1,5,5,6) + (1,3,5,8)\big)+ (1,4,6,8).
\end{align*}

 Hence, the lemma is proved. 
\end{proof}

\begin{proof}[Proof of Lemma \ref{5.5}] Obviously, the first part of the lemma is true. Suppose $s\geqslant 2$. Since $2^{s+1}-1$ is odd, we have either $\tau_1(x) =1$ or $\tau_1(x) =3$.

If $\tau_1(x) =1$ then using Lemmas \ref{5.2}, \ref{5.3}, \ref{5.4}, we get either $\tau(x)=(1;3)$, for $s=2,$ or $\tau(x) = (\underset{\text{$s+1$ times}}{\underbrace{1;1;\ldots ; 1}})$, for $s\geqslant 2$.

Suppose that $\tau_1(x) = 3$. Then $x = z_{i}y^2$,  $1 \leqslant i \leqslant 4$, where $z_1, z_2, z_3, z_4$ are defined as in the proof of Proposition \ref{mdc3.1} and $y$ is a monomial of degree $2^s-2$. Since $x$ is admissible, using Lemma \ref{5.1a} and Theorem \ref{2.4}, we see that $y$ is also admissible and $\tau_i(y) = \tau_{i+1}(x), i\geqslant 1$. Using Lemma \ref{4.1} and Proposition \ref{2.6}, we obtain
$$\tau(y)=(\tau_2(x);\tau_3(x);\ldots ; \tau_i(x); \ldots) = (\underset{\text{$s-1$ times}}{\underbrace{2;2;\ldots ; 2}}).$$
The lemma is proved. 
\end{proof}

\begin{lem}\label{5.6} The following matrices are strictly inadmissible
$$\Delta_{54} = \begin{pmatrix} 0&0&0&1\\ 0&0&1&0\end{pmatrix},  \Delta_{55} = \begin{pmatrix} 0&0&0&1\\ 0&1&0&0\end{pmatrix},  
\Delta_{56} = \begin{pmatrix} 0&0&1&0\\ 0&1&0&0\end{pmatrix}, $$
$$\Delta_{57} = \begin{pmatrix} 0&0&0&1\\ 1&0&0&0\end{pmatrix},  \Delta_{58} = \begin{pmatrix} 0&0&1&0\\ 1&0&0&0\end{pmatrix},  
\Delta_{59} = \begin{pmatrix} 0&1&0&0\\ 1&0&0&0\end{pmatrix}. $$ 
\end{lem}

\begin{proof}  See \cite{ka}.
\end{proof}

\begin{lem}\label{5.7} The following matrices are strictly inadmissible
$$\Delta_{60} = \begin{pmatrix} 0&1&1&1\\ 1&0&0&1\end{pmatrix}, \quad  \Delta_{61} = \begin{pmatrix} 0&1&1&1\\ 1&0&1&0\end{pmatrix},$$  $$\Delta_{62} = \begin{pmatrix} 0&1&1&1\\ 1&1&0&0 \end{pmatrix}, \quad \Delta_{63} = \begin{pmatrix} 1&0&1&1\\ 1&1&0&0\end{pmatrix}.$$
\end{lem}

\begin{proof} A direct computation shows
\begin{align*} (2,1,1,3) &= Sq^1(1,1,1,3) + (1,2,1,3) + (1,1,2,3) + (1,1,1,4),\\
(2,1,3,1) &= Sq^1(1,1,3,1) + (1,2,3,1) + (1,1,3,2) + (1,1,4,1),\\
(2,3,1,1) &= Sq^1(1,3,1,1) + (1,3,2,1) + (1,3,1,2) + (1,4,1,1),\\
(3,2,1,1) &= Sq^1(3,1,1,1) + (3,1,2,1) + (3,1,1,2) + (4,1,1,1).
\end{align*}
The lemma follows.
\end{proof}

\begin{lem}\label{5.8} The following matrices are strictly inadmissible
$$\Delta_{64} = \begin{pmatrix} 0&0&1&0\\ 0&0&1&0\\ 0&0&0&1\end{pmatrix}, \  \Delta_{65} = \begin{pmatrix} 0&1&0&0\\ 0&1&0&0\\ 0&0&0&1\end{pmatrix}, \  \Delta_{66} = \begin{pmatrix} 0&1&0&0\\ 0&1&0&0\\ 0&0&1&0\end{pmatrix}, $$  
$$\Delta_{67} = \begin{pmatrix} 1&0&0&0\\ 1&0&0&0\\ 0&0&0&1\end{pmatrix}, \ \Delta_{68} = \begin{pmatrix} 1&0&0&0\\ 1&0&0&0\\ 0&0&1&0\end{pmatrix}, \  \Delta_{69} = \begin{pmatrix} 1&0&0&0\\ 1&0&0&0\\ 0&1&0&0\end{pmatrix} .$$ 
\end{lem}

\begin{proof} See \cite{ka}.
\end{proof}

\begin{lem}\label{5.9} The following matrices are strictly inadmissible
$$\Delta_{70} =  \begin{pmatrix} 1&0&1&1\\ 1&0&0&1\\ 0&1&0&1\\ 0&1&0&1\end{pmatrix}, \  \Delta_{71} = \begin{pmatrix} 1&0&1&1\\ 1&0&1&0\\ 0&1&1&0\\ 0&1&1&0\end{pmatrix}, \  \Delta_{72} = \begin{pmatrix} 1&1&0&1\\ 1&1&0&0\\ 0&1&1&0\\ 0&1&1&0\end{pmatrix},$$  
$$\Delta_{73} = \begin{pmatrix} 1&1&0&1\\ 1&1&0&0\\ 1&0&1&0\\ 1&0&1&0\end{pmatrix},\ \Delta_{74} = \begin{pmatrix} 1&0&1&1\\ 1&0&0&1\\ 0&1&0&1\\ 0&0&1&1\end{pmatrix}, \  \Delta_{75} = \begin{pmatrix} 1&0&1&1\\ 1&0&1&0\\ 0&1&1&0\\ 0&0&1&1\end{pmatrix},$$ 
$$\Delta_{76} = \begin{pmatrix} 1&1&0&1\\ 1&1&0&0\\ 0&1&1&0\\ 0&1&0&1\end{pmatrix}, \  \Delta_{77} = \begin{pmatrix} 1&1&0&1\\ 1&1&0&0\\ 1&0&1&0\\ 1&0&0&1\end{pmatrix},\ \Delta_{78} = \begin{pmatrix} 1&1&0&1\\ 1&0&0&1\\ 0&1&0&1\\ 0&0&1&1\end{pmatrix},$$ 
$$\Delta_{79} = \begin{pmatrix} 1&1&1&0\\ 1&0&1&0\\ 0&1&1&0\\ 0&0&1&1\end{pmatrix}, \  \Delta_{80} = \begin{pmatrix} 1&1&1&0\\ 1&1&0&0\\ 0&1&1&0\\ 0&1&0&1\end{pmatrix}, \  \Delta_{81} = \begin{pmatrix} 1&1&1&0\\ 1&1&0&0\\ 1&0&1&0\\ 1&0&0&1\end{pmatrix}, $$  
$$ \Delta_{82} = \begin{pmatrix} 1&1&0&1\\ 1&1&0&0\\ 1&0&1&0\\ 0&1&1&0\end{pmatrix}, \  \Delta_{83} = \begin{pmatrix} 1&0&1&1\\ 1&0&1&0\\ 0&1&0&1\\ 0&1&0&1\end{pmatrix}, \  \Delta_{84} = \begin{pmatrix} 1&1&1&0\\ 1&0&0&1\\ 0&1&0&1\\ 0&0&1&1\end{pmatrix},$$ 
$$\Delta_{85} = \begin{pmatrix} 1&1&0&1\\ 1&0&1&0\\ 0&1&1&0\\ 0&0&1&1\end{pmatrix},\  \Delta_{86} = \begin{pmatrix} 1&1&1&0\\ 1&1&0&0\\ 0&1&1&0\\ 0&0&1&1\end{pmatrix}, \  \Delta_{87} = \begin{pmatrix} 1&1&1&0\\ 1&1&0&0\\ 1&0&1&0\\ 0&0&1&1\end{pmatrix},$$ 
$$\Delta_{88} = \begin{pmatrix} 1&1&0&1\\ 1&1&0&0\\ 0&1&1&0\\ 0&0&1&1\end{pmatrix}, \quad  \Delta_{89} = \begin{pmatrix} 1&1&0&1\\ 1&1&0&0\\ 1&0&1&0\\ 0&0&1&1\end{pmatrix}, $$  
$$ \Delta_{90} = \begin{pmatrix} 1&1&0&1\\ 1&1&0&0\\ 1&0&1&0\\ 0&1&0&1\end{pmatrix}, \quad  \Delta_{91} = \begin{pmatrix} 1&1&1&0\\ 1&1&0&0\\ 1&0&1&0\\ 0&1&0&1\end{pmatrix} . $$
\end{lem}

\begin{proof} The monomials corresponding to the above matrices respectively are
\begin{align*} &(3,12,1,15),  (3,12,15,1), (3,15,12,1),  (15,3,12,1), (3,4,9,15),\\
&(3,4,15,9), (3,15,4,9), (15,3,4,9), (3,5,8,15), (3,5,15,8),\\ 
&(3,15,5,8), (15,3,5,8), (7,11,12,1), (3,12,3,13), (3,5,9,14),\\
&(3,5,14,9), (3,7,13,8), (7,3,13,8), (3,7,12,9), (7,3,12,9), \\
&(7,11,4,9), (7,11,5,8).
\end{align*}

If $x$ is one of the four monomials $(3,12,1,15),  (3,12,15,1), (3,15,12,1)$,  $(15,3,12,1)$ then there is a homomorphism  $\psi_1:P_4\to P_4$ of $\mathcal A$-modules induced by a permutation of $\{x_1,x_2,x_3,x_4\}$ such that $x=\psi_1(3,12,1,15)$. 
By a direct calculation, we get
\begin{align*} (3,12,1,15) &= Sq^1\big((3,11,1,15) + (3,7,1,19)\big)\\
&\quad + Sq^2\big((2,11,1,15) + (5,7,2,15)\big)  + Sq^4(3,7,2,15)\\
&\quad+ (2,13,1,15) + (3,9,4,15)\quad  \text{mod }\mathcal L_4(3;2;2;2), 
\end{align*}
where $\tau(2,13,1,15) = \tau (3,9,4,15) = \tau(x)$ and 
$\sigma (2,13,1,15), \sigma (3,9,4,15)$ $< \sigma(x).$

If $x'$ is one of the four monomials $(3,4,9,15),(3,4,15,9), (3,15,4,9)$, $(15,3,4,9)$ then there is a homomorphism  $\psi_2:P_4\to P_4$ of $\mathcal A$-modules induced by a permutation of $\{x_1,x_2,x_3,x_4\}$ such that $x'=\psi_2(3,4,9,15)$. We have
\begin{align*} &(3,4,9,15) = Sq^1\big((3,1,11,15) + (3,1,7,19)\big)\\
&\quad + Sq^2\big((2,1,11,15) + (5,2,7,15)\big) + Sq^4(3,2,7,15)\\
&\quad + (2,1,13,15)+ (3,1,12,15)\quad  \text{mod }\mathcal L_4(3;2;2;2), 
\end{align*}
where $\tau(2,1,13,15) = \tau (3,1,12,15) = \tau(x')$, 
$\sigma (2,1,13,15), \sigma (3,1,12,15)$ $ < \sigma(x').$

If $x''$ is one of the  four monomials $(3,5,8,15), (3,5,15,8),(3,15,5,8)$, $ (15,3,5,8)$ then there is a homomorphism  $\psi_3:P_4\to P_4$ of $\mathcal A$-modules induced by a permutation of $\{x_1,x_2,x_3,x_4\}$ such that $x''=\varphi_3(3,5,8,15)$. A direct computation shows
\begin{align*} (3,5,8,15) &= Sq^1\big((3,3,9,15) + (3,3,5,19)\big)\\
&\quad + Sq^2\big((2,3,9,15) + (5,3,6,15)\big)  + Sq^4(3,3,6,15)\\
&\quad + (2,5,9,15) + (3,4,9,15) \quad  \text{mod }\mathcal L_4(3;2;2;2). \end{align*}
Here $\tau(2,5,9,15) = \tau (3,4,9,15) = \tau(x'')$ and $\sigma (2,5,9,15), \sigma (3,4,9,15) < \sigma(x'')$.

Similarly, we have 
\begin{align*}
&(7,11,12,1) = Sq^1(7,13,9,1) +Sq^2\big((7,11,10,1)\\
&\quad + (7,14,7,1) + (7,11,9,2)\big)+ Sq^4(5,14,7,1)\\
&\quad  +  (5,14,11,1) + (7,11,9,4)\quad  \text{mod }\mathcal L_4(3;2;2;2), \\
&(3,12,3,13) = Sq^1\big((3,11,3,13) + (3,7,3,17)\big) + Sq^2\big((5,7,3,14)\\
&\quad+(2,11,3,13)\big) + Sq^4(3,7,3,14) + (2,13,3,13)+ (2,11,5,13)\\
&\quad + (3,9,5,14)+ (3,11,4,13)\quad  \text{mod }\mathcal L_4(3;2;2;2), \\
&(3,5,9,14) = Sq^1\big((3,3,11,13) + (3,3,7,17)\big) + Sq^2\big((5,3,7,14)\\
&\quad + (2,3,11,13)\big)  + Sq^4(3,3,7,14)  + (3,4,11,13) + (3,3,12,13)\\
&\quad + (2,5,11,13) + (2,3,13,13) \quad  \text{mod }\mathcal L_4(3;2;2;2),
\end{align*} 
\begin{align*} 
&(3,7,13,8) = Sq^1(3,7,11,9) + Sq^2\big((5,7,11,6) + (5,7,7,10)\\ 
&\quad+ (2,7,11,9)\big) + Sq^4\big((3,11,7,6) + (3,5,7,12)+ (3,5,13,6)\big)\\
&\quad   + Sq^8(3,7,7,6) + (3,7,9,12) + (3,7,12,9) + (2,7,13,9)\\
&\quad  + (3,5,11,12) + (3,5,13,10)\quad  \text{mod }\mathcal L_4(3;2;2;2),\\
&(3,7,12,9) = Sq^1(3,7,9,11) + Sq^2\big((5,7,10,7) + (5,7,6,11)\\ 
&\quad + (2,7,9,11)\big) + Sq^4\big((3,11,6,7) + (3,5,12,7)  + (3,5,6,13)\big)\\
&\quad + Sq^8(3,7,6,7) + (3,7,8,13) + (3,7,9,12)  + (2,7,9,13)\\
&\quad + (3,5,12,11) + (3,5,10,13)\quad  \text{mod }\mathcal L_4(3;2;2;2),\\
&(7,11,4,9) = Sq^1(7,13,1,9) + Sq^2\big((7,11,2,9)\\ 
&\quad + (7,14,1,7) + (7,11,1,10)\big)+ Sq^4(5,14,1,7)\\
&\quad  + (7,11,1,12) + (5,14,1,11)\quad  \text{mod }\mathcal L_4(3;2;2;2),\\
&(7,11,5,8) = Sq^1(7,11,3,9) + Sq^2\big((7,13,3,6) + (7,7,3,12)\\ 
&\quad+ (7,10,3,9) \big) + Sq^4\big((11,7,3,6)  + (5,11,5,6) + (5,7,5,10)\big)\\ 
&\quad + Sq^8(7,7,3,6) + (7,9,5,10) + (5,11,9,6)+ (5,7,9,10)\\ 
&\quad + (7,11,4,9)+ (7,10,5,9)\quad  \text{mod }\mathcal L_4(3;2;2;2).
\end{align*}
The lemma is proved. 
\end{proof}

\begin{lem}\label{5.10} The following matrices are strictly inadmissible
$$\Delta_{92} = \begin{pmatrix} 1&1&0&1\\ 1&1&0&0\\ 0&1&0&1\\ 0&0&1&1\\ 0&0&1&1\end{pmatrix}, \quad \Delta_{93} = \begin{pmatrix} 1&1&0&1\\ 1&1&0&0\\ 1&0&0&1\\ 0&0&1&1\\ 0&0&1&1\end{pmatrix}. $$
\end{lem}

\begin{proof} A direct computation shows
\begin{align*}
&(7,3,24,29) = Sq^1\big((7,3,23,29) +(7,3,15,37)+(7,3,19,33)\big)\\ 
&\quad+ Sq^2\big((7,2,23,29) + (7,5,19,30)+ (7,3,15,36)+(7,4,15,35)\big)\\ 
&\quad  + Sq^4\big((4,3,23,29)  + (5,2,23,29)+ (11,3,15,30)+(5,3,21,30)\big)\\ 
&\quad  + Sq^8(7,3,15,30)+ (4,3,27,29)  + (7,2,25,29) + (5,2,27,29) \\ 
&\quad+ (5,3,25,30)\quad  \text{mod }\mathcal L_4(3;2;2;2;2).\\ 
&(3,7,24,29) = \overline{\varphi}_1(7,3,24,29).
\end{align*}

The lemma follows.  
\end{proof}

Now we prove Proposition \ref{mdc5.1}.

\begin{proof}[Proof of Proposition \ref{mdc5.1}] Suppose $x$ is an admissible monomial of degree $2^{s+1}-1$ in $P_4$. For $s=2$,  using Lemma \ref{5.5} we have either $\tau(x) = (1;1;1)$ or $\tau(x) = (1;3)$ or $\tau(x) = (3;2)$. By Lemma \ref{5.1a} and Proposition \ref{5.1}, it suffices to consider $x= a_{1,i}c_{1,j}^2, i = 1,2,3,4,\ 1 \leqslant j \leqslant 14$ or $x = z_ib_{1,j}^2, i = 1,2,3,4,\ j = 1,2,3,4,5,6$. A direct computation shows

\smallskip
\centerline{\begin{tabular}{llll}
$c_{2,1} = c_{1,2}c_{1,7}^2$,& $c_{2,2} = c_{1,3}c_{1,7}^2$,& $c_{2,3} = c_{1,3}c_{1,8}^2$,& $c_{2,4} = c_{1,4}c_{1,7}^2$,\cr   
$c_{2,5} = c_{1,4}c_{1,8}^2$,& $c_{2,6} = c_{1,4}c_{1,9}^2$,& $c_{2,7} = c_{1,1}c_{1,7}^2$,& $c_{2,8} = c_{1,2}c_{1,8}^2$,\cr   
$c_{2,9} = c_{1,3}c_{1,9}^2$,& $c_{2,10} = c_{1,4}c_{1,10}^2$,& $c_{2,11} = c_{1,3}c_{1,1}^2$,& $c_{2,12} = c_{1,4}c_{1,1}^2$,\cr   
$c_{2,13} = c_{1,4}c_{1,2}^2$,& $c_{2,14} = c_{1,4}c_{1,3}^2$,&  $c_{2,15} = z_{1}b_{1,1}^2$,& $c_{2,16} = z_{1}b_{1,2}^2$,\cr   
$c_{2,17} = z_{1}b_{1,3}^2$,& $c_{2,18} = z_{2}b_{1,1}^2$,& $c_{2,19} = z_{3}b_{1,2}^2$,& $c_{2,20} = z_{4}b_{1,3}^2$,\cr   
$c_{2,21} = z_{2}b_{1,4}^2$,& $c_{2,22} = z_{2}b_{1,5}^2$,& $c_{2,23} = z_{3}b_{1,4}^2$,& $c_{2,24} = z_{4}b_{1,5}^2$,\cr   
$c_{2,25} = z_{3}b_{1,6}^2$,& $c_{2,26} = z_{4}b_{1,6}^2$,  & $d_{2,1} = z_{3}b_{1,1}^2$,& $d_{2,2} = z_{4}b_{1,1}^2$,\cr   
$d_{2,3} = z_{2}b_{1,2}^2$,& $d_{2,4} = z_{2}b_{1,3}^2$,& $d_{2,5} = z_{4}b_{1,2}^2$,& $d_{2,6} = z_{3}b_{1,3}^2$,\cr   
$d_{2,7} = z_{4}b_{1,4}^2$,& $d_{2,8} = z_{3}b_{1,5}^2$,&$d_{2,9} = a_{1,4}c_{1,11}^2.$&\cr 
\end{tabular}}

\smallskip
For $x \ne c_{2,i}, 1 \leqslant i \leqslant 26,$ and $ x\ne d_{2,j}, 1 \leqslant j \leqslant 9$, we have $\Delta_{54} \triangleright a_{1,1}c_{1,j}^2$ for $j  =$ 1, 8;
$\Delta_{55} \triangleright a_{1,1}c_{1,j}^2$,
$\Delta_{56} \triangleright a_{1,2}c_{1,j}^2$ for $j  =$ 2, 3, 9;
$\Delta_{57} \triangleright a_{1,1}c_{1,j}^2$,
$\Delta_{58} \triangleright a_{1,2}c_{1,j}^2$,
$\Delta_{59} \triangleright a_{1,3}c_{1,j}^2$ for $j  =$ 4, 5, 6, 10;
$\Delta_{60} \triangleright z_{1}b_{1,4}^2$;
$\Delta_{61} \triangleright z_{1}b_{1,5}^2$; 
$\Delta_{62} \triangleright z_{1}b_{1,6}^2$; 
$\Delta_{63} \triangleright z_{2}b_{1,6}^2$; 
$\Delta_{64} \triangleright a_{1,2}c_{1,1}^2$; 
$\Delta_{65} \triangleright a_{1,3}c_{1,2}^2$; 
$\Delta_{66} \triangleright a_{1,3}c_{1,3}^2$; 
$\Delta_{67} \triangleright a_{1,4}c_{1,4}^2$; 
$\Delta_{68} \triangleright a_{1,4}c_{1,5}^2$; 
$\Delta_{69} \triangleright a_{1,4}c_{1,6}^2$. 
Using Lemma \ref{5.2}, we see that if $x=a_{1,i}c_{1,j}^2, i=1,2,3,4,\ j = 11, 12, 13,14$ and $x \ne (1,2,2,2)$ then $x$ is inadmissible. These equalities, Lemmas \ref{5.6}, \ref{5.7}, \ref{5.8} and Theorem \ref{2.4} imply that the case $s=2$ is true.

\smallskip
Suppose $s \geqslant 3$. Using Lemma \ref{5.5}, we get either $\tau_1(x) =1$ or $\tau_1(x) =3$. If $\tau_1(x) =1$ then $\tau_i(x)=1, 1 \leqslant i \leqslant s+1$. So,  $x = a_{1,i}y_1^2$, where $i=1,2,3,4$ and $y_1$ is a monomial with $\tau_i(y_1) = 1, i = 1,2,\ldots, s$. If $\tau_1(x) =3$ then  $\tau_i(x) = 2, 2 \leqslant i \leqslant s-1$. Hence $x = z_{i}y_2^2$,  where $y_2$ is a monomial of degree $2^s-2$. Now we prove that if $x\ne c_{s,i}$ and $x\ne d_{s,j}$ for all $i, j$ then there is a strictly inadmissible matrix $\Delta$ such that $\Delta \triangleright x$. The proof proceeds by induction on $s$. 

Let $s = 3$. Since the proposition holds for $s=2$, using Lemma \ref{5.1a}, it suffices to consider $y_1 = c_{2,i}$ and $y_2 = b_{2,j}$ for some $i, j$. A direct computation shows
 
\smallskip
\centerline{\begin{tabular}{llll}
$c_{3,1} = a_{1,2}c_{2,7}^2$,& $c_{3,2} = a_{1,3}c_{2,7}^2$,& $c_{3,3} = a_{1,3}c_{2,8}^2$,& $c_{3,4} = a_{1,4}c_{2,7}^2$,\cr   
$c_{3,5} = a_{1,4}c_{2,8}^2$,& $c_{3,6} = a_{1,4}c_{2,9}^2$,& $c_{3,7} = a_{1,1}c_{2,7}^2$,& $c_{3,8} = a_{1,2}c_{2,8}^2$,\cr   
$c_{3,9} = a_{1,3}c_{2,9}^2$,& $c_{3,10} = a_{1,4}c_{2,10}^2$,& $c_{3,11} = a_{1,3}c_{2,1}^2$,& $c_{3,12} = a_{1,4}c_{2,1}^2$,\cr   
$c_{3,13} = a_{1,4}c_{2,2}^2$,& $c_{3,14} = a_{1,4}c_{2,3}^2$,& $c_{3,15} = z_{1}b_{2,13}^2$,& $c_{3,16} = z_{1}b_{2,14}^2$,\cr   
$c_{3,17} = z_{1}b_{2,15}^2$,& $c_{3,18} = z_{2}b_{2,13}^2$,& $c_{3,19} = z_{3}b_{2,14}^2$,& $c_{3,20} = z_{4}b_{2,15}^2$,\cr   
$c_{3,21} = z_{2}b_{2,16}^2$,& $c_{3,22} = z_{2}b_{2,17}^2$,& $c_{3,23} = z_{3}b_{2,16}^2$,& $c_{3,24} = z_{4}b_{2,17}^2$,\cr   
$c_{3,25} = z_{3}b_{2,18}^2$,& $c_{3,26} = z_{4}b_{2,18}^2$,& $c_{3,27} = z_{1}b_{2,1}^2$,& $c_{3,28} = z_{1}b_{2,2}^2$,\cr   
$c_{3,29} = z_{1}b_{2,3}^2$,& $c_{3,30} = z_{2}b_{2,4}^2$,& $c_{3,31} = z_{2}b_{2,5}^2$,& $c_{3,32} = z_{3}b_{2,6}^2$,\cr   
$c_{3,33} = z_{4}b_{2,7}^2$,& $c_{3,34} = z_{3}b_{2,8}^2$,& $c_{3,35} = z_{4}b_{2,9}^2$,& $c_{3,36} = z_{2}b_{2,10}^2$,\cr   
$c_{3,37} = z_{3}b_{2,11}^2$,& $c_{3,38} = z_{4}b_{2,12}^2$,& $d_{3,1} = z_{3}b_{2,13}^2$,& $d_{3,2} = z_{4}b_{2,13}^2$,\cr   
$d_{3,3} = z_{2}b_{2,14}^2$,& $d_{3,4} = z_{2}b_{2,15}^2$,& $d_{3,5} = z_{4}b_{2,14}^2$,& $d_{3,6} = z_{3}b_{2,15}^2$,\cr   
$d_{3,7} = z_{4}b_{2,16}^2$,& $d_{3,8} = z_{3}b_{2,17}^2$,& $d_{3,9} = z_{2}b_{2,1}^2$,& $d_{3,10} = z_{2}b_{2,2}^2$,\cr   
$d_{3,11} = z_{3}b_{2,3}^2$,& $d_{3,12} = z_{3}b_{2,10}^2$,& $d_{3,13} = z_{3}b_{2,1}^2$,& $d_{3,14} = z_{4}b_{2,2}^2$,\cr   
$d_{3,15} = z_{4}b_{2,3}^2$,& $d_{3,16} = z_{3}b_{2,4}^2$,& $d_{3,17} = z_{4}b_{2,5}^2$,& $d_{3,18} = z_{4}b_{2,8}^2$,\cr   
\end{tabular}}
\centerline{\begin{tabular}{llll}
$d_{3,19} = z_{4}b_{2,10}^2$,& $d_{3,20} = z_{4}b_{2,11}^2$,& $d_{3,21} = z_{4}b_{2,1}^2$,& $d_{3,22} = z_{3}b_{2,2}^2$,\cr   
$d_{3,23} = z_{2}b_{2,3}^2$,& $d_{3,24} = z_{4}b_{2,4}^2$,& $d_{3,25} = z_{3}b_{2,5}^2$,& $d_{3,26} = z_{4}b_{2,6}^2$,\cr   
$d_{3,27} = z_{3}b_{2,7}^2$,& $d_{3,28} = z_{3}b_{2,20}^2$,& $d_{3,29} = z_{3}b_{2,19}^2$,& $d_{3,30} = z_{4}b_{2,19}^2$,\cr   
$d_{3,31} = z_{4}b_{2,20}^2$,& $d_{3,32} = z_{2}b_{2,6}^2$,& $d_{3,33} = z_{2}b_{2,7}^2$,& $d_{3,34} = z_{3}b_{2,9}^2$,\cr   
$d_{3,35} = z_{3}b_{2,12}^2$,& $d_{3,36} = z_{2}b_{2,20}^2$,& $d_{3,37} = a_{1,4}c_{2,11}^2$.& \cr 
\end{tabular}}

\smallskip
For $x \ne c_{3,i}, 1 \leqslant i \leqslant 38$ and $ x \ne d_{3,j}, 1 \leqslant j \leqslant 37$, we have 
$\Delta_{54} \triangleright a_{1,1}c_{2,j}^2$ for $j  =$ 1, 8;
$\Delta_{55} \triangleright a_{1,1}c_{2,j}^2$,
$\Delta_{56} \triangleright a_{1,2}c_{2,j}^2$ for $j  =$ 2, 3, 9, 11;
$\Delta_{57} \triangleright a_{1,1}c_{2,j}^2$,
$\Delta_{58} \triangleright a_{1,2}c_{2,j}^2$,
$\Delta_{59} \triangleright a_{1,3}c_{2,j}^2$ for $j  =$ 4, 5, 6, 10, 12, 13, 14;
$\Delta_{60} \triangleright z_{1}b_{2,j}^2$ for $j  =$ 4, 6, 16;
$\Delta_{61} \triangleright z_{1}b_{2,j}^2$ for $j  =$ 5, 7, 10, 17, 20;
$\Delta_{62} \triangleright z_{1}b_{2,j}^2$;
$\Delta_{63} \triangleright z_{2}b_{2,j}^2$ for $j  =$ 8, 9, 11, 12, 18, 19;
$\Delta_{64} \triangleright a_{1,2}c_{2,1}^2$;
$\Delta_{65} \triangleright a_{1,3}c_{2,2}^2$;
$\Delta_{66} \triangleright a_{1,3}c_{2,j}^2$ for $j  =$ 3, 11;
$\Delta_{67} \triangleright a_{1,4}c_{2,4}^2$;
$\Delta_{68} \triangleright a_{1,4}c_{2,j}^2$ for $j  =$ 5, 12;
$\Delta_{69} \triangleright a_{1,4}c_{2,j}^2$ for $j  =$ 6, 13, 14. 
So, the case $s=3$ is true.

\smallskip
Suppose that $s \geqslant 3$ and the proposition holds for $s$. We have

\smallskip
\centerline{\begin{tabular}{llll}
$c_{s+1,1} = a_{1,2}c_{s,7}^2$,& $c_{s+1,2} = a_{1,3}c_{s,7}^2$,& $c_{s+1,3} = a_{1,3}c_{s,8}^2$,& $c_{s+1,4} = a_{1,4}c_{s,7}^2$,\cr   
$c_{s+1,5} = a_{1,4}c_{s,8}^2$,& $c_{s+1,6} = a_{1,4}c_{s,9}^2$,& $c_{s+1,7} = a_{1,1}c_{s,7}^2$,& $c_{s+1,8} = a_{1,2}c_{s,8}^2$,\cr   
$c_{s+1,9} = a_{1,3}c_{s,9}^2$,& $c_{s+1,10} = a_{1,4}c_{s,10}^2$,& $c_{s+1,11} = a_{1,3}c_{s,1}^2$,& $c_{s+1,12} = a_{1,4}c_{s,1}^2$,\cr   
$c_{s+1,13} = a_{1,4}c_{s,2}^2$,& $c_{s+1,14} = a_{1,4}c_{s,3}^2$,& $c_{s+1,15} = z_{4}b_{s,13}^2$,& $c_{s+1,16} = z_{4}b_{s,14}^2$,\cr   
$c_{s+1,17} = z_{4}b_{s,15}^2$,& $c_{s+1,18} = z_{3}b_{s,13}^2$,& $c_{s+1,19} = z_{2}b_{s,14}^2$,& $c_{s+1,20} = z_{1}b_{s,15}^2$,\cr   
$c_{s+1,21} = z_{3}b_{s,16}^2$,& $c_{s+1,22} = z_{3}b_{s,17}^2$,& $c_{s+1,23} = z_{2}b_{s,16}^2$,& $c_{s+1,24} = z_{1}b_{s,17}^2$,\cr   
$c_{s+1,25} = z_{2}b_{s,18}^2$,& $c_{s+1,26} = z_{1}b_{s,18}^2$,& $c_{s+1,27} = z_{4}b_{s,1}^2$,& $c_{s+1,28} = z_{4}b_{s,2}^2$,\cr   
$c_{s+1,29} = z_{4}b_{s,3}^2$,& $c_{s+1,30} = z_{3}b_{s,4}^2$,& $c_{s+1,31} = z_{3}b_{s,5}^2$,& $c_{s+1,32} = z_{2}b_{s,6}^2$,\cr   
$c_{s+1,33} = z_{1}b_{s,7}^2$,& $c_{s+1,34} = z_{2}b_{s,8}^2$,& $c_{s+1,35} = z_{1}b_{s,9}^2$,& $c_{s+1,36} = z_{3}b_{s,10}^2$,\cr   
$c_{s+1,37} = z_{2}b_{s,11}^2$,& $c_{s+1,38} = z_{1}b_{s,12}^2$,& $c_{s+1,39} = z_{4}b_{s,19}^2$,& $c_{s+1,40} = z_{3}b_{s,20}^2$,\cr   
$c_{s+1,41} = z_{2}b_{s,21}^2$,& $c_{s+1,42} = z_{1}b_{s,22}^2$,&$d_{s+1,1} = z_{3}b_{s,13}^2$,& $d_{s+1,2} = z_{4}b_{s,13}^2$,\cr   
$d_{s+1,3} = z_{2}b_{s,14}^2$,& $d_{s+1,4} = z_{2}b_{s,15}^2$,& $d_{s+1,5} = z_{4}b_{s,14}^2$,& $d_{s+1,6} = z_{3}b_{s,15}^2$,\cr   
$d_{s+1,7} = z_{4}b_{s,16}^2$,& $d_{s+1,8} = z_{3}b_{s,17}^2$,& $d_{s+1,9} = z_{2}b_{s,1}^2$,& $d_{s+1,10} = z_{2}b_{s,2}^2$,\cr   
$d_{s+1,11} = z_{3}b_{s,3}^2$,& $d_{s+1,12} = z_{3}b_{s,10}^2$,& $d_{s+1,13} = z_{3}b_{s,1}^2$,& $d_{s+1,14} = z_{4}b_{s,2}^2$,\cr   
$d_{s+1,15} = z_{4}b_{s,3}^2$,& $d_{s+1,16} = z_{3}b_{s,4}^2$,& $d_{s+1,17} = z_{4}b_{s,5}^2$,& $d_{s+1,18} = z_{4}b_{s,8}^2$,\cr   
$d_{s+1,19} = z_{4}b_{s,10}^2$,& $d_{s+1,20} = z_{4}b_{s,11}^2$,& $d_{s+1,21} = z_{4}b_{s,1}^2$,& $d_{s+1,22} = z_{3}b_{s,2}^2$,\cr   
$d_{s+1,23} = z_{2}b_{s,3}^2$,& $d_{s+1,24} = z_{4}b_{s,4}^2$,& $d_{s+1,25} = z_{3}b_{s,5}^2$,& $d_{s+1,26} = z_{4}b_{s,6}^2$,\cr   
$d_{s+1,27} = z_{3}b_{s,7}^2$,& $d_{s+1,28} = z_{3}b_{s,24}^2$,& $d_{s+1,29} = z_{3}b_{s,23}^2$,& $d_{s+1,30} = z_{4}b_{s,23}^2$,\cr   
$d_{s+1,31} = z_{4}b_{s,24}^2$,& $d_{s+1,32} = z_{2}b_{s,19}^2$,& $d_{s+1,33} = z_{3}b_{s,19}^2$,& $d_{s+1,34} = z_{3}b_{s,20}^2$,\cr   
$d_{s+1,35} = z_{4}b_{s,19}^2$,& $d_{s+1,36} = z_{4}b_{s,20}^2$,& $d_{s+1,37} = z_{4}b_{s,21}^2$,& $d_{s+1,38} = z_{2}b_{s,29}^2$,\cr 
$d_{s+1,39} = z_{3}b_{s,29}^2$,& $d_{s+1,40} = z_{4}b_{s,29}^2$,& $d_{s+1,41} = z_{4}b_{s,30}^2$,& $d_{s+1,42} = z_{4}b_{s,32}^2$,\cr
$d_{s+1,43} = a_{1,4}c_{s,11}^2$.&&&\cr
\end{tabular}}

For $s=3$, $d_{4,44} = z_{3}b_{3,30}^2$,\ $d_{4,45} = z_{3}b_{3,32}^2$,\ $d_{4,46} = z_{3}b_{3,35}^2$,\ $d_{4,47} = z_{4}b_{3,35}^2$.

\smallskip
If $x \ne c_{s,i}, 1 \leqslant i \leqslant 42$ and $ x \ne d_{s,j}$ for all $j$ then
$\Delta_{54} \triangleright a_{1,1}c_{s,j}^2$ for $j  =$ 1, 8;
$\Delta_{55} \triangleright a_{1,1}c_{s,j}^2$,
$\Delta_{56} \triangleright a_{1,2}c_{s,j}^2$ for $j  =$ 2, 3, 9, 11;
$\Delta_{57} \triangleright a_{1,1}c_{s,j}^2$,
$\Delta_{58} \triangleright a_{1,2}c_{s,j}^2$,
$\Delta_{59} \triangleright a_{1,3}c_{s,j}^2$ for $j  =$ 4, 5, 6, 10, 12, 13, 14; 
$\Delta_{57} \triangleright a_{1,1}c_{s,37}^2$ for $s  = 3$, $\Delta_{57} \triangleright a_{1,1}c_{s,43}^2$ for $s > 3$;
$\Delta_{58} \triangleright a_{1,2}c_{s,37}^2$ for $s  = 3$, $\Delta_{58} \triangleright a_{1,2}c_{s,43}^2$ for $s > 3$
$\Delta_{59} \triangleright a_{1,3}c_{s,37}^2$ for $s = 3$, $\Delta_{59} \triangleright a_{1,3}c_{s,43}^2$ for $j  > 3$; 
$\Delta_{60} \triangleright z_{1}b_{s,j}^2$ for $j  =$ 4, 6, 16, 25;
$\Delta_{61} \triangleright z_{1}b_{s,j}^2$ for $j  =$ 5, 7, 10, 17, 20, 24, 26, 29;
$\Delta_{62} \triangleright z_{1}b_{s,j}^2$,
$\Delta_{63} \triangleright z_{2}b_{s,j}^2$ for $j  =$ 8, 9, 11, 12, 18, 21, 22, 23, 27, 28, 30, 31, 32, 33, 34, 35;
$\Delta_{64} \triangleright a_{1,2}c_{s,1}^2$; 
$\Delta_{65} \triangleright a_{1,3}c_{s,2}^2$;
$\Delta_{66} \triangleright a_{1,3}c_{s,j}^2$ for $j  =$ 3, 11;
$\Delta_{67} \triangleright a_{1,4}c_{s,4}^2$;
$\Delta_{68} \triangleright z_{1}b_{s,j}^2$ for $j  =$ 5, 12; 
$\Delta_{69} \triangleright z_{1}b_{s,j}^2$ for $j  =$ 6, 13, 14; $\Delta_{69} \triangleright z_{1}b_{s,37}^2$ for $s = 3$, $\Delta_{69} \triangleright z_{1}b_{s,43}^2$ for $s > 3$;
$\Delta_{70} \triangleright z_{2}b_{s,6}^2$;
$\Delta_{71} \triangleright z_{2}b_{s,7}^2$;
$\Delta_{72} \triangleright z_{3}b_{s,9}^2$;
$\Delta_{73} \triangleright z_{3}b_{s,12}^2$;
$\Delta_{74} \triangleright z_{2}b_{s,25}^2$;
$\Delta_{75} \triangleright z_{2}b_{s,26}^2$;
$\Delta_{76} \triangleright z_{3}b_{s,27}^2$;
$\Delta_{77} \triangleright z_{3}b_{s,28}^2$;
$\Delta_{78} \triangleright z_{3}b_{s,25}^2$;
$\Delta_{79} \triangleright z_{4}b_{s,26}^2$;
$\Delta_{80} \triangleright z_{4}b_{s,27}^2$;
$\Delta_{81} \triangleright z_{4}b_{s,28}^2$;
$\Delta_{82} \triangleright z_{3}b_{s,22}^2$;
$\Delta_{83} \triangleright z_{2}b_{s,24}^2$;
$\Delta_{84} \triangleright z_{4}b_{s,25}^2$;
$\Delta_{85} \triangleright z_{3}b_{s,26}^2$;
$\Delta_{86} \triangleright z_{4}b_{s,31}^2$;
$\Delta_{87} \triangleright z_{4}b_{s,33}^2$;
$\Delta_{88} \triangleright z_{3}b_{s,31}^2$;
$\Delta_{89} \triangleright z_{3}b_{s,33}^2$;
$\Delta_{90} \triangleright z_{3}b_{s,34}^2$, for $s=3$, $\Delta_{90} \triangleright z_{3}b_{s,j}^2$ for $j  =$ 34, 35 and $s>3$;
$\Delta_{91} \triangleright z_{4}b_{s,34}^2$ for $s=3$, $\Delta_{91} \triangleright z_{4}b_{s,j}^2$ for $j  =$ 34, 35 and $s >3$;
 $\Delta_{92} \triangleright z_{3}b_{s,30}^2$, 
$\Delta_{93} \triangleright z_{3}b_{s,32}^2$ for $s > 3$.

\smallskip
The proposition is completely proved.
\end{proof}

\medskip
Now, we prove  the $\rho_3(s)$ elements listed in Theorem \ref{dlc5} are linearly independent.

\begin{prop}\label{5.12} The elements $[d_{2,j}], 1\leqslant j \leqslant 9,$ are linearly independent in $(\mathbb F_2\underset{\mathcal A} \otimes R_4)_7$.
\end{prop}

\begin{proof} Suppose that there is a linear relation
\begin{equation}\sum_{1\leqslant j \leqslant 9}\gamma_j[d_{2,j}] = 0, \tag{\ref{5.12}.1}
\end{equation}
with $\gamma _j \in \mathbb F_2$. 

Under the homomorphisms $f_i, 1\leqslant i \leqslant 5$, the images of (\ref{5.12}.1) respectively are
\begin{align*} 
&\gamma_{\{1, 2, 5, 6, 7, 8, 9\}}[1,2,4] + \gamma_{3}[3,1,3] +   \gamma_{4}[3,3,1] =0,\\  
&\gamma_{\{2, 3, 4, 5, 7, 8, 9\}}[1,2,4] + \gamma_{1}[3,1,3] +  \gamma_{6}[3,3,1] =0,\\
&\gamma_{\{1, 3, 4, 6, 7, 8, 9\}}[1,2,4] + \gamma_{2}[3,1,3] +  \gamma_{5}[3,3,1] =0,\\
&\gamma_{\{2, 4, 5, 6, 7, 9\}}[1,2,4] + \gamma_{\{1, 3\}}[1,3,3] +  \gamma_{8}[3,3,1] =0,\\  
&\gamma_{\{1, 3, 5, 6, 8, 9\}}[1,2,4] +  \gamma_{\{2, 4\}}[1,3,3] +  \gamma_{7}[3,3,1] =0.  
\end{align*}

From these above equalities, we get $\gamma_j=0, 1\leqslant j \leqslant 9$. The proposition is proved. 
\end{proof}

\begin{prop}\label{5.13} The elements $[d_{3,j}], 1\leqslant j \leqslant 37,$ are linearly independent in $(\mathbb F_2\underset{\mathcal A} \otimes R_4)_{15}$.
\end{prop}

\begin{proof} Suppose that there is a linear relation
\begin{equation}\sum_{1\leqslant j \leqslant 37}\gamma_j[d_{3,j}] = 0, \tag{\ref{5.13}.1}
\end{equation}
with $\gamma _j \in \mathbb F_2$. We need to prove that $\gamma_j=0$ for all $j$.

Under the homomorphisms $f_i, i=1,2,\ldots , 6$, the images of (\ref{5.13}.1) respectively are
\begin{align*}
&a_1[1,2,12] + \gamma_{\{3, 32\}}[7,1,7] +   \gamma_{\{4, 33\}}[7,7,1]\\
&\quad +  \gamma_{9}[3,5,7] +  \gamma_{10}[3,7,5] +  \gamma_{\{23, 36\}}[7,3,5] =0,
\end{align*}
\begin{align*}
&a_2[1,2,12] + \gamma_{\{1, 16\}}[7,1,7] +  \gamma_{\{6, 34\}}[7,7,1]\\
&\quad +  \gamma_{13}[3,5,7] +  \gamma_{11}[3,7,5] +  \gamma_{\{22, 29\}}[7,3,5] = 0,\\  
&a_3[1,2,12] +  \gamma_{\{2, 17\}}[7,1,7] +  \gamma_{\{5, 18\}}[7,7,1]\\
&\quad +  \gamma_{14}[3,5,7] +  \gamma_{15}[3,7,5] +  \gamma_{\{21, 30, 31\}}[7,3,5] = 0,\\  
&a_4[1,2,12] +  \gamma_{\{1, 3, 9, 13\}}[1,7,7] +  \gamma_{\{8, 35\}}[7,7,1]\\
&\quad +  \gamma_{\{16, 32\}}[3,5,7] +  \gamma_{\{25, 28, 29, 36\}}[3,7,5] +  \gamma_{12}[7,3,5] = 0,\\  
&a_5[1,2,12] +  \gamma_{\{2, 4, 10, 14\}}[1,7,7] +  \gamma_{\{7, 20\}}[7,7,1]\\
&\quad +  \gamma_{\{17, 33\}}[3,5,7] +  \gamma_{\{24, 30\}}[3,7,5] +  \gamma_{19}[7,3,5] = 0,\\  
&a_6[1,2,12] +  \gamma_{\{5, 6, 11, 15\}}[1,7,7] + \gamma_{\{7, 8, 12, 19\}}[7,1,7]\\
&\quad +  \gamma_{\{26, 27, 28, 31\}}[3,5,7] +   \gamma_{\{18, 34\}}[3,7,5] +  \gamma_{\{20, 35\}}[7,3,5] = 0.
\end{align*}
where 
\begin{align*}
&a_1 = \gamma + \gamma_{\{3, 4, 9, 10, 23, 32, 33, 36\}},\
a_3 = \gamma + \gamma_{\{2, 5, 14, 15, 17, 18, 21, 30, 31\}},\\
&a_2 = \gamma + \gamma_{\{1, 6, 11, 13, 16, 22, 29, 34\}},\
a_4 = \gamma + \gamma_{\{1, 3, 8, 9, 12, 13, 16, 25, 28, 29, 32, 35, 36\}},\\
&a_5 = \gamma + \gamma_{\{2, 4, 7, 10, 14, 17, 19, 20, 24, 30, 33\}},\\
&a_6 = \gamma + \gamma_{\{5, 6, 7, 8, 11, 12, 15, 18, 19, 20, 26, 27, 28, 31, 34, 35\}}.
\end{align*} 
with $\gamma = \sum_{j = 1}^{37}\gamma_j$.
 
Computing directly from the above equalities, we get 
\begin{equation}\begin{cases} 
a_i=0,\ i = 1,2,\ldots , 6,\ 
\gamma_{j}=0, j=9, 10, \ldots , 15, 19,\\
\gamma_{\{3, 32\}} =   
\gamma_{\{4, 33\}} =  
\gamma_{\{23, 36\}} = 
\gamma_{\{1, 16\}} =   
\gamma_{\{6, 34\}} = 0,\\  
\gamma_{\{22, 29\}} =   
\gamma_{\{2, 17\}} =   
\gamma_{\{5, 18\}} =   
\gamma_{\{21, 30, 31\}} = 0,\\  
\gamma_{\{1, 3\}} =   
\gamma_{\{8, 35\}} =   
\gamma_{\{16, 32\}} =   
\gamma_{\{25, 28, 29, 36\}} = 0,\\  
\gamma_{\{2, 4\}} =   
\gamma_{\{7, 20\}} =   
\gamma_{\{17, 33\}} =   
\gamma_{\{24, 30\}} =  
\gamma_{\{5, 6\}} = 0,\\  
\gamma_{\{7, 8\}} =   
\gamma_{\{26, 27, 28, 31\}} =   
\gamma_{\{18, 34\}} =   
\gamma_{\{20, 35\}} = 0. 
\end{cases} \tag {\ref{5.13}.2}
\end{equation}  

With the aid of (\ref{5.13}.2), the homomorphisms $g_1, g_2$ send (\ref{5.13}.1) to
\begin{align*}
&\gamma_{\{25, 27, 28\}}[1,2,12] +  \gamma_{27}[7,7,1] +   \gamma_{25}[3,7,5] +  \gamma_{28}[7,3,5] =0,\\  
&\gamma_{\{21, 23, 26\}}[1,2,12] +  \gamma_{26}[7,7,1] +  \gamma_{\{21, 23, 31\}}[3,7,5] +  \gamma_{31}[7,3,5] =0. 
\end{align*}
Computing directly from these equalities gives
\begin{equation}\begin{cases} 
\gamma_j = 0,\ j = 25, 26, 27, 28, 31,\\
\gamma_{\{21, 23\}} = 0.
\end{cases} \tag {\ref{5.13}.3}
\end{equation}

With the aid of (\ref{5.13}.2) and (\ref{5.13}.3), the homomorphisms $g_3,g_4$ send (\ref{5.13}.1) respectively to
\begin{align*} 
&\gamma_{\{21, 22, 24\}}[1,2,12] + \gamma_{24}[7,7,1] +   \gamma_{\{21, 22, 24\}}[3,7,5] +  \gamma_{24}[7,3,5] = 0,\\  
&\gamma_{\{21, 22, 24\}}[1,2,12] + \gamma_{\{21, 22, 23\}}[7,7,1] +  \gamma_{22}[3,7,5] +  \gamma_{\{22, 23, 24\}}[7,3,5] = 0.      
\end{align*}
These equalities imply
\begin{equation}\gamma_j=0,\ j=21, 22, 23, 24. \tag {\ref{5.13}.4}
\end{equation}
Combining (\ref{5.13}.2), (\ref{5.13}.3) and (\ref{5.13}.4), we get 
\begin{equation}\begin{cases}  
\gamma_j = 0,\ j \ne 1, 2, \ldots , 8, 16, 17, 18, 20, 32, 33, 34, 35,\\
\gamma_{1} = \gamma_{3} = \gamma_{16} = \gamma_{32},\ \
\gamma_{2} = \gamma_{4} = \gamma_{17} = \gamma_{33},\\
\gamma_{5} = \gamma_{6} = \gamma_{18} = \gamma_{34},\ \
\gamma_{7} = \gamma_{8} = \gamma_{20} = \gamma_{35}.
\end{cases}\tag{\ref{5.13}.5}
\end{equation}
Substituting (\ref{5.13}.5) into the relation (\ref{5.13}.1), we obtain 
\begin{equation} \gamma_{1}[\theta_{1}]+ \gamma_{2}[\theta_{2}]  +\gamma_{5}[\theta_{3}] + \gamma_{7}[\theta_{4}] = 0,\tag {\ref{5.13}.6}
\end{equation}
where
\begin{align*} 
&\theta_{1} =  d_{3,1}+ d_{3,3}+ d_{3,16}+ d_{3,32}, \ \ 
\theta_{2} =  d_{3,2}+ d_{3,4}+ d_{3,17}+ d_{3,33}, \\ 
&\theta_{3} =  d_{3,5}+ d_{3,6}+ d_{3,18}+ d_{3,34}, \ \ 
\theta_{4} =  d_{3,7}+ d_{3,8}+ d_{3,20}+ d_{3,35}.
\end{align*}

Now, we prove that $\gamma_{1} = \gamma_{2} = \gamma_{5} = \gamma_{7} = 0.$ The proof is divided into 4 steps.

{\it Step 1.} Consider the homomorphisms $\varphi_i, i = 1,2,3,4,$ defined in Section \ref{2}. Under the homomorphism $\varphi _4$, the image of (\ref{5.13}.6) is
\begin{equation} \gamma_{1}[\theta_{1}]+ \gamma_{2}[\theta_{2}]  +\gamma_{5}[\theta_{3}] + \gamma_{7}[\theta_{4}] + \gamma_{7}[\theta_{3}]=0.\tag {\ref{5.13}.7}
\end{equation}
Combining (3.15.6) and (3.15.7), we get
\begin{equation} \gamma_{7}[\theta_{3}] = 0.\tag {\ref{5.13}.8}
\end{equation}
If $[\theta_3]  = 0$ then the polynomial $\theta_3$ is hit.  Hence, we have
$$\theta_3 = Sq^1(A) +Sq^2(B) + Sq^4(C),$$
for some polynomials $A \in (R_4)_{14}, B \in (R_4)_{13}, C \in (R_4)_{11}$. Let $Sq^2Sq^2Sq^2$ act on the both sides of this equality. We get
$$Sq^2Sq^2Sq^2(\theta_3) = Sq^2Sq^2Sq^2Sq^4(C),$$
By a direct calculation, we see that the monomial $(8,7,4,2)$ is a term of $Sq^2Sq^2Sq^2(\theta_3)$. This monomial is not a term of $Sq^2Sq^2Sq^2Sq^4(y)$ for all monomials $y \in (R_4)_{11}$. Hence, 
$$Sq^2Sq^2Sq^2(\theta_3) \ne Sq^2Sq^2Sq^2Sq^4(C),$$
for all $C \in (R_4)_{11}$ and we have a contradiction.
 So, $[\theta_3] \ne 0$ and the relation (\ref{5.13}.8) implies  $\gamma_{7} = 0.$ 

{\it Step 2.} Since $\gamma_{7} = 0$, the homomorphism $\varphi_1$ sends (\ref{5.13}.6) to
\begin{equation} \gamma_{1}[\theta_{1}]+ \gamma_{2}[\theta_{2}]  + \gamma_{5}[\theta_{4}] = 0.\tag {\ref{5.13}.9}
\end{equation}
Using the relation (\ref{5.13}.9) and by the same argument as given in Step 1, we get $\gamma_{5} = 0$.

{\it Step 3.} Since $\gamma_{7}  = \gamma_{5}=0$, the homomorphism $\varphi_2$ sends (\ref{5.13}.6) to
\begin{equation} 
\gamma_{1}[\theta_{1}]+ \gamma_{2}[\theta_{3}] = 0.\tag {\ref{5.13}.10}
\end{equation}
 Using the relation (\ref{5.13}.10) and by the same argument as given in Step 2, we obtain $\gamma_{2}=0$. 

{\it Step 4.} Since $\gamma_{7}  = \gamma_{5}= \gamma_{2} = 0$, the homomorphism $\varphi_3$ sends (\ref{5.13}.6) to
$$\gamma_{1}[\theta_{2}] =0.$$
 Using this relation and by the same argument as given in Step 3, we obtain   $\gamma_{1} = 0.$  

The proposition is completely proved. 
\end{proof}

\begin{prop}\label{5.15} The  elements $[d_{4,j}], 1\leqslant j \leqslant 47,$ are linearly independent in $(\mathbb F_2\underset{\mathcal A} \otimes R_4)_{31}$.
\end{prop}

\begin{proof} Suppose that there is a linear relation
\begin{equation}\sum_{1\leqslant j \leqslant 47}\gamma_j[d_{4,j}] = 0, \tag{\ref{5.15}.1}
\end{equation}
with $\gamma _j \in \mathbb F_2$. 

Applying the homomorphisms $f_i, i=1,2,3,4,5,6,$ to the linear relation (\ref{5.15}.1), we get
\begin{align*}
&a_1[1,2,28] + 1 \gamma_{3}[15,1,15] +   \gamma_{4}[15,15,1] +  \gamma_{9}[3,13,15]\\
&\quad +  \gamma_{10}[3,15,13] +  \gamma_{23}[15,3,13] +  \gamma_{\{32, 38\}}[7,11,13] =0,\\  
&a_2[1,2,28] +  \gamma_{\{1, 16\}}[15,1,15] +  \gamma_{6}[15,15,1] +  \gamma_{13}[3,13,15]\\
&\quad +  \gamma_{11}[3,15,13] +  \gamma_{\{22, 29, 45, 46\}}[15,3,13] +  \gamma_{\{33, 44\}}[7,11,13] = 0,\\  
&a_3[1,2,28] +  \gamma_{\{2, 17\}}[15,1,15] +  \gamma_{\{5, 18\}}[15,15,1] +  \gamma_{14}[3,13,15]\\
&\quad +  \gamma_{15}[3,15,13] +  \gamma_{\{21, 30, 31, 40, 41, 47\}}[15,3,13] +  \gamma_{35}[7,11,13] = 0,\\  
&a_4[1,2,28] +  \gamma_{\{1, 3, 9, 13\}}[1,15,15] +  \gamma_{8}[15,15,1] +  \gamma_{16}[3,13,15]\\
&\quad +  \gamma_{\{25, 28, 29, 38, 39, 44, 46\}}[3,15,13] +  \gamma_{12}[15,3,13] +  \gamma_{\{34, 45\}}[7,11,13] = 0,\\  
&a_5[1,2,28] +  \gamma_{\{2, 4, 10, 14\}}[1,15,15] +  \gamma_{\{7, 20\}}[15,15,1] +  \gamma_{17}[3,13,15]\\
&\quad + \gamma_{\{24, 30, 42, 47\}}[3,15,13] +  \gamma_{19}[15,3,13] +   \gamma_{36}[7,11,13] = 0,\\
&a_6[1,2,28] +  \gamma_{\{5, 6, 11, 15\}}[1,15,15] +  \gamma_{\{7, 8, 12, 19\}}[15,1,15] +  \gamma_{18}[3,15,13]\\
&\quad +  \gamma_{\{26, 27, 28, 31\}}[3,13,15] +  \gamma_{20}[15,3,13] +  \gamma_{37}[7,11,13] =0.    
\end{align*}
where
\begin{align*}
a_1 &=\gamma + \gamma_{\{3, 4, 9, 10, 23, 32, 38\}},\\
a_2 &= \gamma + \gamma_{\{1, 6, 11, 13, 16, 22, 29, 33, 44, 45, 46\}},\\
a_3 &=\gamma +  \gamma_{\{2, 5, 14, 15, 17, 18, 21, 30, 31, 35, 40, 41, 47\}},\\
a_4 &= \gamma + \gamma_{\{1, 3, 8, 9, 12, 13, 16, 25, 28, 29, 34, 38, 39, 44, 45, 46\}},\\
a_5 &=\gamma + \gamma_{\{2, 4, 7, 10, 14, 17, 19, 20, 24, 30, 36, 42, 47\}},\\
a_6 &= \gamma + \gamma_{\{5, 6, 7, 8, 11, 12, 15, 18, 19, 20, 26, 27, 28, 31, 37\}},
\end{align*}
with $\gamma = \sum_{1 \leqslant j \leqslant 47}\gamma_j$.

Computing directly from the above equalities, we get
\begin{equation}\begin{cases} 
a_i = 0,\ i = 1,2 \ldots, 6,\\ 
\gamma_j=0, \ j=1, 2, \ldots, 20, 23, 35, 36, 37,\\
\gamma_{\{32, 38\}} =   
\gamma_{\{22, 29, 34, 46\}} =   
\gamma_{\{33, 44\}} = 0,\\  
\gamma_{\{21, 30, 31, 40, 41, 47\}} =   
\gamma_{\{25, 28, 29, 32, 33, 39, 46\}} = 0,\\  
\gamma_{\{34, 45\}} =  
\gamma_{\{24, 30, 42, 47\}} =   
\gamma_{\{26, 27, 28, 31\}} = 0.                         
\end{cases} \tag {\ref{5.15}.2}
\end{equation}

With the aid of (\ref{5.15}.2), the homomorphisms $g_1, g_2$ send (\ref{5.15}.1) respectively to
 \begin{align*} 
&\gamma_{\{22, 25, 27, 28, 29, 39, 46\}}[1,2,28] +  \gamma_{27}[15,15,1]\\
&\quad +   \gamma_{\{22, 25, 29, 46\}}[3,15,13] +  \gamma_{28}[15,3,13] +  \gamma_{39}[7,11,13] = 0,\\  
&\gamma_{\{21, 24, 26, 30, 32, 41, 42, 47\}}[1,2,28] + \gamma_{26}[15,15,1] \\
&\quad +  \gamma_{\{21, 24, 30, 31, 32, 40, 41, 42, 47\}}[3,15,13] +  \gamma_{31}[15,3,13] +  \gamma_{40}[7,11,13] = 0.    
\end{align*} 

Computing directly from  the above equalities and (\ref{5.15}.2), we have
\begin{equation}\begin{cases} 
\gamma_j = 0,\ j = 26, 27, 28, 31, 39, 40,\\ 
\gamma_{\{22, 25, 29, 46\}} =    
\gamma_{\{21, 24, 30, 32, 41, 42, 47\}} = 0.\\   
\end{cases} \tag {\ref{5.15}.3}
\end{equation}
With the aid of (\ref{5.15}.2) and (\ref{5.15}.3), the homomorphisms $g_3, g_4$ send (\ref{5.15}.1) respectively to
\begin{align*}
&\gamma_{\{21, 22, 24, 25, 33, 34\}}[1,2,28] + \gamma_{\{24, 25, 34\}}[15,15,1] +  \gamma_{41}[7,11,13]\\
&\quad +   \gamma_{\{21, 29, 30, 33, 34, 41, 42, 46, 47\}}[3,15,13] +  \gamma_{\{22, 29, 30, 34, 42, 46, 47\}}[15,3,13] = 0,\\  
&\gamma_{\{21, 22, 24, 25, 33, 34\}}[1,2,28] +  \gamma_{\{21, 22, 32, 33\}}[15,15,1] +  \gamma_{42}[7,11,13]\\
&\quad +  \gamma_{\{24, 29, 30, 33, 34, 41, 42, 46, 47\}}[3,15,13] +  \gamma_{\{25, 29, 30, 32, 33, 41, 46, 47\}}[15,3,13] =0.  
\end{align*} 

From these equalities, it implies
\begin{equation}\begin{cases} 
\gamma_{41} =\gamma_{42}= \gamma_{\{21, 22, 24, 25, 33, 34\}}= 
\gamma_{\{21, 29, 30, 33, 34, 46, 47\}}= 0,\\
\gamma_{\{24, 25, 34\}} = 
\gamma_{\{22, 29, 30, 34, 46, 47\}}= 
\gamma_{\{21, 22, 24, 25, 33, 34\}}= 0,\\
\gamma_{\{21, 22, 32, 33\}}= 
\gamma_{\{24, 29, 30, 33, 34, 46, 47\}}= 
\gamma_{\{25, 29, 30, 32, 33, 46, 47\}}= 0.
\end{cases} \tag {\ref{5.15}.4}
\end{equation} 

Combining (\ref{5.15}.2), (\ref{5.15}.3) and (\ref{5.15}.4), we get
$$\begin{cases} \gamma_j =0, \text{ for } j \ne 22, 25, 29, 30, 33, 34, 44, 45, 46, 47,\\
\gamma_{22} = \gamma_{29} = \gamma_{44},\ \
\gamma_{25} = \gamma_{34} = \gamma_{45},\\
\gamma_{\{30, 47\}}=0,\ \
 \gamma_{\{22, 25, 29, 46\}} = 0.
\end{cases}$$

Substituting the above equalities into the linear relation (\ref{5.15}.1), we have
\begin{equation}
 \gamma_{22}[\theta_{1}] + \gamma_{25}[\theta_{2}] + \gamma_{46}[\theta_{3}]  +\gamma_{47}[\theta_{4}] = 0,\tag {\ref{5.15}.5}
\end{equation}
where
\begin{align*} 
&\theta_{1}= d_{4,22}+ d_{4,29}+ d_{4,33}+ d_{4,44},\\
&\theta_{2} = d_{4,25}+ d_{4,29}+ d_{4,34}+ d_{4,45},\\
&\theta_{3} = d_{4,29} + d_{4,46},\ \
\theta_{4} = d_{4,30} + d_{4,47}.
\end{align*}

Now, we prove $\gamma_{22} = \gamma_{25} = \gamma_{46} = \gamma_{47} = 0.$ The proof is divided into 4 steps.

{\it Step 1.} The homomorphism $\varphi_4$ sends (\ref{5.15}.5) to
\begin{equation}
 \gamma_{22}[\theta_{1}] + \gamma_{25}[\theta_{2}] + \gamma_{46}[\theta_{3}]  +\gamma_{47}[\theta_{4}] + \gamma_{25}[\theta_{1}] = 0.\tag {\ref{5.15}.6}
\end{equation}
Combining (\ref{5.15}.5) and (\ref{5.15}.6) gives
\begin{equation}
 \gamma_{25}[\theta_{1}] = 0.\tag {\ref{5.15}.7}
\end{equation}

By an analogous argument as given in the proof of Proposition  \ref{5.13}, we can show that the polynomial $\theta_1$ is non-hit. So, $[\theta_1] \ne 0$ and $\gamma_{25} = 0.$

{\it Step 2.} Applying the homomorphism $\varphi_1$ to (\ref{5.15}.5), we obtain
\begin{equation}
 \gamma_{22}[\theta_{2}] + \gamma_{46}[\theta_{3}]  +\gamma_{47}[\theta_{4}] =0.\tag {\ref{5.15}.8}
\end{equation}
Using (\ref{5.15}.8)  and by a same argument as given in Step 1, we get $\gamma_{22} = 0$.

{\it Step 3.} Under the homomorphism $\varphi_2$, the image of (\ref{5.15}.5) is
\begin{equation}
  \gamma_{46}[\theta_{2}]  +\gamma_{47}[\theta_{4}] = 0.\tag {\ref{5.15}.9}
\end{equation}
Using (\ref{5.15}.9)  and by a same argument as given in Step 3, we obtain $\gamma_{46}=0$. 

{\it Step 4.} Since $\gamma_{22} =\gamma_{25}= \gamma_{46}=0$, the homomorphism $\varphi_3$ sends (\ref{5.15}.5) to
$$\gamma_{47}[\theta_{3}] = 0.$$
From this equality and by a same argument as given in Step 3, we get $ \gamma_{47}=0.$ The proposition is proved. 
\end{proof}

\begin{prop}\label{5.17} For $s \geqslant 5$, the elements $[d_{s,j}], 1\leqslant j \leqslant 43,$ are linearly independent in $(\mathbb F_2\underset{\mathcal A} \otimes R_4)_{2^{s+1}-1}$.
\end{prop}

\begin{proof} Suppose that there is a linear relation
\begin{equation}\sum_{1\leqslant j \leqslant 43}\gamma_j[d_{s,j}] = 0, \tag{\ref{5.17}.1}\end{equation}
with $\gamma _j \in \mathbb F_2$.

Under the homomorphisms $f_i, i=1,2,3,4,5,6,$ the images of (\ref{5.17}.1) respectively are
 \begin{align*} 
&a_1u_{s,7} +  \gamma_{3}u_{s,9} +   \gamma_{4}u_{s,10} +  \gamma_{9}u_{s,11}\\
&\quad +  \gamma_{10}u_{s,12} +  \gamma_{23}u_{s,13} +  \gamma_{\{32, 38\}}u_{s,14} = 0,\\  
&a_2u_{s,7} + \gamma_{\{1, 16\}}u_{s,9} +  \gamma_{6}u_{s,10} +  \gamma_{13}u_{s,11}\\
&\quad +  \gamma_{11}u_{s,12} +  \gamma_{\{22, 29\}}u_{s,13} +  \gamma_{33}u_{s,14} =0,\\  
&a_3u_{s,7} +  \gamma_{\{2, 17\}}u_{s,9} +  \gamma_{\{5, 18\}}u_{s,10} +  \gamma_{14}u_{s,11}\\
&\quad +  \gamma_{15}u_{s,12} +  \gamma_{\{21, 30, 31, 40, 41\}}u_{s,13} +  \gamma_{35}u_{s,14} = 0,\\  
&a_4u_{s,7} +  \gamma_{\{1, 3, 9, 13\}}u_{s,8} +  \gamma_{8}u_{s,10} +  \gamma_{16}u_{s,11}\\
&\quad +  \gamma_{\{25, 28, 29, 38, 39\}}u_{s,12} +  \gamma_{12}u_{s,13} +  \gamma_{34}u_{s,14} = 0,\\  
&a_5u_{s,7} +  \gamma_{\{2, 4, 10, 14\}}u_{s,8} +  \gamma_{\{7, 20\}}u_{s,10} +  \gamma_{17}u_{s,11}\\
&\quad + \gamma_{\{24, 30, 42\}}u_{s,12} +  \gamma_{19}u_{s,13} +   \gamma_{36}u_{s,14} = 0,\\  
&a_6u_{s,7} +  \gamma_{\{5, 6, 11, 15\}}u_{s,8} +  \gamma_{\{7, 8, 12, 19\}}u_{s,9} +  \gamma_{18}u_{s,12}\\
&\quad +  \gamma_{\{26, 27, 28, 31\}}u_{s,11} +  \gamma_{20}u_{s,13} +  \gamma_{37}u_{s,14} = 0.
\end{align*} 
where 
\begin{align*}
a_1 &= \gamma + \gamma_{\{3, 4, 9, 10, 23, 32, 38\}},\ \
a_2 = \gamma + \gamma_{\{1, 6, 11, 13, 16, 22, 29, 33\}},\\
a_3 &= \gamma + \gamma_{\{2, 5, 14, 15, 17, 18, 21, 30, 31, 35, 40, 41\}},\\
a_4 &= \gamma + \gamma_{\{1, 3, 8, 9, 12, 13, 16, 25, 28, 29, 34, 38, 39\}},\\
a_5 &= \gamma + \gamma_{\{2, 4, 7, 10, 14, 17, 19, 20, 24, 30, 36, 42\}},\\
a_6 &= \gamma + \gamma_{\{5, 6, 7, 8, 11, 12, 15, 18, 19, 20, 26, 27, 28, 31, 37\}},
\end{align*}
with $\gamma =\sum_{1 \leqslant j  \leqslant 43}\gamma_j$.

Computing directly from the above equalities gives
\begin{equation}\begin{cases}
a_i = 0,\ i = 1,2,3,4,5,6,\\
\gamma_j = 0,\ j=1, 2, \ldots, 20, 23, 33, \ldots , 37,\\
\gamma_{\{32, 38\}} =   
\gamma_{\{22, 29\}} =    
\gamma_{\{21, 30, 31, 40, 41\}} = 0,\\  
\gamma_{\{22, 25, 28, 32, 39\}} =   
\gamma_{\{24, 30, 42\}} =  
\gamma_{\{26, 27, 28, 31\}} = 0. 
\end{cases} 
\tag {\ref{5.17}.2}\end{equation}

With the aid of (\ref{5.17}.2), the homomorphisms $g_1, g_2$ send (\ref{5.17}.1) respectively to
\begin{align*}
&\gamma_{\{25, 27, 28, 39\}}u_{s,7} +  \gamma_{27}u_{s,10} +   \gamma_{25}u_{s,12} +  \gamma_{28}u_{s,13} +  \gamma_{39}u_{s,14} =0,\\  
&\gamma_{\{21, 24, 26, 30, 32, 41, 42\}}u_{s,7} +  \gamma_{26}u_{s,10}\\
&\quad +  \gamma_{\{21, 24, 30, 31, 32, 40, 41, 42\}}u_{s,12} +  \gamma_{31}u_{s,13} +  \gamma_{40}u_{s,14} = 0.
\end{align*}

From the above equalities, we obtain
\begin{equation}\begin{cases}
\gamma_j =0,\ j = 25, 26, 27, 28, 31, 39, 40,\\
\gamma_{\{21, 24, 30, 32, 41, 42\}} = 0.
\end{cases} \tag {\ref{5.17}.3}
\end{equation}

With the aid of (\ref{5.17}.2) and (\ref{5.17}.3), the homomorphisms $g_3, g_4$ send (\ref{5.17}.1) respectively to
\begin{align*}
&\gamma_{\{21, 22, 24\}}u_{s,7} + \gamma_{24}u_{s,10} +   \gamma_{\{21, 22, 30, 41, 42\}}u_{s,12} +  \gamma_{\{30, 42\}}u_{s,13} +  \gamma_{41}u_{s,14} =0,\\  
&\gamma_{\{21, 22, 24\}}u_{s,7} +  \gamma_{\{21, 22, 32\}}u_{s,10}\\
&\quad +  \gamma_{\{22, 24, 30, 41, 42\}}u_{s,12} +  \gamma_{\{22, 30, 32, 41\}}u_{s,13} +  \gamma_{42}u_{s,14} = 0.
\end{align*} 
From these equalities, it implies
\begin{equation}
\gamma_j = 0,\ j = 21, 22, 24, 30, 32, 41, 42.
\tag {\ref{5.17}.4}
\end{equation}

Combining (\ref{5.17}.2), (\ref{5.17}.3) and (\ref{5.17}.4), we get $\gamma_j = 0$ for all $j$.
The proposition is proved.  
\end{proof}

\begin{rem}\label{5.14} 
The $\mathbb F_2$-subspace of $(\mathbb F_2\underset{\mathcal A} \otimes P_4)_{15}$ generated by the elements $[\theta_{1}], [\theta_{2}], [\theta_{3}], [\theta_{4}]$, which are defined  as in the proof of Proposition \ref{5.13}, is an $GL_4(\mathbb F_2)$-submodule of $(\mathbb F_2\underset{\mathcal A} \otimes P_4)_{15}$.

The $\mathbb F_2$-subspace of $(\mathbb F_2\underset{\mathcal A} \otimes P_4)_{31}$ generated by  $[\theta_{1}], [\theta_{2}], [\theta_{3}], [\theta_{4}]$, which are defined as in the proof of Proposition \ref{5.15}, is an $GL_4(\mathbb F_2)$-submodule of $(\mathbb F_2\underset{\mathcal A} \otimes P_4)_{31}$.
\end{rem}

\section{The indecomposables of $P_4$ in degree $2^{s+t+1} + 2^{s+1}-3$}\label{6}

\subsection{The $\tau$-sequence of the admissible monomials}\label{6.0}\

\smallskip
First of all, we determine the $\tau$-sequence of an admissible monomial of degree $2^{s+t+1}+2^{s+1} -3$ for any positive integers $s, t$.

\begin{lems}\label{6.1.1} Let  $x$ be a  monomial of degree $2^{s+t+1}+2^{s+1}-3$ in $P_4$ with $s,t$ are the positive integers.  If $x$ is admissible then $\tau (x)$ is one of the following sequences
$$(\underset{\text{$s$ times}}{\underbrace{3;3;\ldots ; 3}};\underset{\text{$t+1$ times}}{\underbrace{1;1;\ldots ; 1}}),\quad (\underset{\text{$s+1$ times}}{\underbrace{3;3;\ldots ; 3}};\underset{\text{$t-1$ times}}{\underbrace{2;2;\ldots ; 2}}).$$
\end{lems}

From the proof of Proposition \ref{mdc5.1} and Lemmas \ref{2.5},   \ref{5.1},  \ref{4.1a}, we easily obtain the following

\begin{lems}\label{6.0a} Let $x$ be a monomial of degree $2^{s}-1$ in $P_4$ with $s \geqslant 2$. If  $x$ is inadmissible then there is a strictly inadmissible matrix $\Delta$ such that $\Delta \triangleright x$.
\end{lems}

\begin{proof}[Proof of Proposition \ref{6.1.1}] Observe  that $z=(2^{s+t+1}-1,2^s-1,2^s-1,0)$ is the minimal spike of degree $2^{s+t+1} + 2^{s+1} - 3$ in $P_4$ and 
$$ \tau (z) = (\underset{\text{$s$ times}}{\underbrace{3;3;\ldots ; 3}};\underset{\text{$t+1$ times}}{\underbrace{1;1;\ldots ; 1}}).$$ 
Since $x$ is admissible and $2^{s+t+1}+2^{s+1}-3$ is odd, using Theorem \ref{2.10}, we obtain $\tau_1(x)=3$.  Using Theorem \ref{2.12} and Proposition \ref{2.6}, we get $\tau_i(x) = 3$ for $i=1,2,\ldots , s$. 

Suppose that $M=(\varepsilon_{ij}(x)), i\geqslant 1, 1\leqslant j\leqslant 4$, is the matrix associated with $x$. We set $M'=(\varepsilon_{ij}(x)), i> s, 1\leqslant j\leqslant 4$ and denote by $x'$ the monomial corresponding to $M'$. Then $\tau_i(x') = \tau_{i+s}(x), i\geqslant 1$. We have
\begin{align*} 2^{s+t+1}+2^{s+1}-3 &= \deg x = \sum_{i\geqslant 1}2^{i-1}\tau_i(x) \\
&= 3(2^s-1) + 2^s\sum_{j\geqslant 1}2^{j-1}\tau_{j+s}(x)\\
&= 2^{s+1} + 2^s - 3 +2^s\deg (x').\end{align*}
From this, it implies $\deg (x') = 2^{t+1}-1$.
Since $x$ is admissible, using Lemma \ref{6.0a} and Theorem \ref{2.4}, we see that $x'$ is also admissible. By Lemmas \ref{5.2}, \ref{5.3}, \ref{5.5}, $\tau(x')$ is one of the following sequences 
$$(\underset{\text{$t+1$ times}}{\underbrace{1;1;\ldots ; 1}}),\quad (3;\underset{\text{$t-1$ times}}{\underbrace{2;2;\ldots ; 2}}).$$
The lemma is proved. 
\end{proof}

\subsection{The case $t=1$}\label{6.1} \ 

\medskip
According to Kameko \cite{ka}, $\dim (\mathbb F_2\underset{\mathcal A}\otimes P_3)_{2^{s+2}+2^{s+1}-3} =7$ with a basis given by the classes

\medskip
\centerline{\begin{tabular}{ll}
$w_{1,s,1} = [2^s-1,2^s - 1,2^{s+2} - 1]$,& $w_{1,s,2} = [2^s-1,2^{s+2} - 1,2^s - 1]$,\cr 
$w_{1,s,3} = [2^{s+2}-1,2^s - 1,2^s - 1]$,& $w_{1,s,4} = [2^s-1,2^{s+1} - 1,3.2^s - 1]$,\cr 
$w_{1,s,5} = [2^{s+1}-1,2^s - 1,3.2^s - 1]$,& $w_{1,s,6} = [2^{s+1}-1,3.2^s - 1,2^s - 1]$,\cr 
$w_{1,s,7} = [2^{s+1}-1,2^{s+1} - 1,2^{s+1} - 1]$.& \cr
\end{tabular}}

\medskip
Hence, we have

\begin{props}\label{6.1.4} $(\mathbb F_2\underset{\mathcal A}\otimes Q_4)_{2^{s+2} +2^{s+1}-3}$ is  an $\mathbb F_2$-vector space of dimension 28 with a basis consisting of all the  classes represented by the monomials $a_{1,s,j}, 1\leqslant j\leqslant 28,$ which are determined as follows:

\medskip
\centerline{\begin{tabular}{ll}
$1.\ (0,2^{s+1} - 1,2^{s+1} - 1,2^{s+1} - 1)$,& $2.\ (2^{s+1}-1,0,2^{s+1} - 1,2^{s+1} - 1)$,\cr 
$3.\ (2^{s+1}-1,2^{s+1} - 1,0,2^{s+1} - 1)$,& $4.\ (2^{s+1}-1,2^{s+1} - 1,2^{s+1} - 1,0)$,\cr 
$5.\ (0,2^s - 1,2^s - 1,2^{s+2} - 1)$,& $6.\ (0,2^s - 1,2^{s+2} - 1,2^s - 1)$,\cr 
$7.\ (0,2^{s+2} - 1,2^s - 1,2^s - 1)$,& $8.\ (2^s-1,0,2^s - 1,2^{s+2} - 1)$,\cr 
$9.\ (2^s-1,0,2^{s+2} - 1,2^s - 1)$,& $10.\ (2^s-1,2^s - 1,0,2^{s+2} - 1)$,\cr 
$11.\ (2^s-1,2^s - 1,2^{s+2} - 1,0)$,& $12.\ (2^s-1,2^{s+2} - 1,0,2^s - 1)$,\cr 
$13.\ (2^s-1,2^{s+2} - 1,2^s - 1,0)$,& $14.\ (2^{s+2}-1,0,2^s - 1,2^s - 1)$,\cr 
$15.\ (2^{s+2}-1,2^s - 1,0,2^s - 1)$,& $16.\ (2^{s+2}-1,2^s - 1,2^s - 1,0)$,\cr 
$17.\ (0,2^s - 1,2^{s+1} - 1,3.2^s - 1)$,& $18.\ (0,2^{s+1} - 1,2^s - 1,3.2^s - 1)$,\cr 
$19.\ (0,2^{s+1} - 1,3.2^s - 1,2^s - 1)$,& $20.\ (2^s-1,0,2^{s+1} - 1,3.2^s - 1)$,\cr 
$21.\ (2^s-1,2^{s+1} - 1,0,3.2^s - 1)$,& $22.\ (2^s-1,2^{s+1} - 1,3.2^s - 1,0)$,\cr 
$23.\ (2^{s+1}-1,0,2^s - 1,3.2^s - 1)$,& $24.\ (2^{s+1}-1,0,3.2^s - 1,2^s - 1)$,\cr 
$25.\ (2^{s+1}-1,2^s - 1,0,3.2^s - 1)$,& $26.\ (2^{s+1}-1,2^s - 1,3.2^s - 1,0)$,\cr 
$27.\ (2^{s+1}-1,3.2^s - 1,0,2^s - 1)$,& $28.\ (2^{s+1}-1,3.2^s - 1,2^s - 1,0)$.\cr
\end{tabular}}
\end{props}

By Proposition \ref{2.7}, we need only to determine $(\mathbb F_2\underset{\mathcal A}\otimes R_4)_{2^{s+2}+2^{s+1}-3}$. 

Set $\mu_2(1) = 46, \mu_2(2) = 94$ and $\mu_2(s) = 105$ for $s\geqslant 3$. We have

\begin{thms}\label{dlc6.1} $(\mathbb F_2 \underset {\mathcal A} \otimes R_4)_{2^{s+2}+2^{s+1}-3}$ is an $\mathbb F_2$-vector space of dimension $\mu_2(s)-28$ with a basis consisting of all the classes represented by the monomials $a_{1,s,j}, 29 \leqslant j \leqslant \mu_2(s)$, which are determined as follows:

\smallskip
For $s \geqslant 1$,

\medskip
\centerline{\begin{tabular}{ll}
$29.\ (1,2^{s+1} - 2,2^{s+1} - 1,2^{s+1} - 1)$,& $30.\ (1,2^{s+1} - 1,2^{s+1} - 2,2^{s+1} - 1)$,\cr 
$31.\ (1,2^{s+1} - 1,2^{s+1} - 1,2^{s+1} - 2)$,& $32.\ (2^{s+1}-1,1,2^{s+1} - 2,2^{s+1} - 1)$,\cr 
$33.\ (2^{s+1}-1,1,2^{s+1} - 1,2^{s+1} - 2)$,& $34.\ (2^{s+1}-1,2^{s+1} - 1,1,2^{s+1} - 2)$,\cr 
$35.\ (1,2^s - 1,2^s - 1,2^{s+2} - 2)$,& $36.\ (1,2^s - 1,2^{s+2} - 2,2^s - 1)$,\cr 
$37.\ (1,2^{s+2} - 2,2^s - 1,2^s - 1)$,& $38.\ (1,2^s - 1,2^{s+1} - 2,3.2^s - 1)$,\cr 
$39.\ (1,2^{s+1} - 2,2^s - 1,3.2^s - 1)$,& $40.\ (1,2^{s+1} - 2,3.2^s - 1,2^s - 1)$.
\end{tabular}}

\smallskip
For $s=1$,

\smallskip
\centerline{\begin{tabular}{lll}
$41.\  (1,1,3,4),$& $42.\  (1,3,1,4),$& $43.\  (1,3,4,1),$\cr 
$44.\  (3,1,1,4),$& $45.\  (3,1,4,1),$& $46.\  (3,4,1,1).$\cr 
\end{tabular}}

\smallskip
For $s \geqslant 2$,

\medskip
\centerline{\begin{tabular}{ll}
$41.\ (2^s-1,1,2^{s+1} - 2,3.2^s - 1)$,& $42.\ (2^s-1,1,2^s - 1,2^{s+2} - 2)$,\cr 
$43.\ (2^s-1,1,2^{s+2} - 2,2^s - 1)$,& $44.\ (2^s-1,2^s - 1,1,2^{s+2} - 2)$,\cr 
$45.\ (3,2^{s+1} - 3,2^{s+1} - 2,2^{s+1} - 1)$,& $46.\ (3,2^{s+1} - 3,2^{s+1} - 1,2^{s+1} - 2)$,\cr 
$47.\ (3,2^{s+1} - 1,2^{s+1} - 3,2^{s+1} - 2)$,& $48.\ (2^{s+1}-1,3,2^{s+1} - 3,2^{s+1} - 2)$,\cr 
$49.\ (1,2^s - 2,2^s - 1,2^{s+2} - 1)$,& $50.\ (1,2^s - 2,2^{s+2} - 1,2^s - 1)$,\cr 
$51.\ (1,2^s - 1,2^s - 2,2^{s+2} - 1)$,& $52.\ (1,2^s - 1,2^{s+2} - 1,2^s - 2)$,\cr 
$53.\ (1,2^{s+2} - 1,2^s - 2,2^s - 1)$,& $54.\ (1,2^{s+2} - 1,2^s - 1,2^s - 2)$,\cr 
$55.\ (2^s-1,1,2^s - 2,2^{s+2} - 1)$,& $56.\ (2^s-1,1,2^{s+2} - 1,2^s - 2)$,\cr 
$57.\ (2^s-1,2^{s+2} - 1,1,2^s - 2)$,& $58.\ (2^{s+2}-1,1,2^s - 2,2^s - 1)$,\cr 
$59.\ (2^{s+2}-1,1,2^s - 1,2^s - 2)$,& $60.\ (2^{s+2}-1,2^s - 1,1,2^s - 2)$,\cr 
$61.\ (1,2^s - 2,2^{s+1} - 1,3.2^s - 1)$,& $62.\ (1,2^{s+1} - 1,2^s - 2,3.2^s - 1)$,\cr 
$63.\ (1,2^{s+1} - 1,3.2^s - 1,2^s - 2)$,& $64.\ (2^{s+1}-1,1,2^s - 2,3.2^s - 1)$,\cr 
$65.\ (2^{s+1}-1,1,3.2^s - 1,2^s - 2)$,& $66.\ (2^{s+1}-1,3.2^s - 1,1,2^s - 2)$,\cr 
$67.\ (1,2^s - 1,2^{s+1} - 1,3.2^s - 2)$,& $68.\ (1,2^{s+1} - 1,2^s - 1,3.2^s - 2)$,\cr 
$69.\ (1,2^{s+1} - 1,3.2^s - 2,2^s - 1)$,& $70.\ (2^s-1,1,2^{s+1} - 1,3.2^s - 2)$,\cr 
$71.\ (2^s-1,2^{s+1} - 1,1,3.2^s - 2)$,& $72.\ (2^{s+1}-1,1,2^s - 1,3.2^s - 2)$,\cr 
$73.\ (2^{s+1}-1,1,3.2^s - 2,2^s - 1)$,& $74.\ (2^{s+1}-1,2^s - 1,1,3.2^s - 2)$,\cr 
$75.\ (3,2^s - 1,2^{s+2} - 3,2^s - 2)$,& $76.\ (3,2^{s+2} - 3,2^s - 2,2^s - 1)$,\cr 
$77.\ (3,2^{s+2} - 3,2^s - 1,2^s - 2)$,& $78.\ (3,2^{s+1} - 3,2^s - 2,3.2^s - 1)$,\cr 
$79.\ (3,2^{s+1} - 3,3.2^s - 1,2^s - 2)$,& $80.\ (3,2^{s+1} - 1,3.2^s - 3,2^s - 2)$,\cr 
$81.\ (2^{s+1}-1,3,3.2^s - 3,2^s - 2)$,& $82.\ (3,2^s - 1,2^{s+1} - 3,3.2^s - 2)$,\cr 
$83.\ (3,2^{s+1} - 3,2^s - 1,3.2^s - 2)$,& $84.\ (3,2^{s+1} - 3,3.2^s - 2,2^s - 1)$.\cr
\end{tabular}}

\medskip
For $s=2$,

\smallskip
\centerline{\begin{tabular}{lll}
$85.\  (3,3,3,12),$& $86.\  (3,3,12,3),$& $87.\  (7,9,2,3),$\cr 
$88.\  (7,9,3,2),$& $89.\  (3,3,4,11),$& $90.\  (3,3,7,8),$\cr 
$91.\  (3,7,3,8),$& $92.\  (3,7,8,3),$& $93.\  (7,3,3,8),$\cr 
$94.\  (7,3,8,3).$& & \cr
\end{tabular}}

\medskip
For $s\geqslant 3$,

\smallskip
\centerline{\begin{tabular}{ll}
$85.\ (3,2^s - 3,2^s - 2,2^{s+2} - 1)$,& $86.\ (3,2^s - 3,2^{s+2} - 1,2^s - 2)$,\cr 
$87.\ (3,2^{s+2} - 1,2^s - 3,2^s - 2)$,& $88.\ (2^{s+2}-1,3,2^s - 3,2^s - 2)$,\cr 
$89.\ (3,2^s - 3,2^s - 1,2^{s+2} - 2)$,& $90.\ (3,2^s - 3,2^{s+2} - 2,2^s - 1)$,\cr 
$91.\ (3,2^s - 1,2^s - 3,2^{s+2} - 2)$,& $92.\ (2^s-1,3,2^s - 3,2^{s+2} - 2)$,\cr 
$93.\ (2^s-1,3,2^{s+2} - 3,2^s - 2)$,& $94.\ (3,2^s - 3,2^{s+1} - 2,3.2^s - 1)$,\cr 
$95.\ (3,2^s - 3,2^{s+1} - 1,3.2^s - 2)$,& $96.\ (3,2^{s+1} - 1,2^s - 3,3.2^s - 2)$,\cr 
$97.\ (2^{s+1}-1,3,2^s - 3,3.2^s - 2)$,& $98.\ (2^s-1,3,2^{s+1} - 3,3.2^s - 2)$,\cr 
$99.\ (7,2^{s+2} - 5,2^s - 3,2^s - 2)$,& $100.\ (7,2^{s+1} - 5,2^s - 3,3.2^s - 2)$,\cr 
$101.\ (7,2^{s+1} - 5,3.2^s - 3,2^s - 2)$,& $102.\ (7,2^{s+1} - 5,2^{s+1} - 3,2^{s+1} - 2)$.\cr
\end{tabular}}

\medskip
For $s=3$,
$$103.\  (7,7,7,24),\quad 104.\  (7,7,9,22),\quad 105.\  (7,7,25,6).$$

For $s\geqslant 4$,

\smallskip
\centerline{\begin{tabular}{ll}
$103.\ (7,2^s - 5,2^s - 3,2^{s+2} - 2)$,& $104.\ (7,2^s - 5,2^{s+1} - 3,3.2^s - 2)$,\cr
$105.\ (7,2^s - 5,2^{s+2} - 3,2^s - 2)$.& \cr
\end{tabular}}
\end{thms}

This theorem is proved by combining some propositions.

\begin{props}\label{mdc6.1} The $\mathbb F_2$-vector space $(\mathbb F_2\underset {\mathcal A}\otimes R_4)_{2^{s+2}+2^{s+1}-3}$ is generated by the $\mu_2(s)-28$ elements listed in Theorem \ref{dlc6.1}.
\end{props}

\begin{lems}\label{6.1.2}  The following matrix is strictly inadmissible
$$ \begin{pmatrix} 1&0&1&1\\ 1&0&1&1\\ 1&0&0&0\\ 0&1&0&0\end{pmatrix}.$$
\end{lems}

\begin{proof}   By a direct computation, we have
\begin{align*} 
&(7,8,3,3)= Sq^1\big((7,5,5,3)+(7,5,3,5) + (7,3,5,5)\big) \\
&\quad+Sq^2\big((7,6,3,3) +(7,3,6,3) + (7,3,3,6)\big) + Sq^4\big((5,6,3,3)\\
&\quad+(5,3,6,3) +(5,3,3,6)\big) +(5,10,3,3)+(5,3,10,3)\\
&\quad +(5,3,3,10)+(7,3,8,3) +(7,3,3,8)\quad \text{mod  } \mathcal L_4(3;3;1;1).
\end{align*} 
The lemma is proved.
\end{proof}

\begin{lems}\label{6.1.3} The following matrices are strictly inadmissible
$$\begin{pmatrix} 1&1&1&0\\ 1&1&0&1\\ 1&0&1&1\\ 1&0&0&0\\ 0&1&0&0\end{pmatrix} \quad \begin{pmatrix} 1&1&1&0\\ 1&1&0&1\\ 1&1&0&1\\ 0&1&0&0\\ 0&0&1&0\end{pmatrix} \quad \begin{pmatrix} 1&1&1&0\\ 1&1&0&1\\ 1&1&0&1\\ 1&0&0&0\\ 0&0&1&0\end{pmatrix} \quad \begin{pmatrix} 1&1&1&0\\ 1&1&1&0\\ 1&1&1&0\\ 0&0&1&0\\ 0&0&0&1\end{pmatrix} $$    
$$\begin{pmatrix} 1&1&1&0\\ 1&1&1&0\\ 1&1&1&0\\ 0&1&0&0\\ 0&0&0&1\end{pmatrix} \quad \begin{pmatrix} 1&1&0&1\\ 1&1&0&1\\ 1&1&0&1\\ 0&1&0&0\\ 0&0&1&0\end{pmatrix} \quad \begin{pmatrix} 1&1&1&0\\ 1&1&1&0\\ 1&1&1&0\\ 1&0&0&0\\ 0&0&0&1\end{pmatrix} \quad \begin{pmatrix} 1&1&0&1\\ 1&1&0&1\\ 1&1&0&1\\ 1&0&0&0\\ 0&0&1&0\end{pmatrix}. $$
\end{lems}

\begin{proof} The monomials corresponding to the above matrices are 
\begin{align*}&(15,19,5,6), (7,15,17,6), (15,7,17,6), (7,7,15,16),\\
&(7,15,7,16), (7,15,16,7), (15,7,7,16), (15,7,16,7).
\end{align*} 
We prove the lemma for the matrices associated to the monomials $$(15,19,5,6), (15,7,17,6), (15,7,7,16), (15,7,16,7).$$ The others can be obtained by a similar calculation. We have
\begin{align*} 
&(15,19,5,6)= Sq^1\big((15,15,7,7)+(15,15,5,9)\big)
 +Sq^2\big((15,15,7,6)\\
&\quad+(15,15,6,7)\big)+Sq^4(15,15,5,6) + Sq^8\big((11,15,5,6)\\
&\quad+(9,15,7,6) +(9,15,6,7)+(8,15,7,7)\big) +(11,23,5,6)\\
&\quad+(9,23,7,6)+(9,23,6,7)+(8,23,7,7) +(15,17,7,6)\\
&\quad+(15,17,6,7)+(15,16,7,7)\quad \text{mod  } \mathcal L_4(3;3;3;1;1),\\
&(15,7,17,6)= Sq^1\big((15,5,15,9)+(15,7,15,7)\big)+ Sq^2(15,7,15,6)\\
&\quad+(15,6,15,7)+Sq^4(15,5,15,6) + Sq^8\big((9,7,15,6)\\
&\quad+(11,5,15,6)+(8,7,15,7)+(9,6,15,7)\big) +(9,7,23,6)\\
&\quad+(11,5,23,6)+(8,7,23,7)+(9,6,23,7)+ (15,5,19,6)\\ 
&\quad +(15,6,17,7)+(15,7,16,7) \quad\text{mod  } \mathcal L_4(3;3;3;1;1),\\
&(15,7,16,7)= Sq^1\big((15,7,13,9)+(15,7,9,13)+ (15,7,11,11) \\
&\quad+ (15,5,11,13)+ (15,5,13,11)\big) +Sq^2\big((15,7,14,7)\\
&\quad+(15,7,11,10)+(15,7,10,11)+(15,7,7,14)+ (15,6,11,11)\big)\\
&\quad+Sq^4\big((15,5,14,7)+(15,5,7,14)\big) + Sq^8\big((9,7,14,7)\\
&\quad+(11,5,14,7)+(9,7,7,14) +(11,5,7,14)\big)+(9,7,22,7)\\
&\quad+(15,5,18,7)+(11,5,22,7)+(9,7,7,22)+(15,7,7,16)\\
&\quad +(15,5,7,18)+(11,5,7,22) \quad\text{mod  } \mathcal L_4(3;3;3;1;1),\\
&(15,7,7,16)= Sq^1\big((15,7,7,15)+(15,3,11,15)\big) +Sq^2\big((15,7,6,15)\\
&\quad +(15,3,10,15)\big) +Sq^4\big((15,4,7,15) + (15,5,6,15)\big)\\
&\quad+ Sq^8\big((8,7,7,15)+(11,4,7,15) +(9,7,6,15)+(11,5,6,15)\big)\\
&\quad+(8,7,7,23)+(11,4,7,23) +(15,4,7,19)+(9,7,6,23)\\
&\quad+(15,7,6,17)+(15,5,6,19)  + (11,5,6,23)\quad\text{mod  } \mathcal L_4(3;3;3;1;1).
\end{align*}
The lemma follows.
\end{proof}

Combining Theorem \ref{2.4}, the lemmas in Sections \ref{3}, \ref{5} and the  Lemmas \ref{6.1.1}, \ref{6.1.2}, \ref{6.1.3}, we get Proposition \ref{mdc6.1}.

\medskip
Now, we show that $[a_{1,s,j}], 29\leqslant j\leqslant \mu_2(s)$, are linearly independent in the space $(\mathbb F_2\underset {\mathcal A}\otimes R_4)_{2^{s+2}+2^{s+1}-3}$.

\begin{props}\label{6.1.5} The elements $[a_{1,1,j}], \ 29 \leqslant j \leqslant 46,$ are linearly independent in $(\mathbb F_2\underset {\mathcal A}\otimes R_4)_9$.
\end{props}

\begin{proof} Suppose there is a linear relation
\begin{equation} \sum_{j=29}^{46}\gamma_j[a_{1,1,j}]=0, \tag{\ref{6.1.5}.1}
\end{equation}
where $\gamma_j \in \mathbb F_2.$ 

Consider the homomorphisms $f_i,\ i=1,2,3,4,5,6$. Under these homomorphisms, the images of the  above relation respectively are
\begin{align*}
&\gamma_{\{37, 46\}}[7,1,1] +  \gamma_{39}[3,1,5] +   \gamma_{40}[3,5,1] +  \gamma_{29}[3,3,3] =0,\\
&\gamma_{\{36, 45\}}[7,1,1] +  \gamma_{\{32, 38\}}[3,1,5] +  \gamma_{43}[3,5,1] +  \gamma_{30}[3,3,3] =0,\\  
&\gamma_{\{35, 44\}}[7,1,1] +  \gamma_{\{33, 41\}}[3,1,5] +
  \gamma_{\{34, 42\}}[3,5,1] +  \gamma_{31}[3,3,3] =0,\\  
&\gamma_{\{36, 37, 40, 43\}}[1,7,1]\\
&\quad +  \gamma_{\{29, 30, 38, 39\}}[1,3,5] +  \gamma_{\{45, 46\}}[3,5,1] +  \gamma_{32}[3,3,3] =0,\\
&\gamma_{\{35, 37, 39, 42\}}[1,7,1]\\
&\quad +  \gamma_{\{29, 31, 40, 41\}}[1,3,5]+  \gamma_{\{34, 44, 46\}}[3,5,1] +  \gamma_{33}[3,3,3] =0,\\
&\gamma_{\{35, 36, 38, 41\}}[1,1,7]\\
&\quad + \gamma_{\{30, 31, 42, 43\}}[1,3,5] +  \gamma_{\{32, 33, 44, 45\}}[3,1,5] +  \gamma_{34}[3,3,3] =0.
\end{align*} 

From the above equalities, we get 
\begin{equation}\begin{cases} \gamma_j= 0,\ \text{for } j \ne 35,36,37,44,45,46,\\
\gamma_{35} = \gamma_{36} = \gamma_{37} = \gamma_{44} = \gamma_{45} =
\gamma_{46}.\end{cases} \tag{\ref{6.1.5}.2}
\end{equation}

Substituting (\ref{6.1.5}.2) into the relation (\ref{6.1.5}.1), we get
\begin{equation} \gamma_{35}[\theta] =0, \tag{\ref{6.1.5}.3}
\end{equation} 
where
$\theta = a_{1,1,35} + a_{1,1,36} + a_{1,1,37} + a_{1,1,44} + a_{1,1,45} + a_{1,1,46}.$

We prove $\gamma_{35} = 0$ by showing $[\theta] \ne 0$. Suppose $[\theta] = 0$. Then $\theta$ is hit and we get
$$\theta = Sq^1(A) + Sq^2 (B) +Sq^4(C),$$
for some polynomials $A\in (R_4)_8, B \in (R_4)_7, C \in (R_4)_5$. Let $(Sq^2)^3 = Sq^2Sq^2Sq^2$ act on the both sides of this equality. It is easy to check that $(Sq^2)^3Sq^4(C)$ $ = 0$ for all $C\in (R_4)_5$. Since $(Sq^2)^3$ annihilates $Sq^1$ and $Sq^2$, the right hand side is sent to zero. On the other hand,  a direct computation shows
$$ (Sq^2)^3(\theta) = (1,2,4,8) + \text{symmetries} \ne 0.$$ 
Hence, we have a contradiction. So, $[\theta] \ne 0$ and the relation (\ref{6.1.5}.3) implies $\gamma_{35} = 0.$
The proposition is proved. 
\end{proof}

\begin{props}\label{6.1.7} The elements $[a_{1,2,j}], \ 29 \leqslant j \leqslant 94,$ are linearly independent in $(\mathbb F_2\underset {\mathcal A}\otimes R_4)_{21}$.
\end{props}

\begin{proof} Suppose there is a linear relation
\begin{equation}\sum_{j=29}^{94}\gamma_j[a_{1,2,j}]=0,\tag{\ref{6.1.7}.1}
\end{equation}
where $\gamma_j \in \mathbb F_2.$ 

Apply the homomorphisms $f_i, i=1,2,\ldots ,6,$ to the relation (\ref{6.1.7}.1) and we get
\begin{align*}
&\gamma_{49}[3,3,15]  +   \gamma_{50}[3,15,3]  +    \gamma_{37}[15,3,3]  +   \gamma_{61}[3,7,11]\\
&\quad  +   \gamma_{39}[7,3,11]  +   \gamma_{40}[7,11,3]  +   \gamma_{29}[7,7,7] =0,\\
&\gamma_{51}[3,3,15]  +   \gamma_{53}[3,15,3]  +   \gamma_{\{36, 86, 94\}}[15,3,3]  +   \gamma_{62}[3,7,11]\\
&\quad   +   \gamma_{\{38, 89\}}[7,3,11]  +   \gamma_{\{69, 92\}}[7,11,3]  +   \gamma_{30}[7,7,7] =0,\\
&\gamma_{52}[3,3,15]  +  \gamma_{54}[3,15,3]  +   \gamma_{\{35, 85, 93\}}[15,3,3]  +   \gamma_{63}[3,7,11] \\
&\quad  +   \gamma_{\{67, 90\}}[7,3,11]  +   \gamma_{\{68, 91\}}[7,11,3]  +   \gamma_{31}[7,7,7] =0,\\  
&\gamma_{55}[3,3,15]  +   \gamma_{\{43, 76, 84, 86, 92\}}[3,15,3]  +   \gamma_{58}[15,3,3] +   \gamma_{64}[7,3,11] \\
&\quad   +   \gamma_{\{41, 45, 78, 89\}}[3,7,11]   +   \gamma_{\{73, 87, 94\}}[7,11,3]  +   \gamma_{32}[7,7,7] =0,\\
&\gamma_{56}[3,3,15]  +   \gamma_{\{42, 77, 83, 85, 91\}}[3,15,3]   +   \gamma_{\{46, 70, 79, 90\}}[3,7,11]  \\
&\quad +  \gamma_{59}[15,3,3] +    \gamma_{65}[7,3,11]  +   \gamma_{\{72, 88, 93\}}[7,11,3]  +   \gamma_{33}[7,7,7] =0,\\
&\gamma_{\{44, 75, 82, 85, 86, 89, 90\}}[3,3,15]  +   \gamma_{57}[3,15,3]  +   \gamma_{60}[15,3,3] +   \gamma_{66}[7,11,3] \\
&\quad  +   \gamma_{\{47, 71, 80, 91, 92\}}[3,7,11]   +  \gamma_{\{48, 74, 81, 93, 94\}}[7,3,11]  +   \gamma_{34}[7,7,7] =0.
\end{align*}
From the above equalities, we obtain
\begin{equation}\begin{cases}
\gamma_j = 0, \ 29\leqslant j \leqslant 34, \text{or } j = 37,39,40, \text{or } 49 \leqslant j \leqslant 66,\\ 
\gamma_{\{36, 86, 94\}} =   
\gamma_{\{38, 89\}} =    
\gamma_{\{69, 92\}} =   
\gamma_{\{35, 85, 93\}} =   
\gamma_{\{67, 90\}} = 0,\\  
\gamma_{\{68, 91\}} = 
\gamma_{\{43, 76, 84, 86, 92\}} =   
\gamma_{\{41, 45, 78, 89\}} =   
\gamma_{\{73, 87, 94\}} = 0,\\  
\gamma_{\{42, 77, 83, 85, 91\}} =
\gamma_{\{46, 70, 79, 90\}} =   
\gamma_{\{72, 88, 93\}} = 0,\\  
\gamma_{\{44, 75, 82, 85, 86, 89, 90\}} =   
\gamma_{\{47, 71, 80, 91, 92\}} =  
\gamma_{\{48, 74, 81, 93, 94\}} = 0. 
\end{cases}\tag{\ref{6.1.7}.2}
\end{equation}

With the aid of (\ref{6.1.7}.2), the homomorphisms $g_1, g_2$ send (\ref{6.1.7}.1) respectively to
\begin{align*}
&\gamma_{\{36, 43, 73, 86, 94\}}[3,15,3]  +  \gamma_{\{76, 87\}}[15,3,3]  +    \gamma_{41}[3,7,11]\\
&\quad  +   \gamma_{78}[7,3,11]  +  \gamma_{\{73, 84, 87, 94\}}[7,11,3]  +   \gamma_{45}[7,7,7]  =0,\\   
&\gamma_{\{35, 42, 72, 85, 93\}}[3,15,3]  +  \gamma_{\{77, 88\}}[15,3,3]  +   \gamma_{70}[3,7,11] \\
&\quad  +   \gamma_{79}[7,3,11]  +   \gamma_{\{72, 83, 88, 93\}}[7,11,3]  +   \gamma_{46}[7,7,7]  =0.        
 \end{align*}

Computing directly from the above equalities, we obtain
\begin{equation}\begin{cases}
\gamma_j = 0,\ j = 41, 45, 46, 70, 78, 79,\\
\gamma_{\{36, 43, 73, 86, 94\}} = 
\gamma_{\{76, 87\}} =    
\gamma_{\{35, 42, 72, 85, 93\}} = 0,\\ 
\gamma_{\{77, 88\}} =   
\gamma_{\{73, 84, 87, 94\}} = 
\gamma_{\{72, 83, 88, 93\}} = 0.
\end{cases} \tag{\ref{6.1.7}.3}
\end{equation}     

With the aid of (\ref{6.1.7}.2) and (\ref{6.1.7}.3), the homomorphisms $g_3, g_4$ send (\ref{6.1.7}.1) respectively to
\begin{align*}
&a_1[3,15,3]  + a_2[15,3,3]  +    \gamma_{71}[3,7,11] \\
&\quad +   \gamma_{80}[7,3,11]  +   a_3[7,11,3]  +   \gamma_{47}[7,7,7]  =0,\\
&a_4[3,15,3]  +  a_5[15,3,3]  +   a_6[3,7,11] \\
&\quad +   a_7[7,3,11]  +   a_8[7,11,3]  +   \gamma_{48}[7,7,7]  =0,    
\end{align*}
where
\begin{align*}
a_1 &= \gamma_{\{35, 44, 74, 85, 93\}},\ \
a_2 = \gamma_{\{36, 75, 81, 86, 94\}},\\
a_3 &= \gamma_{\{74, 81, 82, 93, 94\}},\ \
a_4 = \gamma_{\{42, 44, 68, 71, 76, 77, 83, 85\}},\\
a_5 &=  \gamma_{\{43, 69, 75, 76, 77, 80, 84, 86\}},\ \
a_6 = \gamma_{\{72, 74, 87, 88, 93\}},\\
a_7 &= \gamma_{\{73, 81, 87, 88, 94\}},\ \
a_8 = \gamma_{\{38, 67, 68, 69, 71, 80, 82, 83, 84\}}.
\end{align*}

From the above equalities, we obtain
\begin{equation}\begin{cases}
\gamma_{47} = \gamma_{48} = \gamma_{71} = \gamma_{80} = 0,\\
a_i = 0,\  1 \leqslant i \leqslant 8. 
 \end{cases} \tag{\ref{6.1.7}.4}
\end{equation}       

Combining (\ref{6.1.7}.2), (\ref{6.1.7}.3) and (\ref{6.1.7}.4), we get
\begin{equation}\begin{cases}  
\gamma_j = 0, \ j \ne 35, 36, 42, 43, 44, 68, 69, 72, 73, 74, 75,\\
 \hskip2.5cm 76, 77, 81, 85, 86, 87, 88, 91, 92, 93, 94,\\
\gamma_{35} = \gamma_{36} = \gamma_{68} = \gamma_{69} = \gamma_{91} = \gamma_{92},\ \
\gamma_{42} = \gamma_{43} = \gamma_{72} = \gamma_{73},\\
\gamma_{44} = \gamma_{74} = \gamma_{76} = \gamma_{87},\
\gamma_{75} = \gamma_{77} = \gamma_{81} = \gamma_{88},\\
\gamma_{\{42, 44, 94\}} =  \gamma_{\{42, 75, 93\}} =
\gamma_{\{35, 42, 44, 86\}} = \gamma_{\{35, 42, 75, 85\}} =0.
\end{cases}\tag{\ref{6.1.7}.5}
\end{equation}

Substituting (\ref{6.1.7}.5) into the relation (\ref{6.1.7}.1), we have
\begin{equation} \gamma_{35}[\theta_{1}] + \gamma_{42}[\theta_{2}] + \gamma_{44}[\theta_{3}] + \gamma_{75}[\theta_{4}] =0, \tag{\ref{6.1.7}.6}
\end{equation}
where
\begin{align*}
\theta_{1} &= a_{1,2,35}+ a_{1,2,36}+ a_{1,2,68}+ a_{1,2,69}\\
&\quad+ a_{1,2,85}+ a_{1,2,86}+ a_{1,2,91}+ a_{1,2,92},\\
\theta_{2} &= a_{1,2,42}+ a_{1,2,43}+ a_{1,2,72}+ a_{1,2,73}\\
&\quad+ a_{1,2,85}+ a_{1,2,86}+ a_{1,2,93}+ a_{1,2,94},\\
\theta_{3} &= a_{1,2,44}+ a_{1,2,74}+ a_{1,2,76}+ a_{1,2,86}+ a_{1,2,87}+ a_{1,2,94},\\
\theta_{4} &= a_{1,2,75}+ a_{1,2,77}+ a_{1,2,81}+ a_{1,2,85}+ a_{1,2,88}+ a_{1,2,93}.
\end{align*}

We need to show that $\gamma_{35} = \gamma_{42} = \gamma_{44} = \gamma_{75} = 0.$ The proof is divided into 4 steps.

{\it Step 1.} The homomorphism $\varphi_4$ sends (\ref{6.1.7}.6) to
\begin{equation}(\gamma_{35}+\gamma_{42})[\theta_{1}] + \gamma_{42}[\theta_{2}] + \gamma_{44}[\theta_{3}] + \gamma_{75}[\theta_{4}] =0. \tag{\ref{6.1.7}.7}
\end{equation}
Combining (\ref{6.1.7}.6) and (\ref{6.1.7}.7) gives
$$\gamma_{42}[\theta_{1}] =0.$$
We prove $[\theta_1] \ne 0$. Suppose $[\theta_1] = 0$. Then, the polynomial $\theta_1$ is hit and we have
$$\theta_1 = Sq^1(A) +Sq^2(B) + Sq^4(C) + Sq^8(D),$$
for some polynomials $A \in (R_4)_{20}, B\in (R_4)_{19}, C \in (R_4)_{17}, D\in (R_4)_{13}$.

Let $(Sq^2)^3$ act on the both sides of this equality. Since $(Sq^2)^3Sq^1=0$ and $(Sq^2)^3Sq^2=0$, we get
$$(Sq^2)^3(\theta_1) = (Sq^2)^3Sq^4(C) + (Sq^2)^3Sq^8(D).$$

By a direct computation, we see that the monomial $(8,7,8,3)$ is a term of $(Sq^2)^3(\theta_1)$. This monomial is not a term of $(Sq^2)^3Sq^4(u)$ for $u \ne (3,7,4,3), (4,7,3,3)$. It is also not a term of $(Sq^2)^3Sq^8(D)$ for all polynomial $D$. So, either $(3,7,4,3)$ or $(4,7,3,3)$ is a term of $C$. 

If $(3,7,4,3)$ is a term of $C$ then
$$(Sq^2)^3(\theta_1+Sq^4(3,7,4,3)) = (Sq^2)^3(Sq^4(C') + Sq^8(D)),$$
where $C' = C + (3,7,4,3)$. The monomial $(4,16,4,3)$ is a term of the polynomial $ (Sq^2)^3(\theta_1+Sq^4(3,7,4,3))$. This monomial is not a term of the polynomial $(Sq^2)^3Sq^4(u)$ for $u \ne (3,7,4,3), (4,7,3,3)$. It is also not a term of $(Sq^2)^3Sq^8(D)$ for all $D \in (R_4)_{13}$. So, either $(3,7,4,3)$ or $(4,7,3,3)$ is a term of $C'$. Since $(3,7,4,3)$ is not a term of $C'$, the monomial $(4,7,3,3)$ is a term of $C'$. Then we have 
$$(Sq^2)^3(\theta_1+Sq^4((3,7,4,3) + (4,7,3,3))) = (Sq^2)^3(Sq^4(C'') + Sq^8(D)),$$
where $C'' = C' + (4,7,3,3) = C + (3,7,4,3) + (4,7,3,3)$. Now the monomial $(8,7,8,3)$ is a term of $(Sq^2)^3(\theta_1+Sq^4((3,7,4,3) + (4,7,3,3)))$. Hence, either $(3,7,4,3)$ or $(4,7,3,3)$ is a term of $C''$. On the other hand, 
the two monomials $(3,7,4,3)$ and $(4,7,3,3)$ are not  a term of $C''$, so we have a contradiction.

By a same argument, if $(4,7,3,3)$ is a term of $C$ then we also have a contradiction.
Hence, $[\theta_1] \ne 0$ and $\gamma_{42} = 0.$

{\it Step 2.} Since $\gamma_{42} = 0$, the homomorphism $\varphi_1$ sends (\ref{6.1.7}.6) to
\begin{equation}\gamma_{35}[\theta_{2}] + \gamma_{44}[\theta_{3}] + \gamma_{75}[\theta_{4}] =0. \tag{\ref{6.1.7}.8}
\end{equation}
Using (\ref{6.1.7}.8) and by an analogous argument as given in Step 1, we get  $\gamma_{35} = 0$.

{\it Step 3.} The homomorphism $\varphi_2$ sends (\ref{6.1.7}.8) to
\begin{equation}\gamma_{44}[\theta_{1}] +  \gamma_{75}[\theta_{4}] =0. \tag{\ref{6.1.7}.9}
\end{equation}
Using the relation (\ref{6.1.7}.9) and by a same argument as given in Step 2, we obtain $\gamma_{44} = 0$. 

{\it Step 4.} The homomorphism $\varphi_3$ sends (\ref{6.1.7}.9) to
$$  \gamma_{75}[\theta_{3}] =0.$$
Using this equality and by a same argument as given in Step 3, we get $\gamma_{75}=0.$ The proposition is proved.
\end{proof}

\begin{props}\label{6.1.9} For $s \geqslant 3$,  the elements $[a_{1,s,j}], \ 29 \leqslant j \leqslant 105,$ are linearly independent in $(\mathbb F_2\underset {\mathcal A}\otimes R_4)_{2^{s+2} + 2^{s+1} - 3}$.
\end{props}

\begin{proof} Suppose there is a linear relation
\begin{equation}\sum_{j=29}^{105}\gamma_j[a_{1,s,j}]=0,\tag{\ref{6.1.9}.1}
\end{equation}
where $\gamma_j \in \mathbb F_2.$ 

Applying the homomorphisms $f_i, i=1,2,3,4,5,6,$ to (\ref{6.1.9}.1), we have
\begin{align*}
&\gamma_{49}w_{1,s,1} + \gamma_{50}w_{1,s,2} +   \gamma_{37}w_{1,s,3} +  \gamma_{61}w_{1,s,4}\\
&\qquad +  \gamma_{39}w_{1,s,5} +  \gamma_{40}w_{1,s,6} +  \gamma_{29}w_{1,s,7} =0,\\  
&\gamma_{51}w_{1,s,1} +  \gamma_{53}w_{1,s,2} +  \gamma_{36}w_{1,s,3} +  \gamma_{62}w_{1,s,4}\\
&\qquad  +  \gamma_{38}w_{1,s,5} +  \gamma_{69}w_{1,s,6} +  \gamma_{30}w_{1,s,7} =0,\\ 
&\gamma_{52}w_{1,s,1} +  \gamma_{54}w_{1,s,2} +  a_1w_{1,s,3} +  \gamma_{63}w_{1,s,4}\\
&\qquad  +  \gamma_{67}w_{1,s,5} +  \gamma_{68}w_{1,s,6} +  \gamma_{31}w_{1,s,7} =0,\\
&\gamma_{55}w_{1,s,1} +  \gamma_{43}w_{1,s,2} +  \gamma_{58}w_{1,s,3} +  \gamma_{41}w_{1,s,4}\\
&\qquad  +  \gamma_{64}w_{1,s,5} +  \gamma_{73}w_{1,s,6} +  \gamma_{32}w_{1,s,7} =0,\\  
&\gamma_{56}w_{1,s,1} +  a_2w_{1,s,2} + \gamma_{59}w_{1,s,3} +  \gamma_{70}w_{1,s,4}\\
&\qquad  +   \gamma_{65}w_{1,s,5} +  \gamma_{72}w_{1,s,6} +  \gamma_{33}w_{1,s,7} = 0,\\  
&a_3w_{1,s,1} +  \gamma_{57}w_{1,s,2} +  \gamma_{60}w_{1,s,3} +  \gamma_{71}w_{1,s,4}\\
&\qquad  +  \gamma_{74}w_{1,s,5} +  \gamma_{66}w_{1,s,6} +  \gamma_{34}w_{1,s,7} = 0,
\end{align*}   
where
\begin{align*}
a_1 &= \begin{cases}  \gamma_{\{35, 103\}}, &s=3,\\   \gamma_{35}, &s \geqslant 4, \end{cases} \ \
a_2 = \begin{cases}  \gamma_{\{42, 103\}}, &s=3,\\  \gamma_{42}, &s \geqslant 4, \end{cases} \\
a_3 &= \begin{cases}  \gamma_{\{44, 103, 104, 105\}}, &s=3,\\   \gamma_{44}, &s \geqslant 4. \end{cases}
\end{align*}
From these equalities, we get 
\begin{equation}\begin{cases}
a_1 = a_2 = a_3 = 0,\\
\gamma_j = 0,\ j = 29, \ldots, 34, 36, \ldots , 41, 43, 49, \ldots , 74.
\end{cases}\tag{\ref{6.1.9}.2}
\end{equation}

With the aid of (\ref{6.1.9}.2), the homomorphisms $g_1, g_2, g_3, g_4$ send (\ref{6.1.9}.1) respectively to
\begin{align*}
&\gamma_{85}w_{1,s,1} +  \gamma_{90}w_{1,s,2} +   \gamma_{76}w_{1,s,3} +  \gamma_{94}w_{1,s,4}\\
&\qquad +  \gamma_{78}w_{1,s,5} +  \gamma_{84}w_{1,s,6} +  \gamma_{45}w_{1,s,7} =0,\\  
&\gamma_{86}w_{1,s,1} +  a_4w_{1,s,2} +  \gamma_{77}w_{1,s,3} +  \gamma_{95}w_{1,s,4}\\
&\qquad +  \gamma_{79}w_{1,s,5} +  \gamma_{83}w_{1,s,6} +  \gamma_{46}w_{1,s,7} =0,\\  
&\gamma_{87}w_{1,s,1} +  a_5w_{1,s,2} +  a_6w_{1,s,3} +  \gamma_{96}w_{1,s,4}\\
&\qquad +  \gamma_{80}w_{1,s,5} +  a_7w_{1,s,6} +  \gamma_{47}w_{1,s,7} =0,\\  
&\gamma_{88}w_{1,s,1} +  a_8w_{1,s,2} + a_9w_{1,s,3} +  \gamma_{97}w_{1,s,4}\\
&\qquad +  \gamma_{81}w_{1,s,5} + a_{10}w_{1,s,6} +  \gamma_{48}w_{1,s,7} =0,
\end{align*} 
where
\begin{align*}
a_4 &= \begin{cases}  \gamma_{\{35, 89\}}, & s = 3,\\   \gamma_{89}, & s \geqslant 4, \end{cases} \ \   
a_5 = \begin{cases}  \gamma_{\{44, 91\}}, & s = 3,\\   \gamma_{91}, & s \geqslant 4, \end{cases} \\   
a_6 &= \begin{cases}  \gamma_{\{75, 105\}}, & s = 3,\\   \gamma_{75}, & s \geqslant 4, \end{cases} \ \   
a_7 = \begin{cases}  \gamma_{\{82, 104\}}, & s = 3,\\   \gamma_{82}, & s \geqslant 4, \end{cases} \\    
a_8 &= \begin{cases}  \gamma_{\{44, 92, 99, 100\}}, & s = 3,\\   \gamma_{92}, & s \geqslant 4, \end{cases} \ \   
a_9 = \begin{cases}  \gamma_{\{93, 99, 101, 105\}}, & s = 3,\\    \gamma_{93}, & s \geqslant 4, \end{cases} \\   
a_{10} &= \begin{cases}  \gamma_{\{98, 100, 101, 104\}}, & s = 3,\\    \gamma_{98}, & s \geqslant 4. \end{cases}    
\end{align*}
These equalities imply
\begin{equation}\begin{cases}
a_i = 0, \ i = 4, 5, \ldots , 10,\\
\gamma_j = 0, \ j = 45, \ldots , 48,  76, \ldots , 81,\\ 83, \ldots , 88, 90, 94, \ldots , 97.
\end{cases}\tag{\ref{6.1.9}.3}
\end{equation}

With the aid of (\ref{6.1.9}.2) and (\ref{6.1.9}.3), the homomorphism $h$ sends (\ref{6.1.9}.1) to
\begin{align*} 
 &a_{11}w_{1,s,1} +  a_{12}w_{1,s,2} + \gamma_{99}w_{1,s,3} +  a_{13}w_{1,s,4}\\
&\qquad +   \gamma_{100}w_{1,s,5} +  \gamma_{101}w_{1,s,6} +  \gamma_{102}w_{1,s,7} =0,
\end{align*} 
where 
$$ 
a_{11} = \begin{cases}  \gamma_{92}, & s = 3, \\   \gamma_{103}, & s \geqslant 4, \end{cases}\ \ 
a_{12} = \begin{cases}  \gamma_{93}, & s = 3, \\   \gamma_{105}, & s \geqslant 4, \end{cases}\ \ 
a_{13} = \begin{cases}  \gamma_{98}, & s = 3, \\   \gamma_{104}, & s \geqslant 4.\end{cases} 
$$
From this, we get
\begin{equation}\begin{cases}
a_{11} = a_{12} = a_{13} = 0,\\
\gamma_j = 0, \ j = 99, 100, 101, 102.
\end{cases}\tag{\ref{6.1.9}.4}
\end{equation}

Combining (\ref{6.1.9}.2), \ref{6.1.9}.3) and (\ref{6.1.9}.4), we obtain  $\gamma_j = 0$ for $29 \leqslant j \leqslant 105$.
The proposition is proved.
\end{proof}

\begin{rems} The element $[\theta]$ defined as in the proof of Proposition \ref{6.1.5} is an $GL_4(\mathbb F_2)$-invariant in $(\mathbb F_2\underset {\mathcal A}\otimes R_4)_{9}$.

The $\mathbb F_2$-subspace of $(\mathbb F_2\underset{\mathcal A} \otimes P_4)_{21}$ generated by $[\theta_{1}], [\theta_{2}], [\theta_{3}], [\theta_{4}]$, which are defined as in the proof of Proposition \ref{6.1.7}, is an $GL_4(\mathbb F_2)$-submodule of $(\mathbb F_2\underset{\mathcal A} \otimes P_4)_{21}$.
\end{rems}

\subsection{The case $t=2$}\label{6.2}\

\medskip
Recall that from Kameko \cite{ka}, $\dim (\mathbb F_2\underset{\mathcal A}\otimes P_3)_{2^{s+3}+2^{s+1}-3} = 10$ with a basis consisting of all the classes $w_{2,s,i}, 1\leqslant i \leqslant 10,$ which are represented by the following monomials:

\medskip
\centerline{\begin{tabular}{ll}
$1.\  (2^s-1,2^s - 1,2^{s+3} - 1),$& $2.\  (2^s-1,2^{s+3} - 1,2^s - 1),$\cr 
$3.\  (2^{s+3}-1,2^s - 1,2^s - 1),$& $4.\  (2^s-1,2^{s+1} - 1,7.2^s - 1),$\cr 
$5.\  (2^{s+1}-1,2^s - 1,7.2^s - 1),$& $6.\  (2^{s+1}-1,7.2^s - 1,2^s - 1),$\cr 
$7.\  (2^{s+1}-1,3.2^s - 1,5.2^s - 1),$& $8.\  (2^{s+1}-1,2^{s+2} - 1,2^{s+2} - 1),$\cr 
$9.\  (2^{s+2}-1,2^{s+1} - 1,2^{s+2} - 1),$& $10.\  (2^{s+2}-1,2^{s+2} - 1,2^{s+1} - 1).$\cr
\end{tabular}}

\smallskip
Hence, we easily obtain

\begin{props}\label{6.2.3} $(\mathbb F_2\underset{\mathcal A}\otimes Q_4)_{2^{s+3} +2^{s+1}-3}$ is  an $\mathbb F_2$-vector space of dimension 40 with a basis consisting of all the  classes represented by the monomials $a_{2,s,j}, 1 \leqslant j \leqslant 40,$ which are determined as follows:

\smallskip
\centerline{\begin{tabular}{ll}
$1.\  (0,2^{s+1} - 1,2^{s+2} - 1,2^{s+2} - 1),$& $2.\  (0,2^{s+2} - 1,2^{s+1} - 1,2^{s+2} - 1),$\cr 
$3.\  (0,2^{s+2} - 1,2^{s+2} - 1,2^{s+1} - 1),$& $4.\  (2^{s+1}-1,0,2^{s+2} - 1,2^{s+2} - 1),$\cr 
$5.\  (2^{s+1}-1,2^{s+2} - 1,0,2^{s+2} - 1),$& $6.\  (2^{s+1}-1,2^{s+2} - 1,2^{s+2} - 1,0),$\cr 
$7.\  (2^{s+2}-1,0,2^{s+1} - 1,2^{s+2} - 1),$& $8.\  (2^{s+2}-1,0,2^{s+2} - 1,2^{s+1} - 1),$\cr 
$9.\  (2^{s+2}-1,2^{s+1} - 1,0,2^{s+2} - 1),$& $10.\  (2^{s+2}-1,2^{s+1} - 1,2^{s+2} - 1,0),$\cr 
$11.\  (2^{s+2}-1,2^{s+2} - 1,0,2^{s+1} - 1),$& $12.\  (2^{s+2}-1,2^{s+2} - 1,2^{s+1} - 1,0),$\cr 
$13.\  (0,2^s - 1,2^s - 1,2^{s+3} - 1),$& $14.\  (0,2^s - 1,2^{s+3} - 1,2^s - 1),$\cr 
$15.\  (0,2^{s+3} - 1,2^s - 1,2^s - 1),$& $16.\  (2^s-1,0,2^s - 1,2^{s+3} - 1),$\cr 
$17.\  (2^s-1,0,2^{s+3} - 1,2^s - 1),$& $18.\  (2^s-1,2^s - 1,0,2^{s+3} - 1),$\cr 
$19.\  (2^s-1,2^s - 1,2^{s+3} - 1,0),$& $20.\  (2^s-1,2^{s+3} - 1,0,2^s - 1),$\cr 
$21.\  (2^s-1,2^{s+3} - 1,2^s - 1,0),$& $22.\  (2^{s+3}-1,0,2^s - 1,2^s - 1),$\cr 
$23.\  (2^{s+3}-1,2^s - 1,0,2^s - 1),$& $24.\  (2^{s+3}-1,2^s - 1,2^s - 1,0),$\cr 
$25.\  (0,2^s - 1,2^{s+1} - 1,7.2^s - 1),$& $26.\  (0,2^{s+1} - 1,2^s - 1,7.2^s - 1),$\cr 
$27.\  (0,2^{s+1} - 1,7.2^s - 1,2^s - 1),$& $28.\  (2^s-1,0,2^{s+1} - 1,7.2^s - 1),$\cr 
$29.\  (2^s-1,2^{s+1} - 1,0,7.2^s - 1),$& $30.\  (2^s-1,2^{s+1} - 1,7.2^s - 1,0),$\cr 
$31.\  (2^{s+1}-1,0,2^s - 1,7.2^s - 1),$& $32.\  (2^{s+1}-1,0,7.2^s - 1,2^s - 1),$\cr 
$33.\  (2^{s+1}-1,2^s - 1,0,7.2^s - 1),$& $34.\  (2^{s+1}-1,2^s - 1,7.2^s - 1,0),$\cr 
$35.\  (2^{s+1}-1,7.2^s - 1,0,2^s - 1),$& $36.\  (2^{s+1}-1,7.2^s - 1,2^s - 1,0),$\cr 
$37.\  (0,2^{s+1} - 1,3.2^s - 1,5.2^s - 1),$& $38.\  (2^{s+1}-1,0,3.2^s - 1,5.2^s - 1),$\cr 
$39.\  (2^{s+1}-1,3.2^s - 1,0,5.2^s - 1),$& $40.\  (2^{s+1}-1,3.2^s - 1,5.2^s - 1,0).$\cr
\end{tabular}}
\end{props}

Now, we determine $(\mathbb F_2\underset{\mathcal A}\otimes R_4)_{2^{s+3} +2^{s+1}-3}$. 

Set $\mu_3(1) = 87, \mu_3(2) = 135$ and $\mu_3(s) =150$ for $s \geqslant 3$.

The main result of this subsection is

\begin{thms}\label{dlc6.2}  $(\mathbb F_2 \underset {\mathcal A} \otimes R_4)_{2^{s+3}+2^{s+1}-3}$ is an $\mathbb F_2$-vector space of dimension $\mu_2(s)-40$ with a basis consisting of all the classes represented by the  monomials $a_{2,s,j}, 41 \leqslant j \leqslant \mu_3(s),$ which are determined as follows:

\smallskip
For $s\geqslant 1$,

\medskip
\centerline{\begin{tabular}{ll}
$41.\  (1,2^{s+1} - 2,2^{s+2} - 1,2^{s+2} - 1),$& $42.\  (1,2^{s+2} - 1,2^{s+1} - 2,2^{s+2} - 1),$\cr 
$43.\  (1,2^{s+2} - 1,2^{s+2} - 1,2^{s+1} - 2),$& $44.\  (2^{s+2}-1,1,2^{s+1} - 2,2^{s+2} - 1),$\cr 
$45.\  (2^{s+2}-1,1,2^{s+2} - 1,2^{s+1} - 2),$& $46.\  (2^{s+2}-1,2^{s+2} - 1,1,2^{s+1} - 2),$\cr 
$47.\  (1,2^{s+1} - 1,2^{s+2} - 2,2^{s+2} - 1),$& $48.\  (1,2^{s+1} - 1,2^{s+2} - 1,2^{s+2} - 2),$\cr 
$49.\  (1,2^{s+2} - 2,2^{s+1} - 1,2^{s+2} - 1),$& $50.\  (1,2^{s+2} - 2,2^{s+2} - 1,2^{s+1} - 1),$\cr 
$51.\  (1,2^{s+2} - 1,2^{s+1} - 1,2^{s+2} - 2),$& $52.\  (1,2^{s+2} - 1,2^{s+2} - 2,2^{s+1} - 1),$\cr 
$53.\  (2^{s+1}-1,1,2^{s+2} - 2,2^{s+2} - 1),$& $54.\  (2^{s+1}-1,1,2^{s+2} - 1,2^{s+2} - 2),$\cr 
$55.\  (2^{s+1}-1,2^{s+2} - 1,1,2^{s+2} - 2),$& $56.\  (2^{s+2}-1,1,2^{s+1} - 1,2^{s+2} - 2),$\cr 
$57.\  (2^{s+2}-1,1,2^{s+2} - 2,2^{s+1} - 1),$& $58.\  (2^{s+2}-1,2^{s+1} - 1,1,2^{s+2} - 2),$\cr 
$59.\  (3,2^{s+2} - 3,2^{s+1} - 2,2^{s+2} - 1),$& $60.\  (3,2^{s+2} - 3,2^{s+2} - 1,2^{s+1} - 2),$\cr 
$61.\  (3,2^{s+2} - 1,2^{s+2} - 3,2^{s+1} - 2),$& $62.\  (2^{s+2}-1,3,2^{s+2} - 3,2^{s+1} - 2),$\cr 
$63.\  (3,2^{s+1} - 1,2^{s+2} - 3,2^{s+2} - 2),$& $64.\  (3,2^{s+2} - 3,2^{s+1} - 1,2^{s+2} - 2),$\cr 
$65.\  (3,2^{s+2} - 3,2^{s+2} - 2,2^{s+1} - 1),$& $66.\  (1,2^s - 1,2^s - 1,2^{s+3} - 2),$\cr 
$67.\  (1,2^s - 1,2^{s+3} - 2,2^s - 1),$& $68.\  (1,2^{s+3} - 2,2^s - 1,2^s - 1),$\cr 
$69.\  (1,2^s - 1,2^{s+1} - 2,7.2^s - 1),$& $70.\  (1,2^{s+1} - 2,2^s - 1,7.2^s - 1),$\cr 
$71.\  (1,2^{s+1} - 2,7.2^s - 1,2^s - 1),$& $72.\  (1,2^{s+1} - 2,3.2^s - 1,5.2^s - 1).$\cr 
\end{tabular}}

\medskip
For $s=1$,

\smallskip
\centerline{\begin{tabular}{lll}
$73.\ (3,3,4,7),$& $74.\ (3,3,7,4),$& $75.\ (3,7,3,4),$\cr 
$76.\ (7,3,3,4),$& $77.\ (1,1,3,12),$& $78.\ (1,3,1,12),$\cr 
$79.\ (1,3,12,1),$& $80.\ (3,1,1,12),$& $81.\ (3,1,12,1),$\cr 
$82.\ (1,3,4,9),$& $83.\ (3,1,4,9),$& $84.\ (1,3,5,8),$\cr 
$85.\ (3,1,5,8),$& $86.\ (3,5,1,8),$& $87.\ (3,5,8,1).$\cr 
\end{tabular}}

\medskip
For $s \geqslant 2$,

\smallskip
\centerline{\begin{tabular}{ll}
$73.\  (2^{s+1}-1,3,2^{s+2} - 3,2^{s+2} - 2),$& $74.\  (3,2^{s+1} - 3,2^{s+2} - 2,2^{s+2} - 1),$\cr 
$75.\  (3,2^{s+1} - 3,2^{s+2} - 1,2^{s+2} - 2),$& $76.\  (3,2^{s+2} - 1,2^{s+1} - 3,2^{s+2} - 2),$\cr 
$77.\  (2^{s+2}-1,3,2^{s+1} - 3,2^{s+2} - 2),$& $78.\  (7,2^{s+2} - 5,2^{s+1} - 3,2^{s+2} - 2),$\cr 
$79.\  (7,2^{s+2} - 5,2^{s+2} - 3,2^{s+1} - 2),$& $80.\  (1,2^s - 2,2^s - 1,2^{s+3} - 1),$\cr 
$81.\  (1,2^s - 2,2^{s+3} - 1,2^s - 1),$& $82.\  (1,2^s - 1,2^s - 2,2^{s+3} - 1),$\cr 
$83.\  (1,2^s - 1,2^{s+3} - 1,2^s - 2),$& $84.\  (1,2^{s+3} - 1,2^s - 2,2^s - 1),$\cr 
$85.\  (1,2^{s+3} - 1,2^s - 1,2^s - 2),$& $86.\  (2^s-1,1,2^s - 2,2^{s+3} - 1),$\cr 
$87.\  (2^s-1,1,2^{s+3} - 1,2^s - 2),$& $88.\  (2^s-1,2^{s+3} - 1,1,2^s - 2),$\cr 
$89.\  (2^{s+3}-1,1,2^s - 2,2^s - 1),$& $90.\  (2^{s+3}-1,1,2^s - 1,2^s - 2),$\cr 
$91.\  (2^{s+3}-1,2^s - 1,1,2^s - 2),$& $92.\  (2^s-1,1,2^s - 1,2^{s+3} - 2),$\cr 
$93.\  (2^s-1,1,2^{s+3} - 2,2^s - 1),$& $94.\  (2^s-1,2^s - 1,1,2^{s+3} - 2),$\cr 
$95.\  (1,2^s - 2,2^{s+1} - 1,7.2^s - 1),$& $96.\  (1,2^{s+1} - 1,2^s - 2,7.2^s - 1),$\cr 
\end{tabular}}
\centerline{\begin{tabular}{ll}
$97.\  (1,2^{s+1} - 1,7.2^s - 1,2^s - 2),$& $98.\  (2^{s+1}-1,1,2^s - 2,7.2^s - 1),$\cr 
$99.\  (2^{s+1}-1,1,7.2^s - 1,2^s - 2),$& $100.\  (2^{s+1}-1,7.2^s - 1,1,2^s - 2),$\cr 
$101.\  (2^s-1,1,2^{s+1} - 2,7.2^s - 1),$& $102.\  (1,2^s - 1,2^{s+1} - 1,7.2^s - 2),$\cr 
$103.\  (1,2^{s+1} - 1,2^s - 1,7.2^s - 2),$& $104.\  (1,2^{s+1} - 1,7.2^s - 2,2^s - 1),$\cr 
$105.\  (2^s-1,1,2^{s+1} - 1,7.2^s - 2),$& $106.\  (2^s-1,2^{s+1} - 1,1,7.2^s - 2),$\cr 
$107.\  (2^{s+1}-1,1,2^s - 1,7.2^s - 2),$& $108.\  (2^{s+1}-1,1,7.2^s - 2,2^s - 1),$\cr 
$109.\  (2^{s+1}-1,2^s - 1,1,7.2^s - 2),$& $110.\  (1,2^{s+1} - 1,3.2^s - 2,5.2^s - 1),$\cr 
$111.\  (2^{s+1}-1,1,3.2^s - 2,5.2^s - 1),$& $112.\  (1,2^{s+1} - 1,3.2^s - 1,5.2^s - 2),$\cr 
$113.\  (2^{s+1}-1,1,3.2^s - 1,5.2^s - 2),$& $114.\  (2^{s+1}-1,3.2^s - 1,1,5.2^s - 2),$\cr 
$115.\  (3,2^{s+1} - 3,2^s - 2,7.2^s - 1),$& $116.\  (3,2^{s+1} - 3,7.2^s - 1,2^s - 2),$\cr 
$117.\  (3,2^{s+1} - 1,7.2^s - 3,2^s - 2),$& $118.\  (2^{s+1}-1,3,7.2^s - 3,2^s - 2),$\cr 
$119.\  (3,2^{s+1} - 3,3.2^s - 2,5.2^s - 1),$& $120.\  (3,2^{s+1} - 3,3.2^s - 1,5.2^s - 2),$\cr 
$121.\  (3,2^{s+1} - 1,3.2^s - 3,5.2^s - 2),$& $122.\  (2^{s+1}-1,3,3.2^s - 3,5.2^s - 2),$\cr 
$123.\  (3,2^s - 1,2^{s+3} - 3,2^s - 2),$& $124.\  (3,2^{s+3} - 3,2^s - 2,2^s - 1),$\cr 
$125.\  (3,2^{s+3} - 3,2^s - 1,2^s - 2),$& $126.\  (3,2^s - 1,2^{s+1} - 3,7.2^s - 2),$\cr
$127.\  (3,2^{s+1} - 3,2^s - 1,7.2^s - 2),$& $128.\  (3,2^{s+1} - 3,7.2^s - 2,2^s - 1).$\cr 
\end{tabular}}

\medskip
For $s = 2$,

\smallskip
\centerline{\begin{tabular}{lll}
$129.\ (3,3,3,28),$& $130.\ (3,3,28,3),$& $131.\ (3,3,4,27),$\cr 
$132.\ (3,3,7,24),$& $133.\ (3,7,3,24),$& $134.\ (7,3,3,24),$\cr 
$135.\ (7,7,9,14).$& & \cr 
\end{tabular}}

\medskip
For $s \geqslant 3$,

\medskip
\centerline{\begin{tabular}{ll}
$129.\  (2^s-1,3,2^{s+1} - 3,7.2^s - 2),$& $130.\  (2^s-1,3,2^{s+3} - 3,2^s - 2),$\cr 
$131.\  (3,2^s - 3,2^s - 2,2^{s+3} - 1),$& $132.\  (3,2^s - 3,2^{s+3} - 1,2^s - 2),$\cr 
$133.\  (3,2^{s+3} - 1,2^s - 3,2^s - 2),$& $134.\  (2^{s+3}-1,3,2^s - 3,2^s - 2),$\cr 
$135.\  (3,2^s - 3,2^s - 1,2^{s+3} - 2),$& $136.\  (3,2^s - 3,2^{s+3} - 2,2^s - 1),$\cr 
$137.\  (3,2^s - 1,2^s - 3,2^{s+3} - 2),$& $138.\  (2^s-1,3,2^s - 3,2^{s+3} - 2),$\cr 
$139.\  (3,2^s - 3,2^{s+1} - 2,7.2^s - 1),$& $140.\  (3,2^s - 3,2^{s+1} - 1,7.2^s - 2),$\cr 
$141.\  (3,2^{s+1} - 1,2^s - 3,7.2^s - 2),$& $142.\  (2^{s+1}-1,3,2^s - 3,7.2^s - 2),$\cr 
$143.\  (7,2^{s+3} - 5,2^s - 3,2^s - 2),$& $144.\  (7,2^{s+1} - 5,2^s - 3,7.2^s - 2),$\cr 
$145.\  (7,2^{s+1} - 5,3.2^s - 3,5.2^s - 2),$& $146.\  (7,2^{s+1} - 5,2^{s+2} - 3,2^{s+2} - 2),$\cr 
$147.\  (7,2^{s+1} - 5,7.2^s - 3,2^s - 2).$& \cr
\end{tabular}}

\medskip
For $s=3$,

\smallskip
\centerline{\begin{tabular}{lll}
$148.\ (7,7,7,56),$& $149.\ (7,7,9,54),$& $150.\ (7,7,57,6).$\cr
\end{tabular}}

\medskip
For $s \geqslant 4$,

\medskip
\centerline{\begin{tabular}{ll}
$148.\  (7,2^s - 5,2^s - 3,2^{s+3} - 2),$& $149.\  (7,2^s - 5,2^{s+1} - 3,7.2^s - 2),$\cr 
$150.\  (7,2^s - 5,2^{s+3} - 3,2^s - 2).$& \cr 
\end{tabular}}
\end{thms}

We begin the proof of this theorem with the following proposition.

\begin{props}\label{mdc6.2} The $\mathbb F_2$-vector space $(\mathbb F_2\underset {\mathcal A}\otimes R_4)_{2^{s+3}+2^{s+1}-3}$ is generated by the $\mu_3(s)-40$ elements listed in Theorem \ref{dlc6.2}.
\end{props}

The proof of the proposition is based on some lemmas.

\begin{lems}\label{6.2.1} The following matrices are strictly inadmissible
$$ \begin{pmatrix} 1&0&1&1\\ 1&0&0&0\\ 0&1&0&0\\ 0&1&0&0\end{pmatrix} \quad  \begin{pmatrix} 1&0&1&1\\ 1&0&0&0\\ 0&1&0&0\\ 0&0&0&1\end{pmatrix} \quad  
\begin{pmatrix} 1&0&1&1\\ 1&0&0&0\\ 0&1&0&0\\ 0&0&1&0\end{pmatrix}$$  
$$ \begin{pmatrix} 1&0&1&1\\ 1&0&1&1\\ 0&1&0&1\end{pmatrix}  \quad    
\begin{pmatrix} 1&0&1&1\\ 1&0&1&1\\ 0&1&1&0\end{pmatrix} \quad  
\begin{pmatrix} 1&1&0&1\\ 1&1&0&1\\ 0&1&1&0\end{pmatrix} \quad  
\begin{pmatrix} 1&1&0&1\\ 1&1&0&1\\ 1&0&1&0\end{pmatrix}.  $$ 
\end{lems}

\begin{proof} The monomials corresponding to the above matrices respectively are 
(3,12,1,1),  (3,4,1,9),   (3,4,9,1), (3,4,3,7),  (3,4,7,3),  (3,7,4,3), (7,3,4,3).

By a direct computation, we have
\begin{align*}
(3,12,1,1) &= Sq^1(3,11,1,1) + Sq^2\big((5,7,2,1) + (2,11,1,1) \\ 
&\quad + (5,7,1,2)\big) + Sq^4\big((3,7,2,1) + (3,7,1,2)\big)  + (3,9,4,1)\\ 
&\quad+(3,9,1,4) +(2,13,1,1)\quad  \text{mod }\mathcal L_4(3;1;1;1),\\
(3,4,9,1) &= Sq^1(3,1,11,1) + Sq^2\big((5,2,7,1) + (2,1,11,1) \\ 
&\quad+ (5,1,7,2)\big) + Sq^4\big((3,2,7,1) + (3,1,7,2)\big) + (3,1,12,1)\\ 
&\quad+(2,1,13,1)+(3,1,9,4)\quad  \text{mod }\mathcal L_4(3;1;1;1),\\
(3,4,1,9) &= \overline{\varphi}_3(3,4,9,1).
\end{align*}

If $x$ is one of the four monomials (3,4,3,7),  (3,4,7,3),  (3,7,4,3), (7,3,4,3) then there is a homomorphism  $\psi:P_4\to P_4$  induced by a permutation of $\{x_1,x_2,x_3,x_4\}$ such that $x=\psi (3,4,3,7)$. 

A direct computation shows
\begin{align*}
(3,4,3,7) &= Sq^1(3,3,3,7) + Sq^2(2,3,3,7) + (3,3,4,7)\\ 
&\quad + (2,5,3,7)+(2,3,5,7)\quad  \text{mod }\mathcal L_4(3;3;2).
\end{align*}

So, the lemma is proved.
\end{proof}

\begin{lems}\label{6.2.2}  The following matrices are strictly inadmissible
$$\begin{pmatrix} 1&1&0&1\\ 1&1&0&1\\ 0&1&0&0\\ 0&0&1&0\\ 0&0&1&0\end{pmatrix} \quad \begin{pmatrix} 1&1&0&1\\ 1&1&0&1\\ 1&0&0&0\\ 0&0&1&0\\ 0&0&1&0\end{pmatrix} \quad \begin{pmatrix} 1&1&1&0\\ 1&1&0&1\\ 1&0&0&0\\ 0&1&0&0\\ 0&0&1&0\end{pmatrix}$$    
$$  \begin{pmatrix} 1&1&0&1\\ 1&1&0&1\\ 0&1&0&0\\ 0&0&1&0\\ 0&0&0&1\end{pmatrix}\quad \begin{pmatrix} 1&1&0&1\\ 1&1&0&1\\ 1&0&0&0\\ 0&0&1&0\\ 0&0&0&1\end{pmatrix} \quad \begin{pmatrix} 1&1&1&0\\ 1&1&1&0\\ 0&1&0&0\\ 0&0&1&0\\ 0&0&0&1\end{pmatrix} $$    
$$ \begin{pmatrix} 1&1&1&0\\ 1&1&1&0\\ 1&0&0&0\\ 0&0&1&0\\ 0&0&0&1\end{pmatrix} \quad \begin{pmatrix} 1&1&1&0\\ 1&1&1&0\\ 1&0&0&0\\ 0&1&0&0\\ 0&0&0&1\end{pmatrix} \quad \begin{pmatrix} 1&1&0&1\\ 1&1&0&1\\ 1&0&0&0\\ 0&1&0&0\\ 0&0&1&0\end{pmatrix}$$    
$$ \begin{pmatrix} 1&1&0&1\\ 1&1&0&1\\ 1&1&0&1\\ 0&0&1&1\end{pmatrix} \quad \begin{pmatrix} 1&1&1&0\\ 1&1&1&0\\ 1&1&1&0\\ 0&0&1&1\end{pmatrix} \quad \begin{pmatrix} 1&1&1&0\\ 1&1&1&0\\ 1&1&1&0\\ 0&1&0&1\end{pmatrix}\quad  \begin{pmatrix} 1&1&1&0\\ 1&1&1&0\\ 1&1&1&0\\ 1&0&0&1\end{pmatrix}.  $$ 
\end{lems}

\begin{proof}  The monomials corresponding to the above matrices respectively are 
\begin{align*}&(3,7,24,3),  (7,3,24,3),   (7,11,17,2),  (3,7,8,19),  (7,3,8,19),\\  &(3,7,11,16),  (7,3,11,16),  (7,11,3,16),  (7,11,16,3),  (7,7,8,15),\\
&(7,7,15,8),  (7,15,7,8),  (15,7,7,8).
\end{align*}
By a direct computation, we have
\begin{align*}
&(7,3,24,3) = Sq^1(7,3,23,3)+ Sq^2(7,2,23,3) + Sq^4\big((4,3,23,3)\\ 
&\quad + (5,2,23,3) + (11,2,15,5) +(11,3,15,4) \big)+ Sq^8\big((7,3,15,4)\\ 
&\quad + (7,2,15,5)\big)+ (4,3,27,3)+(5,2,27,3)+(7,2,25,3)\\ 
&\quad  + (7,2,19,9) + (7,3,19,8) \quad  \text{mod }\mathcal L_4(3;3;1;1;1),\\
&(3,7,24,3) = \overline{\varphi}_1(7,3,24,3),
\end{align*}
where $\overline{\varphi}_1$ is defined as in Section \ref{2}.
\begin{align*}
&(7,11,17,2) = Sq^1(7,7,19,3) + Sq^2(7,7,19,2) + Sq^4\big((11,7,12,3)\\ 
&\quad + (11,7,13,2)  +(4,7,19,3) + (5,7,19,2) \big) + Sq^8\big((7,7,12,3)\\ 
&\quad+ (7,7,13,2)\big)+(4,11,19,3)+ (5,11,19,2) + (7,9,19,2)\\ 
&\quad +(7,8,19,3) + (7,11,16,3) \quad  \text{mod }\mathcal L_4(3;3;1;1;1),\\
&(7,3,8,19) = Sq^1(7,3,3,23) + Sq^2(7,2,3,23) + Sq^4\big((11,3,4,15)\\ 
&\quad + (4,3,3,23) +(5,2,3,23) + (11,2,5,15) \big) + Sq^8\big((7,3,4,15)\\ 
&\quad+(7,2,5,15)\big) + (7,3,3,24)+(4,3,3,27)+(7,2,3,25)\\ 
&\quad +(5,2,3,27)+ (7,2,9,19) \quad  \text{mod }\mathcal L_4(3;3;1;1;1),\\
&(3,7,8,19) = \overline{\varphi}_1(7,3,8,19).
\end{align*}

If $x$ is one of the  monomials (3,7,11,16),  (7,3,11,16),  (7,11,3,16),  then there is a homomorphism  $\psi:P_4\to P_4$  induced by a permutation of $\{x_1,x_2,x_3,x_4\}$ such that $x=\psi(7,11,3,6)$. We have
\begin{align*}
&(7,11,3,16) = Sq^1\big((7,7,3,19)+(7,7,1,21)\big) + Sq^2(7,7,2,19) \\ 
&\quad+ Sq^4\big((11,7,3,12) + (4,7,3,19) +(5,7,2,19)  +(11,7,1,14)\big)\\ 
&\quad + Sq^8\big((7,7,3,12) +(7,7,1,14)\big) + (4,11,3,19) +(5,11,2,19)\\ 
&\quad + (7,8,3,19)  +(7,9,2,19) + (7,11,1,18)\quad  \text{mod }\mathcal L_4(3;3;1;1;1),\\
&(7,11,16,3) = Sq^1(11,19,3,3) + Sq^2\big((10,19,3,3) + (7,22,3,3)\big)\\ 
&\quad + Sq^4\big((7,19,4,3) + (7,19,3,4)  + (6,21,3,3)+ (6,19,5,3)\\ 
&\quad+ (6,19,3,5) + (7,20,3,3)  + (5,22,3,3)+ (11,14,5,3) \\  
&\quad+ (11,14,3,5) \big)+ Sq^8\big((7,11,8,3) + (7,11,3,8) + (7,14,5,3)\\ 
&\quad + (7,14,3,5)  + (7,10,9,3) + (7,10,3,9)\big)+(7,11,3,16)\\ 
&\quad +(6,25,3,3) + (6,19,9,3) + (6,19,3,9) + (5,26,3,3)\\ 
&\quad + (7,10,17,3) + (7,10,3,17)\quad  \text{mod }\mathcal L_4(3;3;1;1;1).
\end{align*}

If $x$ is one of the  monomials (7,7,8,15),  (7,7,15,8),  (7,15,7,8), (15,7,7,8),  then there is a homomorphism  $\psi:P_4\to P_4$  induced by a permutation of $\{x_1,x_2,x_3,x_4\}$ such that $x=\psi(7,7,8,15)$. We have
\begin{align*}
(7,7,8,15) &= Sq^1(7,7,7,15) + Sq^2(7,6,7,15) + Sq^4\big((4,7,7,15)\\ 
&\quad +(5,6,7,15)\big) + (4,11,7,15) + (4,7,11,15) + (7,6,9,15)\\ 
&\quad+(5,10,7,15)+(5,6,11,15) \quad  \text{mod }\mathcal L_4(3;3;3;2).
\end{align*}

The lemma follows.
\end{proof}

Combining Theorem \ref{2.4}, the lemmas in Sections \ref{3}, \ref{5} and Lemmas  \ref{6.1.1}, \ref{6.1.2}, \ref{6.1.3}, \ref{6.2.1}, \ref{6.2.2}, we obtain Proposition \ref{mdc6.2}.

\medskip
Now, we show that the elements $[a_{2,s,j}], 41\leqslant j\leqslant \mu_3(s)$, are linearly independent.

\begin{props}\label{6.2.4} The elements $[a_{2,1,j}], \ 41 \leqslant j \leqslant 87,$ are linearly independent in $(\mathbb F_2\underset {\mathcal A}\otimes R_4)_{17}$.
\end{props}

\begin{proof} Suppose there is a linear relation
\begin{equation}\sum_{j=41}^{87}\gamma_j[a_{2,1,j}]=0,\tag{\ref{6.2.4}.1}
\end{equation}
where $\gamma_j \in \mathbb F_2.$

Applying the homomorphisms $f_i, i=1,2,3,$ to the relation (\ref{6.2.4}.1), we get
\begin{align*}
&\gamma_{68}[15,1,1]  +   \gamma_{70}[3,1,13] +    \gamma_{71}[3,13,1] \\
&\quad  +   \gamma_{72}[3,5,9]  +   \gamma_{41}[3,7,7]  +   \gamma_{49}[7,3,7]  +   \gamma_{50}[7,7,3]  =0,\\   
&\gamma_{\{67, 81\}}[15,1,1]  +  \gamma_{\{44, 53, 57, 69, 83\}}[3,1,13]  +   \gamma_{\{79, 87\}}[3,13,1]  \\
&\quad +   \gamma_{\{59, 65, 82\}}[3,5,9] +   \gamma_{42}[3,7,7]  +   \gamma_{\{47, 73\}}[7,3,7]  +   \gamma_{52}[7,7,3]  =0,\\   
&\gamma_{\{66, 80\}}[15,1,1]  +  a_1[3,1,13]  +  a_2[3,13,1]  +   a_3[3,5,9]\\
&\quad  +   \gamma_{43}[3,7,7]  +   \gamma_{\{48, 74\}}[7,3,7]  +   \gamma_{\{51, 75\}}[7,7,3] =0,       
\end{align*}
where
$a_1 =  \gamma_{\{45, 54, 56, 77, 85\}},\
a_2 = \gamma_{\{46, 55, 58, 78, 86\}},\ 
a_3 = \gamma_{\{60, 61, 62, 63, 64, 76, 84\}}.
$

Computing from the above equalities gives
\begin{equation}\begin{cases} 
a_1 = a_2 = a_3 =0,\\  \gamma_j = 0, \ j = 41, 42, 43, 49, 50, 52, 68, 70, 71, 72,\\
\gamma_{\{67, 81\}} = 
\gamma_{\{44, 53, 57, 69, 83\}} =    
\gamma_{\{79, 87\}} =   
\gamma_{\{59, 65, 82\}} = 0,\\  
\gamma_{\{47, 73\}} =   \gamma_{\{66, 80\}} =   \gamma_{\{48, 74\}} =   \gamma_{\{51, 75\}} = 0.      
\end{cases}\tag{\ref{6.2.4}.2}
\end{equation}

With the aid of (\ref{6.2.4}.2), the homomorphisms $f_4, f_5, f_6$ send (\ref{6.2.4}.1) to
\begin{align*}
&\gamma_{\{67, 79\}}[1,15,1]  +  \gamma_{\{47, 69, 82\}}[1,3,13]  +    \gamma_{\{81, 87\}}[3,13,1] \\
&\quad +   \gamma_{\{65, 83\}}[3,5,9]   +   \gamma_{\{53, 59, 73\}}[3,7,7]  +   \gamma_{44}[7,3,7]  +   \gamma_{57}[7,7,3]  =0,\\
&\gamma_{\{66, 78\}}[1,15,1]  +  \gamma_{\{48, 51, 77, 84\}}[1,3,13]   +   a_4[3,13,1]  +   a_5[3,5,9]\\
&\quad  +   \gamma_{\{54, 60, 74\}}[3,7,7]   +   \gamma_{45}[7,3,7]  +   \gamma_{\{56, 76\}}[7,7,3]  =0,\\  
&\gamma_{\{66, 67, 69, 77\}}[1,1,15]  +   a_6[1,3,13]  +  a_7[3,1,13]   +  a_8[3,5,9]\\
&\quad  +   \gamma_{\{55, 61, 75\}}[3,7,7]   +   \gamma_{\{58, 62, 76\}}[7,3,7]  +   \gamma_{46}[7,7,3]  =0,           
\end{align*}
where
\begin{align*}
a_4 &= \gamma_{\{46, 55, 58, 80, 86\}},\ \
a_5 = \gamma_{\{61, 62, 63, 64, 75, 85\}},\\
a_6 &= \gamma_{\{47, 48, 51, 78, 79, 82, 84\}},\ \
a_7 =  \gamma_{\{44, 45, 53, 54, 56, 57, 80, 81, 83, 85\}},\\
a_8 &=  \gamma_{\{59, 60, 63, 64, 65, 73, 74, 86, 87\}}.
\end{align*}
Computing directly from  these equalities and using (\ref{6.2.4}.2), we get
\begin{equation}\begin{cases}
a_4 = a_5 = a_6 = a_7 = a_8 = 0,\\  \gamma_j = 0, \ j = 44, 45, 46, 57,\\
\gamma_{\{67, 79\}} =  
\gamma_{\{47, 69, 82\}} =    
\gamma_{\{81, 87\}} =   
\gamma_{\{65, 83\}} =  0,\\ 
\gamma_{\{53, 59, 73\}} =   
\gamma_{\{66, 78\}} =   
\gamma_{\{48, 51, 77, 84\}} =   
\gamma_{\{54, 60, 74\}} =  0,\\ 
\gamma_{\{56, 76\}} =   
\gamma_{\{66, 67, 69, 77\}} = 
\gamma_{\{55, 61, 75\}} =   
\gamma_{\{58, 62, 76\}} = 0.   
\end{cases}\tag{\ref{6.2.4}.3}
\end{equation}

With the aid of (\ref{6.2.4}.2) and (\ref{6.2.4}.3), the homomorphisms $g_1, g_2$ send (\ref{6.2.4}.1) to
\begin{align*}
&\gamma_{\{53, 65, 69\}}[1,3,13] +  \gamma_{\{47, 59, 82\}}[3,5,9]\\
&\quad +   \gamma_{53}[3,7,7] +  \gamma_{59}[7,3,7] +  \gamma_{65}[7,7,3] =0,\\  
& \gamma_{\{48, 51, 56, 60, 64, 84, 85\}}[3,5,9] +  \gamma_{54}[3,7,7]\\
&\quad +  \gamma_{\{54, 56, 77, 85\}}[1,3,13] +\gamma_{60}[7,3,7] +  \gamma_{64}[7,7,3] =0.       
\end{align*}

From the above equalities, we get
\begin{equation}
\begin{cases}
\gamma_j = 0,\ j = 53, 54, 59, 60, 64, 65,\\
\gamma_{\{53, 65, 69\}} =   
\gamma_{\{54, 56, 77, 85\}} = 0,\\  
\gamma_{\{47, 59, 82\}} =  
\gamma_{\{48, 51, 56, 60, 64, 84, 85\}} = 0.\\ 
\end{cases} \tag{\ref{6.2.4}.4}
\end{equation}  

With the aid of (\ref{6.2.4}.2), (\ref{6.2.4}.3) and (\ref{6.2.4}.4), the homomorphisms $g_3, g_4$ send (\ref{6.2.4}.1) respectively to
\begin{align*}
&\gamma_{\{55, 58, 66, 86\}}[1,3,13]  +  \gamma_{\{67, 79\}}[3,1,13] +    \gamma_{\{66, 67, 69, 77\}}[3,13,1]\\
&\quad  +   a_9[3,5,9]  +   \gamma_{55}[3,7,7]  +   \gamma_{61}[7,3,7]  +   \gamma_{\{56, 58, 62, 63\}}[7,7,3] =0,\\
&\gamma_{\{55, 58, 66, 86\}}[1,3,13]  +  \gamma_{\{67, 79\}}[3,1,13]   +   \gamma_{\{66, 67, 69, 77\}}[3,13,1]\\
&\quad  +  a_{10}[3,5,9]  +   \gamma_{58}[3,7,7]  +   \gamma_{62}[7,3,7]  +   a_{11}[7,7,3]  =0,   
\end{align*}
where
\begin{align*}
a_9 &= \gamma_{\{47, 48, 51, 56, 61, 62, 63, 79, 82, 84, 86\}},\\
a_{10} &=  \gamma_{\{47, 48, 51, 56, 61, 62, 63, 79, 85, 86\}},\ \
a_{11} = \gamma_{\{47, 48, 51, 55, 61, 63\}}.
\end{align*}
Hence, we get
\begin{equation}\begin{cases}
a_9 = a_{10} = a_{11} = 0,\  
\gamma_j = 0,\ j = 55, 58, 61, 62,\\
\gamma_{\{55, 58, 66, 86\}} = 
\gamma_{\{67, 79\}} =  
\gamma_{\{66, 67, 69, 77\}} =   
\gamma_{\{56, 58, 62, 63\}} = 0.   
\end{cases}\tag{\ref{6.2.4}.5}
\end{equation}
    
Combining (\ref{6.2.4}.2), (\ref{6.2.4}.3), (\ref{6.2.4}.4) and (\ref{6.2.4}.5), we obtain
\begin{equation}\begin{cases}
\gamma_j = 0,\ j \ne 66, 67, 78, 79, 80, 81, 86, 87,\\
\gamma_{66} = \gamma_{67} = \gamma_{78} = \gamma_{79} = \gamma_{80} = \gamma_{81} = \gamma_{86} = \gamma_{87}.
\end{cases}\tag{\ref{6.2.4}.6}
\end{equation} 
Substituting (\ref{6.2.4}.6) into the relation (\ref{6.2.4}.1), we have
\begin{equation}\gamma_{66}[\theta] = 0, \tag{\ref{6.2.4}.7}
\end{equation} 
where $\theta = a_{2,1,66} + a_{2,1,67} + a_{2,1,78} + a_{2,1,79} + a_{2,1,80} + a_{2,1,81} + a_{2,1,86} + a_{2,1,87}.$ 

By a same argument as given in the proof of Proposition \ref{6.1.7}, we see that the polynomial $\theta$  is non-hit. Hence, we get $[\theta] \ne 0$ and the relation (\ref{6.2.4}.7) implies $\gamma_{66} = 0.$ The proposition is proved.
\end{proof}

\begin{props}\label{6.2.6} The elements $[a_{2,2,j}], \ 41 \leqslant j \leqslant 135,$ are linearly independent in $(\mathbb F_2\underset {\mathcal A}\otimes R_4)_{37}$.
\end{props}

\begin{proof} Suppose there is a linear relation
\begin{equation}\sum_{j=41}^{135}\gamma_j[a_{2,2,j}]=0,\tag{\ref{6.2.6}.1}
\end{equation}
where $\gamma_j \in \mathbb F_2.$ 

Apply the homomorphisms $f_i, 1 \leqslant i \leqslant 6,$ to the relation (\ref{6.2.6}.1) and we get
\begin{align*}
&\gamma_{80}[3,3,31] +  \gamma_{81}[3,31,3] +   \gamma_{68}[31,3,3] +  \gamma_{95}[3,7,27]\\
&\quad +  \gamma_{70}[7,3,27] +  \gamma_{71}[7,27,3] +  \gamma_{72}[7,11,19]\\
&\quad +  \gamma_{41}[7,15,15] +  \gamma_{49}[15,7,15] +  \gamma_{50}[15,15,7] =0,\\
&\gamma_{82}[3,3,31] +  \gamma_{84}[3,31,3] +  \gamma_{\{67, 130\}}[31,3,3] +  \gamma_{96}[3,7,27]\\
&\quad +  \gamma_{\{69, 131\}}[7,3,27] +  \gamma_{104}[7,27,3] +  \gamma_{110}[7,11,19]\\
&\quad +  \gamma_{42}[7,15,15] +  \gamma_{47}[15,7,15] +  \gamma_{52}[15,15,7] = 0,\\  
&\gamma_{83}[3,3,31] +  \gamma_{85}[3,31,3] +  \gamma_{\{66, 129, 134\}}[31,3,3] +  \gamma_{97}[3,7,27]\\
&\quad +  \gamma_{\{102, 132\}}[7,3,27] +  \gamma_{\{103, 133\}}[7,27,3] +  \gamma_{112}[7,11,19]\\
&\quad +  \gamma_{43}[7,15,15] +  \gamma_{48}[15,7,15] +  \gamma_{51}[15,15,7] = 0,\\ 
&\gamma_{86}[3,3,31]  +  \gamma_{\{93, 124, 128, 130\}}[3,31,3]  +    \gamma_{89}[31,3,3]\\
&\quad  +   a_1[3,7,27]  +   \gamma_{98}[7,3,27]  +   \gamma_{108}[7,27,3]  +   \gamma_{111}[7,11,19]\\
&\quad  +   \gamma_{53}[7,15,15]  +   \gamma_{44}[15,7,15]  +   \gamma_{57}[15,15,7]  = 0,\\   
&\gamma_{87}[3,3,31]  +  \gamma_{\{92, 125, 127, 129, 133\}}[3,31,3]  +   \gamma_{90}[31,3,3]\\
&\quad  +   a_2[3,7,27]  +   \gamma_{99}[7,3,27]  +   \gamma_{\{107, 134\}}[7,27,3]  +   \gamma_{113}[7,11,19]\\
&\quad   +   \gamma_{54}[7,15,15]  +   \gamma_{45}[15,7,15]  +   \gamma_{56}[15,15,7]  = 0,\\   
&a_3[3,3,31]  +   \gamma_{88}[3,31,3]  +   \gamma_{91}[31,3,3]  +   a_4[3,7,27] \\
&\quad +   a_5[7,3,27]  +   \gamma_{100}[7,27,3]  +   \gamma_{\{78, 79, 114, 135\}}[7,11,19]\\
&\quad  +   \gamma_{55}[7,15,15]  +   \gamma_{58}[15,7,15]  +   \gamma_{46}[15,15,7] =0,  
\end{align*}
where
\begin{align*}
a_1 &= \gamma_{\{59, 65, 74, 101, 115, 119, 131\}},\ \
a_2 = \gamma_{\{60, 64, 75, 105, 116, 120, 132\}},\\
a_3 &= \gamma_{\{94, 123, 126, 129, 130, 131, 132\}},\ \
a_4 = \gamma_{\{61, 63, 76, 106, 117, 121, 133\}},\\
a_5 &= \gamma_{\{62, 73, 77, 109, 118, 122, 134\}}.
\end{align*}

Computing from the above equalities gives
\begin{equation}\begin{cases} 
\gamma_j = 0,\ j= 41, \ldots , 57, 58, 68, 70, 71, 72,
 80, \ldots , 91,\\ 
95, \ldots , 100, 104, 108, 110, 111, 112, 113,\\
\gamma_{\{67, 130\}} =   
\gamma_{\{69, 131\}} = 
\gamma_{\{66, 129, 134\}} =   
\gamma_{\{102, 132\}} = 0,\\
\gamma_{\{103, 133\}} =    
\gamma_{\{93, 124, 128, 130\}} =   
\gamma_{\{92, 125, 127, 129, 133\}} = 0,\\  
\gamma_{\{107, 134\}} = 
\gamma_{\{78, 79, 114, 135\}} = 0.
\end{cases}\tag{\ref{6.2.6}.2}
\end{equation}

With the aid of (\ref{6.2.6}.2), the homomorphisms $g_1, g_2$ send (\ref{6.2.6}.1) to
\begin{align*}
&\gamma_{93}[3,31,3] +  \gamma_{124}[31,3,3] +   \gamma_{101}[3,7,27]\\
&\quad +  \gamma_{115}[7,3,27] +  \gamma_{128}[7,27,3] +  \gamma_{119}[7,11,19]\\
&\quad +  \gamma_{74}[7,15,15] +  \gamma_{59}[15,7,15] +  \gamma_{65}[15,15,7] =0,\\
&\gamma_{\{66, 92, 129\}}[3,31,3] +  \gamma_{125}[31,3,3] +  \gamma_{105}[3,7,27]\\
&\quad +  \gamma_{116}[7,3,27] +  \gamma_{127}[7,27,3] +  \gamma_{120}[7,11,19] \\
&\quad+  \gamma_{75}[7,15,15] +  \gamma_{60}[15,7,15] +  \gamma_{64}[15,15,7] =0.    
\end{align*}

From the above equalities, we get
\begin{equation}
\begin{cases}
\gamma_j = 0, \ j = 59, 60, 64, 65, 74, 75, 93, 101, \\
105, 115, 116, 119, 120, 124,125, 127, 128,\\
\gamma_{\{66, 92, 129\}} = 0.
\end{cases} \tag{\ref{6.2.6}.3}
\end{equation}  

With the aid of (\ref{6.2.6}.2) and (\ref{6.2.6}.3), the homomorphisms $g_3, g_4$ send (\ref{6.2.6}.1) respectively to
\begin{align*}
&\gamma_{\{66, 94, 107, 109, 129\}}[3,31,3]  + \gamma_{\{118, 123\}}[31,3,3]  +    \gamma_{\{106, 114\}}[3,7,27]\\
&\quad  +     \gamma_{\{107, 109, 118, 126\}}[7,27,3]  +  \gamma_{117}[7,3,27]  +  \gamma_{\{114, 121, 135\}}[7,11,19] \\
&\quad +   \gamma_{76}[7,15,15]  +   \gamma_{61}[15,7,15]  +   \gamma_{\{63, 135\}}[15,15,7]  =0,\\   
&\gamma_{\{92, 94, 103, 106, 129\}}[3,31,3]  +  \gamma_{\{67, 117, 123\}}[31,3,3]\\
&\quad  +   \gamma_{118}[7,3,27]  +   \gamma_{\{69, 102, 103, 106, 117, 126\}}[7,27,3]\\
&\quad + \gamma_{\{109, 114\}}[3,7,27]  +     \gamma_{\{78, 79, 114, 122, 135\}}[7,11,19]\\
&\quad  +   \gamma_{77}[7,15,15]  +   \gamma_{62}[15,7,15]  +   \gamma_{\{73, 78, 79, 135\}}[15,15,7]  =0.     
\end{align*}

Hence, we get
\begin{equation}\begin{cases}
\gamma_j = 0,\ j =  59, 60, 64, 61, 62, 
76, 77, 117, 118,\\
\gamma_{\{66, 94, 107, 109, 129\}} =   
\gamma_{\{118, 123\}} = 
\gamma_{\{106, 114\}} =  0,\\ 
\gamma_{\{114, 121, 135\}} =   
\gamma_{\{107, 109, 118, 126\}} =  
\gamma_{\{63, 135\}} = 0,\\  
\gamma_{\{92, 94, 103, 106, 129\}} =  
\gamma_{\{109, 114\}} =   
\gamma_{\{69, 102, 103, 106, 117, 126\}} = 0,\\  
\gamma_{\{67, 117, 123\}} =   
\gamma_{\{78, 79, 114, 122, 135\}} =   
\gamma_{\{73, 78, 79, 135\}} = 0. 
\end{cases}\tag{\ref{6.2.6}.4}
\end{equation}
    
With the aid of (\ref{6.2.6}.2), (\ref{6.2.6}.3),  and (\ref{6.2.6}.4), the homomorphism $h$ sends (\ref{6.2.6}.1) to
\begin{align*}
&\gamma_{\{66, 92, 94, 106, 129\}}[3,3,31] + \gamma_{\{106, 126\}}[3,7,27] +   \gamma_{\{106, 121, 122\}}[7,11,19]\\
&\quad +  \gamma_{73}[7,15,15] +  \gamma_{78}[15,7,15] +  \gamma_{79}[15,15,7] =0. 
\end{align*}

From this equality, we get
\begin{equation}\begin{cases}
\gamma_{73} = \gamma_{78} = \gamma_{79} =  0,\\
\gamma_{\{66, 92, 94, 106, 129\}}=
\gamma_{\{106, 126\}} = \gamma_{\{106, 121, 122\}} = 0.
\end{cases}\tag{\ref{6.2.6}.5} 
\end{equation}

Combining (\ref{6.2.6}.2), (\ref{6.2.6}.3), (\ref{6.2.6}.4) and (\ref{6.2.6}.5), we obtain $\gamma_j =0$ for any $j$. The proposition is proved.
\end{proof}

\begin{props}\label{6.2.7} For any $s \geqslant 3$, the elements $[a_{2,s,j}], \ 41 \leqslant j \leqslant 150,$ are linearly independent in $(\mathbb F_2\underset {\mathcal A}\otimes R_4)_{2^{s+3} +2^{s+1} - 3}$.
\end{props}

\begin{proof} Suppose there is a linear relation
\begin{equation}\sum_{j=41}^{150}\gamma_j[a_{2,s,j}]=0,\tag{\ref{6.2.7}.1}
\end{equation}
where $\gamma_j \in \mathbb F_2.$ 

Applying the homomorphisms $f_i, i=1,2, \ldots, 6,$ to  (\ref{6.2.7}.1), we get
\begin{align*}
&\gamma_{80}w_{2,s,1} +  \gamma_{81}w_{2,s,2} +   \gamma_{68}w_{2,s,3} +  \gamma_{95}w_{2,s,4} +  \gamma_{70}w_{2,s,5}\\
&\quad +  \gamma_{71}w_{2,s,6} +  \gamma_{72}w_{2,s,7} +  \gamma_{41}w_{2,s,8} +  \gamma_{49}w_{2,s,9} +  \gamma_{50}w_{2,s,10} =0,\\
&\gamma_{82}w_{2,s,1} + \gamma_{84}w_{2,s,2} +  \gamma_{67}w_{2,s,3} +  \gamma_{96}w_{2,s,4} +  \gamma_{69}w_{2,s,5}\\
&\quad +  \gamma_{104}w_{2,s,6} +  \gamma_{110}w_{2,s,7} +  \gamma_{42}w_{2,s,8} +  \gamma_{47}w_{2,s,9} +  \gamma_{52}w_{2,s,10} = 0,\\  
&\gamma_{83}w_{2,s,1} + \gamma_{85}w_{2,s,2} +  a_1w_{2,s,3} +  \gamma_{97}w_{2,s,4} +  \gamma_{102}w_{2,s,5}\\
&\quad +  \gamma_{103}w_{2,s,6} +  \gamma_{112}w_{2,s,7} +  \gamma_{43}w_{2,s,8} +  \gamma_{48}w_{2,s,9} +  \gamma_{51}w_{2,s,10} =0,\\ 
&\gamma_{86}w_{2,s,1} + \gamma_{93}w_{2,s,2} +   \gamma_{89}w_{2,s,3} +  \gamma_{101}w_{2,s,4} +  \gamma_{98}w_{2,s,5}\\
&\quad +  \gamma_{108}w_{2,s,6} +  \gamma_{111}w_{2,s,7} +  \gamma_{53}w_{2,s,8} +  \gamma_{44}w_{2,s,9} +  \gamma_{57}w_{2,s,10} = 0,\\  
&\gamma_{87}w_{2,s,1} +  a_2w_{2,s,2} +  \gamma_{90}w_{2,s,3} +  \gamma_{105}w_{2,s,4} +  \gamma_{99}w_{2,s,5}\\
&\quad +  \gamma_{107}w_{2,s,6} +  \gamma_{113}w_{2,s,7} +  \gamma_{54}w_{2,s,8} +  \gamma_{45}w_{2,s,9} +  \gamma_{56}w_{2,s,10} = 0,\\  
&a_3w_{2,s,1} +  \gamma_{88}w_{2,s,2} +  \gamma_{91}w_{2,s,3} +  \gamma_{106}w_{2,s,4} +  \gamma_{109}w_{2,s,5}\\
&\quad +  \gamma_{100}w_{2,s,6} +  \gamma_{114}w_{2,s,7} +  \gamma_{55}w_{2,s,8} +  \gamma_{58}w_{2,s,9} +  \gamma_{46}w_{2,s,10} = 0, 
\end{align*}
where
\begin{align*}
a_1 &= \begin{cases}  \gamma_{\{66, 148\}}, &s = 3,\\   \gamma_{66}, & s \geqslant 4, \end{cases}\ 
a_2 = \begin{cases}  \gamma_{\{92, 148\}}, &s = 3,\\   \gamma_{92}, & s \geqslant 4, \end{cases}\\
a_3 &= \begin{cases}  \gamma_{\{94, 148, 149, 150\}}, &s = 3,\\   \gamma_{94}, & s \geqslant 4. \end{cases}
\end{align*}
Computing from the above equalities, we obtain 
\begin{equation}\begin{cases}
\gamma_j = 0, \ j = 29, \ldots, 34, 36, \ldots, 41,  43, 49,\ldots, 74,\\
a_1 = a_2 = a_3 = 0.
\end{cases}\tag{\ref{6.2.7}.2}
\end{equation}

With the aid of (\ref{6.2.7}.2),  the homomorphisms $g_1, g_2, g_3, g_4$ send  (\ref{6.2.7}.1) to
\begin{align*}
&\gamma_{131}w_{2,s,1} +  \gamma_{136}w_{2,s,2} +   \gamma_{124}w_{2,s,3} +  \gamma_{139}w_{2,s,4} +  \gamma_{115}w_{2,s,5}\\
&\quad +  \gamma_{128}w_{2,s,6} +  \gamma_{119}w_{2,s,7} +  \gamma_{74}w_{2,s,8} +  \gamma_{59}w_{2,s,9} +  \gamma_{65}w_{2,s,10} = 0,\\  
&\gamma_{132}w_{2,s,1} + a_4w_{2,s,2} +  \gamma_{125}w_{2,s,3} +  \gamma_{140}w_{2,s,4} +  \gamma_{116}w_{2,s,5}\\
&\quad +  \gamma_{127}w_{2,s,6} +  \gamma_{120}w_{2,s,7} +  \gamma_{75}w_{2,s,8} +  \gamma_{60}w_{2,s,9} +  \gamma_{64}w_{2,s,10} = 0,
\end{align*}
\begin{align*}
&\gamma_{133}w_{2,s,1} + a_5w_{2,s,2} +  a_6w_{2,s,3} +  \gamma_{141}w_{2,s,4} +  \gamma_{117}w_{2,s,5}\\
&\quad +  a_7w_{2,s,6} +  \gamma_{121}w_{2,s,7} +  \gamma_{76}w_{2,s,8} +  \gamma_{61}w_{2,s,9} +  \gamma_{63}w_{2,s,10} = 0,\\ 
&\gamma_{134}w_{2,s,1} + a_8w_{2,s,2} +  a_9w_{2,s,3} +  \gamma_{142}w_{2,s,4} +  \gamma_{118}w_{2,s,5}\\
&\quad +  a_{10}w_{2,s,6} +  \gamma_{122}w_{2,s,7} +  \gamma_{77}w_{2,s,8} +  \gamma_{62}w_{2,s,9} +  \gamma_{73}w_{2,s,10} = 0,  
\end{align*}
where
\begin{align*}
a_4 &= \begin{cases}  \gamma_{\{66, 135\}}, &s = 3,\\  \gamma_{135}, &s \geqslant 4,\end{cases}\ \
a_5 = \begin{cases}  \gamma_{\{94, 137\}}, &s = 3,\\  \gamma_{137}, &s \geqslant 4,\end{cases}\\
a_6 &= \begin{cases}  \gamma_{\{123, 150\}}, &s = 3,\\  \gamma_{123}, &s \geqslant 4,\end{cases}\ \
a_7 = \begin{cases}  \gamma_{\{126, 149\}}, &s = 3,\\  \gamma_{126}, &s \geqslant 4,\end{cases}\\
a_8 &= \begin{cases}  \gamma_{\{94, 138, 143, 144\}}, &s = 3,\\  \gamma_{138}, &s \geqslant 4,\end{cases}\ \
a_9 = \begin{cases}  \gamma_{\{130, 143, 147, 150\}}, &s = 3,\\   \gamma_{130}, &s \geqslant 4,\end{cases}\\
a_{10} &= \begin{cases}  \gamma_{\{129, 144, 147, 149\}}, &s = 3,\\  \gamma_{129}, &s \geqslant 4.\end{cases}
\end{align*}
These equalities imply 
\begin{equation}\begin{cases}
a_i = 0, i = 4,5, \ldots , 10,\\
\gamma_j = 0,\ j= 59, \ldots, 65;\ 73, \ldots, 77;\ 115, \ldots , 122,\\ 
124, 125,127, 128, 131, \ldots, 134, 136, 139,\ldots, 142. 
\end{cases} \tag{\ref{6.2.7}.3}
\end{equation}

Using (\ref{6.2.7}.2) and (\ref{6.2.7}.3), we see that the image of (\ref{6.2.7}.1) under the homomorphism $h$ is
\begin{align*}
&a_{11}w_{2,s,1} + a_{12}w_{2,s,2} +  \gamma_{143}w_{2,s,3} +  a_{13}w_{2,s,4} +  \gamma_{144}w_{2,s,5}\\
&\quad +  \gamma_{147}w_{2,s,6} +  \gamma_{145}w_{2,s,7} +  \gamma_{146}w_{2,s,8} +  \gamma_{78}w_{2,s,9} +  \gamma_{79}w_{2,s,10} =0,
\end{align*}
where 
$$
a_{11} = \begin{cases}  \gamma_{138}, & s = 3,\\   \gamma_{148}, & s \geqslant 4, \end{cases}\\
a_{12} = \begin{cases}  \gamma_{130}, & s = 3,\\   \gamma_{150}, & s \geqslant 4, \end{cases}\\
a_{13} = \begin{cases}  \gamma_{129}, & s = 3,\\   \gamma_{149}, & s \geqslant 4. \end{cases}
$$

From this equality, we obtain 
\begin{equation}\begin{cases}
\gamma_j = 0,\ j =  78, 79, 143, 144, 145, 146, 147,\\
a_{11} = a_{12} = a_{13} = 0.
\end{cases}\tag{\ref{6.2.7}.4}
\end{equation}

Combining (\ref{6.2.7}.2), (\ref{6.2.7}.3) and (\ref{6.2.7}.4), we obtain $\gamma_j = 0$ for all $j$. The proposition follows.
\end{proof}

\begin{rems}\label{6.2.5} The element $[\theta]$ defined in the relation (\ref{6.2.4}.6) is an $GL_4(\mathbb F_2)$-invariant of  $(\mathbb F_2\underset {\mathcal A}\otimes R_4)_{17}$.
\end{rems}

\subsection{The case $t=3$}\label{6.3}\

\medskip
From the results in Kameko \cite{ka}, we see that $\dim (\mathbb F_2\underset{\mathcal A}\otimes P_3)_{2^{s+4}+2^{s+1}-3} = 13$ with a basis consisting of all the classes $w_{3,s,j}, 1 \leqslant j \leqslant 13,$ represented by  the following monomials:

\smallskip
\centerline{\begin{tabular}{ll}
$1.\  (2^s-1,2^s - 1,16.2^s - 1),$& $2.\  (2^s-1,16.2^s - 1,2^s - 1),$\cr 
$3.\  (16.2^s-1,2^s - 1,2^s - 1),$& $4.\  (2^s-1,2.2^s - 1,15.2^s - 1),$\cr 
$5.\  (2.2^s-1,2^s - 1,15.2^s - 1),$& $6.\  (2.2^s-1,15.2^s - 1,2^s - 1),$\cr 
$7.\  (2.2^s-1,3.2^s - 1,13.2^s - 1),$& $8.\  (2.2^s-1,8.2^s - 1,8.2^s - 1),$\cr 
$9.\  (8.2^s-1,2.2^s - 1,8.2^s - 1),$& $10.\  (8.2^s-1,8.2^s - 1,2.2^s - 1),$\cr 
$11.\  (4.2^s-1,6.2^s - 1,8.2^s - 1),$& $12.\  (4.2^s-1,8.2^s - 1,6.2^s - 1),$\cr 
$13.\  (8.2^s-1,4.2^s - 1,6.2^s - 1).$& \cr
\end{tabular}}

\smallskip\noindent
Hence, we easily obtain

\begin{props}\label{6.3.4} $(\mathbb F_2\underset{\mathcal A}\otimes Q_4)_{2^{s+4} +2^{s+1}-3}$ is  an $\mathbb F_2$-vector space of dimension 52 with a basis consisting of all the  classes represented by the  monomials $a_{3,s,j}, 1 \leqslant j \leqslant 52,$ which are determined as follows: 

\smallskip
\centerline{\begin{tabular}{ll}
$1.\  (0,2.2^s - 1,8.2^s - 1,8.2^s - 1),$& $2.\  (0,8.2^s - 1,2.2^s - 1,8.2^s - 1),$\cr 
$3.\  (0,8.2^s - 1,8.2^s - 1,2.2^s - 1),$& $4.\  (2.2^s-1,0,8.2^s - 1,8.2^s - 1),$\cr 
$5.\  (2.2^s-1,8.2^s - 1,0,8.2^s - 1),$& $6.\  (2.2^s-1,8.2^s - 1,8.2^s - 1,0),$\cr 
$7.\  (8.2^s-1,0,2.2^s - 1,8.2^s - 1),$& $8.\  (8.2^s-1,0,8.2^s - 1,2.2^s - 1),$\cr 
$9.\  (8.2^s-1,2.2^s - 1,0,8.2^s - 1),$& $10.\  (8.2^s-1,2.2^s - 1,8.2^s - 1,0),$\cr 
$11.\  (8.2^s-1,8.2^s - 1,0,2.2^s - 1),$& $12.\  (8.2^s-1,8.2^s - 1,2.2^s - 1,0),$\cr 
$13.\  (0,4.2^s - 1,6.2^s - 1,8.2^s - 1),$& $14.\  (0,4.2^s - 1,8.2^s - 1,6.2^s - 1),$\cr 
$15.\  (0,8.2^s - 1,4.2^s - 1,6.2^s - 1),$& $16.\  (4.2^s-1,0,6.2^s - 1,8.2^s - 1),$\cr 
$17.\  (4.2^s-1,0,8.2^s - 1,6.2^s - 1),$& $18.\  (4.2^s-1,6.2^s - 1,0,8.2^s - 1),$\cr 
$19.\  (4.2^s-1,6.2^s - 1,8.2^s - 1,0),$& $20.\  (4.2^s-1,8.2^s - 1,0,6.2^s - 1),$\cr 
$21.\  (4.2^s-1,8.2^s - 1,6.2^s - 1,0),$& $22.\  (8.2^s-1,0,4.2^s - 1,6.2^s - 1),$\cr 
$23.\  (8.2^s-1,4.2^s - 1,0,6.2^s - 1),$& $24.\  (8.2^s-1,4.2^s - 1,6.2^s - 1,0),$\cr 
$25.\  (0,2^s - 1,2^s - 1,16.2^s - 1),$& $26.\  (0,2^s - 1,16.2^s - 1,2^s - 1),$\cr 
$27.\  (0,16.2^s - 1,2^s - 1,2^s - 1),$& $28.\  (2^s-1,0,2^s - 1,16.2^s - 1),$\cr 
$29.\  (2^s-1,0,16.2^s - 1,2^s - 1),$& $30.\  (2^s-1,2^s - 1,0,16.2^s - 1),$\cr 
$31.\  (2^s-1,2^s - 1,16.2^s - 1,0),$& $32.\  (2^s-1,16.2^s - 1,0,2^s - 1),$\cr 
$33.\  (2^s-1,16.2^s - 1,2^s - 1,0),$& $34.\  (16.2^s-1,0,2^s - 1,2^s - 1),$\cr 
$35.\  (16.2^s-1,2^s - 1,0,2^s - 1),$& $36.\  (16.2^s-1,2^s - 1,2^s - 1,0),$\cr 
$37.\  (0,2^s - 1,2.2^s - 1,15.2^s - 1),$& $38.\  (0,2.2^s - 1,2^s - 1,15.2^s - 1),$\cr 
$39.\  (0,2.2^s - 1,15.2^s - 1,2^s - 1),$& $40.\  (2^s-1,0,2.2^s - 1,15.2^s - 1),$\cr 
$41.\  (2^s-1,2.2^s - 1,0,15.2^s - 1),$&   $42.\  (2^s-1,2.2^s - 1,15.2^s - 1,0),$\cr 
$43.\  (2.2^s-1,0,2^s - 1,15.2^s - 1),$&   $44.\  (2.2^s-1,0,15.2^s - 1,2^s - 1),$\cr 
$45.\  (2.2^s-1,2^s - 1,0,15.2^s - 1),$& $46.\  (2.2^s-1,2^s - 1,15.2^s - 1,0),$\cr 
$47.\  (2.2^s-1,15.2^s - 1,0,2^s - 1),$& $48.\  (2.2^s-1,15.2^s - 1,2^s - 1,0),$\cr 
$49.\  (0,2.2^s - 1,3.2^s - 1,13.2^s - 1),$& $50.\  (2.2^s-1,0,3.2^s - 1,13.2^s - 1),$\cr 
$51.\  (2.2^s-1,3.2^s - 1,0,13.2^s - 1),$& $52.\  (2.2^s-1,3.2^s - 1,13.2^s - 1,0).$\cr
\end{tabular}}
\end{props}

\medskip
Now, we determine $(\mathbb F_2\underset{\mathcal A}\otimes R_4)_{2^{s+4}+2^{s+1}-3}$. 

Set $\mu_4(1)=136, \mu_4(2) = 180 $ and $\mu_4(s) =195$ for $s\geqslant 3$.

\begin{thms}\label{dlc6.3} $(\mathbb F_2 \underset {\mathcal A} \otimes R_4)_{2^{s+4}+2^{s+1}-3}$ is an $\mathbb F_2$-vector space of dimension $\mu_3(s)-52$
with a basis consisting of all the classes represented by the monomials $a_{3,s,j}, 53 \leqslant j \leqslant \mu_4(s)$, which are determined as follows:

\smallskip
For $s\geqslant 1$,

\smallskip
\centerline{\begin{tabular}{ll}
 $53.\  (1,2.2^s - 2,8.2^s - 1,8.2^s - 1),$& $54.\  (1,8.2^s - 1,2.2^s - 2,8.2^s - 1),$\cr 
$55.\  (1,8.2^s - 1,8.2^s - 1,2.2^s - 2),$& $56.\  (8.2^s-1,1,2.2^s - 2,8.2^s - 1),$\cr 
$57.\  (8.2^s-1,1,8.2^s - 1,2.2^s - 2),$& $58.\  (8.2^s-1,8.2^s - 1,1,2.2^s - 2),$\cr 
$59.\  (1,2.2^s - 1,8.2^s - 2,8.2^s - 1),$& $60.\  (1,2.2^s - 1,8.2^s - 1,8.2^s - 2),$\cr 
$61.\  (1,8.2^s - 2,2.2^s - 1,8.2^s - 1),$& $62.\  (1,8.2^s - 2,8.2^s - 1,2.2^s - 1),$\cr 
$63.\  (1,8.2^s - 1,2.2^s - 1,8.2^s - 2),$& $64.\  (1,8.2^s - 1,8.2^s - 2,2.2^s - 1),$\cr 
$65.\  (2.2^s-1,1,8.2^s - 2,8.2^s - 1),$& $66.\  (2.2^s-1,1,8.2^s - 1,8.2^s - 2),$\cr 
$67.\  (2.2^s-1,8.2^s - 1,1,8.2^s - 2),$& $68.\  (8.2^s-1,1,2.2^s - 1,8.2^s - 2),$\cr 
$69.\  (8.2^s-1,1,8.2^s - 2,2.2^s - 1),$& $70.\  (8.2^s-1,2.2^s - 1,1,8.2^s - 2),$\cr 
$71.\  (1,4.2^s - 2,6.2^s - 1,8.2^s - 1),$& $72.\  (1,4.2^s - 2,8.2^s - 1,6.2^s - 1),$\cr 
$73.\  (1,8.2^s - 1,4.2^s - 2,6.2^s - 1),$& $74.\  (8.2^s-1,1,4.2^s - 2,6.2^s - 1),$\cr 
$75.\  (1,4.2^s - 1,6.2^s - 2,8.2^s - 1),$& $76.\  (1,4.2^s - 1,8.2^s - 1,6.2^s - 2),$\cr 
$77.\  (1,8.2^s - 1,4.2^s - 1,6.2^s - 2),$& $78.\  (4.2^s-1,1,6.2^s - 2,8.2^s - 1),$\cr 
$79.\  (4.2^s-1,1,8.2^s - 1,6.2^s - 2),$& $80.\  (4.2^s-1,8.2^s - 1,1,6.2^s - 2),$\cr 
$81.\  (8.2^s-1,1,4.2^s - 1,6.2^s - 2),$& $82.\  (8.2^s-1,4.2^s - 1,1,6.2^s - 2),$\cr 
$83.\  (1,4.2^s - 1,6.2^s - 1,8.2^s - 2),$& $84.\  (1,4.2^s - 1,8.2^s - 2,6.2^s - 1),$\cr 
$85.\  (1,8.2^s - 2,4.2^s - 1,6.2^s - 1),$& $86.\  (4.2^s-1,1,6.2^s - 1,8.2^s - 2),$\cr 
$87.\  (4.2^s-1,1,8.2^s - 2,6.2^s - 1),$& $88.\  (4.2^s-1,6.2^s - 1,1,8.2^s - 2),$\cr 
$89.\  (3,8.2^s - 3,2.2^s - 2,8.2^s - 1),$& $90.\  (3,8.2^s - 3,8.2^s - 1,2.2^s - 2),$\cr 
$91.\  (3,8.2^s - 1,8.2^s - 3,2.2^s - 2),$& $92.\  (8.2^s-1,3,8.2^s - 3,2.2^s - 2),$\cr 
$93.\  (3,4.2^s - 3,6.2^s - 2,8.2^s - 1),$& $94.\  (3,4.2^s - 3,8.2^s - 1,6.2^s - 2),$\cr 
$95.\  (3,8.2^s - 1,4.2^s - 3,6.2^s - 2),$& $96.\  (8.2^s-1,3,4.2^s - 3,6.2^s - 2),$\cr 
$97.\  (3,4.2^s - 3,6.2^s - 1,8.2^s - 2),$& $98.\  (3,4.2^s - 3,8.2^s - 2,6.2^s - 1),$\cr 
$99.\  (3,8.2^s - 3,4.2^s - 2,6.2^s - 1),$& $100.\  (3,4.2^s - 1,6.2^s - 3,8.2^s - 2),$\cr 
$101.\  (4.2^s-1,3,6.2^s - 3,8.2^s - 2),$& $102.\  (3,4.2^s - 1,8.2^s - 3,6.2^s - 2),$\cr 
$103.\  (3,8.2^s - 3,4.2^s - 1,6.2^s - 2),$& $104.\  (4.2^s-1,3,8.2^s - 3,6.2^s - 2),$\cr 
$105.\  (7,8.2^s - 5,8.2^s - 3,2.2^s - 2),$& $106.\  (7,8.2^s - 5,4.2^s - 3,6.2^s - 2),$\cr 
$107.\  (1,2.2^s - 2,3.2^s - 1,13.2^s - 1),$& $108.\  (1,2.2^s - 1,3.2^s - 2,13.2^s - 1),$\cr 
$109.\  (2.2^s-1,1,3.2^s - 2,13.2^s - 1),$& $110.\  (1,2.2^s - 1,3.2^s - 1,13.2^s - 2),$\cr 
$111.\  (2.2^s-1,1,3.2^s - 1,13.2^s - 2),$& $112.\  (2.2^s-1,3.2^s - 1,1,13.2^s - 2),$\cr 
$113.\  (3,2.2^s - 1,8.2^s - 3,8.2^s - 2),$& $114.\  (3,8.2^s - 3,2.2^s - 1,8.2^s - 2),$\cr 
$115.\  (3,8.2^s - 3,8.2^s - 2,2.2^s - 1),$& $116.\  (1,2^s - 1,2.2^s - 2,15.2^s - 1),$\cr 
$117.\  (1,2.2^s - 2,2^s - 1,15.2^s - 1),$& $118.\  (1,2.2^s - 2,15.2^s - 1,2^s - 1),$\cr 
$119.\  (1,2^s - 1,2^s - 1,16.2^s - 2),$& $120.\  (1,2^s - 1,16.2^s - 2,2^s - 1),$\cr 
$121.\  (1,16.2^s - 2,2^s - 1,2^s - 1).$& \cr 
\end{tabular}}

\smallskip
 For $s=1$,

\smallskip
\centerline{\begin{tabular}{lll}
$122.\ (3,3,12,15),$& $123.\ (3,3,15,12),$& $124.\ (3,15,3,12),$\cr 
$125.\ (15,3,3,12),$& $126.\ (3,7,11,12),$& $127.\ (7,3,11,12),$\cr 
$128.\ (7,11,3,12),$& $129.\ (7,7,8,11),$& $130.\ (7,7,11,8),$\cr 
$131.\ (7,7,9,10),$& $132.\ (1,1,3,28),$& $133.\ (1,3,1,28),$\cr 
$134.\ (1,3,28,1),$& $135.\ (3,1,1,28),$& $136.\ (3,1,28,1).$\cr  
\end{tabular}}

\smallskip
For $s \geqslant 2$,

\medskip
\centerline{\begin{tabular}{ll}
$122.\  (2^s-1,1,2^s - 1,16.2^s - 2),$& $123.\  (2^s-1,1,16.2^s - 2,2^s - 1),$\cr 
$124.\  (2^s-1,2^s - 1,1,16.2^s - 2),$& $125.\  (2.2^s-1,3,8.2^s - 3,8.2^s - 2),$\cr 
$126.\  (2^s-1,1,2.2^s - 2,15.2^s - 1),$& $127.\  (3,2.2^s - 3,8.2^s - 2,8.2^s - 1),$\cr 
$128.\  (3,2.2^s - 3,8.2^s - 1,8.2^s - 2),$& $129.\  (3,8.2^s - 1,2.2^s - 3,8.2^s - 2),$\cr
 $130.\  (8.2^s-1,3,2.2^s - 3,8.2^s - 2),$& $131.\  (7,8.2^s - 5,2.2^s - 3,8.2^s - 2),$\cr 
$132.\  (7,2.2^s - 5,8.2^s - 3,8.2^s - 2),$& $133.\  (7,4.2^s - 5,6.2^s - 3,8.2^s - 2),$\cr 
$134.\  (7,4.2^s - 5,8.2^s - 3,6.2^s - 2),$& $135.\  (1,2^s - 2,2^s - 1,16.2^s - 1),$\cr 
$136.\  (1,2^s - 2,16.2^s - 1,2^s - 1),$& $137.\  (1,2^s - 1,2^s - 2,16.2^s - 1),$\cr 
$138.\  (1,2^s - 1,16.2^s - 1,2^s - 2),$& $139.\  (1,16.2^s - 1,2^s - 2,2^s - 1),$\cr 
$140.\  (1,16.2^s - 1,2^s - 1,2^s - 2),$& $141.\  (2^s-1,1,2^s - 2,16.2^s - 1),$\cr 
$142.\  (2^s-1,1,16.2^s - 1,2^s - 2),$& $143.\  (2^s-1,16.2^s - 1,1,2^s - 2),$\cr 
$144.\  (16.2^s-1,1,2^s - 2,2^s - 1),$& $145.\  (16.2^s-1,1,2^s - 1,2^s - 2),$\cr 
$146.\  (16.2^s-1,2^s - 1,1,2^s - 2),$& $147.\  (1,2^s - 2,2.2^s - 1,15.2^s - 1),$\cr 
$148.\  (1,2.2^s - 1,2^s - 2,15.2^s - 1),$& $149.\  (1,2.2^s - 1,15.2^s - 1,2^s - 2),$\cr 
$150.\  (2.2^s-1,1,2^s - 2,15.2^s - 1),$& $151.\  (2.2^s-1,1,15.2^s - 1,2^s - 2),$\cr 
$152.\  (2.2^s-1,15.2^s - 1,1,2^s - 2),$& $153.\  (1,2^s - 1,2.2^s - 1,15.2^s - 2),$\cr 
$154.\  (1,2.2^s - 1,2^s - 1,15.2^s - 2),$& $155.\  (1,2.2^s - 1,15.2^s - 2,2^s - 1),$\cr 
$156.\  (2^s-1,1,2.2^s - 1,15.2^s - 2),$& $157.\  (2^s-1,2.2^s - 1,1,15.2^s - 2),$\cr 
$158.\  (2.2^s-1,1,2^s - 1,15.2^s - 2),$& $159.\  (2.2^s-1,1,15.2^s - 2,2^s - 1),$\cr 
$160.\  (2.2^s-1,2^s - 1,1,15.2^s - 2),$& $161.\  (3,2.2^s - 3,2^s - 2,15.2^s - 1),$\cr 
$162.\  (3,2.2^s - 3,15.2^s - 1,2^s - 2),$& $163.\  (3,2.2^s - 1,15.2^s - 3,2^s - 2),$\cr 
$164.\  (2.2^s-1,3,15.2^s - 3,2^s - 2),$& $165.\  (3,2.2^s - 3,3.2^s - 2,13.2^s - 1),$\cr 
$166.\  (3,2.2^s - 3,3.2^s - 1,13.2^s - 2),$& $167.\  (3,2.2^s - 1,3.2^s - 3,13.2^s - 2),$\cr 
$168.\  (2.2^s-1,3,3.2^s - 3,13.2^s - 2),$& $169.\  (3,2^s - 1,16.2^s - 3,2^s - 2),$\cr 
$170.\  (3,16.2^s - 3,2^s - 2,2^s - 1),$& $171.\  (3,16.2^s - 3,2^s - 1,2^s - 2),$\cr 
$172.\  (3,2^s - 1,2.2^s - 3,15.2^s - 2),$& $173.\  (3,2.2^s - 3,2^s - 1,15.2^s - 2),$\cr 
$174.\  (3,2.2^s - 3,15.2^s - 2,2^s - 1).$& \cr
\end{tabular}}

\medskip
For $s=2$,

\smallskip
\centerline{\begin{tabular}{lll}
$175.\ (3,3,3,60),$& $176.\ (3,3,60,3),$& $177.\ (3,3,4,59),$\cr 
$178.\ (3,3,7,56),$& $179.\ (3,7,3,56),$& $180.\ (7,3,3,56).$\cr
\end{tabular}}

\medskip
For $s \geqslant 3$

\medskip
\centerline{\begin{tabular}{ll}
$175.\  (2^s-1,3,2.2^s - 3,15.2^s - 2),$& $176.\  (2^s-1,3,16.2^s - 3,2^s - 2),$\cr 
$177.\  (3,2^s - 3,2^s - 2,16.2^s - 1),$& $178.\  (3,2^s - 3,16.2^s - 1,2^s - 2),$\cr 
$179.\  (3,16.2^s - 1,2^s - 3,2^s - 2),$& $180.\  (16.2^s-1,3,2^s - 3,2^s - 2),$\cr 
$181.\  (3,2^s - 3,2^s - 1,16.2^s - 2),$& $182.\  (3,2^s - 3,16.2^s - 2,2^s - 1),$\cr 
$183.\  (3,2^s - 1,2^s - 3,16.2^s - 2),$& $184.\  (2^s-1,3,2^s - 3,16.2^s - 2),$\cr 
$185.\  (3,2^s - 3,2.2^s - 2,15.2^s - 1),$& $186.\  (3,2^s - 3,2.2^s - 1,15.2^s - 2),$\cr 
\end{tabular}}
\centerline{\begin{tabular}{ll}
$187.\  (3,2.2^s - 1,2^s - 3,15.2^s - 2),$& $188.\  (2.2^s-1,3,2^s - 3,15.2^s - 2),$\cr 
$189.\  (7,16.2^s - 5,2^s - 3,2^s - 2),$& $190.\  (7,2.2^s - 5,2^s - 3,15.2^s - 2),$\cr 
$191.\  (7,2.2^s - 5,15.2^s - 3,2^s - 2),$& $192.\  (7,2.2^s - 5,3.2^s - 3,13.2^s - 2).$\cr
\end{tabular}}

\medskip
For $s =3$,
$$193.\ (7,7,121,6),\quad194.\ (7,7,9,118),\quad195.\ (7,7,7,120).$$

\medskip
For $s \geqslant 4$,

\medskip
\centerline{\begin{tabular}{ll}
$193.\  (7,2^s - 5,16.2^s - 3,2^s - 2),$& $194.\  (7,2^s - 5,2.2^s - 3,15.2^s - 2),$\cr 
$195.\  (7,2^s - 5,2^s - 3,16.2^s - 2).$& \cr 
\end{tabular}}
\end{thms}

We prove the theorem by proving the following propositions.

\begin{props}\label{mdc6.3} The $\mathbb F_2$-vector space $(\mathbb F_2\underset {\mathcal A}\otimes R_4)_{2^{s+4}+2^{s+1}-3}$ is generated by the $\mu_4(s)-52$ elements listed in Theorem \ref{dlc6.3}.
\end{props}

We need the following for the proof of this proposition.

\begin{lems}\label{6.3.1} The following matrices are strictly inadmissible
$$\begin{pmatrix} 1&1&0&1\\ 1&0&0&0\\ 0&1&0&0\\ 0&0&1&0\\ 0&0&1&0\end{pmatrix}  \quad \begin{pmatrix} 1&1&1&0\\ 1&0&0&0\\ 0&1&0&0\\ 0&0&1&0\\ 0&0&0&1\end{pmatrix}\quad \begin{pmatrix} 1&1&0&1\\ 1&0&0&0\\ 0&1&0&0\\ 0&0&1&0\\ 0&0&0&1\end{pmatrix}. $$
\end{lems}

\begin{proof} The monomials corresponding to the above matrices respectively are (3,5,24,1),  (3,5,9,16), (3,5,8,17). By a direct computation we have
\begin{align*}
&(3,5,24,1) = Sq^1(3,3,25,1) + Sq^2\big((5,3,22,1) + (2,3,25,1) \\ 
&\quad+ (5,3,21,2)\big) + Sq^4\big((3,3,22,1) + (3,3,21,2)  \\ 
&\quad+(3,9,13,4) \big) + Sq^8(3,5,13,4)+ (3,4,25,1)  \\ 
&\quad+ (2,5,25,1)+ (3,5,17,8) \quad  \text{mod   }\mathcal L_4(3;1;1;1;1),\\
&(3,5,9,16) = Sq^1(3,3,9,17) + Sq^2\big((5,3,9,14) + (2,3,9,17)\\ 
&\quad + (5,3,6,17)\big) + Sq^4\big((3,3,9,14) + (3,3,6,17) \big)\\ 
&\quad + (3,4,9,17) + (2,5,9,17)+ (3,5,8,17) \quad  \text{mod   }\mathcal L_4(3;1;1;1;1),\\
&(3,5,8,17) = Sq^1(5,9,1,17) + Sq^2\big((3,9,2,17)  + (2,11,1,17) \\ 
&\quad+ (6,7,1,17) + (5,10,1,15) + (3,10,1,17) + (5,6,1,19)\\ 
&\quad + (3,9,1,18)\big)  +Sq^4\big((3,5,4,17) + (4,7,1,17) + (3,10,1,15)\\ 
&\quad + (3,6,1,19) + (3,4,1,21) + (3,5,1,20)\big)+ (2,13,1,17)\\ 
&\quad + (3,4,1,25) + (3,5,1,24) \quad \text{mod  }\mathcal L_4(3;1;1;1;1),
\end{align*}
The lemma is proved.
\end{proof}

\begin{lems}\label{6.3.2} The following matrices are strictly inadmissible
$$ \begin{pmatrix} 1&1&0&1\\ 1&0&1&1\\ 1&0&0&1\\ 0&1&0&1\end{pmatrix} \quad  \begin{pmatrix} 1&1&1&0\\ 1&0&1&1\\ 1&0&1&0\\ 0&1&1&0\end{pmatrix} \quad  \begin{pmatrix} 1&1&1&0\\ 1&1&0&1\\ 1&1&0&0\\ 0&1&1&0\end{pmatrix} \quad  \begin{pmatrix} 1&1&1&0\\ 1&1&0&1\\ 1&1&0&0\\ 1&0&1&0\end{pmatrix} $$    
$$ \begin{pmatrix} 1&1&0&1\\ 1&1&0&1\\ 0&1&0&1\\ 0&0&1&1\end{pmatrix} \quad  \begin{pmatrix} 1&1&1&0\\ 1&1&1&0\\ 0&1&1&0\\ 0&0&1&1\end{pmatrix} \quad  \begin{pmatrix} 1&1&1&0\\ 1&1&1&0\\ 0&1&1&0\\ 0&1&0&1\end{pmatrix} \quad  \begin{pmatrix} 1&1&0&1\\ 1&1&0&1\\ 1&0&0&1\\ 0&0&1&1\end{pmatrix} $$    
$$ \begin{pmatrix} 1&1&1&0\\ 1&1&1&0\\ 1&0&1&0\\ 0&0&1&1\end{pmatrix} \quad  \begin{pmatrix} 1&0&1&1\\ 1&0&1&1\\ 1&0&0&1\\ 0&1&0&1\end{pmatrix} \quad  \begin{pmatrix} 1&0&1&1\\ 1&0&1&1\\ 1&0&1&0\\ 0&1&1&0\end{pmatrix} \quad  \begin{pmatrix} 1&1&1&0\\ 1&1&1&0\\ 1&1&0&0\\ 0&1&0&1\end{pmatrix} $$    
$$ \begin{pmatrix} 1&1&0&1\\ 1&1&0&1\\ 1&1&0&0\\ 0&1&1&0\end{pmatrix} \quad  \begin{pmatrix} 1&1&1&0\\ 1&1&1&0\\ 1&0&1&0\\ 1&0&0&1\end{pmatrix} \quad  \begin{pmatrix} 1&1&1&0\\ 1&1&1&0\\ 1&1&0&0\\ 1&0&0&1\end{pmatrix} \quad  \begin{pmatrix} 1&1&0&1\\ 1&1&0&1\\ 1&1&0&0\\ 1&0&1&0\end{pmatrix} $$    
$$ \begin{pmatrix} 1&1&1&0\\ 1&0&1&1\\ 1&0&0&1\\ 0&1&0&1\end{pmatrix} \quad  \begin{pmatrix} 1&1&0&1\\ 1&0&1&1\\ 1&0&1&0\\ 0&1&1&0\end{pmatrix} \quad  \begin{pmatrix} 1&1&1&0\\ 1&1&0&1\\ 1&0&1&0\\ 0&1&0&1\end{pmatrix} \quad  \begin{pmatrix} 1&1&0&1\\ 1&0&1&1\\ 1&0&1&0\\ 0&1&0&1\end{pmatrix} $$  
$$ \begin{pmatrix} 1&0&1&1\\ 1&0&1&1\\ 1&0&1&0\\ 0&1&0&1\end{pmatrix} \quad  \begin{pmatrix} 1&1&1&0\\ 1&1&1&0\\ 1&0&1&0\\ 0&1&0&1\end{pmatrix} \quad  \begin{pmatrix} 1&1&1&0\\ 1&0&1&1\\ 1&0&1&0\\ 0&1&0&1\end{pmatrix} . $$
\end{lems}

\begin{proof} The monomials corresponding to the above matrices respectively are \begin{align*}
&(7,9,2,15),  (7,9,15,2),   (7,15,9,2),  (15,7,9,2),  (3,7,8,15),  (3,7,15,8),\\
&  (3,15,7,8),  (7,3,8,15),  (7,3,15,8),  (7,8,3,15),  (7,8,15,3),  (7,15,3,8), \\
& (7,15,8,3),  (15,3,7,8),  (15,7,3,8),  (15,7,8,3),  (7,9,3,14),  (7,9,14,3), \\
& (7,11,5,10),  (7,9,6,11),  (7,8,7,11),  (7,11,7,8),  (7,9,7,10).
\end{align*}
We prove  the lemma the matrices corresponding to the monomials 
\begin{align*}
&(7,9,2,15), (7,3,8,15),  (7,9,3,14), (7,9,6,11),\\
&(7,8,7,11),  (7,11,7,8),  (7,9,7,10). 
\end{align*}
The others can be obtained by a similar argument. 
We have
\begin{align*}
&(7,9,2,15) = Sq^1(7,7,3,15) + Sq^2(7,7,2,15) \\ 
&\quad + Sq^4\big((5,7,2,15)+ (4,7,3,15) \big)+ (5,11,2,15)\\ 
&\quad  + (4,11,3,15)  + (7,8,3,15) \  \text{mod   }\mathcal L_4(3;3;2;2),\\
&(7,3,8,15) = Sq^1(7,3,7,15) + Sq^2(7,2,7,15)\\ 
&\quad  + Sq^4\big((4,3,7,15)+(5,2,7,15)\big) + (4,3,11,15) \\ 
&\quad + (7,2,9,15) + (5,2,11,15) \quad  \text{mod   }\mathcal L_4(3;3;2;2),\\
&(7,9,3,14) = Sq^1\big((7,7,11,7)+(7,7,5,13)\big)+ Sq^2\big((7,7,10,7)\\ &\quad+(7,7,6,11)+(7,7,3,14)\big)  + Sq^4\big((4,7,11,7)+(5,7,10,7)\\ 
&\quad+(11,5,6,7)+(5,7,6,11)+(5,7,3,14)\big)  +Sq^8(7,5,6,7)\\ 
&\quad+ (4,11,11,7)  + (4,7,11,11)+ (7,8,11,7) +(7,7,10,9) + (5,11,10,7)\\ 
&\quad +(5,7,10,11) +(7,5,10,11) + (7,7,8,11)+ (7,7,11,8)\\ 
&\quad + (5,11,6,11) + (5,7,10,11) +(5,11,3,14)\quad  \text{mod   }\mathcal L_4(3;3;2;2),\\
&(7,9,6,11)= Sq^1(7,7,5,13) + Sq^2\big((7,7,6,11)  + (7,7,3,14)\big)\\ 
&\quad+ Sq^4\big((5,7,6,11) + (5,7,3,14)\big) + (5,11,6,11) + (5,7,10,11)\\ 
&\quad +(7,7,8,11) + (7,9,3,14) + (5,11,3,14) \quad  \text{mod }\mathcal L_4(3;3;2;2),\\
&(7,8,7,11)= Sq^1(7,5,7,13) + Sq^2\big((7,6,7,11) + (7,3,7,14)\big) \\ 
&\quad+ Sq^4\big((5,6,7,11) + (5,3,7,14)\big) + (5,6,11,11) + (5,10,7,11)\\ 
&\quad +(7,6,9,11) + (7,3,9,14) + (5,3,11,14) \quad  \text{mod }\mathcal L_4(3;3;2;2),\\
&(7,11,7,8)= Sq^1(7,11,7,7) + Sq^2(7,10,7,7) + Sq^4\big((4,11,7,7)  \\ 
&\quad+ (5,10,7,7) + (11,7,4,7)\big) + Sq^8(7,7,4,7)  + (4,11,11,7)\\ 
&\quad + (4,11,7,11) +(7,10,9,7)  + (7,10,7,9) + (5,10,11,7)\\ 
&\quad + (5,10,7,11) + (7,11,4,11) + (7,7,8,11)\quad  \text{mod }\mathcal L_4(3;3;2;2),\\
&(7,9,7,10)= Sq^1(7,7,7,11) + Sq^2(7,7,7,10) + Sq^4\big((4,7,7,11)\\ 
&\quad  + (5,7,7,10) \big) + (4,11,7,11) + (4,7,11,11)\\ 
&\quad +(7,7,9,10)  + (5,10,7,9) + (5,11,7,10) + (5,7,11,10)\\ 
&\quad + (7,7,8,11) + (7,8,7,11)\  \text{mod }\mathcal L_4(3;3;2;2).
\end{align*}
The lemma is proved.
\end{proof}

\begin{lems}\label{6.3.3} The following matrices are strictly inadmissible
$$\begin{pmatrix} 1&1&0&1\\ 1&1&0&1\\ 1&1&0&1\\ 1&1&0&0\\ 0&0&1&1\end{pmatrix} \quad \begin{pmatrix} 1&1&1&0\\ 1&1&1&0\\ 1&1&1&0\\ 1&1&0&0\\ 0&0&1&1\end{pmatrix} \quad \begin{pmatrix} 1&1&1&0\\ 1&1&0&1\\ 1&1&0&1\\ 1&1&0&0\\ 0&0&1&1\end{pmatrix}. $$
\end{lems}

\begin{proof} The monomials corresponding to the above matrices respectively are (15,15,16,23), (15,15,23,16), (15,15,17,22). We prove the lemma for the monomials (15,15,16,23), (15,15,17,22). We have
\begin{align*}
&(15,15,16,23)= Sq^1\big((15,15,13,25) + (15,15,11,27) + (15,15,9,29)\big)\\
&\quad + Sq^2\big((15,15,14,23) + (15,11,14,27) + (15,15,11,26)\\
&\quad + (15,15,10,27) + (15,15,7,30) + (15,11,11,30)\big)\\
&\quad + Sq^4\big((15,13,14,23) + (15,7,14,29) + (15,13,7,30) + (15,7,13,30)\big)\\
&\quad + Sq^8\big((9,15,14,23) + (11,13,14,23) + (11,7,14,29)\\ 
&\quad + (9,15,7,30) + (11,13,7,30) + (11,7,13,30)\big)  + (9,23,14,23)\\ 
&\quad + (9,15,22,23) +(15,13,18,23) + (11,21,14,23) + (11,13,22,23)\\ 
&\quad + (11,7,22,29) + (9,23,7,30) + (11,21,7,30)  + (15,7,17,30)\\
&\quad + (11,7,21,30) + (15,7,18,29)\quad  \text{mod   }\mathcal L_4(3;3;3;2;2),\\
&(15,15,17,22)= Sq^1\big((15,15,15,23) + (15,11,15,27)\big)\\
&\quad + Sq^2\big((15,15,15,22) + (15,11,15,26) \big) + Sq^4\big((15,13,15,22)\\ 
&\quad + (15,12,15,23) \big) + Sq^8\big((9,15,15,22) + (11,13,15,22) + (8,15,15,23)\\ 
&\quad+ (11,12,15,23) \big) +(9,23,15,22) + (9,15,23,22)+ (11,21,15,22)\\ 
&\quad + (11,13,23,22)  + (15,13,19,22)+ (8,23,15,23)\\ 
&\quad + (8,15,23,23) +(15,15,16,23) +  (11,20,15,23) \\ 
&\quad+ (11,12,23,23) + (15,12,19,23)  \quad  \text{mod   }\mathcal L_4(3;3;3;2;2).
\end{align*}
The lemma follows.
\end{proof}

Combining Theorem \ref{2.4}, the lemmas in Sections \ref{3}, \ref{4}, \ref{5} and  Lemmas  \ref{6.1.1}, \ref{6.1.2}, \ref{6.1.3}, \ref{6.2.1}, \ref{6.2.2}, \ref{6.3.1}, \ref{6.3.2}, \ref{6.3.3},  we get Proposition \ref{mdc6.3}.

\medskip
Now, we show that the monomials $a_{3,s,j}, 53 \leqslant j \leqslant \mu_4(s)$ are admissible. 

\begin{props}\label{6.3.5} The elements $[a_{3,1,j}], \ 53 \leqslant j \leqslant 136,$ are linearly independent in $(\mathbb F_2\underset {\mathcal A}\otimes R_4)_{33}$.
\end{props}

\begin{proof} Suppose there is a linear relation
\begin{equation} \sum_{j=53}^{136}\gamma_j[a_{3,1,j}]=0, \tag{\ref{6.3.5}.1}
\end{equation}
where $\gamma_j \in \mathbb F_2.$ 

Consider the homomorphisms $f_i,\ i=1,2,3$. Under these homomorphisms, the images of the  above relation respectively are
\begin{align*}
&\gamma_{121}[31,1,1]  +   \gamma_{117}[3,1,29]  +    \gamma_{118}[3,29,1]  +   \gamma_{107}[3,5,25] \\
&\quad  +   \gamma_{53}[3,15,15] +   \gamma_{61}[15,3,15]  +   \gamma_{62}[15,15,3]\\
&\quad  +   \gamma_{71}[7,11,15]  +   \gamma_{72}[7,15,11]  + \gamma_{85}[15,7,11]  =0,\\  
&\gamma_{\{120, 136\}}[31,1,1]  +   a_1[3,1,29]  +   \gamma_{134}[3,29,1]\\
&\quad  +   a_2[3,5,25]  +   \gamma_{54}[3,15,15]  +   \gamma_{\{59, 122\}}[15,3,15]  +   \gamma_{64}[15,15,3]\\
&\quad  +   \gamma_{75}[7,11,15]  +   \gamma_{73}[7,15,11]  +   \gamma_{\{84, 129\}}[15,7,11]  = 0,\\
&\gamma_{\{119, 135\}}[31,1,1]  +  a_3[3,1,29]  +  a_4[3,29,1]  +   a_5[3,5,25]\\
&\quad  +   \gamma_{55}[3,15,15]  +   \gamma_{\{60, 123\}}[15,3,15]  +   \gamma_{\{63, 124\}}[15,15,3]\\
&\quad  +   \gamma_{76}[7,11,15]  +   \gamma_{77}[7,15,11]  +   \gamma_{\{83, 126, 130\}}[15,7,11]  = 0,
\end{align*}
where
\begin{align*}
a_1 &= \gamma_{\{56, 65, 69, 74, 78, 87, 109, 116\}},\ \
a_2 = \gamma_{\{89, 93, 98, 99, 108, 115\}},\\
a_3 &= \gamma_{\{57, 66, 68, 79, 81, 86, 111, 132\}},\ \
a_4 =  \gamma_{\{58, 67, 70, 80, 82, 88, 112, 133\}},\\
a_5 &= \gamma_{\{90, 91, 92, 94, 95, 96, 97,100, 101, 102, 103, 104, 105, 106, 110, 113, 114, 125, 127, 128, 131\}}.
\end{align*}

From this, it implies
\begin{equation}\begin{cases}
a_1 = a_2 = a_3 = a_4 = a_5 = 0,\\
\gamma_j = 0, \ j = 53, 54, 55, 61, 62, 64, 71, 72,\\
 73, 75, 76, 77, 85, 107, 117, 118, 121, 134,\\
\gamma_{\{120, 136\}} =   
\gamma_{\{59, 122\}} =    
\gamma_{\{84, 129\}} =   
\gamma_{\{119, 135\}} =  0,\\ 
\gamma_{\{60, 123\}} =   
\gamma_{\{63, 124\}} =   
\gamma_{\{83, 126, 130\}} = 0.
\end{cases}\tag{\ref{6.3.5}.2}
\end{equation}

With the aid of (\ref{6.3.5}.2), the homomorphisms $f_4, f_5, f_6$ send (\ref{6.3.5}.1) respectively to
\begin{align*}
&\gamma_{120}[1,31,1]  +  \gamma_{\{59, 84, 108, 116\}}[1,3,29]  +    \gamma_{136}[3,29,1]\\
&\quad  +   \gamma_{\{98, 99, 109, 115\}}[3,5,25]  +   \gamma_{\{65, 89, 93, 122\}}[3,15,15]  +   \gamma_{56}[15,3,15]\\
&\quad +   \gamma_{69}[15,15,3]  +   \gamma_{78}[7,11,15] +   \gamma_{\{87, 129\}}[7,15,11]  +   \gamma_{74}[15,7,11]  =0,\\   
&\gamma_{\{119, 133\}}[1,31,1]  +  \gamma_{\{60, 63, 83, 110, 132\}}[1,3,29] +  a_6[3,29,1]   +   a_7[3,5,25]\\
&\quad  +   \gamma_{\{66, 90, 94, 123\}}[3,15,15] +   \gamma_{57}[15,3,15]  +   \gamma_{\{68, 125\}}[15,15,3]\\
&\quad  +   \gamma_{79}[7,11,15]   +   \gamma_{\{86, 127, 130\}}[7,15,11]  +   \gamma_{81}[15,7,11]  = 0,\\   
&\gamma_{\{116, 119, 120, 132\}}[1,1,31]  +   a_8[1,3,29]  +   a_9[3,1,29]  +   a_{10}[3,5,25]\\
&\quad  +   \gamma_{\{67, 91, 95, 124\}}[3,15,15]  +   \gamma_{\{70, 92, 96, 125\}}[15,3,15]  +   \gamma_{58}[15,15,3]\\
&\quad  +   \gamma_{\{88, 105, 106, 128\}}[7,11,15] +   \gamma_{80}[7,15,11]  +   \gamma_{82}[15,7,11]  = 0,   
\end{align*}
where
\begin{align*}
a_6 &=  \gamma_{\{58, 67, 70, 80, 82, 88, 112, 135\}},\ \
a_8 = \gamma_{\{59, 60, 63, 83, 84, 108, 110, 133\}},
\end{align*}
\begin{align*}
a_7 &= \gamma_{\{91, 92, 95, 96, 97, 100, 101, 102, 103, 104, 105, 106, 111, 113, 114, 124, 126, 128, 131\}},\\
a_9 &= \gamma_{\{56, 57, 65, 66, 68, 69, 74, 78, 79, 81, 86, 87, 109, 111, 135, 136\}},\\
a_{10} &= \gamma_{\{89, 90, 93, 94, 97, 98, 99, 100, 101, 102, 103, 104, 112, 113, 114, 115, 122, 123, 126, 127, 129, 130, 131\}}.
\end{align*}
From this, it implies
\begin{equation}\begin{cases}
a_6 = a_7 = a_8 = a_9 = a_{10} = 0,\\
\gamma_j = 0,\ j = 56, 57, 58, 69, 74, 78, 79, 80, 81, 82,120, 136,\\
\gamma_{\{59, 84, 108, 116\}} =   
\gamma_{\{98, 99, 109, 115\}} =    
\gamma_{\{65, 89, 93, 122\}} =  0,\\ 
\gamma_{\{87, 129\}} =  
\gamma_{\{119, 133\}} =  
\gamma_{\{68, 125\}} = 
\gamma_{\{60, 63, 83, 110, 132\}} = 0,\\  
\gamma_{\{66, 90, 94, 123\}} =   
\gamma_{\{86, 127, 130\}} =   
\gamma_{\{116, 119, 120, 132\}} = 0,\\  
\gamma_{\{67, 91, 95, 124\}} =   
\gamma_{\{70, 92, 96, 125\}} = 
\gamma_{\{88, 105, 106, 128\}} = 0.     
\end{cases}\tag{\ref{6.3.5}.3}
\end{equation}

With the aid of (\ref{6.3.5}.2) and (\ref{6.3.5}.3), the homomorphisms $g_1, g_2$ send (\ref{6.3.5}.1) respectively to
\begin{align*}
&\gamma_{\{65, 84, 109, 116\}}[1,3,29]  +   a_{11}[3,5,25]  +    \gamma_{65}[3,15,15] +   \gamma_{89}[15,3,15]\\
&\quad   +   \gamma_{115}[15,15,3]  +   \gamma_{93}[7,11,15] +   \gamma_{\{84, 98\}}[7,15,11]  +   \gamma_{99}[15,7,11]  =0,\\
&\gamma_{\{66, 68, 86, 111, 132\}}[1,3,29]  +   a_{12}[3,5,25]  +   \gamma_{66}[3,15,15]\\
&\quad  +   \gamma_{90}[15,3,15]  +   \gamma_{\{114, 128\}}[15,15,3]  +   \gamma_{94}[7,11,15]\\
&\quad  +   \gamma_{\{83, 86, 97, 126, 127, 130\}}[7,15,11]  +   \gamma_{103}[15,7,11] = 0,
\end{align*}
where
\begin{align*}
a_{11} &= \gamma_{\{59, 84, 89, 93, 98, 99, 108, 109, 115\}},\\
a_{12} &= \gamma_{\{60, 63, 68, 90, 94, 97, 103, 110, 111, 114, 126, 127, 128, 130\}}.
\end{align*}

From this, it implies
\begin{equation}\begin{cases}
 \gamma_j = 0,\ j = 65, 66, 89, 90,93, 94, 99, 103, 115,\\ 
a_{11} = a_{12} = \gamma_{\{65, 84, 109, 116\}} =   \gamma_{\{84, 98\}} =  
\gamma_{\{114, 128\}} = 0, \\  
\gamma_{\{66, 68, 86, 111, 132\}} = 
\gamma_{\{83, 86, 97, 126, 127, 130\}} = 0.  
\end{cases}\tag{\ref{6.3.5}.4}
\end{equation}

With the aid of (\ref{6.3.5}.2), (\ref{6.3.5}.3) and (\ref{6.3.5}.4), the homomorphisms $g_3, g_4$ send (\ref{6.3.5}.1) respectively to
\begin{align*}
&\gamma_{\{67, 70, 88, 112, 119\}}[1,3,29]  +   \gamma_{\{116, 119, 132\}}[3,29,1]\\
&\quad +    a_{13}[3,5,25]  +   \gamma_{67}[3,15,15]  +   \gamma_{91}[15,3,15]  +   a_{14}[15,15,3]\\
&\quad +   \gamma_{95}[7,11,15]  +   a_{15}[7,15,11]  +   \gamma_{\{102, 105, 130, 131\}}[15,7,11]  = 0,\\   
&\gamma_{\{67, 70, 88, 112, 119\}}[1,3,29]  +  \gamma_{\{116, 119, 132\}}[3,29,1]\\
&\quad +   a_{16}[3,5,25]  +   \gamma_{70}[3,15,15]  +   \gamma_{92}[15,3,15]  +  a_{17}[15,15,3]\\
&\quad  +   \gamma_{96}[7,11,15]  +   a_{18}[7,15,11]  +   \gamma_{\{104, 105, 106, 130, 131\}}[15,7,11] =0,
\end{align*}
where
\begin{align*}
a_{13} &= \gamma_{\{59, 60, 63, 68, 91, 92, 95, 96, 100, 101, 102, 104, 105, 106,108, 110, 112, 113, 126, 127, 128\}},\\
a_{14} &= \gamma_{\{68, 70, 92, 96, 101, 104, 113, 127\}},\ \
a_{15} = \gamma_{\{83, 84, 88, 100, 106, 126, 128, 130, 131\}},
\end{align*}
\begin{align*}
a_{16} &= \gamma_{\{59, 60, 63, 68, 91, 92, 95, 96, 97, 98, 100, 101, 102, 104, 105, 106, 109, 111, 112, 113, 114, 126, 127, 128\}},\\
a_{17} &=  \gamma_{\{59, 60, 63, 67, 91, 95, 97, 98, 100, 102, 113, 114, 126\}},\ \
a_{18} = \gamma_{\{84, 86, 88, 101, 127, 128, 130, 131\}}.
\end{align*}

From this, it implies
\begin{equation}\begin{cases}
a_{13} = a_{14} = a_{15} = a_{16} = a_{17} = a_{18} = 0,\\
\gamma_j = 0, \ j = 67, 91, 95, 70, 92, 96,\\
\gamma_{\{67, 70, 88, 112, 119\}} =  
\gamma_{\{116, 119, 132\}} =    
\gamma_{\{102, 105, 130, 131\}} = 0,\\  
\gamma_{\{116, 119, 132\}} =  
\gamma_{\{67, 70, 88, 112, 119\}} =   
\gamma_{\{104, 105, 106, 130, 131\}} = 0.
\end{cases}\tag{\ref{6.3.5}.5}
\end{equation}
With the aid of (\ref{6.3.5}.2), (\ref{6.3.5}.3), (\ref{6.3.5}.4) and (\ref{6.3.5}.5), the homomorphisms $h$ sends (\ref{6.3.5}.1)  to
\begin{align*}
&a_{19}[3,5,25] + a_{20}[3,15,15] +   \gamma_{88}[15,3,15] +  \gamma_{105}[15,15,3]\\
&\quad+  a_{21}[7,11,15]  + \gamma_{\{102, 104, 105, 130, 131\}}[7,15,11] +  \gamma_{106}[15,7,11] = 0,   
\end{align*}
where
\begin{align*}
a_{19} &= \gamma_{\{86, 88, 101, 104, 105, 106, 110, 111, 112, 114, 127, 130, 131\}},\\  
a_{20} &= \gamma_{\{84, 86, 101, 104, 113, 127\}},\ \ 
a_{21} = \gamma_{\{83, 84, 86, 88, 97, 100, 101, 106, 114, 126, 127, 130, 131\}}.  
\end{align*}

Hence, we obtain
\begin{equation}\begin{cases}
a_{19} = a_{20} = a_{21} = \gamma_{88} = \gamma_{105} = \gamma_{106} =  0,\\
  \gamma_{\{102, 104, 105, 130, 131\}} = 0.    
\end{cases} \tag{\ref{6.3.5}.6}
\end{equation}

Combining (\ref{6.3.5}.2),  (\ref{6.3.5}.3), (\ref{6.3.5}.4), (\ref{6.3.5}.5) and (\ref{6.3.5}.6)  gives
\begin{equation}\begin{cases}
\gamma_j = 0,\ j = 83, 84, 86, 87, 97, 98, 108, 109,\\
\hskip2.3cm 112, 119, 129, 130, 131, 132, 133, 135.\\
\gamma_{83} = \gamma_{84} = \gamma_{86} = \gamma_{87} = \gamma_{97} = \gamma_{98} = \gamma_{108} = \gamma_{109} = \gamma_{112}\\
\qquad = \gamma_{119} = \gamma_{129} = \gamma_{130} = \gamma_{131} = \gamma_{132} = \gamma_{133} = \gamma_{135}.
\end{cases} \tag{\ref{6.3.5}.7}
\end{equation}

Substituting (\ref{6.3.5}.7) into the relation (\ref{6.3.5}.1), we get
\begin{equation} 
\gamma_{83}[\theta] = 0, \tag{\ref{6.3.5}.8} 
\end{equation}
where $\theta = a_{3,1,83} + a_{3,1,84} + a_{3,1,86} + a_{3,1,87} + a_{3,1,97} + a_{3,1,98}
 + a_{3,1,108} + a_{3,1,109} + a_{3,1,112} + a_{3,1,119} + a_{3,1,129} 
 + a_{3,1,130} + a_{3,1,131} + a_{3,1,132} + a_{3,1,133} + a_{3,1,135}.$

By a direct computation, we can show that the polynomial $\theta$ is non-hit. Hence, $[\theta] \ne 0$ and $\gamma_{83} = 0$.  The proposition follows. 
\end{proof}

\begin{props}\label{6.3.7} The elements $[a_{3,2,j}], \ 53 \leqslant j \leqslant 180,$ are linearly independent in $(\mathbb F_2\underset {\mathcal A}\otimes R_4)_{69}$.
\end{props}

\begin{proof} Suppose there is a linear relation
\begin{equation}\sum_{j=53}^{180}\gamma_j[a_{3,2,j}]=0,\tag{\ref{6.3.7}.1}
\end{equation}
where $\gamma_j \in \mathbb F_2.$ 

Apply the homomorphisms $f_i, 1 \leqslant i \leqslant 6,$ to  (\ref{6.3.7}.1) and we get
\begin{align*}
&\gamma_{135}[3,3,63]  +   \gamma_{136}[3,63,3]  +    \gamma_{121}[63,3,3]  +   \gamma_{147}[3,7,59]  +   \gamma_{117}[7,3,59] \\
&\quad +   \gamma_{118}[7,59,3]  +   \gamma_{107}[7,11,51]  +   \gamma_{53}[7,31,31]  +   \gamma_{61}[31,7,31]\\
&\quad  +   \gamma_{62}[31,31,7]  +   \gamma_{71}[15,23,31]  +   \gamma_{72}[15,31,23]  +   \gamma_{85}[31,15,23]  = 0,\\   
&\gamma_{137}[3,3,63]  +  \gamma_{139}[3,63,3]  +   \gamma_{\{120, 176\}}[63,3,3]\\
&\quad +   \gamma_{148}[3,7,59]  +   \gamma_{\{116, 177\}}[7,3,59]  +   \gamma_{155}[7,59,3]\\
&\quad  +   \gamma_{108}[7,11,51]  +   \gamma_{54}[7,31,31]  +   \gamma_{59}[31,7,31]  +   \gamma_{64}[31,31,7]\\
&\quad  +   \gamma_{75}[15,23,31]  +   \gamma_{73}[15,31,23]  +   \gamma_{84}[31,15,23] =0,\\   
&\gamma_{138}[3,3,63]  +  \gamma_{140}[3,63,3]  +   \gamma_{\{119, 175, 180\}}[63,3,3]\\
&\quad  +   \gamma_{149}[3,7,59]  + \gamma_{\{153, 178\}}[7,3,59]  +   \gamma_{\{154, 179\}}[7,59,3]\\
&\quad  +    \gamma_{110}[7,11,51]  +   \gamma_{55}[7,31,31]  +   \gamma_{60}[31,7,31]  +   \gamma_{63}[31,31,7]\\
&\quad  +   \gamma_{76}[15,23,31]  +   \gamma_{77}[15,31,23]  +   \gamma_{83}[31,15,23] = 0,\\   
&\gamma_{141}[3,3,63]  +  \gamma_{\{123, 170, 174, 176\}}[3,63,3]  +   \gamma_{144}[63,3,3]  +   a_1[3,7,59]\\
&\quad  +   \gamma_{150}[7,3,59]  +   \gamma_{159}[7,59,3]  +   \gamma_{109}[7,11,51]\\
&\quad   +   \gamma_{65}[7,31,31]  +   \gamma_{56}[31,7,31]  +   \gamma_{69}[31,31,7]\\
&\quad   +   \gamma_{78}[15,23,31]  +   \gamma_{87}[15,31,23]  +   \gamma_{74}[31,15,23] = 0,\\   
&\gamma_{142}[3,3,63]  +  \gamma_{\{122, 171, 173, 175, 179\}}[3,63,3]  +   \gamma_{145}[63,3,3]\\
&\quad   +   a_2[3,7,59]  +   \gamma_{151}[7,3,59]  +   \gamma_{\{158, 180\}}[7,59,3] \\
&\quad  +   \gamma_{111}[7,11,51] +   \gamma_{66}[7,31,31]  +   \gamma_{57}[31,7,31]  +   \gamma_{68}[31,31,7]\\
&\quad   +    \gamma_{79}[15,23,31]  +   \gamma_{86}[15,31,23]  +   \gamma_{81}[31,15,23]  = 0,\\   
&a_3[3,3,63]  +  \gamma_{143}[3,63,3]  +   \gamma_{146}[63,3,3]  +   a_4[3,7,59]  +  a_5[7,3,59]\\
&\quad   +   \gamma_{152}[7,59,3]  +   a_6[7,11,51]  +   \gamma_{67}[7,31,31]  +   \gamma_{70}[31,7,31] \\
&\quad  +   \gamma_{58}[31,31,7]  +   \gamma_{88}[15,23,31]  +   \gamma_{80}[15,31,23]  +   \gamma_{82}[31,15,23] =0, 
\end{align*}
where
\begin{align*}
a_1 &= \gamma_{\{89, 93, 98, 99, 115, 126, 127, 161, 165, 177\}},\ 
a_2 = \gamma_{\{90, 94, 97, 103, 114, 128, 156, 162, 166, 178\}},\\
a_3 &= \gamma_{\{124, 169, 172, 175, 176, 177, 178\}},\ 
a_4 = \gamma_{\{91, 95, 100, 102, 113, 129, 157, 163, 167, 179\}},\\
a_5 &=  \gamma_{\{92, 96, 101, 104, 125, 130, 160, 164, 168, 180\}},\ \
a_6 = \gamma_{\{105, 106, 112, 131, 132, 133, 134\}}.
\end{align*}

From these equalities, we get
\begin{equation}\begin{cases}
a_1 = a_2 = a_3 = a_4 = a_5 = a_6 = 0,\\
\gamma_j = 0,\ j = 53, \ldots , 88;\ 107, \ldots, 111,\\ 
117,118, 121, 135, \ldots, 152, 155, 159,\\
\gamma_{\{120, 176\}} =  
\gamma_{\{116, 177\}} =    
\gamma_{\{119, 175, 180\}} =   
\gamma_{\{153, 178\}} =  0,\\
\gamma_{\{154, 179\}} = 
\gamma_{\{123, 170, 174, 176\}} = 0,\\  
\gamma_{\{122, 171, 173, 175, 179\}} =   
\gamma_{\{158, 180\}} = 0.               
\end{cases} \tag{\ref{6.3.7}.2}
\end{equation}
With the aid of (\ref{6.3.7}.2), the homomorphisms $g_1, g_2$ send (\ref{6.3.7}.1) respectively to
\begin{align*}
&\gamma_{123}[3,63,3] +  \gamma_{170}[63,3,3] +   \gamma_{126}[3,7,59] +  \gamma_{161}[7,3,59]\\
&\quad  +  \gamma_{174}[7,59,3] +  \gamma_{165}[7,11,51] +  \gamma_{127}[7,31,31] +  \gamma_{89}[31,7,31] \\
&\quad +  \gamma_{115}[31,31,7] +  \gamma_{93}[15,23,31] +  \gamma_{98}[15,31,23] +  \gamma_{99}[31,15,23] =0,\\
&\gamma_{\{119, 122, 175\}}[3,63,3] +  \gamma_{171}[63,3,3] +  \gamma_{156}[3,7,59] +  \gamma_{162}[7,3,59]\\
&\quad +  \gamma_{173}[7,59,3] +  \gamma_{166}[7,11,51] +  \gamma_{128}[7,31,31] +  \gamma_{90}[31,7,31]\\
&\quad +  \gamma_{114}[31,31,7] +  \gamma_{94}[15,23,31] +  \gamma_{97}[15,31,23] +  \gamma_{103}[31,15,23] =0.
\end{align*}

From these equalities, we obtain
\begin{equation}\begin{cases}
\gamma_j = 0,\ j = 89, 90, 93, 94, 97, 98, 99, 103, 114, 115, 123,\\
 126, 127, 128, 156, 161, 162, 165, 166, 170, 171, 173, 174,\\
\gamma_{\{119, 122, 175\}} = 0.
\end{cases} \tag{\ref{6.3.7}.3}
\end{equation}

With the aid of (\ref{6.3.7}.2) and (\ref{6.3.7}.3), the homomorphisms $g_3, g_4$ send (\ref{6.3.7}.1) respectively to
\begin{align*}
&a_7[3,63,3]  +  \gamma_{\{164, 169\}}[63,3,3]  +    \gamma_{\{112, 157\}}[3,7,59]  +   \gamma_{163}[7,3,59]\\
&\quad  +   \gamma_{\{158, 160, 164, 172\}}[7,59,3]  +   \gamma_{\{112, 132, 167\}}[7,11,51]\\
&\quad  +   \gamma_{129}[7,31,31]  +   \gamma_{91}[31,7,31]  +   \gamma_{\{113, 132\}}[31,31,7]\\
&\quad  +   \gamma_{95}[15,23,31]  +   \gamma_{100}[15,31,23]  +   \gamma_{102}[31,15,23]  = 0,\\   
&a_8[3,63,3]  +  \gamma_{\{120, 163, 169\}}[63,3,3]  +   \gamma_{\{112, 160\}}[3,7,59] +   \gamma_{164}[7,3,59]\\
&\quad   +   a_9[7,59,3]  +  a_{10}[7,11,51]  +   \gamma_{130}[7,31,31]  +   \gamma_{92}[31,7,31] \\
&\quad +   a_{11}[31,31,7]  +   \gamma_{96}[15,23,31]  +   \gamma_{101}[15,31,23]  +   \gamma_{104}[31,15,23]  =0,        
\end{align*}
where
\begin{align*}
a_7 &= \gamma_{\{119, 124, 158, 160, 175\}},\ \
a_8 = \gamma_{\{122, 124, 154, 157, 175\}},\\
a_{9} &= \gamma_{\{116, 153, 154, 157, 163, 172\}},\ \
a_{10} =  \gamma_{\{105, 106, 112, 131, 132, 133, 134, 168\}},\\
a_{11} &= \gamma_{\{105, 106, 125, 131, 132, 133, 134\}}.
\end{align*}

Computing directly from these equalities we get
\begin{equation}\begin{cases}
a_7 = a_8 = a_9 = a_{10} = a_{11},\ \
\gamma_j = 0, \ j = 91, 92, 95, 96,\\ 100, 101, 102, 104, 129, 130, 163, 164,\\
\gamma_{\{164, 169\}} =  
\gamma_{\{112, 157\}} = 
\gamma_{\{158, 160, 164, 172\}} =0,\\
\gamma_{\{112, 132, 167\}} = 
\gamma_{\{113, 132\}} = 
 \gamma_{\{120, 163, 169\}} = 
\gamma_{\{112, 160\}}  =0.
\end{cases} \tag{\ref{6.3.7}.4}
\end{equation}

With the aid of (\ref{6.3.7}.2), (\ref{6.3.7}.3) and (\ref{6.3.7}.4), the homomorphism $h$ sends (\ref{6.3.7}.1) to
\begin{align*}
&\gamma_{\{112, 119, 122, 124, 175\}}[3,3,63] + \gamma_{\{112, 172\}}[3,7,59] +   \gamma_{\{112, 167, 168\}}[7,11,51]\\
&\quad +  \gamma_{125}[7,31,31] +  \gamma_{131}[31,7,31] +  \gamma_{105}[31,31,7]\\
&\quad +  \gamma_{133}[15,23,31] +  \gamma_{134}[15,31,23] +  \gamma_{106}[31,15,23] = 0. 
\end{align*}

From this, it implies
\begin{equation}\begin{cases}
\gamma_j = 0,\ j = 105, 106, 125, 131, 133, 134,\\
\gamma_{\{112, 119, 122, 124, 175\}} =
\gamma_{\{112, 172\}} =  
\gamma_{\{112, 167, 168\}} =0.
\end{cases} \tag{\ref{6.3.7}.5}
\end{equation}

Combining (\ref{6.3.7}.2), (\ref{6.3.7}.3), (\ref{6.3.7}.4) and (\ref{6.3.7}.5), we obtain $\gamma_j = 0$ for all $j$. The proposition is proved.
\end{proof}

\begin{props}\label{6.3.8} For $s \geqslant 3$, the elements $[a_{3,s,j}], \ 53 \leqslant j \leqslant 195,$ are linearly independent in $(\mathbb F_2\underset {\mathcal A}\otimes R_4)_{2^{s+4} + 2^{s+1} -3}$.
\end{props}

\begin{proof} Suppose there is a linear relation
\begin{equation}\sum_{j=53}^{195}\gamma_j[a_{3,s,j}]=0,\tag{\ref{6.3.8}.1}
\end{equation}
where $\gamma_j \in \mathbb F_2.$ 

Applying the homomorphisms $f_i, 1 \leqslant i \leqslant 6,$ to the relation (\ref{6.3.8}.1), we obtain
\begin{align*}
&\gamma_{135}w_{3,s,1} +  \gamma_{136}w_{3,s,2} +   \gamma_{121}w_{3,s,3} +  \gamma_{147}w_{3,s,4} +  \gamma_{117}w_{3,s,5}\\
&\quad +  \gamma_{118}w_{3,s,6} +  \gamma_{107}w_{3,s,7} +  \gamma_{53}w_{3,s,8} +  \gamma_{61}w_{3,s,9}\\
&\quad +  \gamma_{62}w_{3,s,10} +  \gamma_{71}w_{3,s,11} +  \gamma_{72}w_{3,s,12} +  \gamma_{85}w_{3,s,13} = 0,\\ 
&\gamma_{137}w_{3,s,1} +  \gamma_{139}w_{3,s,2} +  \gamma_{120}w_{3,s,3} +  \gamma_{148}w_{3,s,4} +  \gamma_{116}w_{3,s,5}\\
&\quad +  \gamma_{155}w_{3,s,6} +  \gamma_{108}w_{3,s,7} +  \gamma_{54}w_{3,s,8} +  \gamma_{59}w_{3,s,9}\\
&\quad +  \gamma_{64}w_{3,s,10} +  \gamma_{75}w_{3,s,11} +  \gamma_{73}w_{3,s,12} +  \gamma_{84}w_{3,s,13} = 0,\\  
&\gamma_{138}w_{3,s,1} + \gamma_{140}w_{3,s,2} +  a_1w_{3,s,3} +  \gamma_{149}w_{3,s,4} + \gamma_{153}w_{3,s,5}\\
&\quad +  \gamma_{154}w_{3,s,6} +   \gamma_{110}w_{3,s,7} +  \gamma_{55}w_{3,s,8} +  \gamma_{60}w_{3,s,9}\\
&\quad +  \gamma_{63}w_{3,s,10} +  \gamma_{76}w_{3,s,11} +  \gamma_{77}w_{3,s,12} +  \gamma_{83}w_{3,s,13} = 0,\\  
&\gamma_{141}w_{3,s,1} + \gamma_{123}w_{3,s,2} +  \gamma_{144}w_{3,s,3} +  \gamma_{126}w_{3,s,4} +  \gamma_{150}w_{3,s,5}\\
&\quad +  \gamma_{159}w_{3,s,6} +  \gamma_{109}w_{3,s,7} +  \gamma_{65}w_{3,s,8} +  \gamma_{56}w_{3,s,9}\\
&\quad +  \gamma_{69}w_{3,s,10} +  \gamma_{78}w_{3,s,11} +  \gamma_{87}w_{3,s,12} +  \gamma_{74}w_{3,s,13} = 0,\\  
&\gamma_{142}w_{3,s,1} +  a_2w_{3,s,2} +  \gamma_{145}w_{3,s,3} +  \gamma_{156}w_{3,s,4} +  \gamma_{151}w_{3,s,5}\\
&\quad +  \gamma_{158}w_{3,s,6} +  \gamma_{111}w_{3,s,7} +  \gamma_{66}w_{3,s,8} + \gamma_{57}w_{3,s,9}\\
&\quad +  \gamma_{68}w_{3,s,10} +   \gamma_{79}w_{3,s,11} +  \gamma_{86}w_{3,s,12} +  \gamma_{81}w_{3,s,13} = 0,\\  
&a_3w_{3,s,1} + \gamma_{143}w_{3,s,2} +  \gamma_{146}w_{3,s,3} +  \gamma_{157}w_{3,s,4} +  \gamma_{160}w_{3,s,5}\\
&\quad +  \gamma_{152}w_{3,s,6} +  \gamma_{112}w_{3,s,7} +  \gamma_{67}w_{3,s,8} +  \gamma_{70}w_{3,s,9}\\
&\quad +  \gamma_{58}w_{3,s,10} +  \gamma_{88}w_{3,s,11} +  \gamma_{80}w_{3,s,12} +  \gamma_{82}w_{3,s,13} = 0,
\end{align*}
where
\begin{align*}
a_1 &= \begin{cases}  \gamma_{\{119, 195\}}, & s = 3,\\   \gamma_{119}, & s \geqslant 4, \end{cases} \ \
a_2 = \begin{cases}  \gamma_{\{122, 195\}}, & s = 3,\\   \gamma_{122}, & s \geqslant 4, \end{cases} \\
a_3 &= \begin{cases}  \gamma_{\{124, 193, 194, 195\}}, & s = 3,\\   \gamma_{124}, & s \geqslant 4. \end{cases} 
\end{align*}
From these equalities, we get
\begin{equation}\begin{cases}
\gamma_j = 0,\ j =53, \ldots , 88;\ 107, \ldots , 112;\ 116, \\
 117, 118,\ 120, 121, 123;\ 135, \ldots , 160,\\
a_1 = a_2 = a_3 = 0.
\end{cases} \tag{\ref{6.3.8}.2}
\end{equation}

With the aid of (\ref{6.3.8}.2)  the homomorphisms $g_j , 1 \leqslant j \leqslant 4,$ send (\ref{6.3.8}.1)  to
\begin{align*}
&\gamma_{177}w_{3,s,1} + \gamma_{182}w_{3,s,2} +  \gamma_{170}w_{3,s,3} +  \gamma_{185}w_{3,s,4} +  \gamma_{161}w_{3,s,5}\\
&\quad +  \gamma_{174}w_{3,s,6} +  \gamma_{165}w_{3,s,7} +  \gamma_{127}w_{3,s,8} +  \gamma_{89}w_{3,s,9}\\
&\quad +  \gamma_{115}w_{3,s,10} +  \gamma_{93}w_{3,s,11} +  \gamma_{98}w_{3,s,12} + \gamma_{99}w_{3,s,13} = 0,\\  
&\gamma_{178}w_{3,s,1} + a_4w_{3,s,2} +  \gamma_{171}w_{3,s,3} +  \gamma_{186}w_{3,s,4} +  \gamma_{162}w_{3,s,5}\\
&\quad +  \gamma_{173}w_{3,s,6} +  \gamma_{166}w_{3,s,7} +  \gamma_{128}w_{3,s,8} +  \gamma_{90}w_{3,s,9}\\
&\quad +  \gamma_{114}w_{3,s,10} +  \gamma_{94}w_{3,s,11} +  \gamma_{97}w_{3,s,12} +  \gamma_{103}w_{3,s,13} = 0,\\  &\gamma_{179}w_{3,s,1} + a_5w_{3,s,2} +  a_6w_{3,s,3} +  \gamma_{187}w_{3,s,4} +  \gamma_{163}w_{3,s,5}\\
&\quad +  a_7w_{3,s,6} +  \gamma_{167}w_{3,s,7} +  \gamma_{129}w_{3,s,8} +  \gamma_{91}w_{3,s,9}\\
&\quad +  \gamma_{113}w_{3,s,10} +  \gamma_{95}w_{3,s,11} +  \gamma_{100}w_{3,s,12} +  \gamma_{102}w_{3,s,13} = 0,\\  
&\gamma_{180}w_{3,s,1} + a_8w_{3,s,2} +  a_9w_{3,s,3} + \gamma_{188}w_{3,s,4} +  \gamma_{164}w_{3,s,5}\\
&\quad +   a_{10}w_{3,s,6} +  \gamma_{168}w_{3,s,7} +  \gamma_{130}w_{3,s,8} +  \gamma_{92}w_{3,s,9}\\
&\quad +  \gamma_{125}w_{3,s,10} +  \gamma_{96}w_{3,s,11} +  \gamma_{101}w_{3,s,12} +  \gamma_{104}w_{3,s,13} = 0,  
\end{align*}
where
\begin{align*}
a_4 & = \begin{cases} \gamma_{\{119, 181\}}, & s = 3,\\   \gamma_{181}, & s \geqslant 4, \end{cases}\ \
a_5  = \begin{cases} \gamma_{\{124, 183\}}, & s = 3,\\   \gamma_{183}, & s \geqslant 4, \end{cases}\\
a_6 & = \begin{cases} \gamma_{\{169, 193\}}, & s = 3,\\   \gamma_{169}, & s \geqslant 4, \end{cases}\ \
a_7  = \begin{cases} \gamma_{\{172, 194\}}, & s = 3,\\   \gamma_{172}, & s \geqslant 4, \end{cases}\\
a_8 & = \begin{cases} \gamma_{\{124, 184, 189, 190\}}, & s = 3,\\    \gamma_{184}, & s \geqslant 4, \end{cases}\ \
a_9  = \begin{cases} \gamma_{\{176, 189, 191, 193\}}, & s = 3,\\   \gamma_{176}, & s \geqslant 4, \end{cases}\\
a_{10} & = \begin{cases} \gamma_{\{175, 190, 191, 194\}}, & s = 3,\\   \gamma_{175}, & s \geqslant 4. \end{cases}
\end{align*}
Hence, we obtain
\begin{equation}\begin{cases}
a_i = 0, \ i = 4,5, \ldots , 10,\\
\gamma_j = 0, j = 89, \ldots, 104; 113, 114, 115, 125, 127, 128, \\
\hskip 2cm 129, 130, 161, \ldots, 168, 170, 171, 173, 174,\\ 
\hskip 2cm 177, 178, 179, 180, 182, 185, 186, 187, 188.
\end{cases}\tag{\ref{6.3.8}.3}
\end{equation}

With the aid of (\ref{6.3.8}.2) and (\ref{6.3.8}.3), the homomorphism $h$ sends (\ref{6.3.8}.1) to
\begin{align*}
 &a_{11}w_{3,s,1} + a_{12}w_{3,s,2} +  \gamma_{189}w_{3,s,3} +  a_{13}w_{3,s,4} +  \gamma_{190}w_{3,s,5}\\
&\quad +  \gamma_{191}w_{3,s,6} +  \gamma_{192}w_{3,s,7} +  \gamma_{132}w_{3,s,8} +  \gamma_{131}w_{3,s,9}\\
&\quad +  \gamma_{105}w_{3,s,10} +  \gamma_{133}w_{3,s,11} +  \gamma_{134}w_{3,s,12} +  \gamma_{106}w_{3,s,13} = 0,
\end{align*}
where
$$ 
a_{11} = \begin{cases}  \gamma_{184}, & s = 3,\\   \gamma_{195}, & s \geqslant 4, \end{cases} \ \  
a_{12} = \begin{cases}  \gamma_{176}, & s = 3,\\  \gamma_{193}, & s \geqslant 4, \end{cases} \ \   
a_{13} = \begin{cases}  \gamma_{175}, & s = 3,\\   \gamma_{194}, & s \geqslant 4. \end{cases}
$$
From this, it implies
\begin{equation}\begin{cases}
a_{11} = a_{12} = a_{13} = 0,\\ 
\gamma_j = 0,\ j = 105, 106, 131,\ldots, 134, 189,\ldots, 192.
\end{cases}\tag{\ref{6.3.8}.4}
\end{equation}

From (\ref{6.3.8}.2), (\ref{6.3.8}.3) and (\ref{6.3.8}.4), we see that $\gamma_j = 0$ for all $j$. So, the proposition is proved.
\end{proof}

\begin{rems}\label{6.3.6} The element $[\theta]$ defined as in the proof of Proposition \ref{6.3.5}, is an $GL_4(\mathbb F_2)$-invariant in $(\mathbb F_2 \underset{\mathcal A}\otimes P_4)_{33}$.
\end{rems}

\subsection{The case $t \geqslant 4$}\label{6.4}\

\medskip
According to Kameko \cite{ka}, for $t \geqslant 4$, $\dim (\mathbb F_2\underset{\mathcal A}\otimes P_3)_{2^{s+t+1}+2^{s+1}-3} =14$ with a basis given by the following classes: 

\medskip
\centerline{\begin{tabular}{l}
$w_{t,s,1} = [2^s - 1,2^s - 1,2^{s+t+1} - 1]$,\cr 
$w_{t,s,2} = [2^s - 1,2^{s+t+1} - 1,2^s - 1]$,\cr 
$w_{t,s,3} = [2^{s+t+1} - 1,2^s - 1,2^s - 1]$,\cr 
$w_{t,s,4} = [2^s - 1,2^{s+1} - 1,2^{s+t+1}-2^s - 1]$,\cr 
$w_{t,s,5} = [2^{s+1} - 1,2^s - 1,2^{s+t+1}-2^s - 1]$,\cr 
$w_{t,s,6} = [2^{s+1} - 1,2^{s+t+1}-2^s - 1,2^s - 1]$,\cr 
$w_{t,s,7} = [2^{s+1} - 1,3.2^s - 1,2^{s+t+1}-3.2^s - 1]$,\cr 
$w_{t,s,8} = [2^{s+1} - 1,2^{s+t} - 1,2^{s+t} - 1]$,\cr 
$w_{t,s,9} = [2^{s+t} -1,2^{s+1} - 1,2^{s+t} - 1]$,\cr 
$w_{t,s,10} = [2^{s+t} -1,2^{s+t} - 1,2^{s+1} - 1]$,\cr 
$w_{t,s,11} = [2^{s+2} - 1,2^{s+t}-2^{s+1} - 1,2^{s+t} - 1]$,\cr 
$w_{t,s,12} = [2^{s+2} - 1,2^{s+t} - 1,2^{s+t}-2^{s+1} - 1]$,\cr 
$w_{t,s,13} = [2^{s+t} -1,2^{s+2} - 1,2^{s+t}-2^{s+1} - 1]$,\cr 
$w_{t,s,14} = [2^{s+3} - 1,2^{s+t}-2^{s+2} - 1,2^{s+t}-2^{s+1} - 1]$.\cr
\end{tabular}}

\medskip
So we easily obtain

\begin{props}\label{6.4.2} For any $t \geqslant 4$, $(\mathbb F_2\underset{\mathcal A}\otimes Q_4)_{2^{s+t+1} +2^{s+1}-3}$ is  an $\mathbb F_2$-vector space of dimension 56 with a basis consisting of all the  classes represented by the following monomials: 

\medskip
\centerline{\begin{tabular}{ll}
&$a_{t,s,1} =  (0,2^{s+1} - 1,2^{s+t} - 1,2^{s+t} - 1),$\cr 
&$a_{t,s,2} =  (0,2^{s+t} - 1,2^{s+1} - 1,2^{s+t} - 1),$\cr 
&$a_{t,s,3} =  (0,2^{s+t} - 1,2^{s+t} - 1,2^{s+1} - 1),$\cr 
&$a_{t,s,4} =  (2^{s+1} - 1,0,2^{s+t} - 1,2^{s+t} - 1),$\cr 
&$a_{t,s,5} =  (2^{s+1} - 1,2^{s+t} - 1,0,2^{s+t} - 1),$\cr 
&$a_{t,s,6} =  (2^{s+1} - 1,2^{s+t} - 1,2^{s+t} - 1,0),$\cr 
&$a_{t,s,7} =  (2^{s+t} - 1,0,2^{s+1} - 1,2^{s+t} - 1),$\cr 
&$a_{t,s,8} =  (2^{s+t} - 1,0,2^{s+t} - 1,2^{s+1} - 1),$\cr 
&$a_{t,s,9} =  (2^{s+t} - 1,2^{s+1} - 1,0,2^{s+t} - 1),$\cr 
&$a_{t,s,10} =  (2^{s+t} - 1,2^{s+1} - 1,2^{s+t} - 1,0),$\cr 
&$a_{t,s,11} =  (2^{s+t} - 1,2^{s+t} - 1,0,2^{s+1} - 1),$\cr 
&$a_{t,s,12} =  (2^{s+t} - 1,2^{s+t} - 1,2^{s+1} - 1,0),$\cr 
&$a_{t,s,13} =  (0,2^{s+2} - 1,2^{s+t}-2^{s+1} - 1,2^{s+t} - 1),$\cr 
&$a_{t,s,14} =  (0,2^{s+2} - 1,2^{s+t} - 1,2^{s+t}-2^{s+1} - 1),$\cr 
&$a_{t,s,15} =  (0,2^{s+t} - 1,2^{s+2} - 1,2^{s+t}-2^{s+1} - 1),$\cr 
&$a_{t,s,16} =  (2^{s+2} - 1,0,2^{s+t}-2^{s+1} - 1,2^{s+t} - 1),$\cr 
&$a_{t,s,17} =  (2^{s+2} - 1,0,2^{s+t} - 1,2^{s+t}-2^{s+1} - 1),$\cr 
&$a_{t,s,18} =  (2^{s+2} - 1,2^{s+t}-2^{s+1} - 1,0,2^{s+t} - 1),$\cr 
&$a_{t,s,19} =  (2^{s+2} - 1,2^{s+t}-2^{s+1} - 1,2^{s+t} - 1,0),$\cr 
&$a_{t,s,20} =  (2^{s+2} - 1,2^{s+t} - 1,0,2^{s+t}-2^{s+1} - 1),$\cr 
&$a_{t,s,21} =  (2^{s+2} - 1,2^{s+t} - 1,2^{s+t}-2^{s+1} - 1,0),$\cr 
&$a_{t,s,22} =  (2^{s+t} - 1,0,2^{s+2} - 1,2^{s+t}-2^{s+1} - 1),$\cr 
&$a_{t,s,23} =  (2^{s+t} - 1,2^{s+2} - 1,0,2^{s+t}-2^{s+1} - 1),$\cr 
&$a_{t,s,24} =  (2^{s+t} - 1,2^{s+2} - 1,2^{s+t}-2^{s+1} - 1,0),$\cr 
&$a_{t,s,25} =  (0,2^{s+3} - 1,2^{s+t}-2^{s+2} - 1,2^{s+t}-2^{s+1} - 1),$\cr 
&$a_{t,s,26} =  (2^{s+3} - 1,0,2^{s+t}-2^{s+2} - 1,2^{s+t}-2^{s+1} - 1),$\cr 
&$a_{t,s,27} =  (2^{s+3} - 1,2^{s+t}-2^{s+2} - 1,0,2^{s+t}-2^{s+1} - 1),$\cr 
&$a_{t,s,28} =  (2^{s+3} - 1,2^{s+t}-2^{s+2} - 1,2^{s+t}-2^{s+1} - 1,0),$\cr 
&$a_{t,s,29} =  (0,2^s - 1,2^s - 1,2^{s+t+1} - 1),$\cr 
&$a_{t,s,30} =  (0,2^s - 1,2^{s+t+1} - 1,2^s - 1),$\cr 
&$a_{t,s,31} =  (0,2^{s+t+1} - 1,2^s - 1,2^s - 1),$\cr 
&$a_{t,s,32} =  (2^{s} - 1,0,2^s - 1,2^{s+t+1} - 1),$\cr 
&$a_{t,s,33} =  (2^{s} - 1,0,2^{s+t+1} - 1,2^s - 1),$\cr 
&$a_{t,s,34} =  (2^{s} - 1,2^s - 1,0,2^{s+t+1} - 1),$\cr 
&$a_{t,s,35} =  (2^{s} - 1,2^s - 1,2^{s+t+1} - 1,0),$\cr 
&$a_{t,s,36} =  (2^{s} - 1,2^{s+t+1} - 1,0,2^s - 1),$\cr 
&$a_{t,s,37} =  (2^{s} - 1,2^{s+t+1} - 1,2^s - 1,0),$\cr 
&$a_{t,s,38} =  (2^{s+t+1} - 1,0,2^s - 1,2^s - 1),$\cr 
&$a_{t,s,39} =  (2^{s+t+1} - 1,2^s - 1,0,2^s - 1),$\cr 
&$a_{t,s,40} =  (2^{s+t+1} - 1,2^s - 1,2^s - 1,0),$\cr 
\end{tabular}}
\centerline{\begin{tabular}{ll}
&$a_{t,s,41} =  (0,2^s - 1,2^{s+1} - 1,2^{s+t+1}-2^s - 1),$\cr 
&$a_{t,s,42} =  (0,2^{s+1} - 1,2^s - 1,2^{s+t+1}-2^s - 1),$\cr 
&$a_{t,s,43} =  (0,2^{s+1} - 1,2^{s+t+1}-2^s - 1,2^s - 1),$\cr 
&$a_{t,s,44} =  (2^{s} - 1,0,2^{s+1} - 1,2^{s+t+1}-2^s - 1),$\cr 
&$a_{t,s,45} =  (2^{s} - 1,2^{s+1} - 1,0,2^{s+t+1}-2^s - 1),$\cr 
&$a_{t,s,46} =  (2^{s} - 1,2^{s+1} - 1,2^{s+t+1}-2^s - 1,0),$\cr 
&$a_{t,s,47} =  (2^{s+1} - 1,0,2^s - 1,2^{s+t+1}-2^s - 1),$\cr 
&$a_{t,s,48} =  (2^{s+1} - 1,0,2^{s+t+1}-2^s - 1,2^s - 1),$\cr 
&$a_{t,s,49} =  (2^{s+1} - 1,2^s - 1,0,2^{s+t+1}-2^s - 1),$\cr 
&$a_{t,s,50} =  (2^{s+1} - 1,2^s - 1,2^{s+t+1}-2^s - 1,0),$\cr 
&$a_{t,s,51} =  (2^{s+1} - 1,2^{s+t+1}-2^s - 1,0,2^s - 1),$\cr 
&$a_{t,s,52} =  (2^{s+1} - 1,2^{s+t+1}-2^s - 1,2^s - 1,0),$\cr 
&$a_{t,s,53} =  (0,2^{s+1} - 1,3.2^s - 1,2^{s+t+1}-3.2^s - 1),$\cr 
&$a_{t,s,54} =  (2^{s+1} - 1,0,3.2^s - 1,2^{s+t+1}-3.2^s - 1),$\cr 
&$a_{t,s,55} =  (2^{s+1} - 1,3.2^s - 1,0,2^{s+t+1}-3.2^s - 1),$\cr 
&$a_{t,s,56} =  (2^{s+1} - 1,3.2^s - 1,2^{s+t+1}-3.2^s - 1,0).$\cr
\end{tabular}}
\end{props}

Now, we  determine $(\mathbb F_2\underset{\mathcal A}\otimes R_4)_{2^{s+t+1}+2^{s+1}-3}$. 

Set $\mu_5(1) = 150, \mu_5(2) = 195$ and $\mu_5(s) = 210$ for $s\geqslant 3$.
The main result of this subsection is

\begin{thms}\label{dlc6.4} For any $t \geqslant 4$, $(\mathbb F_2 \underset {\mathcal A} \otimes R_4)_{2^{s+t+1}+2^{s+1}-3}$ is an $\mathbb F_2$-vector space of dimension $\mu_5(s)-56$
with a basis consisting of all the classes represented by the following monomials:

\smallskip
For $s\geqslant 1$,

\smallskip
\centerline{\begin{tabular}{ll}
&$a_{t,s,57} =  (1,2^{s+1} - 2,2^{s+t} - 1,2^{s+t} - 1),$\cr 
&$a_{t,s,58} =  (1,2^{s+t} - 1,2^{s+1} - 2,2^{s+t} - 1),$\cr 
&$a_{t,s,59} =  (1,2^{s+t} - 1,2^{s+t} - 1,2^{s+1} - 2),$\cr 
&$a_{t,s,60} =  (2^{s+t} - 1,1,2^{s+1} - 2,2^{s+t} - 1),$\cr 
&$a_{t,s,61} =  (2^{s+t} - 1,1,2^{s+t} - 1,2^{s+1} - 2),$\cr 
&$a_{t,s,62} =  (2^{s+t} - 1,2^{s+t} - 1,1,2^{s+1} - 2),$\cr 
&$a_{t,s,63} =  (1,2^{s+1} - 1,2^{s+t} - 2,2^{s+t} - 1),$\cr 
&$a_{t,s,64} =  (1,2^{s+1} - 1,2^{s+t} - 1,2^{s+t} - 2),$\cr 
&$a_{t,s,65} =  (1,2^{s+t} - 2,2^{s+1} - 1,2^{s+t} - 1),$\cr 
&$a_{t,s,66} =  (1,2^{s+t} - 2,2^{s+t} - 1,2^{s+1} - 1),$\cr 
&$a_{t,s,67} =  (1,2^{s+t} - 1,2^{s+1} - 1,2^{s+t} - 2),$\cr 
&$a_{t,s,68} =  (1,2^{s+t} - 1,2^{s+t} - 2,2^{s+1} - 1),$\cr 
&$a_{t,s,69} =  (2^{s+1} - 1,1,2^{s+t} - 2,2^{s+t} - 1),$\cr 
&$a_{t,s,70} =  (2^{s+1} - 1,1,2^{s+t} - 1,2^{s+t} - 2),$\cr 
&$a_{t,s,71} =  (2^{s+1} - 1,2^{s+t} - 1,1,2^{s+t} - 2),$\cr 
&$a_{t,s,72} =  (2^{s+t} - 1,1,2^{s+1} - 1,2^{s+t} - 2),$\cr 
&$a_{t,s,73} =  (2^{s+t} - 1,1,2^{s+t} - 2,2^{s+1} - 1),$\cr 
&$a_{t,s,74} =  (2^{s+t} - 1,2^{s+1} - 1,1,2^{s+t} - 2),$\cr 
&$a_{t,s,75} =  (1,2^{s+2} - 2,2^{s+t}-2^{s+1} - 1,2^{s+t} - 1),$\cr 
&$a_{t,s,76} =  (1,2^{s+2} - 2,2^{s+t} - 1,2^{s+t}-2^{s+1} - 1),$\cr 
\end{tabular}}
\centerline{\begin{tabular}{ll}
&$a_{t,s,77} =  (1,2^{s+t} - 1,2^{s+2} - 2,2^{s+t}-2^{s+1} - 1),$\cr 
&$a_{t,s,78} =  (2^{s+t} - 1,1,2^{s+2} - 2,2^{s+t}-2^{s+1} - 1),$\cr
&$a_{t,s,79} =  (1,2^{s+2} - 1,2^{s+t}-2^{s+1} - 2,2^{s+t} - 1),$\cr 
&$a_{t,s,80} =  (1,2^{s+2} - 1,2^{s+t} - 1,2^{s+t}-2^{s+1} - 2),$\cr 
&$a_{t,s,81} =  (1,2^{s+t} - 1,2^{s+2} - 1,2^{s+t}-2^{s+1} - 2),$\cr 
&$a_{t,s,82} =  (2^{s+2} - 1,1,2^{s+t}-2^{s+1} - 2,2^{s+t} - 1),$\cr 
&$a_{t,s,83} =  (2^{s+2} - 1,1,2^{s+t} - 1,2^{s+t}-2^{s+1} - 2),$\cr 
&$a_{t,s,84} =  (2^{s+2} - 1,2^{s+t} - 1,1,2^{s+t}-2^{s+1} - 2),$\cr 
&$a_{t,s,85} =  (2^{s+t} - 1,1,2^{s+2} - 1,2^{s+t}-2^{s+1} - 2),$\cr 
&$a_{t,s,86} =  (2^{s+t} - 1,2^{s+2} - 1,1,2^{s+t}-2^{s+1} - 2),$\cr 
&$a_{t,s,87} =  (1,2^{s+2} - 1,2^{s+t}-2^{s+1} - 1,2^{s+t} - 2),$\cr 
&$a_{t,s,88} =  (1,2^{s+2} - 1,2^{s+t} - 2,2^{s+t}-2^{s+1} - 1),$\cr 
&$a_{t,s,89} =  (1,2^{s+t} - 2,2^{s+2} - 1,2^{s+t}-2^{s+1} - 1),$\cr 
&$a_{t,s,90} =  (2^{s+2} - 1,1,2^{s+t}-2^{s+1} - 1,2^{s+t} - 2),$\cr 
&$a_{t,s,91} =  (2^{s+2} - 1,1,2^{s+t} - 2,2^{s+t}-2^{s+1} - 1),$\cr 
&$a_{t,s,92} =  (2^{s+2} - 1,2^{s+t}-2^{s+1} - 1,1,2^{s+t} - 2),$\cr 
&$a_{t,s,93} =  (1,2^{s+3} - 2,2^{s+t}-2^{s+2} - 1,2^{s+t}-2^{s+1} - 1),$\cr 
&$a_{t,s,94} =  (1,2^{s+3} - 1,2^{s+t}-2^{s+2} - 2,2^{s+t}-2^{s+1} - 1),$\cr 
&$a_{t,s,95} =  (2^{s+3} - 1,1,2^{s+t}-2^{s+2} - 2,2^{s+t}-2^{s+1} - 1),$\cr 
&$a_{t,s,96} =  (1,2^{s+3} - 1,2^{s+t}-2^{s+2} - 1,2^{s+t}-2^{s+1} - 2),$\cr 
&$a_{t,s,97} =  (2^{s+3} - 1,1,2^{s+t}-2^{s+2} - 1,2^{s+t}-2^{s+1} - 2),$\cr 
&$a_{t,s,98} =  (2^{s+3} - 1,2^{s+t}-2^{s+2} - 1,1,2^{s+t}-2^{s+1} - 2),$\cr 
&$a_{t,s,99} =  (3,2^{s+t} - 3,2^{s+1} - 2,2^{s+t} - 1),$\cr 
&$a_{t,s,100} =  (3,2^{s+t} - 3,2^{s+t} - 1,2^{s+1} - 2),$\cr 
&$a_{t,s,101} =  (3,2^{s+t} - 1,2^{s+t} - 3,2^{s+1} - 2),$\cr 
&$a_{t,s,102} =  (2^{s+t} - 1,3,2^{s+t} - 3,2^{s+1} - 2),$\cr 
&$a_{t,s,103} =  (3,2^{s+2} - 3,2^{s+t}-2^{s+1} - 2,2^{s+t} - 1),$\cr 
&$a_{t,s,104} =  (3,2^{s+2} - 3,2^{s+t} - 1,2^{s+t}-2^{s+1} - 2),$\cr 
&$a_{t,s,105} =  (3,2^{s+t} - 1,2^{s+2} - 3,2^{s+t}-2^{s+1} - 2),$\cr 
&$a_{t,s,106} =  (2^{s+t} - 1,3,2^{s+2} - 3,2^{s+t}-2^{s+1} - 2),$\cr 
&$a_{t,s,107} =  (3,2^{s+2} - 3,2^{s+t}-2^{s+1} - 1,2^{s+t} - 2),$\cr 
&$a_{t,s,108} =  (3,2^{s+2} - 3,2^{s+t} - 2,2^{s+t}-2^{s+1} - 1),$\cr 
&$a_{t,s,109} =  (3,2^{s+t} - 3,2^{s+2} - 2,2^{s+t}-2^{s+1} - 1),$\cr 
&$a_{t,s,110} =  (3,2^{s+2} - 1,2^{s+t}-2^{s+1} - 3,2^{s+t} - 2),$\cr 
&$a_{t,s,111} =  (2^{s+2} - 1,3,2^{s+t}-2^{s+1} - 3,2^{s+t} - 2),$\cr 
&$a_{t,s,112} =  (3,2^{s+2} - 1,2^{s+t} - 3,2^{s+t}-2^{s+1} - 2),$\cr 
&$a_{t,s,113} =  (3,2^{s+t} - 3,2^{s+2} - 1,2^{s+t}-2^{s+1} - 2),$\cr 
&$a_{t,s,114} =  (2^{s+2} - 1,3,2^{s+t} - 3,2^{s+t}-2^{s+1} - 2),$\cr 
&$a_{t,s,115} =  (3,2^{s+3} - 3,2^{s+t}-2^{s+2} - 2,2^{s+t}-2^{s+1} - 1),$\cr 
&$a_{t,s,116} =  (3,2^{s+3} - 3,2^{s+t}-2^{s+2} - 1,2^{s+t}-2^{s+1} - 2),$\cr 
&$a_{t,s,117} =  (3,2^{s+3} - 1,2^{s+t}-2^{s+2} - 3,2^{s+t}-2^{s+1} - 2),$\cr 
&$a_{t,s,118} =  (2^{s+3} - 1,3,2^{s+t}-2^{s+2} - 3,2^{s+t}-2^{s+1} - 2),$\cr 
&$a_{t,s,119} =  (7,2^{s+t} - 5,2^{s+t} - 3,2^{s+1} - 2),$\cr 
&$a_{t,s,120} =  (7,2^{s+t} - 5,2^{s+2} - 3,2^{s+t}-2^{s+1} - 2),$\cr 
&$a_{t,s,121} =  (7,2^{s+3} - 5,2^{s+t}-2^{s+2} - 3,2^{s+t}-2^{s+1} - 2),$\cr 
\end{tabular}}
\centerline{\begin{tabular}{ll}
&$a_{t,s,122} =  (1,2^s - 1,2^s - 1,2^{s+t+1} - 2),$\cr 
&$a_{t,s,123} =  (1,2^s - 1,2^{s+t+1} - 2,2^s - 1),$\cr
&$a_{t,s,124} =  (1,2^{s+t+1} - 2,2^s - 1,2^s - 1),$\cr 
&$a_{t,s,125} =  (1,2^s - 1,2^{s+1} - 2,2^{s+t+1}-2^s - 1),$\cr 
&$a_{t,s,126} =  (1,2^{s+1} - 2,2^s - 1,2^{s+t+1}-2^s - 1),$\cr 
&$a_{t,s,127} =  (1,2^{s+1} - 2,2^{s+t+1}-2^s - 1,2^s - 1),$\cr 
&$a_{t,s,128} =  (1,2^s - 1,2^{s+1} - 1,2^{s+t+1}-2^s - 2),$\cr 
&$a_{t,s,129} =  (1,2^{s+1} - 1,2^s - 1,2^{s+t+1}-2^s - 2),$\cr 
&$a_{t,s,130} =  (1,2^{s+1} - 1,2^{s+t+1}-2^s - 2,2^s - 1),$\cr 
&$a_{t,s,131} =  (2^{s+1} - 1,1,2^s - 1,2^{s+t+1}-2^s - 2),$\cr 
&$a_{t,s,132} =  (2^{s+1} - 1,1,2^{s+t+1}-2^s - 2,2^s - 1),$\cr 
&$a_{t,s,133} =  (1,2^{s+1} - 2,3.2^s - 1,2^{s+t+1}-3.2^s - 1),$\cr 
&$a_{t,s,134} =  (1,2^{s+1} - 1,3.2^s - 2,2^{s+t+1}-3.2^s - 1),$\cr 
&$a_{t,s,135} =  (2^{s+1} - 1,1,3.2^s - 2,2^{s+t+1}-3.2^s - 1),$\cr 
&$a_{t,s,136} =  (1,2^{s+1} - 1,3.2^s - 1,2^{s+t+1}-3.2^s - 2),$\cr 
&$a_{t,s,137} =  (2^{s+1} - 1,1,3.2^s - 1,2^{s+t+1}-3.2^s - 2),$\cr 
&$a_{t,s,138} =  (2^{s+1} - 1,3.2^s - 1,1,2^{s+t+1}-3.2^s - 2),$\cr 
&$a_{t,s,139} =  (3,2^{s+1} - 1,2^{s+t} - 3,2^{s+t} - 2),$\cr 
&$a_{t,s,140} =  (3,2^{s+t} - 3,2^{s+1} - 1,2^{s+t} - 2),$\cr 
&$a_{t,s,141} =  (3,2^{s+t} - 3,2^{s+t} - 2,2^{s+1} - 1).$\cr
\end{tabular}}

\medskip
For $s = 1$,

\medskip
\centerline{\begin{tabular}{ll}
$a_{t,1,142} = (3,3,2^{t+1} - 4,2^{t+1} - 1),$ &$a_{t,1,143} = (3,3,2^{t+1} - 1,2^{t+1} - 4),$\cr
$a_{t,1,144} = (3,2^{t+1} - 1,3,2^{t+1} - 4),$ &$a_{t,1,145} = (2^{t+1} - 1,3,3,2^{t+1} - 4),$\cr
$a_{t,1,146} = (3,7,2^{t+1} - 5,2^{t+1} - 4),$ &$a_{t,1,147} = (7,2^{t+1} - 5,3,2^{t+1} - 4),$\cr
$a_{t,1,148} = (7,3,2^{t+1} - 5,2^{t+1} - 4),$ &$a_{t,1,149} = (7,7,2^{t+1} - 8,2^{t+1} - 5),$\cr
$a_{t,1,150} = (7,7,2^{t+1} - 7,2^{t+1} - 6).$ &\cr
\end{tabular}}

\medskip
For $s \geqslant 2$,

\medskip
\centerline{\begin{tabular}{ll}
&$a_{t,s,142} =  (2^{s+1} - 1,3,2^{s+t} - 3,2^{s+t} - 2),$\cr 
&$a_{t,s,143} =  (3,2^{s+1} - 3,2^{s+t} - 2,2^{s+t} - 1),$\cr 
&$a_{t,s,144} =  (3,2^{s+1} - 3,2^{s+t} - 1,2^{s+t} - 2),$\cr 
&$a_{t,s,145} =  (3,2^{s+t} - 1,2^{s+1} - 3,2^{s+t} - 2),$\cr 
&$a_{t,s,146} =  (2^{s+t} - 1,3,2^{s+1} - 3,2^{s+t} - 2),$\cr 
&$a_{t,s,147} =  (7,2^{s+t} - 5,2^{s+1} - 3,2^{s+t} - 2),$\cr 
&$a_{t,s,148} =  (7,2^{s+2} - 5,2^{s+t}-2^{s+1} - 3,2^{s+t} - 2),$\cr 
&$a_{t,s,149} =  (7,2^{s+2} - 5,2^{s+t} - 3,2^{s+t}-2^{s+1} - 2),$\cr 
&$a_{t,s,150} =  (1,2^s - 2,2^s - 1,2^{s+t+1} - 1),$\cr 
&$a_{t,s,151} =  (1,2^s - 2,2^{s+t+1} - 1,2^s - 1),$\cr 
&$a_{t,s,152} =  (1,2^s - 1,2^s - 2,2^{s+t+1} - 1),$\cr 
&$a_{t,s,153} =  (1,2^s - 1,2^{s+t+1} - 1,2^s - 2),$\cr 
&$a_{t,s,154} =  (1,2^{s+t+1} - 1,2^s - 2,2^s - 1),$\cr 
&$a_{t,s,155} =  (1,2^{s+t+1} - 1,2^s - 1,2^s - 2),$\cr 
&$a_{t,s,156} =  (2^{s} - 1,1,2^s - 2,2^{s+t+1} - 1),$\cr 
&$a_{t,s,157} =  (2^{s} - 1,1,2^{s+t+1} - 1,2^s - 2),$\cr 
\end{tabular}}
\centerline{\begin{tabular}{ll}
&$a_{t,s,158} =  (2^{s} - 1,2^{s+t+1} - 1,1,2^s - 2),$\cr 
&$a_{t,s,159} =  (2^{s+t+1} - 1,1,2^s - 2,2^s - 1),$\cr
&$a_{t,s,160} =  (2^{s+t+1} - 1,1,2^s - 1,2^s - 2),$\cr 
&$a_{t,s,161} =  (2^{s+t+1} - 1,2^s - 1,1,2^s - 2),$\cr 
&$a_{t,s,162} =  (2^{s} - 1,1,2^s - 1,2^{s+t+1} - 2),$\cr 
&$a_{t,s,163} =  (2^{s} - 1,1,2^{s+t+1} - 2,2^s - 1),$\cr 
&$a_{t,s,164} =  (2^{s} - 1,2^s - 1,1,2^{s+t+1} - 2),$\cr 
&$a_{t,s,165} =  (1,2^s - 2,2^{s+1} - 1,2^{s+t+1}-2^s - 1),$\cr 
&$a_{t,s,166} =  (1,2^{s+1} - 1,2^s - 2,2^{s+t+1}-2^s - 1),$\cr 
&$a_{t,s,167} =  (1,2^{s+1} - 1,2^{s+t+1}-2^s - 1,2^s - 2),$\cr 
&$a_{t,s,168} =  (2^{s+1} - 1,1,2^s - 2,2^{s+t+1}-2^s - 1),$\cr 
&$a_{t,s,169} =  (2^{s+1} - 1,1,2^{s+t+1}-2^s - 1,2^s - 2),$\cr 
&$a_{t,s,170} =  (2^{s+1} - 1,2^{s+t+1}-2^s - 1,1,2^s - 2),$\cr 
&$a_{t,s,171} =  (2^{s} - 1,1,2^{s+1} - 2,2^{s+t+1}-2^s - 1),$\cr 
&$a_{t,s,172} =  (2^{s} - 1,1,2^{s+1} - 1,2^{s+t+1}-2^s - 2),$\cr 
&$a_{t,s,173} =  (2^{s} - 1,2^{s+1} - 1,1,2^{s+t+1}-2^s - 2),$\cr 
&$a_{t,s,174} =  (2^{s+1} - 1,2^s - 1,1,2^{s+t+1}-2^s - 2),$\cr 
&$a_{t,s,175} =  (3,2^{s+1} - 3,2^s - 2,2^{s+t+1}-2^s - 1),$\cr 
&$a_{t,s,176} =  (3,2^{s+1} - 3,2^{s+t+1}-2^s - 1,2^s - 2),$\cr 
&$a_{t,s,177} =  (3,2^{s+1} - 1,2^{s+t+1}-2^s - 3,2^s - 2),$\cr 
&$a_{t,s,178} =  (2^{s+1} - 1,3,2^{s+t+1}-2^s - 3,2^s - 2),$\cr 
&$a_{t,s,179} =  (3,2^{s+1} - 3,3.2^s - 2,2^{s+t+1}-3.2^s - 1),$\cr 
&$a_{t,s,180} =  (3,2^{s+1} - 3,3.2^s - 1,2^{s+t+1}-3.2^s - 2),$\cr 
&$a_{t,s,181} =  (3,2^{s+1} - 1,3.2^s - 3,2^{s+t+1}-3.2^s - 2),$\cr 
&$a_{t,s,182} =  (2^{s+1} - 1,3,3.2^s - 3,2^{s+t+1}-3.2^s - 2),$\cr 
&$a_{t,s,183} =  (3,2^s - 1,2^{s+t+1} - 3,2^s - 2),$\cr 
&$a_{t,s,184} =  (3,2^{s+t+1} - 3,2^s - 2,2^s - 1),$\cr 
&$a_{t,s,185} =  (3,2^{s+t+1} - 3,2^s - 1,2^s - 2),$\cr 
&$a_{t,s,186} =  (3,2^s - 1,2^{s+1} - 3,2^{s+t+1}-2^s - 2),$\cr 
&$a_{t,s,187} =  (3,2^{s+1} - 3,2^s - 1,2^{s+t+1}-2^s - 2),$\cr 
&$a_{t,s,188} =  (3,2^{s+1} - 3,2^{s+t+1}-2^s - 2,2^s - 1).$\cr 
\end{tabular}}

\medskip
For $s=2$,

\medskip
\centerline{\begin{tabular}{ll}
$a_{t,2,189} = (7,7,2^{t+2} - 7,2^{t+2} - 2)$, 
&$a_{t,2,190} = (3,3,3,2^{t+3} - 4)$,\cr
$a_{t,2,191} = (3,3,2^{t+3} - 4,3)$, 
&$a_{t,2,192} = (3,3,4,2^{t+3} - 5)$,\cr
$a_{t,2,193} = (3,3,7,2^{t+3} - 8)$, 
&$a_{t,2,194} = (3,7,3,2^{t+3} - 8)$,\cr
$a_{t,2,195} = (7,3,3,2^{t+3} - 8)$. &\cr
\end{tabular}}

\medskip
For $s \geqslant 3$,

\medskip
\centerline{\begin{tabular}{ll}
&$a_{t,s,189} =  (7,2^{s+1} - 5,2^{s+t} - 3,2^{s+t} - 2),$\cr 
&$a_{t,s,190} =  (3,2^s - 3,2^s - 2,2^{s+t+1} - 1),$\cr 
&$a_{t,s,191} =  (3,2^s - 3,2^{s+t+1} - 1,2^s - 2),$\cr 
&$a_{t,s,192} =  (3,2^{s+t+1} - 1,2^s - 3,2^s - 2),$\cr 
&$a_{t,s,193} =  (2^{s+t+1} - 1,3,2^s - 3,2^s - 2),$\cr 
&$a_{t,s,194} =  (3,2^s - 3,2^s - 1,2^{s+t+1} - 2),$\cr 
\end{tabular}}
\centerline{\begin{tabular}{ll}
&$a_{t,s,195} =  (3,2^s - 3,2^{s+t+1} - 2,2^s - 1),$\cr 
&$a_{t,s,196} =  (3,2^s - 1,2^s - 3,2^{s+t+1} - 2),$\cr
&$a_{t,s,197} =  (2^{s} - 1,3,2^s - 3,2^{s+t+1} - 2),$\cr 
&$a_{t,s,198} =  (2^{s} - 1,3,2^{s+t+1} - 3,2^s - 2),$\cr 
&$a_{t,s,199} =  (3,2^s - 3,2^{s+1} - 2,2^{s+t+1}-2^s - 1),$\cr 
&$a_{t,s,200} =  (3,2^s - 3,2^{s+1} - 1,2^{s+t+1}-2^s - 2),$\cr 
&$a_{t,s,201} =  (3,2^{s+1} - 1,2^s - 3,2^{s+t+1}-2^s - 2),$\cr 
&$a_{t,s,202} =  (2^{s+1} - 1,3,2^s - 3,2^{s+t+1}-2^s - 2),$\cr 
&$a_{t,s,203} =  (2^{s} - 1,3,2^{s+1} - 3,2^{s+t+1}-2^s - 2),$\cr 
&$a_{t,s,204} =  (7,2^{s+t+1} - 5,2^s - 3,2^s - 2),$\cr 
&$a_{t,s,205} =  (7,2^{s+1} - 5,2^s - 3,2^{s+t+1}-2^s - 2),$\cr 
&$a_{t,s,206} =  (7,2^{s+1} - 5,2^{s+t+1}-2^s - 3,2^s - 2),$\cr 
&$a_{t,s,207} =  (7,2^{s+1} - 5,3.2^s - 3,2^{s+t+1}-3.2^s - 2).$\cr
\end{tabular}}
\medskip

For $s=3$,

\medskip
\centerline{\begin{tabular}{ll}
$a_{t,3,208} = (7,7,2^{t+4}-7,6),$& $a_{t,3,209} = (7,7,7,2^{t+4}-8),$\cr $a_{t,3,210} = (7,7,9,2^{t+4}-10).$&\cr
\end{tabular}}

\medskip
For $s \geqslant 4$,

\medskip
\centerline{\begin{tabular}{ll}
&$a_{t,s,208} =  (7,2^s - 5,2^s - 3,2^{s+t+1} - 2),$\cr 
&$a_{t,s,209} =  (7,2^s - 5,2^{s+t+1} - 3,2^s - 2),$\cr 
&$a_{t,s,210} =  (7,2^s - 5,2^{s+1} - 3,2^{s+t+1}-2^s - 2).$\cr
\end{tabular}}
\end{thms}

The theorem is proved by combining the following propositions.

\begin{props}\label{mdc6.4}  The $\mathbb F_2$-vector space $(\mathbb F_2\underset {\mathcal A}\otimes R_4)_{2^{s+t+1}+2^{s+1}-3}$ is generated by the $\mu_5(s)-56$ elements listed in Theorem \ref{dlc6.4}.
\end{props}

The proof of this proposition is based on Theorem \ref{2.4} and the following lemma.

\begin{lems}\label{6.4.1} The following matrices are strictly inadmissible
$$\begin{pmatrix} 1&1&1&0\\ 1&1&1&0\\ 1&1&0&0\\ 0&0&1&1\\ 0&0&1&1\end{pmatrix} 
\quad \begin{pmatrix} 1&1&0&1\\ 1&1&0&1\\ 1&1&0&0\\ 0&1&0&1\\ 0&0&1&1\end{pmatrix} 
\quad \begin{pmatrix} 1&1&0&1\\ 1&1&0&1\\ 1&1&0&0\\ 1&0&0&1\\ 0&0&1&1\end{pmatrix} \quad 
\begin{pmatrix} 1&1&1&0\\ 1&1&0&1\\ 1&1&0&0\\ 0&1&0&1\\ 0&0&1&1\end{pmatrix} $$    
$$\begin{pmatrix} 1&1&1&0\\ 1&1&0&1\\ 1&1&0&0\\ 1&0&0&1\\ 0&0&1&1\end{pmatrix} 
\quad \begin{pmatrix} 1&1&0&1\\ 1&1&0&1\\ 1&1&0&0\\ 1&1&0&0\\ 0&0&1&1\end{pmatrix} 
\quad \begin{pmatrix} 1&1&1&0\\ 1&1&1&0\\ 1&1&0&0\\ 1&1&0&0\\ 0&0&1&1\end{pmatrix} 
\quad \begin{pmatrix} 1&1&1&0\\ 1&1&0&1\\ 1&1&0&0\\ 1&1&0&0\\ 0&0&1&1\end{pmatrix}. $$
\end{lems}

\begin{proof}  The monomials corresponding to the above matrices respectively are 
\begin{align*}
&(7,7,27,24),  (7,15,16,27),   (15,7,16,27),  (7,15,17,26),\\  
&(15,7,17,26),  (15,15,16,19),  (15,15,19,16),  (15,15,17,18).\end{align*}
We proved the lemma for the matrices associated with 
\begin{align*}&(7,7,27,24),  (15,7,16,27),   (15,7,17,26),\\  
&(15,15,16,19),  (15,15,19,16), (15,15,17,18).
\end{align*}
 The others can be obtained by a similar computation. By a direct computation, we have
\begin{align*}
&(7,7,27,24)= Sq^1\big((7,7,27,23) + (7,5,29,23) + (7,7,21,29)\big) \\
&\quad+ Sq^2\big((7,6,27,23) + (7,3,30,23) +(7,7,26,23) + (7,7,22,27) \\
&\quad+ (7,7,19,30) + (7,3,23,30) + (7,3,15,38)\big)+ Sq^4\big((4,7,27,23)\\
&\quad + (5,6,27,23) + (5,3,30,23) + (5,7,26,23) + (11,5,22,23)\\
&\quad + (5,7,22,27) + (5,7,19,30) + (11,5,15,30) + (5,3,23,30) \big)\\ 
&\quad  + Sq^8\big((7,5,22,23) + (7,5,15,30)  \big)  +(4,11,27,23) + (4,7,27,27)\\
&\quad + (7,6,27,25) + (5,10,27,23) + (5,6,27,27) + (7,3,30,25)\\ 
&\quad+ (5,3,30,27)  +(7,7,26,25) +  (5,11,26,23)+ (7,5,26,27)   \\ 
&\quad+ (7,7,24,27) +(5,11,22,27) + (5,11,19,30)\\
&\quad + (7,3,25,30) + (5,3,27,30)\quad  \text{mod }\mathcal L_4(3;3;2;2;2),\\
&(15,7,16,27)= Sq^1(15,7,13,29) + Sq^2\big((15,7,11,30)\\
&\quad + (15,7,14,27) \big)+ Sq^4\big((15,5,14,27) + (15,5,11,30) \big)\\ 
&\quad  + Sq^8\big((9,7,14,27) + (11,5,14,27) + (9,7,11,30)\\
&\quad + (11,5,11,30) \big) + (15,5,18,27) + (9,7,22,27) + (11,5,22,27) \\ 
&\quad+ (9,7,19,30) + (11,5,19,30)\quad  \text{mod }\mathcal L_4(3;3;2;2;2),\\
&(15,7,17,26)= Sq^1(15,7,15,27) + Sq^2(15,7,15,26)\\ 
&\quad + Sq^4\big((15,5,15,26)  + (15,4,15,27) \big) + Sq^8\big((9,7,15,26)\\ 
&\quad + (11,5,15,26)+ (8,7,15,27)  + (11,4,15,27) \big) + (15,5,19,26)\\
&\quad + (11,5,23,26) + (15,7,16,27) + (8,7,23,27) + (15,4,19,27)\\
&\quad+(11,4,23,27)+ (9,7,23,26) \quad  \text{mod }\mathcal L_4(3;3;2;2;2),\\
&(15,15,16,19)= Sq^1\big((15,15,13,21) + (15,15,11,23) + (15,15,9,25)\\ 
&\quad + (15,15,7,27) + (15,15,5,29)\big)  +Sq^2\big( (15,15,14,19) + (15,15,11,22) \\ 
&\quad + (15,15,10,23) + (15,15,7,26) + (15,15,6,27) + (15,15,3,30)\big)\\ 
&\quad+ Sq^4\big((15,13,14,19) + (15,13,11,22) + (15,12,11,23)\\ 
&\quad +(15,13,10,23) + (15,13,7,26) + (15,12,7,27) + (15,13,6,27)\\ 
&\quad + (15,13,3,30)\big) + Sq^8\big((9,15,14,19) + (11,13,14,19) + (9,15,11,22)\\ 
&\quad + (11,13,11,22) + (8,15,11,23)+(11,12,11,23) + (9,15,10,23)\\ 
&\quad +(11,13,10,23) + (9,15,7,26) + (11,13,7,26) + (8,15,7,27)
\end{align*}
\begin{align*}
&\quad + (11,12,7,27) + (9,15,6,27) + (11,13,6,27) + (9,15,3,30)\\ 
&\quad + (11,13,3,30)\big) + (9,23,14,19) + (9,15,22,19)+ (15,13,18,19)\\ 
&\quad  + (11,21,14,19) + (11,13,22,19)+ (9,23,11,22)  + (9,15,19,22)\\ 
&\quad+ (11,21,11,22) + (11,13,19,22) + (8,23,11,23) + (8,15,19,23)\\ 
&\quad  + (11,20,11,23) + (11,12,19,23) + (9,23,10,23) + (9,15,18,23)\\ 
&\quad + (11,21,10,23) + (11,13,18,23) + (9,23,7,26) + (11,21,7,26)\\ 
&\quad + (8,23,7,27) + (11,20,7,27) + (9,23,6,27) + (11,21,6,27)\\ 
&\quad + (9,23,3,30) + (11,21,3,30) \quad  \text{mod }\mathcal L_4(3;3;2;2;2),\\
&(15,15,19,16)= Sq^1(15,15,19,15)  +Sq^2(15,15,18,15) \\ 
&\quad+ Sq^4\big((15,12,19,15) + (15,13,18,15) \big) + Sq^8\big((8,15,19,15)\\ 
&\quad + (11,12,19,15) + (9,15,18,15) + (11,13,18,15)\big)  + (15,15,18,17) \\ 
&\quad + (15,12,19,19)+ (15,13,18,19) + (8,23,19,15) + (8,15,19,23)\\ 
&\quad + (11,20,19,15)  + (11,12,19,23)+ (9,23,18,15) + (9,15,18,23)\\ 
&\quad  + (11,21,18,15) + (11,13,18,23)  \quad  \text{mod }\mathcal L_4(3;3;2;2;2),\\
&(15,15,17,18)= Sq^1(15,15,15,19)  +Sq^2(15,15,15,18) \\ 
&\quad + Sq^4\big((15,13,15,18)+ (15,12,15,19) \big) + Sq^8\big((9,15,15,18)\\ 
&\quad + (11,13,15,18) + (8,15,15,19) + (11,12,15,19)\big)  + (15,15,16,19) \\ 
&\quad + (9,23,15,18)+ (9,15,23,18) + (15,13,19,18) + (11,21,15,18)\\ 
&\quad + (11,13,23,18)  + (8,23,15,19)+ (8,15,23,19) + (15,12,19,19) \\ 
&\quad + (11,20,15,19) + (11,12,23,19)  \quad  \text{mod }\mathcal L_4(3;3;2;2;2).
\end{align*}

The lemma is proved.
\end{proof}

Combining Theorem \ref{2.4} the lemmas in Sections \ref{3}, \ref{4}, \ref{5} and  Lemmas \ref{6.1.1}, \ref{6.1.2}, \ref{6.1.3}, \ref{6.2.1}, \ref{6.2.2}, \ref{6.3.1}, \ref{6.3.2}, \ref{6.3.3}, \ref{6.4.1},  we obtain Proposition \ref{mdc6.4}. 

\medskip
Now, we show that the elements $[a_{t,s,j}], 57\leqslant j\leqslant \mu_5(s)$, are linearly independent in $(\mathbb F_2\underset {\mathcal A}\otimes R_4)_{2^{s+t+1}+2^{s+1}-3}$.

\begin{props}\label{6.4.3} For $ t \geqslant 4,$ the elements $[a_{t,1,j}], \ 57 \leqslant j \leqslant 150,$ are linearly independent in $(\mathbb F_2\underset {\mathcal A}\otimes R_4)_{2^{t+2}+1}$.
\end{props}

\begin{proof} Suppose there is a linear relation
\begin{equation} \sum_{j=57}^{150}\gamma_j[a_{t,1,j}]=0, \tag{\ref{6.4.3}.1}
\end{equation}
where $\gamma_j \in \mathbb F_2.$ 

Consider the homomorphisms $f_i,\ i=1,2,3$. Under these homomorphisms, the images of the  relation (\ref{6.4.3}.1) respectively are
\begin{align*}
&\gamma_{124} w_{t,1,3}  +  \gamma_{126} w_{t,1,5}  +    \gamma_{127} w_{t,1,6} +   \gamma_{133} w_{t,1,7}  +   \gamma_{57} w_{t,1,8}  +   \gamma_{65} w_{t,1,9}\\
&\quad  +   \gamma_{66} w_{t,1,10}  +   \gamma_{75} w_{t,1,11} +   \gamma_{76} w_{t,1,12}  +   \gamma_{89} w_{t,1,13}  +   \gamma_{93} w_{t,1,14}  =0,  \\ 
&\gamma_{\{123, 132\}} w_{t,1,3}  +   a_1 w_{t,1,5}  +   \gamma_{130} w_{t,1,6}  +   a_2 w_{t,1,7}  +   \gamma_{58} w_{t,1,8}  +   \gamma_{\{63, 142\}} w_{t,1,9}\\
&\quad  +   \gamma_{68} w_{t,1,10}  +   \gamma_{79} w_{t,1,11}  +   \gamma_{77} w_{t,1,12}  +   \gamma_{\{88, 149\}} w_{t,1,13}  +   \gamma_{94} w_{t,1,14}  =0,\\   
&\gamma_{\{122, 131\}} w_{t,1,3}  +  a_3w_{t,1,5}  +   a_4w_{t,1,6}  +   a_5w_{t,1,7}  +   \gamma_{59} w_{t,1,8}  +   \gamma_{\{64, 143\}} w_{t,1,9}\\
&\ +   \gamma_{\{67, 144\}} w_{t,1,10}  +   \gamma_{80} w_{t,1,11}  +  \gamma_{81} w_{t,1,12}  +   \gamma_{\{87, 146\}} w_{t,1,13}  +    \gamma_{96} w_{t,1,14}  =0,   
\end{align*} 
where
\begin{align*}
a_1 &= \gamma_{\{60, 69, 73, 78, 82, 91, 95, 125, 135\}},\ 
a_2 = \gamma_{\{99, 103, 108, 109, 115, 134, 141\}},\\
a_3 &= \gamma_{\{61, 70, 72, 83, 85, 90, 97, 128, 137\}},\
a_4 = \gamma_{\{62, 71, 74, 84, 86, 92, 98, 129, 138\}},\\ 
a_5 &= \gamma_{\{100, 101, 102, 104, 105, 106, 107, 110, \ldots , 114, 116, 117, \ldots , 121, 136, 139, 140, 145, 147, 148, 150\}}.
\end{align*}
From the above equalities, we get 
\begin{equation}\begin{cases} 
a_1 = a_2 = a_3 = a_4 = a_5 = 0,\\
\gamma_j = 0, \ j = 57, 58, 59, 65, 66, 68, 75, 76, 77,\\ 79, 80, 81, 89, 93, 94, 96, 124, 126, 127, 130, 133,\\
\gamma_{\{123, 132\}}= 
\gamma_{\{63, 142\}} = 
\gamma_{\{88, 149\}}= 0,\\
\gamma_{\{122, 131\}} = 
\gamma_{\{64, 143\}} = 
\gamma_{\{67, 144\}} = 
\gamma_{\{87, 146\}} =0.   
\end{cases} \tag{\ref{6.4.3}.2}
\end{equation}

With the aid of (\ref{6.4.3}.2), the homomorphisms $f_4, f_5, f_6$ send (\ref{6.4.3}.1) respectively to
\begin{align*}
&\gamma_{123} w_{t,1,2}  +   \gamma_{\{63, 88, 125, 134\}} w_{t,1,4}  +   \gamma_{\{108, 109, 115, 135, 141\}}w_{t,1,7}\\
&\quad  +    \gamma_{132} w_{t,1,6} +   \gamma_{60} w_{t,1,9}  +   \gamma_{73} w_{t,1,10}     +   \gamma_{\{69, 99, 103, 142\}} w_{t,1,8}\\
&\quad  +   \gamma_{82} w_{t,1,11} +   \gamma_{\{91, 149\}} w_{t,1,12}  +   \gamma_{78} w_{t,1,13}  +   \gamma_{95} w_{t,1,14}  = 0,\\  
&\gamma_{\{122, 129\}} w_{t,1,2}  +  \gamma_{\{64, 67, 87, 128, 136\}} w_{t,1,4}   +   a_6 w_{t,1,6}+   a_7w_{t,1,7}\\
&\quad  +   \gamma_{\{70, 100, 104, 143\}} w_{t,1,8}  +   \gamma_{61} w_{t,1,9} +   \gamma_{\{72, 145\}} w_{t,1,10}  +   \gamma_{83} w_{t,1,11}\\
&\quad +   \gamma_{\{90, 148\}} w_{t,1,12}  +   \gamma_{85} w_{t,1,13}  +   \gamma_{97} w_{t,1,14}  = 0,\\   
&\gamma_{\{122, 123, 125, 128\}} w_{t,1,1}  +  a_8w_{t,1,4}  +   a_9 w_{t,1,5}  +   a_{10} w_{t,1,7}\\
&\quad  +   \gamma_{\{71, 101, 105, 144\}} w_{t,1,8}  +   \gamma_{\{74, 102, 106, 145\}} w_{t,1,9} +   \gamma_{62} w_{t,1,10}  \\
&\quad +   \gamma_{\{92, 119, 120, 147\}} w_{t,1,11} +  \gamma_{84} w_{t,1,12}  +   \gamma_{86} w_{t,1,13}  +    \gamma_{98} w_{t,1,14}  =0, 
\end{align*} 
where
\begin{align*}
a_6 &= \gamma_{\{62, 71, 74, 84, 86, 92, 98, 131, 138\}},\ \
a_8 = \gamma_{\{63, 64, 67, 87, 88, 129, 134, 136\}},\\ 
a_7 &= \gamma_{\{101, 102, 105, 106, 107, 110, 111, 112, 113, 114, 116, \ldots , 121, 137, 139, 140, 144, 146, 147, 150\}},\\
a_9 &= \gamma_{\{60, 61, 69, 70, 72, 73, 78, 82, 83, 85, 90, 91, 95, 97, 131, 132, 135, 137\}},\\
a_{10} &= \gamma_{\{99, 100, 103, 104, 107, \ldots , 118, 121, 138, 139, 140, 141, 142, 143, 146, 148, 149, 150\}}.
\end{align*}
From theses equalities, we get
\begin{equation}\begin{cases}
a_6 = a_7 = a_8 = a_9 = a_{10} = 0,\
\gamma_j = 0,\ j = 60, 61, 62,\\ 73, 78, 82, 83, 84, 85, 86, 95, 97, 98, 123, 132=0,\\
\gamma_{\{63, 88, 125, 134\}}=
\gamma_{\{69, 99, 103, 142\}}= 
\gamma_{\{91, 149\}} = 0,\\
\gamma_{\{108, 109, 115, 135, 141\}} =   
\gamma_{\{122, 129\}} =  
\gamma_{\{64, 67, 87, 128, 136\}} = 0,\\
\gamma_{\{70, 100, 104, 143\}} =  
\gamma_{\{90, 148\}} =  
\gamma_{\{72, 145\}} = 
\gamma_{\{122, 123, 125, 128\}} = 0,\\ 
\gamma_{\{71, 101, 105, 144\}} = 
\gamma_{\{74, 102, 106, 145\}} =  
\gamma_{\{92, 119, 120, 147\}} =0.
\end{cases}\tag{\ref{6.4.3}.3}
\end{equation}

With the aid of (\ref{6.4.3}.2) and (\ref{6.4.3}.3), the homomorphisms $g_1, g_2$ send (\ref{6.4.3}.1) respectively to
\begin{align*}
&\gamma_{\{69, 88, 125, 135\}}w_{t,1,4}  +  a_{11}w_{t,1,7}  +    \gamma_{69}w_{t,1,8}  +   \gamma_{99}w_{t,1,9}  +   \gamma_{141}w_{t,1,10}\\
&\quad  +   \gamma_{103}w_{t,1,11}  +   \gamma_{\{88, 108\}}w_{t,1,12}  +   \gamma_{109}w_{t,1,13}  +   \gamma_{115}w_{t,1,14}  =0,\\
&\gamma_{\{70, 72, 90, 128, 137\}}w_{t,1,4}  +  a_{12}w_{t,1,7}  +   \gamma_{70}w_{t,1,8}  +   \gamma_{100}w_{t,1,9}  +   \gamma_{\{140, 147\}}w_{t,1,10}\\
&\quad +   \gamma_{104}w_{t,1,11}  +   \gamma_{107}w_{t,1,12}  +   \gamma_{113}w_{t,1,13}  +   \gamma_{116}w_{t,1,14}  = 0, 
\end{align*}
where
\begin{align*}
a_{11} &= \gamma_{\{63, 88, 99, 103, 108, 109, 115, 134, 135, 141\}},\\
a_{12} &= \gamma_{\{64, 67, 72, 87, 90, 100, 104, 107, 113, 116, 136, 137, 140, 147\}}.
\end{align*}
From this, we obtain
\begin{equation}  \begin{cases}
a_{11} = a_{12} = 0, \ \ \gamma_j = 0,\ j = 69, 70, 99,\\ 
100, 103, 104, 107, 109, 113, 115, 116, 141,\\
\gamma_{\{69, 88, 125, 135\}} =
\gamma_{\{88, 108\}}  = 0,\\
\gamma_{\{70, 72, 90, 128, 137\}} = 
\gamma_{\{140, 147\}} =0.  
\end{cases}\tag{\ref{6.4.3}.4}
\end{equation} 

With the aid of (\ref{6.4.3}.2), (\ref{6.4.3}.3) and (\ref{6.4.3}.4), the homomorphisms $g_3, g_4$ send (\ref{6.4.3}.1) respectively to
\begin{align*}
&\gamma_{\{71, 74, 92, 122, 138\}}w_{t,1,4}  +  \gamma_{\{122, 125, 128\}}w_{t,1,6}  +  a_{13}w_{t,1,7}\\
&\quad  +   \gamma_{71}w_{t,1,8}  +   \gamma_{101}w_{t,1,9}  +   a_{14}w_{t,1,10}  +   \gamma_{105}w_{t,1,11} \\
&\quad+   a_{15}w_{t,1,12}  +   \gamma_{\{112, 119, 150\}}w_{t,1,13}  +   \gamma_{117}w_{t,1,14}  = 0,\\   
&\gamma_{\{71, 74, 92, 122, 138\}}w_{t,1,4}  +  \gamma_{\{122, 125, 128\}}w_{t,1,6}  +  a_{16}w_{t,1,7}  +   \gamma_{74}w_{t,1,8}\\
&\quad  +    \gamma_{102}w_{t,1,9}  +   a_{17}w_{t,1,10}  +   \gamma_{106}w_{t,1,11}  +   \gamma_{\{88, 92, 111, 147, 150\}}w_{t,1,12}\\
&\quad  +   \gamma_{\{114, 119, 120, 121, 150\}}w_{t,1,13}  +   \gamma_{118}w_{t,1,14}  =0, 
\end{align*}
where
\begin{align*}
a_{13} &=  \gamma_{\{63, 64, 67, 72, 87, 90, 101, 102, 105, 106, 110, 111, 112, 114, 117, \ldots , 121, 134, 136, 138, 139, 147\}},\\ 
a_{14} &= \gamma_{\{72, 74, 90, 102, 106, 111, 114, 118, 139\}},\ \
a_{15} = \gamma_{\{88, 92, 110, 120, 121, 147, 150\}},\\
a_{16} &= \gamma_{\{63, 64, 67, 72, 87, 90, 101, 102, 105, 106, 108, 110, 111, 112, 114, 117, \ldots , 121, 135, 137, \ldots , 140, 147\}},\\
a_{17} &= \gamma_{\{63, 64, 67, 71, 87, 101, 105, 108, 110, 112, 117, 139, 140\}}.
\end{align*} 
From these equalities, we get
\begin{equation} \begin{cases}
 a_{13} = a_{14} = a_{15} = a_{16} = a_{17} = 0,\\
 \gamma_j = 0, \ j = 71, 74, 101, 102, 105, 106, 117, 118,\\
\gamma_{\{71, 74, 92, 122, 138\}} =  \gamma_{\{122, 125, 128\}} = 0,\\   
\gamma_{\{112, 119, 150\}} =  \gamma_{\{71, 74, 92, 122, 138\}} = 0,\\  
\gamma_{\{122, 125, 128\}} =   \gamma_{\{88, 92, 111, 147, 150\}} =   
\gamma_{\{114, 119, 120, 121, 150\}} = 0.  
\end{cases}\tag{\ref{6.4.3}.5}
\end{equation} 

With the aid of (\ref{6.4.3}.2), (\ref{6.4.3}.3), (\ref{6.4.3}.4) and (\ref{6.4.3}.5) the homomorphism $h$ sends (\ref{6.4.3}.1) to
\begin{align*}
&a_{18}w_{t,1,7}  +  \gamma_{\{88, 111, 114, 139\}}w_{t,1,8}  +    \gamma_{92}w_{t,1,9}  +   \gamma_{119}w_{t,1,10}\\
&\quad   +   a_{19}w_{t,1,11} +   \gamma_{\{112, 114, 119, 150\}}w_{t,1,12}  +   \gamma_{120}w_{t,1,13} +   \gamma_{121}w_{t,1,14}  = 0, 
\end{align*}
where
\begin{align*}
a_{18} &= \gamma_{\{92, 111, 114, 119, 120, 121, 136, 137, 138, 140, 150\}},\\
a_{19} &= \gamma_{\{88, 92, 110, 111, 120, 121, 140, 150\}}.
\end{align*}
This equality implies
\begin{equation}\begin{cases}
a_{18} = a_{19} = \gamma_{92} = \gamma_{119} = \gamma_{120} = \gamma_{121} = 0,\\
\gamma_{\{88, 111, 114, 139\}}= \gamma_{\{112, 114, 119, 150\}}=0.
\end{cases}\tag{\ref{6.4.3}.6}
\end{equation}

Combining (\ref{6.4.3}.2) - (\ref{6.4.3}.6), we get $\gamma_j = 0$ for all $j$. The proposition is proved.
\end{proof}

\begin{props}\label{6.4.4} For any $t \geqslant 4$, the elements $[a_{t,2,j}], \ 57 \leqslant j \leqslant 195,$ are linearly independent in $(\mathbb F_2\underset {\mathcal A}\otimes R_4)_{2^{t+3}+5}$.
\end{props}

\begin{proof} Suppose there is a linear relation
\begin{equation}\sum_{j=57}^{195}\gamma_j[a_{t,2,j}]=0,\tag{\ref{6.4.4}.1}
\end{equation}
where $\gamma_j \in \mathbb F_2.$ 

Applying the homomorphisms $f_i, 1 \leqslant i \leqslant 6,$ to the relation (\ref{6.4.4}.1), we get
\begin{align*}
&\gamma_{150}w_{t,2,1}  +  \gamma_{151}w_{t,2,2}  +    \gamma_{124}w_{t,2,3}  +   \gamma_{165}w_{t,2,4}  +   \gamma_{126}w_{t,2,5} \\
&\quad +   \gamma_{127}w_{t,2,6}  +   \gamma_{133}w_{t,2,7}  +   \gamma_{57}w_{t,2,8}  +   \gamma_{65}w_{t,2,9}  +   \gamma_{66}w_{t,2,10}\\
&\quad  +   \gamma_{75}w_{t,2,11}  +   \gamma_{76}w_{t,2,12}  +   \gamma_{89}w_{t,2,13}  +   \gamma_{93}w_{t,2,14}  = 0,\\
&\gamma_{152}w_{t,2,1}  +   \gamma_{154}w_{t,2,2}  +   \gamma_{\{123, 191\}}w_{t,2,3}  +   \gamma_{166}w_{t,2,4}  +   \gamma_{\{125, 192\}}w_{t,2,5}\\
&\quad  +   \gamma_{130}w_{t,2,6}  +   \gamma_{134}w_{t,2,7}  +   \gamma_{58}w_{t,2,8} +   \gamma_{63}w_{t,2,9}  +   \gamma_{68}w_{t,2,10}\\
&\quad  +   \gamma_{79}w_{t,2,11}  +   \gamma_{77}w_{t,2,12}  +   \gamma_{88}w_{t,2,13}  +   \gamma_{94}w_{t,2,14}  = 0,\\   
&\gamma_{153}w_{t,2,1}  +  \gamma_{155}w_{t,2,2}  + \gamma_{\{122, 190, 195\}}w_{t,2,3}   +   \gamma_{167}w_{t,2,4} +   \gamma_{\{128, 193\}}w_{t,2,5}\\
&\quad  +   \gamma_{\{129, 194\}}w_{t,2,6}   +   \gamma_{136}w_{t,2,7}  +   \gamma_{59}w_{t,2,8} +   \gamma_{64}w_{t,2,9}  +   \gamma_{67}w_{t,2,10}\\
&\quad    +   \gamma_{80}w_{t,2,11} +   \gamma_{81}w_{t,2,12}  +   \gamma_{87}w_{t,2,13}  +   \gamma_{96}w_{t,2,14}  = 0,
\end{align*}
\begin{align*}
&\gamma_{156}w_{t,2,1}  +  \gamma_{\{163, 184, 188, 191\}}w_{t,2,2}  +   \gamma_{159}w_{t,2,3}   +   a_1w_{t,2,4}  +   \gamma_{168}w_{t,2,5}\\
&\quad  +   \gamma_{132}w_{t,2,6}  +   \gamma_{135}w_{t,2,7}  +   \gamma_{69}w_{t,2,8}  +   \gamma_{60}w_{t,2,9}  +   \gamma_{73}w_{t,2,10}\\
&\quad  +   \gamma_{82}w_{t,2,11}  +   \gamma_{91}w_{t,2,12}  +   \gamma_{78}w_{t,2,13}  +   \gamma_{95}w_{t,2,14}  = 0,\\   
&\gamma_{157}w_{t,2,1}  +  a_2w_{t,2,2}  +   \gamma_{160}w_{t,2,3}  +   a_3w_{t,2,4}  +  \gamma_{169}w_{t,2,5}\\
&\quad  +   \gamma_{\{131, 195\}}w_{t,2,6}  +    \gamma_{137}w_{t,2,7}  +   \gamma_{70}w_{t,2,8}  +   \gamma_{61}w_{t,2,9} +   \gamma_{72}w_{t,2,10}\\
&\quad  +   \gamma_{83}w_{t,2,11}  +   \gamma_{90}w_{t,2,12}  +   \gamma_{85}w_{t,2,13}  +   \gamma_{97}w_{t,2,14}  = 0,\\   
&a_4w_{t,2,1}  +  \gamma_{158}w_{t,2,2}  +   \gamma_{161}w_{t,2,3}  +   a_5w_{t,2,4}  +   a_6w_{t,2,5}\\
&\quad  +   \gamma_{170}w_{t,2,6}  +   a_7w_{t,2,7}  +   \gamma_{71}w_{t,2,8}  +   \gamma_{74}w_{t,2,9}  +   \gamma_{62}w_{t,2,10}\\
&\quad  +   \gamma_{92}w_{t,2,11}  +   \gamma_{84}w_{t,2,12}  +   \gamma_{86}w_{t,2,13}  +   \gamma_{98}w_{t,2,14}  = 0,   
\end{align*}
where
\begin{align*}
a_1 &= \gamma_{\{99, 103, 108, 109, 115, 141, 143, 171, 175, 179, 192\}},\ \
a_2 = \gamma_{\{162, 185, 187, 190, 194\}},\\
a_3 &= \gamma_{\{100, 104, 107, 113, 116, 140, 144, 172, 176, 180, 193\}},\ \
a_4 = \gamma_{\{164, 183, 186, 190, 191, 192, 193\}},\\
a_5 &= \gamma_{\{101, 105, 110, 112, 117, 139, 145, 173, 177, 181, 194\}},\\
a_6 &= \gamma_{\{102, 106, 111, 114, 118, 142, 146, 174, 178, 182, 195\}},\\
a_7 &= \gamma_{\{119, 120, 121, 138, 147, 148, 149, 189\}}.
\end{align*}

From these equalities, we get

\begin{equation}\begin{cases}
a_1 = a_2 = a_3 = a_4 = a_5 = a_6 = a_7 = 0,\\
\gamma_j = 0,\ j = 57, \ldots , 98, 124, 126, 127,\\ 130,
 132 \ldots, 137, 150,\ldots, 161, 165,\ldots, 170,\\
\gamma_{\{123, 191\}} =  
\gamma_{\{125, 192\}} =  
\gamma_{\{122, 190, 195\}} =   
\gamma_{\{128, 193\}} =  0,\\
\gamma_{\{129, 194\}} =   
\gamma_{\{131, 195\}} =  
\gamma_{\{163, 184, 188, 191\}} = 0.
\end{cases} \tag{\ref{6.4.4}.2}
\end{equation}

With the aid of (\ref{6.4.4}.2), the homomorphisms $g_1, g_2$ send (\ref{6.4.4}.1) respectively to
\begin{align*}
&\gamma_{163}w_{t,2,2}  +   \gamma_{184}w_{t,2,3}  +    \gamma_{171}w_{t,2,4}  +   \gamma_{175}w_{t,2,5}  +   \gamma_{188}w_{t,2,6}\\
&\quad  +   \gamma_{179}w_{t,2,7}  +   \gamma_{143}w_{t,2,8}  +   \gamma_{99}w_{t,2,9}  +   \gamma_{141}w_{t,2,10}\\
&\quad  +   \gamma_{103}w_{t,2,11}  +   \gamma_{108}w_{t,2,12}  +   \gamma_{109}w_{t,2,13}  +   \gamma_{115}w_{t,2,14}  = 0,\\   
&\gamma_{\{122, 162, 190\}}w_{t,2,2}  +  \gamma_{185}w_{t,2,3}  +   \gamma_{172}w_{t,2,4} +   \gamma_{176}w_{t,2,5}\\
&\quad   +   \gamma_{187}w_{t,2,6}  +   \gamma_{180}w_{t,2,7}  +   \gamma_{144}w_{t,2,8}  +   \gamma_{100}w_{t,2,9}  +   \gamma_{140}w_{t,2,10}\\
&\quad  +   \gamma_{104}w_{t,2,11}  +   \gamma_{107}w_{t,2,12}  +   \gamma_{113}w_{t,2,13}  +   \gamma_{116}w_{t,2,14} = 0.
\end{align*}

From these equalities, we obtain
\begin{equation}\begin{cases}
\gamma_j = 0, \ j = 99, 100, 103, 104, 107, 108, 109,\\ 113, 115, 116, 140, 141, 143, 144, 163, 171,\\ 172, 175, 176, 179, 180, 184, 185, 187, 188,\\
\gamma_{\{122, 162, 190\}}=0.
\end{cases} \tag{\ref{6.4.4}.3}
\end{equation}

With the aid of (\ref{6.4.4}.2) and (\ref{6.4.4}.3), the homomorphisms $g_3, g_4$ send (\ref{6.4.4}.1) respectively to
\begin{align*}
&a_8w_{t,2,2}  +  \gamma_{\{178, 183\}}w_{t,2,3}  +   \gamma_{\{138, 173\}}w_{t,2,4}  +  \gamma_{\{131, 174, 178, 186\}}w_{t,2,6}\\
&\quad   +   \gamma_{177}w_{t,2,5} +   \gamma_{\{138, 181, 189\}}w_{t,2,7}  +    \gamma_{145}w_{t,2,8}  +   \gamma_{\{139, 189\}}w_{t,2,10}\\
&\quad   +   \gamma_{101}w_{t,2,9} +   \gamma_{105}w_{t,2,11}  +   \gamma_{110}w_{t,2,12}  +   \gamma_{112}w_{t,2,13}  +   \gamma_{117}w_{t,2,14}  =0,\\
&a_9w_{t,2,2}  +  \gamma_{\{123, 177, 183\}}w_{t,2,3}  +   \gamma_{\{138, 174\}}w_{t,2,4}  +   \gamma_{178}w_{t,2,5}\\
&\quad  +   a_{10}w_{t,2,6}  +  a_{11}w_{t,2,7}  +   \gamma_{146}w_{t,2,8}  +   \gamma_{102}w_{t,2,9}  +   a_{12}w_{t,2,10}\\
&\quad  +   \gamma_{106}w_{t,2,11}  +   \gamma_{111}w_{t,2,12}  +   \gamma_{114}w_{t,2,13}  +   \gamma_{118}w_{t,2,14}  = 0,     
\end{align*}
where
\begin{align*}
a_8 &= \gamma_{\{122, 131, 164, 174, 190\}},\ \
a_9 = \gamma_{\{129, 162, 164, 173, 190\}},\\
a_{10} &= \gamma_{\{125, 128, 129, 173, 177, 186\}},\ \
a_{11} =  \gamma_{\{119, 120, 121, 138, 147, 148, 149, 182, 189\}},\\
a_{12} &= \gamma_{\{119, 120, 121, 142, 147, 148, 149, 189\}}.
\end{align*}

Computing directly from these equalities, we get
\begin{equation}\begin{cases}
a_8 = a_9 = a_{10} = a_{11} = a_{12} =0,\\
\gamma_j = 0,\ j = 101, 102, 105, 106, 110, 111,\\
112, 114, 117, 118, 145, 146, 177, 178,\\
\gamma_{\{178, 183\}}= 
\gamma_{\{138, 173\}}  = 
\gamma_{\{131, 174, 178, 186\}} = 
\gamma_{\{139, 189\}} = 0,\\
\gamma_{\{138, 181, 189\}} = 
\gamma_{\{123, 177, 183\}} = 
\gamma_{\{138, 174\}}
= 0.
\end{cases} \tag{\ref{6.4.4}.4}
\end{equation}

With the aid of (\ref{6.4.4}.2), (\ref{6.4.4}.3) and (\ref{6.4.4}.4), the homomorphism $h$ sends (\ref{6.4.4}.1) to
\begin{align*}
&\gamma_{\{122, 138, 162, 164, 190\}}w_{t,2,1} + \gamma_{\{138, 186\}}w_{t,2,4} +   \gamma_{\{138, 181, 182\}}w_{t,2,7}\\
&\quad+  \gamma_{142}w_{t,2,8} +  \gamma_{147}w_{t,2,9}+  \gamma_{119}w_{t,2,10} +  \gamma_{148}w_{t,2,11} \\
&\quad +  \gamma_{149}w_{t,2,12} +  \gamma_{120}w_{t,2,13} +  \gamma_{121}w_{t,2,14}= 0. 
\end{align*}

From this, it implies
\begin{equation}\begin{cases}
\gamma_{119} = \gamma_{120} = \gamma_{121} = \gamma_{142} = \gamma_{147} = \gamma_{148} = \gamma_{149} = 0,\\
\gamma_{\{122, 138, 162, 164, 190\}} = 
\gamma_{\{138, 186\}}= \gamma_{\{138, 181, 182\}} =0.
\end{cases} \tag{\ref{6.4.4}.5}
\end{equation}

Combining (\ref{6.4.4}.2), (\ref{6.4.4}.3), (\ref{6.4.4}.4) and (\ref{6.4.4}.5), we obtain $\gamma_j = 0$ for all $j$. The proposition is proved.
\end{proof}

\begin{props}\label{6.4.5} For any $t \geqslant 4$ and $s\geqslant 3$, the elements $[a_{t,s,j}], \ 57 \leqslant j \leqslant 210,$ are linearly independent in $(\mathbb F_2\underset {\mathcal A}\otimes R_4)_{2^{s+t+1} + 2^{s+1} -3}$.
\end{props}

\begin{proof} Suppose there is a linear relation
\begin{equation}\sum_{j=57}^{210}\gamma_j[a_{t,s,j}]=0,\tag{\ref{6.4.5}.1}
\end{equation}
where $\gamma_j \in \mathbb F_2.$ 

Apply the homomorphisms $f_i,  1 \leqslant i \leqslant 6$ to the relation (\ref{6.4.5}.1) and we obtain
\begin{align*}
&\gamma_{150}w_{t,s,1} +  \gamma_{151}w_{t,s,2} +   \gamma_{124}w_{t,s,3} +  \gamma_{165}w_{t,s,4} +  \gamma_{126}w_{t,s,5}\\
&\quad +  \gamma_{127}w_{t,s,6} +  \gamma_{133}w_{t,s,7} +  \gamma_{57}w_{t,s,8} +  \gamma_{65}w_{t,s,9} +  \gamma_{66}w_{t,s,10}\\
&\quad +  \gamma_{75}w_{t,s,11} +  \gamma_{76}w_{t,s,12} +  \gamma_{89}w_{t,s,13} +  \gamma_{93}w_{t,s,14} =0,\\
&\gamma_{152}w_{t,s,1} + \gamma_{154}w_{t,s,2} +  \gamma_{123}w_{t,s,3} +  \gamma_{166}w_{t,s,4} +  \gamma_{125}w_{t,s,5}\\
&\quad +  \gamma_{130}w_{t,s,6} +  \gamma_{134}w_{t,s,7} +  \gamma_{58}w_{t,s,8} +  \gamma_{63}w_{t,s,9} +  \gamma_{68}w_{t,s,10}\\
&\quad +  \gamma_{79}w_{t,s,11} +  \gamma_{77}w_{t,s,12} +  \gamma_{88}w_{t,s,13} +  \gamma_{94}w_{t,s,14} = 0,\\  
&\gamma_{153}w_{t,s,1} +  \gamma_{155}w_{t,s,2} + a_1w_{t,s,3} +  \gamma_{167}w_{t,s,4} +   \gamma_{128}w_{t,s,5}\\
&\quad +  \gamma_{129}w_{t,s,6} +  \gamma_{136}w_{t,s,7} +  \gamma_{59}w_{t,s,8} +  \gamma_{64}w_{t,s,9} +  \gamma_{67}w_{t,s,10}\\
&\quad +  \gamma_{80}w_{t,s,11} +  \gamma_{81}w_{t,s,12} +  \gamma_{87}w_{t,s,13} +  \gamma_{96}w_{t,s,14} = 0,\\
&\gamma_{156}w_{t,s,1} +  \gamma_{163}w_{t,s,2} +  \gamma_{159}w_{t,s,3} +  \gamma_{171}w_{t,s,4} +  \gamma_{168}w_{t,s,5}\\
&\quad +  \gamma_{132}w_{t,s,6} +  \gamma_{135}w_{t,s,7} +  \gamma_{69}w_{t,s,8} +  \gamma_{60}w_{t,s,9} +  \gamma_{73}w_{t,s,10}\\
&\quad +  \gamma_{82}w_{t,s,11} +  \gamma_{91}w_{t,s,12} +  \gamma_{78}w_{t,s,13} +  \gamma_{95}w_{t,s,14} =0,\\  
&\gamma_{157}w_{t,s,1} + a_2w_{t,s,2} +  \gamma_{160}w_{t,s,3} +  \gamma_{172}w_{t,s,4} + \gamma_{169}w_{t,s,5}\\
&\quad +  \gamma_{131}w_{t,s,6} +   \gamma_{137}w_{t,s,7} +  \gamma_{70}w_{t,s,8} +  \gamma_{61}w_{t,s,9} +  \gamma_{72}w_{t,s,10}\\
&\quad +  \gamma_{83}w_{t,s,11} +  \gamma_{90}w_{t,s,12} +  \gamma_{85}w_{t,s,13} +  \gamma_{97}w_{t,s,14} = 0,\\  
&a_3w_{t,s,1} + \gamma_{158}w_{t,s,2} +  \gamma_{161}w_{t,s,3} +  \gamma_{173}w_{t,s,4} +  \gamma_{174}w_{t,s,5}\\
&\quad +  \gamma_{170}w_{t,s,6} +  \gamma_{138}w_{t,s,7} +  \gamma_{71}w_{t,s,8} +  \gamma_{74}w_{t,s,9} +  \gamma_{62}w_{t,s,10}\\
&\quad +  \gamma_{92}w_{t,s,11} +  \gamma_{84}w_{t,s,12} +  \gamma_{86}w_{t,s,13} +  \gamma_{98}w_{t,s,14} = 0, 
\end{align*}
where
\begin{align*}
a_1 &= \begin{cases}  \gamma_{\{122, 209\}}, & s = 3,\\  \gamma_{122}, &s \geqslant 4,\end{cases}\ \
a_2 = \begin{cases}  \gamma_{\{162, 209\}}, & s = 3,\\  \gamma_{162}, &s \geqslant 4,\end{cases}\\
a_3 &= \begin{cases}  \gamma_{\{164, 208, 209, 210\}}, & s = 3,\\  \gamma_{164}, &s \geqslant 4.\end{cases}\\
\end{align*}
From this, it implies 
\begin{equation}\begin{cases}
a_1 = a_2 = a_3 = 0,\\
\gamma_j = 0,\ j= 57, \ldots,98, 123, \ldots , 138,\\
 150, \ldots, 161, 163, 165, \ldots, 174.
\end{cases} \tag{\ref{6.4.5}.2}
\end{equation}

With the aid of (\ref{6.4.5}.2), the homomorphisms $g_i,\ i=1,2,3,4,$ send (\ref{6.4.5}.1) respectively to
\begin{align*}
&\gamma_{190}w_{t,s,1} + \gamma_{195}w_{t,s,2} +   \gamma_{184}w_{t,s,3} +  \gamma_{199}w_{t,s,4} +  \gamma_{175}w_{t,s,5}\\
&\quad +  \gamma_{188}w_{t,s,6} +  \gamma_{179}w_{t,s,7} +  \gamma_{143}w_{t,s,8} +  \gamma_{99}w_{t,s,9} +  \gamma_{141}w_{t,s,10}\\
&\quad +  \gamma_{103}w_{t,s,11} +  \gamma_{108}w_{t,s,12} +  \gamma_{109}w_{t,s,13} +  \gamma_{115}w_{t,s,14} = 0,
\end{align*}
\begin{align*}
&\gamma_{191}w_{t,s,1} + a_4w_{t,s,2} +  \gamma_{185}w_{t,s,3} +  \gamma_{200}w_{t,s,4} +  \gamma_{176}w_{t,s,5}\\
&\quad +  \gamma_{187}w_{t,s,6} +  \gamma_{180}w_{t,s,7} +  \gamma_{144}w_{t,s,8} +  \gamma_{100}w_{t,s,9} +  \gamma_{140}w_{t,s,10}\\
&\quad +  \gamma_{104}w_{t,s,11} +  \gamma_{107}w_{t,s,12} +  \gamma_{113}w_{t,s,13} +  \gamma_{116}w_{t,s,14} = 0,\\
&\gamma_{192}w_{t,s,1} + a_5w_{t,s,2} + a_6w_{t,s,3} +  \gamma_{201}w_{t,s,4} +   \gamma_{177}w_{t,s,5}\\
&\quad +  a_7w_{t,s,6} +  \gamma_{181}w_{t,s,7} +  \gamma_{145}w_{t,s,8} +  \gamma_{101}w_{t,s,9} +  \gamma_{139}w_{t,s,10}\\
&\quad +  \gamma_{105}w_{t,s,11} +  \gamma_{110}w_{t,s,12} +  \gamma_{112}w_{t,s,13} +  \gamma_{117}w_{t,s,14} = 0,\\
&\gamma_{193}w_{t,s,1} + a_8w_{t,s,2} +  a_9w_{t,s,3} +  \gamma_{202}w_{t,s,4} +  \gamma_{178}w_{t,s,5}\\
&\quad +  a_{10}w_{t,s,6} +  \gamma_{182}w_{t,s,7} +  \gamma_{146}w_{t,s,8} +  \gamma_{102}w_{t,s,9} +  \gamma_{142}w_{t,s,10}\\
&\quad +  \gamma_{106}w_{t,s,11} +  \gamma_{111}w_{t,s,12} +  \gamma_{114}w_{t,s,13} +  \gamma_{118}w_{t,s,14} = 0,  
\end{align*}
where
\begin{align*}
a_4 &= \begin{cases}  \gamma_{\{122, 194\}}, & s = 3,\\  \gamma_{194}, &s \geqslant 4,\end{cases}\ \
a_5 = \begin{cases}  \gamma_{\{164, 196\}}, & s = 3,\\  \gamma_{196}, &s \geqslant 4,\end{cases}\\
a_6 &= \begin{cases}  \gamma_{\{183, 208\}}, & s = 3,\\  \gamma_{183}, &s \geqslant 4,\end{cases}\ \
a_7 = \begin{cases}  \gamma_{\{186, 210\}}, & s = 3,\\  \gamma_{186}, &s \geqslant 4,\end{cases}\\
a_8 &= \begin{cases}  \gamma_{\{164, 197, 204, 205\}}, & s = 3,\\  \gamma_{197}, &s \geqslant 4,\end{cases}\ \
a_9 = \begin{cases}  \gamma_{\{198, 204, 206, 208\}}, & s = 3,\\  \gamma_{198}, &s \geqslant 4,\end{cases}\\
a_{10} &= \begin{cases}  \gamma_{\{203, 205, 206, 210\}}, & s = 3,\\  \gamma_{203}, &s \geqslant 4.\end{cases}
\end{align*}

Hence, we obtain
\begin{equation}\begin{cases}
a_i = 0, i = 4,5, \ldots , 10,\\
\gamma_j = 0,\ j = 99,\ldots, 118, 139,\ldots, 146,175,\ldots, 182,\\
 184,185, 187,188, 190,\ldots, 193, 195, 199,\ldots, 202.
\end{cases} \tag{\ref{6.4.5}.3}
\end{equation}

With the aid of (\ref{6.4.5}.2) and (\ref{6.4.5}.3), the homomorphism $h$ sends (\ref{6.4.5}.1)  to
\begin{align*}
&a_{11}w_{t,s,1} +  a_{12}w_{t,s,2} +  \gamma_{204}w_{t,s,3} +  a_{13}w_{t,s,4} + \gamma_{205}w_{t,s,5}\\
&\quad +  \gamma_{206}w_{t,s,6} +   \gamma_{207}w_{t,s,7} +  \gamma_{189}w_{t,s,8} +  \gamma_{147}w_{t,s,9} +  \gamma_{119}w_{t,s,10}\\
&\quad +  \gamma_{148}w_{t,s,11} +  \gamma_{149}w_{t,s,12} +  \gamma_{120}w_{t,s,13} +  \gamma_{121}w_{t,s,14} = 0, 
\end{align*}
where
$$
a_{11} = \begin{cases}  \gamma_{197}, & s = 3,\\  \gamma_{208}, &s \geqslant 4,\end{cases}\ \
a_{12} = \begin{cases}  \gamma_{198}, & s = 3,\\  \gamma_{209}, &s \geqslant 4,\end{cases}\ \
a_{13} = \begin{cases}  \gamma_{203}, & s = 3,\\  \gamma_{210}, &s \geqslant 4.\end{cases}
$$

From these equalities, we get
\begin{equation}\begin{cases}
a_{11} = a_{12} = a_{13} = 0,\
\gamma_j = 0, \ j = 119, 120,\\ 121, 147, 148, 149, 189, 204, 205, 206, 207.
\end{cases}\tag{\ref{6.4.5}.4}
\end{equation}

Combining (\ref{6.4.5}.2), (\ref{6.4.5}.3) and (\ref{6.4.5}.4), we obtain $\gamma_j = 0$ for all $j$. The proposition is proved.
\end{proof}

\section{The indecomposables of $P_4$ in degree $2^{s+t} + 2^s-2$}\label{7}

\subsection{The $\tau$-sequence of the admissible monomials}\label{7.1}\

\medskip
In this subsection, we prove the following.

\begin{lems}\label{7.1.5} Let $x$ be an admissible monomial of degree $2^{s+t} + 2^s-2$ in $P_4$. We have

{\rm 1.} If $s=1$ and $ t=1$ then either $\tau(x) = (2;1)$ or $\tau(x) = (4;0)$. 

{\rm 2.} If $s=1$ and $t=2$ then either $\tau(x) = (2;1;1)$ or $\tau(x) = (2;3)$ or $\tau(x) = (4;2)$.

{\rm 3.} If $s=1$ and $t=3$ then either $\tau(x) = (2;1;1;1)$ or $\tau(x) = (2;3;2)$ or $\tau(x) = (4;2;2)$ or $\tau(x) = (4;4;1)$.

{\rm 4.} If $s = 1$ and $t \geqslant 4$ then $\tau(x)$ is one of the following sequences
$$(2;\underset{\text{$t$ times}}{\underbrace{1;1;\ldots ; 1}}),\quad 
(4;\underset{\text{$t-1$ times}}{\underbrace{2;2;\ldots ; 2}}),\quad
(4;4;\underset{\text{$t-3$ times}}{\underbrace{3;3;\ldots ; 3}};1).$$

{\rm 5.} If $s\geqslant 2$ then $\tau(x)$ is one of the following sequences
\begin{align*}
&(\underset{\text{$s$ times}}{\underbrace{2;2;\ldots ; 2}};\underset{\text{$t$ times}}{\underbrace{1;1;\ldots ; 1}}),\ \ 
(4;\underset{\text{$s-2$ times}}{\underbrace{3;3;\ldots ; 3}};\underset{\text{$t+1$ times}}{\underbrace{1;1;\ldots ; 1}}),\\
&(4;\underset{\text{$s-1$ times}}{\underbrace{3;3;\ldots ; 3}};\underset{\text{$t-1$ times}}{\underbrace{2;2;\ldots ; 2}}).
\end{align*}
\end{lems}

We need the following for the proof of this lemma. From Lemmas \ref{5.1}, \ref{2.5}, Theorem \ref{2.4} and the results in Section \ref{6}, we get

\begin{lems}\label{7.1.1a} Let $x$ be a monomial of degree $2^{s+1}+2^s-3$ in $P_4$ with $s \geqslant 2$. If  $x$ is inadmissible then there is a strictly inadmissible matrix $\Delta$ such that $\Delta \triangleright x$.
\end{lems} 

\begin{lems}\label{7.1.1} The following matrices are strictly inadmissible
$$ \begin{pmatrix} 0&1&0&1\\ 1&0&1&1\end{pmatrix} \quad  \begin{pmatrix} 0&1&1&0\\ 1&0&1&1\end{pmatrix} \quad  \begin{pmatrix} 0&0&1&1\\ 1&1&0&1\end{pmatrix} \quad  \begin{pmatrix} 0&0&1&1\\ 1&1&1&0\end{pmatrix} $$    $$ \begin{pmatrix} 0&1&1&0\\ 1&1&0&1\end{pmatrix} \quad  \begin{pmatrix} 0&1&0&1\\ 1&1&1&0\end{pmatrix} \quad  \begin{pmatrix} 1&0&1&0\\ 1&1&0&1\end{pmatrix} \quad  \begin{pmatrix} 1&0&0&1\\ 1&1&1&0\end{pmatrix} .$$ 
\end{lems}

\begin{proof} The monomials corresponding to the above matrices respectively are
(2,1,2,3), (2,1,3,2),  (2,2,1,3), (2,2,3,1), (2,3,1,2), (2,3,2,1), (3,2,1,2), (3,2,2,1). We prove the lemma for the matrix associated with the monomial (3,2,1,2). The others can be obtained by a similar argument. Obviously, we have
$$(3,2,1,2) = Sq^1(3,1,1,2) + (4,1,1,2) + (3,1,2,2).$$
The lemma is proved.
\end{proof}

\begin{lems}\label{7.1.2} The following matrices are strictly inadmissible
 $$ \begin{pmatrix} 1&0&0&1\\ 0&1&1&1\\ 0&0&1&1\end{pmatrix} \quad  \begin{pmatrix} 1&0&1&0\\ 0&1&1&1\\ 0&0&1&1\end{pmatrix} \quad  \begin{pmatrix} 1&0&0&1\\ 0&1&1&1\\ 0&1&0&1\end{pmatrix} \quad  \begin{pmatrix} 1&0&1&0\\ 0&1&1&1\\ 0&1&1&0\end{pmatrix} $$    
$$ \begin{pmatrix} 1&1&0&0\\ 0&1&1&1\\ 0&1&0&1\end{pmatrix} \quad  \begin{pmatrix} 1&1&0&0\\ 0&1&1&1\\ 0&1&1&0\end{pmatrix} \quad  \begin{pmatrix} 1&1&0&0\\ 1&0&1&1\\ 1&0&0&1\end{pmatrix} \quad  \begin{pmatrix} 1&1&0&0\\ 1&0&1&1\\ 1&0&1&0\end{pmatrix} $$    
$$ \begin{pmatrix} 1&0&1&0\\ 0&1&1&1\\ 0&1&0&1\end{pmatrix} \quad  \begin{pmatrix} 1&0&0&1\\ 0&1&1&1\\ 0&1&1&0\end{pmatrix}.$$
\end{lems}

\begin{proof} The monomials corresponding to the above matrices respectively are
(1,2,6,7), (1,2,7,6),  (1,6,2,7), (1,6,7,2), (1,7,2,6), (1,7,6,2), (7,1,2,6), (7,1,6,2), (1,6,3,6), (1,6,6,3). It suffices to prove the lemma for the matrices associated with the monomials (1,2,6,7), (1,6,3,6). A simple computation shows
\begin{align*}
(1,2,6,7) &= Sq^1(2,1,5,7)+ Sq^2\big((1,1,5,7)+(1,1,3,9)) \ \text{mod }\mathcal L_4(2;3;2),\\
(1,6,3,6) &= Sq^1\big((2,5,3,5) + (2,3,5,5)\big) + Sq^2\big((1,5,3,5) \\
&\quad+ (1,3,5,5) \big)  + (1,3,6,6) \ \text{mod }\mathcal L_4(2;3;2).
\end{align*}
The lemma follows.
\end{proof}

\begin{lems}\label{7.1.3}  The following matrices are strictly inadmissible
$$ \begin{pmatrix} 1&0&0&1\\ 1&0&0&1\\ 0&1&1&1\end{pmatrix} \quad  \begin{pmatrix} 1&0&1&0\\ 1&0&1&0\\ 0&1&1&1\end{pmatrix} \quad  \begin{pmatrix} 1&1&0&0\\ 1&1&0&0\\ 0&1&1&1\end{pmatrix} \quad  \begin{pmatrix} 1&1&0&0\\ 1&1&0&0\\ 1&0&1&1\end{pmatrix} $$    
$$ \begin{pmatrix} 1&0&1&0\\ 1&0&0&1\\ 0&1&1&1\end{pmatrix} \quad  \begin{pmatrix} 1&1&0&0\\ 1&0&0&1\\ 0&1&1&1\end{pmatrix} \quad  \begin{pmatrix} 1&1&0&0\\ 1&0&1&0\\ 0&1&1&1\end{pmatrix} .$$
\end{lems}

\begin{proof} The monomials corresponding to the above matrices respectively are
(3,4,4,7), (3,4,7,4),  (3,7,4,4), (7,3,4,4), (3,4,5,6), (3,5,4,6), (3,5,6,4). It suffices to prove the lemma for the matrices associated with the monomials (3,4,4,7), (3,4,5,6). By a direct computation, we have
\begin{align*}
(3,4,4,7) &= Sq^1(3,1,2,11) + Sq^2(5,2,2,7)\\
&\quad  + Sq^4(3,2,2,7) \ \text{mod }\mathcal L_4(2;2;3),\\
(3,4,5,6) &= Sq^1(3,1,3,10) + Sq^2(5,2,3,6)\\
&\quad  + Sq^4(3,2,3,6)   \ \text{mod }\mathcal L_4(2;2;3).
\end{align*}
The lemma follows.
\end{proof}

\begin{lems}\label{7.1.4} The following matrices are strictly inadmissible
$$ \begin{pmatrix} 1&1&0&0\\ 0&1&1&1\\ 0&0&1&1\\ 0&0&1&1\end{pmatrix} \quad  \begin{pmatrix} 1&1&0&0\\ 1&0&1&1\\ 0&0&1&1\\ 0&0&1&1\end{pmatrix} \quad  \begin{pmatrix} 1&1&0&0\\ 1&0&1&1\\ 0&1&0&1\\ 0&1&0&1\end{pmatrix} \quad  \begin{pmatrix} 1&1&0&0\\ 1&0&1&1\\ 0&1&1&0\\ 0&1&1&0\end{pmatrix} $$    
$$ \begin{pmatrix} 1&1&0&0\\ 1&0&1&1\\ 0&1&0&1\\ 0&0&1&1\end{pmatrix} \quad  \begin{pmatrix} 1&1&0&0\\ 1&0&1&1\\ 0&1&1&0\\ 0&0&1&1\end{pmatrix} \quad  \begin{pmatrix} 1&1&0&0\\ 1&0&1&1\\ 0&1&1&0\\ 0&1&0&1\end{pmatrix} . $$
\end{lems}

\begin{proof} The monomials corresponding to the above matrices respectively are
\begin{align*}
&(1,3,14,14), (3,1,14,14),  (3,13,2,14), (3,13,14,2),\\
&(3,5,10,14), (3,5,14,10), (3,13,6,10).
\end{align*}

It suffices to prove the lemma for the matrices associated with the monomials (3,1,14,14), (3,13,2,14), (3,5,10,14), (3,13,6,10). By a direct computation, we have
\begin{align*}
&(3,1,14,14) = Sq^1\big((3,1,13,14) + (3,1,14,13)\big) + Sq^2\big((5,1,11,13)\\
&\quad + (5,1,7,17)\big) +Sq^4\big((3,1,11,13) + (3,1,7,17)\big) \quad
 \text{mod }\mathcal L_4(2;3;2;2).\\
&(3,13,2,14) = Sq^1\big((3,13,2,13) + (3,13,1,14)\big) + Sq^2\big((5,11,1,13)\\
&\quad + (5,7,1,17)\big) +Sq^4\big((3,11,1,13) + (3,7,1,17)\big) \quad
 \text{mod }\mathcal L_4(2;3;2;2).\\
&(3,5,10,14) = Sq^1\big((3,5,10,13) + (3,5,9,14)\big) + Sq^2\big((5,3,9,13)\\
&\quad + (5,3,5,17)\big) +Sq^4\big((3,3,9,13) + (3,3,5,17)\big) \quad
 \text{mod }\mathcal L_4(2;3;2;2).\\
&(3,13,6,10) = Sq^1\big((3,13,6,9) + (3,13,5,10)\big) + Sq^2\big((5,11,5,9)\\
&\quad + (5,7,9,9)\big) +Sq^4\big((3,11,5,9) + (3,7,9,9)\big) \quad
 \text{mod }\mathcal L_4(2;3;2;2).
\end{align*}
The lemma follows.
\end{proof}

\begin{proof}[Proof of Lemma \ref{7.1.5}] Observe that $z = (2^{s+t}-1, 2^s-1,0,0)$ is the minimal spike of degree $2^{s+t} + 2^s-2$ and $\tau(z) = (\underset{\text{$s$ times}}{\underbrace{2;2;\ldots ; 2}};\underset{\text{$t$ times}}{\underbrace{1;1;\ldots ; 1}})$. Since $2^{s+t} + 2^s - 2$ is even and $x$ is admissible, using Theorem \ref{2.12}, we get  either $\tau_1(x) =2$ or $\tau_1(x) =4$. Let $M=(\varepsilon_{ij}(x)), i\geqslant 1, 1\leqslant j\leqslant 4$, be the matrix associated with $x$. We set $M'=(\varepsilon_{ij}(x)), i\geqslant 2, 1\leqslant j\leqslant 4$ and denote by $x'$ the monomial corresponding to $M'$. 

Suppose $s=1$ and $t=1$. Then $\deg x = 4$. If $\tau_1(x) = 2$ then $\deg x' =1$. Hence $\tau(x') = (\tau_2(x)) = (1;0).$ So $\tau(x) = (2;1)$.  If $\tau_1(x) = 4$ then $\deg x' =0$. Hence $\tau(x) = (4;0)$. The first part of the lemma holds.

Suppose $s=1$ and $t=2$. Then we have $\deg x = 8$. If $\tau_1(x) = 2$ then $\deg x' =3$ and $x'$ is admissible. According to Lemma \ref{5.5}, either $\tau(x') = (\tau_2(x);\tau_3(x)) = (1;1),$ or $\tau(x') =  (3;0).$ So, either $\tau(x) = (2;1;1)$ or $\tau(x) = (2;3)$.  If $\tau_1(x) = 4$ then $\deg x' = 2$. Hence $\tau(x') = (2;0)$ and $\tau(x) = (4;2)$. The second part of the lemma holds.

Let $s=1$ and $t=3$. Then $\deg x = 16$. If $\tau_1(x) = 2$ then $\deg x' =7$. Since $x$ is admissible, by Lemma \ref{6.0a} and Theorem \ref{2.4}, $x'$ is admissible. According to Lemma \ref{5.5}, either $\tau(x') = (\tau_2(x);\tau_3(x)) = (1;1;1),$ or $\tau(x') =  (3;2),$ so, either $\tau(x) = (2;1;1;1)$ or $\tau(x) = (2;3;2)$. 
 
If $\tau_1(x) = 4$ then $\deg x' = 6$. Since $x$ is admissible, by Lemma \ref{5.1a} and Theorem \ref{2.4}, $x'$ is admissible. Using Lemma \ref{4.1}, we get either $\tau(x') = (2;2)$ or  $\tau(x') = (4;1)$. So, we obtain either $\tau(x) = (4;2;2)$ or $\tau(x) = (4;4;1)$. The lemma holds.

Let $s=1$ and $t\geqslant 4$. Then $\deg x = 2^{t+1}$. 
If $\tau_1(x) = 2$ then $\deg x' =2^t-1$. Since $x$ is admissible, by Lemma \ref{6.0a} and Theorem \ref{2.4}, $x'$ is admissible.  According to Lemma \ref{5.5}, either $\tau(x') = (\tau_2(x);\tau_3(x);\ldots) = (\underset{\text{$t$ times}}{\underbrace{1;1;\ldots ; 1}})$ or $\tau(x') =  (3;\underset{\text{$t-2$ times}}{\underbrace{2;2;\ldots ; 2}})$. 

If $\tau(x') =  (3;\underset{\text{$t-2$ times}}{\underbrace{2;2;\ldots ; 2}})$  then $\tau(x) = (2;3;\underset{\text{$t-2$ times}}{\underbrace{2;2;\ldots ; 2}})$ . From Lemmas \ref{4.2}, \ref{4.3}, \ref{5.1}, \ref{5.7}, \ref{7.1.1}, \ref{7.1.2},  \ref{7.1.4}, we see that $x$ is inadmissible and we have a contradiction. Hence, we get $\tau(x) = (2;\underset{\text{$t$ times}}{\underbrace{1;1;\ldots ; 1}})$. 

If $\tau_1(x) = 4$ then $\deg x' = 2^t-2$. By Lemma \ref{5.1a} and Theorem \ref{2.4}, $x'$ is admissible. Using Lemma \ref{4.1}, we get either $\tau(x') = (\underset{\text{$t-1$ times}}{\underbrace{2;2;\ldots ; 2}})$ or  $\tau(x') = (4;\underset{\text{$t-3$ times}}{\underbrace{3;3;\ldots ; 3}};1)$.  The lemma holds.

Suppose that $s \geqslant 2$.  If $\tau_1(x) = 2$ then $\deg x' =2^{s+t-1} + 2^{s-1}-2$. Using  Lemmas \ref{4.2}, \ref{4.3}, \ref{5.1}, \ref{5.7}, \ref{7.1.1},  \ref{7.1.3},  Proposition \ref{2.6} and by an inductive argument we see that $x'$ is admissible and $\tau(x') = (\underset{\text{$s-1$ times}}{\underbrace{2;2;\ldots ; 2}};\underset{\text{$t$ times}}{\underbrace{1;1;\ldots ; 1}}).$ Hence, we get
$$\tau(x) = (\underset{\text{$s$ times}}{\underbrace{2;2;\ldots ; 2}};\underset{\text{$t$ times}}{\underbrace{1;1;\ldots ; 1}}).$$  

If $\tau_1(x) = 4$ then $\deg x' = 2^{s+t-1} + 2^{s-1}-3$. Since $x$ is admissible, by Lemma \ref{7.1.1a} and Theorem \ref{2.4}, $x'$ is admissible. According to Lemma \ref{5.5} and Lemma \ref{6.1.1},  $\tau(x')$ is one of the following sequences
$$(\underset{\text{$s-2$ times}}{\underbrace{3;3;\ldots ; 3}};\underset{\text{$t+1$ times}}{\underbrace{1;1;\ldots ; 1}}),\quad
(\underset{\text{$s-1$ times}}{\underbrace{3;3;\ldots ; 3}};\underset{\text{$t-1$ times}}{\underbrace{2;2;\ldots ; 2}}).$$

The lemma is completely proved. 
\end{proof}

\subsection{The case $s=1$}\label{7.2}\

\medskip
According to Kameko \cite{ka}, $(\mathbb F_2\underset{\mathcal A}\otimes P_3)_{2^{t+1}}$ is an  $\mathbb F_2$-vector space with basis of the consisting of all the following classes: 

\smallskip
For $t\geqslant 1$,

\smallskip
\centerline{\begin{tabular}{ll}
$v_{t,1,1} = [1,1,2^{t+1} - 2],$& $v_{t,1,2} = [1,2^{t+1} - 2,1],$\cr 
$v_{t,1,3} = [0,1,2^{t+1} - 1],$& $v_{t,1,4} = [0,2^{t+1} - 1,1],$\cr 
$v_{t,1,5} = [1,0,2^{t+1} - 1],$& $v_{t,1,6} = [1,2^{t+1} - 1,0],$\cr 
$v_{t,1,7} = [2^{t+1}-1,0,1],$& $v_{t,1,8} = [2^{t+1}-1,1,0].$\cr 
\end{tabular}}

\smallskip
For $t\geqslant 2$

\smallskip
\centerline{\begin{tabular}{ll}
$v_{t,1,9} = [0,3,2^{t+1} - 3],$& $v_{t,1,10} = [3,0,2^{t+1} - 3],$\cr 
$v_{t,1,11} = [3,2^{t+1} - 3,0],$& $v_{t,1,12} = [1,2,2^{t+1} - 3],$\cr 
$v_{t,1,13} = [1,3,2^{t+1} - 4],$& $v_{t,1,14} = [3,1,2^{t+1} - 4].$\cr 
\end{tabular}}

\smallskip
For $t=2$,\  $v_{2,1,15} = [3,4,1].$

\medskip
So, we have

\begin{props}\label{7.2.4}\  $(\mathbb F_2\underset{\mathcal A}\otimes Q_4)_{2^{t+1}}$ is  an $\mathbb F_2$-vector space with a basis consisting of all the classes represented by the following monomials:

\smallskip
For $t \geqslant 1$,

\smallskip
\centerline{\begin{tabular}{ll}
$b_{t,1,1} = (0,1,1,2^{t+1} - 2),$& $b_{t,1,2} = (0,1,2^{t+1} - 2,1),$\cr 
$b_{t,1,3} = (1,0,1,2^{t+1} - 2),$& $b_{t,1,4} = (1,0,2^{t+1} - 2,1),$\cr 
$b_{t,1,5} = (1,1,0,2^{t+1} - 2),$& $b_{t,1,6} = (1,1,2^{t+1} - 2,0),$\cr 
$b_{t,1,7} = (1,2^{t+1} - 2,0,1),$& $b_{t,1,8} = (1,2^{t+1} - 2,1,0),$\cr 
$b_{t,1,9} = (0,0,1,2^{t+1} - 1),$& $b_{t,1,10} = (0,0,2^{t+1} - 1,1),$\cr 
$b_{t,1,11} = (0,1,0,2^{t+1} - 1),$& $b_{t,1,12} = (0,1,2^{t+1} - 1,0),$\cr 
$b_{t,1,13} = (0,2^{t+1} - 1,0,1),$& $b_{t,1,14} = (0,2^{t+1} - 1,1,0),$\cr 
$b_{t,1,15} = (1,0,0,2^{t+1} - 1),$& $b_{t,1,16} = (1,0,2^{t+1} - 1,0),$\cr 
$b_{t,1,17} = (1,2^{t+1} - 1,0,0),$& $b_{t,1,18} = (2^{t+1}-1,0,0,1),$\cr 
$b_{t,1,19} = (2^{t+1}-1,0,1,0),$& $b_{t,1,20} = (2^{t+1}-1,1,0,0),$\cr 
\end{tabular}}
\smallskip
For $t \geqslant 2$,

\smallskip
\centerline{\begin{tabular}{ll}
$b_{t,1,21} = (0,0,3,2^{t+1} - 3),$& $b_{t,1,22} = (0,3,0,2^{t+1} - 3).$\cr 
$b_{t,1,23} = (0,3,2^{t+1} - 3,0),$& $b_{t,1,24} = (3,0,0,2^{t+1} - 3),$\cr 
$b_{t,1,25} = (3,0,2^{t+1} - 3,0),$& $b_{t,1,26} = (3,2^{t+1} - 3,0,0),$\cr 
$b_{t,1,27} = (0,1,2,2^{t+1} - 3),$& $b_{t,1,28} = (1,0,2,2^{t+1} - 3),$\cr 
$b_{t,1,29} = (1,2,0,2^{t+1} - 3),$& $b_{t,1,30} = (1,2,2^{t+1} - 3,0),$\cr 
$b_{t,1,31} = (0,1,3,2^{t+1} - 4),$& $b_{t,1,32} = (0,3,1,2^{t+1} - 4),$\cr 
$b_{t,1,33} = (1,0,3,2^{t+1} - 4),$& $b_{t,1,34} = (1,3,0,2^{t+1} - 4),$\cr 
$b_{t,1,35} = (1,3,2^{t+1} - 4,0),$& $b_{t,1,36} = (3,0,1,2^{t+1} - 4),$\cr 
$b_{t,1,37} = (3,1,0,2^{t+1} - 4),$& $b_{t,1,38} = (3,1,2^{t+1} - 4,0).$\cr
\end{tabular}}

\smallskip
For $t = 2$,

\smallskip
\centerline{\begin{tabular}{ll}
$b_{2,1,39} = (0,3,4,1),$& $b_{2,1,40} = (3,0,4,1),$\cr 
$b_{2,1,41} = (3,4,0,1),$& $b_{2,1,42} = (3,4,1,0).$\cr 
\end{tabular}}
\end{props}

Now, we determine $(\mathbb F_2\underset{\mathcal A}\otimes R_4)_{2^{t+1}}$. We have 

\begin{thms}\label{dlc7.2} $(\mathbb F_2\underset{\mathcal A}\otimes R_4)_{2^{t+1}}$ is  an $\mathbb F_2$-vector space with a basis consisting of all the classes represented by the following monomials:

1. $\phi (b_{1,j}), 1 \leqslant j \leqslant 6$ for $t=2$.

2. $\phi (b_{t-1,j}), 1 \leqslant j \leqslant \mu_6(t-1)$, $\phi^2(a_{t-2,j}), 1 \leqslant j \leqslant \mu_1(t-2)$ for $t \geqslant 3$. Here $b_{t-1,j}$, is defined as in Section \ref{4},  $a_{t-3,j},  \mu_1(t-2)$ are defined as in Section \ref{3}, the homomorphism $\phi$ is defined as in Definition \ref{2.8}  and $\mu_6(1) = 6, \mu_6(2) = 24, \mu_6(t-1) = 35$ for any $t\geqslant 4$.

3.  For $t=1$, $c_{1,1,1} = (1,1,1,1).$

\smallskip
For $t\geqslant 2$,

\medskip
\centerline{\begin{tabular}{ll}
$c_{t,1,1} = (1,1,2,2^{t+1} - 4),$& $c_{t,1,2} = (1,2,1,2^{t+1} - 4),$\cr 
$c_{t,1,3} = (1,2,2^{t+1} - 4,1).$& \cr
\end{tabular}}

\medskip
For $t = 2$,

\smallskip
\centerline{\begin{tabular}{ll}
$c_{2,1,4} = (1,2,2,3),$& $c_{2,1,5} = (1,2,3,2),$\cr 
$c_{2,1,6} = (1,3,2,2),$& $c_{2,1,7} = (3,1,2,2).$ \cr
\end{tabular}}

\medskip
For $t\geqslant 3$,

\centerline{\begin{tabular}{ll}
$c_{t,1,4} = (1,2,4,2^{t+1} - 7),$& $c_{t,1,5} = (1,2,5,2^{t+1} - 8),$\cr
 $c_{t,1,6} = (1,3,4,2^{t+1} - 8),$& $c_{t,1,7} = (3,1,4,2^{t+1} - 8),$\cr
\end{tabular}}

\smallskip
For $t=3$,

\smallskip
\centerline{\begin{tabular}{ll}
$c_{3,1,8} = (1,3,6,6),$& $c_{3,1,9} = (3,1,6,6),$\cr 
$c_{3,1,10} = (3,5,2,6),$& $c_{3,1,11} = (3,5,6,2).$\cr
\end{tabular}}
\end{thms}

We prove this theorem by proving some propositions.

\begin{props}\label{mdc7.2} The $\mathbb F_2$-vector space $(\mathbb F_2\underset {\mathcal A}\otimes R_4)_{2^{t+1}}$ is generated by the  elements listed in Theorem \ref{dlc7.2}.
\end{props}

The proof of the proposition is based on the following lemmas.

\begin{lems}\label{7.2.1} The following matrices are strictly inadmissible
$$ \begin{pmatrix} 0&0&1&1\\ 0&1&0&0\end{pmatrix} \quad  \begin{pmatrix} 0&0&1&1\\ 1&0&0&0\end{pmatrix} \quad  \begin{pmatrix} 0&1&0&1\\ 1&0&0&0\end{pmatrix} \quad  \begin{pmatrix} 0&1&1&0\\ 1&0&0&0\end{pmatrix} .$$
\end{lems}

\begin{proof} See \cite{ka}
\end{proof}

\begin{lems}\label{7.2.2} The following matrices are strictly inadmissible
$$ \begin{pmatrix} 0&1&0&1\\ 0&1&0&0\\ 0&0&1&0\\ 0&0&1&0\end{pmatrix} \quad  \begin{pmatrix} 1&0&0&1\\ 1&0&0&0\\ 0&0&1&0\\ 0&0&1&0\end{pmatrix} \quad  \begin{pmatrix} 1&0&0&1\\ 1&0&0&0\\ 0&1&0&0\\ 0&1&0&0\end{pmatrix} \quad  \begin{pmatrix} 1&0&1&0\\ 1&0&0&0\\ 0&1&0&0\\ 0&1&0&0\end{pmatrix} $$    $$ \begin{pmatrix} 0&1&0&1\\ 0&1&0&0\\ 0&0&1&0\\ 0&0&0&1\end{pmatrix} \quad  \begin{pmatrix} 1&0&0&1\\ 1&0&0&0\\ 0&0&1&0\\ 0&0&0&1\end{pmatrix} \quad  \begin{pmatrix} 1&0&0&1\\ 1&0&0&0\\ 0&1&0&0\\ 0&0&0&1\end{pmatrix} \quad  \begin{pmatrix} 1&0&1&0\\ 1&0&0&0\\ 0&1&0&0\\ 0&0&1&0\end{pmatrix} $$    $$ \begin{pmatrix} 0&1&1&0\\ 0&1&0&0\\ 0&0&1&0\\ 0&0&0&1\end{pmatrix} \quad  \begin{pmatrix} 1&0&1&0\\ 1&0&0&0\\ 0&0&1&0\\ 0&0&0&1\end{pmatrix} \quad  \begin{pmatrix} 1&1&0&0\\ 1&0&0&0\\ 0&1&0&0\\ 0&0&0&1\end{pmatrix} \quad  \begin{pmatrix} 1&1&0&0\\ 1&0&0&0\\ 0&1&0&0\\ 0&0&1&0\end{pmatrix} .$$
\end{lems}

\begin{proof} See \cite{ka}
\end{proof}

\begin{lems}\label{7.2.3} The following matrices are strictly inadmissible
$$ \begin{pmatrix} 1&0&1&0\\ 1&0&0&0\\ 0&1&0&0\\ 0&0&0&1\end{pmatrix} \quad  \begin{pmatrix} 1&0&0&1\\ 1&0&0&0\\ 0&1&0&0\\ 0&0&1&0\end{pmatrix}.$$
\end{lems}

\begin{proof}  The monomials corresponding to the above matrices respectively are
(3,4,1,8), (3,4,8,1). We have
\begin{align*}
(3,4,1,8) &= Sq^1(3,1,1,10)+ Sq^2\big((5,2,1,6)+(5,1,2,6)\\ 
&\quad +  (2,1,1,10)\big)  +Sq^4\big((3,2,1,6) + (3,1,2,6) \big)\\ 
&\quad  + (2,1,1,12) + (3,1,4,8)  \quad \text{mod  }\mathcal L_4(2;1;1;1),\\
(3,4,8,1) &= Sq^1(3,1,10,1)+ Sq^2\big((5,2,6,1)+(5,1,6,2)\\ 
&\quad +  (2,1,10,1)\big)  +Sq^4\big((3,2,6,1) + (3,1,6,2) \big)\\ 
&\quad  + (2,1,12,1) + (3,1,8,4)   \quad \text{mod  }\mathcal L_4(2;1;1;1).
\end{align*}
The lemma follows.
\end{proof}

Using the results in Section \ref{4}, Section \ref{5}, Lemmas \ref{7.1.1}-\ref{7.1.5}, \ref{7.2.1}, \ref{7.2.2}, \ref{7.2.3} and Theorem \ref{2.4},  we get Proposition \ref{mdc7.2}.

\medskip
Now, we show that the elements listed in Theorem \ref{dlc7.2} are linearly independent.

The case $t=1$ is trivial. We begin with the case $t=2$.

\begin{props}\label{7.2.5} The elements $[c_{2,1,i}], 1 \leqslant i \leqslant 7$ and $[\phi(b_{1,j})], 1 \leqslant j \leqslant 6,$  are linearly independent in $(\mathbb F_2\underset {\mathcal A}\otimes R_4)_8$.
\end{props}

\begin{proof} Suppose that there is a linear relation
\begin{equation}\sum_{1\leqslant i\leqslant 7}\gamma_i[c_{2,1,i}] +\sum_{1\leqslant j \leqslant 6}\eta_j[\phi(b_{1,j})] = 0, \tag {\ref{7.2.5}.1}
\end{equation}
with $\gamma_i, \eta_j \in \mathbb F_2$. 

Apply the squaring operation $Sq^0_*$ to (\ref{7.2.5}.1) and we get
$$\sum_{1\leqslant j \leqslant 6}\eta_j[b_{1,j}]=0.$$
Since $\{[b_{1,j}]; 1\leqslant j \leqslant 6\}$ is a basis of $(\mathbb F_2 \underset {\mathcal A} \otimes P_4)_2,$ we obtain $\eta_j=0, 1\leqslant j \leqslant 6$. Hence, (\ref{7.2.5}.1) becomes
\begin{equation}\sum_{1\leqslant i\leqslant 7}\gamma_i[c_{2,1,i}] =0. \tag {\ref{7.2.5}.2}
\end{equation}
Now, we prove $\gamma_i=0, 1\leqslant i\leqslant 7.$ 

Applying the homomorphisms $f_1, f_2, f_4, f_6$ to the relation (\ref{7.2.5}.2), we obtain
\begin{align*}
&\gamma_{4}[1,1,6]  +  \gamma_{5}[1,6,1]  +    \gamma_{4}[1,2,5]\\
&\quad  +   \gamma_{5}[1,3,4]  +   \gamma_{\{2, 4\}}[3,1,4]  +   \gamma_{\{3, 5\}}[3,4,1]  = 0,\\   
&\gamma_{4}[1,1,6]  +  \gamma_{6}[1,6,1]  +   \gamma_{4}[1,2,5]  +   \gamma_{6}[1,3,4] \\
&\quad  +   \gamma_{\{1, 4, 7\}}[3,1,4]  +   \gamma_{\{3, 6\}}[3,4,1]  =0,\\   
&\gamma_{\{3, 7\}}[1,6,1]  +  \gamma_{4}[1,2,5]  +  \gamma_{\{1, 2, 5, 6, 7\}}[1,3,4]  +   \gamma_{7}[3,4,1]  =0,\\   
&\gamma_{1}[1,1,6]  +  \gamma_{\{2, 3, 4, 5\}}[1,2,5]  +   \gamma_{6}[1,3,4]  +   \gamma_{7}[3,1,4]  = 0.  
\end{align*}

From these equalities, we get
$\gamma_i = 0, \ 1 \leqslant i \leqslant 7. $
The proposition is proved.
\end{proof}

\begin{props}\label{7.2.6} The elements $[c_{3,1,i}], 1 \leqslant i \leqslant 11$,  $[\phi(b_{2,j})], 1 \leqslant j \leqslant 24,$ and $[\phi^2(a_{1,\ell})], 1 \leqslant \ell \leqslant 4,$  are linearly independent in $(\mathbb F_2\underset {\mathcal A}\otimes R_4)_{16}$.
\end{props}

\begin{proof} Observe that 
$$Sq^0_*([c_{3,1,i}]) = 0, Sq^0_*([\phi(b_{2,j})]) = [b_{2,j}], Sq^0_*([\phi^2(a_{1,\ell})]) = [\phi(a_{1,\ell})],$$ and $\{[b_{2,j}],  [\phi(a_{1,\ell})]; 1\leqslant j \leqslant 20,1 \leqslant \ell \leqslant 4\}$ is a basis of $(\mathbb F_2 \underset {\mathcal A} \otimes P_4)_6.$ Hence, it suffices to show that the elements $[c_{3,1,i}], 1 \leqslant i \leqslant 11,$ are linearly independent.

Suppose that there is a linear relation
\begin{equation}\sum_{1\leqslant i\leqslant 11}\gamma_i[c_{3,1,i}] =0, \tag {\ref{7.2.6}.1}
\end{equation}
with $\gamma_i \in \mathbb F_2, 1 \leqslant i \leqslant 11$. 
Applying the homomorphisms $f_1, f_2, f_3, f_4$ to the relation (\ref{7.2.6}.1), we obtain
\begin{align*}
&\gamma_{\{3, 4, 5\}}[1,1,14]  +  \gamma_{3}[1,14,1]  +    \gamma_{4}[1,2,13] \\
&\quad +   \gamma_{5}[1,3,12]  +   \gamma_{\{2, 3, 4, 5\}}[3,1,12]  =0,\\
&\gamma_{\{3, 4, 6, 8, 10, 11\}}[1,1,14]  +   \gamma_{3}[1,14,1]  +   \gamma_{4}[1,2,13]\\
&\quad  +   \gamma_{\{6, 8, 10, 11\}}[1,3,12] +   \gamma_{\{1, 3, 4, 6, 7, 8, 9, 10, 11\}}[3,1,12]  =0,\\   
&\gamma_{\{2, 5, 6, 8, 10, 11\}}[1,1,14]  +  \gamma_{2}[1,14,1]  +   \gamma_{5}[1,2,13]\\
&\quad  +   \gamma_{\{6, 8, 10, 11\}}[1,3,12] +   \gamma_{\{1, 2, 5, 6, 7, 8, 9, 10, 11\}}[3,1,12]  = 0,\\   
&\gamma_{\{7, 9, 10, 11\}}[1,1,14]  +  \gamma_{3}[1,14,1]  +   \gamma_{4}[1,2,13] \\
&\quad +   \gamma_{\{1, 2, 5, 6, 7, 8, 9, 10, 11\}}[1,3,12] +   \gamma_{\{7, 9, 10, 11\}}[3,1,12]  = 0. 
\end{align*}

Computing directly from the above equalities gives
\begin{equation}\begin{cases}
\gamma_{2} = \gamma_{3} = \gamma_{4} = \gamma_{5} = 0,\\
\gamma_{\{1, 10, 11\}} = 
\gamma_{\{1, 7, 9\}} =  
\gamma_{\{1, 6, 8\}} = 0.  
\end{cases}\tag{\ref{7.2.6}.2}
\end{equation}

Substituting (\ref{7.2.6}.2) into the relation (\ref{7.2.6}.1), we get
\begin{equation}
\gamma_{1}[\theta_1] + \gamma_{6}[\theta_2] + \gamma_{7}[\theta_3] + \gamma_{11}[\theta_4] = 0, \tag{\ref{7.2.6}.3}
\end{equation}
where $\theta_1 = b_{2,1,1} + b_{2,1,8} + b_{2,1,9} + b_{2,1,10}, \theta_2 = b_{2,1,6} + b_{2,1,8}, \theta_3 =  b_{2,1,7} + b_{2,1,9}, \theta_4 =  b_{2,1,10} + b_{2,1,11}.$ 

We prove $\gamma_{1}= \gamma_{6} = \gamma_{7} = \gamma_{11} = 0.$ The proof is divided into 4 steps.

{\it Step 1.} The homomorphism $\varphi_3$ sends (\ref{7.2.6}.3) to
\begin{equation}
\gamma_{1}([\theta_1]+ [\theta_4]) + \gamma_{6}[\theta_2] + \gamma_{7}[\theta_3] + \gamma_{11}[\theta_4] = 0. \tag{\ref{7.2.6}.4}
\end{equation}
Combining (\ref{7.2.6}.3) and (\ref{7.2.6}.4), we obtain
\begin{equation} \gamma_{1}[\theta_4] = 0. \tag{\ref{7.2.6}.5}
\end{equation}
Now we prove $[\theta_4] \ne 0.$ Suppose the contrary, that $[\theta_4] = 0$. Then, the polynomial $\theta_4$ is hit. Hence, we have
$$\theta_4 = Sq^1(A) + Sq^2(B) + Sq^4(C) + Sq^8 (D),$$
for some polynomials $A \in (R_4)_{15}, B \in (R_4)_{14}, C \in (R_4)_{12}, D \in (R_4)_8$. Let $(Sq^2)^3$ act on the both sides of this equality. Since $Sq^8(D) = D^2$, $(Sq^2)^3Sq^1 = 0$ and $(Sq^2)^3Sq^2 = 0$, we obtain
$$(Sq^2)^3(\theta_4) = (Sq^2)^3Sq^4(C).$$

On the other hand, by a direct calculation, we see that the monomial (6,6,6,4) is a term of $(Sq^2)^3(\theta_4)$. This monomial is not a term of $ (Sq^2)^3Sq^4(y)$ for all monomial $y \in (R_4)_{12}$. So, $(Sq^2)^3(\theta_4) \ne (Sq^2)^3Sq^4(C),$ for all $C \in (R_4)_{12}$ and we have a contradiction. Hence, $[\theta_4] \ne 0$ and $\gamma_{1} = 0.$

{\it Step 2.} Since $\gamma_{1} = 0$, the relation (\ref{7.2.6}.3) becomes
\begin{equation}
\gamma_{6}[\theta_2] + \gamma_{7}[\theta_3] + \gamma_{11}[\theta_4] = 0. \tag{\ref{7.2.6}.6}
\end{equation}
 The homomorphism $\varphi_2$ sends (\ref{7.2.6}.6) to
\begin{equation}
(\gamma_{7} + \gamma_{11})[\theta_1] + (\gamma_{6}+\gamma_{7} + \gamma_{11})[\theta_2] + \gamma_{7}[\theta_3] + \gamma_{11}[\theta_4] = 0. \tag{\ref{7.2.6}.7}
\end{equation}
Using the relation  (\ref{7.2.6}.7) and by an analogous argument as given in Step 1, we get
$$ \gamma_{7} = \gamma_{11}. $$

{\it Step 3.} The homomorphism $\varphi_1$ sends (\ref{7.2.6}.6) to
\begin{equation}
\gamma_{7}[\theta_2] + \gamma_{6}[\theta_3] + \gamma_{11}[\theta_4] = 0. \tag{\ref{7.2.6}.8}
\end{equation}
Using the relation (\ref{7.2.6}.8) and by a same argument as given in Step 2, we obtain
$$ \gamma_{6} = \gamma_{11}. $$

{\it Step 4.} The homomorphism $\varphi_4$ sends (\ref{7.2.6}.6) to
\begin{equation}
\gamma_{6}[\theta_2] + \gamma_{7}([\theta_2] + [\theta_3]) +  \gamma_{11}[\theta_4] = 0. \tag{\ref{7.2.6}.9}
\end{equation}
Combining (\ref{7.2.6}.8) and (\ref{7.2.6}.9), we obtain
\begin{equation}
\gamma_{7}[\theta_2] = 0. \tag{\ref{7.2.6}.10}
\end{equation}
Using the relation (\ref{7.2.6}.10) and by an analogous argument as given in Step 3, we get
$\gamma_{7} =0.$ Hence, the proposition is proved.
\end{proof}

\begin{props}\label{7.2.8} For $t \geqslant 4$, the elements $[c_{t,1,i}], 1 \leqslant i \leqslant 7,$ and $[\phi(b_{t-1,j})], 1 \leqslant j \leqslant \mu_6(t-1),$ and $[\phi^2(a_{t-2,\ell})], 1 \leqslant \ell \leqslant \mu_1(t-2),$  are linearly independent in $(\mathbb F_2\underset {\mathcal A}\otimes R_4)_{2^{t+1}}$.
\end{props}

\begin{proof} By an analogous argument as given in the proof of Proposition \ref{7.2.6}, it suffices to show that the elements $[c_{t,1,j}], 1 \leqslant j \leqslant 7$, are linearly independent.
Suppose that there is a linear relation
$$\sum_{i=1}^{7}\gamma_i[c_{t,1,i}]  = 0,
$$
with $\gamma_i \in \mathbb F_2$ for all $i$. 

Applying the homomorphisms $f_1, f_2, f_4$ to this relation, we obtain
\begin{align*}
&\gamma_{\{3, 4, 5\}}v_{t,1,1} +  \gamma_{3}v_{t,1,2} +   \gamma_{4}v_{t,1,12}  +  \gamma_{5}v_{t,1,13} +  \gamma_{\{2, 3, 4, 5\}}v_{t,1,14} = 0,\\
&\gamma_{\{3, 4, 6\}}v_{t,1,1} + \gamma_{3}v_{t,1,2} +  \gamma_{4}v_{t,1,12}+  \gamma_{6}v_{t,1,13} +  \gamma_{\{1, 3, 4, 6, 7\}}v_{t,1,14} =0,\\  
&\gamma_{7}v_{t,1,1} +  \gamma_{3}v_{t,1,2} +  \gamma_{4}v_{t,1,12} +  \gamma_{\{1, 2, 5, 6, 7\}}v_{t,1,13} +  \gamma_{7}v_{t,1,14} =0.                
\end{align*}

From these equalities, we obtain $\gamma_i = 0$ for $i=1, 2, \ldots , 7$. The proposition follows.
\end{proof}

\begin{rems}\label{7.2.7} The $\mathbb F_2$-subspace of $(\mathbb F_2\underset {\mathcal A}\otimes R_4)_{16}$ generated by $[\theta_1], [\theta_2]$, $[\theta_3]$, $[\theta_4]$, which are defined as in the proof of Proposition \ref{7.2.6}, is an $GL_4(\mathbb F_2)$-submodule of $(\mathbb F_2\underset {\mathcal A}\otimes R_4)_{16}$.
\end{rems}

\subsection{The case $s \geqslant 2$ and $t=1$}\label{7.3}\  

\medskip
According to Kameko \cite{ka}, for $s \geqslant 2, \dim (\mathbb F_2\underset{\mathcal A}\otimes P_3)_{2^{s+1} + 2^s -2} = 14$ with a basis given by the following classes:

\medskip
\centerline{\begin{tabular}{ll}
$v_{1,s,1} = [1,2^s - 2,2^{s+1} - 1],$& $v_{1,s,2} = [1,2^{s+1} - 1,2^s - 2],$\cr 
$v_{1,s,3} = [2^{s+1} - 1,1,2^s - 2],$& $v_{1,s,4} = [1,2^s - 1,2^{s+1} - 2],$\cr 
$v_{1,s,5} = [1,2^{s+1} - 2,2^s - 1],$& $v_{1,s,6} = [2^s - 1,1,2^{s+1} - 2],$\cr 
$v_{1,s,7} = [0,2^s - 1,2^{s+1} - 1],$& $v_{1,s,8} = [0,2^{s+1} - 1,2^s - 1],$\cr 
$v_{1,s,9} = [2^s - 1,0,2^{s+1} - 1],$& $v_{1,s,10} = [2^s - 1,2^{s+1} - 1,0],$\cr 
$v_{1,s,11} = [2^{s+1} - 1,0,2^s - 1],$& $v_{1,s,12} = [2^{s+1} - 1,2^s - 1,0],$\cr 
$v_{1,s,13} = [3,2^{s+1} - 3,2^s - 2].$&\cr
\end{tabular}}

\medskip
For $s = 2$,\  $v_{1,2,14} = [3,3,4].$

For $s \geqslant 3$,\  $v_{1,s,14} = [3,2^s - 3,2^{s+1} - 2].$

\medskip
So, we easily obtain

\begin{props}\label{7.3.4} $(\mathbb F_2\underset{\mathcal A}\otimes Q_4)_{2^{s+1} + 2^s -2}$ is  an $\mathbb F_2$-vector space of dimension 44 with a basis consisting of all the classes represented by the following monomials:

For $s \geqslant 2$,

\medskip
\centerline{\begin{tabular}{ll}
$b_{1,s,1} = (0,1,2^s - 2,2^{s+1} - 1),$& $b_{1,s,2} = (0,1,2^{s+1} - 1,2^s - 2),$\cr 
$b_{1,s,3} = (0,2^{s+1} - 1,1,2^s - 2),$& $b_{1,s,4} = (1,0,2^s - 2,2^{s+1} - 1),$\cr 
$b_{1,s,5} = (1,0,2^{s+1} - 1,2^s - 2),$& $b_{1,s,6} = (1,2^s - 2,0,2^{s+1} - 1),$\cr 
$b_{1,s,7} = (1,2^s - 2,2^{s+1} - 1,0),$& $b_{1,s,8} = (1,2^{s+1} - 1,0,2^s - 2),$\cr 
$b_{1,s,9} = (1,2^{s+1} - 1,2^s - 2,0),$& $b_{1,s,10} = (2^{s+1} - 1,0,1,2^s - 2),$\cr 
$b_{1,s,11} = (2^{s+1} - 1,1,0,2^s - 2),$& $b_{1,s,12} = (2^{s+1} - 1,1,2^s - 2,0),$\cr 
$b_{1,s,13} = (0,1,2^s - 1,2^{s+1} - 2),$& $b_{1,s,14} = (0,1,2^{s+1} - 2,2^s - 1),$\cr 
$b_{1,s,15} = (0,2^s - 1,1,2^{s+1} - 2),$& $b_{1,s,16} = (1,0,2^s - 1,2^{s+1} - 2),$\cr 
$b_{1,s,17} = (1,0,2^{s+1} - 2,2^s - 1),$& $b_{1,s,18} = (1,2^s - 1,0,2^{s+1} - 2),$\cr 
$b_{1,s,19} = (1,2^s - 1,2^{s+1} - 2,0),$& $b_{1,s,20} = (1,2^{s+1} - 2,0,2^s - 1),$\cr 
$b_{1,s,21} = (1,2^{s+1} - 2,2^s - 1,0),$& $b_{1,s,22} = (2^s - 1,0,1,2^{s+1} - 2),$\cr 
$b_{1,s,23} = (2^s - 1,1,0,2^{s+1} - 2),$& $b_{1,s,24} = (2^s - 1,1,2^{s+1} - 2,0),$\cr 
$b_{1,s,25} = (0,0,2^s - 1,2^{s+1} - 1),$& $b_{1,s,26} = (0,0,2^{s+1} - 1,2^s - 1),$\cr 
$b_{1,s,27} = (0,2^s - 1,0,2^{s+1} - 1),$& $b_{1,s,28} = (0,2^s - 1,2^{s+1} - 1,0),$\cr 
$b_{1,s,29} = (0,2^{s+1} - 1,0,2^s - 1),$& $b_{1,s,30} = (0,2^{s+1} - 1,2^s - 1,0),$\cr 
$b_{1,s,31} = (2^s - 1,0,0,2^{s+1} - 1),$& $b_{1,s,32} = (2^s - 1,0,2^{s+1} - 1,0),$\cr 
$b_{1,s,33} = (2^s - 1,2^{s+1} - 1,0,0),$& $b_{1,s,34} = (2^{s+1} - 1,0,0,2^s - 1),$\cr 
$b_{1,s,35} = (2^{s+1} - 1,0,2^s - 1,0),$& $b_{1,s,36} = (2^{s+1} - 1,2^s - 1,0,0),$\cr 
$b_{1,s,37} = (0,3,2^{s+1} - 3,2^s - 2),$& $b_{1,s,38} = (3,0,2^{s+1} - 3,2^s - 2),$\cr 
$b_{1,s,39} = (3,2^{s+1} - 3,0,2^s - 2),$& $b_{1,s,40} = (3,2^{s+1} - 3,2^s - 2,0),$\cr 
\end{tabular}}

\medskip
For $s=2$,

\medskip
\centerline{\begin{tabular}{ll}
$b_{1,2,41} = (0,3,3,4)$,& $b_{1,2,42} = (3,0,3,4)$,\cr 
$b_{1,2,43} = (3,3,0,4)$,& $b_{1,2,44} = (3,3,4,0)$. \cr 
\end{tabular}}

\medskip
For $s \geqslant 3$,

\medskip
\centerline{\begin{tabular}{ll}
$b_{1,s,41} = (0,3,2^s - 3,2^{s+1} - 2),$& $b_{1,s,42} = (3,0,2^s - 3,2^{s+1} - 2),$\cr 
$b_{1,s,43} = (3,2^s - 3,0,2^{s+1} - 2),$& $b_{1,s,44} = (3,2^s - 3,2^{s+1} - 2,0).$\cr
\end{tabular}}
\end{props}

Now, we determine $(\mathbb F_2\underset{\mathcal A}\otimes R_4)_{2^{s+1} + 2^s -2}$. We have

\begin{thms}\label{dlc7.3} $(\mathbb F_2\underset{\mathcal A}\otimes R_4)_{2^{s+1} + 2^s -2}$ is  an $\mathbb F_2$-vector space with a basis consisting of all the classes represented by the following monomials:

1. $\phi (c_{1,j}), 1 \leqslant j \leqslant 14$ for $s = 2$.

2. $\phi (a_{1,s-2,j}), 1 \leqslant j \leqslant \mu_2(s-2)$ for $ s\geqslant 3$. Here $c_{1,j}$, $a_{1,s-2,j}$, $\mu_2(s-2)$ are defined as in Section \ref{5} and Section \ref{6},  the homomorphism $\phi$ is defined as in Definition \ref{2.8}.

3.  For $s \geqslant 2$,

\smallskip
\centerline{\begin{tabular}{ll}
$b_{1,s,45} = (1,1,2^s - 2,2^{s+1} - 2),$& $b_{1,s,46} = (1,1,2^{s+1} - 2,2^s - 2),$\cr 
$b_{1,s,47} = (1,2^s - 2,1,2^{s+1} - 2),$& $b_{1,s,48} = (1,2^{s+1} - 2,1,2^s - 2),$\cr 
$b_{1,s,49} = (1,2,2^{s+1} - 3,2^s - 2),$& $b_{1,s,50} = (1,2,2^s - 1,2^{s+1} - 4),$\cr 
$b_{1,s,51} = (1,2,2^{s+1} - 4,2^s - 1),$& $b_{1,s,52} = (1,2^s - 1,2,2^{s+1} - 4),$\cr 
$b_{1,s,53} = (2^s - 1,1,2,2^{s+1} - 4),$& $b_{1,s,54} = (1,3,2^{s+1} - 4,2^s - 2),$\cr 
$b_{1,s,55} = (3,1,2^{s+1} - 4,2^s - 2).$&\cr
\end{tabular}}

\smallskip
For $s=2$, \  $b_{1,2,56} = (3,4,1,2).$

For $s \geqslant 3$,

\medskip
\centerline{\begin{tabular}{ll}
$b_{1,s,56} = (1,3,2^s - 2,2^{s+1} - 4),$& $b_{1,s,57} = (3,1,2^s - 2,2^{s+1} - 4),$\cr 
$b_{1,s,58} = (1,2,2^s - 4,2^{s+1} - 1),$& $b_{1,s,59} = (1,2,2^{s+1} - 1,2^s - 4),$\cr 
$b_{1,s,60} = (1,2^{s+1} - 1,2,2^s - 4),$& $b_{1,s,61} = (2^{s+1} - 1,1,2,2^s - 4),$\cr 
$b_{1,s,62} = (1,2,2^s - 3,2^{s+1} - 2),$& $b_{1,s,63} = (1,3,2^s - 4,2^{s+1} - 2),$\cr 
$b_{1,s,64} = (1,3,2^{s+1} - 2,2^s - 4),$& $b_{1,s,65} = (3,1,2^s - 4,2^{s+1} - 2),$\cr 
$b_{1,s,66} = (3,1,2^{s+1} - 2,2^s - 4),$& $b_{1,s,67} = (3,2^{s+1} - 3,2,2^s - 4),$\cr 
$b_{1,s,68} = (3,2^s - 3,2,2^{s+1} - 4),$& $b_{1,s,69} = (3,5,2^{s+1} - 6,2^s - 4).$\cr 
\end{tabular}}

\medskip
For $s=3$, \ $b_{1,3,70} = (3,5,6,8).$

For $s \geqslant 4$, \ $b_{1,s,70} = (3,5,2^s - 6,2^{s+1} - 4).$
\end{thms}

To prove this theorem, we prove the following propositions.

\begin{props}\label{mdc7.3} The $\mathbb F_2$-vector space $(\mathbb F_2\underset {\mathcal A}\otimes R_4)_{2^{s+1}+2^s-2}$ is generated by the  elements listed in Theorem \ref{dlc7.3}.
\end{props}

The proof of the proposition is based on some lemmas.

\begin{lems}\label{7.3.1} The following matrices are strictly inadmissible
 $$ \begin{pmatrix} 0&1&0&1\\ 0&1&0&1\\ 0&0&1&0\end{pmatrix} \quad  \begin{pmatrix} 1&0&0&1\\ 1&0&0&1\\ 0&0&1&0\end{pmatrix} \quad  \begin{pmatrix} 1&0&0&1\\ 1&0&0&1\\ 0&1&0&0\end{pmatrix} \quad  \begin{pmatrix} 1&0&1&0\\ 1&0&1&0\\ 0&1&0&0\end{pmatrix}. $$
\end{lems}

\begin{proof} See \cite{ka}.
\end{proof}

\begin{lems}\label{7.3.2} The following matrices are strictly inadmissible
 $$ \begin{pmatrix} 0&1&1&0\\ 0&1&1&0\\ 0&1&1&0\\ 0&0&0&1\end{pmatrix} \quad  \begin{pmatrix} 1&0&1&0\\ 1&0&1&0\\ 1&0&1&0\\ 0&0&0&1\end{pmatrix} \quad  \begin{pmatrix} 1&1&0&0\\ 1&1&0&0\\ 1&1&0&0\\ 0&0&0&1\end{pmatrix} \quad  \begin{pmatrix} 1&1&0&0\\ 1&1&0&0\\ 1&1&0&0\\ 0&0&1&0\end{pmatrix} .$$
\end{lems}

\begin{proof} See \cite{ka}.
\end{proof}

\begin{lems}\label{7.3.3} The following matrices are strictly inadmissible
 $$ \begin{pmatrix} 1&1&0&0\\ 0&1&1&0\\ 0&1&0&1\\ 0&0&1&0\end{pmatrix} \quad  \begin{pmatrix} 1&1&0&0\\ 1&0&1&0\\ 1&0&0&1\\ 0&0&1&0\end{pmatrix} \quad  \begin{pmatrix} 1&0&1&0\\ 0&1&1&0\\ 0&1&1&0\\ 0&0&0&1\end{pmatrix} \quad  \begin{pmatrix} 1&1&0&0\\ 0&1&1&0\\ 0&1&1&0\\ 0&0&0&1\end{pmatrix} $$    
$$ \begin{pmatrix} 1&1&0&0\\ 1&0&1&0\\ 1&0&1&0\\ 0&0&0&1\end{pmatrix} \quad  \begin{pmatrix} 1&1&0&0\\ 1&1&0&0\\ 0&0&1&1\\ 0&0&0&1\end{pmatrix} \quad  \begin{pmatrix} 1&1&0&0\\ 1&1&0&0\\ 0&0&1&1\\ 0&0&1&0\end{pmatrix} \quad  \begin{pmatrix} 1&1&0&0\\ 1&0&1&0\\ 1&0&0&1\\ 0&1&0&0\end{pmatrix} $$    
$$ \begin{pmatrix} 1&0&1&0\\ 1&0&1&0\\ 1&0&0&1\\ 0&1&0&0\end{pmatrix} \quad  \begin{pmatrix} 1&1&0&0\\ 1&0&0&1\\ 0&1&0&1\\ 0&0&1&0\end{pmatrix} .$$ 
\end{lems}

\begin{proof} The monomials corresponding to these matrices respectively are
(1,7,10,4),  (7,1,10,4),   (1,6,7,8),  (1,7,6,8),  (7,1,6,8),  (3,3,4,12),  (3,3,12,4),  (7,9,2,4),  (7,8,3,4),  (3,5,8,6). We prove the lemma for the matrices corresponding to the monomials  (7,1,10,4),   (1,6,7,8),   (3,3,4,12),  (7,9,2,4),  (7,8,3,4),  (3,5,8,6). The others can be obtained by a same computation. We have
\begin{align*}
(7,1,10,4) &= Sq^1\big((7,1,9,4)+(7,4,7,3)+(7,1,8,5)\big)\\
&\quad +Sq^2\big((7,2,7,4) + (7,2,8,3)\big) + Sq^4\big((5,2,7,4)\\
&\quad +(4,4,7,3)\big)+ (5,2,11,4) + (5,2,7,8) + (4,4,11,3) \\
&\quad + (4,8,7,3) + (7,1,8,6) \  \text{mod }\mathcal L_4(2;2;2;1),\\
(1,7,6,8) &= Sq^1\big((1,7,5,8)+(4,7,3,7)+(1,8,5,7)+(1,9,4,7)\\
&\quad +(1,7,4,9)\big) + Sq^2\big((2,7,3,8)+(2,8,3,7)+(2,7,4,7)\big)\\
&\quad +Sq^4\big((1,6,4,7)+(1,4,6,7)\big) + (1,4,10,7) +  (1,4,6,11)\\
&\quad +(1,6,8,7) + (1,6,4,11) + (1,7,4,10) \ \text{mod }\mathcal L_4(2;2;2;1),\\
(3,3,4,12)&= Sq^1\big((3,3,4,11) + (3,8,3,7) + (3,3,8,7)\big) +Sq^2\big((2,3,4,11)\\
&\quad+(2,8,3,7)+(2,3,8,7)\big)+Sq^4(3,4,4,7) +(2,5,4,11)\\
&\quad+(2,3,4,13)+(2,8,5,7)+(2,5,8,7)\quad  \text{mod }\mathcal L_4(2;2;2;1),\\
(7,9,2,4) &= Sq^1(7,7,3,4) + Sq^2(7,7,2,4) + Sq^4\big((5,7,2,4)\\
&\quad  +(4,7,3,4)\big) +(5,11,2,4)+(5,7,2,8)+(4,11,3,4)\\
&\quad+(4,7,3,8) + (7,8,3,4) \quad  \text{mod   }\mathcal L_4(2;2;2;1),\\
(7,8,3,4) &=  Sq^1(7,8,3,3) + Sq^2(7,8,2,3)\\
&\quad + (7,8,2,5) \quad  \text{mod }\mathcal L_4(2;2;2;1),\\
(3,5,8,6)&= Sq^1\big((3,3,9,6) + (3,3,6,9)\big)+ Sq^2\big((5,3,6,6) + (2,3,6,9)\\
&\quad + (2,3,9,6)\big) + Sq^4(3,3,6,6)+ (3,5,6,8) + (3,4,6,9) \\
&\quad + (3,4,9,6) + (2,5,6,9) + (2,5,9,6)\  \text{mod }\mathcal L_4(2;2;2;1).
\end{align*}
The lemma is proved.
\end{proof}

 Using the results in Section \ref{4}, Section \ref{5}, Lemma \ref{7.1.3}, \ref{7.1.4}, \ref{7.1.5}, \ref{7.2.1}, \ref{7.2.2}, \ref{7.2.3}, \ref{7.3.1}, \ref{7.3.2} and Theorem \ref{2.4},  we get Proposition \ref{mdc7.3}.

\medskip
Now, we show that the elements listed in Theorem \ref{dlc7.3} are linearly independent.

\begin{props}\label{7.3.5} The elements $[b_{1,2,i}], 45 \leqslant i \leqslant 56,$ and $[\phi(c_{1,j})], 1 \leqslant j \leqslant 14,$  are linearly independent in $(\mathbb F_2\underset {\mathcal A}\otimes R_4)_{10}$.
\end{props}

\begin{proof} Suppose that there is a linear relation
\begin{equation}\sum_{i=45}^{56}\gamma_i[b_{1,2,i}] +\sum_{j=1}^{14}\eta_j[\phi(c_{1,j})] = 0, \tag {\ref{7.3.5}.1}
\end{equation}
with $\gamma_i, \eta_j \in \mathbb F_2$. 

Applying the squaring operation $Sq^0_*$ to (\ref{7.3.5}.1), we obtain
$$\sum_{j=1}^{14}\eta_j[c_{1,j}]=0.$$
Since $\{[c_{1,j}]; 1\leqslant j \leqslant 14\}$ is a basis of $(\mathbb F_2 \underset {\mathcal A} \otimes P_4)_3,$ we obtain $\eta_j=0, 1\leqslant j \leqslant 14$. Hence, (\ref{7.3.5}.1) becomes
\begin{equation}\sum_{45\leqslant i\leqslant 56}\gamma_i[b_{1,2,i}] =0. \tag {\ref{7.3.5}.2}
\end{equation}
Now, we prove $\gamma_i=0, 45\leqslant i\leqslant 56.$ 

Apply the homomorphisms $f_1, f_2, f_3$ to the relation (\ref{7.3.5}.2) and we get
\begin{align*}
&\gamma_{\{48, 56\}}[7,1,2] +  \gamma_{51}[1,3,6] +   \gamma_{51}[1,6,3]\\
&\quad +  \gamma_{47}[3,1,6] +  \gamma_{49}[3,5,2] +  \gamma_{\{50, 51\}}[3,3,4] =0,\\
&\gamma_{\{46, 55\}}[7,1,2] + \gamma_{51}[1,3,6] +  \gamma_{51}[1,6,3]\\
&\quad +  \gamma_{\{45, 51\}}[3,1,6]  +  \gamma_{54}[3,5,2] +  \gamma_{\{51, 52, 54\}}[3,3,4] =0,\\  
&\gamma_{\{45, 47, 53\}}[7,1,2] + \gamma_{50}[1,3,6] +  \gamma_{\{46, 49, 50\}}[3,1,6] \\
&\quad +  \gamma_{50}[1,6,3]+  \gamma_{\{48, 52\}}[3,5,2] +  \gamma_{\{50, 52, 54\}}[3,3,4] = 0.  
\end{align*}

From these equalities, we get
\begin{equation}\begin{cases}
\gamma_{47} = \gamma_{49} = \gamma_{50} = \gamma_{51} = \gamma_{54} =  0,\\
\gamma_{\{48, 56\}} =   
\gamma_{\{50, 51\}} =   
\gamma_{\{46, 55\}} =   
\gamma_{\{45, 51\}} =   
\gamma_{\{51, 52, 54\}} = 0,\\  
\gamma_{\{45, 47, 53\}} =  
\gamma_{\{46, 49, 50\}} =   
\gamma_{\{48, 52\}} =   
\gamma_{\{50, 52, 54\}} = 0.  
\end{cases}\tag{\ref{7.3.5}.3}
\end{equation}

The relation (\ref{7.3.5}.3) implies $\gamma_j = 0$ for $ j= 45, \ldots, 56$. So, the proposition is proved.
\end{proof}

\begin{props}\label{7.3.6} The elements $[b_{1,3,i}], 45 \leqslant i \leqslant 70,$ and $[\phi(a_{1,1,j})], 1 \leqslant j \leqslant 46,$  are linearly independent in $(\mathbb F_2\underset {\mathcal A}\otimes R_4)_{22}$.
\end{props}

\begin{proof} Suppose that there is a linear relation
\begin{equation}\sum_{i=45}^{70}\gamma_i[b_{1,3,i}] +\sum_{j=1}^{46} \eta_j [\phi(a_{1,1,j})] = 0, \tag {\ref{7.3.6}.1}
\end{equation}
with $\gamma_i, \eta_j \in \mathbb F_2$. 

Applying the squaring operation $Sq^0_*$ to (\ref{7.3.6}.1), we get
$$\sum_{j=1}^{46}\eta_j[a_{1,1,j}]=0.$$
Since $\{[a_{1,1,j}]; 1\leqslant j \leqslant 46\}$ is a basis of $(\mathbb F_2 \underset {\mathcal A} \otimes P_4)_9,$ we obtain $\eta_j=0, 1\leqslant j \leqslant 46$. Hence, (\ref{7.3.6}.1) becomes
\begin{equation}\sum_{45\leqslant i\leqslant 70}\gamma_i[b_{1,3,i}] =0. \tag {\ref{7.3.6}.2}
\end{equation}
Now, we prove $\gamma_i=0, 45\leqslant i\leqslant 70.$ 

Applying the homomorphisms $f_t, t= 1,\ldots , 6,$ to the relation (\ref{7.3.6}.2), we obtain
\begin{align*}
&\gamma_{58}[1,6,15] +  \gamma_{59}[1,15,6] +   \gamma_{48}[15,1,6] +  \gamma_{50}[1,7,14]\\
&\quad +  \gamma_{51}[1,14,7] +  \gamma_{47}[7,1,14] +  \gamma_{62}[3,5,14] +  \gamma_{49}[3,13,6] = 0,\\  
&\gamma_{58}[1,6,15] +  \gamma_{60}[1,15,6] +  \gamma_{\{46, 55, 64\}}[15,1,6]  +  \gamma_{\{45, 56, 65\}}[7,1,14]\\
&\quad +  \gamma_{52}[1,7,14] +  \gamma_{51}[1,14,7] +  \gamma_{\{51, 54, 63\}}[3,5,14] +  \gamma_{54}[3,13,6] = 0,\\  
&\gamma_{59}[1,6,15] +  \gamma_{60}[1,15,6] +  \gamma_{\{45, 47, 62, 57, 63, 68\}}[15,1,6]\\
&\quad +  \gamma_{52}[1,7,14] +  \gamma_{50}[1,14,7] +  \gamma_{\{46, 48, 49 54, 66, 67, 69\}}[7,1,14]\\
&\quad +  \gamma_{\{48, 50 56, 64, 67, 70\}}[3,5,14] +  \gamma_{\{48, 52, 56, 67, 70\}}[3,13,6] = 0,\\  
&\gamma_{58}[1,6,15] + \gamma_{\{46, 48, 49, 54, 66, 67, 69\}}[1,15,6] +  \gamma_{61}[15,1,6] +  \gamma_{51}[1,14,7]\\
&\ +  \gamma_{\{45, 47, 62, 57, 63, 68\}}[1,7,14] +  \gamma_{53}[7,1,14] + \gamma_{65}[3,5,14] +  \gamma_{55}[3,13,6] = 0,\\   
&\gamma_{59}[1,6,15] + \gamma_{\{45, 56, 65\}}[1,15,6] +  \gamma_{61}[15,1,6] +  \gamma_{\{46, 55, 64\}}[1,7,14]\\
&\quad +  \gamma_{50}[1,14,7] +  \gamma_{53}[7,1,14] +  \gamma_{\{53, 66\}}[3,5,14] +  \gamma_{\{53, 57, 70\}}[3,13,6] = 0,\\  
&\gamma_{47}[1,6,15] +  \gamma_{60}[1,15,6] +  \gamma_{61}[15,1,6] +  \gamma_{52}[1,7,14]\\
&\quad +  \gamma_{48}[1,14,7] +  \gamma_{53}[7,1,14] +  \gamma_{\{68, 69, 70\}}[3,5,14] +  \gamma_{67}[3,13,6] = 0. 
\end{align*}
From the above equalities, we get
\begin{equation}\begin{cases}
\gamma_j = 0,\ j = 47, \ldots , 63, 65, 66, 54, 55, 67,\\
\gamma_{\{46, 64\}} =   
\gamma_{\{45, 56\}} =    
\gamma_{\{45, 57, 68\}} =   
\gamma_{\{46, 69\}} = 0,\\  
\gamma_{\{56, 64, 70\}} =   
\gamma_{\{56, 70\}} =   
\gamma_{\{57, 70\}} =   
\gamma_{\{68, 69, 70\}} = 0. 
\end{cases}\tag{\ref{7.3.6}.3}
\end{equation}
Computing directly from (\ref{7.3.6}.3), we obtain $\gamma_i = 0$ for any $i$. The proposition is proved.
\end{proof}

\begin{props}\label{7.3.7} For $s\geqslant 4$, the elements $[b_{1,s,i}], 45 \leqslant i \leqslant 70,$ and $[\phi(a_{1,s-2,j})], 1 \leqslant j \leqslant \mu_2(s-2),$  are linearly independent in $(\mathbb F_2\underset {\mathcal A}\otimes R_4)_{3.2^s -2}$.
\end{props}

\begin{proof} Suppose that there is a linear relation
\begin{equation}\sum_{i=45}^{70}\gamma_i[b_{1,s-2,i}] +\sum_{j=1}^{\mu_2(s-2)} \eta_j [\phi(a_{1,s-2,j})] = 0, \tag {\ref{7.3.7}.1}
\end{equation}
with $\gamma_i, \eta_j \in \mathbb F_2$. 

Applying the squaring operation $Sq^0_*$ to (\ref{7.3.7}.1), we get
$$\sum_{j=1}^{\mu_2(s-2)}\eta_j[a_{1,s-2,j}]=0.$$
Since $\{[a_{1,s-2,j}]; 1\leqslant j \leqslant \mu_2(s-2)\}$ is a basis of $(\mathbb F_2 \underset {\mathcal A} \otimes P_4)_{2^s + 2^{s-1} - 3},$ we obtain $\eta_j=0, 1\leqslant j \leqslant \mu_2(s-2)$. Hence, (\ref{7.3.7}.1) becomes
\begin{equation}\sum_{45\leqslant i\leqslant 70}\gamma_i[b_{1,s,i}] =0. \tag {\ref{7.3.7}.2}
\end{equation}
We need to prove $\gamma_i=0, 45\leqslant i\leqslant 70.$ 

Apply the homomorphisms $f_t, t= 1,\ldots , 6,$ to the relation (\ref{7.3.7}.2) and we obtain
\begin{align*}
&\gamma_{58}v_{1,s,1} +  \gamma_{59}v_{1,s,2} +   \gamma_{48}v_{1,s,3} +  \gamma_{50}v_{1,s,4}\\ 
&\quad +  \gamma_{51}v_{1,s,5} +  \gamma_{47}v_{1,s,6} +  \gamma_{49}v_{1,s,13} +  \gamma_{62}v_{1,s,14} =0,\\  
&\gamma_{58}v_{1,s,1} +  \gamma_{60}v_{1,s,2} +  \gamma_{\{46, 64, 66\}}v_{1,s,3} +  \gamma_{52}v_{1,s,4}  +  \gamma_{51}v_{1,s,5}\\
&\quad +  \gamma_{\{45, 56, 65\}}v_{1,s,6} +  \gamma_{54}v_{1,s,13} +  \gamma_{\{51, 54, 63\}}v_{1,s,14} =0,\\  
&\gamma_{59}v_{1,s,1, 60}v_{1,s,2} +  \gamma_{\{45, 47, 57, 62, 63, 68, 70\}}v_{1,s,3} +  \gamma_{52}v_{1,s,4} +  \gamma_{50}v_{1,s,5}\\
&\quad  +  a_2v_{1,s,6} +  \gamma_{\{52, 56\}}v_{1,s,13} +  \gamma_{\{50, 56, 64\}}v_{1,s,14} = 0,\\  
&\gamma_{58}v_{1,s,1} +  \gamma_{\{46, 48, 49, 54, 66, 67, 69\}}v_{1,s,2} +  \gamma_{61}v_{1,s,3} +  a_1v_{1,s,4} \\
&\quad +  \gamma_{51}v_{1,s,5} +  \gamma_{55}v_{1,s,13} +  \gamma_{53}v_{1,s,6, 65}v_{1,s,14} = 0,\\   
&\gamma_{59}v_{1,s,1} + \gamma_{\{45, 56, 65\}}v_{1,s,2} +  \gamma_{46, 55, 64}v_{1,s,4}+  \gamma_{61}v_{1,s,3}\\
&\quad   +  \gamma_{50}v_{1,s,5} +  \gamma_{53}v_{1,s,6} +  \gamma_{57}v_{1,s,13} +  \gamma_{66}v_{1,s,14} = 0,\\  
&\gamma_{47}v_{1,s,1} +  \gamma_{60}v_{1,s,2} +  \gamma_{61}v_{1,s,3} +  \gamma_{52}v_{1,s,4}\\
&\quad  +  \gamma_{48}v_{1,s,5} +  \gamma_{53}v_{1,s,6} +  \gamma_{67}v_{1,s,13} +  \gamma_{68}v_{1,s,14} = 0.
\end{align*}

From these equalities, we get
\begin{equation}\begin{cases}
\gamma_j=0,\ j= 47, \ldots, 62, 65, 66, 54,\ldots , 68,\\
\gamma_{\{46, 64, 55\}} = 
\gamma_{\{45, 65, 56\}} = 
\gamma_{\{51, 63, 54\}} = 0,\\  
\gamma_{\{50, 64, 56\}}=  
\gamma_{\{52, 56\}} =0.   
\end{cases}\tag{\ref{7.3.7}.3}
\end{equation}

Computing directly from (\ref{7.3.7}.3), we obtain $\gamma_i = 0$ for any $i$. The proposition is proved.
\end{proof}

\subsection{The case $s \geqslant 2$ and $t\geqslant 2$}\label{7.4}\ 

\medskip
Using the result in Kameko \cite{ka}, we see that for $t \geqslant 2$ and $s \geqslant 2$, $\dim (\mathbb F_2\underset{\mathcal A}\otimes P_3)_{2^{t + s} + 2^s -2} = 21$ with a basis given by the following classes:

\medskip
\centerline{\begin{tabular}{ll}
$v_{t,s,1} = [1,2^s - 2,2^{s+t} - 1],$& $v_{t,s,2} = [1,2^{s+t} - 1,2^s - 2],$\cr 
$v_{t,s,3} = [2^{s+t}-1,1,2^s - 2],$& $v_{t,s,4} = [1,2^s - 1,2^{s+t} - 2],$\cr 
$v_{t,s,5} = [1,2^{s+t} - 2,2^s - 1],$& $v_{t,s,6} = [2^s-1,1,2^{s+t} - 2],$\cr 
$v_{t,s,7} = [1,2^{s+1} - 2,2^{s+t} - 2^s - 1],$& $v_{t,s,8} = [1,2^{s+1} - 1,2^{s+t} - 2^s - 2],$\cr 
$v_{t,s,9} = [2^{s+1}-1,1,2^{s+t} - 2^s - 2],$& $v_{t,s,10} = [3,2^{s+1} - 3,2^{s+t} - 2^s - 2],$\cr 
$v_{t,s,11} = [3,2^{s+t} - 3,2^s - 2],$& $v_{t,s,12} = [0,2^s - 1,2^{s+t} - 1],$\cr
 $v_{t,s,13} = [0,2^{s+1} - 1,2^{s+t} - 2^s - 1],$& $v_{t,s,14} = [0,2^{s+t} - 1,2^s - 1],$\cr 
$v_{t,s,15} = [2^s-1,0,2^{s+t} - 1],$& $v_{t,s,16} = [2^s-1,2^{s+t} - 1,0],$\cr
$v_{t,s,17} = [2^{s+1}-1,0,2^{s+t} - 2^s - 1],$& $v_{t,s,18} = [2^{s+1}-1,2^{s+t} - 2^s - 1,0],$\cr 
$v_{t,s,19} = [2^{s+t}-1,0,2^s - 1],$& $v_{t,s,20} = [2^{s+t}-1,2^s - 1,0],$ \cr
$v_{t,2,21} = [3,3,2^{t+2} - 4]$, $s=2$,& $v_{t,s,21} = [3,2^s - 3,2^{s+t} - 2],$ $s \geqslant 3$. \cr
\end{tabular}}

\medskip
Hence, we easily obtain

\begin{props}\label{7.4.5} $(\mathbb F_2\underset{\mathcal A}\otimes Q_4)_{2^{s+t} + 2^s -2}$ is  an $\mathbb F_2$-vector space of dimension 66 with a basis consisting of all the classes represented by the  monomials $b_{t,s,j}, j\geqslant 1,$ which are determined as follows:

\smallskip
For $s \geqslant 2$,

\medskip
\centerline{\begin{tabular}{ll}
$1.\  (0,1,2^s - 2,2^{s+t} - 1),$& $2.\  (0,1,2^{s+t} - 1,2^s - 2),$\cr 
$3.\  (0,2^{s+t} - 1,1,2^s - 2),$& $4.\  (1,0,2^s - 2,2^{s+t} - 1),$\cr 
$5.\  (1,0,2^{s+t} - 1,2^s - 2),$& $6.\  (1,2^s - 2,0,2^{s+t} - 1),$\cr 
$7.\  (1,2^s - 2,2^{s+t} - 1,0),$& $8.\  (1,2^{s+t} - 1,0,2^s - 2),$\cr 
$9.\  (1,2^{s+t} - 1,2^s - 2,0),$& $10.\  (2^{s+t}-1,0,1,2^s - 2),$\cr 
$11.\  (2^{s+t}-1,1,0,2^s - 2),$& $12.\  (2^{s+t}-1,1,2^s - 2,0),$\cr 
$13.\  (0,1,2^s - 1,2^{s+t} - 2),$& $14.\  (0,1,2^{s+t} - 2,2^s - 1),$\cr 
$15.\  (0,2^s - 1,1,2^{s+t} - 2),$& $16.\  (1,0,2^s - 1,2^{s+t} - 2),$\cr 
$17.\  (1,0,2^{s+t} - 2,2^s - 1),$& $18.\  (1,2^s - 1,0,2^{s+t} - 2),$\cr 
$19.\  (1,2^s - 1,2^{s+t} - 2,0),$& $20.\  (1,2^{s+t} - 2,0,2^s - 1),$\cr 
$21.\  (1,2^{s+t} - 2,2^s - 1,0),$& $22.\  (2^s-1,0,1,2^{s+t} - 2),$\cr 
$23.\  (2^s-1,1,0,2^{s+t} - 2),$& $24.\  (2^s-1,1,2^{s+t} - 2,0),$\cr 
$25.\  (0,0,2^s - 1,2^{s+t} - 1),$& $26.\  (0,0,2^{s+t} - 1,2^s - 1),$\cr 
$27.\  (0,2^s - 1,0,2^{s+t} - 1),$& $28.\  (0,2^s - 1,2^{s+t} - 1,0),$\cr 
$29.\  (0,2^{s+t} - 1,0,2^s - 1),$& $30.\  (0,2^{s+t} - 1,2^s - 1,0),$\cr 
$31.\  (2^s-1,0,0,2^{s+t} - 1),$& $32.\  (2^s-1,0,2^{s+t} - 1,0),$\cr 
$33.\  (2^s-1,2^{s+t} - 1,0,0),$& $34.\  (2^{s+t}-1,0,0,2^s - 1),$\cr 
$35.\  (2^{s+t}-1,0,2^s - 1,0),$& $36.\  (2^{s+t}-1,2^s - 1,0,0),$\cr 
$37.\  (0,1,2^{s+1} - 2,2^{s+t} - 2^s - 1),$& $38.\  (1,0,2^{s+1} - 2,2^{s+t} - 2^s - 1),$\cr 
$39.\  (1,2^{s+1} - 2,0,2^{s+t} - 2^s - 1),$& $40.\  (1,2^{s+1} - 2,2^{s+t} - 2^s - 1,0),$\cr 
$41.\  (0,1,2^{s+1} - 1,2^{s+t} - 2^s - 2),$& $42.\  (0,2^{s+1} - 1,1,2^{s+t} - 2^s - 2),$\cr 
$43.\  (1,0,2^{s+1} - 1,2^{s+t} - 2^s - 2),$& $44.\  (1,2^{s+1} - 1,0,2^{s+t} - 2^s - 2),$\cr 
$45.\  (1,2^{s+1} - 1,2^{s+t} - 2^s - 2,0),$& $46.\  (2^{s+1}-1,0,1,2^{s+t} - 2^s - 2),$\cr 
$47.\  (2^{s+1}-1,1,0,2^{s+t} - 2^s - 2),$& $48.\  (2^{s+1}-1,1,2^{s+t} - 2^s - 2,0),$\cr 
$49.\  (0,0,2^{s+1} - 1,2^{s+t} - 2^s - 1),$& $50.\  (0,2^{s+1} - 1,0,2^{s+t} - 2^s - 1),$\cr 
$51.\  (0,2^{s+1} - 1,2^{s+t} - 2^s - 1,0),$& $52.\  (2^{s+1}-1,0,0,2^{s+t} - 2^s - 1),$\cr 
$53.\  (2^{s+1}-1,0,2^{s+t} - 2^s - 1,0),$& $54.\  (2^{s+1}-1,2^{s+t} - 2^s - 1,0,0),$\cr 
\end{tabular}}
\centerline{\begin{tabular}{ll}
$55.\  (0,3,2^{s+t} - 3,2^s - 2),$& $56.\  (3,0,2^{s+t} - 3,2^s - 2),$\cr 
$57.\  (3,2^{s+t} - 3,0,2^s - 2),$& $58.\  (3,2^{s+t} - 3,2^s - 2,0),$\cr 
$59.\  (0,3,2^{s+1} - 3,2^{s+t} - 2^s - 2),$& $60.\  (3,0,2^{s+1} - 3,2^{s+t} - 2^s - 2),$\cr 
$61.\  (3,2^{s+1} - 3,0,2^{s+t} - 2^s - 2),$& $62.\  (3,2^{s+1} - 3,2^{s+t} - 2^s - 2,0).$\cr
\end{tabular}}

\medskip
For $s=2$,

\medskip
\centerline{\begin{tabular}{ll}
$63.\ (0,3,3,2^{t+2}-4),$ & $64.\ (3,0,3,2^{t+2}-4),$ \cr
$65.\ (3,3,0,2^{t+2}-4),$ & $66.\ (3,3,2^{t+2}-4,0).$ \cr
\end{tabular}}

\medskip
For $s \geqslant 3$,

\medskip
\centerline{\begin{tabular}{ll}
$63.\  (0,3,2^s - 3,2^{s+t} - 2),$& $64.\  (3,0,2^s - 3,2^{s+t} - 2),$\cr 
$65.\  (3,2^s - 3,0,2^{s+t} - 2),$& $66.\  (3,2^s - 3,2^{s+t} - 2,0).$\cr
\end{tabular}}
\end{props}

\medskip
Now, we determine $(\mathbb F_2\underset{\mathcal A}\otimes R_4)_{2^{s+t} + 2^s -2}$. We have

\begin{thms}\label{dlc7.4} $(\mathbb F_2\underset{\mathcal A}\otimes R_4)_{2^{s+t} + 2^s -2}$ is  an $\mathbb F_2$-vector space with a basis consisting of all the classes represented by the following monomials:

1. $\phi (c_{t,i}), 1 \leqslant i \leqslant \rho_2(t)$ and $\phi(d_{t,j}), 1 \leqslant j \leqslant \rho_3(t)$, for $s = 2$.

2. $\phi (a_{t,s-2,j}), 1 \leqslant j \leqslant \mu_{t+1}(s-2)$ for $ s\geqslant 3$. 

Here $c_{t,i}, d_{t,j}$, $\rho_2(t), \rho_3(t)$ are defined as in Section \ref{3}, $a_{t,s-2,j}$ is defined as in Section \ref{6},  the homomorphism $\phi$ is defined as in Definition \ref{2.8}. By convention, we set $\mu_{t+1} (s-2) = \mu_5(s-2)$ for $t\geqslant 4$.

3. The monomials $b_{t,s,j}, j \geqslant 67,$ are determined as follows:

\smallskip
For $s \geqslant 2$,

\medskip
\centerline{\begin{tabular}{ll}
$67.\  (1,1,2^s - 2,2^{s+t} - 2),$& $68.\  (1,1,2^{s+t} - 2,2^s - 2),$\cr 
$69.\  (1,2^s - 2,1,2^{s+t} - 2),$& $70.\  (1,2^{s+t} - 2,1,2^s - 2),$\cr 
$71.\  (1,2,2^{s+t} - 3,2^s - 2),$& $72.\  (1,1,2^{s+1} - 2,2^{s+t} - 2^s - 2),$\cr 
$73.\  (1,2^{s+1} - 2,1,2^{s+t} - 2^s - 2),$& $74.\  (1,2,2^s - 1,2^{s+t} - 4),$\cr 
$75.\  (1,2,2^{s+t} - 4,2^s - 1),$& $76.\  (1,3,2^s - 2,2^{s+t} - 4),$\cr 
$77.\  (1,3,2^{s+t} - 4,2^s - 2),$& $78.\  (3,1,2^s - 2,2^{s+t} - 4),$\cr 
$79.\  (3,1,2^{s+t} - 4,2^s - 2),$& $80.\  (1,2,2^{s+1} - 4,2^{s+t} - 2^s - 1),$\cr 
$81.\  (1,2,2^{s+1} - 3,2^{s+t} - 2^s - 2),$& $82.\  (1,2,2^{s+1} - 1,2^{s+t} - 2^s - 4),$\cr 
$83.\  (1,2^{s+1} - 1,2,2^{s+t} - 2^s - 4),$& $84.\  (2^{s+1}-1,1,2,2^{s+t} - 2^s - 4),$\cr 
$85.\  (1,3,2^{s+1} - 4,2^{s+t} - 2^s - 2),$& $86.\  (3,1,2^{s+1} - 4,2^{s+t} - 2^s - 2),$\cr 
$87.\  (1,3,2^{s+1} - 2,2^{s+t} - 2^s - 4),$& $88.\  (3,1,2^{s+1} - 2,2^{s+t} - 2^s - 4),$\cr 
$89.\  (3,2^{s+1} - 3,2,2^{s+t} - 2^s - 4).$& \cr 
\end{tabular}}

\smallskip
For $s=2$ and $ t \geqslant 2$, \ $b_{t,2,90} = (3,3,4, 2^{t+2}-8)$. 

\medskip
For $s=2$ and $t=2$,  \ $b_{2,2,91} = (3,5,8,2).$

For $s \geqslant 3$,

\medskip
\centerline{\begin{tabular}{ll}
$90.\  (1,2,2^s - 4,2^{s+t} - 1),$& $91.\  (1,2,2^{s+t} - 1,2^s - 4),$\cr 
$92.\  (1,2^{s+t} - 1,2,2^s - 4),$& $93.\  (2^{s+t}-1,1,2,2^s - 4),$\cr 
$94.\  (1,2,2^s - 3,2^{s+t} - 2),$& $95.\  (1,3,2^s - 4,2^{s+t} - 2),$\cr 
$96.\  (1,3,2^{s+t} - 2,2^s - 4),$& $97.\  (3,1,2^s - 4,2^{s+t} - 2),$\cr 
$98.\  (3,1,2^{s+t} - 2,2^s - 4),$& $99.\  (1,2^s - 1,2,2^{s+t} - 4),$\cr 
\end{tabular}}
\centerline{\begin{tabular}{ll}
$100.\  (2^s-1,1,2,2^{s+t} - 4),$& $101.\  (3,2^{s+t} - 3,2,2^s - 4),$\cr 
$102.\  (3,2^s - 3,2,2^{s+t} - 4),$& $103.\  (3,5,2^{s+t} - 6,2^s - 4),$\cr 
$104.\  (3,5,2^{s+1} - 6,2^{s+t} - 2^s - 4).$& \cr
\end{tabular}}

\medskip
For $s=3$, $b_{t,3,105} = (3,5,6,2^{t+3}-8)$.

\smallskip
For $s \geqslant 4$, $b_{t,s,105} =  (3,5,2^s - 6,2^{s+t} - 4).$
\end{thms}

We prove the theorem by proving the following propositions.
 
\begin{props}\label{mdc7.4} The $\mathbb F_2$-vector space $(\mathbb F_2\underset {\mathcal A}\otimes R_4)_{2^{s+t}+2^s-2}$ is generated by the  elements listed in Theorem \ref{dlc7.4}.
\end{props}
 The proof of this proposition is based on the following lemmas.
 
\begin{lems}\label{7.4.1} The following matrices are strictly inadmissible
$$ \begin{pmatrix} 0&1&1&0\\ 0&1&0&1\\ 0&1&0&0\\ 0&0&1&0\end{pmatrix} \quad  \begin{pmatrix} 1&0&1&0\\ 1&0&0&1\\ 1&0&0&0\\ 0&0&1&0\end{pmatrix} \quad  \begin{pmatrix} 1&1&0&0\\ 1&0&0&1\\ 1&0&0&0\\ 0&1&0&0\end{pmatrix} \quad  \begin{pmatrix} 1&1&0&0\\ 1&0&1&0\\ 1&0&0&0\\ 0&1&0&0\end{pmatrix} $$    
$$ \begin{pmatrix} 0&1&1&0\\ 0&1&1&0\\ 0&0&1&0\\ 0&0&0&1\end{pmatrix} \quad  \begin{pmatrix} 0&1&1&0\\ 0&1&1&0\\ 0&1&0&0\\ 0&0&0&1\end{pmatrix} \quad  \begin{pmatrix} 0&1&0&1\\ 0&1&0&1\\ 0&1&0&0\\ 0&0&1&0\end{pmatrix} \quad  \begin{pmatrix} 1&0&1&0\\ 1&0&1&0\\ 0&0&1&0\\ 0&0&0&1\end{pmatrix} $$    
$$ \begin{pmatrix} 1&1&0&0\\ 1&1&0&0\\ 0&1&0&0\\ 0&0&0&1\end{pmatrix} \quad  \begin{pmatrix} 1&1&0&0\\ 1&1&0&0\\ 0&1&0&0\\ 0&0&1&0\end{pmatrix} \quad  \begin{pmatrix} 1&0&1&0\\ 1&0&1&0\\ 1&0&0&0\\ 0&0&0&1\end{pmatrix} \quad  \begin{pmatrix} 1&0&0&1\\ 1&0&0&1\\ 1&0&0&0\\ 0&0&1&0\end{pmatrix} $$    
$$ \begin{pmatrix} 1&1&0&0\\ 1&1&0&0\\ 1&0&0&0\\ 0&0&0&1\end{pmatrix} \quad  \begin{pmatrix} 1&1&0&0\\ 1&1&0&0\\ 1&0&0&0\\ 0&0&1&0\end{pmatrix} \quad  \begin{pmatrix} 1&0&0&1\\ 1&0&0&1\\ 1&0&0&0\\ 0&1&0&0\end{pmatrix} \quad  \begin{pmatrix} 1&0&1&0\\ 1&0&1&0\\ 1&0&0&0\\ 0&1&0&0\end{pmatrix}. $$
\end{lems}

\begin{proof} See \cite{ka}
\end{proof}

\begin{lems}\label{7.4.2} The following matrices are strictly inadmissible
$$ \begin{pmatrix} 1&0&1&0\\ 1&0&0&1\\ 0&1&0&0\\ 0&1&0&0\end{pmatrix} \quad  \begin{pmatrix} 1&0&1&0\\ 0&1&0&1\\ 0&1&0&0\\ 0&0&1&0\end{pmatrix} \quad  \begin{pmatrix} 1&1&0&0\\ 0&1&0&1\\ 0&1&0&0\\ 0&0&1&0\end{pmatrix} $$  
$$ \begin{pmatrix} 1&1&0&0\\ 1&0&0&1\\ 1&0&0&0\\ 0&0&1&0\end{pmatrix} \quad  \begin{pmatrix} 1&0&1&0\\ 1&0&0&1\\ 1&0&0&0\\ 0&1&0&0\end{pmatrix} \quad  \begin{pmatrix} 1&0&1&0\\ 1&0&0&1\\ 0&1&0&0\\ 0&0&0&1\end{pmatrix} $$  
$$ \begin{pmatrix} 1&0&1&0\\ 0&1&1&0\\ 0&1&0&0\\ 0&0&0&1\end{pmatrix} \quad  \begin{pmatrix} 1&0&0&1\\ 0&1&0&1\\ 0&1&0&0\\ 0&0&1&0\end{pmatrix} \quad  \begin{pmatrix} 1&0&1&0\\ 1&0&0&1\\ 0&1&0&0\\ 0&0&1&0\end{pmatrix}. $$
\end{lems}

\begin{proof} The monomials corresponding to the above matrices respectively are
\begin{align*}&(3,12,1,2),  (1,6,9,2),   (1,7,8,2),  (7,1,8,2),  (7,8,1,2),\\
&(3,4,1,10),  (1,6,3,8),  (1,6,8,3),  (3,4,9,2).
\end{align*} 

A direct computation shows
\begin{align*}
&(3,12,1,2) = Sq^1\big((3,11,1,2)+(3,7,3,4)+(1,7,5,4)\big)\\
&\quad  + Sq^2\big((2,11,1,2)+(5,7,2,2)+(2,7,3,4)+(1,7,6,2)\big)\\
&\quad +Sq^4(3,7,2,2) + (2,13,1,2)  + (2,11,1,4)  +(3,9,4,2) + (3,9,2,4)\\
&\quad + (2,9,3,4)  + (3,8,3,4) + (1,9,6,2) + (1,7,8,2) \ \text{mod  }\mathcal L_4(2;2;1;1),\\
&(1,6,9,2) = Sq^1\big((1,5,7,4)+(1,3,9,4)+(4,3,7,3)+(1,5,8,3)\\ 
&\quad+(1,3,8,5) + (1,4,9,3) + (1,4,7,5)\big) + Sq^2\big((1,6,7,2)\\ 
&\quad+(2,3,7,4)+(2,3,8,3) +(2,4,7,3) + (1,2,7,6) \big) +Sq^4(1,4,7,2)\\ 
&\quad + (1,4,11,2) +(1,3,10,4) + (1,6,8,3)+ (1,3,8,6)\\
&\quad + (1,4,10,3) + (1,2,9,6) +(1,2,7,8)\quad \text{mod  }\mathcal L_4(2;2;1;1),\\
&(1,7,8,2) = Sq^1\big((1,7,5,4) + (1,5,3,8) + (4,7,3,3) + (1,8,5,3)\\
&\quad + (1,8,3,5) + (1,9,4,3) + (1,7,4,5)\big) + Sq^2\big((1,7,6,2) + (2,7,3,4)\\
&\quad + (2,8,3,3) + (2,7,4,3) + (1,7,2,6)\big) +Sq^4\big((1,5,6,2) + (2,5,3,4)\\
&\quad + (1,4,6,3) + (1,4,3,6) + (1,6,4,3) + (1,5,2,6)\big) +(1,5,10,2)\\
&\quad + (1,6,3,8) + (1,4,10,3) + (1,4,3,10) + (1,6,8,3)\\
&\quad + (1,7,2,8) + (1,5,2,10)\quad \text{mod }\mathcal L_4(2;2;1;1),\\
&(7,1,8,2)  = \overline{\varphi}_1(1,7,8,2),\\
&(7,8,1,2) = Sq^1\big((7,5,1,4)+(7,3,4,3)+ (7,1,4,5)\big) + Sq^2\big((7,6,1,2)\\ 
&\quad+(7,3,2,4) + (7,2,4,3)+(7,1,2,6)  \big) +Sq^4\big((5,6,1,2) +  (5,3,2,4)\\ 
&\quad+ (4,3,4,3) + (5,2,4,3) + (5,1,2,6)  \big) + (5,10,1,2) + (5,3,2,8)\\
&\quad +(4,3,8,3) + (5,2,8,3)+ (7,1,2,8) + (5,1,2,10) \quad \text{mod  }\mathcal L_4(2;2;1;1),\\
&(3,4,1,10) = Sq^1\big((3,4,1,9)+(3,3,4,7) + (1,5,4,7)+(3,1,2,11)\big) \\ 
&\quad+ Sq^2\big((5,2,2,7)+(2,3,4,7) + (1,6,2,7)+(2,1,2,11)  \big)\\ 
&\quad +Sq^4(3,2,2,7) + (3,2,4,9) + (3,3,4,8) +(2,3,4,9)+ (2,1,2,13)\\
&\quad + (1,8,2,7) + (1,6,2,9) + (3,1,2,12) + (2,1,4,11) \ \text{mod  }\mathcal L_4(2;2;1;1),
\end{align*}
\begin{align*}
&(1,6,3,8) = Sq^1\big((2,5,3,7) + (2,3,5,7) + (2,3,3,9)\big)\\
&\quad +Sq^2\big((1,5,3,7) + (1,6,2,7) + (1,3,5,7) + (1,2,6,7) + (1,3,3,9)\big)\\
&\quad + Sq^4(1,4,2,7) + (1,6,2,9) + (1,4,2,11) + (1,2,8,7) + (1,2,6,9)\\
&\quad + (1,3,6,8) + (1,4,3,10) + (1,3,4,10)\quad \text{mod }\mathcal L_4(2;2;1;1),\\
&(1,6,8,3) = Sq^1\big((1,6,5,5)+(1,8,3,5) + (4,5,3,5)+(1,5,3,8)\\ 
&\quad+ (1,5,5,6)\big) + Sq^2\big((1,6,6,3)+(2,6,3,5) + (2,5,3,6)+(1,3,6,6)\big)\\ 
&\quad +Sq^4\big((1,4,6,3) + (1,4,3,6)\big)+ (1,4,10,3) + (1,4,3,10)\\
&\quad +(1,6,3,8) + (1,3,8,6)+ (1,3,6,8) \quad \text{mod  }\mathcal L_4(2;2;1;1),\\
&(3,4,9,2) = Sq^1\big((3,1,11,2)+(3,4,7,3) + (1,4,7,5) \big)\\ 
&\quad + Sq^2\big((5,2,7,2)+(2,1,11,2) + (2,4,7,3)+(1,2,7,6)\big)\\
&\quad  +Sq^4(3,2,7,2) + (3,2,9,4) + (3,1,12,2) +(2,1,13,2) + (2,1,11,4)\\
&\quad+ (3,4,8,3) + (2,4,9,3) + (1,2,7,8) + (1,2,9,6)  \quad \text{mod  }\mathcal L_4(2;2;1;1).
\end{align*}

The lemma is proved.
\end{proof}

\begin{lems}\label{7.4.3} The following matrix is strictly inadmissible
$$\begin{pmatrix} 1&1&0&0\\ 1&0&1&0\\ 0&1&1&0\\ 0&0&1&0\\ 0&0&0&1\end{pmatrix}  .$$
\end{lems}

\begin{proof} The monomials corresponding to the above matrix is  (3,5,14,16). By a direct calculation, we have
\begin{align*}
&(3,5,14,16) = Sq^1\big((3,3,17,14)+(3,3,14,17)  \big) \\ 
&\quad+ Sq^2\big((5,3,14,14)+(2,3,17,14) + (2,3,14,17)\big)\\ 
&\quad  +Sq^4(3,3,14,14) + Sq^8(3,5,8,14)+ (3,5,8,22)\\ 
&\quad  + (3,4,17,14) +(3,4,14,17)+ (2,5,17,14) \\
&\quad+ (2,5,14,17)  \quad \text{mod  }\mathcal L_4(2;2;2;1;1).
\end{align*}
The lemma follows.
\end{proof}

\begin{lems}\label{7.4.4} The following matrices are strictly inadmissible
$$\begin{pmatrix} 1&1&0&0\\ 1&0&0&1\\ 0&1&0&0\\ 0&0&1&0\\ 0&0&1&0\end{pmatrix} \quad \begin{pmatrix} 1&1&0&0\\ 1&0&0&1\\ 0&1&0&0\\ 0&0&1&0\\ 0&0&0&1\end{pmatrix} \quad \begin{pmatrix} 1&1&0&0\\ 1&0&1&0\\ 0&1&0&0\\ 0&0&1&0\\ 0&0&0&1\end{pmatrix} .$$
\end{lems}

\begin{proof} The monomials corresponding to the above matrices respectively are (3,5,24,2),  (3,5,8,18),   (3,5,10,16).  By a direct calculation, we have
\begin{align*}
&(3,5,24,2) = Sq^1(3,3,25,2)+ Sq^2\big((5,3,22,2)+(2,3,25,2) \big)\\ 
&\quad  +Sq^4\big((3,9,14,4) + (3,3,22,2) \big)  + Sq^8(3,5,14,4) \\ 
&\quad  + (3,5,18,8) + (3,3,24,4) +  (3,4,25,2) \\ 
&\quad  + (2,5,25,2) + (2,3,25,4)\quad \text{mod  }\mathcal L_4(2;2;1;1;1),\\
&(3,5,8,18) = Sq^1(3,3,1,26)+ Sq^2\big((5,3,2,22)+(2,3,1,26) \big)\\ 
&\quad  +Sq^4\big((3,9,4,14) + (3,3,2,22) \big)  + Sq^8(3,5,4,14) \\ 
&\quad  + (3,5,2,24) + (3,3,4,24) +  (3,4,1,26) \\ 
&\quad  + (2,5,1,26) + (2,3,1,28)\quad \text{mod  }\mathcal L_4(2;2;1;1;1),\\
&(3,5,10,16) = Sq^1(3,3,17,10)+ Sq^2\big((5,3,18,6)+(2,3,17,10) \big)\\ 
&\quad  +Sq^4\big((3,3,18,6) + (3,9,12,6) \big)   + Sq^8\big((3,5,10,8) + (3,5,12,6)\\ 
&\quad + (3,5,8,10)\big)  + (3,3,20,8) + (3,5,8,18) +  (3,4,17,10) \\ 
&\quad  + (2,5,17,10) + (2,3,17,12)\quad \text{mod  }\mathcal L_4(2;2;1;1;1).
\end{align*}
The lemma is proved.
\end{proof}

Using the results in Section \ref{5}, Section \ref{6}, Lemmas \ref{7.4.1}-\ref{7.4.4}, \ref{7.1.3}, \ref{7.1.5} and Theorem \ref{2.4},  we get Proposition \ref{mdc7.4}.

\medskip
Now, we show that the elements listed in Theorem \ref{dlc7.4} are linearly independent.

\begin{props}\label{7.4.6} The elements $[b_{2,2,i}], 67 \leqslant i \leqslant 91,$  $[\phi(c_{2,j})], 1 \leqslant j \leqslant 26,$ and $[\phi(d_{2,\ell})], 1 \leqslant \ell \leqslant 9,$ are linearly independent in $(\mathbb F_2\underset {\mathcal A}\otimes R_4)_{18}$.
\end{props}

\begin{proof} Suppose that there is a linear relation
\begin{equation}\sum_{i=67}^{91}\gamma_i[b_{2,2,i}] +\sum_{j=1}^{26}\eta_j[\phi(c_{2,j})] +\sum_{\ell=1}^{9}\xi_\ell [\phi(d_{2,\ell})]= 0, \tag {\ref{7.4.6}.1}
\end{equation}
with $\gamma_i, \eta_j, \xi_\ell \in \mathbb F_2$. 

Apply the squaring operation $Sq^0_*$ to (\ref{7.4.6}.1) and we get
$$\sum_{j=1}^{26}\eta_j[c_{2,j}] +\sum_{\ell=1}^{9}\xi_\ell [d_{2,j}]=0.$$
Since $\{[c_{3,i}]; 1\leqslant i \leqslant 26\}\cup \{[d_{3,j}]; 1\leqslant j \leqslant 9\}$ is a basis of $(\mathbb F_2 \underset {\mathcal A} \otimes P_4)_{7},$ we obtain $\eta_i=0, 1\leqslant i \leqslant 26, \xi_j = 0, 1 \leqslant j \leqslant 9$. Hence, (\ref{7.4.6}.1) becomes
\begin{equation}\sum_{i=67}^{91}\gamma_i[b_{2,2,i}] =0. \tag {\ref{7.4.6}.2}
\end{equation}

Now, we prove $\gamma_i=0, 67\leqslant i\leqslant 91.$ 

Applying the homomorphisms $f_1, f_6,$ to the relation (\ref{7.4.6}.2), we obtain
\begin{align*}
&\gamma_{70}[15,1,2]  +  \gamma_{\{75, 80, 82\}}[1,3,14]  +    \gamma_{75}[1,14,3] \\
&\quad  +   \gamma_{69}[3,1,14]  +   \gamma_{80}[1,6,11] +   \gamma_{82}[1,7,10]  +   \gamma_{73}[7,1,10]\\
&\quad   +   \gamma_{71}[3,13,2] +   \gamma_{\{74, 75, 80, 82\}}[3,3,12]  +   \gamma_{81}[3,5,10]  = 0,\\
&\gamma_{\{69, 71, 74, 75, 80, 81, 82\}}[1,2,15]  + \gamma_{\{76, 77, 85, 87\}}[1,3,14]  +   \gamma_{70}[1,14,3]\\
&\quad   +   \gamma_{\{78, 79, 86, 88\}}[3,1,14] +   \gamma_{73}[1,6,11] +   \gamma_{83}[1,7,10] \\
&\quad  +   \gamma_{84}[7,1,10]  +   \gamma_{90}[3,3,12]  +   \gamma_{\{89, 91\}}[3,5,10]  =0.   
\end{align*}

From these equalities, we get
\begin{equation}\begin{cases}
\gamma_i = 0, \ i= 69, 70, 71, 73, 74, 75, 81, 82, 83, 84, 90,\\ 
\gamma_{\{76, 77, 85, 87\}} = 
\gamma_{\{78, 79, 86, 88\}}  = \gamma_{\{89, 91\}} = 0.
\end{cases}\tag{\ref{7.4.6}.3}
\end{equation}

With the aid of (\ref{7.4.6}.3), the homomorphisms $f_2, f_3$ send (\ref{7.4.6}.2) to
\begin{align*}
&\gamma_{\{68, 79\}}[15,1,2]  +  \gamma_{\{67, 87, 91\}}[3,1,14]  +    \gamma_{\{72, 86, 87, 91\}}[7,1,10]\\
&\quad  +   \gamma_{\{77, 91\}}[3,13,2]  +   \gamma_{\{76, 77, 85, 87\}}[3,3,12]  +   \gamma_{85}[3,5,10]  = 0,\\   
&\gamma_{\{67, 78\}}[15,1,2]  +  \gamma_{\{68, 85, 89\}}[3,1,14]  +   \gamma_{\{72, 85, 88, 89\}}[7,1,10]\\
&\quad  +   \gamma_{\{76, 89\}}[3,13,2]  +   \gamma_{\{76, 77, 85, 87\}}[3,3,12]  +   \gamma_{87}[3,5,10]  = 0.  
\end{align*}

From these equalities, we get
\begin{equation}\begin{cases}
\gamma_{85} = \gamma_{87} = 0,\\
\gamma_{\{68, 79\}} =   
\gamma_{\{67, 91\}} =   
\gamma_{\{72, 86, 91\}} =  
\gamma_{\{77, 91\}} =   
\gamma_{\{76, 77\}} =  0,\\
\gamma_{\{67, 78\}} =  
\gamma_{\{68, 89\}} =    
\gamma_{\{72, 88, 89\}} =  
\gamma_{\{76, 89\}} = 0.  
\end{cases}\tag{\ref{7.4.6}.4}
\end{equation}

With the aid of (\ref{7.4.6}.3) and (\ref{7.4.6}.4),  the homomorphism $f_4$ send (\ref{7.4.6}.2) to
$$\gamma_{88}[1,3,14] + \gamma_{\{67, 72, 88\}}[1,7,10] +   \gamma_{88}[3,3,12] +  \gamma_{86}[3,5,10] = 0.$$
From this equality, we get
\begin{equation}
\gamma_{86} = \gamma_{88} = \gamma_{\{67,72\}} =  0. 
\tag{\ref{7.4.6}.5}
\end{equation}
Combining (\ref{7.4.6}.3), (\ref{7.4.6}.4) and (\ref{7.4.6}.5), we obtain
\begin{equation}\begin{cases}
\gamma_i = 0, i\ne 67, 68, 72, 76, 77, 78 ,79, 89, 91,\\
\gamma_{67} = \gamma_{68} = \gamma_{72} = \gamma_{76} = \gamma_{77} = \gamma_{78} = \gamma_{79} = \gamma_{89} = \gamma_{91}
\end{cases}\tag{\ref{7.4.6}.6}
\end{equation}
Substituting (\ref{7.4.6}.6) into the relation (\ref{7.4.6}.2), we obtain
\begin{equation}
\gamma_{67}[\theta] = 0,\tag{\ref{7.4.6}.7}
\end{equation}
where  $\theta = b_{2,2,67} + b_{2,2,68} + b_{2,2,72} + b_{2,2,76} + b_{2,2,77}
 + b_{2,2,78} + b_{2,2,79} + b_{2,2,89} + b_{2,2,91}.$

Now, we prove that the polynomial $\theta$ is non-hit. Suppose the contrary, that $\theta$ is hit. Then, we have
$$\theta = Sq^1(A) + Sq^2(B) + Sq^4(C) + Sq^8(D),$$
for some polynomials $A, B, C, D$ in $R_4$. Let $(Sq^2)^3$ act on the both sides of the above equality. We get
\begin{equation}(Sq^2)^3(\theta) = (Sq^2)^3Sq^4(C) + (Sq^2)^3Sq^8(D).\tag{\ref{7.4.6}.8}\end{equation}

 On the other hand, by a direct computation, it is not difficult to check that the relation (\ref{7.4.6}.8) is not true for all $C \in (R_4)_{14}$ and $D \in (R_4)_{10}$. This is a contradiction. Hence, $[\theta] \ne 0$ and the relation (\ref{7.4.6}.7) implies $\gamma_{67} = 0.$ 
The proposition is proved.
\end{proof}

\begin{props}\label{7.4.8} For $t \geqslant 3$, the elements $[b_{t,2,i}],$ for $ 67 \leqslant i \leqslant 90,$  $[\phi(c_{t,j})]$, for $1 \leqslant j \leqslant \rho_2(t),$ and $[\phi(d_{t,\ell})]$, for $1 \leqslant \ell \leqslant \rho_3(t),$ are linearly independent in $(\mathbb F_2\underset {\mathcal A}\otimes R_4)_{2^{t+2}+2}$.
\end{props}

\begin{proof} By a simple computation, we have
$$Sq^0_*([b_{t,2,i}]) = 0, Sq^0_*([\phi(c_{t,j}]) = [c_{t,j}], Sq^0_*([\phi(d_{t,\ell}]) = [d_{t,\ell}].$$
Since $\{[c_{t,j}]; 1\leqslant j \leqslant \rho_2(t)\}\cup \{[d_{t,\ell}]; 1\leqslant \ell \leqslant \rho_3(t)\}$ is a basis of the space $(\mathbb F_2\underset {\mathcal A}\otimes P_3)_{2^{t+1}-1}$, it suffices to show that the elements $[b_{t,2,i}], 67 \leqslant i \leqslant 90,$ are linearly independent. 

Suppose that there is a linear relation
\begin{equation}\sum_{i=67}^{90}\gamma_i[b_{t,2,i}] =0, \tag {\ref{7.4.8}.1}
\end{equation}
with $\gamma_i \in \mathbb F_2$. We prove $\gamma_i=0, 67\leqslant i\leqslant 90.$ 

Apply the homomorphisms $f_1,  f_6$ to the relation (\ref{7.4.8}.1) and we obtain
\begin{align*}
&\gamma_{70}v_{t,2,3}  +  \gamma_{\{75, 80, 82\}}v_{t,2,4}  +    \gamma_{75}v_{t,2,5}  +   \gamma_{69}v_{t,2,6}  +   \gamma_{80}v_{t,2,7}\\
&\quad   +   \gamma_{82}v_{t,2,8}  +   \gamma_{73}v_{t,2,9} +   \gamma_{81}v_{t,2,10}   +   \gamma_{71}v_{t,2,11}  +   \gamma_{\{74, 75, 80, 82\}}v_{t,2,21}  =0,\\   
&\gamma_{\{69, 71, 74, 75, 80, 81, 82\}}v_{t,2,1} +  \gamma_{\{76, 77, 85, 87\}}v_{t,2,4}  +   \gamma_{70}v_{t,2,5} +   \gamma_{73}v_{t,2,7}\\
&\quad  +   \gamma_{\{78, 79, 86, 88\}}v_{t,2,6} +   \gamma_{83}v_{t,2,8}  +   \gamma_{84}v_{t,2,9}  +  \gamma_{89}v_{t,2,10}  +   \gamma_{90}v_{t,2,21}  =0.   
\end{align*}

From these equalities, we get
\begin{equation}\begin{cases}
\gamma_i = 0, \ i =  69, 70, 71, 73, 74, 75, 80, 81, 82, 83, 84, 89, 90,\\
\gamma_{\{76, 77, 85, 87\}} = \gamma_{\{78, 79, 86, 88\}} = 0.
\end{cases}\tag{\ref{7.4.8}.2}
\end{equation}

With the aid of (\ref{7.4.8}.2), the homomorphisms $f_2, f_3, f_4$ send (\ref{7.4.8}.1) to
\begin{align*}
&\gamma_{\{68, 79\}}v_{t,2,3}  +  \gamma_{\{67, 87\}}v_{t,2,6}  +     \gamma_{\{72, 86, 87\}}v_{t,2,9} \\
&\quad +   \gamma_{85}v_{t,2,10}  +   \gamma_{77}v_{t,2,11}  +  \gamma_{\{76, 77, 85, 87\}}v_{t,2,21}  = 0,\\
&\gamma_{\{67, 78\}}v_{t,2,3}  +  \gamma_{\{68, 85\}}v_{t,2,6}    +   \gamma_{\{72, 85, 88\}}v_{t,2,9}\\
&\quad +   \gamma_{87}v_{t,2,10} +   \gamma_{76}v_{t,2,11}  +   \gamma_{\{76, 77, 85, 87\}}v_{t,2,21}   = 0,\\
&\gamma_{\{68, 77\}}v_{t,2,2}  +  \gamma_{\{67, 88\}}v_{t,2,4}  +   \gamma_{\{72, 85, 88\}}v_{t,2,8}\\
&\quad +   \gamma_{86}v_{t,2,10}  +   \gamma_{79}v_{t,2,11}   +   \gamma_{\{78, 88\}}v_{t,2,21} = 0.
\end{align*}

Computing from these equalities, we obtain
\begin{equation}\begin{cases}
\gamma_i = 0,\ i = 76, 77, 79, 85, 86, 87,\\
\gamma_{\{68, 79\}} =  
\gamma_{\{67, 87\}} =    
\gamma_{\{72, 86, 87\}} =   
\gamma_{\{76, 77, 85, 87\}} = 0,\\  
\gamma_{\{67, 78\}} =  
\gamma_{\{68, 85\}} =   
\gamma_{\{72, 85, 88\}} =  
\gamma_{\{76, 77, 85, 87\}} = 0,\\  
\gamma_{\{68, 77\}} =   
\gamma_{\{67, 88\}} =   
\gamma_{\{72, 85, 88\}} =  
\gamma_{\{78, 88\}} = 0.           
\end{cases}\tag{\ref{7.4.8}.3}
\end{equation}

Combining (\ref{7.4.8}.2) and (\ref{7.4.8}.3), we get $\gamma_i =0$ for any $i$.
The proposition is proved.
\end{proof}

\begin{props}\label{7.4.9} For $t \geqslant 2$, the elements $[b_{t,3,i}], 67 \leqslant i \leqslant 105,$ and $[\phi(a_{t,1,j})], 1 \leqslant j \leqslant \nu (t,3),$  are linearly independent in $(\mathbb F_2\underset {\mathcal A}\otimes R_4)_{2^{t+3}+6}$. Here $\nu (2,3)  = 87, \nu (3,3)= 136, \nu (t,3) = 150$ for $t \geqslant 4$.
\end{props}

\begin{proof} Suppose that there is a linear relation
\begin{equation}\sum_{i=67}^{105}\gamma_i[b_{t,3,i}] +\sum_{j = 1}^{\nu (t,3)}\eta_j[\phi(a_{t,1,j})] = 0, \tag {\ref{7.4.9}.1}
\end{equation}
with $\gamma_i, \eta_j \in \mathbb F_2$. 

Applying the squaring operation $Sq^0_*$ to (\ref{7.4.9}.1), we get
$$\sum_{j=1}^{\nu (t,3)}\eta_j[a_{t,1,j}]=0.$$
Since $\{[a_{t,1,j}]; 1\leqslant j \leqslant \nu (t,3)\}$ is a basis of $(\mathbb F_2 \underset {\mathcal A} \otimes P_4)_{2^{t+2}+1},$ we obtain $\eta_j=0, 1\leqslant j \leqslant \nu (t,3)$. Hence, (\ref{7.4.9}.1) becomes
\begin{equation}\sum_{i=67}^{105}\gamma_i[b_{t,2,i}] =0. \tag {\ref{7.4.9}.2}
\end{equation}
Now, we prove $\gamma_i=0, 67\leqslant i\leqslant 105.$ 

Applying the homomorphisms $f_1,  f_6$ to the relation (\ref{7.4.9}.2), we obtain
\begin{align*}
&\gamma_{90}v_{t,3,1} +  \gamma_{91}v_{t,3,2} +   \gamma_{70}v_{t,3,3} +  \gamma_{74}v_{t,3,4}\\
&\quad +  \gamma_{75}v_{t,3,5} +  \gamma_{69}v_{t,3,6} +  \gamma_{80}v_{t,3,7} +  \gamma_{82}v_{t,3,8}\\
&\quad +  \gamma_{73}v_{t,3,9}+  \gamma_{81}v_{t,3,10} +  \gamma_{71}v_{t,3,11} +  \gamma_{94}v_{t,3,21}  = 0,\\
&  \gamma_{69}v_{t,3,1} +  \gamma_{92}v_{t,3,2} +  \gamma_{93}v_{t,3,3} +  \gamma_{99}v_{t,3,4} +  \gamma_{70}v_{t,3,5}\\
&\quad +  \gamma_{100}v_{t,3,6} +  \gamma_{73}v_{t,3,7} +  \gamma_{83}v_{t,3,8} +  \gamma_{84}v_{t,3,9}\\
&\quad +  \gamma_{89}v_{t,3,10}  +  \gamma_{101}v_{t,3,11}+  \gamma_{\{102, 103, 104, 105\}}v_{t,3,21} =0.     
\end{align*}

From these equalities, we get
\begin{equation}\begin{cases}
\gamma_i = 0,\ i = 69, 70, 71, 73, 74, 75, 80, 81,\\ 82, 83,
  84, 89, 90, 91, 92, 93, 94, 99, 100, 101,\\
\gamma_{\{102, 103, 104, 105\}} = 0.
\end{cases}\tag{\ref{7.4.9}.3}
\end{equation}

With the aid of (\ref{7.4.9}.3), the homomorphisms $f_4, f_5$ send (\ref{7.4.9}.2) to
\begin{align*}
&\gamma_{\{68, 77, 98, 103\}}v_{t,3,2}  +  \gamma_{\{67, 78, 95, 102\}}v_{t,3,4}  +  \gamma_{\{72, 85, 88, 104\}}v_{t,3,8}\\
&\quad  +   \gamma_{86}v_{t,3,10} +   \gamma_{79}v_{t,3,11} +   \gamma_{97}v_{t,3,21} = 0,\\   
&\gamma_{\{67, 76, 97\}}v_{t,3,2}  +  \gamma_{\{68, 79, 96\}}v_{t,3,4}  +   \gamma_{\{72, 86, 87\}}v_{t,3,8}\\
&\quad  +   \gamma_{88}v_{t,3,10}+   \gamma_{\{78, 105\}}v_{t,3,11}   +   \gamma_{98}v_{t,3,21}  = 0.     
\end{align*}
Computing from these equalities, we obtain
\begin{equation}\begin{cases}
\gamma_{79} = \gamma_{86} = \gamma_{88} = \gamma_{97} = \gamma_{98} = 0,\\
\gamma_{\{68, 77, 98, 103\}} = 
\gamma_{\{67, 78, 95, 102\}}=  
\gamma_{\{72, 85, 88, 104\}} = 0,\\
\gamma_{\{67, 76, 97\}}= 
\gamma_{\{68, 79, 96\}} = 
\gamma_{\{72, 86, 87\}}= 
\gamma_{\{78, 105\}} = 0.
\end{cases}\tag{\ref{7.4.9}.4}
\end{equation}

With the aid of (\ref{7.4.9}.3) and (\ref{7.4.9}.4), the homomorphisms $f_2, f_3$ send the relation (\ref{7.4.9}.2) respectively to
\begin{align*}
&\gamma_{\{68, 96\}}v_{t,3,3}  +  \gamma_{\{67, 76\}}v_{t,3,6}  +    \gamma_{\{72, 87\}}v_{t,3,9}\\
&\quad +   \gamma_{85}v_{t,3,10}   +     \gamma_{77}v_{t,3,11}  +  \gamma_{\{77, 85, 95\}}v_{t,3,21}  = 0,\\
&\gamma_{\{67, 78, 95, 102\}}v_{t,3,3}  +  \gamma_{\{68, 77, 103\}}v_{t,3,6}  +   \gamma_{\{72, 85, 104\}}v_{t,3,9}\\
&\quad +   \gamma_{87}v_{t,3,10}  +   \gamma_{\{76, 105\}}v_{t,3,11}    +   \gamma_{\{76, 87, 96, 105\}}v_{t,3,21} = 0.
\end{align*}

The above equalities imply
\begin{equation}\begin{cases}
\gamma_{77} = \gamma_{85} = \gamma_{87} = 0,\\
\gamma_{\{68, 96\}} =   
\gamma_{\{67, 76\}} = 
 \gamma_{\{72, 87\}} =   
\gamma_{\{77, 85, 95\}} = 0,\\
\gamma_{\{67, 78, 95, 102\}} =   
\gamma_{\{68, 77, 103\}} =   
\gamma_{\{72, 85, 104\}} = 0,\\  
\gamma_{\{76, 87, 96, 105\}} =   
\gamma_{\{76, 105\}} =  0.
\end{cases}\tag{\ref{7.4.9}.5}
\end{equation}

Combining (\ref{7.4.9}.3), (\ref{7.4.9}.4) and  (\ref{7.4.9}.5), we get $\gamma_i =0$ for any $i$.
The proposition is proved.
\end{proof}

\begin{props}\label{7.4.10} For $t \geqslant 2, s \geqslant 4$, the elements $[b_{t,s,i}],$ for $67 \leqslant i \leqslant 105,$ and $[\phi(a_{t,s-2,j})]$, for $1 \leqslant j \leqslant \nu (t,s),$  are linearly independent in $(\mathbb F_2\underset {\mathcal A}\otimes R_4)_{2^{s+t} + 2^s - 2}$. Here $\nu (2,4) = 135, \nu (2,s) = 150$ for $ s\geqslant 5$, $\nu (3,4) = 180, \nu (3,s) = 195$ for $s \geqslant 5$, $\nu (t,4) = 195, \nu (t,s) = 210$ for $t \geqslant 4, s\geqslant 5$.
\end{props}

\begin{proof} Suppose that there is a linear relation
\begin{equation}\sum_{i=67}^{105}\gamma_i[b_{t,s,i}] +\sum_{j = 1}^{\nu (t,s)}\eta_j[\phi(a_{t,s-2,j})] = 0, \tag {\ref{7.4.10}.1}
\end{equation}
with $\gamma_i, \eta_j \in \mathbb F_2$. 

Applying the squaring operation $Sq^0_*$ to (\ref{7.4.10}.1), we get
$$\sum_{j=1}^{\nu (t,s)}\eta_j[a_{t,s-2,j}]=0.$$
Since $\{[a_{t,s-2,j}]; 1\leqslant j \leqslant \nu (t,s)\}$ is a basis of $(\mathbb F_2 \underset {\mathcal A} \otimes P_4)_{2^{s+t-1}+2^{s-1} - 3},$ we obtain $\eta_j=0, 1\leqslant j \leqslant \nu (t,s)$. Hence, (\ref{7.4.10}.1) becomes
\begin{equation}\sum_{i=67}^{105}\gamma_i[b_{t,s,i}] =0. \tag {\ref{7.4.10}.2}
\end{equation}
Now, we prove $\gamma_i=0, 67\leqslant i\leqslant 105.$ 

Apply the homomorphisms $f_1,  f_6$ to the relation (\ref{7.4.10}.2) and we obtain
\begin{align*}
&\gamma_{90}v_{t,s,1}   +  \gamma_{91}v_{t,s,2}   +    \gamma_{70}v_{t,s,3}   +   \gamma_{74}v_{t,s,4}   +   \gamma_{75}v_{t,s,5}   +   \gamma_{69}v_{t,s,6}\\
&\  +   \gamma_{80}v_{t,s,7}   +   \gamma_{82}v_{t,s,8}   +   \gamma_{73}v_{t,s,9}   +   \gamma_{81}v_{t,s,10}   +   \gamma_{71}v_{t,s,11}   + \gamma_{94}v_{t,s,21}   = 0,\\   
&\gamma_{69}v_{t,s,1}   +   \gamma_{92}v_{t,s,2}   +   \gamma_{93}v_{t,s,3}   +   \gamma_{99}v_{t,s,4}   +   \gamma_{70}v_{t,s,5}   +   \gamma_{100}v_{t,s,6} \\
&\  +   \gamma_{73}v_{t,s,7}   +   \gamma_{83}v_{t,s,8}   +   \gamma_{84}v_{t,s,9}   +   \gamma_{89}v_{t,s,10}  +   \gamma_{101}v_{t,s,11}    +   \gamma_{102}v_{t,s,21} =0. 
\end{align*}

From these equalities, we get
\begin{equation}\begin{cases}
\gamma_i = 0,\ i= 69, 70, 71, 73, 74, 75, 80, 81, 82, 83, 84, \\
\hskip2.3cm 89, 90, 91, 92, 93, 94, 99, 100, 101, 102.
\end{cases}\tag{\ref{7.4.10}.3}
\end{equation}

With the aid of (\ref{7.4.10}.3), the homomorphisms $f_4, f_5$ send (\ref{7.4.10}.2) to
\begin{align*}
&\gamma_{\{68, 77, 98, 103\}}v_{t,s,2}   + \gamma_{\{67, 78, 95, 105\}}v_{t,s,4}\\
&\quad   +   \gamma_{\{72, 85, 88, 104\}}v_{t,s,8}   +   \gamma_{86}v_{t,s,10}   +   \gamma_{79}v_{t,s,11}   +   \gamma_{97}v_{t,s,21}   =0,\\  
&\gamma_{91}v_{t,s,1}   +  \gamma_{\{67, 76, 97\}}v_{t,s,2}   +    \gamma_{93}v_{t,s,3}   +   \gamma_{\{68, 79, 96\}}v_{t,s,4}\\
&\quad  +   \gamma_{74}v_{t,s,5}   +   \gamma_{100}v_{t,s,6}   +   \gamma_{82}v_{t,s,7}   +   \gamma_{\{72, 86, 87\}}v_{t,s,8}   +   \gamma_{84}v_{t,s,9}\\
&\quad   +   \gamma_{88}v_{t,s,10}   +   \gamma_{\{78, 84, 100\}}v_{t,s,11}   +   \gamma_{\{84, 98, 100\}}v_{t,s,21}   = 0.   
\end{align*}
Computing from these equalities, we obtain
\begin{equation}\begin{cases}
\gamma_i = 0,\ i= 74, 79, 82, 84, 86, 88, 91, 93, 97, 100,\\
\gamma_{\{68, 77, 98, 103\}} = 
\gamma_{\{67, 78, 95, 105\}} =  
\gamma_{\{72, 85, 88, 104\}} = 0,\\
\gamma_{\{67, 76, 97\}} = 
\gamma_{\{68, 79, 96\}} = 
\gamma_{\{72, 86, 87\}} = 0,\\
\gamma_{\{84, 98, 100\}} = 
\gamma_{\{78, 84, 100\}} =  0.  
\end{cases}\tag{\ref{7.4.10}.4}
\end{equation}
With the aid of (\ref{7.4.10}.3) and (\ref{7.4.10}.4), the homomorphisms $f_2, f_3$ send the relation (\ref{7.4.10}.2) respectively to
\begin{align*}
&\gamma_{\{68, 96\}}v_{t,s,3}   +  \gamma_{\{67, 76\}}v_{t,s,6}   +   \gamma_{\{72, 87\}}v_{t,s,9}\\
&\quad  +   \gamma_{85}v_{t,s,10}    +   \gamma_{77}v_{t,s,11}   +  \gamma_{\{77, 85, 95\}}v_{t,s,21}   =0,\\   
&\gamma_{\{67, 78, 95, 105\}}v_{t,s,3}   +  \gamma_{\{68, 77, 98, 103\}}v_{t,s,6}  +   \gamma_{87}v_{t,s,10}\\
&\quad   +   \gamma_{\{72, 85, 104\}}v_{t,s,9}     +   \gamma_{76}v_{t,s,11}   +   \gamma_{\{76, 87, 96\}}v_{t,s,21} =0.   
\end{align*}

The above equalities imply
\begin{equation}\begin{cases}
\gamma_{76} = \gamma_{77} = \gamma_{85} = \gamma_{87} = 0,\\
\gamma_{\{68, 96\}} = 
\gamma_{\{67, 76\}} =  
\gamma_{\{77, 85, 95\}} = 
\gamma_{\{67, 78, 95, 105\}} = 0,\\
\gamma_{\{68, 77, 98, 103\}} = 
\gamma_{\{72, 85, 104\}} = 
\gamma_{\{76, 87, 96\}} = 
\gamma_{\{72, 87\}} = 0.
\end{cases}\tag{\ref{7.4.10}.5}
\end{equation}

Combining (\ref{7.4.10}.3), (\ref{7.4.10}.4) and  (\ref{7.4.10}.5), we get $\gamma_i =0$ for any $i$.
The proposition is proved.
\end{proof}

\begin{rems}\label{7.4.7} The element $[\theta]$ determined as in the proof of Proposition \ref{7.4.6}, is an $GL_4(\mathbb F_2)$-invariant in  $(\mathbb F_2\underset {\mathcal A}\otimes R_4)_{18}$.
\end{rems}

\section{The indecomposables of $P_4$ in degree $2^{s+t+u} + 2^{s+t} + 2^s-3$}\label{8}

\smallskip
\subsection{The $\tau$-sequences of the admissible monomials}\label{8.1}\ 

\medskip
In this subsection, we prove the following.

\begin{lems}\label{8.1.2}  If $x$ is an admissible monomial of degree $2^{s+t+u}+ 2^{s+t} + 2^s-3$ in $P_4$ then
$$\tau(x) = (\underset{\text{$s$ times}}{\underbrace{3;3;\ldots ; 3}};\underset{\text{$t$ times}}{\underbrace{2;2;\ldots ; 2}};\underset{\text{$u$ times}}{\underbrace{1;1;\ldots ; 1}}).$$
\end{lems}

We need the following for the proof of this lemma.
From Lemmas \ref{5.1}, \ref{2.5}, Theorem \ref{2.4} and the results in Section \ref{7}, we get

\begin{lems}\label{8.1.1a} Let $x$ be a monomial of degree $2^{s+1}+2^s-2$ in $P_4$. If  $x$ is inadmissible then there is a strictly inadmissible matrix $\Delta$ such that $\Delta \triangleright x$.
\end{lems} 

\begin{lems}\label{8.1.1} The following matrices are strictly inadmissible
$$ \begin{pmatrix} 1&0&1&1\\ 1&0&0&1\\ 0&1&1&1\end{pmatrix} \quad  \begin{pmatrix} 1&0&1&1\\ 1&0&1&0\\ 0&1&1&1\end{pmatrix} \quad  \begin{pmatrix} 1&1&0&1\\ 1&0&0&1\\ 0&1&1&1\end{pmatrix} \quad  \begin{pmatrix} 1&1&1&0\\ 1&0&1&0\\ 0&1&1&1\end{pmatrix} $$    $$ \begin{pmatrix} 1&1&0&1\\ 1&1&0&0\\ 0&1&1&1\end{pmatrix} \quad  \begin{pmatrix} 1&1&1&0\\ 1&1&0&0\\ 0&1&1&1\end{pmatrix} \quad  \begin{pmatrix} 1&1&0&1\\ 1&1&0&0\\ 1&0&1&1\end{pmatrix} \quad  \begin{pmatrix} 1&1&1&0\\ 1&1&0&0\\ 1&0&1&1\end{pmatrix} $$    $$ \begin{pmatrix} 1&1&1&0\\ 1&0&0&1\\ 0&1&1&1\end{pmatrix} \quad  \begin{pmatrix} 1&1&0&1\\ 1&0&1&0\\ 0&1&1&1\end{pmatrix} .$$
\end{lems}

\begin{proof} The monomials corresponding to the above matrices respectively are
 (3,4,5,7),  (3,4,7,5),   (3,5,4,7),  (3,5,7,4),  (3,7,4,5),  (3,7,5,4),  (7,3,4,5),  (7,3,5,4),  (3,5,5,6), (3,5,6,5). We proved the lemma for the matrices associated to the monomials (3,4,5,7) and (3,5,5,6).  By a direct computation, we have
\begin{align*}
(3,4,5,7) &= Sq^1(3,1,3,11) + Sq^2(5,2,3,7)\\
&\quad + Sq^4(3,2,3,7) \ \text{mod } L_4(3;2;3),\\
(3,5,5,6) &= Sq^1(3,3,3,9) +Sq^2(5,3,3,6)\\
&\quad + Sq^4(3,3,3,6)  \ \text{mod } L_4(3;2;3).
\end{align*}
The lemma is proved.
\end{proof}

\begin{proof}[Proof of Lemma \ref{8.1.2}] Observe that $z = (2^{s+t+u}-1,2^{s+t}-1, 2^s-1,0)$ is the minimal spike of degree $2^{s+t+u} + 2^{s+t} + 2^s-3$ and $$\tau(z) = (\underset{\text{$s$ times}}{\underbrace{3;3;\ldots ; 3}};\underset{\text{$t$ times}}{\underbrace{2;2;\ldots ; 2}};\underset{\text{$u$ times}}{\underbrace{1;1;\ldots ; 1}}).$$ 
Since $2^{s+t+u} + 2^{s+t} + 2^s - 3$ is odd and $x$ is admissible, using Proposition \ref{2.6} and Theorem \ref{2.12}, we get  $\tau_i(x) =3$ for $1 \leqslant i \leqslant s$. Let $M=(\varepsilon_{ij}(x)), i\geqslant 1, 1\leqslant j\leqslant 4$, be the matrix associated with $x$. We set $M'=(\varepsilon_{ij}(x)), i > s, 1\leqslant j\leqslant 4$ and denote by $x'$ the monomial corresponding to $M'$. We have $\tau_j(x') = \tau_{j+s}(x)$ for all $j \geqslant 1$ and
\begin{align*}
2^{s+t+u} + 2^{s+t} + 2^s-3 &= \deg x = \sum_{i\geqslant 1}2^{i-1}\tau_i(x)\\
&= 3(2^s-1) + 2^s\sum_{j\geqslant 1}2^{j-1}\tau_{j+s}(x)\\
&= 3.2^s -3 + 2^s\deg x'.
\end{align*}
This equality implies $\deg x' = 2^{t+u} + 2^u -2$. Since $x$ is admissible, using Lemma \ref{8.1.1a}, Theorem \ref{2.4}, we see that $x'$ is also admissible. Now, using Lemmas \ref{7.1.5}, \ref{8.1.1}, Proposition \ref{2.6} and Theorem \ref{2.4}, we obtain
$$\tau(x') = (\underset{\text{$t$ times}}{\underbrace{2;2;\ldots ; 2}};\underset{\text{$u$ times}}{\underbrace{1;1;\ldots ; 1}}).$$
The lemma is proved.
\end{proof}

\subsection{The case $s=1$ and $t=1$}\label{8.2} \

\medskip
According to Kameko \cite{ka},  $(\mathbb F_2\underset{\mathcal A}\otimes P_3)_{2^{u+2}+3}$ is an $\mathbb F_2$-vector space with a basis given by the following classes:

\smallskip
For $u \geqslant 1$,

\centerline{\begin{tabular}{ll}
\smallskip
$w_{u,1,1,1} = [1,3,2^{u + 2} - 1],$& $w_{u,1,1,2} = [1,2^{u + 2} - 1,3],$\cr
$w_{u,1,1,3} = [3,1,2^{u + 2} - 1],$& $w_{u,1,1,4} = [3,2^{u + 2} - 1,1],$\cr 
$w_{u,1,1,5} = [2^{u+2}-1,1,3],$& $w_{u,1,1,6} = [2^{u+2}-1,3,1],$\cr 
$w_{u,1,1,7} = [1,7,2^{u + 2} - 5],$& $w_{u,1,1,8} = [7,1,2^{u + 2} - 5],$\cr
\end{tabular}}

\smallskip
For $u \geqslant 2$,

\smallskip
\centerline{\begin{tabular}{ll}
$w_{u,1,1,9} = [7,2^{u + 2} - 5,1],$&$w_{u,1,1,10} = [3,3,2^{u + 2} - 3],$\cr
$w_{u,1,1,11} = [3,2^{u + 2} - 3,3],$& $w_{u,1,1,12} = [3,5,2^{u + 2} - 5],$\cr 
$w_{u,1,1,13} = [3,7,2^{u + 2} - 7],$& $w_{u,1,1,14} = [7,3,2^{u + 2} - 7].$ \cr 
\end{tabular}}

\smallskip
For $u = 2$, \ $w_{2,1,1,15} = [7,9,3].$

\smallskip
So, we easily obtain

\begin{props}\label{8.2.2} For any positive integer $u$, the dimension of the space $(\mathbb F_2\underset{\mathcal A}\otimes Q_4)_{2^{u+2} +3}$ is  given by
$$\dim (\mathbb F_2\underset{\mathcal A}\otimes Q_4)_{2^{u+2} +3} = \begin{cases} 32, & u=1,\\ 60,& u=2,\\ 56, &u \geqslant 3.
\end{cases}$$
\end{props}

\medskip
Now, we determine $(\mathbb F_2\underset{\mathcal A}\otimes R_4)_{2^{u+2}+3}$. We have

\begin{thms}\label{dlc8.2} $(\mathbb F_2\underset{\mathcal A}\otimes R_4)_{2^{u+2} + 3}$ is  an $\mathbb F_2$-vector space with a basis consisting of all the classes represented by the monomials $a_{u,1,1,j}, j \geqslant 1$, which are determined as follows:

\smallskip
For $u \geqslant 1$,

\medskip
\centerline{\begin{tabular}{lll}
$1.\ (1,1,2,2^{u + 2} - 1),$& $2.\ (1,1,2^{u + 2} - 1,2),$& $3.\ (1,2,1,2^{u + 2} - 1),$\cr 
$4.\ (1,2,2^{u + 2} - 1,1),$& $5.\ (1,2^{u + 2} - 1,1,2),$& $6.\ (1,2^{u + 2} - 1,2,1),$\cr 
$7.\ (2^{u + 2} - 1,1,1,2),$& $8.\ (2^{u + 2} - 1,1,2,1),$& $9.\ (1,1,3,2^{u + 2} - 2),$\cr 
$10.\ (1,1,2^{u + 2} - 2,3),$& $11.\ (1,3,1,2^{u + 2} - 2),$& $12.\ (1,3,2^{u + 2} - 2,1),$\cr 
$13.\ (1,2^{u + 2} - 2,1,3),$& $14.\ (1,2^{u + 2} - 2,3,1),$& $15.\ (3,1,1,2^{u + 2} - 2),$\cr 
$16.\ (3,1,2^{u + 2} - 2,1),$& $17.\ (1,2,3,2^{u + 2} - 3),$& $18.\ (1,2,2^{u + 2} - 3,3),$\cr 
$19.\ (1,3,2,2^{u + 2} - 3),$& $20.\ (1,3,2^{u + 2} - 3,2),$& $21.\ (3,1,2,2^{u + 2} - 3),$\cr 
$22.\ (3,1,2^{u + 2} - 3,2),$& $23.\ (3,2^{u + 2} - 3,1,2),$& $24.\ (3,2^{u + 2} - 3,2,1),$\cr 
$25.\ (1,3,3,2^{u + 2} - 4),$& $26.\ (1,3,2^{u + 2} - 4,3),$& $27.\ (3,1,3,2^{u + 2} - 4),$\cr 
$28.\ (3,1,2^{u + 2} - 4,3),$& $29.\ (3,3,1,2^{u + 2} - 4),$& $30.\ (3,3,2^{u + 2} - 4,1).$\cr
\end{tabular}}

\smallskip
For $u=1, \ a_{1,1,1,31} = (3,4,1,3), \quad a_{1,1,1,32} = (3,4,3,1).$

\smallskip
For $u\geqslant 2$,

\medskip
\centerline{\begin{tabular}{lll}
$31.\ (1,1,6,2^{u + 2} - 5),$& $32.\ (1,6,1,2^{u + 2} - 5),$& $33.\ (1,6,2^{u + 2} - 5,1),$\cr 
$34.\ (1,1,7,2^{u + 2} - 6),$& $35.\ (1,7,1,2^{u + 2} - 6),$& $36.\ (1,7,2^{u + 2} - 6,1),$\cr 
$37.\ (7,1,1,2^{u + 2} - 6),$& $38.\ (7,1,2^{u + 2} - 6,1),$& $39.\ (1,2,5,2^{u + 2} - 5),$\cr 
$40.\ (1,3,4,2^{u + 2} - 5),$& $41.\ (3,1,4,2^{u + 2} - 5),$& $42.\ (3,4,1,2^{u + 2} - 5),$\cr 
$43.\ (3,4,2^{u + 2} - 5,1),$& $44.\ (1,2,7,2^{u + 2} - 7),$& $45.\ (1,7,2,2^{u + 2} - 7),$\cr 
$46.\ (7,1,2,2^{u + 2} - 7),$& $47.\ (1,3,5,2^{u + 2} - 6),$& $48.\ (3,1,5,2^{u + 2} - 6),$\cr 
$49.\ (3,5,1,2^{u + 2} - 6),$& $50.\ (3,5,2^{u + 2} - 6,1),$& $51.\ (1,3,6,2^{u + 2} - 7),$\cr 
$52.\ (1,6,3,2^{u + 2} - 7),$& $53.\ (3,1,6,2^{u + 2} - 7),$& $54.\ (1,3,7,2^{u + 2} - 8),$\cr 
$55.\ (1,7,3,2^{u + 2} - 8),$& $56.\ (3,1,7,2^{u + 2} - 8),$& $57.\ (3,7,1,2^{u + 2} - 8),$\cr 
$58.\ (7,1,3,2^{u + 2} - 8),$& $59.\ (7,3,1,2^{u + 2} - 8),$& $60.\ (3,5,2,2^{u + 2} - 7),$\cr 
$61.\ (3,3,4,2^{u + 2} - 7),$& $62.\ (3,4,3,2^{u + 2} - 7),$& $63.\ (3,3,5,2^{u + 2} - 8),$\cr 
$64.\ (3,5,3,2^{u + 2} - 8).$& & \cr
\end{tabular}}

\smallskip
For $u=2$,

\medskip
\centerline{\begin{tabular}{llll}
$65.\  (3,12,1,3),$& $66.\  (3,12,3,1),$& $67.\  (1,7,9,2),$& $68.\  (7,1,9,2),$\cr
$69.\  (7,9,1,2),$& $70.\  (7,9,2,1),$& $71.\  (1,6,9,3),$& $72.\  (1,7,8,3),$\cr 
$73.\  (3,7,8,1),$& $74.\  (7,1,8,3),$& $75.\  (7,3,8,1),$& $76.\  (7,8,1,3),$\cr 
$77.\  (7,8,3,1),$& $78.\  (3,5,9,2),$& $79.\  (3,4,9,3),$& $80.\  (3,5,8,3).$ \cr
\end{tabular}}
\end{thms}

We prove the theorem by proving some propositions.

\begin{props}\label{mdc8.2}The $\mathbb F_2$-vector space $(\mathbb F_2\underset {\mathcal A}\otimes R_4)_{2^{u+2}+3}$ is generated by the  elements listed in Theorem \ref{dlc8.2}.
\end{props}

The proposition is prove by combining Theorem \ref{2.4} and the following.

\begin{lems}\label{8.2.1} The following matrices are strictly inadmissible
 $$\begin{pmatrix} 1&0&1&1\\ 1&0&0&1\\ 0&1&0&0\\ 0&1&0&0\\ 0&1&0&0\end{pmatrix} \quad \begin{pmatrix} 1&0&1&1\\ 1&0&1&0\\ 0&1&0&0\\ 0&1&0&0\\ 0&1&0&0\end{pmatrix} \quad \begin{pmatrix} 1&1&0&1\\ 1&1&0&0\\ 0&1&0&0\\ 0&0&1&0\\ 0&0&1&0\end{pmatrix} \quad \begin{pmatrix} 1&1&0&1\\ 1&1&0&0\\ 1&0&0&0\\ 0&0&1&0\\ 0&0&1&0\end{pmatrix} $$    
$$\begin{pmatrix} 1&1&1&0\\ 1&0&0&1\\ 0&1&0&0\\ 0&0&1&0\\ 0&0&1&0\end{pmatrix} \quad \begin{pmatrix} 1&0&1&1\\ 1&0&0&1\\ 0&1&0&0\\ 0&0&1&0\\ 0&0&1&0\end{pmatrix} \quad \begin{pmatrix} 1&1&0&1\\ 1&0&0&1\\ 0&1&0&0\\ 0&0&1&0\\ 0&0&1&0\end{pmatrix} \quad \begin{pmatrix} 1&0&1&1\\ 1&0&0&1\\ 0&1&0&0\\ 0&0&1&0\\ 0&0&0&1\end{pmatrix} $$    
$$\begin{pmatrix} 1&0&1&1\\ 1&0&1&0\\ 0&1&0&0\\ 0&0&1&0\\ 0&0&0&1\end{pmatrix} \quad \begin{pmatrix} 1&1&0&1\\ 1&0&0&1\\ 0&1&0&0\\ 0&0&1&0\\ 0&0&0&1\end{pmatrix} \quad \begin{pmatrix} 1&1&1&0\\ 1&0&0&1\\ 0&1&0&0\\ 0&0&1&0\\ 0&0&0&1\end{pmatrix} \quad \begin{pmatrix} 1&1&0&1\\ 1&0&1&0\\ 0&1&0&0\\ 0&0&1&0\\ 0&0&0&1\end{pmatrix} $$    
$$\begin{pmatrix} 1&1&1&0\\ 1&0&1&0\\ 0&1&0&0\\ 0&0&1&0\\ 0&0&0&1\end{pmatrix} \quad \begin{pmatrix} 1&1&0&1\\ 1&1&0&0\\ 0&1&0&0\\ 0&0&1&0\\ 0&0&0&1\end{pmatrix} \quad \begin{pmatrix} 1&1&0&1\\ 1&1&0&0\\ 1&0&0&0\\ 0&0&1&0\\ 0&0&0&1\end{pmatrix} \quad \begin{pmatrix} 1&1&1&0\\ 1&1&0&0\\ 0&1&0&0\\ 0&0&1&0\\ 0&0&0&1\end{pmatrix} $$    
$$\begin{pmatrix} 1&1&1&0\\ 1&1&0&0\\ 1&0&0&0\\ 0&0&1&0\\ 0&0&0&1\end{pmatrix} . $$
\end{lems}

\begin{proof} The monomials corresponding to the above matrices respectively are
  \begin{align*}
&(3,28,1,3), (3,28,3,1), (3,7,24,1), (7,3,24,1), (3,5,25,2),(3,4,25,3),\\
& (3,5,24,3), (3,4,9,19), (3,4,11,17), (3,5,8,19),(3,5,9,18), (3,5,10,17),\\
& (3,5,11,16),  (3,7,8,17),  (7,3,8,17), (3,7,9,16),  (7,3,9,16). 
\end{align*}
We prove the lemma for the matrices associated to the monomials 
\begin{align*}
&(3,28,1,3),  (3,7,24,1), (3,5,25,2),  (3,4,25,3), (3,5,24,3),\\  
&(3,4,9,19),  (3,4,11,17),  (3,5,8,19),  (3,5,10,17), (3,7,9,16). 
\end{align*}
The others can be proved by a similar computation. 
We have
\begin{align*}
&(3,28,1,3) = Sq^1(3,27,1,3) + Sq^2\big((2,27,1,3) + (5,23,2,3)\big)\\ 
&\quad +Sq^4\big((3,23,2,3) + (3,15,8,5)\big)+Sq^8(3,15,4,5)\\
&\quad  + (2,29,1,3) + (2,27,1,5) + (3,27,1,4) + (3,25,4,3)\\ 
&\quad+ (3,25,2,5) + (3,19,4,9)\quad \text{mod  }\mathcal L_4(3;2;1;1;1),\\
&(3,7,24,1) = Sq^1(3,7,23,1) + Sq^2\big((2,7,23,1) + (5,7,19,2)\big)\\ 
&\quad +Sq^4\big((3,4,23,1)+ (3,11,15,2)  + (3,5,21,2) + (3,5,19,4) \\ 
&\quad+ (3,11,13,4)\big)+Sq^8\big((3,7,15,2) + (3,7,13,4)\big)\\ 
&\quad + (2,9,23,1) + (2,7,25,1) + (3,4,27,1) + (3,5,25,2)\\ 
&\quad + (3,5,19,8) + (3,7,17,8)\quad \text{mod  }\mathcal L_4(3;2;1;1;1),\\
&(3,5,25,2) = Sq^1(3,3,27,1) + Sq^2\big((5,3,23,2) + (2,3,27,1)\big) \\ 
&\quad+Sq^4\big((3,3,23,2) + (3,9,15,4)\big)+Sq^8(3,5,15,4)\\ 
&\quad + (3,3,25,4) + (3,3,28,1) + (3,4,27,1) + (2,5,27,1) \\ 
&\quad+ (2,3,29,1) + (3,5,19,8)\quad \text{mod  }\mathcal L_4(3;2;1;1;1),\\
&(3,4,25,3) = Sq^1(3,1,27,3) + Sq^2\big((5,2,23,3) + (2,1,27,3)\big) \\ 
&\quad+Sq^4\big((3,2,23,3) + (3,8,15,5)\big)+Sq^8(3,4,15,5)\\ 
&\quad + (3,2,25,5) + (3,1,28,3) + (3,1,27,4) + (2,1,27,5) \\ 
&\quad+ (2,1,29,3) + (3,4,19,9)\quad \text{mod  }\mathcal L_4(3;2;1;1;1),\\
&(3,5,24,3) = Sq^1(3,3,25,3) + Sq^2\big((5,3,22,3) + (2,3,25,3)\big) \\ 
&\quad+Sq^4\big((3,3,22,3) + (3,9,14,5)\big)+Sq^8(3,5,14,5)\\ 
&\quad + (3,3,24,5) + (3,3,25,4) + (3,4,25,3) + (2,5,25,3) \\ 
&\quad+ (2,3,25,5) + (3,5,18,9)\quad \text{mod  }\mathcal L_4(3;2;1;1;1),\\
&(3,4,9,19) = Sq^1(3,1,3,27) + Sq^2\big((5,2,3,23) + (2,1,3,27)\big) \\ 
&\quad+Sq^4\big((3,8,5,15) + (3,2,3,23)\big)+Sq^8(3,4,5,15)\\ 
&\quad + (3,4,3,25) + (3,2,5,25) + (3,1,3,28) + (3,1,4,27) \\ 
&\quad+ (2,1,5,27) + (2,1,3,29)\quad \text{mod  }\mathcal L_4(3;2;1;1;1),\\
&(3,4,11,17) = Sq^1(3,1,11,19) + Sq^2\big((5,2,7,19) + (2,1,11,19)\big) \\ 
&\quad+Sq^4\big((3,8,7,13) + (3,2,7,19)\big)+Sq^8(3,4,7,13)\\ 
&\quad + (3,4,9,19) + (3,2,9,21) + (3,1,12,19) + (3,1,11,20) \\ 
&\quad+ (2,1,13,19) + (2,1,11,21)\quad \text{mod  }\mathcal L_4(3;2;1;1;1),\\
&(3,5,8,19) = Sq^1(3,3,1,27) + Sq^2\big((5,3,2,23) + (2,3,1,27)\big) \\ 
&\quad+Sq^4\big((3,9,4,15) + (3,3,2,23)\big)+Sq^8(3,5,4,15)\\ 
&\quad + (3,5,2,25) + (3,3,4,25) + (3,3,1,28) + (3,4,1,27) \\ 
&\quad+ (2,5,1,27) + (2,3,1,29)\quad \text{mod  }\mathcal L_4(3;2;1;1;1),
\end{align*}
\begin{align*}
&(3,5,9,18) = Sq^1(3,3,3,25) + Sq^2\big((5,3,3,22) + (2,3,3,25)\big) \\ 
&\quad+Sq^4\big((3,9,5,14) + (3,3,3,22)\big)+Sq^8(3,5,5,14)\\ 
&\quad + (3,5,3,24) + (3,3,5,24) + (3,4,3,25) + (3,3,4,25) \\ 
&\quad+ (2,5,3,25) + (2,3,5,25)\quad \text{mod  }\mathcal L_4(3;2;1;1;1),\\
&(3,5,10,17) = Sq^1(3,3,9,19) + Sq^2\big((5,3,6,19) + (2,3,9,19)\big) \\ 
&\quad+Sq^4\big((3,9,6,13) + (3,3,6,19)\big)+Sq^8(3,5,6,13)\\ 
&\quad + (3,3,8,21) + (3,5,8,19) + (3,3,9,20) + (3,4,9,19) \\ 
&\quad+ (2,5,9,19) + (2,3,9,21)\quad \text{mod  }\mathcal L_4(3;2;1;1;1),\\
&(3,7,8,17) = Sq^1(3,7,1,23) + Sq^2\big((5,7,2,19) + (2,7,1,23)\big)\\ 
&\quad +Sq^4\big((3,11,4,13) + (3,11,2,15) + (3,5,4,19) + (3,5,2,21)\\ 
&\quad + (3,4,1,23)\big)+Sq^8\big((3,7,4,13) + (3,7,2,15)\big)\\ 
&\quad + (3,5,8,19) + (3,5,2,25) + (3,7,1,24) + (2,9,1,23)\\ 
&\quad + (2,7,1,25) + (3,4,1,27)\quad \text{mod  }\mathcal L_4(3;2;1;1;1),\\
&(3,7,9,16) = Sq^1\big((3,7,3,21) + (1,7,5,21) + (3,7,1,23)\big) \\ 
&\quad+ Sq^2\big((5,7,3,18) + (2,7,3,21) + (1,7,6,19) + (1,7,3,22)\\ 
&\quad + (5,7,2,19) + (2,7,1,23)\big) +Sq^4\big((3,11,5,12) + (3,11,3,14)\\ 
&\quad + (3,5,5,18) + (3,5,3,20) + (3,4,3,21) + (3,11,4,13)\big)\\ 
&\quad+Sq^8\big((3,7,5,12) + (3,7,3,14) + (3,7,4,13)\big) + (3,5,9,18)\\ 
&\quad + (3,5,3,24) + (3,4,3,25) + (2,9,3,21) + (1,9,3,22) + (1,7,3,24)\\ 
&\quad + (1,7,8,19) + (1,9,6,19)+ (3,7,8,17) \quad \text{mod  }\mathcal L_4(3;2;1;1;1).
\end{align*}

The lemma is proved.
\end{proof}

Combining Lemmas \ref{5.7}, \ref{8.1.1}, \ref{8.1.2}, \ref{8.2.1},  Theorem \ref{2.4} and the results in Section \ref{7}, we get Proposition \ref{mdc8.2}.

\medskip
Now, we show that the elements listed in Theorem \ref{dlc8.2} are linearly independent.

\begin{props}\label{8.2.3} The elements $[a_{1,1,1,j}], 1 \leqslant j \leqslant 32,$  are linearly independent in $(\mathbb F_2\underset {\mathcal A}\otimes R_4)_{11}$.
\end{props}

\begin{proof} Suppose that there is a linear relation
\begin{equation}\sum_{j=1}^{32}\gamma_j[a_{1,1,1,j}] = 0, \tag {\ref{8.2.3}.1}
\end{equation}
with $\gamma_j \in \mathbb F_2$.

Applying the homomorphisms $f_1, f_2, f_3$ to the relation (\ref{8.2.3}.1), we obtain
\begin{align*}
&\gamma_{3}[3,1,7] +  \gamma_{4}[3,7,1] +   \gamma_{\{13, 31\}}[7,1,3]\\
&\quad + \gamma_{\{14, 32\}}[7,3,1] +  \gamma_{17}[3,3,5] +  \gamma_{18}[3,5,3] = 0,\\ 
&\gamma_{1}[3,1,7] +  \gamma_{6}[3,7,1] +  \gamma_{\{10, 28\}}[7,1,3]\\
&\quad +  \gamma_{\{12, 30\}}[7,3,1] + \gamma_{\{19, 26\}}[3,3,5] +  \gamma_{26}[3,5,3] =0,\\  
&\gamma_{2}[3,1,7] + \gamma_{5}[3,7,1] +  \gamma_{\{9, 27\}}[7,1,3]\\
&\quad +  \gamma_{\{11, 29\}}[7,3,1] + \gamma_{\{20, 25\}}[3,3,5] +  \gamma_{25}[3,5,3] =0.
\end{align*}

Computing from these equalities, we obtain
\begin{equation}\begin{cases}
\gamma_i = 0,\ i= 1,2,3,4,5,6,17,18,19, 20, 25,26,\\
\gamma_{\{13, 31\}} = \gamma_{\{14, 32\}}=  \gamma_{\{10, 28\}}= 0,\\
 \gamma_{\{12, 30\}}= \gamma_{\{9, 27\}}= \gamma_{\{11, 29\}}= 0.    
\end{cases}\tag{\ref{8.2.3}.2}
\end{equation}

With the aid of (\ref{8.2.3}.2), the homomorphisms $f_4, f_5,f_6$ send (\ref{8.2.3}.1) to
\begin{align*}
&\gamma_{\{10, 13\}}[1,7,3] +  \gamma_{\{16, 24, 30, 32\}}[3,7,1]\\
&\quad +   \gamma_{8}[7,3,1] +  \gamma_{21}[3,3,5] +  \gamma_{\{28, 31\}}[3,5,3] =0,\\  
&\gamma_{\{9, 14\}}[1,7,3] + \gamma_{\{15, 23, 29, 31\}}[3,7,1]\\
&\quad +  \gamma_{7}[7,3,1] +  \gamma_{22}[3,3,5] +  \gamma_{\{27, 32\}}[3,5,3] =0,\\  
&\gamma_{\{11, 12\}}[1,3,7] +  \gamma_{\{15, 16, 21, 22, 27, 28\}}[3,1,7]\\
&\quad +  \gamma_{\{7, 8\}}[7,1,3] +  \gamma_{\{29, 30\}}[3,3,5] +  \gamma_{\{23, 24\}}[3,5,3] =0.   
\end{align*}
From these equalities, we get
\begin{equation}\begin{cases}
\gamma_j = 0,\ i= 7, 8, 21, 22,\\
\gamma_{\{10, 13\}}=  
\gamma_{\{16, 24, 30, 32\}}=   
\gamma_{\{28, 31\}}=  
\gamma_{\{9, 14\}}=  0,\\
\gamma_{\{15, 23, 29, 31\}}=  
\gamma_{\{27, 32\}}=  
\gamma_{\{11, 12\}}= 0,\\
\gamma_{\{15, 16, 27, 28\}}=  
\gamma_{\{29, 30\}}=  
\gamma_{\{23, 24\}}=0.     
\end{cases}\tag{\ref{8.2.3}.3}
\end{equation}

With the aid of (\ref{8.2.3}.2) and (\ref{8.2.3}.3), the homomorphisms $g_1, g_2$ send (\ref{8.2.3}.1) to
\begin{align*}
&\gamma_{16}[3,7,1] + \gamma_{23}[7,3,1] +   \gamma_{10}[3,3,5] =0,\\  &\gamma_{15}[3,7,1] +  \gamma_{23}[7,3,1] +  \gamma_{9}[3,3,5] =0.   
\end{align*}

These equalities imply
\begin{equation} \gamma_j = 0, j = 9, 10, 15, 16, 23. \tag{\ref{8.2.3}.4}
\end{equation}

Combining (\ref{8.2.3}.2), (\ref{8.2.3}.3) and (\ref{8.2.3}.4), we get $\gamma_j = 0$ for $1 \leqslant j \leqslant 32$. The proposition is proved.
\end{proof}

\begin{props}\label{8.2.4} The elements $[a_{2,1,1,j}], 1 \leqslant j \leqslant 80,$  are linearly independent in $(\mathbb F_2\underset {\mathcal A}\otimes R_4)_{19}$.
\end{props}

\begin{proof} Suppose that there is a linear relation
\begin{equation}\sum_{j=1}^{80}\gamma_j[a_{2,1,1,j}] = 0, \tag {\ref{8.2.4}.1}
\end{equation}
with $\gamma_j \in \mathbb F_2$.

Applying the homomorphisms $f_1, f_2, f_3$ to the relation (\ref{8.2.4}.1), we obtain
\begin{align*}
&\gamma_{3}[3,1,15] + \gamma_{4}[3,15,1] +   \gamma_{\{13, 65, 76\}}[15,1,3] +  \gamma_{\{14, 66, 77\}}[15,3,1]\\
&\quad +  \gamma_{17}[3,3,13] +  \gamma_{18}[3,13,3] +  \gamma_{\{32, 42\}}[7,1,11] +  \gamma_{\{33, 43\}}[7,11,1]\\
&\quad +  \gamma_{39}[3,5,11] +  \gamma_{44}[3,7,9] +  \gamma_{\{52, 62\}}[7,3,9] +  \gamma_{\{71, 79\}}[7,9,3] = 0,\\  
&\gamma_{1}[3,1,15] + \gamma_{6}[3,15,1] +  \gamma_{\{10, 28, 74\}}[15,1,3] +  \gamma_{\{12, 30, 75\}}[15,3,1]\\
&\quad +  \gamma_{\{19, 40\}}[3,3,13] +  \gamma_{72}[3,13,3] +  \gamma_{\{31, 41\}}[7,1,11] +  \gamma_{\{36, 73\}}[7,11,1]\\
&\quad +  \gamma_{40}[3,5,11] +  \gamma_{\{45, 72\}}[3,7,9] +  \gamma_{\{26, 51, 61\}}[7,3,9] +  \gamma_{\{26, 80\}}[7,9,3] = 0,\\  
&\gamma_{2}[3,1,15] + \gamma_{5}[3,15,1] +  \gamma_{\{9, 27, 58\}}[15,1,3] +  \gamma_{\{11, 29, 59\}}[15,3,1]\\
&\quad +  \gamma_{\{20, 54\}}[3,3,13] +  \gamma_{55}[3,13,3] +  \gamma_{\{34, 56\}}[7,1,11] +  \gamma_{\{35, 57\}}[7,11,1]\\
&\quad + \gamma_{54}[3,5,11] +  \gamma_{\{55, 67\}}[3,7,9] +   \gamma_{\{25, 47, 63\}}[7,3,9] +  \gamma_{\{25, 64\}}[7,9,3] = 0. 
\end{align*}

Computing from these equalities, we obtain
\begin{equation}\begin{cases}
\gamma_i = 0,\ i= 1, 2, 3, 4, 5, 6, 17, 18, \\
19, 20, 39, 40, 44, 45, 54, 55, 67, 72, \\
\gamma_{\{13, 65, 76\}}=  
\gamma_{\{14, 66, 77\}}=   
\gamma_{\{32, 42\}}=  
\gamma_{\{33, 43\}}= 
\gamma_{\{52, 62\}}=  0,\\
\gamma_{\{71, 79\}}= 
\gamma_{\{10, 28, 74\}}= 
 \gamma_{\{12, 30, 75\}}=  
\gamma_{\{31, 41\}}=  
\gamma_{\{36, 73\}}= 0,\\
\gamma_{\{26, 51, 61\}}=  
\gamma_{\{26, 80\}}=  
\gamma_{\{9, 27, 58\}}=  
\gamma_{\{11, 29, 59\}}= 0,\\
\gamma_{\{34, 56\}}=  
\gamma_{\{35, 57\}}=  
\gamma_{\{25, 47, 63\}}=  
\gamma_{\{25, 64\}}=0.  
\end{cases}\tag{\ref{8.2.4}.2}
\end{equation}

With the aid of (\ref{8.2.4}.2), the homomorphisms $f_4, f_5,f_6$ send (\ref{8.2.4}.1) to
\begin{align*}
&\gamma_{\{10, 13, 26, 71\}}[1,15,3] + a_1[3,15,1] +   \gamma_{8}[15,3,1] +  \gamma_{21}[3,3,13]\\
&\quad +  a_2[3,13,3] +  \gamma_{\{31, 32\}}[1,7,11] +  \gamma_{\{38, 70, 75, 77\}}[7,11,1] +  \gamma_{\{41, 42\}}[3,5,11]\\
&\quad +  \gamma_{\{53, 60, 61, 62\}}[3,7,9] +  \gamma_{46}[7,3,9] +  \gamma_{\{74, 76\}}[7,9,3] = 0,\\  
&\gamma_{\{9, 14, 25, 52\}}[1,15,3] + a_3[3,15,1] +  \gamma_{7}[15,3,1] +  \gamma_{22}[3,3,13] +  \gamma_{68}[7,3,9]\\
&\quad +  \gamma_{\{27, 62, 64, 66\}}[3,13,3] +  \gamma_{\{33, 34\}}[1,7,11] +  \gamma_{\{37, 59, 69, 76\}}[7,11,1]\\
&\quad +  \gamma_{\{43, 56\}}[3,5,11] +  \gamma_{\{48, 63, 78, 79\}}[3,7,9] +  \gamma_{\{58, 77\}}[7,9,3] = 0,\\ 
&a_4[1,3,15] + a_5[3,1,15] +  \gamma_{\{7, 8\}}[15,1,3] +  \gamma_{\{29, 30, 61, 63\}}[3,3,13]\\
&\quad +  \gamma_{\{23, 24\}}[3,13,3] +  \gamma_{\{35, 36\}}[1,7,11] +  a_6[7,1,11] +  a_7[3,5,11]\\
&\quad +  \gamma_{\{57, 73\}}[3,7,9] +  \gamma_{\{59, 75\}}[7,3,9] + \gamma_{\{69, 70\}}[7,9,3] = 0. 
\end{align*}
where
\begin{align*}
a_1 &= \gamma_{\{16, 24, 30, 43, 50, 66, 73\}},\ \
a_2 = \gamma_{\{28, 65, 79, 80\}},\ \
a_3 = \gamma_{\{15, 23, 29, 42, 49, 57, 65\}},\\
a_4 &= \gamma_{\{11, 12, 25, 26, 47, 51\}},\ \
a_6 = \gamma_{\{37, 38, 46, 58, 68, 74\}},\ \
a_7 = \gamma_{\{49, 50, 60, 64, 78, 80\}},\\
a_5 &= \gamma_{\{15, 16, 21, 22, 27, 28, 41, 48, 53, 56\}}.
\end{align*}

From these equalities, we get
\begin{equation}
\begin{cases}
a_i = 0,\ i = 1,2,3,4,5,6,7,\\
\gamma_j = 0,\ j= 7, 8, 21, 22, 46, 68,\\
\gamma_{\{10, 13, 26, 71\}} =
\gamma_{\{31, 32\}} =
\gamma_{\{38, 58, 59, 69\}} =
\gamma_{\{52, 53, 60, 61\}} =0,\\
\gamma_{\{74, 76\}} =
\gamma_{\{9, 14, 25, 52\}} =
\gamma_{\{25, 27, 52, 66\}} =
\gamma_{\{33, 34\}} = 0,\\
\gamma_{\{37, 59, 69, 74\}} =
\gamma_{\{48, 63, 71, 78\}} =
\gamma_{\{58, 77\}} =
\gamma_{\{23, 24\}} = 0,\\
\gamma_{\{29, 30, 61, 63\}} =
\gamma_{\{35, 36\}} =
\gamma_{\{59, 75\}} =
\gamma_{\{69, 70\}} = 0.        
\end{cases}\tag{\ref{8.2.4}.3}
\end{equation}

With the aid of (\ref{8.2.4}.2) and (\ref{8.2.4}.3), the homomorphisms $g_1, g_2$ send (\ref{8.2.4}.1) respectively to
\begin{align*}
&\gamma_{\{10, 28, 74\}}[1,15,3] +  a_8[3,15,1] +   \gamma_{\{13, 65, 74\}}[15,1,3] +  a_9[15,3,1]\\
&\quad +  \gamma_{31}[3,3,13] +  \gamma_{\{28, 71, 74\}}[3,13,3] +  a_{10}[7,11,1]\\
&\quad +  \gamma_{\{26, 51, 53, 61, 71\}}[3,7,9] +  \gamma_{\{60, 65, 74\}}[7,3,9] +  \gamma_{\{26, 65, 74\}}[7,9,3] = 0,\\  
&\gamma_{\{9, 27, 58\}}[1,15,3] + a_{11}[3,15,1] +  \gamma_{\{14, 58, 66\}}[15,1,3]\\
&\quad +  a_{12}[15,3,1] +  \gamma_{33}[3,3,13] +  \gamma_{\{27, 52, 58\}}[3,13,3] +  a_{13}[7,11,1]\\
&\quad +  a_{14}[3,7,9] +  \gamma_{\{58, 66, 78\}}[7,3,9] +  \gamma_{\{25, 58, 66\}}[7,9,3] = 0.             
\end{align*}
where
\begin{align*}
a_8 &= \gamma_{\{12, 16, 30, 38, 59\}},\ \
a_9 = \gamma_{\{14, 23, 58, 66, 69\}},\ \
a_{10} = \gamma_{\{38, 50, 58, 59, 69\}},\\
a_{11} &= \gamma_{\{11, 15, 29, 37, 59\}},\ \
a_{12} = \gamma_{\{13, 23, 65, 69, 74\}},\ \
a_{13} = \gamma_{\{37, 49, 59, 69, 74\}},\\
a_{14} &= \gamma_{\{25, 47, 48, 52, 63\}}
\end{align*}
From these  equalities, we obtain
\begin{equation}\begin{cases}
\gamma_{31} = \gamma_{33} = a_j = 0, \ j = 8,9,\ldots , 14,\\
\gamma_{\{10, 28, 74\}} =
\gamma_{\{13, 65, 74\}} =
\gamma_{\{28, 71, 74\}} = 
\gamma_{\{60, 65, 74\}} = 0,\\
\gamma_{\{26, 51, 53, 61, 71\}} =
\gamma_{\{26, 65, 74\}} =
\gamma_{\{9, 27, 58\}} = 0,\\
\gamma_{\{14, 58, 66\}} =
\gamma_{\{27, 52, 58\}} =
\gamma_{\{58, 66, 78\}} =
\gamma_{\{25, 58, 66\}} = 0. 
\end{cases}\tag{\ref{8.2.4}.4}
\end{equation}

 With the aid of (\ref{8.2.4}.2), (\ref{8.2.4}.3) and (\ref{8.2.4}.4), the homomorphisms $g_3, g_4$ send (\ref{8.2.4}.1) to 
\begin{align*}
&\gamma_{\{11, 29, 59\}}[1,15,3] +  a_{15}[3,15,1] +   \gamma_{\{12, 30, 59\}}[15,1,3]\\
&\quad +  a_{16}[15,3,1] +  \gamma_{35}[3,3,13] +  a_{17}[3,13,3] +  a_{18}[7,11,1]\\
&\quad +  a_{19}[3,7,9] +  \gamma_{\{30, 50, 59, 78\}}[7,3,9] +  a_{20}[7,9,3] = 0,\\  
\end{align*}
\begin{align*}
&a_{21}[1,15,3] + a_{22}[3,15,1] +  a_{23}[15,1,3] +  a_{24}[15,3,1] +  \gamma_{\{59, 69\}}[3,3,13]\\
&\quad +  a_{25}[3,13,3] +  \gamma_{\{37, 59, 69, 74\}}[1,7,11] +  \gamma_{\{38, 58, 59, 69\}}[7,1,11]\\
&\quad +  \gamma_{\{47, 51, 52, 71\}}[7,11,1] +  a_{26}[3,7,9] +  a_{27}[7,3,9] +  a_{28}[7,9,3] = 0.           
\end{align*}
where
\begin{align*}
a_{15} &= \gamma_{\{9, 15, 27, 37, 58\}},\ \ 
a_{16} = \gamma_{\{10, 16, 28, 38, 74\}},\ \ 
a_{17} = \gamma_{\{25, 29, 47, 59, 61, 63\}},\\ 
a_{18} &= \gamma_{\{37, 38, 48, 53, 58, 74\}},\ \ 
a_{19} = \gamma_{\{25, 47, 49, 60, 61, 63\}},\ \ 
a_{20} = \gamma_{\{26, 30, 51, 59, 61, 63\}},\\ 
a_{21} &= \gamma_{\{15, 23, 29, 35, 49, 65\}},\ \ 
a_{22} = \gamma_{\{9, 11, 13, 14, 25, 35, 52\}},\ \ 
a_{23} = \gamma_{\{16, 23, 30, 35, 50, 66\}},\\
a_{24} &= \gamma_{\{10, 12, 13, 14, 26, 35, 71\}},\ \
a_{25} = \gamma_{\{23, 25, 26, 27, 29, 35, 48, 49, 52, 60, 61, 63, 66, 71, 78\}},\\
a_{26} &= \gamma_{\{26, 48, 49, 61, 63, 71, 78\}},\ \
a_{27} = \gamma_{\{23, 26, 28, 30, 35, 65, 71, 78\}},\\
a_{28} &= \gamma_{\{23, 25, 26, 28, 30, 35, 50, 52, 53, 60, 61, 63, 65, 71, 78\}}.
\end{align*}

So, we obtain
\begin{equation}\begin{cases}
\gamma_{35} = a_j = 0, \ j = 15,16,\ldots , 28,\\
\gamma_{\{11, 29, 59\}} =
\gamma_{\{12, 30, 59\}} =
\gamma_{\{30, 50, 59, 78\}}= 
\gamma_{\{59, 69\}} = 0,\\
\gamma_{\{37, 59, 69, 74\}} =
\gamma_{\{38, 58, 59, 69\}} =
\gamma_{\{47, 51, 52, 71\}} = 0.
\end{cases}\tag{\ref{8.2.4}.5}
\end{equation}
Substituting (\ref{8.2.4}.2), (\ref{8.2.4}.3), (\ref{8.2.4}.4), and (\ref{8.2.4}.5) into the relation (\ref{8.2.4}.1), we have
\begin{equation} 
\gamma_{15}[\theta_1] + \gamma_{16}[\theta_2] + \gamma_{12}[\theta_3] + \gamma_{11}[\theta_4] + \gamma_{9}[\theta_5]  + \gamma_{23}[\theta_6]= 0,\tag{\ref{8.2.4}.6}
\end{equation}
where 
\begin{align*}
\theta_1 &= a_{2,1,1,15} + a_{2,1,1,28} + a_{2,1,1,37} + a_{2,1,1,65} + a_{2,1,1,74} + a_{2,1,1,76},\\
\theta_2 &= a_{2,1,1,16} + a_{2,1,1,27} + a_{2,1,1,38} + a_{2,1,1,58} + a_{2,1,1,66} + a_{2,1,1,77},\\
\theta_3 &= a_{2,1,1,12} + a_{2,1,1,14} + a_{2,1,1,25} + a_{2,1,1,30}\\
&\quad + a_{2,1,1,63} + a_{2,1,1,64} + a_{2,1,1,66} + a_{2,1,1,78},\\
\theta_4 &= a_{2,1,1,11} + a_{2,1,1,13} + a_{2,1,1,26} + a_{2,1,1,29}\\
&\quad + a_{2,1,1,60} + a_{2,1,1,61} + a_{2,1,1,65} + a_{2,1,1,80},\\
\theta_5 &= a_{2,1,1,9} + a_{2,1,1,10} + a_{2,1,1,27} + a_{2,1,1,28}+ a_{2,1,1,48}\\
&\quad  + a_{2,1,1,52} + a_{2,1,1,53} + a_{2,1,1,62} + a_{2,1,1,71} + a_{2,1,1,79},\\
\theta_6 &= a_{2,1,1,23} + a_{2,1,1,24} + a_{2,1,1,29} + a_{2,1,1,30}\\
&\quad + a_{2,1,1,59} + a_{2,1,1,69} + a_{2,1,1,70} + a_{2,1,1,75}.
\end{align*}

We need to prove $\gamma_{9} = \gamma_{11} = \gamma_{12} = \gamma_{15} =\gamma_{16} = \gamma_{23} = 0.$ The proof is divided into 4 steps.

{\it Step 1.} First we prove $\gamma_{23} = 0$ by showing the polynomial $\theta = \beta_1\theta_1 + \beta_2\theta_2 + \beta_3\theta_3 + \beta_4\theta_4 + \beta_5\theta_5 + \theta_6$ is non-hit for all $\beta_1, \beta_2, \beta_3, \beta_4, \beta_5 \in \mathbb F_2$. Suppose the contrary that this polynomial is hit. Then we have
$$\theta = Sq^1(A) + Sq^2(B) + Sq^4(C) + Sq^8(D),$$
for some polynomials $A, B, C, D$ in $R_4$. Let $(Sq^2)^3$ act on the both sides of this equality. Using the relations $(Sq^2)^3Sq^1 =0, (Sq^2)^3Sq^2 =0$, we get
$$(Sq^2)^3(\theta) = (Sq^2)^3Sq^4(C) + (Sq^2)^3Sq^8(D).$$
The polynomial $(7,12,4,2)+(7,12,2,4)$ is a summand of $(Sq^2)^3(\theta)$. This polynomial is not a summand of $(Sq^2)^3Sq^8(D)$ for all $D$. It is also not a summand of $(Sq^2)^3Sq^4(C)$ if either $(7,5,1,2) + (7,5,2,1)$ or $(7,6,1,1)$ is not a summand of $C$.

Suppose $(7,5,1,2) + (7,5,2,1)$ is a summand of $C$. Then 
$$(Sq^2)^3(\theta + Sq^4((7,5,1,2)+(7,5,2,1))) = (Sq^2)^3(Sq^4(C') + Sq^8(D)),$$
where $C' = C + (7,5,1,2) + (7,5,2,1)$. We see that the polynomial $(16,6,1,2) + (16,6,2,1)$ is a summand of $(Sq^2)^3(\theta+Sq^4((7,5,1,2)+(7,5,2,1)))$. This polynomial is not a summand of $(Sq^2)^3Sq^8(D)$ for all $D$. It is also not a summand of $(Sq^2)^3Sq^4(C')$ for $C' \ne (7,5,1,2) + (7,5,2,1), (7,6,1,1)$. So, either $(7,5,1,2) + (7,5,2,1)$ or $(7,6,1,1)$ is a summand of $C'$. Since $(7,5,1,2) + (7,5,2,1)$ is not a summand of $C'$, $(7,6,1,1)$ is a summand of $C'$. Hence, we obtain
\begin{align*}(Sq^2)^3(\theta + Sq^4((7,5,1,2)&+(7,5,2,1)+(7,6,1,1)))\\
& = (Sq^2)^3(Sq^4(C'') + Sq^8(D)),
\end{align*}
where $C'' = C + (7,5,1,2) + (7,5,2,1) + (7,6,1,1)$. Now $(7,12,2,4) + (7,12,4,2)$ is a summand of 
$$(Sq^2)^3(\theta + Sq^4((7,5,1,2)+(7,5,2,1)+(7,6,1,1))).$$
 So, either $(7,5,1,2)+ (7,5,2,1)$ or $(7,6,1,1)$ is a summand of $C''$ and we have a contradiction. Hence, $[\theta] \ne 0$ and $\gamma_{23} = 0$. 
 
{\it Step 2.} By a direct computation, we see that the homomorphism $\varphi_2$ sends (\ref{8.2.4}.6) to
$\gamma_{15}[\theta_1] + \gamma_{16}[\theta_6] + \gamma_{12}[\theta_3] + \gamma_{11}[\theta_5] + \gamma_{9}[\theta_4]= 0.
$
By Step 1, we obtain $\gamma_{16} =0$.

{\it Step 3.} The homomorphism $\varphi_2\varphi_1$ sends (\ref{8.2.4}.6) to
$
\gamma_{15}[\theta_5] + \gamma_{12}[\theta_6] + \gamma_{11}[\theta_1] + \gamma_{9}[\theta_4] = 0.
$ By Step 1, we obtain $\gamma_{12} =0$.

{\it Step 4.} Now the homomorphism $\varphi_3$ sends (\ref{8.2.4}.6) to
$\gamma_{15}[\theta_2] + \gamma_{11}[\theta_3] + \gamma_{9}[\theta_5] = 0.
$ Combining Step 2 and Step 3, we obtain $\gamma_{11} = \gamma_{15} = 0$.

Since $\varphi_2([\theta_5]) = [\theta_2]$, we get $\gamma_9 = 0.$
So, we obtain $\gamma_j = 0$ for all $j$. The proposition follows.
\end{proof}

\begin{props}\label{8.2.5} For $u \geqslant 3$, the elements $[a_{u,1,1,j}], 1 \leqslant j \leqslant 64,$  are linearly independent in $(\mathbb F_2\underset {\mathcal A}\otimes R_4)_{2^{u+2}+3}$.
\end{props}

\begin{proof} Suppose that there is a linear relation
\begin{equation}\sum_{j=1}^{64}\gamma_j[a_{u,1,1,j}] = 0, \tag {\ref{8.2.5}.1}
\end{equation}
with $\gamma_j \in \mathbb F_2$.

Applying the homomorphisms $f_1, f_2, f_3$ to the relation (\ref{8.2.5}.1), we obtain
\begin{align*}
&\gamma_{3}w_{u,1,1,3}  +  \gamma_{4}w_{u,1,1,4}  +    \gamma_{13}w_{u,1,1,5}  +   \gamma_{14}w_{u,1,1,6} \\
&\quad +   \gamma_{\{32, 42\}}w_{u,1,1,8}  +   \gamma_{\{33, 43\}}w_{u,1,1,9}  +   \gamma_{17}w_{u,1,1,10}  +   \gamma_{18}w_{u,1,1,11}\\
&\quad  +   \gamma_{39}w_{u,1,1,12}  +   \gamma_{44}w_{u,1,1,13}  +   \gamma_{\{52, 62\}}w_{u,1,1,14}  = 0,\\   
&\gamma_{1}w_{u,1,1,3}  +   \gamma_{6}w_{u,1,1,4}  +   \gamma_{\{10, 28\}}w_{u,1,1,5}  +   \gamma_{\{12, 30\}}w_{u,1,1,6}\\
&\quad  +   \gamma_{\{31, 41\}}w_{u,1,1,8}  +   \gamma_{36}w_{u,1,1,9}  +   \gamma_{\{19, 26, 40\}}w_{u,1,1,10}\\
&\quad  +   \gamma_{26}w_{u,1,1,11}  +   \gamma_{40}w_{u,1,1,12}  +   \gamma_{45}w_{u,1,1,13}  +   \gamma_{\{51, 61\}}w_{u,1,1,14}  = 0,\\   
&\gamma_{2}w_{u,1,1,3}  +  \gamma_{5}w_{u,1,1,4}  +   \gamma_{\{9, 27, 58\}}w_{u,1,1,5} +  \gamma_{\{11, 29, 59\}}w_{u,1,1,6}\\
&\quad   +   \gamma_{\{34, 56\}}w_{u,1,1,8}  +   \gamma_{\{35, 57\}}w_{u,1,1,9}   +   \gamma_{\{20, 25, 54, 64\}}w_{u,1,1,10} \\
&\quad+   \gamma_{\{25, 55, 64\}}w_{u,1,1,11} +  \gamma_{54}w_{u,1,1,12} +   \gamma_{55}w_{u,1,1,13}  +    \gamma_{\{47, 63, 64\}}w_{u,1,1,14} = 0.   
\end{align*}

Computing from these equalities, we get
\begin{equation}\begin{cases}
\gamma_j = 0, \ j= 1, 2, 3, 4, 5, 6, 13, 14, 17, 18,\\
\hskip2.2cm 19, 26, 36, 39, 40, 44, 45, 54, 55,\\
\gamma_{\{32, 42\}} =   
\gamma_{\{33, 43\}} =    
\gamma_{\{52, 62\}} =   
\gamma_{\{10, 28\}} =  
\gamma_{\{12, 30\}} = 0,\\   
\gamma_{\{31, 41\}} =   
\gamma_{\{51, 61\}} =   
\gamma_{\{9, 27, 58\}} =   
\gamma_{\{11, 29, 59\}} =  0,\\ 
\gamma_{\{34, 56\}} =   
\gamma_{\{35, 57\}} =   
\gamma_{\{20, 25, 64\}} =   
\gamma_{\{25, 64\}} =   
\gamma_{\{47, 63, 64\}} = 0. 
\end{cases}\tag{\ref{8.2.5}.2}
\end{equation}

With the aid of (\ref{8.2.5}.2), the homomorphisms $f_4, f_5,f_6$ send (\ref{8.2.5}.1) to
\begin{align*}
&\gamma_{10}w_{u,1,1,2}  +  \gamma_{\{16, 24, 30, 43, 50\}}w_{u,1,1,4}  +   \gamma_{8}w_{u,1,1,6}\\
&\quad  +   \gamma_{\{31, 32\}}w_{u,1,1,7}  +   \gamma_{38}w_{u,1,1,9}  +   \gamma_{21}w_{u,1,1,10}  +   \gamma_{28}w_{u,1,1,11} \\
&\quad +  \gamma_{\{41, 42\}}w_{u,1,1,12}  +   \gamma_{\{53, 60, 61, 62\}}w_{u,1,1,13}  +   \gamma_{46}w_{u,1,1,14}  = 0,\\   
&\gamma_{\{9, 25, 52\}}w_{u,1,1,2}  +  \gamma_{\{15, 23, 29, 42, 49, 57\}}w_{u,1,1,4}  +   \gamma_{7}w_{u,1,1,6}  +   \gamma_{\{33, 34\}}w_{u,1,1,7}\\
&\quad  +   \gamma_{\{37, 59\}}w_{u,1,1,9}  +   \gamma_{\{22, 58\}}w_{u,1,1,10}  +   \gamma_{\{27, 58, 62, 64\}}w_{u,1,1,11}\\
&\quad  +   \gamma_{\{43, 56\}}w_{u,1,1,12}  +   \gamma_{\{48, 63\}}w_{u,1,1,13}  +   \gamma_{58}w_{u,1,1,14}  = 0,\\   
& \gamma_{\{11, 12, 20, 25, 47, 51\}}w_{u,1,1,1}  +  \gamma_{\{15, 16, 21, 22, 27, 28, 41, 48, 53, 56\}}w_{u,1,1,3} +   \gamma_{\{7, 8\}}w_{u,1,1,5}\\
&\quad  +   \gamma_{35}w_{u,1,1,7}  +   \gamma_{\{37, 38, 46, 58\}}w_{u,1,1,8}  +   \gamma_{\{29, 30, 61, 63\}}w_{u,1,1,10}\\
&\quad  +   \gamma_{\{23, 24\}}w_{u,1,1,11}  +  \gamma_{\{49, 50, 60, 64\}}w_{u,1,1,12}   +   \gamma_{57}w_{u,1,1,13}  +    \gamma_{59}w_{u,1,1,14}  = 0. 
\end{align*}

Computing from these equalities and (\ref{8.2.5}.2), we get
\begin{equation}\begin{cases}
 \gamma_j = 0, \ j= 7, 8, 10, 20, 21, 22, 28, 35, 37, 38, 46, 57, 58, 59,\\
\gamma_{\{16, 24, 30, 43, 50\}} =   
\gamma_{\{31, 32\}} =    
\gamma_{\{41, 42\}} =   
\gamma_{\{53, 60, 61, 62\}} = 0,\\   
\gamma_{\{9, 25, 52\}} =   
\gamma_{\{15, 23, 29, 42, 49\}} =   
\gamma_{\{33, 34\}} =   
\gamma_{\{27, 62, 64\}} =  0,\\
\gamma_{\{43, 56\}} =   
\gamma_{\{48, 63\}} =   
\gamma_{\{29, 30, 61, 63\}} = 
\gamma_{\{49, 50, 60, 64\}} = 0,\\
\gamma_{\{23, 24\}} =   
\gamma_{\{11, 12, 20, 25, 47, 51\}}=
\gamma_{\{15, 16, 21, 22, 27, 28, 41, 48, 53, 56\}}=0. 
\end{cases}\tag{\ref{8.2.5}.3}
\end{equation}

With the aid of (\ref{8.2.5}.2) and (\ref{8.2.5}.3), the homomorphisms $g_1, g_2, g_3$ send (\ref{8.2.5}.1) to
\begin{align*}
&\gamma_{16}w_{u,1,1,4}  +   \gamma_{23}w_{u,1,1,6}  +    \gamma_{50}w_{u,1,1,9}\\
&\quad  +   \gamma_{31}w_{u,1,1,10}  +   \gamma_{53}w_{u,1,1,13}  +   \gamma_{60}w_{u,1,1,14}  = 0,\\   
&\gamma_{15}w_{u,1,1,4}  +   \gamma_{23}w_{u,1,1,6}  +   \gamma_{49}w_{u,1,1,9}  +   \gamma_{\{25, 33\}}w_{u,1,1,10}\\
&\quad  +   \gamma_{\{9, 25, 52\}}w_{u,1,1,11}  +  \gamma_{\{25, 47, 52\}}w_{u,1,1,13}  +   \gamma_{25}w_{u,1,1,14}  = 0,\\   
&\gamma_{15}w_{u,1,1,4}  +  \gamma_{16}w_{u,1,1,6}  +   \gamma_{\{48, 53\}}w_{u,1,1,9}  +   \gamma_{\{11, 12, 25, 47, 51\}}w_{u,1,1,11} \\
&\quad +   \gamma_{\{12, 48\}}w_{u,1,1,10}  +   \gamma_{\{25, 47, 48, 49, 51, 60\}}w_{u,1,1,13}  +   \gamma_{\{48, 50\}}w_{u,1,1,14}  = 0.  
\end{align*}

Computing directly from these equalities, we get
\begin{equation}\begin{cases}
\gamma_j = 0,\ j = 12, 15, 16, 23, 25, 31, 33, 48, 49, 50, 53, 60,\\
\gamma_{\{9, 52\}}= \gamma_{\{47, 52\}}=  \gamma_{\{11, 47, 51\}}= \gamma_{\{47, 51\}}=0.
\end{cases}\tag{\ref{8.2.5}.4}
\end{equation}

The relations (\ref{8.2.5}.2), (\ref{8.2.5}.3) and (\ref{8.2.5}.4) imply $\gamma_j = 0$ for $1 \leqslant j \leqslant 64$. The proposition is proved.
\end{proof}

\subsection{The case $s=1, t\geqslant 2 $ and $u=1$}\label{8.3}\ 

\medskip
It is well known that for $t\geqslant 2$, $\dim (\mathbb F_2\underset{\mathcal A}\otimes P_3)_{2^{t+2}+ 2^{t+1} -1} = 14$ with a basis given by the classes $w_{1,t,1,j}, 1\leqslant j \leqslant 14,$ which are determined as follows:

For $t \geqslant 2$,

\smallskip
\centerline{\begin{tabular}{ll}
$1.\  [1,2^{t + 1} - 1,2^{t + 2} - 1],$& $2.\  [1,2^{t + 2} - 1,2^{t + 1} - 1],$\cr 
$3.\  [2^{t+1}-1,1,2^{t + 2} - 1],$& $4.\  [2^{t+1}-1,2^{t + 2} - 1,1],$\cr 
$5.\  [2^{t+2}-1,1,2^{t + 1} - 1],$& $6.\  [2^{t+2}-1,2^{t + 1} - 1,1],$\cr 
$7.\  [3,2^{t + 1} - 3,2^{t + 2} - 1],$& $8.\  [3,2^{t + 2} - 1,2^{t + 1} - 3],$\cr 
$9.\  [2^{t+2}-1,3,2^{t + 1} - 3],$& $10.\  [3,2^{t + 1} - 1,2^{t + 2} - 3],$\cr 
$11.\  [3,2^{t + 2} - 3,2^{t + 1} - 1],$& $12.\  [2^{t+1}-1,3,2^{t + 2} - 3],$\cr 
$13.\  [7,2^{t + 2} - 5,2^{t + 1} - 3].$&\cr
\end{tabular}}

\smallskip
For $t=2$, $ w_{1,2,1,14} = [7,7,9].$

For $t \geqslant 3,$
$w_{1,t,1,14} =  [7,2^{t + 1} - 5,2^{t + 2} - 3] .$

\smallskip
So, we easily obtain

\begin{props}\label{8.3.3} For any  integer $t \geqslant 2$, 
$$\dim (\mathbb F_2\underset{\mathcal A}\otimes Q_4)_{2^{t+2} + 2^{t+1} -1} = 56.$$
\end{props}

Now, we determine $(\mathbb F_2\underset{\mathcal A}\otimes R_4)_{2^{t+2}+ 2^{t+1} -1}$. We have

\begin{thms}\label{dlc8.3} $(\mathbb F_2\underset{\mathcal A}\otimes R_4)_{2^{t+2}+2^{t+1}-1}$ is  an $\mathbb F_2$-vector space with a basis consisting of all the classes represented by the monomials $a_{1,t,1,j}, j \geqslant 1$, which are determined as follows:

\smallskip
For $t \geqslant 2$,

\medskip
\centerline{\begin{tabular}{ll}
$1.\  (1,1,2^{t + 1} - 2,2^{t + 2} - 1),$& $2.\  (1,1,2^{t + 2} - 1,2^{t + 1} - 2),$\cr 
$3.\  (1,2^{t + 1} - 2,1,2^{t + 2} - 1),$& $4.\  (1,2^{t + 1} - 2,2^{t + 2} - 1,1),$\cr 
$5.\  (1,2^{t + 2} - 1,1,2^{t + 1} - 2),$& $6.\  (1,2^{t + 2} - 1,2^{t + 1} - 2,1),$\cr 
$7.\  (2^{t+2}-1,1,1,2^{t + 1} - 2),$& $8.\  (2^{t+2}-1,1,2^{t + 1} - 2,1),$\cr 
$9.\  (1,1,2^{t + 1} - 1,2^{t + 2} - 2),$& $10.\  (1,1,2^{t + 2} - 2,2^{t + 1} - 1),$\cr 
$11.\  (1,2^{t + 1} - 1,1,2^{t + 2} - 2),$& $12.\  (1,2^{t + 1} - 1,2^{t + 2} - 2,1),$\cr 
$13.\  (1,2^{t + 2} - 2,1,2^{t + 1} - 1),$& $14.\  (1,2^{t + 2} - 2,2^{t + 1} - 1,1),$\cr 
$15.\  (2^{t+1}-1,1,1,2^{t + 2} - 2),$& $16.\  (2^{t+1}-1,1,2^{t + 2} - 2,1),$\cr 
$17.\  (1,2,2^{t + 1} - 3,2^{t + 2} - 1),$& $18.\  (1,2,2^{t + 2} - 1,2^{t + 1} - 3),$\cr 
$19.\  (1,2^{t + 2} - 1,2,2^{t + 1} - 3),$& $20.\  (2^{t+2}-1,1,2,2^{t + 1} - 3),$\cr 
$21.\  (1,2,2^{t + 1} - 1,2^{t + 2} - 3),$& $22.\  (1,2,2^{t + 2} - 3,2^{t + 1} - 1),$\cr 
$23.\  (1,2^{t + 1} - 1,2,2^{t + 2} - 3),$& $24.\  (2^{t+1}-1,1,2,2^{t + 2} - 3),$\cr 
$25.\  (1,3,2^{t + 1} - 4,2^{t + 2} - 1),$& $26.\  (1,3,2^{t + 2} - 1,2^{t + 1} - 4),$\cr 
$27.\  (1,2^{t + 2} - 1,3,2^{t + 1} - 4),$& $28.\  (3,1,2^{t + 1} - 4,2^{t + 2} - 1),$\cr 
$29.\  (3,1,2^{t + 2} - 1,2^{t + 1} - 4),$& $30.\  (3,2^{t + 2} - 1,1,2^{t + 1} - 4),$\cr 
$31.\  (2^{t+2}-1,1,3,2^{t + 1} - 4),$& $32.\  (2^{t+2}-1,3,1,2^{t + 1} - 4),$\cr 
$33.\  (1,3,2^{t + 1} - 3,2^{t + 2} - 2),$& $34.\  (1,3,2^{t + 2} - 2,2^{t + 1} - 3),$\cr 
$35.\  (1,2^{t + 2} - 2,3,2^{t + 1} - 3),$& $36.\  (3,1,2^{t + 1} - 3,2^{t + 2} - 2),$\cr 
$37.\  (3,1,2^{t + 2} - 2,2^{t + 1} - 3),$& $38.\  (3,2^{t + 1} - 3,1,2^{t + 2} - 2),$\cr 
$39.\  (3,2^{t + 1} - 3,2^{t + 2} - 2,1),$& $40.\  (1,3,2^{t + 1} - 2,2^{t + 2} - 3),$\cr 
$41.\  (1,3,2^{t + 2} - 3,2^{t + 1} - 2),$& $42.\  (1,2^{t + 1} - 2,3,2^{t + 2} - 3),$\cr 
$43.\  (3,1,2^{t + 1} - 2,2^{t + 2} - 3),$& $44.\  (3,1,2^{t + 2} - 3,2^{t + 1} - 2),$\cr 
$45.\  (3,2^{t + 2} - 3,1,2^{t + 1} - 2),$& $46.\  (3,2^{t + 2} - 3,2^{t + 1} - 2,1),$\cr 
\end{tabular}}
\centerline{\begin{tabular}{ll}
$47.\  (1,3,2^{t + 1} - 1,2^{t + 2} - 4),$& $48.\  (1,3,2^{t + 2} - 4,2^{t + 1} - 1),$\cr 
$49.\  (1,2^{t + 1} - 1,3,2^{t + 2} - 4),$& $50.\  (3,1,2^{t + 1} - 1,2^{t + 2} - 4),$\cr 
$51.\  (3,1,2^{t + 2} - 4,2^{t + 1} - 1),$& $52.\  (3,2^{t + 1} - 1,1,2^{t + 2} - 4),$\cr 
$53.\  (2^{t+1}-1,1,3,2^{t + 2} - 4),$& $54.\  (2^{t+1}-1,3,1,2^{t + 2} - 4),$\cr 
$55.\  (3,2^{t + 1} - 3,2,2^{t + 2} - 3),$& $56.\  (3,2^{t + 2} - 3,2,2^{t + 1} - 3),$\cr 
$57.\  (3,3,2^{t + 1} - 4,2^{t + 2} - 3),$& $58.\  (3,3,2^{t + 2} - 3,2^{t + 1} - 4),$\cr 
$59.\  (3,2^{t + 2} - 3,3,2^{t + 1} - 4),$& $60.\  (3,3,2^{t + 1} - 3,2^{t + 2} - 4),$\cr 
$61.\  (3,3,2^{t + 2} - 4,2^{t + 1} - 3),$& $62.\  (3,2^{t + 1} - 3,3,2^{t + 2} - 4).$ \cr
$63.\  (3,4,2^{t + 2} - 5,2^{t + 1} - 3),$& $64.\  (3,4,2^{t + 1} - 5,2^{t + 2} - 3),$\cr 
$65.\  (3,5,2^{t + 2} - 5,2^{t + 1} - 4),$& $66.\  (3,5,2^{t + 2} - 6,2^{t + 1} - 3),$\cr 
$67.\  (3,7,2^{t + 2} - 7,2^{t + 1} - 4),$& $68.\  (7,3,2^{t + 2} - 7,2^{t + 1} - 4),$\cr 
$69.\  (1,6,2^{t + 2} - 5,2^{t + 1} - 3),$& $70.\  (1,7,2^{t + 2} - 6,2^{t + 1} - 3),$\cr 
$71.\  (7,1,2^{t + 2} - 6,2^{t + 1} - 3).$& $72.\  (1,7,2^{t + 2} - 5,2^{t + 1} - 4),$\cr 
$73.\  (7,1,2^{t + 2} - 5,2^{t + 1} - 4),$& $74.\  (7,2^{t + 2} - 5,1,2^{t + 1} - 4),$\cr 
\end{tabular}}

\medskip
For $t = 2$,

\medskip
\centerline{\begin{tabular}{lll}
$75.\  (3,5,6,9),$& $76.\  (3,5,7,8),$& $77.\  (3,7,5,8),$\cr 
$78.\  (7,3,5,8),$& $79.\  (1,6,7,9),$& $80.\  (1,7,7,8),$\cr 
$81.\  (7,1,7,8),$& $82.\  (7,7,1,8),$& $83.\  (1,7,6,9),$\cr 
$84.\  (7,1,6,9),$& $85.\  (3,7,8,5),$& $86.\  (7,3,8,5),$\cr 
$87.\  (3,4,1,15),$& $88.\  (3,4,15,1),$& $89.\  (3,15,4,1),$\cr 
$90.\  (15,3,4,1),$& $91.\  (3,7,12,1),$& $92.\  (7,3,12,1),$\cr 
\end{tabular}}
\centerline{\begin{tabular}{lll}
$93.\  (7,11,4,1),$& $94.\  (7,7,8,1),$& $95.\  (7,9,2,5),$\cr 
$96.\  (3,4,7,9),$& 97.\  $(3,7,4,9),$& $98.\  (7,3,4,9),$\cr 
$99.\  (7,9,3,4)$.& &\cr
\end{tabular}}

\smallskip
For $t \geqslant 3$,

\medskip
\centerline{\begin{tabular}{ll}
$75.\  (3,5,2^{t + 1} - 6,2^{t + 2} - 3),$& $76.\  (3,5,2^{t + 1} - 5,2^{t + 2} - 4),$\cr 
$77.\  (3,7,2^{t + 1} - 7,2^{t + 2} - 4),$& $78.\  (7,3,2^{t + 1} - 7,2^{t + 2} - 4),$\cr 
$79.\  (1,6,2^{t + 1} - 5,2^{t + 2} - 3),$& $80.\  (1,7,2^{t + 1} - 5,2^{t + 2} - 4),$\cr $81.\  (7,1,2^{t + 1} - 5,2^{t + 2} - 4),$& $82.\  (7,2^{t + 1} - 5,1,2^{t + 2} - 4),$\cr 
$83.\  (1,7,2^{t + 1} - 6,2^{t + 2} - 3),$& $84.\  (7,1,2^{t + 1} - 6,2^{t + 2} - 3),$\cr 
\end{tabular}}
\end{thms}

We prove the theorem by proving the following propositions.

\begin{props}\label{mdc8.3} The $\mathbb F_2$-vector space $(\mathbb F_2\underset {\mathcal A}\otimes R_4)_{2^{t+2}+ 2^{t+1} -1}$ is generated by the  elements listed in Theorem \ref{dlc8.3}.
\end{props}

The proof of this proposition is based on the following lemmas.

\begin{lems}\label{8.3.1} The following matrices are strictly inadmissible
 $$ \begin{pmatrix} 1&0&1&1\\ 1&0&0&1\\ 0&1&0&1\\ 0&1&0&0\end{pmatrix} \quad  \begin{pmatrix} 1&0&1&1\\ 1&0&1&0\\ 0&1&1&0\\ 0&1&0&0\end{pmatrix} \quad  \begin{pmatrix} 1&0&1&1\\ 1&0&1&0\\ 0&1&0&1\\ 0&1&0&0\end{pmatrix} \quad  \begin{pmatrix} 1&0&1&1\\ 1&0&0&1\\ 0&1&0&1\\ 0&0&1&0\end{pmatrix} $$    
$$ \begin{pmatrix} 1&1&1&0\\ 1&0&0&1\\ 0&1&0&1\\ 0&0&1&0\end{pmatrix} \quad  \begin{pmatrix} 1&1&0&1\\ 1&0&0&1\\ 0&1&0&1\\ 0&0&1&0\end{pmatrix} \quad  \begin{pmatrix} 1&0&1&1\\ 1&0&1&0\\ 1&0&0&1\\ 0&1&0&0\end{pmatrix}. $$
\end{lems}

\begin{proof} The monomials corresponding to the above matrices respectively are (3,12,1,7),  (3,12,7,1),   (3,12,3,5),  (3,4,9,7),  (3,5,9,6),  (3,5,8,7),  (7,8,3,5). 
 By a direct computation, we have
\begin{align*}
&(3,12,1,7) = Sq^1\big((3,11,1,7) +(3,7,1,11) \big)+ Sq^2\big((2,11,1,7)\\ 
&\quad  +(5,7,2,7) + (2,7,1,11)\big)+Sq^4(3,7,2,7)+ (2,13,1,7)  + (3,9,4,7)\\ 
&\quad + (3,7,4,9) + (2,7,1,13) + (3,7,1,12) \quad \text{mod  }\mathcal L_4(3;2;2;1),\\
&(3,12,7,1) = \overline{\varphi}_3(3,12,1,7),\\
&(3,12,3,5) = Sq^1\big((3,11,3,5) +(3,7,3,9) \big)+ Sq^2\big((2,11,3,5)\\ 
&\quad  +(5,7,3,6) + (2,7,3,9)\big)+Sq^4(3,7,3,6)+ (2,13,3,5)\\ 
&\quad  + (2,11,5,5) + (3,11,4,5) + (3,9,5,6) + (3,7,5,8)\\ 
&\quad + (2,7,5,9) + (3,7,4,9)\quad \text{mod  }\mathcal L_4(3;2;2;1),\\
&(3,4,9,7) = Sq^1\big((3,1,11,7) +(3,1,7,11) \big)+ Sq^2\big((2,1,11,7)\\ 
&\quad  +(5,2,7,7) + (2,1,7,11)\big)+Sq^4(3,2,7,7) + (2,1,13,7)  + (3,4,7,9)\\ 
&\quad+  (2,1,7,13) +(3,1,12,7) + (3,1,7,12) \quad \text{mod  }\mathcal L_4(3;2;2;1),
\end{align*}
\begin{align*}
&(3,5,9,6) = Sq^1\big((3,3,11,5) +(3,3,7,9) \big)+ Sq^2\big((5,3,7,6) \\ 
&\quad +(2,3,11,5) + (2,3,7,9)\big)+Sq^4(3,3,7,6)+ (3,5,7,8)\\ 
&\quad  + (3,4,11,5) +  (3,3,12,5) +(3,4,7,9) + (2,5,11,5) \\ 
&\quad+ (2,3,13,5) + (2,5,7,9)\quad \text{mod  }\mathcal L_4(3;2;2;1),\\
&(3,5,8,7) = Sq^1\big((3,3,5,11) +(3,3,9,7) \big)+ Sq^2\big((5,3,6,7)\\ 
&\quad  +(2,3,5,11) + (2,3,9,7)\big)+Sq^4(3,3,6,7)+ (3,5,6,9)\\ 
&\quad  + (3,4,5,11) +  (3,3,5,12) +(3,4,9,7) + (2,5,5,11) \\ 
&\quad+ (2,3,5,13) + (2,5,9,7)\quad \text{mod  }\mathcal L_4(3;2;2;1),\\
&(7,8,3,5) = Sq^1(7,5,5,5) + Sq^2\big((7,6,3,5)  +(7,3,6,5)+ (7,3,5,6)\big)\\ 
&\quad+Sq^4\big((5,6,3,5)+ (5,3,5,6)  + (5,3,6,5)\big) +  (5,10,3,5)\\ 
&\quad +(5,6,3,9) + (7,3,5,8) + (7,3,8,5) + (5,3,9,6)\\ 
&\quad + (5,3,5,10) + (5,3,10,5) + (5,3,6,9)\quad \text{mod  }\mathcal L_4(3;2;2;1).
\end{align*}

The lemma is proved.
\end{proof}

\begin{lems}\label{8.3.2} The following matrices are strictly inadmissible
$$\begin{pmatrix} 1&1&0&1\\ 1&1&0&0\\ 0&1&0&1\\ 0&0&1&1\\ 0&0&0&1\end{pmatrix} \quad \begin{pmatrix} 1&1&0&1\\ 1&1&0&0\\ 1&0&0&1\\ 0&0&1&1\\ 0&0&0&1\end{pmatrix} \quad \begin{pmatrix} 1&1&0&1\\ 1&1&0&0\\ 0&1&0&1\\ 0&0&1&1\\ 0&0&1&0\end{pmatrix} $$  $$\begin{pmatrix} 1&1&0&1\\ 1&1&0&0\\ 1&0&0&1\\ 0&0&1&1\\ 0&0&1&0\end{pmatrix} \quad \begin{pmatrix} 1&1&1&0\\ 1&1&0&0\\ 1&0&1&0\\ 0&1&1&0\\ 0&0&0&1\end{pmatrix} . $$
\end{lems}

\begin{proof} The monomials corresponding to the above matrices respectively are $(3,7,8,29), (7,3,8,29), (3,7,24,13),  (7,3,24,13),   (7,11,13,16).$ 

For simplicity, we prove the lemma for the matrices associated to the monomials $(7,3,24,13), (3,7,8,29), (7,11,13,16).$
 By a direct computation, we have
 \begin{align*}
&(7,3,24,13) = Sq^1\big((7,3,23,13) + (7,3,19,17) + (7,3,15,21)\big)\\ 
&\quad+ Sq^2\big((7,2,23,13) + (7,5,19,14) + (7,5,15,18) + (7,2,15,21)\big)\\ 
&\quad+Sq^4\big((4,3,23,13)+ (5,2,23,13) + (11,3,15,14) + (5,3,21,14)\\ 
&\quad + (5,3,15,20) + (4,3,15,21) + (5,2,15,21)\big)+Sq^8(7,3,15,14)  \\ 
&\quad+(4,3,27,13) + (5,2,27,13)+ (7,2,25,13) + (5,3,25,14)\\ 
&\quad+ (5,3,15,24)  + (4,3,15,25) + (5,2,15,25) \quad \text{mod  }\mathcal L_4(3;2;2;2;1),
 \end{align*}
 \begin{align*}
&(3,7,8,29) = Sq^1\big((3,7,9,27) + (3,7,11,25) + (3,7,13,23)\\ 
&\quad + (3,7,21,15) + (3,7,17,19)\big) + Sq^2\big((5,7,10,23) + (5,7,6,27)\\ 
&\quad + (2,7,9,27) + (5,7,11,22) + (5,7,7,26) + (2,7,11,25)\\ 
&\quad + (2,7,13,23) + (5,7,18,15) + (5,7,14,19) + (2,7,21,15)\big)\\ 
&\quad +Sq^4\big((3,5,12,23) + (3,5,6,29) + (3,11,6,23) + (3,11,7,22)\\ 
&\quad + (3,5,7,28) + (3,5,13,22) + (3,4,13,23) + (3,11,14,15)\\ 
&\quad + (3,5,20,15) + (3,5,14,21) + (3,4,21,15)\big)\\ 
&\quad+Sq^8\big((3,7,6,23) + (3,7,7,22) + (3,7,14,15)\big)+ (3,5,12,27)\\ 
&\quad  + (3,5,10,29) + (2,7,9,29) + (3,5,13,26) + (3,5,11,28)\\ 
&\quad + (2,9,13,23) + (3,4,13,27) + (3,5,24,15) + (3,5,14,25)\\ 
&\quad + (2,9,21,15) + (3,4,25,15)\  \text{mod  }\mathcal L_4(3;2;2;2;1),\\
&(7,11,13,16) = Sq^1(7,13,13,13)+ Sq^2\big((7,11,13,14) + (7,14,11,13)\\ 
&\quad + (7,11,14,13) + (7,7,18,13) + (7,7,10,21)\big) + Sq^4\big((5,11,13,14)\\ 
&\quad+ (5,14,11,13) + (5,11,14,13) + (11,7,12,13) + (5,7,18,13)\\ 
&\quad + (5,7,10,21)\big) + Sq^8\big((7,8,11,13)  +(7,7,12,13)\big)+ (5,11,17,14)\\ 
&\quad + (5,11,13,18) + (5,18,11,13)  + (5,14,11,17) + (7,8,19,13)\\ 
&\quad + (7,8,11,21) + (5,11,14,17)  + (7,11,12,17)  + (7,9,18,13)\\ 
&\quad + (7,9,10,21)  + (5,11,10,21) + (5,7,10,25)  \quad \text{mod  }\mathcal L_4(3;2;2;2;1).
\end{align*}

 The lemma is proved.
\end{proof}

Using the results in Section \ref{7}, Lemmas \ref{5.7}, \ref{5.9}, \ref{5.10}, \ref{8.3.1}, \ref{8.3.2} and Theorem \ref{2.4},  we get Proposition \ref{mdc8.3}.

\medskip
Now, we show that the elements listed in Theorem \ref{dlc8.3} are linearly independent. 

\begin{props}\label{8.3.4} The elements $[a_{1,2,1,j}], 1 \leqslant j \leqslant 99,$  are linearly independent in $(\mathbb F_2\underset {\mathcal A}\otimes R_4)_{23}$.
\end{props}

\begin{proof} Suppose that there is a linear relation
\begin{equation}\sum_{j=1}^{99}\gamma_j[a_{1,2,1,j}] = 0, \tag {\ref{8.3.4}.1}
\end{equation}
with $\gamma_j \in \mathbb F_2$.

Apply the homomorphisms $f_1, f_2, f_3$ to the relation (\ref{8.3.4}.1) and we get
\begin{align*}
&\gamma_{\{3, 87\}}[7,1,15]  +  \gamma_{\{4, 88\}}[7,15,1]  +    \gamma_{13}[15,1,7]  +   \gamma_{14}[15,7,1]\\
&\quad  +   \gamma_{17}[3,5,15]  +   \gamma_{18}[3,15,5]  +   \gamma_{35}[15,3,5]  +   \gamma_{21}[3,7,13]  +   \gamma_{22}[3,13,7]\\
&\quad  +   \gamma_{\{42, 64\}}[7,3,13]  +   \gamma_{\{63, 69\}}[7,11,5]  +   \gamma_{\{79, 96\}}[7,7,9]  =0,
\end{align*}
\begin{align*}
&\gamma_{\{1, 28\}}[7,1,15]  +  \gamma_{\{6, 89\}}[7,15,1]  +   \gamma_{\{10, 51\}}[15,1,7]\\
&\quad  +   \gamma_{\{12, 91, 94\}}[15,7,1]  +   \gamma_{25}[3,5,15]  +   \gamma_{19}[3,15,5]\\
&\quad +   \gamma_{\{34, 61, 86\}}[15,3,5]  +   \gamma_{\{48, 70, 83, 85, 97\}}[7,7,9]  +   \gamma_{48}[3,13,7]\\
&\quad  +   \gamma_{\{40, 48, 57\}}[7,3,13]   +   \gamma_{\{70, 85\}}[7,11,5] +   \gamma_{\{23, 48\}}[3,7,13]   = 0,\\
&\gamma_{\{2, 29\}}[7,1,15]  +  \gamma_{\{5, 30\}}[7,15,1]  +   \gamma_{\{9, 50, 81\}}[15,1,7]  +   \gamma_{\{11, 52, 82\}}[15,7,1]\\
&\quad  +   \gamma_{26}[3,5,15]  +   \gamma_{27}[3,15,5]  +  \gamma_{\{33, 60, 62, 78\}}[15,3,5] \\
&\quad  +   \gamma_{\{47, 72, 76, 80\}}[3,7,13] +    \gamma_{\{47, 76, 80\}}[3,13,7]  +   \gamma_{\{41, 47, 58, 65, 76\}}[7,3,13]\\
&\quad  +   \gamma_{\{49, 59, 77\}}[7,11,5]  +   \gamma_{\{47, 49, 67, 76, 77\}}[7,7,9]  =0. 
\end{align*}

Computing from these equalities, we obtain
\begin{equation}\begin{cases}
\gamma_j = 0,\ j = 13, 14, 17, 18, 19, 21, 22, 23, 25, 26, 27, 35, 48,\\
\gamma_{\{3, 87\}} =  
\gamma_{\{4, 88\}} =   
\gamma_{\{42, 64\}} =  
\gamma_{\{63, 69\}} =  0,\\ 
\gamma_{\{79, 96\}} =   
\gamma_{\{1, 28\}} = 
\gamma_{\{6, 89\}} =   
\gamma_{\{10, 51\}} =   0,\\
\gamma_{\{12, 91, 94\}} =  
\gamma_{\{34, 61, 86\}} =   
\gamma_{\{40, 57\}} = 
\gamma_{\{70, 85\}} =  0,\\ 
\gamma_{\{70, 83, 85, 97\}} =   
\gamma_{\{2, 29\}} =   
\gamma_{\{5, 30\}} =  
\gamma_{\{9, 50, 81\}} = 0,\\  
\gamma_{\{11, 52, 82\}} =   
\gamma_{\{33, 60, 62, 78\}} =   
\gamma_{\{47, 72, 76, 80\}} = 
\gamma_{\{47, 76, 80\}} = 0,\\  
\gamma_{\{41, 47, 58, 65, 76\}} =    
\gamma_{\{49, 59, 77\}} = 
\gamma_{\{47, 49, 67, 76, 77\}} = 0.        
\end{cases}\tag{\ref{8.3.4}.2}
\end{equation}

With the aid of (\ref{8.3.4}.2), the homomorphisms $f_4, f_5,f_6$ send (\ref{8.3.4}.1) to
\begin{align*}
&\gamma_{\{1, 3\}}[1,7,15]  +   \gamma_{10}[1,15,7]  +    \gamma_{\{16, 92, 93, 94\}}[7,15,1]  +   \gamma_{\{8, 90\}}[15,7,1]\\
&\quad  +   \gamma_{\{28, 87\}}[3,5,15]  +   a_1[3,15,5]  +   \gamma_{20}[15,3,5]  +   \gamma_{\{43, 55, 57, 64\}}[3,7,13] \\
&\quad +   \gamma_{51}[3,13,7]  +   \gamma_{24}[7,3,13]  +   \gamma_{\{71, 86, 95\}}[7,11,5]  +   \gamma_{\{84, 98\}}[7,7,9]  = 0,\\   
&\gamma_{\{2, 4\}}[1,7,15]  +  \gamma_{\{9, 47, 79, 80\}}[1,15,7]  +   \gamma_{\{15, 54, 74, 82\}}[7,15,1]\\
&\quad  +   \gamma_{\{7, 32\}}[15,7,1]  +   \gamma_{\{29, 88\}}[3,5,15]  +   \gamma_{\{36, 60, 77\}}[3,15,5]\\
&\quad   +   \gamma_{31}[15,3,5]  +   \gamma_{\{44, 58, 81\}}[3,7,13] +   \gamma_{\{50, 76, 81, 96\}}[3,13,7]\\
&\quad  +   \gamma_{73}[7,3,13]  +   \gamma_{\{53, 78, 99\}}[7,11,5]  +   \gamma_{\{68, 81\}}[7,7,9]  = 0,\\   
&a_2[1,7,15]  +  \gamma_{\{5, 6\}}[1,15,7]  +   a_3[7,1,15]   +   \gamma_{\{7, 8, 20, 31\}}[15,1,7]\\
&\quad +   a_4[3,5,15]  +   \gamma_{\{30, 89\}}[3,15,5]  +  \gamma_{\{32, 90\}}[15,3,5]  +   a_5[3,7,13]\\
&\quad  +    a_6[3,13,7]  +   a_7[7,3,13]  +   \gamma_{\{74, 93\}}[7,11,5]  +   \gamma_{\{82, 94, 95, 99\}}[7,7,9]  = 0,
\end{align*}
where
\begin{align*}
a_1 &= \gamma_{\{37, 56, 61, 63, 66, 85\}},\ \
a_2  = \gamma_{\{11, 12, 49, 70, 72, 80, 83\}},\\
a_3 &= \gamma_{\{15, 16, 24, 53, 71, 73, 81, 84\}}, \ \
a_4 = \gamma_{\{38, 39, 55, 62, 65, 66, 75, 76\}},\\
a_5 &= \gamma_{\{52, 67, 77, 85, 91, 95, 97, 99\}},\ \
a_6 = \gamma_{\{45, 46, 56, 59, 95, 99\}},\\
a_7 &= \gamma_{\{54, 68, 78, 86, 92, 98\}}.
\end{align*}
From these equalities, we get
\begin{equation}\begin{cases}
a_i = 0, \ i = 1,2,\ldots ,7,\
\gamma_j = 0,\ j = 10, 20, 51, 24, 31, 73,\\
\gamma_{\{15, 54, 74, 82\}} = 
\gamma_{\{16, 92, 93, 94\}} =   
\gamma_{\{1, 3\}} = 
\gamma_{\{2, 4\}} = 0,\\
\gamma_{\{28, 87\}} =  
\gamma_{\{29, 88\}} =  
\gamma_{\{30, 89\}} =   
\gamma_{\{32, 90\}} =   
\gamma_{\{36, 60, 77\}} = 0,\\  
\gamma_{\{43, 55, 57, 64\}} = 
\gamma_{\{44, 58, 81\}} =   
\gamma_{\{5, 6\}} =   
\gamma_{\{50, 76, 81, 96\}} = 0,\\ 
\gamma_{\{53, 78, 99\}} =   
\gamma_{\{7, 32\}} =
\gamma_{\{7, 8\}} = 
\gamma_{\{68, 81\}} =   
\gamma_{\{71, 86, 95\}} = 0,\\   
\gamma_{\{74, 93\}} =  
\gamma_{\{8, 90\}} =   
\gamma_{\{82, 94, 95, 99\}} =   
\gamma_{\{84, 98\}} =  
\gamma_{\{9, 47, 79, 80\}} = 0.    
\end{cases}\tag{\ref{8.3.4}.3}
\end{equation}

With the aid of (\ref{8.3.4}.2) and (\ref{8.3.4}.3), the homomorphisms $g_1, g_2$ send (\ref{8.3.4}.1) to
\begin{align*}
&a_8[7,15,1]  +  \gamma_{\{46, 74\}}[15,7,1]  +    a_9[3,15,5]\\
&\quad  +   \gamma_{\{56, 95\}}[15,3,5]  +   \gamma_{43}[3,7,13]  +   \gamma_{55}[7,3,13]\\
&\quad  +   \gamma_{\{66, 71, 86, 95\}}[7,11,5]  +   \gamma_{\{75, 95\}}[7,7,9] =0,\\   
&\gamma_{\{9, 50, 68\}}[1,15,7]  +  a_{10}[7,15,1]  +   \gamma_{\{45, 74\}}[15,7,1]\\
&\quad  +   a_{11}[3,15,5]  +   \gamma_{\{59, 99\}}[15,3,5]  +   a_{12}[3,7,13]  +   a_{13}[3,13,7]\\
&\quad  +   \gamma_{\{65, 72\}}[7,3,13]  +   a_{14}[7,11,5]  +   a_{15}[7,7,9]  = 0,
\end{align*}
where
\begin{align*}
a_8 &= \gamma_{\{12, 16, 39, 91, 92, 94\}},\ \
a_9 = \gamma_{\{34, 37, 61, 71, 86\}},\ \
a_{10} =  \gamma_{\{11, 15, 38, 52, 54, 82\}},\\
a_{11} &= \gamma_{\{33, 36, 42, 53, 60, 62, 78\}},\ \
a_{12} = \gamma_{\{41, 44, 58, 63, 65, 68, 76, 80\}},\\
a_{13} &= \gamma_{\{47, 50, 68, 79, 80\}},\ \
a_{14} = \gamma_{\{49, 53, 59, 62, 77, 78, 99\}},\ \
a_{15} = \gamma_{\{59, 67, 76, 80, 99\}}.
\end{align*}
These equalities imply
\begin{equation} \begin{cases}
\gamma_j = 0, j = 43, 55,\ a_i = 0, i = 8, 9, \ldots, 15,\\ 
\gamma_{\{45, 74\}} =   
\gamma_{\{46, 74\}} =    
\gamma_{\{56, 95\}} =   
\gamma_{\{59, 99\}} = 0,\\  
\gamma_{\{66, 71, 86, 95\}} =   
\gamma_{\{65, 72\}} =   
\gamma_{\{75, 95\}} =  
\gamma_{\{9, 50, 68\}} = 0.    
\end{cases}\tag{\ref{8.3.4}.4}
\end{equation}

 With the aid of (\ref{8.3.4}.2), (\ref{8.3.4}.3) and (\ref{8.3.4}.4), the homomorphisms $g_3, g_4$ send (\ref{8.3.4}.1) to 
\begin{align*}
&\gamma_{\{11, 52, 82\}}[1,15,7]  +  a_{16}[7,15,1]  +    \gamma_{\{12, 91, 94\}}[15,1,7]+   a_{17}[3,15,5]  \\
&\quad  +   \gamma_{\{16, 37, 44, 71\}}[15,7,1]  +   a_{18}[15,3,5]  +   a_{19}[3,7,13]  +   a_{20}[3,13,7]\\
&\quad  +   \gamma_{\{46, 67, 72, 91\}}[7,3,13]  +   a_{21}[7,11,5]  +   a_{22}[7,7,9]  =0,\\
&\gamma_{\{15, 54, 74, 82\}}[1,15,7]  +  a_{23}[7,15,1]  +   \gamma_{\{16, 74, 92, 94\}}[15,1,7]\\
&\quad  +   a_{24}[15,7,1]   +   a_{25}[3,15,5] +   a_{26}[15,3,5]  +   a_{27}[3,7,13]\\
&\quad   +   a_{28}[3,13,7]  +   a_{29}[7,3,13]  +   a_{30}[7,11,5]  +   a_{31}[7,7,9] =0,          
\end{align*}
where
\begin{align*}
a_{16} &= \gamma_{\{9, 15, 36, 50, 53, 68, 84\}},\ \
a_{17} = \gamma_{\{33, 38, 40, 54, 60, 62, 78, 84\}},\\ 
a_{18} &= \gamma_{\{34, 39, 58, 61, 68, 86, 92\}},\ \
a_{19} = \gamma_{\{45, 56, 59, 70, 72, 77, 80, 82, 91, 94\}},\\
a_{20} &= \gamma_{\{49, 52, 70, 72, 80, 82, 83, 91, 94\}},\ \
a_{21} = \gamma_{\{41, 47, 54, 58, 60, 61, 65, 66, 68, 75, 76, 78, 84, 86, 92\}},
\end{align*}
\begin{align*}
a_{22} &= \gamma_{\{41, 58, 61, 65, 66, 67, 68, 72, 77, 80, 82, 86, 91, 92, 95, 99\}},\\
a_{23} &= \gamma_{\{9, 11, 33, 40, 42, 47, 49, 79, 80, 83\}},\ \
a_{24} = \gamma_{\{12, 34, 41, 63, 70, 72\}},\\
a_{25} &= \gamma_{\{36, 38, 40, 45, 52, 60, 77, 83\}},\ \
a_{26} = \gamma_{\{37, 39, 46, 56, 58, 59, 61, 63, 65, 66, 67, 70, 91\}},\\
a_{27} &= \gamma_{\{68, 71, 78, 82, 92, 94, 95, 99\}},\ \
a_{28} = \gamma_{\{53, 54, 68, 71, 82, 84, 92, 94\}},\\
a_{29} &= \gamma_{\{68, 71, 86, 92, 95, 99\}},\ \
a_{31} = \gamma_{\{44, 45, 58, 61, 67, 70, 71, 74, 78, 82, 86, 91, 92, 95, 99\}}\\
a_{30} &= \gamma_{\{40, 42, 44, 45, 46, 50, 52, 56, 58, 59, 60, 61, 62, 67, 70, 75, 76, 77, 79, 83, 91\}}.
\end{align*}

From these equalities, we get
\begin{equation} \begin{cases}
 a_i = 0, i = 16, 17, \ldots, 31,\\ 
\gamma_{\{11, 52, 82\}} =   
\gamma_{\{15, 54, 74, 82\}} =   
\gamma_{\{16, 37, 44, 71\}} = 0,\\  
\gamma_{\{12, 91, 94\}} =    
\gamma_{\{16, 74, 92, 94\}} = 
\gamma_{\{46, 67, 72, 91\}} = 0. 
\end{cases}\tag{\ref{8.3.4}.5}
\end{equation}
Substituting (\ref{8.3.4}.2), (\ref{8.3.4}.3), (\ref{8.3.4}.4) and (\ref{8.3.4}.5) into  the relation (\ref{8.3.4}.1), we obtain
\begin{align*}\sum_{i=1}^{15}c_i[\theta_i] = 0,\tag{\ref{8.3.4}.6}\\
\end{align*}
where $c_1= \gamma_1, c_2=\gamma_2, c_3=\gamma_5, c_4= \gamma_7, c_5 = \gamma_9, c_6 = \gamma_{11}, c_7 = \gamma_{12}, c_8 = \gamma_{15}, c_9 = \gamma_{16}, c_{10} = \gamma_{33}, c_{11} = \gamma_{37}, c_{12} = \gamma_{38}, c_{13} = \gamma_{40}, c_{14} = \gamma_{45}, c_{15} = \gamma_{59}$ and
\begin{align*}
\theta_{1} &= a_{1,2,1,1} + a_{1,2,1,3} + a_{1,2,1,28} + a_{1,2,1,87},\\
\theta_{2} &= a_{1,2,1,2} + a_{1,2,1,4} + a_{1,2,1,29} + a_{1,2,1,88},\\
\theta_{3} &= a_{1,2,1,5} + a_{1,2,1,6} + a_{1,2,1,30} + a_{1,2,1,89},\\
\theta_{4} &= a_{1,2,1,7} + a_{1,2,1,8} + a_{1,2,1,32} + a_{1,2,1,90},\\
\theta_{5} &= a_{1,2,1,9} + a_{1,2,1,50} + a_{1,2,1,79} + a_{1,2,1,96},\\
\theta_{6} &= a_{1,2,1,11} + a_{1,2,1,52} + a_{1,2,1,83} + a_{1,2,1,97},\\
\theta_{7} &= a_{1,2,1,12} + a_{1,2,1,41} + a_{1,2,1,47} + a_{1,2,1,67} + a_{1,2,1,80} + a_{1,2,1,91},\\
\theta_{8} &= a_{1,2,1,15} + a_{1,2,1,54} + a_{1,2,1,84} + a_{1,2,1,98},\\
\theta_{9} &= a_{1,2,1,16} + a_{1,2,1,44} + a_{1,2,1,50} + a_{1,2,1,68} + a_{1,2,1,81} + a_{1,2,1,92},\\
\theta_{10} &= a_{1,2,1,33} + a_{1,2,1,34} + a_{1,2,1,49} + a_{1,2,1,60}\\
&\quad + a_{1,2,1,61} + a_{1,2,1,70} + a_{1,2,1,77} + a_{1,2,1,85},\\
\theta_{11} &= a_{1,2,1,36} + a_{1,2,1,37} + a_{1,2,1,53} + a_{1,2,1,60}\\
&\quad + a_{1,2,1,61} + a_{1,2,1,71} + a_{1,2,1,78} + a_{1,2,1,86},\\
\theta_{12} &= a_{1,2,1,38} + a_{1,2,1,52} + a_{1,2,1,54} + a_{1,2,1,56} + a_{1,2,1,61}\\
&\quad + a_{1,2,1,75} + a_{1,2,1,82} + a_{1,2,1,86} + a_{1,2,1,95},\\
\end{align*}
\begin{align*}
\theta_{13} &= a_{1,2,1,36} + a_{1,2,1,39} + a_{1,2,1,40} + a_{1,2,1,42} + a_{1,2,1,47} + a_{1,2,1,49}\\
&\quad + a_{1,2,1,50} + a_{1,2,1,56} + a_{1,2,1,57} + a_{1,2,1,58} + a_{1,2,1,61}\\
&\quad + a_{1,2,1,63} + a_{1,2,1,64} + a_{1,2,1,67} + a_{1,2,1,68}+ a_{1,2,1,69} \\
&\quad + a_{1,2,1,70} + a_{1,2,1,75}+ a_{1,2,1,77}  + a_{1,2,1,80} + a_{1,2,1,81}\\
&\quad + a_{1,2,1,83} + a_{1,2,1,84} + a_{1,2,1,85} + a_{1,2,1,86} + a_{1,2,1,91}\\
&\quad + a_{1,2,1,92} + a_{1,2,1,94} + a_{1,2,1,95} + a_{1,2,1,97} + a_{1,2,1,98},\\
\theta_{14} &= a_{1,2,1,45} + a_{1,2,1,46} + a_{1,2,1,52} + a_{1,2,1,74}\\
&\quad + a_{1,2,1,82} + a_{1,2,1,91} + a_{1,2,1,93} + a_{1,2,1,94},\\
\theta_{15} &= a_{1,2,1,34} + a_{1,2,1,39} + a_{1,2,1,47} + a_{1,2,1,59} + a_{1,2,1,60} + a_{1,2,1,61}\\
&\quad + a_{1,2,1,63} + a_{1,2,1,67} + a_{1,2,1,76} + a_{1,2,1,77} + a_{1,2,1,78} + a_{1,2,1,69}\\
&\quad + a_{1,2,1,79} + a_{1,2,1,91} + a_{1,2,1,92} + a_{1,2,1,94} + a_{1,2,1,96} + a_{1,2,1,99}.
\end{align*}

Now, we show that $c_i = 0$ for $i = 1, 2,\ldots , 15$. The proof is divided into 6 steps.

{\it Step 1.} Set $\theta = \theta_1 + \sum_{i=2}^{15}\beta_i\theta_i$ for $\beta_i \in \mathbb F_2, i= 2, 3, \ldots, 15$. We prove that $\theta$ is non-hit. Suppose the contrary that $\theta$ is hit. Then we have
$$\theta = Sq^1(A) + Sq^2(B) +Sq^4(C) + Sq^8(D)$$
for some polynomials $A, B, C, D \in R_4$. Let $(Sq^2)^3$ act to the both sides of the above equality, we obtain
$$(Sq^2)^3(\theta) = (Sq^2)^3Sq^4(C) + (Sq^2)^3Sq^8(D).$$
By a direct computation we see that the monomial $(8,4,2,15)$ is a term of $(Sq^2)^3(\theta)$. This monomial is not a term of $(Sq^2)^3(Sq^4(C)+Sq^8(D))$ for all polynomials $C, D$ and we have a contradiction. So, $\theta$ is non-hit and we get $c_1 = \gamma_1 = 0$. 

By an argument analogous to the previous one, we get $c_2 = c_3 = c_4 = 0$. Now, the relation (\ref{8.3.4}.6) becomes
\begin{equation} \sum_{i=5}^{15}c_i[\theta_i] = 0. \tag{\ref{8.3.4}.7}
\end{equation}

{\it Step 2.} The homomorphism $\varphi_4$ sends (\ref{8.3.4}.7) to
\begin{equation}  c_{14}[\theta_3] + \sum_{i=5}^{15}d_i[\theta_i]  = 0, \tag{\ref{8.3.4}.8}
\end{equation}
for some $d_i \in \mathbb F_2, i=5, 6, \ldots , 15$. Using (\ref{8.3.4}.8) and by a same argument as given in Step 1, we get $c_{14} = \gamma_{45} = 0$.

Under the homomorphism  induced by the linear transformation $x_1 \mapsto x_1+x_3, x_2 \mapsto x_2, x_3 \mapsto x_3, x_4 \mapsto x_4$, the linear relation (\ref{8.3.4}.7) is sent to 
\begin{equation}  c_{9}[\theta_2] + \sum_{i=5}^{15}d_i[\theta_i]  = 0, \tag{\ref{8.3.4}.9}
\end{equation}
for some $d_i \in \mathbb F_2, i=5, 6, \ldots , 15$. Using (\ref{8.3.4}.9) and by a same argument as given in Step 1, we get $c_{9} = \gamma_{16} = 0$.

Consider the homomorphism $\mathbb F_2\underset {\mathcal A}\otimes R_4 \to  \mathbb F_2\underset {\mathcal A}\otimes R_4$ induced by the linear transformation $x_1 \mapsto x_1+x_4, x_2 \mapsto x_2, x_3 \mapsto x_3, x_4 \mapsto x_4$. Under this homomorphism, the linear relation (\ref{8.3.4}.7) is sent to 
\begin{equation}  c_{8}[\theta_1] + \sum_{i=5}^{15}d_i[\theta_i]  = 0, \tag{\ref{8.3.4}.10}
\end{equation}
for some $d_i \in \mathbb F_2, i=5, 6, \ldots , 15$. Using (\ref{8.3.4}.10) and by a same argument as given in Step 1, we get $c_{8} = \gamma_{15} = 0$.

 The homomorphism induced by the linear transformation $x_1 \mapsto x_1, x_2 \mapsto x_2 + x_3, x_3 \mapsto x_3, x_4 \mapsto x_4$ sends the linear relation (\ref{8.3.4}.7)  to 
\begin{equation}  c_{7}[\theta_3] + \sum_{i=5}^{15}d_i[\theta_i]  = 0, \tag{\ref{8.3.4}.11}
\end{equation}
for some $d_i \in \mathbb F_2, i=5, 6, \ldots , 15$. Using (\ref{8.3.4}.11) and by a same argument as given in Step 1, we get $c_{7} = \gamma_{12} = 0$.

Applying the homomorphism  induced by the linear map $x_1 \mapsto x_1, x_2 \mapsto x_2+x_4, x_3 \mapsto x_3, x_4 \mapsto x_4$ to the linear relation (\ref{8.3.4}.7), we get 
\begin{equation}  c_{6}[\theta_1] + \sum_{i=5}^{15}d_i[\theta_i]  = 0, \tag{\ref{8.3.4}.12}
\end{equation}
for some $d_i \in \mathbb F_2, i=5, 6, \ldots , 15$. Using (\ref{8.3.4}.12) and by a same argument as given in Step 1, we get $c_{6} = \gamma_{11} = 0$.

Consider the homomorphism  induced by the linear map $x_1 \mapsto x_1, x_2 \mapsto x_2, x_3 \mapsto x_3+x_4, x_4 \mapsto x_4$. This homomorphism sends the linear relation (\ref{8.3.4}.7) to 
\begin{equation}  c_{5}[\theta_1] + \sum_{i=5}^{15}d_i[\theta_i]  = 0, \tag{\ref{8.3.4}.13}
\end{equation}
for some $d_i \in \mathbb F_2, i=5, 6, \ldots , 15$. Using (\ref{8.3.4}.13) and by a same argument as given in Step 1, we get $c_{5} = \gamma_{9} = 0$.

So, the relation (\ref{8.3.4}.7) becomes
\begin{equation} 
c_{10}[\theta_{10}] + c_{11}[\theta_{11}] + c_{12}[\theta_{12}] + c_{13}[\theta_{13}] + c_{15}[\theta_{15}] = 0. \tag{\ref{8.3.4}.14}
\end{equation}
{\it Step 3.} The homomorphism $\varphi_4$ sends (\ref{8.3.4}.14) to
\begin{align*}
&c_{12}[\theta_6]+ (c_{13}+c_{15})[\theta_7]+(c_{10}+c_{11})[\theta_{10}]\\
&\quad + c_{11}[\theta_{11}] + c_{12}[\theta_{12}] + c_{13}[\theta_{13}] + c_{15}[\theta_{15}] = 0. 
\end{align*}
Using this equality and by a same argument as given in Step 2, we get $c_{12} = 0$ and $c_{13}=c_{15}$. So, the relation (\ref{8.3.4}.14) becomes
\begin{equation} 
c_{10}[\theta_{10}] + c_{11}[\theta_{11}] + c_{13}[\theta_{13}] + c_{13}[\theta_{15}] = 0. \tag{\ref{8.3.4}.15}
\end{equation}

{\it Step 4.} The homomorphism $\varphi_2$ sends (\ref{8.3.4}.15) to
$$c_{10}[\theta_{10}] + c_{13}[\theta_{11}] + (c_{11}+c_{13})[\theta_{12}] +c_{13}[\theta_{13}] + c_{13}[\theta_{15}] = 0.$$
Using this equality and by a same argument as given in Step 2, we get $c_{11}=c_{13}$. Then the relation (\ref{8.3.4}.15) becomes
\begin{equation} 
c_{10}[\theta_{10}] + c_{11}[\theta_{11}] + c_{11}[\theta_{13}] + c_{11}[\theta_{15}] = 0. \tag{\ref{8.3.4}.16}
\end{equation}

{\it Step 5.} The homomorphism $\varphi_1$ send (\ref{8.3.4}.16) to
$$(c_{10}+c_{11})[\theta_{11}] + c_{11}[\theta_{13}] + c_{11}[\theta_{15}] = 0.$$
Using this equality and by a same argument as given in Step 4, we get $c_{10}=0$. Then the relation (\ref{8.3.4}.16) becomes
\begin{equation} 
c_{11}([\theta_{11}] + [\theta_{13}] + [\theta_{15}]) = 0. \tag{\ref{8.3.4}.17}
\end{equation}
{\it Step 6.} The homomorphism $\varphi_4$ sends (\ref{8.3.4}.17) to
$$c_{11}[\theta_{10}] + c_{11}([\theta_{11}] + [\theta_{13}] + [\theta_{15}]) = 0. 
$$
By an argument analogous to Step 5, we have $c_{11} = 0.$

The proposition is completely proved.
\end{proof}

\begin{props}\label{8.3.6} For $t \geqslant 3$, the elements $[a_{1,t,1,j}], 1 \leqslant j \leqslant 84,$  are linearly independent in $(\mathbb F_2\underset {\mathcal A}\otimes R_4)_{2^{t+2}+2^{t+1}-1}$.
\end{props}

\begin{proof} Suppose that there is a linear relation
\begin{equation}\sum_{j=1}^{84}\gamma_j[a_{1,t,1,j}] = 0, \tag {\ref{8.3.6}.1}
\end{equation}
with $\gamma_j \in \mathbb F_2$.

Applying the homomorphisms $f_1, f_2, f_3$ to the relation (\ref{8.3.6}.1), we obtain
\begin{align*}
&\gamma_{3}w_{1,t,1,3}  +  \gamma_{4}w_{1,t,1,4}  +    \gamma_{13}w_{1,t,1,5}  +   \gamma_{14}w_{1,t,1,6}  +   \gamma_{17}w_{1,t,1,7}\\
&\quad  +   \gamma_{18}w_{1,t,1,8}  +   \gamma_{35}w_{1,t,1,9}  +   \gamma_{21}w_{1,t,1,10}  +   \gamma_{22}w_{1,t,1,11}\\
&\quad  +   \gamma_{42}w_{1,t,1,12}  +   \gamma_{\{63, 69\}}w_{1,t,1,13}  +   \gamma_{\{64, 79\}}w_{1,t,1,14}  =0,\\   
&\gamma_{\{1, 28\}}w_{1,t,1,3}  +   \gamma_{6}w_{1,t,1,4}  +   \gamma_{\{10, 51\}}w_{1,t,1,5}  +   \gamma_{12}w_{1,t,1,6}\\
&\quad  +   \gamma_{25}w_{1,t,1,7}  +   \gamma_{19}w_{1,t,1,8}  +   \gamma_{\{34, 61\}}w_{1,t,1,9}  +   \gamma_{23}w_{1,t,1,10}  +   \gamma_{48}w_{1,t,1,11}\\
&\quad  +   \gamma_{\{40, 57\}}w_{1,t,1,12}  +   \gamma_{70}w_{1,t,1,13}  +   \gamma_{\{48, 70, 83\}}w_{1,t,1,14}  = 0,\\   
&\gamma_{\{2, 29\}}w_{1,t,1,3}  +  \gamma_{\{5, 30\}}w_{1,t,1,4}  +   \gamma_{\{9, 50\}}w_{1,t,1,5}  +   \gamma_{\{11, 52\}}w_{1,t,1,6}\\
&\quad  +   \gamma_{26}w_{1,t,1,7}  +   \gamma_{27}w_{1,t,1,8}  +  \gamma_{\{33, 60, 62, 76, 77\}}w_{1,t,1,9}\\
&\quad  +   \gamma_{49}w_{1,t,1,10}  +    \gamma_{47}w_{1,t,1,11}  +   \gamma_{\{41, 58, 59, 65, 67\}}w_{1,t,1,12}\\
&\quad  +   \gamma_{\{49, 59, 80\}}w_{1,t,1,13}  +   \gamma_{\{47, 59, 72, 80\}}w_{1,t,1,14}  = 0.  
\end{align*}

Computing from these equalities, we obtain
\begin{equation}\begin{cases}
\gamma_j = 0,\ j = 
3, 4, 6, 12, 13, 14, 17, 18, 19, 21, 22, 23,\\
\hskip2.2cm 25, 26, 27, 35, 42, 47, 48, 49, 70, 83,\\ 
\gamma_{\{64, 79\}} = 
\gamma_{\{63, 69\}} =  
\gamma_{\{34, 61\}} = 
\gamma_{\{33, 60, 62, 76, 77\}} = 0,\\
\gamma_{\{40, 57\}} = 
\gamma_{\{59, 72, 80\}} = 
\gamma_{\{41, 58, 59, 65, 67\}} = 
\gamma_{\{59, 80\}} =  0.  
\end{cases}\tag{\ref{8.3.6}.2}
\end{equation}

With the aid of (\ref{8.3.6}.2), the homomorphisms $f_4, f_5,f_6$ send (\ref{8.3.6}.1) to
\begin{align*}
&\gamma_{1}w_{1,t,1,1}  +  \gamma_{10}w_{1,t,1,2}  +   \gamma_{16}w_{1,t,1,4}  +   \gamma_{8}w_{1,t,1,6}  +   \gamma_{28}w_{1,t,1,7}\\
&\quad  +   \gamma_{\{37, 56, 61, 63, 66\}}w_{1,t,1,8}  +   \gamma_{20}w_{1,t,1,9}  +   \gamma_{\{43, 55, 57, 64, 75\}}w_{1,t,1,10}\\
&\quad  +   \gamma_{51}w_{1,t,1,11}  +   \gamma_{24}w_{1,t,1,12}  +   \gamma_{71}w_{1,t,1,13}  +   \gamma_{84}w_{1,t,1,14} = 0,\\   
&\gamma_{2}w_{1,t,1,1}  + \gamma_{9}w_{1,t,1,2}  +   \gamma_{\{15, 54\}}w_{1,t,1,4}  +   \gamma_{\{7, 32\}}w_{1,t,1,6}  +   \gamma_{29}w_{1,t,1,7}\\
&\quad  +   \gamma_{\{36, 60, 78\}}w_{1,t,1,8}  +   \gamma_{31}w_{1,t,1,9}  +   \gamma_{\{44, 58, 68\}}w_{1,t,1,10}\\
&\quad +   \gamma_{50}w_{1,t,1,11}  +   \gamma_{53}w_{1,t,1,12}  +   \gamma_{\{53, 81\}}w_{1,t,1,13}  +   \gamma_{\{53, 73\}}w_{1,t,1,14} =0,\\  
&\gamma_{11}w_{1,t,1,1}  +  \gamma_{5}w_{1,t,1,2}  +    \gamma_{\{15, 16, 24, 53\}}w_{1,t,1,3}  +   \gamma_{\{7, 8, 20, 31\}}w_{1,t,1,5}\\
&\quad  +   \gamma_{\{38, 39, 55, 62\}}w_{1,t,1,7} +   \gamma_{30}w_{1,t,1,8}  +   \gamma_{32}w_{1,t,1,9}  +   \gamma_{52}w_{1,t,1,10}\\
&\quad  +   \gamma_{\{45, 46, 56, 59\}}w_{1,t,1,11}  +   \gamma_{54}w_{1,t,1,12}  +   \gamma_{74}w_{1,t,1,13}  +   \gamma_{82}w_{1,t,1,14}  = 0. 
\end{align*}

Computing from these equalities and using (\ref{8.3.6}.2), we get
\begin{equation}\begin{cases}
\gamma_j = 0,\ j = 1, 2, 5, 7, 8, 9, 10, 11, 15, 16, 20, 24, 28, 29, \\
\hskip1.5cm30, 31, 32, 50, 51, 52, 53, 54, 71, 73, 74, 81, 82, 84,\\ 
\gamma_{\{37, 56, 61, 63, 66\}} = 
\gamma_{\{36, 60, 78\}} = 
\gamma_{\{43, 55, 57, 64, 75\}} =  0,\\
\gamma_{\{44, 58, 68\}} = 
\gamma_{\{38, 39, 55, 62\}} = 
\gamma_{\{45, 46, 56, 59\}} =   0.
\end{cases}\tag{\ref{8.3.6}.3}
\end{equation}

With the aid of (\ref{8.3.6}.2) and (\ref{8.3.6}.3), the homomorphisms $g_1, g_2$ send (\ref{8.3.6}.1) to
\begin{align*}
&\gamma_{39}w_{1,t,1,4}  +  \gamma_{46}w_{1,t,1,6}  +    \gamma_{37}w_{1,t,1,8}  +   \gamma_{\{40, 43, 57\}}w_{1,t,1,10}\\
&\quad   +   \gamma_{56}w_{1,t,1,9} +   \gamma_{55}w_{1,t,1,12}  +   \gamma_{66}w_{1,t,1,13}  +   \gamma_{75}w_{1,t,1,14}  =0,\\   
&\gamma_{38}w_{1,t,1,4}  +   \gamma_{45}w_{1,t,1,6}  +   \gamma_{\{33, 36, 60, 62, 64, 76, 77, 78\}}w_{1,t,1,8}\\
&\quad  +   \gamma_{59}w_{1,t,1,9}  +   \gamma_{\{41, 44, 58, 59, 63, 65, 67, 68\}}w_{1,t,1,10}\\
&\quad  +   \gamma_{62}w_{1,t,1,12}  +   \gamma_{\{62, 76\}}w_{1,t,1,13}  +   \gamma_{\{62, 65\}}w_{1,t,1,14} = 0.  
\end{align*}

These equalities imply
\begin{equation} \begin{cases}
\gamma_j = 0,\ j = 37, 38, 39, 45, 46, 55, 56, 59, 62, 65, 66, 75, 76,\\
\gamma_{\{40, 43, 57\}}= 
\gamma_{\{33, 36, 60, 64, 77, 78\}}=  
\gamma_{\{41, 44, 58, 63, 67, 68\}}= 0\\
\gamma_{\{33, 36, 60, 62, 64, 76, 77, 78\}}=
 \gamma_{\{41, 44, 58, 59, 63, 65, 67, 68\}}=0.
\end{cases}\tag{\ref{8.3.6}.4}
\end{equation}
 With the aid of (\ref{8.3.6}.2), (\ref{8.3.6}.3) and (\ref{8.3.6}.4), the homomorphisms $g_3, g_4$ send (\ref{8.3.6}.1) to 
\begin{align*}
&\gamma_{\{36, 43\}}w_{1,t,1,4}  +  \gamma_{44}w_{1,t,1,6}  +   \gamma_{\{33, 57, 60\}}  + \gamma_{\{77, 78\}}w_{1,t,1,8}\\
&\quad  +   \gamma_{\{58, 68\}}w_{1,t,1,9}  +   \gamma_{\{34, 41, 58\}}  + \gamma_{\{67, 68\}}w_{1,t,1,10}\\
&\quad  +   \gamma_{\{40, 57, 60, 78\}}w_{1,t,1,12}  +   a_2w_{1,t,1,13}  +   a_1w_{1,t,1,14}  = 0,\\   
&\gamma_{\{33, 40, 64\}}w_{1,t,1,4}  +  \gamma_{\{34, 41, 63\}}w_{1,t,1,6} +   \gamma_{\{36, 57, 60, 77, 78\}}w_{1,t,1,8}  +   a_3w_{1,t,1,10}\\
&\quad +   \gamma_{\{34, 58, 63, 67\}}w_{1,t,1,9}   +  a_4w_{1,t,1,12}  +   a_6w_{1,t,1,13} +  a_5w_{1,t,1,14}  =0,  
\end{align*}
where
\begin{align*}
a_1 &= \gamma_{\{40, 57, 60, 67, 78\}},\ \
a_2 = \gamma_{\{34, 40, 41, 57, 58, 60, 68, 77, 78\}},\ \
a_3 = \gamma_{\{34, 44, 58, 67, 68\}},\\
a_4 &= \gamma_{\{43, 57, 60, 64, 77\}},\ \
a_5 = \gamma_{\{43, 57, 60, 64, 68, 77\}},\ \
a_6 = \gamma_{\{34, 43, 44, 57, 58, 60, 64, 67, 77, 78\}}.
\end{align*}

From these equalities, we get
\begin{equation} \begin{cases}
\gamma_{44} = 0,\ a_i = 0, i = 1,2,\ldots , 6,\\
\gamma_{\{36, 43\}} =   
\gamma_{\{33, 57, 60, 77, 78\}} =   
\gamma_{\{58, 68\}} =   
\gamma_{\{34, 41, 58, 67, 68\}} = 0,\\  
\gamma_{\{40, 57, 60, 78\}} =   
\gamma_{\{33, 40, 64\}} =   
\gamma_{\{34, 58, 63, 67\}} =  0,\\
\gamma_{\{34, 41, 63\}} =   
\gamma_{\{36, 57, 60, 77, 78\}} = 0.  
\end{cases}\tag{\ref{8.3.6}.5}
\end{equation}

Combining (\ref{8.3.6}.2), (\ref{8.3.6}.3), (\ref{8.3.6}.4) and (\ref{8.3.6}.5), we get $\gamma_j = 0$ for all $j$.  The proposition is proved.
\end{proof}

\begin{rems}\label{8.3.5} From the proof of Proposition \ref{8.3.4}, we see that the $\mathbb F_2$-subspace of $(\mathbb F_2\underset{\mathcal A} \otimes P_4)_{23}$ generated by $[\theta_{i}], i = 1, 2, \ldots , 15,$ is an $GL_4(\mathbb F_2)$-submodule of $(\mathbb F_2\underset{\mathcal A} \otimes P_4)_{23}$.
\end{rems}

\subsection{The case $s=1, t\geqslant 2 $ and $ u\geqslant 2$}\label{8.4}\ 

\medskip
Using the results in \cite{ka}, for $t\geqslant 2$ and $u\geqslant 2$, we see that the dimension of the space
$\dim (\mathbb F_2\underset{\mathcal A}\otimes P_3)_{2^{t+u+1}+ 2^{t+1} -1}$ is 21 with a basis given by the classes $w_{u,t,1,j}, 1\leqslant j \leqslant 21,$ which are determined as follows:

For $t \geqslant 2$,

\smallskip
\centerline{\begin{tabular}{ll}
$1.\  [1,2^{t + 1} - 1,2^{t + u + 1} - 1],$& $2.\  [1,2^{t + u + 1} - 1,2^{t + 1} - 1],$\cr 
$3.\  [2^{t + 1} - 1,1,2^{t + u + 1} - 1],$& $4.\  [2^{t + 1} - 1,2^{t + u + 1} - 1,1],$\cr 
$5.\  [2^{t + u + 1} - 1,1,2^{t + 1} - 1],$& $6.\  [2^{t + u + 1} - 1,2^{t + 1} - 1,1],$\cr 
$7.\  [1,2^{t + 2} - 1,2^{t + u + 1} - 2^{t+1} -1],$& $8.\  [2^{t + 2} - 1,1,2^{t + u + 1} - 2^{t+1} -1],$\cr 
$9.\  [2^{t + 2} - 1,2^{t + u + 1} - 2^{t+1} -1,1],$& $10.\  [3,2^{t + 1} - 3,2^{t + u + 1} - 1],$\cr 
$11.\  [3,2^{t + u + 1} - 1,2^{t + 1} - 3],$& $12.\  [2^{t + u + 1} - 1,3,2^{t + 1} - 3],$\cr 
$13.\  [3,2^{t + 1} - 1,2^{t + u + 1} - 3],$& $14.\  [3,2^{t + u + 1} - 3,2^{t + 1} - 1],$\cr 
$15.\  [2^{t + 1} - 1,3,2^{t + u + 1} - 3],$& $16.\  [3,2^{t + 2} - 3,2^{t + u + 1} - 2^{t+1} -1],$\cr 
$17.\  [3,2^{t + 2} - 1,2^{t + u + 1} - 2^{t+1} -3],$& $18.\  [2^{t + 2} - 1,3,2^{t + u + 1} - 2^{t+1} -3],$\cr 
$19.\  [7,2^{t + u + 1} - 5,2^{t + 1} - 3],$& $20.\  [7,2^{t + 2} - 5,2^{t + u + 1} - 2^{t+1} -3].$\cr
\end{tabular}}

\smallskip
For $t=2$, \ $w_{u,2,1,21} = [7,7,2^{u+3} - 7].$ 

For $t \geqslant 3,$ \ $w_{u,t,1,21} =  [7,2^{t + 1} - 5,2^{t + u + 1} - 3].$

\smallskip
So, we easily obtain

\begin{props}\label{8.4.2} For any  integer $t \geqslant 2$ and $u \geqslant 2$, 
$$\dim (\mathbb F_2\underset{\mathcal A}\otimes Q_4)_{2^{t+u+1} + 2^{t+1} -1} = 84.$$
\end{props}

Now, we determine $(\mathbb F_2\underset{\mathcal A}\otimes R_4)_{2^{t+u+1}+ 2^{t+1} -1}$. We have

\begin{thms}\label{dlc8.4} For any $u\geqslant 2$, $(\mathbb F_2\underset{\mathcal A}\otimes R_4)_{2^{t+u+1}+ 2^{t+1}-1}$ is  an $\mathbb F_2$-vector space with a basis consisting of all the classes represented by the monomials $a_{u,t,1,j}, j \geqslant 1$, which are determined as follows:

For $t \geqslant 2$,

\medskip
\centerline{\begin{tabular}{ll}
$1.\  (1,1,2^{t + 1} - 2,2^{t + u + 1} - 1),$& $2.\  (1,1,2^{t + u + 1} - 1,2^{t + 1} - 2),$\cr 
$3.\  (1,2^{t + 1} - 2,1,2^{t + u + 1} - 1),$& $4.\  (1,2^{t + 1} - 2,2^{t + u + 1} - 1,1),$\cr 
$5.\  (1,2^{t + u + 1} - 1,1,2^{t + 1} - 2),$& $6.\  (1,2^{t + u + 1} - 1,2^{t + 1} - 2,1),$\cr 
$7.\  (2^{t + u + 1} - 1,1,1,2^{t + 1} - 2),$& $8.\  (2^{t + u + 1} - 1,1,2^{t + 1} - 2,1),$\cr
$9.\  (1,1,2^{t + 1} - 1,2^{t + u + 1} - 2),$& $10.\  (1,1,2^{t + u + 1} - 2,2^{t + 1} - 1),$\cr 
$11.\  (1,2^{t + 1} - 1,1,2^{t + u + 1} - 2),$& $12.\  (1,2^{t + 1} - 1,2^{t + u + 1} - 2,1),$\cr 
$13.\  (1,2^{t + u + 1} - 2,1,2^{t + 1} - 1),$& $14.\  (1,2^{t + u + 1} - 2,2^{t + 1} - 1,1),$\cr 
$15.\  (2^{t + 1} - 1,1,1,2^{t + u + 1} - 2),$& $16.\  (2^{t + 1} - 1,1,2^{t + u + 1} - 2,1),$\cr 
$17.\  (1,1,2^{t + 2} - 2,2^{t + u + 1} - 2^{t+1} -1),$& $18.\  (1,2^{t + 2} - 2,1,2^{t + u + 1} - 2^{t+1} -1),$\cr 
$19.\  (1,2^{t + 2} - 2,2^{t + u + 1} - 2^{t+1} -1,1),$& $20.\  (1,1,2^{t + 2} - 1,2^{t + u + 1} - 2^{t+1} -2),$\cr 
$21.\  (1,2^{t + 2} - 1,1,2^{t + u + 1} - 2^{t+1} -2),$& $22.\  (1,2^{t + 2} - 1,2^{t + u + 1} - 2^{t+1} -2,1),$\cr
$23.\  (2^{t + 2} - 1,1,1,2^{t + u + 1} - 2^{t+1} -2),$& $24.\  (2^{t + 2} - 1,1,2^{t + u + 1} - 2^{t+1} -2,1),$\cr 
$25.\  (1,2,2^{t + 1} - 3,2^{t + u + 1} - 1),$& $26.\  (1,2,2^{t + u + 1} - 1,2^{t + 1} - 3),$\cr $27.\  (1,2^{t + u + 1} - 1,2,2^{t + 1} - 3),$& $28.\  (2^{t + u + 1} - 1,1,2,2^{t + 1} - 3),$\cr 
$29.\  (1,2,2^{t + 1} - 1,2^{t + u + 1} - 3),$& $30.\  (1,2,2^{t + u + 1} - 3,2^{t + 1} - 1),$\cr 
$31.\  (1,2^{t + 1} - 1,2,2^{t + u + 1} - 3),$& $32.\  (2^{t + 1} - 1,1,2,2^{t + u + 1} - 3),$\cr 
$33.\  (1,3,2^{t + 1} - 4,2^{t + u + 1} - 1),$& $34.\  (1,3,2^{t + u + 1} - 1,2^{t + 1} - 4),$\cr 
$35.\  (1,2^{t + u + 1} - 1,3,2^{t + 1} - 4),$& $36.\  (3,1,2^{t + 1} - 4,2^{t + u + 1} - 1),$\cr 
$37.\  (3,1,2^{t + u + 1} - 1,2^{t + 1} - 4),$& $38.\  (3,2^{t + u + 1} - 1,1,2^{t + 1} - 4),$\cr 
$39.\  (2^{t + u + 1} - 1,1,3,2^{t + 1} - 4),$& $40.\  (2^{t + u + 1} - 1,3,1,2^{t + 1} - 4),$\cr 
$41.\  (1,3,2^{t + 1} - 3,2^{t + u + 1} - 2),$& $42.\  (1,3,2^{t + u + 1} - 2,2^{t + 1} - 3),$\cr 
$43.\  (1,2^{t + u + 1} - 2,3,2^{t + 1} - 3),$& $44.\  (3,1,2^{t + 1} - 3,2^{t + u + 1} - 2),$\cr 
$45.\  (3,1,2^{t + u + 1} - 2,2^{t + 1} - 3),$& $46.\  (3,2^{t + 1} - 3,1,2^{t + u + 1} - 2),$\cr 
$47.\  (3,2^{t + 1} - 3,2^{t + u + 1} - 2,1),$& $48.\  (1,3,2^{t + 1} - 2,2^{t + u + 1} - 3),$\cr 
$49.\  (1,3,2^{t + u + 1} - 3,2^{t + 1} - 2),$& $50.\  (1,2^{t + 1} - 2,3,2^{t + u + 1} - 3),$\cr 
$51.\  (3,1,2^{t + 1} - 2,2^{t + u + 1} - 3),$& $52.\  (3,1,2^{t + u + 1} - 3,2^{t + 1} - 2),$\cr 
$53.\  (3,2^{t + u + 1} - 3,1,2^{t + 1} - 2),$& $54.\  (3,2^{t + u + 1} - 3,2^{t + 1} - 2,1),$\cr 
$55.\  (1,3,2^{t + 1} - 1,2^{t + u + 1} - 4),$& $56.\  (1,3,2^{t + u + 1} - 4,2^{t + 1} - 1),$\cr 
$57.\  (1,2^{t + 1} - 1,3,2^{t + u + 1} - 4),$& $58.\  (3,1,2^{t + 1} - 1,2^{t + u + 1} - 4),$\cr 
$59.\  (3,1,2^{t + u + 1} - 4,2^{t + 1} - 1),$& $60.\  (3,2^{t + 1} - 1,1,2^{t + u + 1} - 4),$\cr 
$61.\  (2^{t + 1} - 1,1,3,2^{t + u + 1} - 4),$& $62.\  (2^{t + 1} - 1,3,1,2^{t + u + 1} - 4),$\cr 
$63.\  (1,2,2^{t + 2} - 3,2^{t + u + 1} - 2^{t+1} -1),$& $64.\  (1,2,2^{t + 2} - 1,2^{t + u + 1} - 2^{t+1} -3),$\cr 
$65.\  (1,2^{t + 2} - 1,2,2^{t + u + 1} - 2^{t+1} -3),$& $66.\  (2^{t + 2} - 1,1,2,2^{t + u + 1} - 2^{t+1} -3),$\cr 
$67.\  (1,3,2^{t + 2} - 4,2^{t + u + 1} - 2^{t+1} -1),$& $68.\  (3,1,2^{t + 2} - 4,2^{t + u + 1} - 2^{t+1} -1),$\cr 
$69.\  (1,3,2^{t + 2} - 3,2^{t + u + 1} - 2^{t+1} -2),$& $70.\  (3,1,2^{t + 2} - 3,2^{t + u + 1} - 2^{t+1} -2),$\cr 
\end{tabular}}
\centerline{\begin{tabular}{ll}
$71.\  (3,2^{t + 2} - 3,1,2^{t + u + 1} - 2^{t+1} -2),$& $72.\  (3,2^{t + 2} - 3,2^{t + u + 1} - 2^{t+1} -2,1),$\cr 
$73.\  (1,3,2^{t + 2} - 2,2^{t + u + 1} - 2^{t+1} -3),$& $74.\  (1,2^{t + 2} - 2,3,2^{t + u + 1} - 2^{t+1} -3),$\cr 
$75.\  (3,1,2^{t + 2} - 2,2^{t + u + 1} - 2^{t+1} -3),$& $76.\  (1,3,2^{t + 2} - 1,2^{t + u + 1} - 2^{t+1} -4),$\cr 
$77.\  (1,2^{t + 2} - 1,3,2^{t + u + 1} - 2^{t+1} -4),$& $78.\  (3,1,2^{t + 2} - 1,2^{t + u + 1} - 2^{t+1} -4),$\cr 
$79.\  (3,2^{t + 2} - 1,1,2^{t + u + 1} - 2^{t+1} -4),$& $80.\  (2^{t + 2} - 1,1,3,2^{t + u + 1} - 2^{t+1} -4),$\cr 
$81.\  (2^{t + 2} - 1,3,1,2^{t + u + 1} - 2^{t+1} -4),$& $82.\  (1,6,2^{t + u + 1} - 5,2^{t + 1} - 3),$\cr 
$83.\  (3,2^{t + 1} - 3,2,2^{t + u + 1} - 3),$& $84.\  (3,2^{t + u + 1} - 3,2,2^{t + 1} - 3),$\cr 
$85.\  (1,7,2^{t + u + 1} - 5,2^{t + 1} - 4),$& $86.\  (7,1,2^{t + u + 1} - 5,2^{t + 1} - 4),$\cr 
$87.\  (7,2^{t + u + 1} - 5,1,2^{t + 1} - 4),$& $88.\  (1,7,2^{t + u + 1} - 6,2^{t + 1} - 3),$\cr 
$89.\  (7,1,2^{t + u + 1} - 6,2^{t + 1} - 3),$& $90.\  (3,3,2^{t + 1} - 4,2^{t + u + 1} - 3),$\cr 
$91.\  (3,3,2^{t + u + 1} - 3,2^{t + 1} - 4),$& $92.\  (3,2^{t + u + 1} - 3,3,2^{t + 1} - 4),$\cr 
$93.\  (3,3,2^{t + 1} - 3,2^{t + u + 1} - 4),$& $94.\  (3,3,2^{t + u + 1} - 4,2^{t + 1} - 3),$\cr 
$95.\  (3,2^{t + 1} - 3,3,2^{t + u + 1} - 4),$& $96.\  (1,6,2^{t + 2} - 5,2^{t + u + 1} - 2^{t+1} -3),$\cr 
$97.\  (3,2^{t + 2} - 3,2,2^{t + u + 1} - 2^{t+1} -3),$& $98.\  (3,4,2^{t + u + 1} - 5,2^{t + 1} - 3),$\cr 
$99.\  (1,7,2^{t + 2} - 6,2^{t + u + 1} - 2^{t+1} -3),$& $100.\  (7,1,2^{t + 2} - 6,2^{t + u + 1} - 2^{t+1} -3),$\cr 
$101.\  (1,7,2^{t + 2} - 5,2^{t + u + 1} - 2^{t+1} -4),$& $102.\  (7,1,2^{t + 2} - 5,2^{t + u + 1} - 2^{t+1} -4),$\cr 
$103.\  (7,2^{t + 2} - 5,1,2^{t + u + 1} - 2^{t+1} -4),$& $104.\  (3,3,2^{t + 2} - 4,2^{t + u + 1} - 2^{t+1} -3),$\cr 
$105.\  (3,3,2^{t + 2} - 3,2^{t + u + 1} - 2^{t+1} -4),$& $106.\  (3,2^{t + 2} - 3,3,2^{t + u + 1} - 2^{t+1} -4),$\cr 
$107.\  (3,7,2^{t + u + 1} - 7,2^{t + 1} - 4),$& $108.\  (7,3,2^{t + u + 1} - 7,2^{t + 1} - 4),$\cr 
$109.\  (3,4,2^{t + 2} - 5,2^{t + u + 1} - 2^{t+1} -3),$& $110.\  (3,5,2^{t + 2} - 6,2^{t + u + 1} - 2^{t+1} -3),$\cr 
$111.\  (3,5,2^{t + 2} - 5,2^{t + u + 1} - 2^{t+1} -4),$& $112.\  (3,7,2^{t + 2} - 7,2^{t + u + 1} - 2^{t+1} -4),$\cr 
$113.\  (7,3,2^{t + 2} - 7,2^{t + u + 1} - 2^{t+1} -4).$&\cr 
\end{tabular}}
\medskip
For $t=2$,

\medskip
\centerline{\begin{tabular}{lll}
$114.\  (1,6,7,2^{u+3} - 7),$& $115.\  (1,7,6,2^{u+3} - 7),$\cr 
$116.\  (7,1,6,2^{u+3} - 7),$& $117.\  (1,7,7,2^{u+3} - 8),$\cr 
$118.\  (7,1,7,2^{u+3} - 8),$& $119.\  (7,7,1,2^{u+3} - 8),$\cr 
$120.\  (3,4,7,2^{u+3} - 7),$& $121.\  (3,5,6,2^{u+3} - 7),$\cr 
$122.\  (3,5,7,2^{u+3} - 8),$& $123.\  (3,5,2^{u+3} - 5,4),$\cr 
$124.\  (3,5,2^{u+3} - 6,5),$& $125.\  (3,7,5,2^{u+3} - 8),$\cr 
$126.\  (7,3,5,2^{u+3} - 8),$& $127.\  (3,7,8,2^{u+3} - 11),$\cr 
$128.\  (7,3,8,2^{u+3} - 11),$& $129.\  (3,4,1,2^{u+3} - 1),$\cr 
$130.\  (3,4,2^{u+3} - 1,1),$& $131.\  (3,2^{u+3} - 1,4,1),$\cr 
$132.\  (2^{u+3}-1,3,4,1),$& $133.\  (3,7,2^{u+3} - 4,1),$\cr 
$134.\  (7,3,2^{u+3} - 4,1),$& $135.\  (7,2^{u+3} - 5,4,1),$\cr 
$136.\  (3,4,3,2^{u+3} - 3),$& $137.\  (7,7,2^{u+3} - 8,1),$\cr 
$138.\  (3,7,4,2^{u+3} - 7),$& $139.\  (7,3,4,2^{u+3} - 7),$\cr 
$140.\  (7,7,8,2^{u+3} - 15),$& $141.\  (7,7,9,2^{u+3} - 16).$\cr
\end{tabular}}

\medskip
For $t \geqslant 3$,

\medskip
\centerline{\begin{tabular}{ll}
$114.\  (1,6,2^{t + 1} - 5,2^{t + u + 1} - 3),$& $115.\  (1,7,2^{t + 1} - 6,2^{t + u + 1} - 3),$\cr 
$116.\  (7,1,2^{t + 1} - 6,2^{t + u + 1} - 3),$& $117.\  (1,7,2^{t + 1} - 5,2^{t + u + 1} - 4),$\cr 
$118.\  (7,1,2^{t + 1} - 5,2^{t + u + 1} - 4),$& $119.\  (7,2^{t + 1} - 5,1,2^{t + u + 1} - 4),$\cr 
$120.\  (3,4,2^{t + 1} - 5,2^{t + u + 1} - 3),$& $121.\  (3,5,2^{t + 1} - 6,2^{t + u + 1} - 3),$\cr 
\end{tabular}}
\centerline{\begin{tabular}{ll}
$122.\  (3,5,2^{t + 1} - 5,2^{t + u + 1} - 4),$& $123.\  (3,5,2^{t + u + 1} - 5,2^{t + 1} - 4),$\cr 
$124.\  (3,5,2^{t + u + 1} - 6,2^{t + 1} - 3),$& $125.\  (3,7,2^{t + 1} - 7,2^{t + u + 1} - 4),$\cr 
$126.\  (7,3,2^{t + 1} - 7,2^{t + u + 1} - 4).$&\cr
\end{tabular}}
\end{thms}

We prove the theorem by proving the following propositions.

\begin{props}\label{mdc8.4} The $\mathbb F_2$-vector space $(\mathbb F_2\underset {\mathcal A}\otimes R_4)_{2^{t+u+1}+ 2^{t+1} -1}$ is generated by the  elements listed in Theorem \ref{dlc8.4}.
\end{props}

The proof of this proposition is based on Theorem \ref{2.4} and the following lemma.

\begin{lems}\label{8.4.1} The following matrices are strictly inadmissible
$$\begin{pmatrix} 1&1&0&1\\ 1&1&0&0\\ 0&1&1&0\\ 0&1&0&0\\ 0&0&1&0\end{pmatrix} \quad \begin{pmatrix} 1&1&0&1\\ 1&1&0&0\\ 1&0&1&0\\ 1&0&0&0\\ 0&0&1&0\end{pmatrix} \quad \begin{pmatrix} 1&1&0&1\\ 1&1&0&0\\ 1&0&1&0\\ 0&1&0&0\\ 0&0&1&0\end{pmatrix} \quad \begin{pmatrix} 1&1&0&1\\ 1&1&0&0\\ 0&1&0&1\\ 0&0&1&0\\ 0&0&1&0\end{pmatrix} $$  $$\begin{pmatrix} 1&1&0&1\\ 1&1&0&0\\ 1&0&0&1\\ 0&0&1&0\\ 0&0&1&0\end{pmatrix} \quad \begin{pmatrix} 1&0&1&1\\ 1&0&1&0\\ 0&1&1&0\\ 0&0&1&0\\ 0&0&0&1\end{pmatrix} \quad \begin{pmatrix} 1&1&0&1\\ 1&1&0&0\\ 0&1&1&0\\ 0&1&0&0\\ 0&0&0&1\end{pmatrix} \quad \begin{pmatrix} 1&1&0&1\\ 1&1&0&0\\ 1&0&1&0\\ 1&0&0&0\\ 0&0&0&1\end{pmatrix} $$  $$\begin{pmatrix} 1&1&0&1\\ 1&0&1&0\\ 0&1&1&0\\ 0&0&1&0\\ 0&0&0&1\end{pmatrix} \quad \begin{pmatrix} 1&1&1&0\\ 1&0&1&0\\ 0&1&1&0\\ 0&0&1&0\\ 0&0&0&1\end{pmatrix} \quad \begin{pmatrix} 1&1&1&0\\ 1&1&0&0\\ 0&1&1&0\\ 0&1&0&0\\ 0&0&0&1\end{pmatrix} \quad \begin{pmatrix} 1&1&1&0\\ 1&1&0&0\\ 1&0&1&0\\ 1&0&0&0\\ 0&0&0&1\end{pmatrix} $$  $$\begin{pmatrix} 1&1&0&1\\ 1&1&0&0\\ 0&1&1&0\\ 0&0&1&0\\ 0&0&0&1\end{pmatrix} \quad \begin{pmatrix} 1&1&0&1\\ 1&1&0&0\\ 1&0&1&0\\ 0&0&1&0\\ 0&0&0&1\end{pmatrix} \quad \begin{pmatrix} 1&1&1&0\\ 1&1&0&0\\ 0&1&1&0\\ 0&0&1&0\\ 0&0&0&1\end{pmatrix} \quad \begin{pmatrix} 1&1&1&0\\ 1&1&0&0\\ 1&0&1&0\\ 0&0&1&0\\ 0&0&0&1\end{pmatrix} $$  $$\begin{pmatrix} 1&1&0&1\\ 1&1&0&0\\ 1&0&1&0\\ 0&1&0&0\\ 0&0&0&1\end{pmatrix} \quad \begin{pmatrix} 1&1&1&0\\ 1&1&0&0\\ 1&0&0&1\\ 0&1&0&0\\ 0&0&1&0\end{pmatrix} \quad \begin{pmatrix} 1&1&1&0\\ 1&1&0&0\\ 1&0&1&0\\ 0&1&0&0\\ 0&0&0&1\end{pmatrix} \quad \begin{pmatrix} 1&1&0&1\\ 1&1&0&0\\ 1&0&0&1\\ 0&1&0&0\\ 0&0&1&0\end{pmatrix}. $$
\end{lems}

\begin{proof} The monomials corresponding to the above matrices respectively are
 \begin{align*}
&(3,15,20,1), (15,3,20,1), (7,11,20,1), (3,7,24,5),  (7,3,24,5),\\ 
&(3,4,15,17), (3,15,4,17), (15,3,4,17), (3,5,14,17), (3,5,15,16),\\ 
&(3,15,5,16), (15,3,5,16), (3,7,12,17), (7,3,12,17), (3,7,13,16),
 \end{align*}
 \begin{align*}
&(7,3,13,16), (7,11,4,17), (7,11,17,4), (7,11,5,16), (7,11,16,5). 
\end{align*}
For simplicity, we prove the lemma for the matrices associated to the monomials 
\begin{align*}&(3,15,20,1), (7,11,20,1),  (7,3,24,5), (3,15,4,17),  (3,5,14,17),\\
& (15,3,5,16), (7,3,12,17), (7,3,13,16), (7,11,17,4), (7,11,5,16).
\end{align*}
By a direct computation, we have
\begin{align*}
&(3,15,20,1) = Sq^1\big((3,15,19,1) + (3,19,15,1)\big) + Sq^2\big((2,15,19,1)\\ 
&\quad + (2,19,15,1) + (5,15,15,2) \big) +Sq^4(3,15,15,2)\\ 
&\quad+Sq^8\big((3,9,15,4) + (3,12,15,1)\big) + (2,15,21,1) + (2,21,15,1)\\ 
&\quad + (3,15,17,4) + (3,9,23,4) + (3,12,23,1) \quad \text{mod  }\mathcal L_4(3;2;2;1;1),\\
&(7,11,20,1) = Sq^1(7,13,17,1)  + Sq^2\big((7,11,18,1) + (7,14,15,1)\\ 
&\quad + (7,11,17,2)\big) +Sq^4(5,14,15,1)+Sq^8(7,8,15,1)+ (5,18,15,1)\\ 
&\quad + (5,14,19,1) + (7,8,23,1) + (7,11,17,4)\quad \text{mod  }\mathcal L_4(3;2;2;1;1),\\
&(7,3,24,5) = Sq^1\big((7,3,23,5) + (7,3,19,9)\big) + Sq^2\big((7,2,23,5)\\ 
&\quad + (7,5,19,6) + (7,5,15,10) + (7,2,19,9)\big) +Sq^4\big((4,3,23,5)\\ 
&\quad+  (5,2,23,5) +(11,3,15,6) + (5,3,21,6) + (5,3,15,12)\\ 
&\quad + (3,5,15,12) \big)+Sq^8(7,3,15,6)  + (4,3,27,5) + (4,3,23,9)\\ 
&\quad + (7,2,25,5)+ (5,2,27,5)  + (5,2,23,9) + (7,3,21,8)\\ 
&\quad + (7,3,17,12) + (5,3,25,6)  + (5,3,21,10) + (5,3,19,12)\\ 
&\quad + (5,3,15,16) + (7,3,20,9) + (7,2,21,9)\quad \text{mod  }\mathcal L_4(3;2;2;1;1),\\
&(3,15,4,17) = Sq^1\big((3,19,1,15) + (3,15,1,19)\big) + Sq^2\big((5,15,2,15) \\ 
&\quad+ (2,19,1,15) + (2,15,1,19)\big) +Sq^4(3,15,2,15)\\ 
&\quad+Sq^8\big((3,12,1,15) + (3,9,4,15)\big) + (2,21,1,15) + (2,15,1,21)\\ 
&\quad + (3,12,1,23) + (3,9,4,23) + (3,15,1,20)\quad \text{mod  }\mathcal L_4(3;2;2;1;1),\\
&(3,5,14,17) = Sq^1\big((3,3,17,15) + (3,3,13,19)\big) + Sq^2\big((5,3,14,15) \\ 
&\quad+ (2,3,17,15) + (2,3,13,19)\big) +Sq^4(3,3,14,15)+Sq^8(3,5,8,15)\\ 
&\quad  + (3,4,17,15) + (2,5,17,15)  + (3,4,13,19) + (3,3,13,20)\\ 
&\quad+ (2,5,13,19) + (2,3,13,21) + (3,5,8,23)\quad \text{mod  }\mathcal L_4(3;2;2;1;1),\\
&(15,3,5,16) = Sq^1\big((19,3,3,13) + (15,3,3,17)\big)\\ 
&\quad + Sq^2\big((15,5,3,14)  + (15,2,3,17)\big) +Sq^4(15,3,3,14)\\ 
&\quad + Sq^8\big((9,3,5,14)+ (12,3,3,13)  + (11,4,3,13) + (11,3,4,13)\big)\\ 
&\quad + (9,3,5,22) + (12,3,3,21) + (11,3,4,21) + (11,4,3,21)\\ 
&\quad + (15,3,4,17) + (15,2,5,17)\quad \text{mod  }\mathcal L_4(3;2;2;1;1),
\end{align*}
\begin{align*}
&(7,3,12,17) = Sq^1\big((7,3,5,23) + (7,3,9,19)\big) + Sq^2\big((7,5,10,15)\\ 
&\quad  + (7,5,6,19) + (7,2,9,19) + (7,2,5,23)\big) +Sq^4\big((11,3,6,15)\\ 
&\quad + (5,3,12,15) + (5,3,6,21)  + (4,3,5,23) + (5,2,5,23)\big)\\ 
&\quad+Sq^8(7,3,6,15) + (7,3,8,21) + (5,3,16,15) + (5,3,12,19)\\ 
&\quad + (5,3,10,21) + (5,3,6,25) + (7,3,9,20) + ( 7,2,9,21)\\ 
&\quad + (7,3,5,24) + (4,3,9,23) + (4,3,5,27) + (7,2,5,25) \\ 
&\quad+ (5,2,9,23) + (5,2,5,27)\quad \text{mod  }\mathcal L_4(3;2;2;1;1),\\
&(7,11,17,4) = Sq^1(7,13,13,5) + Sq^2\big((7,7,19,4) + (7,11,13,6) \\ 
&\quad+ (3,11,13,10) + (7,14,11,5) + (3,14,11,9) + (7,11,14,5) \\ 
&\quad+ (3,11,14,9)\big) +Sq^4\big((11,7,13,4) + (5,7,19,4) + (5,11,13,6)\\ 
&\quad + (3,13,13,6) + (3,7,13,12) + (5,14,11,5) + (3,14,13,5)\\ 
&\quad + (5,11,14,5) + (3,13,14,5)\big)+Sq^8\big((7,7,13,4) + (7,8,11,5)\big)\\ 
&\quad + (7,7,17,8) + (7,9,19,4) + (5,11,19,4) + (5,7,19,8) \\ 
&\quad+ (5,11,17,6) + (3,17,13,6) + (3,13,17,6) + (3,7,17,12) \\ 
&\quad+ (3,7,13,16) + (7,8,19,5) + (5,18,11,5) + (3,18,13,5) \\ 
&\quad+ (3,14,17,5) + (7,11,16,5) + (5,11,18,5) \\ 
&\quad+ (3,17,14,5) + (3,13,18,5)\quad \text{mod  }\mathcal L_4(3;2;2;1;1),\\
&(7,11,5,16) = Sq^1(7,13,5,13) + Sq^2\big((7,11,5,14) + (3,11,9,14) \\ 
&\quad+ (7,14,3,13) + (7,11,6,13) + (3,11,10,13) + (7,7,2,21) \big)\\ 
&\quad +Sq^4\big((5,11,5,14) + (3,13,5,14) + (5,14,3,13) + (5,11,6,13)\\ 
&\quad + (3,7,12,13) + (3,13,6,13) + (11,7,4,13) + (5,7,2,21)\big)\\ 
&\quad+Sq^8\big((7,8,3,13) + (7,7,4,13)\big)  + (5,11,5,18) + (3,17,5,14)\\ 
&\quad + (3,13,5,18) + (7,8,3,21)  + (5,18,3,13) + (5,14,3,17)\\ 
&\quad+ (5,11,6,17) + (3,7,16,13)  + (3,7,12,17) + (3,17,6,13)\\ 
&\quad+ (3,13,6,17) + (7,11,4,17) + (7,7,8,17) + (7,9,2,21) \\ 
&\quad+ (5,11,2,21) + (5,7,2,25)\quad \text{mod  }\mathcal L_4(3;2;2;1;1).
\end{align*}
The lemma is proved.
\end{proof}

Using the results in Section \ref{7}, Lemmas \ref{5.7}, \ref{5.9}, \ref{5.10}, \ref{8.3.1}, \ref{8.3.2}, \ref{8.4.1} and Theorem \ref{2.4},  we obtain Proposition \ref{mdc8.4}.

\medskip
Now, we show that the elements listed in Theorem \ref{dlc8.4}  are linearly independent in $(\mathbb F_2\underset {\mathcal A}\otimes R_4)_{2^{t+u+1}+2^{t+1}-1}$. 

\begin{props}\label{8.4.3} For any $u \geqslant 2$, the elements $[a_{u,2,1,j}], 1 \leqslant j \leqslant 141,$  are linearly independent in $(\mathbb F_2\underset {\mathcal A}\otimes R_4)_{2^{u+3}+7}$.
\end{props}

\begin{proof} Suppose that there is a linear relation
\begin{equation}\sum_{j=1}^{141}\gamma_j[a_{u,2,1,j}] = 0, \tag {\ref{8.4.3}.1}
\end{equation}
with $\gamma_j \in \mathbb F_2$.

Applying the homomorphisms $f_1, f_2, f_3$ to  (\ref{8.4.3}.1), we obtain
\begin{align*}
&\gamma_{\{3, 129\}}w_{u,2,1,3} +  \gamma_{\{4, 130\}}w_{u,2,1,4} +   \gamma_{13}w_{u,2,1,5} +  \gamma_{14}w_{u,2,1,6} +  \gamma_{18}w_{u,2,1,8}\\
&\ +  \gamma_{19}w_{u,2,1,9} +  \gamma_{25}w_{u,2,1,10} +  \gamma_{26}w_{u,2,1,11} +  \gamma_{43}w_{u,2,1,12} +  \gamma_{29}w_{u,2,1,13}\\
&\ +  \gamma_{30}w_{u,2,1,14} +  \gamma_{\{50, 136\}}w_{u,2,1,15} +  \gamma_{63}w_{u,2,1,16} +  \gamma_{64}w_{u,2,1,17}\\
&\ +  \gamma_{74}w_{u,2,1,18} +  \gamma_{\{82, 98\}}w_{u,2,1,19} +  \gamma_{\{96, 109\}}w_{u,2,1,20} +  \gamma_{\{114, 120\}}w_{u,2,1,21} =0,\\  
&\gamma_{\{1, 36\}}w_{u,2,1,3} +  \gamma_{\{6, 131\}}w_{u,2,1,4} +  \gamma_{\{10, 59\}}w_{u,2,1,5} +  \gamma_{\{12, 133, 137\}}w_{u,2,1,6}\\
&\ +  \gamma_{\{17, 68\}}w_{u,2,1,8} +  \gamma_{22}w_{u,2,1,9} +  \gamma_{33}w_{u,2,1,10} +  \gamma_{27}w_{u,2,1,11} +  \gamma_{\{42, 94\}}w_{u,2,1,12}\\
&\ +  \gamma_{\{31, 56, 67\}}w_{u,2,1,13} +  \gamma_{56}w_{u,2,1,14} +  a_1w_{u,2,1,15} +  \gamma_{67}w_{u,2,1,16} +  \gamma_{65}w_{u,2,1,17}\\
&\ + a_2w_{u,2,1,18} +  \gamma_{88}w_{u,2,1,19} +  \gamma_{\{99, 127\}}w_{u,2,1,20} +   a_3w_{u,2,1,21} =0,\\  
&\gamma_{\{2, 37\}}w_{u,2,1,3} +  \gamma_{\{5, 38\}}w_{u,2,1,4} +  \gamma_{\{9, 58, 118\}}w_{u,2,1,5} +  \gamma_{\{11, 60, 119\}}w_{u,2,1,6}\\
&\ +  \gamma_{\{20, 78\}}w_{u,2,1,8} +  \gamma_{\{21, 79\}}w_{u,2,1,9} +  \gamma_{34}w_{u,2,1,10} +  \gamma_{35}w_{u,2,1,11}\\
&\ +  \gamma_{\{41, 93, 95, 126\}}w_{u,2,1,12} +  a_4w_{u,2,1,13} +  \gamma_{\{55, 117, 122\}}w_{u,2,1,14}\\
&\ +  a_5w_{u,2,1,15} +  \gamma_{76}w_{u,2,1,16} +  \gamma_{77}w_{u,2,1,17} +  a_6w_{u,2,1,18}\\
&\ +  \gamma_{\{57, 77, 92, 106, 125\}}w_{u,2,1,19} +  \gamma_{101}w_{u,2,1,20} +  a_7w_{u,2,1,21} = 0,    
\end{align*}
where
\begin{align*}
a_1 &= \gamma_{\{48, 56, 67, 90, 140\}},\ \
a_2 = \gamma_{\{73, 104, 128, 140\}},\ \
a_3 = \gamma_{\{56, 67, 88, 99, 115, 127, 138, 140\}}\\
a_4& = \gamma_{\{55, 76, 77, 85, 117, 122\}},\ \
a_5 = \gamma_{\{49, 55, 76, 91, 106, 112, 122, 123, 141\}},\\
a_6 &= \gamma_{\{69, 105, 106, 111, 112, 141\}},\ \
a_7 = \gamma_{\{55, 57, 76, 77, 101, 107, 112, 122, 125, 141\}}.
\end{align*}

Computing from these equalities, we obtain
\begin{equation}\begin{cases}
a_i = 0, i = 1, 2, 3, 4, 5, 6, 7,\\
\gamma_j = 0,\ j= 13, 14, 18, 19, 22, 25, 26, 27, 29, 30, 31,\\ 
33, 34, 35, 43, 56, 63, 64, 65, 67, 74, 76, 77, 88, 101,\\
\gamma_{\{3, 129\}} =   
\gamma_{\{4, 130\}} =    
\gamma_{\{50, 136\}} =   
\gamma_{\{82, 98\}} =   
\gamma_{\{114, 120\}} = 0,\\  
\gamma_{\{96, 109\}} =   
\gamma_{\{1, 36\}} =   
\gamma_{\{6, 131\}} =   
\gamma_{\{10, 59\}} =   
\gamma_{\{17, 68\}} = 0,\\  
\gamma_{\{12, 133, 137\}} =   
\gamma_{\{42, 94\}} =   
\gamma_{\{99, 127\}} =   
\gamma_{\{2, 37\}} =  
\gamma_{\{5, 38\}} = 0,\\  
\gamma_{\{9, 58, 118\}} =   
\gamma_{\{11, 60, 119\}} =   
\gamma_{\{20, 78\}} =   
\gamma_{\{41, 93, 95, 126\}} = 0,\\  
\gamma_{\{21, 79\}} =   
\gamma_{\{55, 117, 122\}} =   
\gamma_{\{57, 92, 106, 125\}} = 0.  
\end{cases}\tag{\ref{8.4.3}.2}
\end{equation}

With the aid of (\ref{8.4.3}.2), the homomorphisms $f_4, f_5,f_6$ send (\ref{8.4.3}.1) to
\begin{align*}
&\gamma_{\{1, 3\}}w_{u,2,1,1} + \gamma_{10}w_{u,2,1,2} +   \gamma_{\{16, 134, 135, 137\}}w_{u,2,1,4} +  \gamma_{\{8, 132\}}w_{u,2,1,6}\\
&\quad +  \gamma_{17}w_{u,2,1,7} +  \gamma_{24}w_{u,2,1,9} +  \gamma_{\{36, 129\}}w_{u,2,1,10} +  \gamma_{\{45, 84, 94, 98, 124\}}w_{u,2,1,11}
\end{align*}
\begin{align*}
&\quad +  \gamma_{28}w_{u,2,1,12} +  \gamma_{\{51, 83, 90, 136, 140\}}w_{u,2,1,13} +  \gamma_{59}w_{u,2,1,14} +  \gamma_{32}w_{u,2,1,15}\\
&\quad +  \gamma_{68}w_{u,2,1,16} +  a_8w_{u,2,1,17} +  \gamma_{66}w_{u,2,1,18} +  \gamma_{89}w_{u,2,1,19}\\
&\quad +  \gamma_{\{100, 128\}}w_{u,2,1,20} +  \gamma_{\{116, 139, 140\}}w_{u,2,1,21} =0,\\  
&\gamma_{\{2, 4\}}w_{u,2,1,1} + \gamma_{\{9, 55, 114, 117\}}w_{u,2,1,2} +  \gamma_{\{15, 62, 87, 103, 119\}}w_{u,2,1,4} +  \gamma_{\{7, 40\}}w_{u,2,1,6}\\
&\quad +  \gamma_{20}w_{u,2,1,7} +  \gamma_{\{23, 81\}}w_{u,2,1,9} +  \gamma_{\{37, 130\}}w_{u,2,1,10} +  \gamma_{\{44, 93, 125\}}w_{u,2,1,11}\\
&\quad +  \gamma_{39}w_{u,2,1,12} +  \gamma_{\{52, 91, 113, 118, 141\}}w_{u,2,1,13} +  \gamma_{\{58, 118, 120, 122\}}w_{u,2,1,14}\\
&\quad +  \gamma_{\{80, 86\}}w_{u,2,1,15} +  \gamma_{78}w_{u,2,1,16} +  \gamma_{\{70, 105, 113, 141\}}w_{u,2,1,17} + \gamma_{80}w_{u,2,1,18}\\
&\quad +  \gamma_{\{61, 80, 126\}}w_{u,2,1,19} +  \gamma_{102}w_{u,2,1,20} +   \gamma_{\{108, 113, 118, 141\}}w_{u,2,1,21} = 0,\\ 
&a_9w_{u,2,1,1} +  \gamma_{\{5, 6\}}w_{u,2,1,2} +  a_{10}w_{u,2,1,3} +  \gamma_{\{7, 8, 28, 39\}}w_{u,2,1,5} +  \gamma_{21}w_{u,2,1,7}\\
&\quad +  \gamma_{\{23, 24, 66, 80\}}w_{u,2,1,8} +  a_{11}w_{u,2,1,10} +  \gamma_{\{38, 131\}}w_{u,2,1,11} +  \gamma_{\{40, 132\}}w_{u,2,1,12}\\
&\quad +  a_{12}w_{u,2,1,13} +  \gamma_{\{53, 54, 84, 92\}}w_{u,2,1,14} +  a_{13}w_{u,2,1,15} +  \gamma_{\{71, 72, 97, 106\}}w_{u,2,1,16}\\
&\quad +  \gamma_{79}w_{u,2,1,17} +  \gamma_{81}w_{u,2,1,18} +  \gamma_{\{87, 135\}}w_{u,2,1,19}\\
&\quad +  \gamma_{103}w_{u,2,1,20} +  \gamma_{\{119, 137, 140, 141\}}w_{u,2,1,21} = 0, 
\end{align*}
where
\begin{align*}
a_8 &= \gamma_{\{75, 97, 104, 109, 110, 127, 140\}},\ \
a_9 = \gamma_{\{11, 12, 57, 85, 99, 115, 117\}},\\
a_{10} &= \gamma_{\{15, 16, 32, 61, 86, 89, 100, 102, 116, 118\}},\ \
a_{11} = \gamma_{\{46, 47, 83, 95, 110, 111, 121, 122, 123, 124\}},\\
a_{12} &= \gamma_{\{60, 107, 112, 125, 127, 133, 138\}},\ \
a_{13} = \gamma_{\{62, 108, 113, 126, 128, 134, 139\}}.
\end{align*}
From these equalities, we get
\begin{equation}\begin{cases}
a_i = 0,\ i = 8, 9, \ldots, 13,\
\gamma_j = 0,\ j = 10, 17, 20, 21, 23,\\ 
24, 28, 32, 39, 59, 66, 68, 78, 79, 80, 81, 86, 89, 102, 103,\\
\gamma_{\{1, 3\}} =  
\gamma_{\{16, 134, 135, 137\}} =    
\gamma_{\{8, 132\}} = \  
\gamma_{\{36, 129\}} = 0,\\  
\gamma_{\{45, 84, 94, 98, 124\}} =   
\gamma_{\{51, 83, 90, 136, 140\}} =   
\gamma_{\{116, 139, 140\}} = 0,\\  
\gamma_{\{100, 128\}} =   
\gamma_{\{2, 4\}} =  
\gamma_{\{9, 55, 114, 117\}} =   
\gamma_{\{15, 62, 87, 119\}} = 0,\\  
\gamma_{\{7, 40\}} =   
\gamma_{\{37, 130\}} =   
\gamma_{\{44, 93, 125\}} =   
\gamma_{\{52, 91, 113, 118, 141\}} = 0,\\  
\gamma_{\{58, 118, 120, 122\}} =   
\gamma_{\{70, 105, 113, 141\}} =  
\gamma_{\{108, 113, 118, 141\}} = 0,\\  
\gamma_{\{61, 126\}} =   
\gamma_{\{5, 6\}} =   
\gamma_{\{7, 8\}} =   
\gamma_{\{38, 131\}} =   
\gamma_{\{53, 54, 84, 92\}} = 0,\\  
\gamma_{\{40, 132\}} =   
\gamma_{\{71, 72, 97, 106\}} =   
\gamma_{\{87, 135\}} =  
\gamma_{\{119, 137, 140, 141\}} = 0. 
\end{cases}\tag{\ref{8.4.3}.3}
\end{equation}

With the aid of (\ref{8.4.3}.2) and (\ref{8.4.3}.3), the homomorphisms $g_1, g_2$ send (\ref{8.4.3}.1) to
\begin{align*}
&a_{14}w_{u,2,1,4} +  \gamma_{\{54, 87\}}w_{u,2,1,6} +   \gamma_{72}w_{u,2,1,9} +  \gamma_{45}w_{u,2,1,11} +  \gamma_{84}w_{u,2,1,12}\\
&\quad +  \gamma_{\{48, 51, 90, 140\}}w_{u,2,1,13} +  \gamma_{83}w_{u,2,1,15} +  \gamma_{\{73, 75, 104, 140\}}w_{u,2,1,17}\\
&\quad +  \gamma_{97}w_{u,2,1,18} +  \gamma_{124}w_{u,2,1,19} +  \gamma_{110}w_{u,2,1,20} +  a_{15}w_{u,2,1,21} = 0,\\  
\end{align*}
\begin{align*}
&\gamma_{\{9, 58, 118\}}w_{u,2,1,2} + a_{16}w_{u,2,1,4} +  \gamma_{\{53, 87\}}w_{u,2,1,6} +  \gamma_{71}w_{u,2,1,9}\\
&\quad +  \gamma_{\{41, 44, 50, 93, 95\}}w_{u,2,1,11} +  \gamma_{92}w_{u,2,1,12} +  a_{17}w_{u,2,1,13}\\
&\quad +  \gamma_{\{55, 58, 114, 117, 118\}}w_{u,2,1,14} +  \gamma_{\{85, 106, 123\}}w_{u,2,1,15} +  a_{18}w_{u,2,1,17}\\
&\quad +  \gamma_{106}w_{u,2,1,18} +  \gamma_{\{57, 92, 95, 125\}}w_{u,2,1,19} +  \gamma_{111}w_{u,2,1,20} +  a_{19}w_{u,2,1,21} = 0,       
\end{align*}
where
\begin{align*}
a_{14} &= \gamma_{\{12, 16, 47, 133, 134, 137\}},\ \
a_{15}  = \gamma_{\{115, 116, 121, 138, 139, 140\}},\\
a_{16} &= \gamma_{\{11, 15, 46, 60, 62, 119\}},\ \
a_{17} = \gamma_{\{49, 52, 82, 91, 106, 112, 113, 117, 118, 122, 123, 141\}},\\
a_{18} &= \gamma_{\{69, 70, 96, 105, 106, 111, 112, 113, 141\}},\ \
a_{19} = \gamma_{\{92, 106, 107, 108, 112, 113, 117, 118, 122, 141\}}.
\end{align*}

From these equalities, we get
\begin{equation}\begin{cases}
a_i = 0,\ i = 14, 15, \ldots, 19,\\
\gamma_{\{54, 87\}} =    
\gamma_{\{48, 51, 90, 140\}} =   
\gamma_{\{73, 75, 104, 140\}} = 0,\\   
\gamma_{\{9, 58, 118\}} =  
\gamma_{\{53, 87\}} =   
\gamma_{\{41, 44, 50, 93, 95\}} = 0,\\  
\gamma_{\{55, 58, 114, 117, 118\}} =   
\gamma_{\{85, 123\}} =   
\gamma_{\{57, 95, 125\}} = 0.  
\end{cases}\tag{\ref{8.4.3}.4}
\end{equation}

With the aid of (\ref{8.4.3}.2), (\ref{8.4.3}.3) and (\ref{8.4.3}.4), the homomorphisms $g_3, g_4$ send (\ref{8.4.3}.1) to
\begin{align*}
&\gamma_{\{11, 60, 119\}}w_{u,2,1,2} + a_{20}w_{u,2,1,4} +   \gamma_{\{12, 133, 137\}}w_{u,2,1,5} +  \gamma_{\{16, 52\}}w_{u,2,1,6}\\
&\quad +  \gamma_{\{70, 75, 100\}}w_{u,2,1,9} +  a_{21}w_{u,2,1,11} +  \gamma_{\{47, 91, 108, 134\}}w_{u,2,1,12}\\
&\quad +  a_{22}w_{u,2,1,13} +  a_{23}w_{u,2,1,14} +  a_{24}w_{u,2,1,15} +  a_{25}w_{u,2,1,17}\\
&\quad +  a_{26}w_{u,2,1,18} +  a_{27}w_{u,2,1,19} +  \gamma_{112}w_{u,2,1,20} +  a_{28}w_{u,2,1,21} = 0,\\  
&\gamma_{\{15, 62, 87, 119\}}w_{u,2,1,2} + a_{29}w_{u,2,1,4} +  \gamma_{\{16, 87, 134, 137\}}w_{u,2,1,5}\\
&\quad +  \gamma_{\{12, 42, 49, 82, 85\}}w_{u,2,1,6} +  \gamma_{\{69, 73, 96, 99\}}w_{u,2,1,9} +  a_{30}w_{u,2,1,11}\\
&\quad +  a_{31}w_{u,2,1,12} +  a_{32}w_{u,2,1,13} +  a_{33}w_{u,2,1,14} +  a_{34}w_{u,2,1,15} +  a_{35}w_{u,2,1,17}\\
&\quad +  a_{36}w_{u,2,1,18} +  a_{37}w_{u,2,1,19} +  \gamma_{113}w_{u,2,1,20} +  a_{38}w_{u,2,1,21} = 0,
\end{align*}
where
\begin{align*}
a_{20} &=  \gamma_{\{9, 15, 44, 51, 58, 61, 116, 118\}},\ \
a_{21} =  \gamma_{\{41, 46, 61, 62, 90, 93, 95, 139\}},\\
a_{22}& =  \gamma_{\{53, 69, 85, 100, 104, 105, 112, 113, 115, 117,119, 125, 133, 137, 138, 140, 141\}},\\
a_{23} &=  \gamma_{\{57, 60, 85, 99, 115, 117, 119, 133, 137\}},\ \
a_{25} =   \gamma_{\{69, 99, 100, 104, 105, 112, 113, 140, 141\}},\\
a_{24} &=  \gamma_{\{54, 73, 85, 100, 104, 105, 107, 113, 133, 140, 141\}},\
a_{26} =  \gamma_{\{73, 100, 104, 105, 113, 140, 141\}},\\
a_{27} &=  \gamma_{\{42, 48, 49, 55, 61, 62, 69, 73, 90, 91, 93, 108, 121, 122, 123, 134, 139\}},\\
a_{28} &=   \gamma_{\{42, 49, 69, 85, 91, 100, 104, 105, 107, 108, 112, 113, 115, 117, 119, 123, 125, 133, 134, 138, 140, 141\}},\\
a_{29} &=  \gamma_{\{9, 11, 41, 48, 50, 55, 57, 114, 115, 117\}},\ \
a_{30} = \gamma_{\{44, 46, 53, 60, 90, 93, 125, 138\}},\\
a_{32} &=  \gamma_{\{61, 70, 99, 104, 105, 112, 113, 116, 118, 119, 134, 137, 139, 140, 141\}},\\
a_{31} &=  \gamma_{\{42, 47, 54, 82, 91, 107, 123, 133\}},\ \
a_{33} =  \gamma_{\{61, 62, 100, 116, 118, 119, 134, 137\}},\\
a_{34} &=  \gamma_{\{75, 96, 99, 104, 105, 108, 112, 134, 140, 141\}},\ \
a_{35} =  \gamma_{\{70, 99, 100, 104, 105, 112, 113, 140, 141\}},\\
a_{36} &=  \gamma_{\{75, 96, 99, 104, 105, 112, 140, 141\}},\ \
a_{38} =  \gamma_{\{42, 52, 53, 61, 70, 87, 91, 99, 104, 105\}},
\end{align*}
\begin{align*}
a_{37} &=  \gamma_{\{42, 50, 51, 52, 53, 54, 58, 60, 70, 75, 90, 91, 93, 95, 96, 107, 114, 121, 122, 125, 133, 138\}}.
\end{align*}
These equalities imply
\begin{equation}\begin{cases}
\gamma_j = 0,\ i= 112, 113,\  a_i =0, \ i= 20, 21, \ldots , 38,\\
\gamma_{\{11, 60, 119\}} =   
\gamma_{\{12, 133, 137\}} =  
\gamma_{\{16, 52\}} = 0,\\  
\gamma_{\{70, 75, 100\}} =   
\gamma_{\{47, 91, 108, 134\}} =   
\gamma_{\{15, 62, 87, 119\}} = 0,\\  
\gamma_{\{16, 87, 134, 137\}} =   
\gamma_{\{12, 42, 49, 82, 85\}} =   
\gamma_{\{69, 73, 96, 99\}} = 0.  
\end{cases}\tag{\ref{8.4.3}.5}
\end{equation}

Substituting (\ref{8.4.3}.2), (\ref{8.4.3}.3), (\ref{8.4.3}.4) and (\ref{8.4.3}.5) into the relation (\ref{8.4.3}.1), we obtain
\begin{equation}
\sum_{i=1}^{15}b_i[\theta_i],\tag{\ref{8.4.3}.6}
\end{equation}
where $b_{1} = \gamma_{1}, b_{2} = \gamma_{2}, b_{3} = \gamma_{5}, b_{4} = \gamma_{7}, b_{5} = \gamma_{9}, b_{6} = \gamma_{11}, b_{7} = \gamma_{12}, b_{8} = \gamma_{15}, b_{9} = \gamma_{16}, b_{10} = \gamma_{41}, b_{11} = \gamma_{42}, b_{12} = \gamma_{44}, b_{13} = \gamma_{46}, b_{14} = \gamma_{48}, b_{15} = \gamma_{53} $ and
\begin{align*}
\theta_{1} &=  a_{u,2,1,1} + a_{u,2,1,3} + a_{u,2,1,36} + a_{u,2,1,129},\\
\theta_{2} &=  a_{u,2,1,2} + a_{u,2,1,4} + a_{u,2,1,37} + a_{u,2,1,130},\\
\theta_{3} &=  a_{u,2,1,5} + a_{u,2,1,6} + a_{u,2,1,38} + a_{u,2,1,131},\\
\theta_{4} &=  a_{u,2,1,7} + a_{u,2,1,8} + a_{u,2,1,40} + a_{u,2,1,132},\\
\theta_{5} &=  a_{u,2,1,9} + a_{u,2,1,58} + a_{u,2,1,114} + a_{u,2,1,120},\\
\theta_{6} &=  a_{u,2,1,11} + a_{u,2,1,60} + a_{u,2,1,115} + a_{u,2,1,138},\\
\theta_{7} &=  a_{u,2,1,12} + a_{u,2,1,49} + a_{u,2,1,55} + a_{u,2,1,107} + a_{u,2,1,117} + a_{u,2,1,133},\\
\theta_{8} &=  a_{u,2,1,15} + a_{u,2,1,62} + a_{u,2,1,116} + a_{u,2,1,139},\\
\theta_{9} &=  a_{u,2,1,16} + a_{u,2,1,52} + a_{u,2,1,58} + a_{u,2,1,108} + a_{u,2,1,118} + a_{u,2,1,134},\\
\theta_{10} &=  a_{u,2,1,41} + a_{u,2,1,57} + a_{u,2,1,73} + a_{u,2,1,93}\\
&\quad + a_{u,2,1,99} + a_{u,2,1,104} + a_{u,2,1,125} + a_{u,2,1,127},\\
\theta_{11} &=  a_{u,2,1,42} + a_{u,2,1,47} + a_{u,2,1,55} + a_{u,2,1,82} + a_{u,2,1,91} + a_{u,2,1,94}\\
&\quad + a_{u,2,1,98} + a_{u,2,1,105} + a_{u,2,1,107} + a_{u,2,1,108} + a_{u,2,1,114} + a_{u,2,1,120}\\
&\quad + a_{u,2,1,122} + a_{u,2,1,133} + a_{u,2,1,134} + a_{u,2,1,137} + a_{u,2,1,141},\\
\theta_{12} &=  a_{u,2,1,44} + a_{u,2,1,61} + a_{u,2,1,75} + a_{u,2,1,93}\\
\theta_{13} &=  a_{u,2,1,46} + a_{u,2,1,60} + a_{u,2,1,62} + a_{u,2,1,90} + a_{u,2,1,104}\\
&\quad + a_{u,2,1,119} + a_{u,2,1,121} + a_{u,2,1,140} + a_{u,2,1,139} + a_{u,2,1,140},\\
\theta_{14} &=  a_{u,2,1,48} + a_{u,2,1,50} + a_{u,2,1,57} + a_{u,2,1,58} + a_{u,2,1,61} + a_{u,2,1,73}\\
&\quad + a_{u,2,1,75} + a_{u,2,1,93} + a_{u,2,1,99} + a_{u,2,1,100} + a_{u,2,1,104} + a_{u,2,1,105}\\
&\quad + a_{u,2,1,114} + a_{u,2,1,115} + a_{u,2,1,116} + a_{u,2,1,117} + a_{u,2,1,118}\\
&\quad + a_{u,2,1,120} + a_{u,2,1,121} + a_{u,2,1,122} + a_{u,2,1,125} + a_{u,2,1,126}\\
&\quad + a_{u,2,1,127} + a_{u,2,1,128} + a_{u,2,1,136} + a_{u,2,1,140} + a_{u,2,1,141},
\end{align*}
\begin{align*}
\theta_{15} &=  a_{u,2,1,53} + a_{u,2,1,54} + a_{u,2,1,60} + a_{u,2,1,87}\\
&\quad + a_{u,2,1,119} + a_{u,2,1,133} + a_{u,2,1,135} + a_{u,2,1,137}.
\end{align*}
Now, we prove $b_i = 0$ for $ i = 1,2, \ldots , 15$. The proof is divided into 6 steps.

{\it Step 1.} First, we prove $b_1 = 0$ by showing that the polynomial $\theta = \theta_1 + \sum_{j=2}^{15}c_j\theta_j$ is non-hit for all $c_j \in \mathbb F_2, j = 2,3, \ldots, 15$. Suppose the contrary that $\theta$ is hit. Then we have
$$\theta = \sum_{m=0}^{u+2}Sq^{2^m}(A_m),$$
 for some polynomials $A_m, m= 0,1,\ldots, u+2$. Let $(Sq^2)^3$ act on the both sides of this equality. Since $(Sq^2)^3Sq^1 = 0, (Sq^2)^3Sq^2 =0$, we get
$$(Sq^2)^3(\theta) = \sum_{m=2}^{u+2}(Sq^2)^3Sq^{2^m}(A_m).$$
It is not difficult to check that the monomial $(8,4,2,2^{u+3}-1)$ is a term of $(Sq^2)^3(\theta)$. This monomial is not a term of $\sum_{m=2}^{u+2}(Sq^2)^3Sq^{2^m}(A_m)$ for all polynomials $A_m, m=2,3,\ldots, u+2$. So, we have a contradiction. 

By an argument analogous to the previous one, we get $b_2 = b_3 = b_4 = 0$. Then, the relation (\ref{8.4.3}.6) becomes
\begin{equation} \sum_{i=5}^{15}b_i[\theta_i] = 0.\tag{\ref{8.4.3}.7}
\end{equation}

{\it Step 2.}  Under the homomorphism $\varphi_4$, the linear relation (\ref{8.4.3}.7) is sent to 
\begin{equation}  b_{15}[\theta_3] + \sum_{i=5}^{15}d_i[\theta_i]  = 0, \tag{\ref{8.4.3}.8}
\end{equation}
for some $d_i \in \mathbb F_2, i = 5,6,\ldots , 15$. Using (\ref{8.4.3}.8) and by a same argument as given in Step 1, we get $b_{15} = \gamma_{53} = 0$.

Consider the homomorphism  induced by the linear map $x_1 \mapsto x_1+x_3, x_2 \mapsto x_2, x_3 \mapsto x_3, x_4 \mapsto x_4$. Under this homomorphism, the linear relation (\ref{8.4.3}.7) is sent to 
\begin{equation}  b_{9}[\theta_2] + \sum_{i=5}^{15}d_i[\theta_i]  = 0, \tag{\ref{8.4.3}.9}
\end{equation}
for some $d_i \in \mathbb F_2, i = 5,6,\ldots , 15$. Using (\ref{8.4.3}.9) and by an analogous argument as given in Step 1, we get $b_9 = \gamma_{16} = 0$.

The homomorphism  induced by the linear transformation $x_1 \mapsto x_1+x_4, x_2 \mapsto x_2, x_3 \mapsto x_3, x_4 \mapsto x_4$ sends the linear relation (\ref{8.4.3}.7) to 
\begin{equation}  b_{8}[\theta_1] + \sum_{i=5}^{15}d_i[\theta_i]  = 0, \tag{\ref{8.4.3}.10}
\end{equation}
for some $d_i \in \mathbb F_2, i = 5,6,\ldots , 15$. Using (\ref{8.4.3}.10) and by a same argument as given in Step 1, we get $b_8 = \gamma_{15}= 0$.

Applying the homomorphism  induced by the linear transformation $x_1 \mapsto x_1, x_2 \mapsto x_2+x_3, x_3 \mapsto x_3, x_4 \mapsto x_4$ to the linear relation (\ref{8.4.3}.7), we get 
\begin{equation}  b_{7}[\theta_2] + \sum_{i=5}^{15}d_i[\theta_i]  = 0, \tag{\ref{8.4.3}.11}
\end{equation}
for some $d_i \in \mathbb F_2, i = 5,6,\ldots , 15$. Using (\ref{8.4.3}.11) and by a same argument as given in Step 1, we get $b_7 = \gamma_{12} = 0$.

Consider the homomorphism  induced by the linear map $x_1 \mapsto x_1, x_2 \mapsto x_2+x_4, x_3 \mapsto x_3, x_4 \mapsto x_4$. This homomorphism sends the linear relation (\ref{8.4.3}.7)  to 
\begin{equation}  b_{6}[\theta_1] + \sum_{i=5}^{15}d_i[\theta_i]  = 0, \tag{\ref{8.4.3}.12}
\end{equation}
for some $d_i \in \mathbb F_2, i = 5,6,\ldots , 15$. Using (\ref{8.4.3}.12) and by a same argument as given in Step 1, we get $b_6 = \gamma_{11} = 0$.

The homomorphism  induced by the linear transformation $x_1 \mapsto x_1, x_2 \mapsto x_2, x_3 \mapsto x_3+x_4, x_4 \mapsto x_4$ sends the linear relation (\ref{8.4.3}.7) to 
\begin{equation} b_{5}[\theta_1] + \sum_{i=5}^{15}d_i[\theta_i]  = 0. \tag{\ref{8.4.3}.13}
\end{equation}
for some $d_i \in \mathbb F_2, i = 5,6,\ldots , 15$. Using (\ref{8.4.3}.13) and by a same argument as given in Step 1, we get $b_5 =\gamma_9 =0$.

Now, the relation (\ref{8.4.3}.7) becomes
\begin{equation} 
\sum_{i=10}^{14}b_i[\theta_i]  = 0. \tag{\ref{8.4.3}.14}
\end{equation}

{\it Step 3.} Applying the homomorphism $\varphi_4$ to (\ref{8.4.3}.14), we get
$$b_{13}[\theta_6] + b_{11}[\theta_7] + (b_{10} + b_{12} + b_{14})[\theta_{10}]  + \sum_{i=11}^{14}b_i[\theta_i]   = 0. $$

Using this equality and by a same argument as given in Step 2, we obtain $b_{11}= \gamma_{42} = 0, b_{13} = \gamma_{46} = 0.$ Then, the relation (\ref{8.4.3}.14) becomes
\begin{equation} 
b_{10}[\theta_{10}] + b_{12}[\theta_{12}] + b_{14}[\theta_{14}] = 0. \tag{\ref{8.4.3}.15}
\end{equation}

{\it Step 4.} Applying the homomorphism $\varphi_2$ to the relation (\ref{8.4.3}.15) we obtain
 $$b_{10}[\theta_{10}] + b_{12}[\theta_{13}] + b_{14}[\theta_{14}] = 0.$$

Using this equality and by an argument analogous to Step 3, we get $b_{12} = \gamma_{44} = 0$. So, the relation (\ref{8.4.3}.15) becomes
\begin{equation} 
b_{10}[\theta_{10}]  + b_{14}[\theta_{14}] = 0. \tag{\ref{8.4.3}.16}
\end{equation}

{\it Step 5.} Applying the homomorphism $\varphi_1$ to the relation (\ref{8.4.3}.16) we obtain
$$b_{10}[\theta_{12}] + b_{14}[\theta_{14}] = 0.$$

Using this equality and by an argument analogous to Step 4, we get $b_{10} = \gamma_{41} = 0$. So, the relation (\ref{8.4.3}.16) becomes
\begin{equation} 
 b_{14}[\theta_{14}] = 0. \tag{\ref{8.4.3}.17}
\end{equation}

{\it Step 6.} Applying the homomorphism $\varphi_4$ to the relation (\ref{8.4.3}.17) we obtain
$$ b_{14}[\theta_{10}] + b_{14}[\theta_{14}] = 0.$$

Using this equality and by an argument analogous to Step 5, we get $b_{14} = \gamma_{48} = 0$. 
 
The proposition is completely proved.
\end{proof}

\begin{props}\label{8.4.5} For any $t \geqslant 3$ and $u \geqslant 2$, the elements $[a_{u,t,1,j}], 1 \leqslant j \leqslant 126,$  are linearly independent in $(\mathbb F_2\underset {\mathcal A}\otimes R_4)_{2^{u+t+1}+2^{t+1} - 1}$.
\end{props}

\begin{proof} Suppose that there is a linear relation
\begin{equation}\sum_{j=1}^{126}\gamma_j[a_{u,t,1,j}] = 0, \tag {\ref{8.4.5}.1}
\end{equation}
with $\gamma_j \in \mathbb F_2$.

Apply the homomorphisms $f_1, f_2, f_3$ to  (\ref{8.4.5}.1) and we have
\begin{align*}
&\gamma_{3}w_{u,t,1,3}  +   \gamma_{4}w_{u,t,1,4}  +    \gamma_{13}w_{u,t,1,5}  +   \gamma_{14}w_{u,t,1,6}  +   \gamma_{18}w_{u,t,1,8}  +   \gamma_{19}w_{u,t,1,9}\\
&\quad  +   \gamma_{25}w_{u,t,1,10}  +   \gamma_{26}w_{u,t,1,11}  +   \gamma_{43}w_{u,t,1,12}  +   \gamma_{29}w_{u,t,1,13}  +   \gamma_{30}w_{u,t,1,14}\\
&\quad   +   \gamma_{50}w_{u,t,1,15}  +   \gamma_{63}w_{u,t,1,16}  +   \gamma_{64}w_{u,t,1,17}  +   \gamma_{74}w_{u,t,1,18}\\
&\quad  +   \gamma_{\{82, 98\}}w_{u,t,1,19}  +   \gamma_{\{96, 109\}}w_{u,t,1,20}   +   \gamma_{\{114, 120\}}w_{u,t,1,21} =0,\\   
&\gamma_{\{1, 36\}}w_{u,t,1,3}  +  \gamma_{6}w_{u,t,1,4}  +   \gamma_{\{10, 59\}}w_{u,t,1,5}  +   \gamma_{12}w_{u,t,1,6}\\
&\quad   +   \gamma_{\{17, 68\}}w_{u,t,1,8}  +   \gamma_{22}w_{u,t,1,9}  +   \gamma_{33}w_{u,t,1,10}  +   \gamma_{27}w_{u,t,1,11}\\
&\quad   +   \gamma_{\{42, 94\}}w_{u,t,1,12}  +   \gamma_{31}w_{u,t,1,13}  +   \gamma_{56}w_{u,t,1,14}  +   \gamma_{\{48, 90\}}w_{u,t,1,15}\\
&\quad   +   \gamma_{67}w_{u,t,1,16}  +   \gamma_{65}w_{u,t,1,17}  +  \gamma_{\{73, 104\}}w_{u,t,1,18} +    \gamma_{88}w_{u,t,1,19}\\
&\quad  +   \gamma_{99}w_{u,t,1,20}  +   \gamma_{\{56, 67, 88, 99, 115\}}w_{u,t,1,21}  = 0,\\
\end{align*}
\begin{align*}
&\gamma_{\{2, 37\}}w_{u,t,1,3}  +  \gamma_{\{5, 38\}}w_{u,t,1,4}  +   \gamma_{\{9, 58\}}w_{u,t,1,5}  +   \gamma_{\{11, 60\}}w_{u,t,1,6}\\
&\quad   +   \gamma_{\{20, 78\}}w_{u,t,1,8}  +   \gamma_{\{21, 79\}}w_{u,t,1,9}  +   \gamma_{34}w_{u,t,1,10}  +   \gamma_{35}w_{u,t,1,11}\\
&\quad   +   \gamma_{\{41, 93, 95, 122, 125\}}w_{u,t,1,12}  +   \gamma_{57}w_{u,t,1,13}  +   \gamma_{55}w_{u,t,1,14}\\
&\quad   +   \gamma_{\{49, 91, 92, 107, 123\}}w_{u,t,1,15}  +   \gamma_{76}w_{u,t,1,16}  +   \gamma_{77}w_{u,t,1,17}\\
&\quad   +   \gamma_{\{69, 105, 106, 111, 112\}}w_{u,t,1,18}  +   \gamma_{\{57, 77, 92, 106, 117\}}w_{u,t,1,19}  \\
&\quad +   \gamma_{101}w_{u,t,1,20}  +   \gamma_{\{55, 76, 85, 92, 101, 106, 117\}}w_{u,t,1,21}  = 0.
\end{align*}

Computing from these equalities, we obtain
\begin{equation}\begin{cases}
\gamma_j = 0,\ j = 3, 4, 6, 12, 13, 14, 18, 19, 22,\\ 
25, 26, 27, 29, 30, 31, 33, 34, 35, 43, 50, 55, 56,\\ 
57, 63, 64, 65, 67, 74, 76, 77, 88, 99, 101, 115, \\
\gamma_{\{114, 120\}} =   
\gamma_{\{82, 98\}} =    
\gamma_{\{96, 109\}} =   
\gamma_{\{1, 36\}} = 0,\\  
\gamma_{\{10, 59\}} =   
\gamma_{\{17, 68\}} =   
\gamma_{\{42, 94\}} =  
\gamma_{\{48, 90\}} = 0,\\  
\gamma_{\{73, 104\}} =   
\gamma_{\{2, 37\}} =  
\gamma_{\{5, 38\}} =   
\gamma_{\{9, 58\}} = 0,\\  
\gamma_{\{11, 60\}} =   
\gamma_{\{20, 78\}} =   
\gamma_{\{41, 93, 95, 122, 125\}} = 0,\\  
\gamma_{\{55, 76, 85, 92, 101, 106, 117\}}=
\gamma_{\{49, 91, 92, 107, 123\}} = 0,\\ 
\gamma_{\{21, 79\}} =   
\gamma_{\{69, 105, 106, 111, 112\}} = 
\gamma_{\{92, 106, 117\}} = 0.             
\end{cases}\tag{\ref{8.4.5}.2}
\end{equation}

With the aid of (\ref{8.4.5}.2), the homomorphisms $f_4, f_5,f_6$ send (\ref{8.4.5}.1) to
\begin{align*}
&\gamma_{1}w_{u,t,1,1}  +   \gamma_{10}w_{u,t,1,2}  +   \gamma_{16}w_{u,t,1,4}  +   \gamma_{8}w_{u,t,1,6}  +   \gamma_{17}w_{u,t,1,7}\\
&\quad  +   \gamma_{24}w_{u,t,1,9}  +   \gamma_{36}w_{u,t,1,10}  +   \gamma_{\{45, 84, 94, 98, 124\}}w_{u,t,1,11}  +   \gamma_{28}w_{u,t,1,12}\\
&\quad   +   \gamma_{\{51, 83, 90, 120, 121\}}w_{u,t,1,13}  +  \gamma_{59}w_{u,t,1,14}  +   \gamma_{32}w_{u,t,1,15}\\
&\quad   +    \gamma_{68}w_{u,t,1,16}  +   \gamma_{\{75, 97, 104, 109, 110\}}w_{u,t,1,17}  +   \gamma_{66}w_{u,t,1,18}\\
&\quad  +   \gamma_{89}w_{u,t,1,19}  +   \gamma_{100}w_{u,t,1,20}   +   \gamma_{116}w_{u,t,1,21} =0,\\   
&\gamma_{2}w_{u,t,1,1}  +  \gamma_{9}w_{u,t,1,2}  +   \gamma_{\{15, 62\}}w_{u,t,1,4}  +   \gamma_{\{7, 40\}}w_{u,t,1,6}  +   \gamma_{20}w_{u,t,1,7}\\
&\quad   +   \gamma_{\{23, 81\}}w_{u,t,1,9}  +   \gamma_{37}w_{u,t,1,10}  +   \gamma_{\{44, 93, 126\}}w_{u,t,1,11}\\
&\quad   +   \gamma_{39}w_{u,t,1,12}  +   \gamma_{\{52, 91, 108\}}w_{u,t,1,13}  +   \gamma_{58}w_{u,t,1,14}  +   \gamma_{61}w_{u,t,1,15}\\
&\quad   +   \gamma_{78}w_{u,t,1,16}  +   \gamma_{\{70, 105, 113\}}w_{u,t,1,17}  +   \gamma_{80}w_{u,t,1,18}\\
&\quad   +   \gamma_{\{61, 80, 118\}}w_{u,t,1,19}  +   \gamma_{102}w_{u,t,1,20}  +   \gamma_{\{61, 80, 86\}}w_{u,t,1,21} = 0,\\   
&\gamma_{11}w_{u,t,1,1}  +  \gamma_{5}w_{u,t,1,2}  +   \gamma_{\{15, 16, 32, 61\}}w_{u,t,1,3}  +   \gamma_{38}w_{u,t,1,11}\\
&\quad    +   \gamma_{\{7, 8, 28, 39\}}w_{u,t,1,5}+   \gamma_{\{23, 24, 66, 80\}}w_{u,t,1,8} +  \gamma_{\{46, 47, 83, 95\}}w_{u,t,1,10}\\
&\quad    +   \gamma_{21}w_{u,t,1,7}   +    \gamma_{40}w_{u,t,1,12}  +   \gamma_{60}w_{u,t,1,13}    +   \gamma_{\{53, 54, 84, 92\}}w_{u,t,1,14}\\
&\quad  +   \gamma_{\{71, 72, 97, 106\}}w_{u,t,1,16}    +   \gamma_{62}w_{u,t,1,15}  +   \gamma_{79}w_{u,t,1,17}  +   \gamma_{81}w_{u,t,1,18}\\
&\quad    +   \gamma_{87}w_{u,t,1,19}  +   \gamma_{103}w_{u,t,1,20}  +   \gamma_{119}w_{u,t,1,21}  = 0.
\end{align*}

From these equalities, we get
\begin{equation}\begin{cases}
\gamma_j = 0,\ j = 1, 2, 5, 7, 8, 9, 10, 11, 15, 16, 17, 20, 21,\\ 23, 24, 28, 32, 36, 37, 38, 39, 40, 58, 59, 60, 61, 62, 66, 68,\\ 78, 79, 80, 81, 86, 87, 89, 100, 102, 103, 116, 118, 119,\\ 
\gamma_{\{45, 84, 94, 98, 124\}} =  
\gamma_{\{44, 93, 126\}} =   
\gamma_{\{51, 83, 90, 120, 121\}} = 0,\\  
\gamma_{\{52, 91, 108\}} =   
\gamma_{\{75, 97, 104, 109, 110\}} =   
\gamma_{\{70, 105, 113\}} =  0,\\
\gamma_{\{46, 47, 83, 95\}} =   
\gamma_{\{53, 54, 84, 92\}} =  
\gamma_{\{71, 72, 97, 106\}} = 0.    
\end{cases}\tag{\ref{8.4.5}.3}
\end{equation}

With the aid of (\ref{8.4.5}.2) and (\ref{8.4.5}.3), the homomorphisms $g_1, g_2$ send (\ref{8.4.5}.1) to
\begin{align*}
&\gamma_{47}w_{u,t,1,4}  +   \gamma_{54}w_{u,t,1,6}  +    \gamma_{72}w_{u,t,1,9}  +   \gamma_{45}w_{u,t,1,11}\\
&\quad  +   \gamma_{84}w_{u,t,1,12}  +   \gamma_{51}w_{u,t,1,13}  +   \gamma_{83}w_{u,t,1,15}  +   \gamma_{75}w_{u,t,1,17}\\
&\quad  +   \gamma_{97}w_{u,t,1,18}  +   \gamma_{124}w_{u,t,1,19}  +   \gamma_{110}w_{u,t,1,20}  +   \gamma_{121}w_{u,t,1,21}  =0,\\   &\gamma_{46}w_{u,t,1,4}  +  \gamma_{53}w_{u,t,1,6}  +   \gamma_{71}w_{u,t,1,9}  +   a_1w_{u,t,1,11}  +   \gamma_{92}w_{u,t,1,12}\\
&\quad  +   a_2w_{u,t,1,13}  +   \gamma_{95}w_{u,t,1,15}  +   a_3w_{u,t,1,17}  +   \gamma_{106}w_{u,t,1,18}\\
&\quad  +   \gamma_{\{92, 95, 117, 122\}}w_{u,t,1,19} +   \gamma_{111}w_{u,t,1,20}  +   \gamma_{\{85, 95, 106, 123\}}w_{u,t,1,21} = 0,   
\end{align*}
where
\begin{align*}
a_1 &= \gamma_{\{41, 44, 93, 95, 114, 122, 125, 126\}},\ \
a_2 = \gamma_{\{49, 52, 82, 91, 92, 107, 108, 123\}},\\
a_3 &= \gamma_{\{69, 70, 96, 105, 106, 111, 112, 113\}}.
\end{align*}

From these equalities, we get
\begin{equation}\begin{cases}
a_i = 0,\ i = 1, 2, 3, \ \gamma_j = 0,\ j = 45, 46, 47, 51, 53, 54,\\ 71, 72, 75, 83, 84, 92, 95, 97, 106, 110, 111, 121, 124,\\
\gamma_{\{117, 122\}} = 0.
\end{cases}\tag{\ref{8.4.5}.4}
\end{equation}

With the aid of (\ref{8.4.5}.2), (\ref{8.4.5}.3) and (\ref{8.4.5}.4), the homomorphisms $g_3, g_4$ send (\ref{8.4.5}.1) to
\begin{align*}
&\gamma_{44}w_{u,t,1,4}  +  \gamma_{52}w_{u,t,1,6}  +   \gamma_{70}w_{u,t,1,9}  +   a_4w_{u,t,1,11}  +   \gamma_{\{91, 108\}}w_{u,t,1,12}\\
&\quad  +   \gamma_{\{93, 126\}}w_{u,t,1,15}  +   \gamma_{\{69, 73, 105, 112, 113\}}w_{u,t,1,17}  +  \gamma_{\{105, 113\}}w_{u,t,1,18}\\
&\quad +   a_5w_{u,t,1,13}  +    a_6w_{u,t,1,19}  +   \gamma_{112}w_{u,t,1,20}    +   a_7w_{u,t,1,21} = 0,\\   
&\gamma_{\{41, 48, 114, 117\}}w_{u,t,1,4}  +  \gamma_{\{42, 49, 82, 85\}}w_{u,t,1,6}  +   \gamma_{\{69, 73, 96\}}w_{u,t,1,9}\\
&\quad  +   \gamma_{\{44, 48, 93, 125, 126\}}w_{u,t,1,11}  +   \gamma_{\{42, 82, 91, 107, 123\}}w_{u,t,1,12}\\
&\quad  +   \gamma_{\{42, 52, 91, 107, 108\}}w_{u,t,1,13}   +   \gamma_{\{48, 93, 114, 122, 125\}}w_{u,t,1,15}\\
&\quad  +   \gamma_{\{70, 73, 105, 112, 113\}}w_{u,t,1,17}   +   \gamma_{\{73, 96, 105, 112\}}w_{u,t,1,18}\\
&\quad +  a_8w_{u,t,1,19} +   \gamma_{113}w_{u,t,1,20} +   a_9w_{u,t,1,21} = 0,      
\end{align*}
where
\begin{align*}
a_4 &= \gamma_{\{41, 48, 93, 122, 125, 126\}},\ \
a_5 = \gamma_{\{42, 49, 91, 107, 108, 123\}},\\
a_6 &= \gamma_{\{42, 49, 69, 73, 85, 91, 93, 108, 117, 123, 125, 126\}},\ \
a_7 = \gamma_{\{85, 93, 105, 107, 113, 126\}},\\
a_8 &=  \gamma_{\{42, 48, 52, 70, 91, 93, 96, 107, 114, 122, 126\}},\ \
a_9 = \gamma_{\{48, 73, 93, 96, 105, 108, 112, 114, 122, 125\}}.
\end{align*}

These equalities imply
\begin{equation}\begin{cases}
a_i =0, \ i= 4, 5, \ldots , 9,\\
\gamma_j = 0,\ j=  44, 52, 69, 70, 73, 96, 105, 112, 113, \\
\gamma_{\{91, 108\}} =   
\gamma_{\{93, 126\}} =   
\gamma_{\{41, 48, 114, 117\}} =  0,\\ 
\gamma_{\{42, 49, 82, 85\}} =  
\gamma_{\{48, 93, 125, 126\}} =   
\gamma_{\{42, 82, 91, 107, 123\}} =  0,\\ 
\gamma_{\{42, 91, 107, 108\}} =   
\gamma_{\{48, 93, 114, 122, 125\}} = 0.      
\end{cases}\tag{\ref{8.4.5}.5}
\end{equation}

Combining (\ref{8.4.5}.2), (\ref{8.4.5}.3), (\ref{8.4.5}.4) and (\ref{8.4.5}.5), we get $\gamma_j = 0$ for all $j$. So, the proposition is proved.
\end{proof}

\begin{rems}\label{8.4.4} By a direct computation, we see that the $\mathbb F_2$-subspace of $(\mathbb F_2\underset{\mathcal A} \otimes P_4)_{2^{u+3}+7}$ generated by $[\theta_{i}], i = 1, 2, \ldots , 15$, which are defined as in the proof of Proposition \ref{8.4.3}, is an $GL_4(\mathbb F_2)$-submodule of $(\mathbb F_2\underset{\mathcal A} \otimes P_4)_{2^{u+3}+7}$.
\end{rems}

\subsection{The case $s \geqslant 2, t = 1$ and $u=1$}\label{8.5}\

\medskip
 According to Kameko \cite{ka}, for $s\geqslant 2$, $\dim (\mathbb F_2\underset{\mathcal A}\otimes P_3)_{2^{s+2}+2^{s+1} + 2^s -3} = 8$ with a basis given by the classes $w_{1,1,s,j}, 1\leqslant j \leqslant 8,$ which are determined as follows:

\medskip
\centerline{\begin{tabular}{lll}
$1.\ [2^{s}-1,2^{s+1} - 1,2^{s+2} - 1],$& $2.\ [2^{s}-1,2^{s+2} - 1,2^{s+1} - 1],$\cr 
$3.\ [2^{s+1}-1,2^{s} - 1,2^{s+2} - 1],$& $4.\ [2^{s+1}-1,2^{s+2} - 1,2^{s} - 1],$\cr 
$5.\ [2^{s+2}-1,2^{s} - 1,2^{s+1} - 1],$& $6.\ [2^{s+2}-1,2^{s+1} - 1,2^{s} - 1],$\cr 
$7.\ [2^{s+1}-1,2^{s+1} - 1,2^{s+1}+2^{s} - 1],$& $8.\ [2^{s+1}-1,2^{s+1}+2^{s} - 1,2^{s+1} - 1].$\cr 
\end{tabular}}

\medskip
So, we easily obtain

\begin{props}\label{8.5.3} For any positive integer $s \geqslant 2$, we have
$$\dim (\mathbb F_2\underset{\mathcal A}\otimes Q_4)_{2^{s+2}+ 2^{s+1} + 2^s -3} = 32.$$
\end{props}

\medskip
By Proposition \ref{2.7}, we need only to determine $(\mathbb F_2\underset{\mathcal A}\otimes R_4)_{2^{s+2}+ 2^{s+1} + 2^s -3}$. The main result of this subsection is the following.

\begin{thms}\label{dlc8.5} For $s\geqslant 2$,  $(\mathbb F_2\underset{\mathcal A}\otimes R_4)_{2^{s+2}+ 2^{s+1} + 2^s -3}$ is  an $\mathbb F_2$-vector space of dimension 88 with a basis consisting of all the classes represented by the monomials $a_{1,1,s,j}, 1 \leqslant j \leqslant 88$, which are determined as follows:

\smallskip
For $s \geqslant 2$,

\medskip
\centerline{\begin{tabular}{ll}
$1.\  (1,2^{s} - 2,2.2^{s} - 1,4.2^{s} - 1),$& $2.\  (1,2^{s} - 2,4.2^{s} - 1,2.2^{s} - 1),$\cr 
$3.\  (1,2.2^{s} - 1,2^{s} - 2,4.2^{s} - 1),$& $4.\  (1,2.2^{s} - 1,4.2^{s} - 1,2^{s} - 2),$\cr 
$5.\  (1,4.2^{s} - 1,2^{s} - 2,2.2^{s} - 1),$& $6.\  (1,4.2^{s} - 1,2.2^{s} - 1,2^{s} - 2),$\cr 
$7.\  (2.2^{s} - 1,1,2^{s} - 2,4.2^{s} - 1),$& $8.\  (2.2^{s} - 1,1,4.2^{s} - 1,2^{s} - 2),$\cr 
$9.\  (2.2^{s} - 1,4.2^{s} - 1,1,2^{s} - 2),$& $10.\  (4.2^{s} - 1,1,2^{s} - 2,2.2^{s} - 1),$\cr 
$11.\  (4.2^{s} - 1,1,2.2^{s} - 1,2^{s} - 2),$& $12.\  (4.2^{s} - 1,2.2^{s} - 1,1,2^{s} - 2),$\cr 
$13.\  (1,2^{s} - 1,2.2^{s} - 2,4.2^{s} - 1),$& $14.\  (1,2^{s} - 1,4.2^{s} - 1,2.2^{s} - 2),$\cr 
$15.\  (1,2.2^{s} - 2,2^{s} - 1,4.2^{s} - 1),$& $16.\  (1,2.2^{s} - 2,4.2^{s} - 1,2^{s} - 1),$\cr 
$17.\  (1,4.2^{s} - 1,2^{s} - 1,2.2^{s} - 2),$& $18.\  (1,4.2^{s} - 1,2.2^{s} - 2,2^{s} - 1),$\cr 
$19.\  (2^{s} - 1,1,2.2^{s} - 2,4.2^{s} - 1),$& $20.\  (2^{s} - 1,1,4.2^{s} - 1,2.2^{s} - 2),$\cr 
$21.\  (2^{s} - 1,4.2^{s} - 1,1,2.2^{s} - 2),$& $22.\  (4.2^{s} - 1,1,2^{s} - 1,2.2^{s} - 2),$\cr 
$23.\  (4.2^{s} - 1,1,2.2^{s} - 2,2^{s} - 1),$& $24.\  (4.2^{s} - 1,2^{s} - 1,1,2.2^{s} - 2),$\cr 
$25.\  (1,2^{s} - 1,2.2^{s} - 1,4.2^{s} - 2),$& $26.\  (1,2^{s} - 1,4.2^{s} - 2,2.2^{s} - 1),$\cr 
$27.\  (1,2.2^{s} - 1,2^{s} - 1,4.2^{s} - 2),$& $28.\  (1,2.2^{s} - 1,4.2^{s} - 2,2^{s} - 1),$\cr 
$29.\  (1,4.2^{s} - 2,2^{s} - 1,2.2^{s} - 1),$& $30.\  (1,4.2^{s} - 2,2.2^{s} - 1,2^{s} - 1),$\cr 
$31.\  (2^{s} - 1,1,2.2^{s} - 1,4.2^{s} - 2),$& $32.\  (2^{s} - 1,1,4.2^{s} - 2,2.2^{s} - 1),$\cr 
$33.\  (2^{s} - 1,2.2^{s} - 1,1,4.2^{s} - 2),$& $34.\  (2.2^{s} - 1,1,2^{s} - 1,4.2^{s} - 2),$\cr 
$35.\  (2.2^{s} - 1,1,4.2^{s} - 2,2^{s} - 1),$& $36.\  (2.2^{s} - 1,2^{s} - 1,1,4.2^{s} - 2),$\cr 
$37.\  (1,2.2^{s} - 2,2.2^{s} - 1,3.2^{s} - 1),$& $38.\  (1,2.2^{s} - 2,3.2^{s} - 1,2.2^{s} - 1),$\cr 
$39.\  (1,2.2^{s} - 1,2.2^{s} - 2,3.2^{s} - 1),$& $40.\  (1,2.2^{s} - 1,3.2^{s} - 1,2.2^{s} - 2),$\cr 
$41.\  (2.2^{s} - 1,1,2.2^{s} - 2,3.2^{s} - 1),$& $42.\  (2.2^{s} - 1,1,3.2^{s} - 1,2.2^{s} - 2),$\cr 
$43.\  (2.2^{s} - 1,3.2^{s} - 1,1,2.2^{s} - 2),$& $44.\  (1,2.2^{s} - 1,2.2^{s} - 1,3.2^{s} - 2),$\cr 
$45.\  (1,2.2^{s} - 1,3.2^{s} - 2,2.2^{s} - 1),$& $46.\  (2.2^{s} - 1,1,2.2^{s} - 1,3.2^{s} - 2),$\cr 
$47.\  (2.2^{s} - 1,1,3.2^{s} - 2,2.2^{s} - 1),$& $48.\  (2.2^{s} - 1,2.2^{s} - 1,1,3.2^{s} - 2),$\cr 
$49.\  (3,2.2^{s} - 3,2^{s} - 2,4.2^{s} - 1),$& $50.\  (3,2.2^{s} - 3,4.2^{s} - 1,2^{s} - 2),$\cr 
$51.\  (3,4.2^{s} - 1,2.2^{s} - 3,2^{s} - 2),$& $52.\  (4.2^{s} - 1,3,2.2^{s} - 3,2^{s} - 2),$\cr 
$53.\  (3,2.2^{s} - 1,4.2^{s} - 3,2^{s} - 2),$& $54.\  (3,4.2^{s} - 3,2^{s} - 2,2.2^{s} - 1),$\cr 
$55.\  (3,4.2^{s} - 3,2.2^{s} - 1,2^{s} - 2),$& $56.\  (2.2^{s} - 1,3,4.2^{s} - 3,2^{s} - 2),$\cr 
$57.\  (3,2^{s} - 1,2.2^{s} - 3,4.2^{s} - 2),$& $58.\  (3,2.2^{s} - 3,2^{s} - 1,4.2^{s} - 2),$\cr 
$59.\  (3,2.2^{s} - 3,4.2^{s} - 2,2^{s} - 1),$& $60.\  (3,2^{s} - 1,4.2^{s} - 3,2.2^{s} - 2),$\cr 
$61.\  (3,4.2^{s} - 3,2^{s} - 1,2.2^{s} - 2),$& $62.\  (3,4.2^{s} - 3,2.2^{s} - 2,2^{s} - 1),$\cr 
$63.\  (3,2.2^{s} - 3,2.2^{s} - 2,3.2^{s} - 1),$& $64.\  (3,2.2^{s} - 3,3.2^{s} - 1,2.2^{s} - 2),$\cr 
$65.\  (3,2.2^{s} - 3,2.2^{s} - 1,3.2^{s} - 2),$& $66.\  (3,2.2^{s} - 3,3.2^{s} - 2,2.2^{s} - 1),$\cr 
$67.\  (3,2.2^{s} - 1,2.2^{s} - 3,3.2^{s} - 2),$& $68.\  (2.2^{s} - 1,3,2.2^{s} - 3,3.2^{s} - 2),$\cr 
$69.\  (3,2.2^{s} - 1,3.2^{s} - 3,2.2^{s} - 2),$& $70.\  (2.2^{s} - 1,3,3.2^{s} - 3,2.2^{s} - 2),$\cr 
$71.\  (7,4.2^{s} - 5,2.2^{s} - 3,2^{s} - 2).$& \cr 
\end{tabular}}

\medskip
For $s = 2$,

\medskip
\centerline{\begin{tabular}{lll}
$72.\ (3,3,4,15),$& $73.\ (3,3,15,4),$& $74.\ (3,15,3,4),$\cr 
$75.\ (15,3,3,4),$& $76.\ (3,3,7,12),$& $77.\ (3,3,12,7),$\cr 
$78.\ (3,7,3,12),$& $79.\ (7,3,3,12),$& $80.\ (7,7,9,2),$\cr 
$81.\ (3,7,11,4),$& $82.\ (7,3,11,4),$& $83.\ (7,11,3,4),$\cr 
$84.\ (7,9,3,6),$& $85.\ (3,7,7,8),$& $86.\ (7,3,7,8),$\cr 
$87.\ (7,7,3,8),$& $88.\ (7,7,8,3).$& \cr
\end{tabular}}

\medskip
For $s\geqslant 3$,

\medskip
\centerline{\begin{tabular}{ll}
$72.\  (3,2^{s} - 3,2.2^{s} - 2,4.2^{s} - 1),$& $73.\  (3,2^{s} - 3,4.2^{s} - 1,2.2^{s} - 2),$\cr 
$74.\  (3,4.2^{s} - 1,2^{s} - 3,2.2^{s} - 2),$& $75.\  (4.2^{s} - 1,3,2^{s} - 3,2.2^{s} - 2),$\cr 
$76.\  (3,2^{s} - 3,2.2^{s} - 1,4.2^{s} - 2),$& $77.\  (3,2^{s} - 3,4.2^{s} - 2,2.2^{s} - 1),$\cr 
$78.\  (3,2.2^{s} - 1,2^{s} - 3,4.2^{s} - 2),$& $79.\  (2.2^{s} - 1,3,2^{s} - 3,4.2^{s} - 2),$\cr 
$80.\  (2^{s} - 1,3,2.2^{s} - 3,4.2^{s} - 2),$& $81.\  (2^{s} - 1,3,4.2^{s} - 3,2.2^{s} - 2),$\cr 
$82.\  (7,2.2^{s} - 5,2^{s} - 3,4.2^{s} - 2),$& $83.\  (7,4.2^{s} - 5,2^{s} - 3,2.2^{s} - 2),$\cr 
$84.\  (7,2.2^{s} - 5,4.2^{s} - 3,2^{s} - 2),$& $85.\  (7,2.2^{s} - 5,2.2^{s} - 3,3.2^{s} - 2),$\cr 
$86.\  (7,2.2^{s} - 5,3.2^{s} - 3,2.2^{s} - 2).$& \cr
\end{tabular}}
\centerline{\begin{tabular}{ll}
\end{tabular}}

\medskip
For $s=3$,
$$87.\ (7,7,9,30),\quad 88.\ (7,7,25,14).$$

For $s \geqslant 4$,
$$87.\  (7,2^{s} - 5,2.2^{s} - 3,4.2^{s} - 2),\quad 88.\  (7,2^{s} - 5,4.2^{s} - 3,2.2^{s} - 2).$$ 
\end{thms}

We prove the theorem by proving the following.

\begin{props}\label{mdc8.5} The $\mathbb F_2$-vector space $(\mathbb F_2\underset {\mathcal A}\otimes R_4)_{2^{s+2}+ 2^{s+1}+2^s -3}$ is generated by the  elements listed in Theorem \ref{dlc8.5}.
\end{props}

The proof of this proposition is based on the following lemmas.

\begin{lems}\label{8.5.1} The following matrices are strictly inadmissible
 $$ \begin{pmatrix} 1&1&0&1\\ 1&0&1&1\\ 1&0&0&1\\ 0&1&0&0\end{pmatrix} \quad  \begin{pmatrix} 1&1&1&0\\ 1&0&1&1\\ 1&0&1&0\\ 0&1&0&0\end{pmatrix} \quad  \begin{pmatrix} 1&1&0&1\\ 1&0&1&1\\ 1&0&1&0\\ 0&1&0&0\end{pmatrix} \quad  \begin{pmatrix} 1&1&0&1\\ 1&1&0&1\\ 0&1&0&1\\ 0&0&1&0\end{pmatrix} $$    
$$ \begin{pmatrix} 1&1&0&1\\ 1&1&0&1\\ 1&0&0&1\\ 0&0&1&0\end{pmatrix} \quad  \begin{pmatrix} 1&0&1&1\\ 1&0&1&1\\ 1&0&0&1\\ 0&1&0&0\end{pmatrix} \quad  \begin{pmatrix} 1&0&1&1\\ 1&0&1&1\\ 1&0&1&0\\ 0&1&0&0\end{pmatrix} .$$
\end{lems}

\begin{proof} The monomials corresponding to the above matrices respectively are \begin{align*}
&(7,9,2,7), (7,9,7,2),  (7,9,6,3), (3,7,8,7), \\
&(7,3,8,7), (7,8,3,7), (7,8,7,3).
\end{align*} 
We prove the lemma for the matrices associated with the monomials 
$$(7,9,2,7),   (7,9,6,3), (7,3,8,7), (7,8,3,7).$$
 By a direct computation, we have
\begin{align*}
&(7,9,2,7) = Sq^1(7,7,3,7)+ Sq^2(7,7,2,7)  +Sq^4\big((5,7,2,7) + (4,7,3,7) \big)  \\ 
&\quad+  (5,11,2,7) + (5,7,2,11) + (7,7,2,9) + (4,11,3,7) \\ 
&\quad+ (4,7,3,11) + (7,8,3,7) + (7,7,3,8)\quad \text{mod  }\mathcal L_4(3;3;2;1),
\end{align*}
\begin{align*}
&(7,9,6,3) = Sq^1(7,7,5,5)+ Sq^2\big((7,7,6,3) + (7,7,3,6)\big)  \\ 
&\quad+Sq^4\big((5,7,6,3) + (5,7,3,6) \big) +  (7,7,8,3) \\ 
&\quad + (5,11,6,3) + (5,7,10,3) + (5,11,3,6)+ (5,7,3,10)\\ 
&\quad + (7,9,3,6) + (7,7,3,8)\ \text{mod  }\mathcal L_4(3;3;2;1),\\
&(7,3,8,7) = Sq^1(7,3,7,7)+ Sq^2(7,2,7,7)  +Sq^4\big((4,3,7,7)\\ 
&\quad + (5,2,7,7)\big)+  (5,2,11,7) + (5,2,7,11) + (7,2,9,7) + (7,2,7,9)\\ 
&\quad + (7,3,7,8) + (4,3,11,7) + (4,3,7,11) \quad \text{mod  }\mathcal L_4(3;3;2;1),\\
&(7,8,3,7) = Sq^1(7,5,5,7)+ Sq^2\big((7,6,3,7) + (7,3,6,7)\big)\\ 
&\quad  +Sq^4\big((5,6,3,7) + (5,3,6,7)\big)+  (5,10,3,7)\\ 
&\quad + (5,6,3,11) + (7,6,3,9) + (7,3,8,7)+ (7,3,6,9)\\ 
&\quad  + (5,3,10,7) + (5,3,6,11) \quad \text{mod  }\mathcal L_4(3;3;2;1).
\end{align*}

The lemma is proved.
\end{proof}

\begin{lems}\label{8.5.2} The following matrices are strictly inadmissible
$$\begin{pmatrix} 1&1&1&0\\ 1&1&0&1\\ 1&1&0&1\\ 1&1&0&0\\ 0&0&1&0\end{pmatrix} \quad \begin{pmatrix} 1&1&1&0\\ 1&1&1&0\\ 1&1&1&0\\ 0&1&1&0\\ 0&0&0&1\end{pmatrix} \quad \begin{pmatrix} 1&1&1&0\\ 1&1&1&0\\ 1&1&1&0\\ 1&0&1&0\\ 0&0&0&1\end{pmatrix} $$  
$$\begin{pmatrix} 1&1&1&0\\ 1&1&1&0\\ 1&1&1&0\\ 1&1&0&0\\ 0&0&0&1\end{pmatrix} \quad \begin{pmatrix} 1&1&0&1\\ 1&1&0&1\\ 1&1&0&1\\ 1&1&0&0\\ 0&0&1&0\end{pmatrix} . $$
\end{lems}

\begin{proof} The monomials corresponding to the above matrices respectively are $$(15,15,17,6), (7,15,15,16),  (15,7,15,16), (15,15,7,16), (15,15,16,7).$$ 
 By a direct computation, we have
\begin{align*}
&(15,15,17,6) = Sq^1\big((15,15,15,7) + (15,11,15,11)\big)+ Sq^2\big((15,15,15,6) \\ 
&\quad+ (15,11,15,10)\big)  +Sq^4\big((15,13,15,6) + (15,12,15,7)\big) \\ 
&\quad+ Sq^8\big((9,15,15,6) + (11,13,15,6) + (8,15,15,7) + (11,12,15,7)\big)\\ 
&\quad+  (9,23,15,6) + (9,15,23,6) + (15,13,19,6) + (11,21,15,6)\\ 
&\quad+ (11,13,23,6)  + (8,23,15,7) + (8,15,23,7) + (15,12,19,7)\\ 
&\quad +(15,15,16,7) +(11,20,15,7) + (11,12,23,7)\quad \text{mod  }\mathcal L_4(3;3;3;2;1),
\end{align*}
\begin{align*}
&(15,7,15,16) = Sq^1(15,7,15,15)+ Sq^2(15,7,14,15)  +Sq^4\big((15,4,15,15)\\ 
&\quad + (15,5,14,15)\big) + Sq^8\big((8,7,15,15) + (11,4,15,15) + (9,7,14,15)\\ 
&\quad + (11,5,14,15)\big)+  (8,7,23,15) + (8,7,15,23) + (15,4,19,15)\\ 
&\quad + (15,4,15,19)+ (11,4,23,15)  + (11,4,15,23) + (9,7,22,15)\\ 
&\quad + (9,7,14,23) +(15,7,14,17) +(11,5,22,15) + (11,5,14,23) \\ 
&\quad+ (15,5,18,15) + (15,5,14,19)\quad \text{mod  }\mathcal L_4(3;3;3;2;1),\\
&(7,15,15,16) = \overline{\varphi}_1(15,7,15,16),\\
&(15,15,7,16) = Sq^1\big((15,15,7,15) +(15,11,11,15)\big)+ Sq^2\big((15,15,6,15) \\ 
&\quad +(15,11,10,15)\big)+Sq^4\big((15,12,7,15) + (15,13,6,15)\big) \\ 
&\quad+ Sq^8\big((8,15,7,15) + (11,12,7,15) + (9,15,6,15)\\ 
&\quad + (11,13,6,15)\big)+  (8,23,7,15) + (8,15,7,23)\\ 
&\quad + (15,12,7,19) + (11,20,7,15)+ (11,12,7,23) \\ 
&\quad + (15,15,6,17) + (9,23,6,15) + (9,15,6,23) +(15,13,6,19)\\ 
&\quad +(11,21,6,15) + (11,13,6,23) \quad \text{mod }\mathcal L_4(3;3;3;2;1),\\
&(15,15,16,7) = Sq^1\big((15,15,13,9) +(15,15,11,11) + (15,15,9,13)\\ 
&\quad + (15,11,13,13)\big)+ Sq^2\big((15,15,14,7)  +(15,15,11,10)\\ 
&\quad + (15,15,10,11) + (15,15,7,14) + (15,11,14,11) \\ 
&\quad+ (15,11,11,14)\big)+Sq^4\big((15,13,14,7) + (15,13,7,14)\big)\\ 
&\quad + Sq^8\big((9,15,14,7) + (11,13,14,7) + (9,15,7,14)\\ 
&\quad + (11,13,7,14)\big)+  (9,23,14,7) + (9,15,22,7)\\ 
&\quad + (15,13,18,7) + (11,21,14,7)+ (11,13,22,7)  + (15,15,7,16)\\ 
&\quad + (9,23,7,14) + (9,15,7,22) +(15,13,7,18)\\ 
&\quad +(11,21,7,14) + (11,13,7,22) \quad \text{mod  }\mathcal L_4(3;3;3;2;1).
\end{align*}
The lemma is proved.
\end{proof}

Combining Lemmas \ref{3.2}, \ref{3.3}, \ref{5.7}, \ref{6.2.1}, \ref{6.2.2}, \ref{8.5.1}, \ref{8.5.2}, Theorem \ref{2.4} and the results in Section \ref{7},   we get Proposition \ref{mdc8.5}.

\medskip
Now, we prove that the classes listed in Theorem \ref{dlc8.5} are linearly independent. 

\begin{props}\label{8.5.4} The elements $[a_{1,1,2,j}], 1 \leqslant j \leqslant 88,$  are linearly independent in $(\mathbb F_2\underset {\mathcal A}\otimes R_4)_{25}$.
\end{props}

\begin{proof} Suppose that there is a linear relation
\begin{equation}\sum_{j=1}^{88}\gamma_j[a_{1,1,2,j}] = 0, \tag {\ref{8.5.4}.1}
\end{equation}
with $\gamma_j \in \mathbb F_2$.

Applying the homomorphisms $f_1, f_2, \ldots, f_6$ to the relation (\ref{8.5.4}.1), we obtain
\begin{align*}
&\gamma_{1}[3,7,15]  +   \gamma_{2}[3,15,7]  +    \gamma_{15}[7,3,15]  +   \gamma_{16}[7,15,3]\\
&\quad  +   \gamma_{29}[15,3,7]  +   \gamma_{30}[15,7,3]  +   \gamma_{37}[7,7,11]  +   \gamma_{38}[7,11,7]  = 0,\\   
&\gamma_{3}[3,7,15]  +  \gamma_{5}[3,15,7]  +   \gamma_{\{13, 72\}}[7,3,15]  +   \gamma_{18}[7,15,3]  +   \gamma_{45}[7,11,7]\\
&\quad  +   \gamma_{\{26, 77\}}[15,3,7]  +   \gamma_{\{28, 88\}}[15,7,3]  +   \gamma_{\{39, 45\}}[7,7,11]  = 0,\\   
&\gamma_{4}[3,7,15]  +   \gamma_{6}[3,15,7]  +   \gamma_{\{14, 73\}}[7,3,15]  +   \gamma_{\{17, 74\}}[7,15,3]\\
&\quad  +   \gamma_{\{25, 76, 86\}}[15,3,7]  +   \gamma_{\{27, 78, 87\}}[15,7,3]\\
&\quad  +   \gamma_{\{40, 44, 81, 85\}}[7,7,11]  +   \gamma_{\{44, 85\}}[7,11,7]  = 0,\\  
&\gamma_{\{19, 49, 72\}}[3,7,15]   +   \gamma_{7}[7,3,15]  +   \gamma_{\{35, 88\}}[7,15,3]  +   \gamma_{10}[15,3,7]\\
&\quad +  \gamma_{\{32, 54, 66, 77\}}[3,15,7] +   \gamma_{23}[15,7,3]  +  \gamma_{41}[7,7,11]  +   \gamma_{47}[7,11,7]  = 0,\\    
&\gamma_{\{20, 50, 73\}}[3,7,15]  +  \gamma_{\{31, 55, 65, 76, 85\}}[3,15,7]\\
&\quad  +   \gamma_{8}[7,3,15]  +   \gamma_{\{34, 79, 83, 84, 87\}}[7,15,3]  +   \gamma_{11}[15,3,7]\\
&\quad  +   \gamma_{\{22, 75\}}[15,7,3]  +   \gamma_{\{42, 82\}}[7,7,11]  +   \gamma_{\{46, 86\}}[7,11,7]  = 0,\\   
&\gamma_{\{33, 53, 67, 69, 78, 81, 85\}}[3,7,15]  +   \gamma_{\{36, 56, 68, 70, 79, 82, 86\}}[7,3,15]\\
&\quad  +   \gamma_{9}[7,15,3] +  \gamma_{\{21, 51, 74\}}[3,15,7]  +   \gamma_{\{24, 52, 75\}}[15,3,7]\\
&\quad  +   \gamma_{12}[15,7,3]  +   \gamma_{\{48, 80, 87, 88\}}[7,7,11]  +   \gamma_{\{43, 71, 83\}}[7,11,7]  = 0.  
\end{align*}

Computing from these equalities, we obtain
\begin{equation}\begin{cases}
\gamma_j = 0, \ j = 1, 2, \ldots , 12,\\ \quad 15, 16, 18, 23, 29, 30, 37, 38, 39, 41, 45, 47,\\
\gamma_{\{13, 72\}} =   
\gamma_{\{26, 77\}} =    
\gamma_{\{28, 88\}} =   
\gamma_{\{24, 52, 75\}} =   
\gamma_{\{17, 74\}} = 0,\\ 
\gamma_{\{25, 76, 86\}} =   
\gamma_{\{27, 78, 87\}} =   
\gamma_{\{40, 44, 81, 85\}} =   
\gamma_{\{44, 85\}} =   0,\\
\gamma_{\{19, 49, 72\}} =   
\gamma_{\{32, 54, 66, 77\}} =   
\gamma_{\{35, 88\}} =   
\gamma_{\{20, 50, 73\}} =  0,\\ 
\gamma_{\{31, 55, 65, 76, 85\}} = 
\gamma_{\{22, 75\}} = 
\gamma_{\{43, 71, 83\}} =       
\gamma_{\{42, 82\}} = 0,\\  
\gamma_{\{34, 79, 83, 84, 87\}} =   
\gamma_{\{46, 86\}} =   
\gamma_{\{48, 80, 87, 88\}} = 
\gamma_{\{81, 85\}} = 0,\\  
\gamma_{\{33, 53, 67, 69, 78, 21, 51, 74\}} =   
\gamma_{\{36, 56, 68, 70, 79, 82, 86\}} =   
\gamma_{\{14, 73\}} =  0. 
\end{cases}\tag{\ref{8.5.4}.2}
\end{equation}

With the aid of (\ref{8.5.4}.2), the homomorphisms $g_1, g_2$ send (\ref{8.5.4}.1) to
\begin{align*}
&\gamma_{19}[3,7,15] +  \gamma_{32}[3,15,7] +   \gamma_{49}[7,3,15] +  \gamma_{\{28, 59\}}[7,15,3]\\
&\quad +  \gamma_{54}[15,3,7] +  \gamma_{62}[15,7,3] +  \gamma_{63}[7,7,11] +  \gamma_{66}[7,11,7] = 0,\\  
&\gamma_{20}[3,7,15] + \gamma_{\{25, 31, 76\}}[3,15,7] +  \gamma_{50}[7,3,15] +  \gamma_{\{27, 34, 58, 78, 79, 87\}}[7,15,3]\\
&\quad +  \gamma_{55}[15,3,7] +  \gamma_{\{61, 83, 84\}}[15,7,3] +  \gamma_{64}[7,7,11] +  \gamma_{65}[7,11,7] = 0.    
\end{align*}

These equalities imply
\begin{equation}\begin{cases}
\gamma_j = 0,\ j = 19, 20, 32, 49, 50, 54, 55, 62, 63, 64, 65, 66,\\
\gamma_{\{28, 59\}} = \gamma_{\{25, 31, 76\}} =  
\gamma_{\{27, 34, 58, 78, 79, 87\}} = \gamma_{\{61, 83, 84\}} = 0.
\end{cases}\tag{\ref{8.5.4}.3}
\end{equation}

With the aid of (\ref{8.5.4}.2) and (\ref{8.5.4}.3), the homomorphisms $g_3, g_4$ send (\ref{8.5.4}.1) to
\begin{align*}
&\gamma_{21}[3,7,15]  +  \gamma_{\{27, 33, 48, 78, 87\}}[3,15,7]   +   \gamma_{\{53, 80\}}[15,3,7]  \\
&\quad+    \gamma_{51}[7,3,15] +   \gamma_{\{28, 43, 69, 71, 80, 83\}}[7,7,11]  +   a_1[7,15,3]  \\
&\quad  +   \gamma_{\{28, 48, 67, 69, 80, 87\}}[7,11,7]+   a_2[15,7,3]  = 0,\\   
&\gamma_{24}[3,7,15]  +  a_3[3,15,7]  +   \gamma_{52}[7,3,15]  +   a_4[7,15,3]  +   \gamma_{\{56, 71, 80\}}[15,3,7]\\
&\quad  +   a_5[15,7,3] +   \gamma_{\{28, 43, 70, 80, 83, 84\}}[7,7,11]  +   a_6[7,11,7] = 0, 
\end{align*}
where
\begin{align*}
a_1 & = \gamma_{\{22, 24, 25, 36, 46, 52, 57, 68, 76, 79\}},\ \
a_2 = \gamma_{\{22, 24, 42, 52, 56, 60, 70\}},\\
a_3 &= \gamma_{\{34, 36, 43, 48, 79, 83, 84, 87\}},\ \
a_4 = \gamma_{\{13, 17, 21, 31, 33, 44, 51, 57, 58, 61, 67, 76, 78\}},\\
a_5 &= \gamma_{\{14, 17, 21, 26, 40, 51, 53, 59, 60, 61, 69\}},\ \
a_6 = \gamma_{\{28, 43, 48, 68, 70, 71, 80, 83, 84, 87\}}.
\end{align*}
From the above equalities, we obtain
\begin{equation}\begin{cases}
a_i = 0, \ i = 1,2 \ldots, 6,\ \gamma_j = 0,\ j = 21, 24, 51, 52,\\
\gamma_{\{27, 33, 48, 78, 87\}} = 
\gamma_{\{53, 80\}} =    
\gamma_{\{28, 43, 69, 71, 80, 83\}} = 0,\\  
\gamma_{\{28, 48, 67, 69, 80, 87\}} =  
\gamma_{\{56, 71, 80\}} =  
\gamma_{\{28, 43, 70, 80, 83, 84\}} = 0.          
\end{cases}\tag{\ref{8.5.4}.4}
\end{equation}
With the aid of (\ref{8.5.4}.2), (\ref{8.5.4}.3) and (\ref{8.5.4}.4), the homomorphism $h$ sends (\ref{8.5.4}.1) to
\begin{align*}
&a_7[3,7,15]  +  \gamma_{\{28, 56, 60, 70\}}[3,15,7]  +    a_8[7,3,15]  +   \gamma_{56}[7,15,3]\\
&\quad  +   \gamma_{\{43, 61, 83, 84\}}[15,3,7] +   \gamma_{71}[15,7,3]  +   a_9[7,7,11]  +   a_{10}[7,11,7]  = 0,   
\end{align*}
where
\begin{align*}
a_7 & = \gamma_{\{25, 31, 34, 36, 57, 68, 76, 79\}},\ \
a_8 = \gamma_{\{27, 33, 34, 36, 48, 58, 78, 79, 87\}},\\
a_9 &= \gamma_{\{43, 48, 67, 68, 71, 83, 84, 87\}},\ \
a_{10} = \gamma_{\{28, 43, 53, 69, 70, 71, 83, 84\}}.
\end{align*}
From the above equalities, it implies
\begin{equation}\begin{cases}
a_i = 0, \ i = 7, 8, 9,10,\ \gamma_j = 0,\ j = 56, 71,\\
\gamma_{\{28, 60, 70\}}= \gamma_{\{43, 61, 83, 84\}}  =0.  
\end{cases}\tag{\ref{8.5.4}.5}
\end{equation}

Combining (\ref{8.5.4}.2), (\ref{8.5.4}.3), (\ref{8.5.4}.4) and (\ref{8.5.4}.5), we get $\gamma_j = 0$ for $1 \leqslant j \leqslant 88$. The proposition is proved.
\end{proof}

\begin{props}\label{8.5.5} For $ s\geqslant 3$, the elements $[a_{1,1,s,j}], 1 \leqslant j \leqslant 88,$  are linearly independent in $(\mathbb F_2\underset {\mathcal A}\otimes R_4)_{2^{s+2} + 2^{s+1}+2^s-3}$.
\end{props}

\begin{proof} Suppose that there is a linear relation
\begin{equation}\sum_{j=1}^{88}\gamma_j[a_{1,1,s,j}] = 0, \tag {\ref{8.5.5}.1}
\end{equation}
with $\gamma_j \in \mathbb F_2$.

Applying the homomorphisms $f_1, f_2, \ldots, f_6$ to the relation (\ref{8.5.5}.1), we have
\begin{align*}
&\gamma_{1}w_{1,1,s,1} +  \gamma_{2}w_{1,1,s,2} +   \gamma_{15}w_{1,1,s,3} +  \gamma_{16}w_{1,1,s,4}\\
&\quad +  \gamma_{29}w_{1,1,s,5} +  \gamma_{30}w_{1,1,s,6} +  \gamma_{37}w_{1,1,s,7} +  \gamma_{38}w_{1,1,s,8} = 0,\\  &\gamma_{3}w_{1,1,s,1} + \gamma_{5}w_{1,1,s,2} +  \gamma_{13}w_{1,1,s,3} +  \gamma_{18}w_{1,1,s,4}\\
&\quad +  \gamma_{26}w_{1,1,s,5} +  \gamma_{28}w_{1,1,s,6} +  \gamma_{\{39, 45\}}w_{1,1,s,7} +  \gamma_{45}w_{1,1,s,8} = 0,\\  
&\gamma_{4}w_{1,1,s,1} +  \gamma_{6}w_{1,1,s,2} +  \gamma_{14}w_{1,1,s,3} +  \gamma_{17}w_{1,1,s,4}\\
&\quad +  \gamma_{25}w_{1,1,s,5} +  \gamma_{27}w_{1,1,s,6} +  \gamma_{\{40, 44\}}w_{1,1,s,7} +  \gamma_{44}w_{1,1,s,8} = 0,\\  
&\gamma_{19}w_{1,1,s,1} +  \gamma_{32}w_{1,1,s,2} +  \gamma_{7}w_{1,1,s,3} +  \gamma_{35}w_{1,1,s,4}\\
&\quad +  \gamma_{10}w_{1,1,s,5} +  \gamma_{23}w_{1,1,s,6} + \gamma_{41}w_{1,1,s,7} +  \gamma_{47}w_{1,1,s,8} =0,\\  
&\gamma_{20}w_{1,1,s,1} +  \gamma_{31}w_{1,1,s,2} +  \gamma_{8}w_{1,1,s,3} +  \gamma_{34}w_{1,1,s,4}\\
&\quad +  \gamma_{11}w_{1,1,s,5} +  \gamma_{22}w_{1,1,s,6} +  \gamma_{42}w_{1,1,s,7} +  \gamma_{46}w_{1,1,s,8} = 0,\\  
&\gamma_{33}w_{1,1,s,1} +  \gamma_{21}w_{1,1,s,2} +  \gamma_{36}w_{1,1,s,3} +  \gamma_{9}w_{1,1,s,4}\\
&\quad +  \gamma_{24}w_{1,1,s,5} +  \gamma_{12}w_{1,1,s,6} +  \gamma_{48}w_{1,1,s,7} +  \gamma_{43}w_{1,1,s,8} =0. 
\end{align*}

Computing from these equalities, we obtain
\begin{equation}
\gamma_j = 0, \ j = 1,2, \ldots , 48.
\tag{\ref{8.5.5}.2}
\end{equation}

With the aid of (\ref{8.5.5}.2), the homomorphisms $g_1, g_2, g_3, g_4$ send (\ref{8.5.5}.1) to
\begin{align*}
&\gamma_{72}w_{1,1,s,1} +  \gamma_{77}w_{1,1,s,2} +   \gamma_{49}w_{1,1,s,3} +  \gamma_{59}w_{1,1,s,4}\\
&\quad +  \gamma_{54}w_{1,1,s,5} +  \gamma_{62}w_{1,1,s,6} +  \gamma_{63}w_{1,1,s,7} +  \gamma_{66}w_{1,1,s,8} =0,\\  
&\gamma_{73}w_{1,1,s,1} +  \gamma_{76}w_{1,1,s,2} +  \gamma_{50}w_{1,1,s,3} +  \gamma_{58}w_{1,1,s,4}\\
&\quad +  \gamma_{55}w_{1,1,s,5} +  \gamma_{61}w_{1,1,s,6} +  \gamma_{64}w_{1,1,s,7} +  \gamma_{65}w_{1,1,s,8} =0,\\  
&\gamma_{74}w_{1,1,s,1} + \gamma_{78}w_{1,1,s,2} +  \gamma_{51}w_{1,1,s,3} +  a_1w_{1,1,s,4}\\
&\quad +  \gamma_{53}w_{1,1,s,5} + a_2w_{1,1,s,6} +  \gamma_{69}w_{1,1,s,7} +  \gamma_{\{67, 69\}}w_{1,1,s,8} = 0,\\  
&\gamma_{75}w_{1,1,s,1} +  \gamma_{79}w_{1,1,s,2} +  \gamma_{52}w_{1,1,s,3} +  a_3w_{1,1,s,4}\\
&\quad +  \gamma_{56}w_{1,1,s,5} +  a_4w_{1,1,s,6} + \gamma_{70}w_{1,1,s,7} +  \gamma_{\{68, 70\}}w_{1,1,s,8} = 0,
\end{align*}
where
\begin{align*}
a_1 &= \begin{cases} \gamma_{\{57, 87\}}, & s=3,\\
\gamma_{57} & s \geqslant 4,\end{cases}\quad
a_2 = \begin{cases}  \gamma_{\{60, 88\}}, & s= 3,\\
\gamma_{60}, & s\geqslant 4,\end{cases}\\
a_3 &= \begin{cases} \gamma_{\{71, 80, 82, 83, 85, 87\}}, & s=3,\\
 \gamma_{80}, & s\geqslant 4,\end{cases}\ \
a_4 = \begin{cases}   \gamma_{\{71, 81, 83, 84, 86, 88\}}, & s=3,\\
 \gamma_{81}, & s\geqslant 4.\end{cases}
\end{align*}

These equalities imply
\begin{equation}\begin{cases}
a_1 = a_2 = a_3 = a_4 = 0,\\
\gamma_j = 0,\ j = 49, 50, \ldots, 56, 58, 59,\\ 
 61, 62, \ldots, 66, 69, 70, 72,73, \ldots, 79.
\end{cases}\tag{\ref{8.5.5}.3}
\end{equation}

With the aid of (\ref{8.5.5}.2) and (\ref{8.5.5}.3), the homomorphism $h$ sends (\ref{8.5.5}.1) to
\begin{align*}
&a_5w_{1,1,s,1} +  a_6w_{1,1,s,2} +   \gamma_{82}w_{1,1,s,3} +  \gamma_{84}w_{1,1,s,4}\\
&\quad +  \gamma_{83}w_{1,1,s,5} +  \gamma_{71}w_{1,1,s,6} +  \gamma_{85}w_{1,1,s,7} +  \gamma_{86}w_{1,1,s,8} = 0,       
\end{align*}
where
$$ a_5 = \begin{cases} \gamma_{80}, & s=3,\\
\gamma_{87},& s\geqslant 4,\end{cases} \quad
a_6 = \begin{cases}  \gamma_{81}, &s=3,\\
\gamma_{88}, &s \geqslant 4.\end{cases} $$
From the above equalities, it implies
\begin{equation}
a_5 = a_6 = 0, \  \gamma_j = 0,\ j = 71, 82, 83, 84, 85, 86.
\tag{\ref{8.5.5}.4}
\end{equation}

Combining (\ref{8.5.5}.2), (\ref{8.5.5}.3) and (\ref{8.5.5}.4), we get $\gamma_j = 0$ for $1 \leqslant j \leqslant 88$. The proposition is proved.
\end{proof}

\subsection{The case $s \geqslant 2, t = 1$ and $u=2$}\label{8.6}\ 

\medskip
According to Kameko \cite{ka}, for $s\geqslant 2$, $\dim (\mathbb F_2\underset{\mathcal A}\otimes P_3)_{2^{s+3}+2^{s+1} + 2^s -3} = 15$ with a basis given by the classes $w_{2,1,s,j}, 1\leqslant j \leqslant 15,$ which are determined as follows:

\smallskip
\centerline{\begin{tabular}{lll}
$1.\  [2^{s} - 1,2.2^{s} - 1,8.2^{s} - 1],$& $2.\  [2^{s} - 1,8.2^{s} - 1,2.2^{s} - 1],$\cr 
$3.\  [2.2^{s} - 1,2^{s} - 1,8.2^{s} - 1],$& $4.\  [2.2^{s} - 1,8.2^{s} - 1,2^{s} - 1],$\cr 
$5.\  [8.2^{s} - 1,2^{s} - 1,2.2^{s} - 1],$& $6.\  [8.2^{s} - 1,2.2^{s} - 1,2^{s} - 1],$\cr 
$7.\  [2^{s} - 1,4.2^{s} - 1,6.2^{s} - 1],$& $8.\  [4.2^{s} - 1,2^{s} - 1,6.2^{s} - 1],$\cr 
$9.\  [4.2^{s} - 1,6.2^{s} - 1,2^{s} - 1],$& $10.\  [2.2^{s} - 1,2.2^{s} - 1,7.2^{s} - 1],$\cr 
$11.\  [2.2^{s} - 1,7.2^{s} - 1,2.2^{s} - 1],$& $12.\  [2.2^{s} - 1,3.2^{s} - 1,6.2^{s} - 1],$\cr 
$13.\  [2.2^{s} - 1,4.2^{s} - 1,5.2^{s} - 1],$& $14.\  [4.2^{s} - 1,2.2^{s} - 1,5.2^{s} - 1],$\cr 
$15.\  [4.2^{s} - 1,5.2^{s} - 1,2.2^{s} - 1].$& \cr
\end{tabular}}

\smallskip
So, we easily obtain

\begin{props}\label{8.6.3} For any positive integer $s \geqslant 2$, we have
$$\dim (\mathbb F_2\underset{\mathcal A}\otimes Q_4)_{2^{s+3}+ 2^{s+1} + 2^s -3} = 60.$$
\end{props}

Now, we determine $(\mathbb F_2\underset{\mathcal A}\otimes R_4)_{2^{s+3}+ 2^{s+1} + 2^s -3}$. The main result in this subsection is

\begin{thms}\label{dlc8.6} For $s\geqslant 2$, $(\mathbb F_2\underset{\mathcal A}\otimes R_4)_{2^{s+3}+ 2^{s+1} + 2^s -3}$ is  an $\mathbb F_2$-vector space with of dimension 165 a basis consisting of all the classes represented by the monomials $a_{2,1,s,j}, 1 \leqslant j \leqslant 165$, which are determined as follows:

\medskip
For $s \geqslant 2$,

\smallskip
\centerline{\begin{tabular}{ll}
$1.\  (1,2^{s} - 2,2.2^{s} - 1,8.2^{s} - 1),$& $2.\  (1,2^{s} - 2,8.2^{s} - 1,2.2^{s} - 1),$\cr 
$3.\  (1,2.2^{s} - 1,2^{s} - 2,8.2^{s} - 1),$& $4.\  (1,2.2^{s} - 1,8.2^{s} - 1,2^{s} - 2),$\cr 
$5.\  (1,8.2^{s} - 1,2^{s} - 2,2.2^{s} - 1),$& $6.\  (1,8.2^{s} - 1,2.2^{s} - 1,2^{s} - 2),$\cr 
$7.\  (2.2^{s} - 1,1,2^{s} - 2,8.2^{s} - 1),$& $8.\  (2.2^{s} - 1,1,8.2^{s} - 1,2^{s} - 2),$\cr 
\end{tabular}}
\centerline{\begin{tabular}{ll}
$9.\  (2.2^{s} - 1,8.2^{s} - 1,1,2^{s} - 2),$& $10.\  (8.2^{s} - 1,1,2^{s} - 2,2.2^{s} - 1),$\cr 
$11.\  (8.2^{s} - 1,1,2.2^{s} - 1,2^{s} - 2),$& $12.\  (8.2^{s} - 1,2.2^{s} - 1,1,2^{s} - 2),$\cr 
$13.\  (1,2^{s} - 1,2.2^{s} - 2,8.2^{s} - 1),$& $14.\  (1,2^{s} - 1,8.2^{s} - 1,2.2^{s} - 2),$\cr 
$15.\  (1,2.2^{s} - 2,2^{s} - 1,8.2^{s} - 1),$& $16.\  (1,2.2^{s} - 2,8.2^{s} - 1,2^{s} - 1),$\cr 
$17.\  (1,8.2^{s} - 1,2^{s} - 1,2.2^{s} - 2),$& $18.\  (1,8.2^{s} - 1,2.2^{s} - 2,2^{s} - 1),$\cr 
$19.\  (2^{s} - 1,1,2.2^{s} - 2,8.2^{s} - 1),$& $20.\  (2^{s} - 1,1,8.2^{s} - 1,2.2^{s} - 2),$\cr 
$21.\  (2^{s} - 1,8.2^{s} - 1,1,2.2^{s} - 2),$& $22.\  (8.2^{s} - 1,1,2^{s} - 1,2.2^{s} - 2),$\cr 
$23.\  (8.2^{s} - 1,1,2.2^{s} - 2,2^{s} - 1),$& $24.\  (8.2^{s} - 1,2^{s} - 1,1,2.2^{s} - 2),$\cr 
$25.\  (1,2^{s} - 1,2.2^{s} - 1,8.2^{s} - 2),$& $26.\  (1,2^{s} - 1,8.2^{s} - 2,2.2^{s} - 1),$\cr 
$27.\  (1,2.2^{s} - 1,2^{s} - 1,8.2^{s} - 2),$& $28.\  (1,2.2^{s} - 1,8.2^{s} - 2,2^{s} - 1),$\cr 
$29.\  (1,8.2^{s} - 2,2^{s} - 1,2.2^{s} - 1),$& $30.\  (1,8.2^{s} - 2,2.2^{s} - 1,2^{s} - 1),$\cr 
$31.\  (2^{s} - 1,1,2.2^{s} - 1,8.2^{s} - 2),$& $32.\  (2^{s} - 1,1,8.2^{s} - 2,2.2^{s} - 1),$\cr 
$33.\  (2^{s} - 1,2.2^{s} - 1,1,8.2^{s} - 2),$& $34.\  (2.2^{s} - 1,1,2^{s} - 1,8.2^{s} - 2),$\cr 
$35.\  (2.2^{s} - 1,1,8.2^{s} - 2,2^{s} - 1),$& $36.\  (2.2^{s} - 1,2^{s} - 1,1,8.2^{s} - 2),$\cr 
$37.\  (1,2^{s} - 2,4.2^{s} - 1,6.2^{s} - 1),$& $38.\  (1,4.2^{s} - 1,2^{s} - 2,6.2^{s} - 1),$\cr 
$39.\  (1,4.2^{s} - 1,6.2^{s} - 1,2^{s} - 2),$& $40.\  (4.2^{s} - 1,1,2^{s} - 2,6.2^{s} - 1),$\cr 
$41.\  (4.2^{s} - 1,1,6.2^{s} - 1,2^{s} - 2),$& $42.\  (4.2^{s} - 1,6.2^{s} - 1,1,2^{s} - 2),$\cr 
$43.\  (1,2^{s} - 1,4.2^{s} - 2,6.2^{s} - 1),$& $44.\  (1,4.2^{s} - 2,2^{s} - 1,6.2^{s} - 1),$\cr 
$45.\  (1,4.2^{s} - 2,6.2^{s} - 1,2^{s} - 1),$& $46.\  (2^{s} - 1,1,4.2^{s} - 2,6.2^{s} - 1),$\cr 
$47.\  (1,2^{s} - 1,4.2^{s} - 1,6.2^{s} - 2),$& $48.\  (1,4.2^{s} - 1,2^{s} - 1,6.2^{s} - 2),$\cr 
$49.\  (1,4.2^{s} - 1,6.2^{s} - 2,2^{s} - 1),$& $50.\  (2^{s} - 1,1,4.2^{s} - 1,6.2^{s} - 2),$\cr 
$51.\  (2^{s} - 1,4.2^{s} - 1,1,6.2^{s} - 2),$& $52.\  (4.2^{s} - 1,1,2^{s} - 1,6.2^{s} - 2),$\cr 
$53.\  (4.2^{s} - 1,1,6.2^{s} - 2,2^{s} - 1),$& $54.\  (4.2^{s} - 1,2^{s} - 1,1,6.2^{s} - 2),$\cr 
$55.\  (1,2.2^{s} - 2,2.2^{s} - 1,7.2^{s} - 1),$& $56.\  (1,2.2^{s} - 2,7.2^{s} - 1,2.2^{s} - 1),$\cr 
$57.\  (1,2.2^{s} - 1,2.2^{s} - 2,7.2^{s} - 1),$& $58.\  (1,2.2^{s} - 1,7.2^{s} - 1,2.2^{s} - 2),$\cr 
$59.\  (2.2^{s} - 1,1,2.2^{s} - 2,7.2^{s} - 1),$& $60.\  (2.2^{s} - 1,1,7.2^{s} - 1,2.2^{s} - 2),$\cr 
$61.\  (2.2^{s} - 1,7.2^{s} - 1,1,2.2^{s} - 2),$& $62.\  (1,2.2^{s} - 1,2.2^{s} - 1,7.2^{s} - 2),$\cr 
$63.\  (1,2.2^{s} - 1,7.2^{s} - 2,2.2^{s} - 1),$& $64.\  (2.2^{s} - 1,1,2.2^{s} - 1,7.2^{s} - 2),$\cr 
$65.\  (2.2^{s} - 1,1,7.2^{s} - 2,2.2^{s} - 1),$& $66.\  (2.2^{s} - 1,2.2^{s} - 1,1,7.2^{s} - 2),$\cr 
$67.\  (3,2.2^{s} - 3,2^{s} - 2,8.2^{s} - 1),$& $68.\  (3,2.2^{s} - 3,8.2^{s} - 1,2^{s} - 2),$\cr 
$69.\  (3,8.2^{s} - 1,2.2^{s} - 3,2^{s} - 2),$& $70.\  (8.2^{s} - 1,3,2.2^{s} - 3,2^{s} - 2),$\cr 
$71.\  (1,2.2^{s} - 2,3.2^{s} - 1,6.2^{s} - 1),$& $72.\  (1,2.2^{s} - 1,3.2^{s} - 2,6.2^{s} - 1),$\cr 
$73.\  (2.2^{s} - 1,1,3.2^{s} - 2,6.2^{s} - 1),$& $74.\  (1,2.2^{s} - 1,3.2^{s} - 1,6.2^{s} - 2),$\cr 
$75.\  (2.2^{s} - 1,1,3.2^{s} - 1,6.2^{s} - 2),$& $76.\  (2.2^{s} - 1,3.2^{s} - 1,1,6.2^{s} - 2),$\cr 
$77.\  (3,2.2^{s} - 1,8.2^{s} - 3,2^{s} - 2),$& $78.\  (3,8.2^{s} - 3,2^{s} - 2,2.2^{s} - 1),$\cr 
$79.\  (3,8.2^{s} - 3,2.2^{s} - 1,2^{s} - 2),$& $80.\  (2.2^{s} - 1,3,8.2^{s} - 3,2^{s} - 2),$\cr 
$81.\  (1,2.2^{s} - 2,4.2^{s} - 1,5.2^{s} - 1),$& $82.\  (1,4.2^{s} - 1,2.2^{s} - 2,5.2^{s} - 1),$\cr 
$83.\  (1,4.2^{s} - 1,5.2^{s} - 1,2.2^{s} - 2),$& $84.\  (3,2^{s} - 1,2.2^{s} - 3,8.2^{s} - 2),$\cr 
$85.\  (3,2.2^{s} - 3,2^{s} - 1,8.2^{s} - 2),$& $86.\  (3,2.2^{s} - 3,8.2^{s} - 2,2^{s} - 1),$\cr 
$87.\  (4.2^{s} - 1,1,2.2^{s} - 2,5.2^{s} - 1),$& $88.\  (4.2^{s} - 1,1,5.2^{s} - 1,2.2^{s} - 2),$\cr 
$89.\  (4.2^{s} - 1,5.2^{s} - 1,1,2.2^{s} - 2),$& $90.\  (1,2.2^{s} - 1,4.2^{s} - 2,5.2^{s} - 1),$\cr 
$91.\  (1,4.2^{s} - 2,2.2^{s} - 1,5.2^{s} - 1),$& $92.\  (1,4.2^{s} - 2,5.2^{s} - 1,2.2^{s} - 1),$\cr 
$93.\  (2.2^{s} - 1,1,4.2^{s} - 2,5.2^{s} - 1),$& $94.\  (1,2.2^{s} - 1,4.2^{s} - 1,5.2^{s} - 2),$\cr 
$95.\  (1,4.2^{s} - 1,2.2^{s} - 1,5.2^{s} - 2),$& $96.\  (1,4.2^{s} - 1,5.2^{s} - 2,2.2^{s} - 1),$\cr 
$97.\  (2.2^{s} - 1,1,4.2^{s} - 1,5.2^{s} - 2),$& $98.\  (2.2^{s} - 1,4.2^{s} - 1,1,5.2^{s} - 2),$\cr 
\end{tabular}}
\centerline{\begin{tabular}{ll}
$99.\  (4.2^{s} - 1,1,2.2^{s} - 1,5.2^{s} - 2),$& $100.\  (4.2^{s} - 1,1,5.2^{s} - 2,2.2^{s} - 1),$\cr 
$101.\  (4.2^{s} - 1,2.2^{s} - 1,1,5.2^{s} - 2),$& $102.\  (3,2^{s} - 1,8.2^{s} - 3,2.2^{s} - 2),$\cr 
$103.\  (3,8.2^{s} - 3,2^{s} - 1,2.2^{s} - 2),$& $104.\  (3,8.2^{s} - 3,2.2^{s} - 2,2^{s} - 1),$\cr 
$105.\  (3,4.2^{s} - 3,2^{s} - 2,6.2^{s} - 1),$& $106.\  (3,4.2^{s} - 3,6.2^{s} - 1,2^{s} - 2),$\cr 
$107.\  (3,4.2^{s} - 1,6.2^{s} - 3,2^{s} - 2),$& $108.\  (4.2^{s} - 1,3,6.2^{s} - 3,2^{s} - 2),$\cr 
$109.\  (3,2^{s} - 1,4.2^{s} - 3,6.2^{s} - 2),$& $110.\  (3,4.2^{s} - 3,2^{s} - 1,6.2^{s} - 2),$\cr 
$111.\  (3,4.2^{s} - 3,6.2^{s} - 2,2^{s} - 1),$& $112.\  (3,2.2^{s} - 3,2.2^{s} - 2,7.2^{s} - 1),$\cr 
$113.\  (3,2.2^{s} - 3,7.2^{s} - 1,2.2^{s} - 2),$& $114.\  (3,2.2^{s} - 3,2.2^{s} - 1,7.2^{s} - 2),$\cr 
$115.\  (3,2.2^{s} - 3,7.2^{s} - 2,2.2^{s} - 1),$& $116.\  (3,2.2^{s} - 1,2.2^{s} - 3,7.2^{s} - 2),$\cr 
$117.\  (2.2^{s} - 1,3,2.2^{s} - 3,7.2^{s} - 2),$& $118.\  (3,2.2^{s} - 1,7.2^{s} - 3,2.2^{s} - 2),$\cr 
$119.\  (2.2^{s} - 1,3,7.2^{s} - 3,2.2^{s} - 2),$& $120.\  (7,8.2^{s} - 5,2.2^{s} - 3,2^{s} - 2),$\cr 
$121.\  (3,2.2^{s} - 3,3.2^{s} - 2,6.2^{s} - 1),$& $122.\  (3,2.2^{s} - 3,3.2^{s} - 1,6.2^{s} - 2),$\cr 
$123.\  (3,2.2^{s} - 1,3.2^{s} - 3,6.2^{s} - 2),$& $124.\  (2.2^{s} - 1,3,3.2^{s} - 3,6.2^{s} - 2),$\cr 
$125.\  (3,2.2^{s} - 3,4.2^{s} - 2,5.2^{s} - 1),$& $126.\  (3,2.2^{s} - 3,4.2^{s} - 1,5.2^{s} - 2),$\cr 
$127.\  (3,4.2^{s} - 1,2.2^{s} - 3,5.2^{s} - 2),$& $128.\  (4.2^{s} - 1,3,2.2^{s} - 3,5.2^{s} - 2),$\cr 
$129.\  (3,4.2^{s} - 3,2.2^{s} - 2,5.2^{s} - 1),$& $130.\  (3,4.2^{s} - 3,5.2^{s} - 1,2.2^{s} - 2),$\cr 
$131.\  (3,4.2^{s} - 1,5.2^{s} - 3,2.2^{s} - 2),$& $132.\  (4.2^{s} - 1,3,5.2^{s} - 3,2.2^{s} - 2),$\cr 
$133.\  (3,2.2^{s} - 1,4.2^{s} - 3,5.2^{s} - 2),$& $134.\  (3,4.2^{s} - 3,2.2^{s} - 1,5.2^{s} - 2),$\cr 
$135.\  (3,4.2^{s} - 3,5.2^{s} - 2,2.2^{s} - 1),$& $136.\  (2.2^{s} - 1,3,4.2^{s} - 3,5.2^{s} - 2),$\cr 
$137.\  (7,4.2^{s} - 5,6.2^{s} - 3,2^{s} - 2),$& $138.\  (7,4.2^{s} - 5,2.2^{s} - 3,5.2^{s} - 2),$\cr 
$139.\  (7,4.2^{s} - 5,5.2^{s} - 3,2.2^{s} - 2).$& \cr
\end{tabular}}

\medskip
For $s = 2$,

\medskip
\centerline{\begin{tabular}{lll}
$140.\ (3,3,4,31),$& $141.\ (3,3,31,4),$& $142.\ (3,31,3,4),$\cr 
$143.\ (31,3,3,4),$& $144.\ (3,3,7,28),$& $145.\ (3,3,28,7),$\cr 
$146.\ (3,7,3,28),$& $147.\ (7,3,3,28),$& $148.\ (3,7,27,4),$\cr 
$149.\ (7,3,27,4),$& $150.\ (7,27,3,4),$& $151.\ (7,7,25,2),$\cr 
$152.\ (3,3,12,23),$& $153.\ (3,3,15,20),$& $154.\ (3,15,3,20),$\cr 
$155.\ (15,3,3,20),$& $156.\ (3,7,7,24),$& $157.\ (7,3,7,24),$\cr 
$158.\ (7,7,3,24),$& $159.\ (7,7,24,3),$& $160.\ (3,7,11,20),$\cr 
$161.\ (7,3,11,20),$& $162.\ (7,11,3,20),$& $163.\ (7,7,8,19),$\cr 
$164.\ (7,7,9,18),$& $165.\ (7,7,11,16).$& \cr
\end{tabular}}

\medskip
For $s \geqslant 3$,

\medskip
\centerline{\begin{tabular}{ll}
$140.\  (3,2^{s} - 3,2.2^{s} - 2,8.2^{s} - 1),$& $141.\  (3,2^{s} - 3,8.2^{s} - 1,2.2^{s} - 2),$\cr 
$142.\  (3,8.2^{s} - 1,2^{s} - 3,2.2^{s} - 2),$& $143.\  (8.2^{s} - 1,3,2^{s} - 3,2.2^{s} - 2),$\cr 
$144.\  (3,2^{s} - 3,2.2^{s} - 1,8.2^{s} - 2),$& $145.\  (3,2^{s} - 3,8.2^{s} - 2,2.2^{s} - 1),$\cr 
$146.\  (3,2.2^{s} - 1,2^{s} - 3,8.2^{s} - 2),$& $147.\  (2.2^{s} - 1,3,2^{s} - 3,8.2^{s} - 2),$\cr 
$148.\  (2^{s} - 1,3,2.2^{s} - 3,8.2^{s} - 2),$& $149.\  (2^{s} - 1,3,8.2^{s} - 3,2.2^{s} - 2),$\cr 
$150.\  (3,2^{s} - 3,4.2^{s} - 2,6.2^{s} - 1),$& $151.\  (3,2^{s} - 3,4.2^{s} - 1,6.2^{s} - 2),$\cr 
$152.\  (3,4.2^{s} - 1,2^{s} - 3,6.2^{s} - 2),$& $153.\  (4.2^{s} - 1,3,2^{s} - 3,6.2^{s} - 2),$\cr 
$154.\  (2^{s} - 1,3,4.2^{s} - 3,6.2^{s} - 2),$& $155.\  (7,2.2^{s} - 5,2^{s} - 3,8.2^{s} - 2),$\cr 
$156.\  (7,2.2^{s} - 5,8.2^{s} - 3,2^{s} - 2),$& $157.\  (7,8.2^{s} - 5,2^{s} - 3,2.2^{s} - 2),$\cr 
$158.\  (7,4.2^{s} - 5,2^{s} - 3,6.2^{s} - 2),$& $159.\  (7,2.2^{s} - 5,2.2^{s} - 3,7.2^{s} - 2),$\cr 
\end{tabular}}
\centerline{\begin{tabular}{ll}
$160.\  (7,2.2^{s} - 5,7.2^{s} - 3,2.2^{s} - 2),$& $161.\  (7,2.2^{s} - 5,3.2^{s} - 3,6.2^{s} - 2),$\cr 
$162.\  (7,2.2^{s} - 5,4.2^{s} - 3,5.2^{s} - 2).$& \cr
\end{tabular}}

\medskip
For $s = 3$,
$$163.\ (7,7,9,62),\quad 164.\  (7,7,57,14),\quad 165.\  (7,7,25,46).$$

\medskip
For $s \geqslant 4$,

\medskip
\centerline{\begin{tabular}{ll}
$163.\  (7,2^{s} - 5,2.2^{s} - 3,8.2^{s} - 2),$& $164.\  (7,2^{s} - 5,8.2^{s} - 3,2.2^{s} - 2),$\cr 
$165.\  (7,2^{s} - 5,4.2^{s} - 3,6.2^{s} - 2).$& \cr
\end{tabular}}
\end{thms}

We prove the theorem by proving the following.

\begin{props}\label{mdc8.6} The $\mathbb F_2$-vector space $(\mathbb F_2\underset {\mathcal A}\otimes R_4)_{2^{s+3}+ 2^{s+1}+2^s -3}$ is generated by the  elements listed in Theorem \ref{dlc8.6}.
\end{props}

The proof of this proposition is based on the following lemmas.

\begin{lems}\label{8.6.1} The following matrices are strictly inadmissible
 $$\begin{pmatrix} 1&1&1&0\\ 1&1&0&1\\ 1&0&1&0\\ 1&0&0&0\\ 0&1&0&0\end{pmatrix} \quad \begin{pmatrix} 1&1&1&0\\ 1&0&1&1\\ 1&0&0&1\\ 0&1&0&0\\ 0&1&0&0\end{pmatrix} \quad \begin{pmatrix} 1&1&1&0\\ 1&1&1&0\\ 0&1&0&1\\ 0&1&0&0\\ 0&0&1&0\end{pmatrix} \quad \begin{pmatrix} 1&1&1&0\\ 1&1&1&0\\ 1&0&0&1\\ 1&0&0&0\\ 0&0&1&0\end{pmatrix} $$  
$$\begin{pmatrix} 1&1&1&0\\ 1&1&1&0\\ 1&0&0&1\\ 1&0&0&0\\ 0&1&0&0\end{pmatrix} \quad \begin{pmatrix} 1&1&1&0\\ 1&1&0&1\\ 1&1&0&0\\ 0&1&0&0\\ 0&0&1&0\end{pmatrix} \quad \begin{pmatrix} 1&1&1&0\\ 1&1&0&1\\ 1&1&0&0\\ 1&0&0&0\\ 0&0&1&0\end{pmatrix} \quad \begin{pmatrix} 1&1&0&1\\ 1&0&1&1\\ 1&0&0&1\\ 1&0&0&0\\ 0&1&0&0\end{pmatrix} $$  
$$\begin{pmatrix} 1&1&1&0\\ 1&0&1&1\\ 1&0&1&0\\ 1&0&0&0\\ 0&1&0&0\end{pmatrix} \quad \begin{pmatrix} 1&1&1&0\\ 1&0&1&1\\ 1&0&0&1\\ 0&1&0&0\\ 0&0&0&1\end{pmatrix} \quad \begin{pmatrix} 1&1&1&0\\ 1&0&1&1\\ 1&0&0&1\\ 1&0&0&0\\ 0&1&0&0\end{pmatrix} \quad \begin{pmatrix} 1&1&0&1\\ 1&0&1&1\\ 1&0&1&0\\ 1&0&0&0\\ 0&1&0&0\end{pmatrix} $$  
$$\begin{pmatrix} 1&1&1&0\\ 1&1&1&0\\ 0&1&1&0\\ 0&0&1&0\\ 0&0&0&1\end{pmatrix} \quad \begin{pmatrix} 1&1&1&0\\ 1&1&1&0\\ 0&1&1&0\\ 0&1&0&0\\ 0&0&0&1\end{pmatrix} \quad \begin{pmatrix} 1&1&0&1\\ 1&1&0&1\\ 0&1&0&1\\ 0&1&0&0\\ 0&0&1&0\end{pmatrix} \quad \begin{pmatrix} 1&1&1&0\\ 1&1&1&0\\ 1&0&1&0\\ 0&0&1&0\\ 0&0&0&1\end{pmatrix} $$  
$$\begin{pmatrix} 1&1&1&0\\ 1&1&1&0\\ 1&1&0&0\\ 0&1&0&0\\ 0&0&0&1\end{pmatrix} \quad \begin{pmatrix} 1&1&0&1\\ 1&1&0&1\\ 1&1&0&0\\ 0&1&0&0\\ 0&0&1&0\end{pmatrix} \quad \begin{pmatrix} 1&1&1&0\\ 1&1&1&0\\ 1&0&1&0\\ 1&0&0&0\\ 0&0&0&1\end{pmatrix} \quad \begin{pmatrix} 1&1&0&1\\ 1&1&0&1\\ 1&0&0&1\\ 1&0&0&0\\ 0&0&1&0\end{pmatrix} $$  
$$\begin{pmatrix} 1&1&1&0\\ 1&1&1&0\\ 1&1&0&0\\ 1&0&0&0\\ 0&0&0&1\end{pmatrix} \quad \begin{pmatrix} 1&1&0&1\\ 1&1&0&1\\ 1&1&0&0\\ 1&0&0&0\\ 0&0&1&0\end{pmatrix} \quad \begin{pmatrix} 1&0&1&1\\ 1&0&1&1\\ 1&0&0&1\\ 1&0&0&0\\ 0&1&0&0\end{pmatrix} \quad \begin{pmatrix} 1&0&1&1\\ 1&0&1&1\\ 1&0&1&0\\ 1&0&0&0\\ 0&1&0&0\end{pmatrix} $$  
$$\begin{pmatrix} 1&1&1&0\\ 1&1&1&0\\ 1&0&0&1\\ 0&1&0&0\\ 0&0&1&0\end{pmatrix} \quad \begin{pmatrix} 1&1&1&0\\ 1&0&1&1\\ 1&0&0&1\\ 0&1&0&0\\ 0&0&1&0\end{pmatrix} \quad \begin{pmatrix} 1&1&1&0\\ 1&1&1&0\\ 1&0&1&0\\ 0&1&0&0\\ 0&0&0&1\end{pmatrix} \quad \begin{pmatrix} 1&1&0&1\\ 1&1&0&1\\ 1&0&0&1\\ 0&1&0&0\\ 0&0&1&0\end{pmatrix} .$$ 
\end{lems}

\begin{proof} The monomials corresponding to the above matrices respectively are  
\begin{align*}
&(15,19,5,2),   (7,25,3,6),    (3,15,19,4),   (15,3,19,4),   (15,19,3,4), \\
&(7,15,17,2),   (15,7,17,2),   (15,17,2,7),   (15,17,7,2),   (7,9,3,22),\\
&(15,17,3,6),   (15,17,6,3),   (3,7,15,16),   (3,15,7,16),   (3,15,16,7),\\
&(7,3,15,16),   (7,15,3,16),   (7,15,16,3),   (15,3,7,16),   (15,3,16,7),\\   
&(15,7,3,16),   (15,7,16,3),   (15,16,3,7),   (15,16,7,3),  (7,11,19,4),\\   
&(7,9,19,6),   (7,11,7,16),   (7,11,16,7).
\end{align*}
 We prove the lemma for the matrices associated with the monomials 
 \begin{align*}
&(15,19,5,2), (7,25,3,6), (3,15,19,4),   (15,19,3,4), (15,7,17,2),\\ 
& (15,17,7,2),(7,9,3,22), (15,17,3,6),  (7,3,15,16), (7,15,3,16),  \\ &(15,3,7,16),(15,16,3,7), (7,11,19,4), (7,9,19,6),   (7,11,7,16).
\end{align*}
By a direct computation, we have
\begin{align*}
&(15,19,5,2) = Sq^1(15,15,7,3)  + Sq^2(15,15,7,2) +Sq^4\big((15,15,5,2)\\ 
&\quad+ (15,15,4,3)  \big)+Sq^8\big((11,15,5,2) + (9,15,7,2)\\ 
&\quad + (8,15,7,3)+ (11,15,4,3)\big)+  (11,23,5,2) \\ 
&\quad+ (9,23,7,2) + (15,17,7,2) + (8,23,7,3)+ (15,16,7,3) \\ 
&\quad+ (11,23,4,3)  + (15,19,4,3)\quad \text{mod  }\mathcal L_4(3;3;2;1;1),\\
&(7,25,3,6) = Sq^1\big((7,15,7,11) + (7,15,5,13)\big)  + Sq^2\big((7,23,3,6)\\ 
&\quad + (7,15,7,10)+ (3,15,11,10) + (2,15,11,11) + (7,15,6,11)\\ 
&\quad+ (3,15,10,11) + (7,15,3,14)\big) +Sq^4\big((5,23,3,6)+ (11,15,5,6)\\ 
&\quad  + (5,15,7,10)  + (3,15,13,6)  + (3,15,7,12) + (4,15,7,11)\\ 
&\quad + (2,15,7,13) + (2,15,13,7) + (5,15,6,11) + (3,15,6,13)\\ 
&\quad + (3,15,12,7) + (5,15,3,14)\big)+Sq^8(7,15,5,6) + (7,23,3,8)\\
&\quad +  (5,27,3,6) + (5,23,3,10) + (7,19,9,6) + (7,19,5,10)
\end{align*}
\begin{align*}
&\quad+ (5,19,7,10)  + (7,17,7,10) + (3,19,13,6) + (3,15,17,6)\\ 
&\quad + (3,19,7,12) + (3,15,7,16) + (4,19,7,11)  + (2,19,13,7) \\ 
&\quad+ (2,15,17,7) + (2,19,7,13) + (2,15,7,17)  + (7,16,7,11)\\ 
&\quad + (5,19,6,11) + (3,19,6,13)  + (3,15,6,17) + (3,19,12,7)\\ 
&\quad + (3,15,16,7) + (7,17,6,11) + (7,17,3,14) + (7,15,3,16)\\ 
&\quad + (5,19,3,14) + (5,15,3,18)\quad \text{mod  }\mathcal L_4(3;3;2;1;1),\\
&(3,15,19,4) = Sq^1(3,15,15,7)  + Sq^2(2,15,15,7) +Sq^4\big((3,15,15,4)\\ 
&\quad+ (2,15,15,5)  \big)+Sq^8\big((3,11,15,4) + (3,8,15,7)\big)+  (3,11,23,4)\\ 
&\quad + (2,17,15,7) + (2,15,17,7)+ (2,19,15,5) + (2,15,19,5)\\ 
&\quad + (3,8,23,7)+ (3,15,16,7)\quad \text{mod  }\mathcal L_4(3;3;2;1;1),\\
&(15,19,3,4) = Sq^1(15,15,3,7)  + Sq^2(15,15,2,7)\\ 
&\quad +Sq^4\big((15,15,3,4)+ (15,15,2,5)  \big)+Sq^8\big((11,15,3,4)\\ 
&\quad + (8,15,3,7)+(9,15,2,7)+ (11,15,2,5)\big)+  (11,23,3,4)\\ 
&\quad + (9,23,2,7) + (15,17,2,7) + (11,23,2,5) + (15,19,2,5)\\ 
&\quad+ (8,23,3,7)+ (15,16,3,7)\quad \text{mod  }\mathcal L_4(3;3;2;1;1),\\
&(15,7,17,2) = Sq^1(15,7,15,3)  + Sq^2(15,7,15,2)\\ 
&\quad +Sq^4\big((15,5,15,2)+ (15,4,15,3)  \big)+Sq^8\big((9,7,15,2)\\ 
&\quad + (11,5,15,2)+(8,7,15,3)+ (11,4,15,3)\big)+  (9,7,23,2)\\ 
&\quad + (11,5,23,2) + (15,5,19,2) + (8,7,23,3) + (15,7,16,3)\\ 
&\quad+ (11,4,23,3)+ (15,4,19,3)\quad \text{mod  }\mathcal L_4(3;3;2;1;1),\\
&(15,17,7,2) = Sq^1\big((15,7,11,7) + (15,7,9,9)  + (15,7,7,11) \\ 
&\quad + (15,7,5,13)\big) + Sq^2\big((15,11,11,2) + (15,7,10,7) + (15,7,7,10)\\ 
&\quad + (15,7,6,11) + (15,7,3,14) + (15,3,11,10)\big) +Sq^4\big((15,13,7,2)\\ 
&\quad+ (15,7,13,2)  + (15,7,11,4)  + (15,4,11,7) + (15,5,10,7)\\ 
&\quad + (15,5,7,10) + (15,4,7,11) + (15,5,6,11) + (15,5,3,14)\\ 
&\quad + (15,3,13,6) + (15,3,7,12)\big)+Sq^8\big((11,13,7,2) + (11,7,13,2)\\ 
&\quad+(11,7,11,4)+ (8,7,11,7) + (11,4,11,7) + (9,7,10,7)\\ 
&\quad + (11,5,10,7) + (9,7,7,10) + (11,5,7,10) + (8,7,7,11)\\ 
&\quad + (11,4,7,11) + (9,7,6,11) + (11,5,6,11) + (9,7,3,14)\\ 
&\quad + (11,5,3,14) + (11,3,13,6) + (11,3,7,12)\big)+  (11,21,7,2) \\ 
&\quad+ (11,7,21,2) + (15,7,17,2) + (11,7,19,4) + (8,7,19,7)\\ 
&\quad+ (11,4,19,7)+ (9,7,18,7) + (11,5,18,7) + (9,7,7,18)\\ 
&\quad + (11,5,7,18) + (8,7,7,19) + (11,4,7,19) + (9,7,6,19)\\ 
&\quad + (11,5,6,19) + (9,7,3,22) + (11,5,3,22) + (15,5,3,18)
\end{align*}
\begin{align*}
&\quad + (15,7,3,16) + (11,3,21,6) + (15,3,17,6) \\ 
&\quad+ (11,3,7,20) + (15,3,7,16)\quad \text{mod  }\mathcal L_4(3;3;2;1;1),\\
&(7,9,3,22) = Sq^1\big((7,7,5,21) + (7,7,11,15)\big) + Sq^2\big((7,7,3,22)\\ 
&\quad + (7,7,6,19) + (7,3,6,23) + (7,7,10,15) + (3,11,10,15) \\ 
&\quad+ (2,11,11,5) + (7,2,7,23)\big) +Sq^4\big((5,7,3,22)\\ 
&\quad+ (5,7,6,19) + (11,5,6,15) + (5,3,6,23) + (5,7,10,15) \\ 
&\quad+ (3,13,6,15) + (3,7,12,15) + (4,7,11,15) + (2,13,7,15)\\ 
&\quad + (2,7,13,15) + (11,4,7,15) + (5,2,7,23)\big)\\ 
&\quad+Sq^8\big((7,5,6,15) + (7,4,7,15)\big)+  (7,7,3,24) + (5,11,3,22)\\ 
&\quad + (5,7,3,26) + (7,7,8,19) + (5,11,6,19)+ (5,7,10,19)\\ 
&\quad+ (7,5,10,19) + (7,3,8,23) + (5,3,10,23) + (5,3,6,27) \\ 
&\quad+ (7,7,10,17) + (5,7,10,19) + (3,17,6,15) + (3,13,6,19) \\ 
&\quad+ (3,7,16,15) + (3,7,12,19) + (7,7,11,16) + (4,7,11,19) \\ 
&\quad+ (2,17,7,15) + (2,13,7,19) + (2,7,17,15) + (2,7,13,19) \\ 
&\quad+ (7,4,11,19) + (7,2,9,23) + (7,2,7,25) +(7,8,7,19)\\ 
&\quad+ (5,2,11,23) + (5,2,7,27) + (7,3,6,25)\quad \text{mod  }\mathcal L_4(3;3;2;1;1),\\
&(15,17,3,6) = Sq^1\big((15,7,7,11) + (15,7,5,13) + (15,3,11,11)\big) \\ 
&\quad+ Sq^2\big((15,11,3,10) + (15,7,7,10) + (15,7,6,11) + (15,7,3,14)\\ 
&\quad + (15,3,11,10) + (15,3,10,11)\big) +Sq^4\big(15,13,3,6) + (15,7,3,12)\\ 
&\quad + (15,7,5,10) + (15,5,7,10) + (15,4,7,11) + (15,5,6,11)\\ 
&\quad + (15,5,3,14) + (15,3,13,6) + (15,3,6,13)\big)+Sq^8\big((11,13,3,6)\\ 
&\quad + (11,7,3,12) + (11,7,5,10) + (9,7,7,10) + (11,5,7,10)\\ 
&\quad + (8,7,7,11) + (11,4,7,11) + (9,7,6,11) + (11,5,6,11) \\ 
&\quad+ (9,7,3,14) + (11,5,3,14) + (11,3,13,6) + (11,3,6,13)\big)\\ 
&\quad + (11,21,3,6) + (11,7,3,20) + (15,7,3,16) + (11,7,5,18)\\ 
&\quad + (9,7,7,18) + (11,5,7,18) + (8,7,7,19) + (11,4,7,19)\\ 
&\quad + (9,7,6,19) + (11,5,6,19) + (9,7,3,22) + (11,5,3,22)\\ 
&\quad + (15,5,3,18) + (15,7,3,16) + (15,3,17,6) + (15,3,6,17)\\
&\quad + (11,3,21,6) + (11,3,6,21)\quad \text{mod  }\mathcal L_4(3;3;2;1;1),\\
&(7,3,15,16) = Sq^1(7,3,15,15) + Sq^2(7,2,15,15) +Sq^4\big((4,3,15,15)\\ 
&\quad + (5,2,15,15)\big)+Sq^8(7,3,8,15) + (4,3,19,15)+ (4,3,15,19)\\ 
&\quad  + (7,2,17,15) + (7,2,15,17) + (5,2,19,15) \\ 
&\quad+ (5,2,15,19) + (7,3,8,23)\quad \text{mod  }\mathcal L_4(3;3;2;1;1),
\end{align*}
\begin{align*}
&(15,3,7,16) = Sq^1(15,3,7,15) + Sq^2(15,2,7,15) +Sq^4\big((15,2,5,15)\\ 
&\quad + (15,3,4,15)\big)+Sq^8\big((8,3,7,15) + (9,2,7,15) + (11,2,5,15)\\ 
&\quad + (11,3,4,15)\big)  + (8,3,7,23) + (9,2,7,23)\\ 
&\quad + (15,2,7,17) + (11,2,5,23) + (15,2,5,19)\\
&\quad + (11,3,4,23) + (15,3,4,19)\quad \text{mod  }\mathcal L_4(3;3;2;1;1),\\
&(15,16,3,7) = Sq^1(15,9,5,11) + Sq^2\big((15,10,3,11)+ (15,3,10,11) \big)\\ 
&\quad +Sq^4\big((15,12,3,7)+ (15,6,3,13)+ (15,5,6,11)+ (15,3,6,13)\\ 
&\quad+ (15,3,12,7)\big)+Sq^8\big((11,12,3,7)+ (11,6,3,13)+ (11,5,6,11)\\ 
&\quad+ (11,3,6,13)+ (11,3,12,7)\big)+ (11,20,3,7)+ (15,6,3,17)\\ 
&\quad+ (11,6,3,21)+ (11,5,6,19)+ (11,3,6,21)+ (15,3,6,17)\\
&\quad+ (11,3,20,7)+ (15,3,16,7) \quad \text{mod  }\mathcal L_4(3;3;2;1;1),\\
&(7,11,19,4) = Sq^1\big((7,7,23,3) + (7,11,15,7) + (7,5,23,5)\big)\\ 
&\quad+ Sq^2\big((2,11,15,11)  + (7,10,15,7) + (3,10,15,11) + (7,3,23,6)\big)\\ 
&\quad +Sq^4\big((11,7,15,4) + (4,7,23,3) + (4,11,15,7) + (2,7,15,13)\\ 
&\quad + (2,13,15,7) + (5,10,15,7) + (3,6,15,13) + (3,12,15,7)\\ 
&\quad + (11,6,15,5) + (5,3,23,6)\big)+Sq^8\big((7,7,15,4) + (7,6,15,5)\big)\\ 
&\quad  + (7,7,19,8) + (4,11,23,3) + (4,7,27,3) + (7,8,23,3)\\ 
&\quad + (7,7,24,3) + (4,11,19,7) + (2,7,19,13) + (2,7,15,17)\\ 
&\quad + (2,17,15,7) + (2,13,19,7) + (7,11,16,7) + (5,10,19,7)\\ 
&\quad + (3,6,19,13) + (3,6,15,17) + (3,16,15,7) + (3,12,19,7)\\ 
&\quad + (7,10,17,7) + (7,10,19,5) + (7,6,19,9) + (7,3,25,6)\\ 
&\quad + (7,3,23,8) + (5,3,27,6) + (5,3,23,10)\quad \text{mod  }\mathcal L_4(3;3;2;1;1),\\
&(7,9,19,6) = Sq^1\big((7,7,15,11) + (7,3,23,7)\big) + Sq^2\big((7,3,23,6) \\ 
&\quad+ (7,7,15,10)+ (3,11,15,10)+ (2,11,15,11)\big) \\ 
&\quad+Sq^4\big((11,5,15,6)+ (5,3,23,6)+ (5,7,15,10)+ (3,13,15,6)\\ 
&\quad+ (3,7,15,12)+ (4,7,15,11)+ (2,13,15,7)+ (2,7,15,13)\\ 
&\quad+ (11,4,15,7)+ (4,3,23,7)\big)+Sq^8\big((7,5,15,6) + (7,4,15,7)\big) \\ 
&\quad+ (7,5,19,10)+ (7,3,25,6)+ (7,3,23,8)+ (5,3,27,6)\\ 
&\quad+ (5,3,23,10)+ (7,7,17,10)+ (5,7,19,10)+ (3,17,15,6)\\ 
&\quad+ (3,13,19,6)+ (3,7,19,12)+ (3,7,15,16)+ (7,7,16,11)\\ 
&\quad+ (4,7,19,11)+ (2,17,15,7)+ (2,13,19,7)+ (2,7,19,13)\\ 
&\quad+ (2,7,15,17)+ (7,4,19,11)+ (7,8,19,7)+ (7,3,24,7)\\ 
&\quad+ (7,3,23,8)+ (4,3,27,7)+ (4,3,23,11)\quad \text{mod  }\mathcal L_4(3;3;2;1;1),
\end{align*}
\begin{align*}
&(7,11,7,16) = Sq^1\big((7,11,7,15) + (7,7,3,23)\big) + Sq^2\big((2,11,11,15)\\ 
&\quad + (7,10,7,15)  + (3,10,11,15) + (7,6,3,23)\big) \\ 
&\quad+Sq^4\big((4,11,7,15) + (2,7,13,15) + (2,13,7,15)\\ 
&\quad + (5,10,7,15) + (3,6,13,15) + (3,12,7,15) + (11,6,5,15)\\ 
&\quad + (5,6,3,23) + (11,7,4,15) + (4,7,3,23)\big)+Sq^8\big((7,6,5,15)\\ 
&\quad + (7,7,4,15)\big)  + (4,11,7,19) + (2,7,17,15) + (2,7,13,19)\\ 
&\quad + (2,17,7,15) + (2,13,7,19) + (7,10,7,17) + (5,10,7,19)\\ 
&\quad + (3,6,17,15) + (3,6,13,19) + (3,16,7,15) + (3,12,7,19)\\ 
&\quad + (7,10,5,19) + (7,6,9,19)  + (7,6,3,25) + (5,10,3,23) \\ 
&\quad+ (5,6,3,27) + (7,11,4,19) + (7,7,8,19)  + (7,7,3,24)\\ 
&\quad + (4,11,3,23) + (4,7,3,27)\quad \text{mod  }\mathcal L_4(3;3;2;1;1).
\end{align*}

The lemma is proved.
\end{proof}

\begin{lems}\label{8.6.2} The following matrix is strictly inadmissible
$$\begin{pmatrix} 1&1&1&0\\ 1&1&1&0\\ 1&1&1&0\\ 1&1&0&0\\ 0&0&1&0\\ 0&0&0&1\end{pmatrix} .$$
\end{lems}

\begin{proof} The monomial corresponding to the above matrix is (15,15,23,32).  By a direct computation, we have
\begin{align*}
&(15,15,23,32) = Sq^1(15,15,23,31) + Sq^2\big((15,10,27,31) + (15,15,22,31) \\ 
&\quad+ (15,11,26,31)\big) +Sq^4\big((4,23,23,31) + (15,12,23,31) + (7,20,23,31)\\ 
&\quad + (15,6,29,31) + (5,23,22,31) + (15,13,22,31) + (7,21,22,31)\\ 
&\quad + (15,7,28,31)\big)+Sq^8\big((8,15,23,31) + (4,27,15,31) + (4,15,27,31)\\ 
&\quad + (11,12,23,31) + (7,24,15,31) + (7,12,27,31) + (11,6,29,31)\\ 
&\quad + (9,15,22,31) + (5,15,26,31) + (5,27,14,31)+ (11,13,22,31)\\ 
&\quad + (7,13,26,31) + (7,25,14,31) + (11,7,28,31)\big) + (8,15,23,39)\\ 
&\quad + (4,27,15,39) + (4,35,15,31) + (4,15,35,31) + (4,15,27,39)\\ 
&\quad + (15,12,23,35) + (11,12,23,39) + (7,24,15,39) + (7,32,15,31)\\ 
&\quad + (7,12,35,31) + (7,12,27,39) + (11,6,37,31) + (11,6,29,39)\\ 
&\quad + (15,6,33,31) + (15,6,29,35) + (15,15,22,33) + (9,15,22,39)\\ 
&\quad + (5,15,34,31) + (5,15,26,39) + (5,35,14,31) + (5,27,14,39)\\ 
&\quad + (15,13,22,35) + (11,13,22,39) + (7,13,34,31) + (7,13,26,39)
\end{align*}
\begin{align*}
&\quad + (7,33,14,31) + (7,25,14,39) + (15,7,32,31) + (15,7,28,35)\\ 
&\quad + (11,7,36,31) + (11,7,28,39)\quad \text{mod  }\mathcal L_4(3;3;3;2;1;1).
\end{align*}
The lemma follows.
\end{proof}

Using the results in Section \ref{7}, Lemmas \ref{3.2}, \ref{3.3},  \ref{5.7}, \ref{6.2.1}, \ref{6.2.2}, \ref{8.5.1}, \ref{8.5.2}, \ref{8.6.1}, \ref{8.6.2} and Theorem \ref{2.4},  we obtain Proposition \ref{mdc8.6}.

\medskip
Now, we prove that the classes listed in Theorem \ref{dlc8.6} are linearly independent.

\begin{props}\label{8.6.4} The elements $[a_{2,1,2,j}], 1 \leqslant j \leqslant 165,$  are linearly independent in $(\mathbb F_2\underset {\mathcal A}\otimes R_4)_{41}$.
\end{props}

\begin{proof} Suppose that there is a linear relation
\begin{equation}\sum_{j=1}^{165}\gamma_j[a_{2,1,2,j}] = 0, \tag {\ref{8.6.4}.1}
\end{equation}
with $\gamma_j \in \mathbb F_2$.

Applying the homomorphisms $f_1, f_2, \ldots, f_6$ to the relation (\ref{8.6.4}.1), we obtain
\begin{align*}
&\gamma_{1}[3,7,31]  +  \gamma_{2}[3,31,7]  +    \gamma_{15}[7,3,31]  +   \gamma_{16}[7,31,3] \\
&\quad +   \gamma_{29}[31,3,7]  +   \gamma_{30}[31,7,3]  +   \gamma_{37}[3,15,23]  +   \gamma_{44}[15,3,23]\\
&\quad  +   \gamma_{45}[15,23,3]  +   \gamma_{55}[7,7,27]  +   \gamma_{56}[7,27,7]  +   \gamma_{71}[7,11,23]\\
&\quad  +   \gamma_{81}[7,15,19]  +   \gamma_{91}[15,7,19]  +   \gamma_{92}[15,19,7]  =0,\\   
&\gamma_{3}[3,7,31]  +  \gamma_{5}[3,31,7]  +   \gamma_{\{13, 140\}}[7,3,31]  +   \gamma_{18}[7,31,3]\\
&\quad  +   \gamma_{\{26, 145\}}[31,3,7]  +   \gamma_{\{28, 159\}}[31,7,3]  +   \gamma_{38}[3,15,23]  +   \gamma_{\{43, 152\}}[15,3,23]\\
&\quad  +   \gamma_{49}[15,23,3]  +   \gamma_{\{57, 72\}}[7,7,27] +   \gamma_{96}[7,27,7]   +   \gamma_{72}[7,11,23]\\
&\quad  +   \gamma_{\{82, 96\}}[7,15,19]  +   \gamma_{\{63, 90, 163\}}[15,7,19]  +   \gamma_{63}[15,19,7]  = 0,\\  
&\gamma_{4}[3,7,31]  +  \gamma_{6}[3,31,7]  +    \gamma_{\{14, 141\}}[7,3,31]  +   \gamma_{\{17, 142\}}[7,31,3]\\
&\  +   \gamma_{\{25, 144, 157\}}[31,3,7]  +   \gamma_{\{27, 146, 158\}}[31,7,3]  +   \gamma_{95}[7,27,7]  +   \gamma_{94}[7,11,23]\\
&\  +   \gamma_{39}[3,15,23]  +   \gamma_{\{47, 153\}}[15,3,23]  +   \gamma_{\{48, 154\}}[15,23,3]  +   \gamma_{\{83, 95\}}[7,15,19]\\
&\  +   \gamma_{\{58, 94, 148\}}[7,7,27]  +   \gamma_{\{62, 74, 156, 160, 165\}}[15,7,19]  +   \gamma_{\{62, 156\}}[15,19,7]  = 0,\\  
&\gamma_{\{19, 67, 140\}}[3,7,31]  +  \gamma_{\{32, 78, 115, 135, 145\}}[3,31,7]  +   \gamma_{7}[7,3,31]\\
&\quad  +   \gamma_{\{35, 159\}}[7,31,3]  +   \gamma_{10}[31,3,7]  +   \gamma_{23}[31,7,3] +   \gamma_{40}[15,3,23]  \\
&\quad  +   \gamma_{\{46, 105, 121, 152\}}[3,15,23]  +   \gamma_{53}[15,23,3]  +   \gamma_{59}[7,7,27]  +   \gamma_{65}[7,27,7]\\
&\quad  +   \gamma_{73}[7,11,23]  +   \gamma_{\{93, 163\}}[7,15,19]  +   \gamma_{87}[15,7,19]  +   \gamma_{100}[15,19,7]  = 0,
\end{align*}
\begin{align*}
&\gamma_{\{20, 68, 141\}}[3,7,31]  +  a_1[3,31,7] +    \gamma_{8}[7,3,31]  +   \gamma_{11}[31,3,7]+   \gamma_{41}[15,3,23]  \\
&\quad   +   \gamma_{\{34, 147, 150, 158, 162\}}[7,31,3]  +   \gamma_{\{22, 143\}}[31,7,3]   +   \gamma_{\{50, 106, 126, 153\}}[3,15,23]\\
&\quad +   \gamma_{\{52, 155\}}[15,23,3]     +   \gamma_{\{64, 157\}}[7,27,7]+   \gamma_{\{75, 161, 165\}}[7,15,19]   \\
&\quad    +   \gamma_{\{60, 149\}}[7,7,27]+   \gamma_{97}[7,11,23]  +   \gamma_{88}[15,7,19]  +   \gamma_{99}[15,19,7]  = 0,\\   
&a_2[3,7,31]  +  \gamma_{\{21, 69, 142\}}[3,31,7]  +   a_3[7,3,31]  +   \gamma_{9}[7,31,3] +   \gamma_{101}[15,7,19]   \\
&\quad   +   \gamma_{\{24, 70, 143\}}[31,3,7]  +   \gamma_{12}[31,7,3]  +   \gamma_{42}[15,23,3] +   \gamma_{98}[7,15,19] \\
&\quad  +   \gamma_{\{51, 107, 127, 131, 154\}}[3,15,23]  +   a_4[7,7,27]  +   \gamma_{\{54, 108, 128, 132, 155\}}[15,3,23]\\
&\quad  +   \gamma_{\{61, 120, 150\}}[7,27,7]   +   \gamma_{\{76, 137, 138, 139, 162\}}[7,11,23]  +   \gamma_{89}[15,19,7]  = 0,  
\end{align*}
where
\begin{align*}
a_1 &= \gamma_{\{31, 79, 114, 134, 144, 156\}},\ \
a_2 = \gamma_{\{33, 77, 116, 118, 123, 133, 146, 148, 156, 160\}},\\
a_3 &= \gamma_{\{36, 80, 117, 119, 124, 136, 147, 149, 157, 161\}},\ \
a_4 = \gamma_{\{66, 151, 158, 159, 163, 164, 165\}}.
\end{align*}
Computing from these equalities, we obtain
\begin{equation}\begin{cases}
a_i = 0, \ i = 1,2,3,4,\\
\gamma_j = 0, \ j = 1, 2, \ldots , 12, 15, 16, 18, 23, 29, 30,\\
 37, \ldots , 42, 44, 45, 49, 53, 55, 56, 57, 59, 63, 65,\\ 71, 72, 73, 81, 82, 83, 87, 88, 89, 91, 92, 94, \ldots , 101,\\
\gamma_{\{13, 140\}} =   
\gamma_{\{26, 145\}} =   
\gamma_{\{28, 159\}} =   
\gamma_{\{43, 152\}} =   
\gamma_{\{90, 163\}} =  0,\\ 
\gamma_{\{14, 141\}} =   
\gamma_{\{17, 142\}} =  
\gamma_{\{47, 153\}} =   
\gamma_{\{62, 156\}} =   
\gamma_{\{48, 154\}} = 0,\\  
\gamma_{\{25, 144, 157\}} =   
\gamma_{\{27, 146, 158\}} =  
\gamma_{\{58, 148\}} = 
\gamma_{\{19, 67, 140\}} = 0,\\ 
\gamma_{\{62, 74, 156, 160, 165\}} =   
\gamma_{\{32, 78, 115, 135, 145\}} =   
\gamma_{\{35, 159\}} =  0,\\
\gamma_{\{46, 105, 121, 152\}} =   
\gamma_{\{20, 68, 141\}} =  
\gamma_{\{34, 147, 150, 158, 162\}} = 0,\\  
\gamma_{\{93, 163\}} =  
\gamma_{\{22, 143\}} =  
\gamma_{\{50, 106, 126, 153\}} =  
\gamma_{\{52, 155\}} = 0,\\  
\gamma_{\{60, 149\}} =   
\gamma_{\{75, 161, 165\}} =   
\gamma_{\{21, 69, 142\}} =   
\gamma_{\{24, 70, 143\}} =  0,\\
\gamma_{\{51, 107, 127, 131, 154\}} =  
\gamma_{\{54, 108, 128, 132, 155\}} = 
\gamma_{\{61, 120, 150\}} = 0,\\   
\gamma_{\{64, 157\}} =   
\gamma_{\{76, 137, 138, 139, 162\}} = 0. 
\end{cases}\tag{\ref{8.6.4}.2}
\end{equation}

With the aid of (\ref{8.6.4}.2), the homomorphisms $g_1, g_2$ send (\ref{8.6.4}.1) to
\begin{align*}
&\gamma_{19}[3,7,31] +  \gamma_{32}[3,31,7] +   \gamma_{67}[7,3,31] +  \gamma_{\{28, 86\}}[7,31,3]\\
&\quad +  \gamma_{78}[31,3,7] +  \gamma_{104}[31,7,3] +  \gamma_{46}[3,15,23] +  \gamma_{105}[15,3,23]\\
&\quad +  \gamma_{111}[15,23,3] +  \gamma_{112}[7,7,27] +  \gamma_{115}[7,27,7] +  \gamma_{121}[7,11,23]\\
&\quad +  \gamma_{\{90, 125\}}[7,15,19] +  \gamma_{129}[15,7,19] +  \gamma_{135}[15,19,7] =0,\\  
&\gamma_{20}[3,7,31] + \gamma_{\{25, 31, 144\}}[3,31,7]  +  \gamma_{\{27, 34, 85, 146, 147, 158\}}[7,31,3]\\
&\quad +  \gamma_{68}[7,3,31] +  \gamma_{79}[31,3,7] +  \gamma_{\{103, 150\}}[31,7,3] +  \gamma_{50}[3,15,23]\\
&\quad +  \gamma_{106}[15,3,23] +  \gamma_{\{110, 162\}}[15,23,3] +  \gamma_{113}[7,7,27] \\
&\quad +  \gamma_{114}[7,27,7] +  \gamma_{126}[7,11,23] +  \gamma_{\{74, 75, 122, 160, 161, 165\}}[7,15,19]\\
&\quad +  \gamma_{130}[15,7,19] +  \gamma_{134}[15,19,7] =0. 
\end{align*}

These equalities imply
\begin{equation}\begin{cases}
\gamma_j = 0,\ j = 19, 20, 32, 46, 50, 67, 68, 78, 79, 104, 105,\\ 106, 111, 112, 113, 114, 115, 121, 126, 129, 130, 134, 135,\\
\gamma_{\{28, 86\}} =  
\gamma_{\{90, 125\}} =    
\gamma_{\{25, 31, 144\}} =  
\gamma_{\{27, 34, 85, 146, 147, 158\}} =  0,\\
\gamma_{\{103, 150\}} =     
\gamma_{\{110, 162\}} = 
\gamma_{\{74, 75, 122, 160, 161, 165\}} = 0. 
\end{cases}\tag{\ref{8.6.4}.3}
\end{equation}

With the aid of (\ref{8.6.4}.2) and (\ref{8.6.4}.3), the homomorphisms $g_3, g_4$ send (\ref{8.6.4}.1) to
\begin{align*}
&\gamma_{21}[3,7,31]  +  \gamma_{\{27, 33, 66, 146, 158\}}[3,31,7]  +    \gamma_{69}[7,3,31]  +   a_5[7,31,3]\\
&\quad  +   \gamma_{\{77, 151\}}[31,3,7]  +  a_6[31,7,3]  +   \gamma_{51}[3,15,23]  +   \gamma_{107}[15,3,23]\\
&\quad  +   a_7[15,23,3]  +   \gamma_{\{61, 120, 131, 150\}}[7,7,27]  +   a_8[7,27,7]+   a_9[7,15,19]\\
&\quad  +   \gamma_{\{127, 131\}}[7,11,23]    +   \gamma_{\{28, 118, 137, 139, 151\}}[15,7,19]  +   a_{10}[15,19,7]  = 0,\\
&\gamma_{24}[3,7,31]  +   a_{11}[3,31,7]  +   \gamma_{70}[7,3,31]  +   a_{12}[7,31,3] +   a_{13}[31,7,3]  \\
&\quad  +   \gamma_{\{80, 120, 137, 151\}}[31,3,7] +   \gamma_{54}[3,15,23]  +   \gamma_{108}[15,3,23]\\
&\quad  +   a_{14}[15,23,3]  +   \gamma_{132}[7,7,27]  +   a_{15}[7,27,7]  +   \gamma_{\{128, 132\}}[7,11,23]\\
&\quad  +   a_{16}[7,15,19]  +   a_{17}[15,7,19]  +   a_{18}[15,19,7]  = 0,  
\end{align*}
where
\begin{align*}
a_5 &= \gamma_{\{22, 24, 25, 36, 52, 54, 64, 70, 84, 117, 128, 144, 147\}},\ \
a_6 =  \gamma_{\{22, 24, 60, 70, 80, 102, 108, 119, 132\}},\\
a_7 &= \gamma_{\{52, 54, 108, 109, 124, 128, 132, 136, 161\}},\ \
a_8 = \gamma_{\{66, 74, 90, 116, 123, 158, 160, 164, 165\}},\\
a_9 &= \gamma_{\{74, 76, 90, 123, 138, 160, 162, 164, 165\}},\ \
a_{10} = \gamma_{\{28, 118, 133, 151, 164, 165\}},\\
a_{11} &= \gamma_{\{34, 36, 61, 66, 76, 147, 150, 158, 162\}},\\
a_{12} &=  \gamma_{\{13, 17, 21, 31, 33, 48, 51, 62, 69, 84, 85, 103, 110, 116, 127, 144, 146\}},\\
a_{13} &= \gamma_{\{14, 17, 21, 26, 58, 69, 77, 86, 102, 103, 107, 118, 131\}},\\
a_{14} &= \gamma_{\{43, 47, 48, 51, 107, 109, 110, 122, 123, 125, 127, 131, 133, 160\}},\\
a_{15} &= \gamma_{\{66, 75, 76, 90, 117, 120, 124, 138, 139, 158, 161, 162, 164, 165\}},\\
a_{16} &= \gamma_{\{75, 76, 90, 124, 139, 161, 162, 164, 165\}},\ \
a_{17} = \gamma_{\{28, 61, 119, 139, 150, 151\}},\\
a_{18} &= \gamma_{\{28, 61, 119, 136, 137, 138, 139, 150, 151, 164, 165\}}.
\end{align*}
From the above equalities, we obtain
\begin{equation}\begin{cases}
a_i = 0, \ i = 5,6, \ldots, 18,\\ \gamma_j = 0,\ j = 21, 24, 51, 54, 69, 70, 107, 108, 128, 132,\\
\gamma_{\{27, 33, 66, 146, 158\}} = 
  \gamma_{\{77, 151\}} =    
\gamma_{\{127, 131\}} =   0,\\
\gamma_{\{28, 118, 137, 139, 151\}} = 
\gamma_{\{61, 120, 131, 150\}} =  
\gamma_{\{80, 120, 137, 151\}} = 0.  
\end{cases}\tag{\ref{8.6.4}.4}
\end{equation}
With the aid of (\ref{8.6.4}.2), (\ref{8.6.4}.3) and (\ref{8.6.4}.4), the homomorphism $h$ sends (\ref{8.6.4}.1) to
\begin{align*}
&a_{19}[3,7,31]  +  \gamma_{\{28, 80, 102, 119\}}[3,31,7]  +    a_{20}[7,3,31]  +   \gamma_{80}[7,31,3]\\
&\  +   \gamma_{61}[31,3,7]  +   \gamma_{120}[31,7,3]  +   a_{21}[3,15,23]  +   \gamma_{76}[15,3,23]\\
&\  +   \gamma_{137}[15,23,3] +   a_{22}[7,7,27]  +   a_{23}[7,27,7]  +   \gamma_{\{61, 103, 120, 127, 138\}}[15,7,19] \\
&\ +  a_{24}[7,11,23] +   a_{25}[7,15,19]   +   \gamma_{\{61, 103, 120, 127, 139\}}[15,19,7]  =0,
\end{align*}
where
\begin{align*}
a_{19} &= \gamma_{\{25, 31, 34, 36, 84, 117, 144, 147\}},\ \
a_{20} = \gamma_{\{27, 33, 34, 36, 66, 85, 146, 147, 158\}},\\
a_{21} &= \gamma_{\{75, 90, 109, 124, 136, 161\}},\ \
a_{22} = \gamma_{\{66, 76, 110, 116, 117, 138, 158\}},\\
a_{23} &= \gamma_{\{28, 77, 118, 119, 137, 139\}},\ \
a_{25} = \gamma_{\{133, 136, 137, 139, 164, 165\}},\\
a_{24} &=  \gamma_{\{74, 75, 76, 90, 110, 122, 123, 124, 138, 160, 161, 164, 165\}}.
\end{align*}
From the above equalities, it implies
\begin{equation}\begin{cases}
a_i = 0, \ i = 19, 20, \ldots, 25,\\ 
\gamma_j = 0,\ j = 61, 76, 80, 120, 137,\\
\gamma_{\{28, 102, 119\}} =   
\gamma_{\{103, 127, 138\}} =  
\gamma_{\{103, 127, 139\}} = 0. 
\end{cases}\tag{\ref{8.6.4}.5}
\end{equation}

Combining (\ref{8.6.4}.2), (\ref{8.6.4}.3), (\ref{8.6.4}.4) and (\ref{8.6.4}.5), we get $\gamma_j = 0$ for $1 \leqslant j \leqslant 165$. The proposition is proved.
\end{proof}

\begin{props}\label{8.6.5} For $s \geqslant 3$, the elements $[a_{2,1,s,j}], 1 \leqslant j \leqslant 165,$  are linearly independent in $(\mathbb F_2\underset {\mathcal A}\otimes R_4)_{2^{s+3}+2^{s+1}+2^s-3}$.
\end{props}

\begin{proof} Suppose that there is a linear relation
\begin{equation}\sum_{j=1}^{165}\gamma_j[a_{2,1,s,j}] = 0, \tag {\ref{8.6.5}.1}
\end{equation}
with $\gamma_j \in \mathbb F_2$.

Applying the homomorphisms $f_1, f_2, \ldots, f_6$ to the relation (\ref{8.6.5}.1), we obtain
\begin{align*}
&\gamma_{1}w_{2,1,s,1} +  \gamma_{2}w_{2,1,s,2} +   \gamma_{15}w_{2,1,s,3} +  \gamma_{16}w_{2,1,s,4} +  \gamma_{29}w_{2,1,s,5} +  \gamma_{30}w_{2,1,s,6}\\
&\quad +  \gamma_{37}w_{2,1,s,7} +  \gamma_{44}w_{2,1,s,8} +  \gamma_{45}w_{2,1,s,9} +  \gamma_{55}w_{2,1,s,10} +  \gamma_{56}w_{2,1,s,11}\\
&\quad +  \gamma_{71}w_{2,1,s,12} +  \gamma_{81}w_{2,1,s,13} +  \gamma_{91}w_{2,1,s,14} +  \gamma_{92}w_{2,1,s,15} =0,\\  
&\gamma_{3}w_{2,1,s,1} +  \gamma_{5}w_{2,1,s,2} +  \gamma_{13}w_{2,1,s,3} +  \gamma_{18}w_{2,1,s,4} +  \gamma_{26}w_{2,1,s,5} +  \gamma_{28}w_{2,1,s,6}\\
&\quad +  \gamma_{38}w_{2,1,s,7} +  \gamma_{43}w_{2,1,s,8} +  \gamma_{49}w_{2,1,s,9} +  \gamma_{\{57, 72\}}w_{2,1,s,10} +  \gamma_{96}w_{2,1,s,11}\\
&\quad +  \gamma_{72}w_{2,1,s,12} +  \gamma_{\{82, 96\}}w_{2,1,s,13} +  \gamma_{\{63, 90\}}w_{2,1,s,14} +  \gamma_{63}w_{2,1,s,15} = 0,\\  
&\gamma_{4}w_{2,1,s,1} + \gamma_{6}w_{2,1,s,2} + \gamma_{14}w_{2,1,s,3} +  \gamma_{17}w_{2,1,s,4} +   \gamma_{25}w_{2,1,s,5} +  \gamma_{27}w_{2,1,s,6}\\
&\quad +  \gamma_{39}w_{2,1,s,7} +  \gamma_{47}w_{2,1,s,8} +  \gamma_{48}w_{2,1,s,9} +  \gamma_{\{58, 94\}}w_{2,1,s,10} +  \gamma_{95}w_{2,1,s,11}\\
&\quad +  \gamma_{94}w_{2,1,s,12} +  \gamma_{\{83, 95\}}w_{2,1,s,13} +  \gamma_{\{62, 74\}}w_{2,1,s,14} +  \gamma_{62}w_{2,1,s,15} = 0,
\end{align*}
\begin{align*}
&\gamma_{19}w_{2,1,s,1} +  \gamma_{32}w_{2,1,s,2} +  \gamma_{7}w_{2,1,s,3} +  \gamma_{35}w_{2,1,s,4} +  \gamma_{10}w_{2,1,s,5} +  \gamma_{23}w_{2,1,s,6}\\
&\quad +  \gamma_{46}w_{2,1,s,7} +  \gamma_{40}w_{2,1,s,8} +  \gamma_{53}w_{2,1,s,9} +  \gamma_{59}w_{2,1,s,10} +  \gamma_{65}w_{2,1,s,11}\\
&\quad +  \gamma_{73}w_{2,1,s,12} +  \gamma_{93}w_{2,1,s,13} +  \gamma_{87}w_{2,1,s,14} +  \gamma_{100}w_{2,1,s,15} = 0,\\  
&\gamma_{20}w_{2,1,s,1} +  \gamma_{31}w_{2,1,s,2} +  \gamma_{8}w_{2,1,s,3} +  \gamma_{34}w_{2,1,s,4} + \gamma_{11}w_{2,1,s,5} +  \gamma_{22}w_{2,1,s,6}\\
&\quad +   \gamma_{50}w_{2,1,s,7} +  \gamma_{41}w_{2,1,s,8} +  \gamma_{52}w_{2,1,s,9} +  \gamma_{60}w_{2,1,s,10} +  \gamma_{64}w_{2,1,s,11}\\
&\quad +  \gamma_{97}w_{2,1,s,12} +  \gamma_{75}w_{2,1,s,13} +  \gamma_{88}w_{2,1,s,14} +  \gamma_{99}w_{2,1,s,15} = 0,\\  
&\gamma_{33}w_{2,1,s,1} +  \gamma_{21}w_{2,1,s,2} +  \gamma_{36}w_{2,1,s,3} +  \gamma_{9}w_{2,1,s,4} +  \gamma_{24}w_{2,1,s,5} +  \gamma_{12}w_{2,1,s,6}\\
&\quad +  \gamma_{51}w_{2,1,s,7} +  \gamma_{54}w_{2,1,s,8} +  \gamma_{42}w_{2,1,s,9} +  \gamma_{66}w_{2,1,s,10} +  \gamma_{61}w_{2,1,s,11}\\
&\quad +  \gamma_{76}w_{2,1,s,12} +  \gamma_{98}w_{2,1,s,13} +  \gamma_{101}w_{2,1,s,14} +  \gamma_{89}w_{2,1,s,15} =0. 
\end{align*}

Computing from these equalities, we obtain
\begin{equation}
\gamma_j = 0, j = 1,\ldots, 66, 71,\ldots ,76,  81, 82, 83, 87, \ldots , 101.\tag{\ref{8.6.5}.2}
\end{equation}

With the aid of (\ref{8.6.5}.2), the homomorphisms $g_1, g_2, g_3, g_4$ send (\ref{8.6.5}.1) to
\begin{align*}
&\gamma_{140}w_{2,1,s,1} +  \gamma_{145}w_{2,1,s,2} +   \gamma_{67}w_{2,1,s,3} +  \gamma_{86}w_{2,1,s,4}\\
&\quad +  \gamma_{78}w_{2,1,s,5} +  \gamma_{104}w_{2,1,s,6} +  \gamma_{150}w_{2,1,s,7} +  \gamma_{105}w_{2,1,s,8}\\
&\quad +  \gamma_{111}w_{2,1,s,9} +  \gamma_{112}w_{2,1,s,10} +  \gamma_{115}w_{2,1,s,11} +  \gamma_{121}w_{2,1,s,12}\\
&\quad +  \gamma_{125}w_{2,1,s,13} +  \gamma_{129}w_{2,1,s,14} +  \gamma_{135}w_{2,1,s,15} =0,\\  
&\gamma_{141}w_{2,1,s,1} +  \gamma_{144}w_{2,1,s,2} +  \gamma_{68}w_{2,1,s,3} +  \gamma_{85}w_{2,1,s,4}\\
&\quad +  \gamma_{79}w_{2,1,s,5} +  \gamma_{103}w_{2,1,s,6} +  \gamma_{151}w_{2,1,s,7} +  \gamma_{106}w_{2,1,s,8}\\
&\quad +  \gamma_{110}w_{2,1,s,9} +  \gamma_{113}w_{2,1,s,10} +  \gamma_{114}w_{2,1,s,11} +  \gamma_{126}w_{2,1,s,12}\\
&\quad +  \gamma_{122}w_{2,1,s,13} +  \gamma_{130}w_{2,1,s,14} +  \gamma_{134}w_{2,1,s,15} = 0,\\  
&\gamma_{142}w_{2,1,s,1} +  \gamma_{146}w_{2,1,s,2} + \gamma_{69}w_{2,1,s,3} +  a_1w_{2,1,s,4} +   \gamma_{77}w_{2,1,s,5}\\
&\quad +  a_2w_{2,1,s,6} +  \gamma_{152}w_{2,1,s,7} +  \gamma_{107}w_{2,1,s,8} +  a_3w_{2,1,s,9}\\
&\quad +  \gamma_{131}w_{2,1,s,10} +  \gamma_{\{116, 123\}}w_{2,1,s,11} +  \gamma_{\{127, 131\}}w_{2,1,s,12}\\
&\quad +  \gamma_{123}w_{2,1,s,13} +  \gamma_{118}w_{2,1,s,14} +  \gamma_{\{118, 133\}}w_{2,1,s,15} = 0,\\  
&\gamma_{143}w_{2,1,s,1} + \gamma_{147}w_{2,1,s,2} +  \gamma_{70}w_{2,1,s,3} +  a_4w_{2,1,s,4} +  \gamma_{80}w_{2,1,s,5}\\
&\quad +  a_5w_{2,1,s,6} +  \gamma_{153}w_{2,1,s,7} +  \gamma_{108}w_{2,1,s,8} +  a_6w_{2,1,s,9}\\
&\quad +  \gamma_{132}w_{2,1,s,10} +  \gamma_{\{117, 124\}}w_{2,1,s,11} +  \gamma_{\{128, 132\}}w_{2,1,s,12}\\
&\quad +  \gamma_{124}w_{2,1,s,13} +  \gamma_{119}w_{2,1,s,14} +  \gamma_{\{119, 136\}}w_{2,1,s,15} =0, 
\end{align*}
where
\begin{align*}
&a_1 = \begin{cases} \gamma_{\{84, 163\}}, &s = 3,\\
\gamma_{84}, &s\geqslant 4, \end{cases}\ \
a_4 = \begin{cases} \gamma_{\{120, 138, 148, 155, 157, 158, 159, 163\}}, &s = 3,\\
                      \gamma_{148}, &s\geqslant 4, \end{cases}\\
&a_2 = \begin{cases} \gamma_{\{102, 164\}}, &s = 3,\\
                     \gamma_{102}, &s\geqslant 4, \end{cases}\ \
a_5 = \begin{cases} \gamma_{\{120, 137, 139, 149, 156, 157, 160, 164\}}, &s = 3,\\
                      \gamma_{149}, &s\geqslant 4, \end{cases}
\end{align*}
\begin{align*}
&a_3 = \begin{cases} \gamma_{\{109, 165\}}, &s = 3,\\
                      \gamma_{109}, &s\geqslant 4, \end{cases}\ \
a_6 = \begin{cases} \gamma_{\{137, 138, 139, 154, 158, 161, 162, 165\}}, &s = 3,\\
                      \gamma_{154}, &s\geqslant 4. \end{cases}
\end{align*}
These equalities imply
\begin{equation}\begin{cases}
a_i = 0, i = 1, 2, \ldots , 6,\\
\gamma_j = 0,\ j = 67, \ldots, 70, 77, \ldots, 80, 85, 86, 103, \ldots,108,\\
110,\ldots, 119, 121,  \ldots, 136, 140, \ldots, 147, 150,  \ldots, 153.
\end{cases}\tag{\ref{8.6.5}.3}
\end{equation}

With the aid of (\ref{8.6.5}.2), (\ref{8.6.5}.3) and (\ref{8.6.5}.4), the homomorphism $h$ sends (\ref{8.6.5}.1) to
\begin{align*}
&a_7w_{2,1,s,1} + a_8w_{2,1,s,2} +  \gamma_{155}w_{2,1,s,3} +  \gamma_{156}w_{2,1,s,4} + \gamma_{157}w_{2,1,s,5} +  \gamma_{120}w_{2,1,s,6}\\
&\quad +   a_9w_{2,1,s,7} +  \gamma_{158}w_{2,1,s,8} +  \gamma_{137}w_{2,1,s,9} +  \gamma_{159}w_{2,1,s,10} +  \gamma_{160}w_{2,1,s,11}\\
&\quad +  \gamma_{161}w_{2,1,s,12} +  \gamma_{162}w_{2,1,s,13} +  \gamma_{138}w_{2,1,s,14} +  \gamma_{139}w_{2,1,s,15} =0,
\end{align*}
where
\begin{align*}
a_7 = \begin{cases} \gamma_{148}, &s=3,\\    \gamma_{163}, & s\geqslant 4, \end{cases}\quad
a_8 = \begin{cases} \gamma_{149}, &s=3,\\    \gamma_{164}, & s\geqslant 4, \end{cases}\quad
a_9 = \begin{cases} \gamma_{154}, &s=3,\\    \gamma_{165}, & s\geqslant 4. \end{cases}
\end{align*}
From the above equalities, it implies
\begin{equation}
a_7 = a_8= a_9 = 0,  \gamma_j = 0,\ j = 120, 137, 138, 139, 155, \ldots , 162.\tag{\ref{8.6.5}.4}
\end{equation}

Combining (\ref{8.6.5}.2), (\ref{8.6.5}.3) and (\ref{8.6.5}.4), we get $\gamma_j = 0$ for $1 \leqslant j \leqslant 165$. The proposition is proved.
\end{proof}

\subsection{The case $s \geqslant 2, t = 1$ and $u \geqslant 3$}\label{8.7}\ 

\medskip
According to Kameko \cite{ka}, for $s\geqslant 2$ and $u \geqslant 3$, the dimension of the space $(\mathbb F_2\underset{\mathcal A}\otimes P_3)_{2^{s+u+1}+2^{s+1} + 2^s -3}$  is 14 with a basis given by the following classes:

\smallskip
\centerline{\begin{tabular}{ll}
&$w_{u,1,s,1} = [2^{s} - 1,2^{s+1} - 1,2^{s+u+1} - 1],$\cr 
&$w_{u,1,s,2} = [2^{s} - 1,2^{s+u+1} - 1,2^{s+1} - 1],$\cr 
&$w_{u,1,s,3} = [2^{s+1} - 1,2^{s} - 1,2^{s+u+1} - 1],$\cr 
&$w_{u,1,s,4} = [2^{s+1} - 1,2^{s+u+1} - 1,2^{s} - 1],$\cr 
&$w_{u,1,s,5} = [2^{s+u+1} - 1,2^{s} - 1,2^{s+1} - 1],$\cr 
&$w_{u,1,s,6} = [2^{s+u+1} - 1,2^{s+1} - 1,2^{s} - 1],$\cr 
&$w_{u,1,s,7} = [2^{s} - 1,2^{s+2} - 1,2^{s+u+1}-2^{s+1} - 1],$\cr 
&$w_{u,1,s,8} = [2^{s+2} - 1,2^{s} - 1,2^{s+u+1}-2^{s+1} - 1],$\cr 
&$w_{u,1,s,9} = [2^{s+2} - 1,2^{s+u+1}-2^{s+1} - 1,2^{s} - 1],$\cr 
&$w_{u,1,s,10} = [2^{s+1} - 1,2^{s+1} - 1,2^{s+u+1}-2^{s} - 1],$\cr 
&$w_{u,1,s,11} = [2^{s+1} - 1,2^{s+u+1}-2^{s} - 1,2^{s+1} - 1],$\cr 
&$w_{u,1,s,12} = [2^{s+1} - 1,2^{s+1}+2^{s} - 1,2^{s+u+1}-2^{s+1} - 1],$\cr 
&$w_{u,1,s,13} = [2^{s+1} - 1,2^{s+2} - 1,2^{s+u+1}- 2^{s+1}- 2^{s} - 1],$\cr 
&$w_{u,1,s,14} = [2^{s+2} - 1,2^{s+1} - 1,2^{s+u+1}- 2^{s+1}- 2^{s} - 1].$\cr
\end{tabular}}

\smallskip
So, we easily obtain

\begin{props}\label{8.7.1} For any positive integers $s \geqslant 2$ and $u \geqslant 3$, we have
$$\dim (\mathbb F_2\underset{\mathcal A}\otimes Q_4)_{2^{s+u+1}+ 2^{s+1} + 2^s -3} = 56.$$
\end{props}

Now, we determine $(\mathbb F_2\underset{\mathcal A}\otimes R_4)_{2^{s+u+1}+ 2^{s+1} + 2^s -3}$. We have

\begin{thms}\label{dlc8.7} For $s \geqslant 2, u\geqslant 3$, $(\mathbb F_2\underset{\mathcal A}\otimes R_4)_{2^{s+u+1}+ 2^{s+1} + 2^s -3}$ is  an $\mathbb F_2$-vector space of dimension 154 with a basis consisting of all the classes represented by the monomials $a_{u,1,s,j}, 1 \leqslant j \leqslant 154$, which are determined as follows:

\medskip
For $s \geqslant 2$,

\smallskip
\centerline{\begin{tabular}{ll}
&$1.\ (1,2^{s} - 2,2^{s+1} - 1,2^{s+u+1} - 1),$\cr
&$2.\ (1,2^{s} - 2,2^{s+u+1} - 1,2^{s+1} - 1),$\cr 
&$3.\ (1,2^{s+1} - 1,2^{s} - 2,2^{s+u+1} - 1),$\cr 
&$4.\ (1,2^{s+1} - 1,2^{s+u+1} - 1,2^{s} - 2),$\cr 
&$5.\ (1,2^{s+u+1} - 1,2^{s} - 2,2^{s+1} - 1),$\cr
&$6.\ (1,2^{s+u+1} - 1,2^{s+1} - 1,2^{s} - 2),$\cr 
&$7.\ (2^{s+1} - 1,1,2^{s} - 2,2^{s+u+1} - 1),$\cr
&$8.\ (2^{s+1} - 1,1,2^{s+u+1} - 1,2^{s} - 2),$\cr 
&$9.\ (2^{s+1} - 1,2^{s+u+1} - 1,1,2^{s} - 2),$\cr
&$10.\ (2^{s+u+1} - 1,1,2^{s} - 2,2^{s+1} - 1),$\cr 
&$11.\ (2^{s+u+1} - 1,1,2^{s+1} - 1,2^{s} - 2),$\cr
&$12.\ (2^{s+u+1} - 1,2^{s+1} - 1,1,2^{s} - 2),$\cr 
&$13.\ (1,2^{s} - 1,2^{s+1} - 2,2^{s+u+1} - 1),$\cr
&$14.\ (1,2^{s} - 1,2^{s+u+1} - 1,2^{s+1} - 2),$\cr 
&$15.\ (1,2^{s+1} - 2,2^{s} - 1,2^{s+u+1} - 1),$\cr
&$16.\ (1,2^{s+1} - 2,2^{s+u+1} - 1,2^{s} - 1),$\cr 
&$17.\ (1,2^{s+u+1} - 1,2^{s} - 1,2^{s+1} - 2),$\cr
&$18.\ (1,2^{s+u+1} - 1,2^{s+1} - 2,2^{s} - 1),$\cr 
&$19.\ (2^{s} - 1,1,2^{s+1} - 2,2^{s+u+1} - 1),$\cr
&$20.\ (2^{s} - 1,1,2^{s+u+1} - 1,2^{s+1} - 2),$\cr 
&$21.\ (2^{s} - 1,2^{s+u+1} - 1,1,2^{s+1} - 2),$\cr
&$22.\ (2^{s+u+1} - 1,1,2^{s} - 1,2^{s+1} - 2),$\cr 
&$23.\ (2^{s+u+1} - 1,1,2^{s+1} - 2,2^{s} - 1),$\cr
&$24.\ (2^{s+u+1} - 1,2^{s} - 1,1,2^{s+1} - 2),$\cr 
&$25.\ (1,2^{s} - 1,2^{s+1} - 1,2^{s+u+1} - 2),$\cr
&$26.\ (1,2^{s} - 1,2^{s+u+1} - 2,2^{s+1} - 1),$\cr 
&$27.\ (1,2^{s+1} - 1,2^{s} - 1,2^{s+u+1} - 2),$\cr
&$28.\ (1,2^{s+1} - 1,2^{s+u+1} - 2,2^{s} - 1),$\cr 
&$29.\ (1,2^{s+u+1} - 2,2^{s} - 1,2^{s+1} - 1),$\cr
&$30.\ (1,2^{s+u+1} - 2,2^{s+1} - 1,2^{s} - 1),$\cr 
&$31.\ (2^{s} - 1,1,2^{s+1} - 1,2^{s+u+1} - 2),$\cr
&$32.\ (2^{s} - 1,1,2^{s+u+1} - 2,2^{s+1} - 1),$\cr 
&$33.\ (2^{s} - 1,2^{s+1} - 1,1,2^{s+u+1} - 2),$\cr
&$34.\ (2^{s+1} - 1,1,2^{s} - 1,2^{s+u+1} - 2),$\cr 
\end{tabular}}
\centerline{\begin{tabular}{ll}
&$35.\ (2^{s+1} - 1,1,2^{s+u+1} - 2,2^{s} - 1),$\cr
&$36.\ (2^{s+1} - 1,2^{s} - 1,1,2^{s+u+1} - 2),$\cr
&$37.\ (1,2^{s} - 2,2^{s+2} - 1,2^{s+u+1}-2^{s+1} - 1),$\cr 
&$38.\ (1,2^{s+2} - 1,2^{s} - 2,2^{s+u+1}-2^{s+1} - 1),$\cr 
&$39.\ (1,2^{s+2} - 1,2^{s+u+1}-2^{s+1} - 1,2^{s} - 2),$\cr 
&$40.\ (2^{s+2} - 1,1,2^{s} - 2,2^{s+u+1}-2^{s+1} - 1),$\cr 
&$41.\ (2^{s+2} - 1,1,2^{s+u+1}-2^{s+1} - 1,2^{s} - 2),$\cr 
&$42.\ (2^{s+2} - 1,2^{s+u+1}-2^{s+1} - 1,1,2^{s} - 2),$\cr 
&$43.\ (1,2^{s} - 1,2^{s+2} - 2,2^{s+u+1}-2^{s+1} - 1),$\cr 
&$44.\ (1,2^{s+2} - 2,2^{s} - 1,2^{s+u+1}-2^{s+1} - 1),$\cr 
&$45.\ (1,2^{s+2} - 2,2^{s+u+1}-2^{s+1} - 1,2^{s} - 1),$\cr 
&$46.\ (2^{s} - 1,1,2^{s+2} - 2,2^{s+u+1}-2^{s+1} - 1),$\cr 
&$47.\ (1,2^{s} - 1,2^{s+2} - 1,2^{s+u+1}-2^{s+1} - 2),$\cr 
&$48.\ (1,2^{s+2} - 1,2^{s} - 1,2^{s+u+1}-2^{s+1} - 2),$\cr 
&$49.\ (1,2^{s+2} - 1,2^{s+u+1}-2^{s+1} - 2,2^{s} - 1),$\cr 
&$50.\ (2^{s} - 1,1,2^{s+2} - 1,2^{s+u+1}-2^{s+1} - 2),$\cr 
&$51.\ (2^{s} - 1,2^{s+2} - 1,1,2^{s+u+1}-2^{s+1} - 2),$\cr 
&$52.\ (2^{s+2} - 1,1,2^{s} - 1,2^{s+u+1}-2^{s+1} - 2),$\cr 
&$53.\ (2^{s+2} - 1,1,2^{s+u+1}-2^{s+1} - 2,2^{s} - 1),$\cr 
&$54.\ (2^{s+2} - 1,2^{s} - 1,1,2^{s+u+1}-2^{s+1} - 2),$\cr 
&$55.\ (1,2^{s+1} - 2,2^{s+1} - 1,2^{s+u+1}-2^{s} - 1),$\cr 
&$56.\ (1,2^{s+1} - 2,2^{s+u+1}-2^{s} - 1,2^{s+1} - 1),$\cr 
&$57.\ (1,2^{s+1} - 1,2^{s+1} - 2,2^{s+u+1}-2^{s} - 1),$\cr 
&$58.\ (1,2^{s+1} - 1,2^{s+u+1}-2^{s} - 1,2^{s+1} - 2),$\cr 
&$59.\ (2^{s+1} - 1,1,2^{s+1} - 2,2^{s+u+1}-2^{s} - 1),$\cr 
&$60.\ (2^{s+1} - 1,1,2^{s+u+1}-2^{s} - 1,2^{s+1} - 2),$\cr 
&$61.\ (2^{s+1} - 1,2^{s+u+1}-2^{s} - 1,1,2^{s+1} - 2),$\cr 
&$62.\ (1,2^{s+1} - 1,2^{s+1} - 1,2^{s+u+1}-2^{s} - 2),$\cr 
&$63.\ (1,2^{s+1} - 1,2^{s+u+1}-2^{s} - 2,2^{s+1} - 1),$\cr 
&$64.\ (2^{s+1} - 1,1,2^{s+1} - 1,2^{s+u+1}-2^{s} - 2),$\cr 
&$65.\ (2^{s+1} - 1,1,2^{s+u+1}-2^{s} - 2,2^{s+1} - 1),$\cr 
&$66.\ (2^{s+1} - 1,2^{s+1} - 1,1,2^{s+u+1}-2^{s} - 2),$\cr
&$67.\ (3,2^{s+1} - 3,2^{s} - 2,2^{s+u+1} - 1),$\cr 
&$68.\ (3,2^{s+1} - 3,2^{s+u+1} - 1,2^{s} - 2),$\cr 
&$69.\ (3,2^{s+u+1} - 1,2^{s+1} - 3,2^{s} - 2),$\cr 
&$70.\ (2^{s+u+1} - 1,3,2^{s+1} - 3,2^{s} - 2),$\cr 
&$71.\ (1,2^{s+1} - 2,2^{s+1}+2^{s} - 1,2^{s+u+1}-2^{s+1} - 1),$\cr 
&$72.\ (1,2^{s+1} - 1,2^{s+1}+2^{s} - 2,2^{s+u+1}-2^{s+1} - 1),$\cr 
&$73.\ (2^{s+1} - 1,1,2^{s+1}+2^{s} - 2,2^{s+u+1}-2^{s+1} - 1),$\cr 
&$74.\ (1,2^{s+1} - 1,2^{s+1}+2^{s} - 1,2^{s+u+1}-2^{s+1} - 2),$\cr 
&$75.\ (2^{s+1} - 1,1,2^{s+1}+2^{s} - 1,2^{s+u+1}-2^{s+1} - 2),$\cr 
&$76.\ (2^{s+1} - 1,2^{s+1}+2^{s} - 1,1,2^{s+u+1}-2^{s+1} - 2),$\cr 
&$77.\ (3,2^{s+1} - 1,2^{s+u+1} - 3,2^{s} - 2),$\cr 
&$78.\ (3,2^{s+u+1} - 3,2^{s} - 2,2^{s+1} - 1),$\cr 
&$79.\ (3,2^{s+u+1} - 3,2^{s+1} - 1,2^{s} - 2),$\cr 
\end{tabular}}
\centerline{\begin{tabular}{ll}
&$80.\ (2^{s+1} - 1,3,2^{s+u+1} - 3,2^{s} - 2),$\cr 
&$81.\ (3,2^{s} - 1,2^{s+1} - 3,2^{s+u+1} - 2),$\cr 
&$82.\ (3,2^{s+1} - 3,2^{s} - 1,2^{s+u+1} - 2),$\cr 
&$83.\ (3,2^{s+1} - 3,2^{s+u+1} - 2,2^{s} - 1),$\cr 
&$84.\ (3,2^{s} - 1,2^{s+u+1} - 3,2^{s+1} - 2),$\cr 
&$85.\ (3,2^{s+u+1} - 3,2^{s} - 1,2^{s+1} - 2),$\cr 
&$86.\ (3,2^{s+u+1} - 3,2^{s+1} - 2,2^{s} - 1),$\cr 
&$87.\ (1,2^{s+1} - 2,2^{s+2} - 1,2^{s+u+1}- 2^{s+1}- 2^{s} - 1),$\cr 
&$88.\ (1,2^{s+2} - 1,2^{s+1} - 2,2^{s+u+1}- 2^{s+1}- 2^{s} - 1),$\cr 
&$89.\ (2^{s+2} - 1,1,2^{s+1} - 2,2^{s+u+1}- 2^{s+1}- 2^{s} - 1),$\cr 
&$90.\ (1,2^{s+1} - 1,2^{s+2} - 2,2^{s+u+1}- 2^{s+1}- 2^{s} - 1),$\cr 
&$91.\ (1,2^{s+2} - 2,2^{s+1} - 1,2^{s+u+1}- 2^{s+1}- 2^{s} - 1),$\cr 
&$92.\ (2^{s+1} - 1,1,2^{s+2} - 2,2^{s+u+1}- 2^{s+1}- 2^{s} - 1),$\cr 
&$93.\ (1,2^{s+1} - 1,2^{s+2} - 1,2^{s+u+1}- 2^{s+1}- 2^{s} - 2),$\cr 
&$94.\ (1,2^{s+2} - 1,2^{s+1} - 1,2^{s+u+1}- 2^{s+1}- 2^{s} - 2),$\cr 
&$95.\ (2^{s+1} - 1,1,2^{s+2} - 1,2^{s+u+1}- 2^{s+1}- 2^{s} - 2),$\cr 
&$96.\ (2^{s+1} - 1,2^{s+2} - 1,1,2^{s+u+1}- 2^{s+1}- 2^{s} - 2),$\cr 
&$97.\ (2^{s+2} - 1,1,2^{s+1} - 1,2^{s+u+1}- 2^{s+1}- 2^{s} - 2),$\cr 
&$98.\ (2^{s+2} - 1,2^{s+1} - 1,1,2^{s+u+1}- 2^{s+1}- 2^{s} - 2),$\cr 
&$99.\ (3,2^{s+2} - 3,2^{s} - 2,2^{s+u+1}-2^{s+1} - 1),$\cr 
&$100.\ (3,2^{s+2} - 3,2^{s+u+1}-2^{s+1} - 1,2^{s} - 2),$\cr 
&$101.\ (3,2^{s+2} - 1,2^{s+u+1}-2^{s+1} - 3,2^{s} - 2),$\cr 
&$102.\ (2^{s+2} - 1,3,2^{s+u+1}-2^{s+1} - 3,2^{s} - 2),$\cr 
&$103.\ (3,2^{s} - 1,2^{s+2} - 3,2^{s+u+1}-2^{s+1} - 2),$\cr 
&$104.\ (3,2^{s+2} - 3,2^{s} - 1,2^{s+u+1}-2^{s+1} - 2),$\cr 
&$105.\ (3,2^{s+2} - 3,2^{s+u+1}-2^{s+1} - 2,2^{s} - 1),$\cr 
&$106.\ (3,2^{s+1} - 3,2^{s+1} - 2,2^{s+u+1}-2^{s} - 1),$\cr 
&$107.\ (3,2^{s+1} - 3,2^{s+u+1}-2^{s} - 1,2^{s+1} - 2),$\cr 
&$108.\ (3,2^{s+1} - 3,2^{s+1} - 1,2^{s+u+1}-2^{s} - 2),$\cr 
&$109.\ (3,2^{s+1} - 3,2^{s+u+1}-2^{s} - 2,2^{s+1} - 1),$\cr 
&$110.\ (3,2^{s+1} - 1,2^{s+1} - 3,2^{s+u+1}-2^{s} - 2),$\cr 
&$111.\ (2^{s+1} - 1,3,2^{s+1} - 3,2^{s+u+1}-2^{s} - 2),$\cr 
&$112.\ (3,2^{s+1} - 1,2^{s+u+1}-2^{s} - 3,2^{s+1} - 2),$\cr 
&$113.\ (2^{s+1} - 1,3,2^{s+u+1}-2^{s} - 3,2^{s+1} - 2),$\cr 
&$114.\ (7,2^{s+u+1} - 5,2^{s+1} - 3,2^{s} - 2),$\cr 
&$115.\ (3,2^{s+1} - 3,2^{s+1}+2^{s} - 2,2^{s+u+1}-2^{s+1} - 1),$\cr 
&$116.\ (3,2^{s+1} - 3,2^{s+1}+2^{s} - 1,2^{s+u+1}-2^{s+1} - 2),$\cr 
&$117.\ (3,2^{s+1} - 1,2^{s+1}+2^{s} - 3,2^{s+u+1}-2^{s+1} - 2),$\cr 
&$118.\ (2^{s+1} - 1,3,2^{s+1}+2^{s} - 3,2^{s+u+1}-2^{s+1} - 2),$\cr 
&$119.\ (3,2^{s+1} - 3,2^{s+2} - 2,2^{s+u+1}- 2^{s+1}- 2^{s} - 1),$\cr 
&$120.\ (3,2^{s+1} - 3,2^{s+2} - 1,2^{s+u+1}- 2^{s+1}- 2^{s} - 2),$\cr 
&$121.\ (3,2^{s+2} - 1,2^{s+1} - 3,2^{s+u+1}- 2^{s+1}- 2^{s} - 2),$\cr 
&$122.\ (2^{s+2} - 1,3,2^{s+1} - 3,2^{s+u+1}- 2^{s+1}- 2^{s} - 2),$\cr 
&$123.\ (3,2^{s+2} - 3,2^{s+1} - 2,2^{s+u+1}- 2^{s+1}- 2^{s} - 1),$\cr 
&$124.\ (3,2^{s+1} - 1,2^{s+2} - 3,2^{s+u+1}- 2^{s+1}- 2^{s} - 2),$\cr 
\end{tabular}}
\centerline{\begin{tabular}{ll}
&$125.\ (3,2^{s+2} - 3,2^{s+1} - 1,2^{s+u+1}- 2^{s+1}- 2^{s} - 2),$\cr 
&$126.\ (2^{s+1} - 1,3,2^{s+2} - 3,2^{s+u+1}- 2^{s+1}- 2^{s} - 2),$\cr 
&$127.\ (7,2^{s+2} - 5,2^{s+u+1}-2^{s+1} - 3,2^{s} - 2),$\cr 
&$128.\ (7,2^{s+2} - 5,2^{s+1} - 3,2^{s+u+1}- 2^{s+1}- 2^{s} - 2),$\cr
\end{tabular}}

\medskip
For $s = 2$,

\medskip
\centerline{\begin{tabular}{ll}
$129.\  (3,3,4,2^{u+3} - 1),$& $130.\  (3,3,2^{u+3} - 1,4),$\cr 
$131.\  (3,2^{u+3} - 1,3,4),$& $132.\  (2^{u+3}-1,3,3,4),$\cr 
$133.\  (3,3,7,2^{u+3} - 4),$& $134.\  (3,3,2^{u+3} - 4,7),$\cr 
$135.\  (3,7,3,2^{u+3} - 4),$& $136.\  (7,3,3,2^{u+3} - 4),$\cr 
$137.\  (3,7,2^{u+3} - 5,4),$& $138.\  (7,3,2^{u+3} - 5,4),$\cr 
$139.\  (7,2^{u+3} - 5,3,4),$& $140.\  (7,7,2^{u+3} - 7,2),$\cr 
$141.\  (3,3,12,2^{u+3} - 9),$& $142.\  (3,3,15,2^{u+3} - 12),$\cr 
$143.\  (3,15,3,2^{u+3} - 12),$& $144.\  (15,3,3,2^{u+3} - 12),$\cr 
$145.\  (3,7,7,2^{u+3} - 8),$& $146.\  (7,3,7,2^{u+3} - 8),$\cr 
$147.\  (7,7,3,2^{u+3} - 8),$& $148.\  (7,7,2^{u+3} - 8,3),$\cr 
$149.\  (3,7,11,2^{u+3} - 12),$& $150.\  (7,3,11,2^{u+3} - 12),$\cr 
$151.\  (7,11,3,2^{u+3} - 12),$& $152.\  (7,7,8,2^{u+3} - 13),$\cr 
$153.\  (7,7,9,2^{u+3} - 14),$& $154.\  (7,7,11,2^{u+3} - 16).$\cr
\end{tabular}}

\medskip
For $s \geqslant 3$

\medskip
\centerline{\begin{tabular}{ll}
&$129.\ (3,2^{s} - 3,2^{s+1} - 2,2^{s+u+1} - 1),$\cr 
&$130.\ (3,2^{s} - 3,2^{s+u+1} - 1,2^{s+1} - 2),$\cr 
&$131.\ (3,2^{s+u+1} - 1,2^{s} - 3,2^{s+1} - 2),$\cr 
&$132.\ (2^{s+u+1} - 1,3,2^{s} - 3,2^{s+1} - 2),$\cr 
&$133.\ (3,2^{s} - 3,2^{s+1} - 1,2^{s+u+1} - 2),$\cr 
&$134.\ (3,2^{s} - 3,2^{s+u+1} - 2,2^{s+1} - 1),$\cr 
&$135.\ (3,2^{s+1} - 1,2^{s} - 3,2^{s+u+1} - 2),$\cr 
&$136.\ (2^{s+1} - 1,3,2^{s} - 3,2^{s+u+1} - 2),$\cr 
&$137.\ (2^{s} - 1,3,2^{s+1} - 3,2^{s+u+1} - 2),$\cr 
&$138.\ (2^{s} - 1,3,2^{s+u+1} - 3,2^{s+1} - 2),$\cr 
&$139.\ (3,2^{s} - 3,2^{s+2} - 2,2^{s+u+1}-2^{s+1} - 1),$\cr 
&$140.\ (3,2^{s} - 3,2^{s+2} - 1,2^{s+u+1}-2^{s+1} - 2),$\cr 
&$141.\ (3,2^{s+2} - 1,2^{s} - 3,2^{s+u+1}-2^{s+1} - 2),$\cr 
&$142.\ (2^{s+2} - 1,3,2^{s} - 3,2^{s+u+1}-2^{s+1} - 2),$\cr 
&$143.\ (2^{s} - 1,3,2^{s+2} - 3,2^{s+u+1}-2^{s+1} - 2),$\cr 
&$144.\ (7,2^{s+1} - 5,2^{s} - 3,2^{s+u+1} - 2),$\cr 
&$145.\ (7,2^{s+1} - 5,2^{s+u+1} - 3,2^{s} - 2),$\cr 
&$146.\ (7,2^{s+u+1} - 5,2^{s} - 3,2^{s+1} - 2),$\cr 
&$147.\ (7,2^{s+2} - 5,2^{s} - 3,2^{s+u+1}-2^{s+1} - 2),$\cr 
&$148.\ (7,2^{s+1} - 5,2^{s+1} - 3,2^{s+u+1}-2^{s} - 2),$\cr 
&$149.\ (7,2^{s+1} - 5,2^{s+u+1}-2^{s} - 3,2^{s+1} - 2),$\cr 
&$150.\ (7,2^{s+1} - 5,2^{s+1}+2^{s} - 3,2^{s+u+1}-2^{s+1} - 2),$\cr 
&$151.\ (7,2^{s+1} - 5,2^{s+2} - 3,2^{s+u+1}- 2^{s+1}- 2^{s} - 2).$\cr
\end{tabular}}

\medskip
For $s = 3$,
$$152.\ (7,7,9,2^{u+4}-2),\ 153.\ (7,7,2^{u+4}-7,14),\ 154.\ (7,7,25,2^{u+4}-18).$$

For $s \geqslant 4$,

\medskip
\centerline{\begin{tabular}{ll}
&$152.\ (7,2^{s} - 5,2^{s+1} - 3,2^{s+u+1} - 2),$\cr 
&$153.\ (7,2^{s} - 5,2^{s+u+1} - 3,2^{s+1} - 2),$\cr 
&$154.\ (7,2^{s} - 5,2^{s+2} - 3,2^{s+u+1}-2^{s+1} - 2).$\cr
\end{tabular}}
\end{thms}

We prove the theorem by proving some propositions.

\smallskip
Combining Lemmas \ref{3.2}, \ref{3.3},  \ref{5.7}, \ref{6.2.1}, \ref{6.2.2}, \ref{8.2.1}, \ref{8.4.1}, \ref{8.4.2}. \ref{8.5.1}, \ref{8.5.2}, \ref{8.6.1}, \ref{8.6.2}, Theorem \ref{2.4} and the results in Section \ref{7}, we obtain the following.

\begin{props}\label{mdc8.7} The $\mathbb F_2$-vector space $(\mathbb F_2\underset {\mathcal A}\otimes R_4)_{2^{s+u+1}+ 2^{s+1}+2^s -3}$ is generated by the  elements listed in Theorem \ref{dlc8.7}.
\end{props}

Now, we prove that the elements listed in Theorem \ref{dlc8.7} are linearly independent.

\begin{props}\label{8.7.2} For $u \geqslant 3$, the elements $[a_{u,1,2,j}], 1 \leqslant j \leqslant 154,$  are linearly independent in $(\mathbb F_2\underset {\mathcal A}\otimes R_4)_{2^{u+3}+9}$.
\end{props}

\begin{proof} Suppose that there is a linear relation
\begin{equation}\sum_{j=1}^{154}\gamma_j[a_{u,1,2,j}] = 0, \tag {\ref{8.7.2}.1}
\end{equation}
with $\gamma_j \in \mathbb F_2$.

Applying the homomorphisms $f_1, f_2, f_3$ to the relation (\ref{8.7.2}.1), we have
\begin{align*}
&\gamma_{1}w_{u,1,2,1}  +   \gamma_{2}w_{u,1,2,2}  +    \gamma_{15}w_{u,1,2,3}  +   \gamma_{16}w_{u,1,2,4}  +   \gamma_{29}w_{u,1,2,5}\\
&\quad  +   \gamma_{30}w_{u,1,2,6}  +   \gamma_{37}w_{u,1,2,7}  +   \gamma_{44}w_{u,1,2,8}  +   \gamma_{45}w_{u,1,2,9}  +   \gamma_{55}w_{u,1,2,10}\\
&\quad  +   \gamma_{56}w_{u,1,2,11}  +   \gamma_{71}w_{u,1,2,12}  +   \gamma_{87}w_{u,1,2,13}  +   \gamma_{91}w_{u,1,2,14}  =0,\\   
&\gamma_{3}w_{u,1,2,1}  +  \gamma_{5}w_{u,1,2,2}  +   \gamma_{\{13, 129\}}w_{u,1,2,3}  +   \gamma_{18}w_{u,1,2,4}\\
&\quad  +   \gamma_{\{26, 134\}}w_{u,1,2,5}  +   \gamma_{\{28, 148\}}w_{u,1,2,6}  +   \gamma_{38}w_{u,1,2,7}\\
&\quad  +   \gamma_{\{43, 141\}}w_{u,1,2,8}  +   \gamma_{49}w_{u,1,2,9}  +   \gamma_{\{57, 63, 72\}}w_{u,1,2,10}\\
&\quad  +   \gamma_{63}w_{u,1,2,11}  +   \gamma_{72}w_{u,1,2,12}  +   \gamma_{88}w_{u,1,2,13}  +   \gamma_{\{90, 152\}}w_{u,1,2,14}  = 0,\\   
&\gamma_{4}w_{u,1,2,1}  +  \gamma_{6}w_{u,1,2,2}  +   \gamma_{\{14, 130\}}w_{u,1,2,3}  +   \gamma_{\{17, 131\}}w_{u,1,2,4}\\
&\quad  +  \gamma_{\{25, 133, 146\}}w_{u,1,2,5}  +   \gamma_{\{27, 135, 147\}}w_{u,1,2,6}  +    \gamma_{39}w_{u,1,2,7}  +   \gamma_{\{47, 142\}}w_{u,1,2,8}\\
&\quad  +   \gamma_{\{48, 143\}}w_{u,1,2,9} +   \gamma_{\{58, 62, 93, 137, 145\}}w_{u,1,2,10}  +   \gamma_{\{62, 94, 145\}}w_{u,1,2,11}\\
&\quad  +   \gamma_{93}w_{u,1,2,12}  +   \gamma_{94}w_{u,1,2,13}  +   \gamma_{\{74, 149, 154\}}w_{u,1,2,14}  = 0.   
\end{align*}
Computing from these equalities, we obtain
\begin{equation}\begin{cases}
\gamma_j = 0, \ j = 1, 2, 3, 4, 5, 6, 15, 16, 18, 29, 30, 37, 38,\\ 39, 44, 45, 49, 55, 56, 57, 63, 71, 72, 87, 88, 91, 93, 94, \\
\gamma_{\{13, 129\}} =   
\gamma_{\{26, 134\}} =  
\gamma_{\{28, 148\}} =  
\gamma_{\{43, 141\}} =  0,\\ 
\gamma_{\{90, 152\}} =   
\gamma_{\{14, 130\}} =   
\gamma_{\{17, 131\}} =   
\gamma_{\{47, 142\}} =   0,\\
\gamma_{\{48, 143\}} =  
\gamma_{\{25, 133, 146\}} =    
\gamma_{\{27, 135, 147\}} = 0,\\ 
\gamma_{\{62, 145\}} =   
\gamma_{\{58, 62, 137, 145\}} =   
\gamma_{\{74, 149, 154\}} = 0.
\end{cases}\tag{\ref{8.7.2}.2}
\end{equation}

With the aid of (\ref{8.7.2}.2), the homomorphisms $f_4, f_5, f_6$ send (\ref{8.7.2}.1) to
\begin{align*}
&\gamma_{\{19, 67, 129\}}w_{u,1,2,1}  +  \gamma_{\{32, 78, 109, 134\}}w_{u,1,2,2}  +   \gamma_{7}w_{u,1,2,3}\\
&\quad  +   \gamma_{\{35, 148\}}w_{u,1,2,4}  +   \gamma_{10}w_{u,1,2,5}  +   \gamma_{23}w_{u,1,2,6}  +   \gamma_{40}w_{u,1,2,8}\\
&\quad   +   \gamma_{\{46, 99, 115, 141\}}w_{u,1,2,7}  +   \gamma_{53}w_{u,1,2,9} +   \gamma_{59}w_{u,1,2,10}\\
&\quad     +   \gamma_{65}w_{u,1,2,11}  +   \gamma_{73}w_{u,1,2,12}  +   \gamma_{\{92, 152\}}w_{u,1,2,13}  +   \gamma_{89}w_{u,1,2,14}  = 0,\\   
&\gamma_{\{20, 68, 130\}}w_{u,1,2,1}  +  a_1w_{u,1,2,2}  +   \gamma_{8}w_{u,1,2,3}  +   \gamma_{11}w_{u,1,2,5}\\
&\quad  +   \gamma_{\{34, 136, 139, 147, 151\}}w_{u,1,2,4}  +   \gamma_{\{22, 132\}}w_{u,1,2,6} +   \gamma_{\{50, 100, 120, 142\}}w_{u,1,2,7} \\
&\quad  +   \gamma_{41}w_{u,1,2,8}  +  \gamma_{\{52, 144\}}w_{u,1,2,9}  +   \gamma_{\{60, 97, 138\}}w_{u,1,2,10}  +    \gamma_{\{64, 97, 146\}}w_{u,1,2,11}\\
&\quad  +   \gamma_{95}w_{u,1,2,12}  +   \gamma_{\{75, 150, 154\}}w_{u,1,2,13}  +   \gamma_{97}w_{u,1,2,14}  = 0,\\   
&a_2w_{u,1,2,1}  +  \gamma_{\{21, 69, 131\}}w_{u,1,2,2}  +   a_3w_{u,1,2,3}  +   \gamma_{9}w_{u,1,2,4}\\
&\quad   +   \gamma_{\{24, 70, 132\}}w_{u,1,2,5}  +   \gamma_{\{51, 101, 121, 143\}}w_{u,1,2,7}   +   \gamma_{12}w_{u,1,2,6}\\
&\quad  +   \gamma_{\{54, 102, 122, 144\}}w_{u,1,2,8}  +   \gamma_{42}w_{u,1,2,9}   +   a_4w_{u,1,2,10}  +   \gamma_{\{61, 114, 139\}}w_{u,1,2,11}\\
&\quad  +   \gamma_{96}w_{u,1,2,13}  +   \gamma_{\{76, 127, 128, 151\}}w_{u,1,2,12}  +   \gamma_{98}w_{u,1,2,14}  = 0,              
\end{align*}
where
\begin{align*}
a_1 &= \gamma_{\{31, 79, 108, 125, 133, 145\}},\ \
a_2 = \gamma_{\{33, 77, 110, 112, 117, 124, 135, 137, 145, 149\}},\\
a_3 &= \gamma_{\{36, 80, 111, 113, 118, 126, 136, 138, 146, 150\}},\ \
a_4 = \gamma_{\{66, 140, 147, 148, 152, 153, 154\}}.
\end{align*}
Computing from these equalities, we obtain
\begin{equation}\begin{cases}
a_i = 0, \ i = 1,2,3,4,\\
\gamma_j = 0, \ j = 7, \ldots , 12, 23, 40, 41, 42,\\ 
\hskip2.2cm 53, 59, 65, 73, 89, 95, 96, 97, 98,  \\
\gamma_{\{32, 78, 109, 134\}} =   
\gamma_{\{35, 148\}} =    
\gamma_{\{92, 152\}} =   
\gamma_{\{46, 99, 115, 141\}} =  0,\\ 
\gamma_{\{22, 132\}} =   
\gamma_{\{52, 144\}} =  
\gamma_{\{20, 68, 130\}} =   
\gamma_{\{34, 136, 139, 147, 151\}} = 0,\\  
\gamma_{\{50, 100, 120, 142\}} =   
\gamma_{\{76, 127, 128, 151\}} =               
\gamma_{\{60, 138\}} =   
\gamma_{\{64, 146\}} = 0\\  
\gamma_{\{75, 150, 154\}} =    
\gamma_{\{21, 69, 131\}} =   
\gamma_{\{24, 70, 132\}} =   
\gamma_{\{61, 114, 139\}} = 0,\\  
\gamma_{\{51, 101, 121, 143\}} =   
\gamma_{\{54, 102, 122, 144\}} = 0.  
\end{cases}\tag{\ref{8.7.2}.3}
\end{equation}

With the aid of (\ref{8.7.2}.2) and (\ref{8.7.2}.3), the homomorphisms $g_1, g_2$ send (\ref{8.7.2}.1) to
\begin{align*}
&\gamma_{19}w_{u,1,2,1}  +  \gamma_{32}w_{u,1,2,2}  +   \gamma_{67}w_{u,1,2,3}  +  \gamma_{\{28, 83\}}w_{u,1,2,4}  +  \gamma_{78}w_{u,1,2,5}\\
&\quad  +  \gamma_{86}w_{u,1,2,6}  +  \gamma_{46}w_{u,1,2,7}  +  \gamma_{99}w_{u,1,2,8}  +  \gamma_{105}w_{u,1,2,9}  +  \gamma_{106}w_{u,1,2,10}\\
&\quad  +  \gamma_{109}w_{u,1,2,11}  +  \gamma_{115}w_{u,1,2,12}  +  \gamma_{\{90, 119\}}w_{u,1,2,13}  +  \gamma_{123}w_{u,1,2,14} =0,\\  
&\gamma_{20}w_{u,1,2,1}  + \gamma_{\{25, 31, 133\}}w_{u,1,2,2}  +  \gamma_{68}w_{u,1,2,3}  +  \gamma_{\{27, 34, 82, 135, 136, 147\}}w_{u,1,2,4}\\
&\quad  +  \gamma_{79}w_{u,1,2,5} +  \gamma_{\{85, 139\}}w_{u,1,2,6}  +  \gamma_{50}w_{u,1,2,7}  +  \gamma_{100}w_{u,1,2,8}\\
&\quad  +  \gamma_{\{104, 151\}}w_{u,1,2,9}  +  \gamma_{\{107, 125\}}w_{u,1,2,10}  +  \gamma_{\{108, 125\}}w_{u,1,2,11}\\
&\quad  +  \gamma_{120}w_{u,1,2,12}  +  \gamma_{\{74, 75, 116, 149, 150, 154\}}w_{u,1,2,13}  +  \gamma_{125}w_{u,1,2,14} =0.  
\end{align*}

These equalities imply
\begin{equation}\begin{cases}
\gamma_j = 0,\ j = 19, 20, 32, 46, 50, 67, 68, 78, 79, 86,\\ 
99, 100, 105, 106, 107, 108, 109, 115, 120, 123, 125,\\ 
\gamma_{\{28, 83\}} =   
\gamma_{\{90, 119\}} =   
\gamma_{\{25, 31, 133\}} =   
\gamma_{\{85, 139\}} = 0,\\  
\gamma_{\{27, 34, 82, 135, 136, 147\}} =  
\gamma_{\{104, 151\}} = \  
\gamma_{\{74, 75, 116, 149, 150, 154\}} = 0.    
\end{cases}\tag{\ref{8.7.2}.4}
\end{equation}

With the aid of (\ref{8.7.2}.2), (\ref{8.7.2}.3) and (\ref{8.7.2}.4), the homomorphisms $g_3, g_4$ send (\ref{8.7.2}.1) to
\begin{align*}
&\gamma_{21}w_{u,1,2,1}  + \gamma_{\{27, 33, 66, 135, 147\}}w_{u,1,2,2}  +  \gamma_{69}w_{u,1,2,3} +  a_5w_{u,1,2,4}  +  a_6w_{u,1,2,6}\\
&\  + \gamma_{\{77, 140\}}w_{u,1,2,5}  +   \gamma_{51}w_{u,1,2,7} +  \gamma_{101}w_{u,1,2,8}  +  a_7w_{u,1,2,9}  +  a_8w_{u,1,2,10}\\
&\   +  a_9w_{u,1,2,11}  +  \gamma_{121}w_{u,1,2,12}  +  a_{10}w_{u,1,2,13}  +  \gamma_{\{124, 127, 153, 154\}}w_{u,1,2,14} =0,\\  
&\gamma_{24}w_{u,1,2,1}  + a_{11}w_{u,1,2,2}  +  \gamma_{70}w_{u,1,2,3}  +  a_{12}w_{u,1,2,4}  +  \gamma_{\{80, 114, 127, 140\}}w_{u,1,2,5}\\
&\  +  a_{13}w_{u,1,2,6}  +  \gamma_{54}w_{u,1,2,7}  +  \gamma_{102}w_{u,1,2,8}  +  a_{14}w_{u,1,2,9}  +  a_{15}w_{u,1,2,10}\\
&\  +  a_{16}w_{u,1,2,11}  +  \gamma_{122}w_{u,1,2,12}  +  a_{17}w_{u,1,2,13}  +  \gamma_{\{126, 127, 128, 153, 154\}}w_{u,1,2,14} = 0, 
\end{align*}
where
\begin{align*}
a_5 &= \gamma_{\{22, 24, 25, 36, 52, 54, 64, 70, 81, 111, 122, 133, 136\}},\ \
a_6 =  \gamma_{\{22, 24, 60, 70, 80, 84, 102, 113\}},\\
a_7 &= \gamma_{\{52, 54, 102, 103, 118, 122, 126, 150\}},\ \
a_8 = \gamma_{\{28, 61, 112, 114, 124, 139, 140, 153, 154\}},\\
a_9 &= \gamma_{\{28, 66, 74, 90, 110, 112, 117, 124, 140, 147, 149\}},
a_{10} = \gamma_{\{74, 76, 90, 117, 128, 149, 151, 153, 154\}},\\
a_{12} &= \gamma_{\{13, 17, 21, 31, 33, 48, 51, 62, 69, 81, 82, 85, 104, 110, 121, 133, 135\}},\\
a_{11} &= \gamma_{\{34, 36, 61, 66, 76, 136, 139, 147, 151\}},\ \
a_{13} = \gamma_{\{14, 17, 21, 26, 58, 69, 77, 83, 84, 85, 101, 112\}},\\
a_{14} &= \gamma_{\{43, 47, 48, 51, 101, 103, 104, 116, 117, 119, 121, 124, 149\}},\\
a_{15} &= \gamma_{\{28, 61, 113, 126, 127, 128, 139, 140, 153, 154\}},\ \
a_{17} = \gamma_{\{75, 76, 90, 118, 150, 151, 153, 154\}},\\
a_{16} &= \gamma_{\{28, 61, 66, 75, 76, 90, 111, 113, 114, 118, 126, 127, 139, 140, 147, 150, 151\}}.
\end{align*}
From the above equalities, we obtain
\begin{equation}\begin{cases}
a_i = 0, \ i = 5,6, \ldots, 17,\\ \gamma_j = 0,\ j = 21, 24, 51, 54, 69, 70, 101, 102, 121, 122, \\
\gamma_{\{27, 33, 66, 135, 147\}} =   
\gamma_{\{77, 140\}} =    
\gamma_{\{124, 127, 153, 154\}} = 0,\\  
\gamma_{\{80, 114, 127, 140\}} =   
\gamma_{\{126, 127, 128, 153, 154\}} = 0.\\   
\end{cases}\tag{\ref{8.7.2}.5}
\end{equation}
With the aid of (\ref{8.7.2}.2), (\ref{8.7.2}.3), (\ref{8.7.2}.4) and (\ref{8.7.2}.5), the homomorphism $h$ sends (\ref{8.7.2}.1) to
\begin{align*}
&a_{18}w_{u,1,2,1}  + \gamma_{\{28, 80, 84, 113\}}w_{u,1,2,2}  +  a_{19}w_{u,1,2,3}  +  \gamma_{80}w_{u,1,2,4}\\
&\quad  +  \gamma_{61}w_{u,1,2,5}  +  \gamma_{114}w_{u,1,2,6}  +  \gamma_{\{75, 90, 103, 118, 126, 150\}}w_{u,1,2,7}\\
&\quad  +  \gamma_{76}w_{u,1,2,8}  + \gamma_{127}w_{u,1,2,9}  +  a_{20}w_{u,1,2,10}  +   a_{21}w_{u,1,2,11}\\
&\quad +  a_{22}w_{u,1,2,12}   +  \gamma_{\{124, 126, 127, 153, 154\}}w_{u,1,2,13}  +  \gamma_{128}w_{u,1,2,14} =0,
\end{align*}
where
\begin{align*}
a_{18} &= \gamma_{\{25, 31, 34, 36, 81, 111, 133, 136\}},\ \
a_{19} = \gamma_{\{27, 33, 34, 36, 66, 82, 135, 136, 147\}},\\
a_{20} &= \gamma_{\{61, 66, 76, 85, 104, 110, 111, 114, 128, 147\}},\ \
a_{21} = \gamma_{\{28, 61, 77, 85, 112, 113, 114, 127\}},\\
a_{22} &= \gamma_{\{74, 75, 76, 90, 104, 116, 117,118, 128, 149, 150, 153, 154\}}.
\end{align*}
From the above equalities, it implies
\begin{equation}\begin{cases}
a_i = 0, \ i = 18, 19, \ldots, 22,\\ 
\gamma_j = 0,\ j = 61, 76, 80, 114, 127, 128,\\
\gamma_{\{75, 90, 103, 118, 126, 150\}} =    
\gamma_{\{28, 84, 113\}} =  
\gamma_{\{124, 126, 153, 154\}} = 0.  
\end{cases}\tag{\ref{8.7.2}.6}
\end{equation}

Combining (\ref{8.7.2}.2), (\ref{8.7.2}.3), (\ref{8.7.2}.4), (\ref{8.7.2}.5) and (\ref{8.7.2}.6), we get $\gamma_j = 0$ for $1 \leqslant j \leqslant 154$. The proposition is proved.
\end{proof}

\begin{props}\label{8.7.3} For $s \geqslant 3$ and $u \geqslant 3$, the elements $[a_{u,1,s,j}], 1 \leqslant j \leqslant 154,$  are linearly independent in $(\mathbb F_2\underset {\mathcal A}\otimes R_4)_{2^{s+u+1}+2^{s+1}+2^s-3}$.
\end{props}

\begin{proof} Suppose that there is a linear relation
\begin{equation}\sum_{j=1}^{154}\gamma_j[a_{u,1,s,j}] = 0, \tag {\ref{8.7.3}.1}
\end{equation}
with $\gamma_j \in \mathbb F_2$.

Applying the homomorphisms $f_1, f_2, \ldots, f_6$ to the relation (\ref{8.7.3}.1), we obtain
\begin{align*}
&\gamma_{1}w_{u,1,s,1} +  \gamma_{2}w_{u,1,s,2} +   \gamma_{15}w_{u,1,s,3} +  \gamma_{16}w_{u,1,s,4} +  \gamma_{29}w_{u,1,s,5}\\
&\quad +  \gamma_{30}w_{u,1,s,6} +  \gamma_{37}w_{u,1,s,7} +  \gamma_{44}w_{u,1,s,8} +  \gamma_{45}w_{u,1,s,9} +  \gamma_{55}w_{u,1,s,10}\\
&\quad +  \gamma_{56}w_{u,1,s,11} +  \gamma_{71}w_{u,1,s,12} +  \gamma_{87}w_{u,1,s,13} +  \gamma_{91}w_{u,1,s,14} =0,
\end{align*}
\begin{align*}
&\gamma_{3}w_{u,1,s,1} +  \gamma_{5}w_{u,1,s,2} +  \gamma_{13}w_{u,1,s,3} +  \gamma_{18}w_{u,1,s,4} +  \gamma_{26}w_{u,1,s,5} +  \gamma_{28}w_{u,1,s,6}\\
&\quad +  \gamma_{38}w_{u,1,s,7} +  \gamma_{43}w_{u,1,s,8} +  \gamma_{49}w_{u,1,s,9} +  \gamma_{\{57, 63, 72\}}w_{u,1,s,10}\\
&\quad +  \gamma_{63}w_{u,1,s,11} +  \gamma_{72}w_{u,1,s,12} +  \gamma_{88}w_{u,1,s,13} +  \gamma_{90}w_{u,1,s,14} =0,\\
&\gamma_{4}w_{u,1,s,1} +  \gamma_{6}w_{u,1,s,2} +  \gamma_{14}w_{u,1,s,3} +  \gamma_{17}w_{u,1,s,4} + \gamma_{ 25}w_{u,1,s,5} +  \gamma_{27}w_{u,1,s,6}\\
&\quad +   \gamma_{39}w_{u,1,s,7} +  \gamma_{47}w_{u,1,s,8} +  \gamma_{48}w_{u,1,s,9} +  \gamma_{\{58, 62, 93\}}w_{u,1,s,10}\\
&\quad +  \gamma_{\{62, 94\}}w_{u,1,s,11} +  \gamma_{93}w_{u,1,s,12} +  \gamma_{94}w_{u,1,s,13} +  \gamma_{74}w_{u,1,s,14} = 0,\\  
&\gamma_{19}w_{u,1,s,1} + \gamma_{32}w_{u,1,s,2} +  \gamma_{7}w_{u,1,s,3} +  \gamma_{35}w_{u,1,s,4} +  \gamma_{10}w_{u,1,s,5}\\
&\quad +  \gamma_{23}w_{u,1,s,6} +  \gamma_{46}w_{u,1,s,7} +  \gamma_{40}w_{u,1,s,8} +  \gamma_{53}w_{u,1,s,9} +  \gamma_{59}w_{u,1,s,10}\\
&\quad +  \gamma_{65}w_{u,1,s,11} +  \gamma_{73}w_{u,1,s,12} +  \gamma_{92}w_{u,1,s,13} +  \gamma_{89}w_{u,1,s,14} = 0,\\  
&\gamma_{20}w_{u,1,s,1} +  \gamma_{31}w_{u,1,s,2} +  \gamma_{8}w_{u,1,s,3} +  \gamma_{34}w_{u,1,s,4} +  \gamma_{11}w_{u,1,s,5}\\
&\quad +  \gamma_{22}w_{u,1,s,6} +  \gamma_{50}w_{u,1,s,7} +  \gamma_{41}w_{u,1,s,8} + \gamma_{52}w_{u,1,s,9} +  \gamma_{\{60, 97\}}w_{u,1,s,10}\\
&\quad +   \gamma_{\{64, 97\}}w_{u,1,s,11} +  \gamma_{95}w_{u,1,s,12} +  \gamma_{75}w_{u,1,s,13} +  \gamma_{97}w_{u,1,s,14} = 0,\\  
&\gamma_{33}w_{u,1,s,1} +  \gamma_{21}w_{u,1,s,2} +  \gamma_{36}w_{u,1,s,3} +  \gamma_{9}w_{u,1,s,4} +  \gamma_{24}w_{u,1,s,5}\\
&\quad +  \gamma_{12}w_{u,1,s,6} +  \gamma_{51}w_{u,1,s,7} +  \gamma_{54}w_{u,1,s,8} +  \gamma_{42}w_{u,1,s,9} +  \gamma_{66}w_{u,1,s,10}\\
&\quad +  \gamma_{61}w_{u,1,s,11} +  \gamma_{76}w_{u,1,s,12} +  \gamma_{96}w_{u,1,s,13} +  \gamma_{98}w_{u,1,s,14} = 0.
\end{align*}

Computing from these equalities, we obtain
\begin{equation}
\gamma_j = 0, j = 1,\ldots, 66, 71,\ldots ,76,  87, \ldots , 98.\tag{\ref{8.7.3}.2}
\end{equation}

With the aid of (\ref{8.7.3}.2), the homomorphisms $g_1, g_2, g_3, g_4$ send (\ref{8.7.3}.1) to
\begin{align*}
&\gamma_{129}w_{u,1,s,1} + \gamma_{134}w_{u,1,s,2} +   \gamma_{67}w_{u,1,s,3} +  \gamma_{83}w_{u,1,s,4} +  \gamma_{78}w_{u,1,s,5}\\
&\quad +  \gamma_{86}w_{u,1,s,6} +  \gamma_{139}w_{u,1,s,7} +  \gamma_{99}w_{u,1,s,8} +  \gamma_{105}w_{u,1,s,9} +  \gamma_{106}w_{u,1,s,10}\\
&\quad +  \gamma_{109}w_{u,1,s,11} +  \gamma_{115}w_{u,1,s,12} +  \gamma_{119}w_{u,1,s,13} +  \gamma_{123}w_{u,1,s,14} = 0,\\  
&\gamma_{130}w_{u,1,s,1} + \gamma_{133}w_{u,1,s,2} +  \gamma_{68}w_{u,1,s,3} +  \gamma_{82}w_{u,1,s,4} +  \gamma_{79}w_{u,1,s,5} +  \gamma_{85}w_{u,1,s,6}\\
&\quad +  \gamma_{140}w_{u,1,s,7} +  \gamma_{100}w_{u,1,s,8} +  \gamma_{104}w_{u,1,s,9} +  \gamma_{\{107, 125\}}w_{u,1,s,10}\\
&\quad +  \gamma_{\{108, 125\}}w_{u,1,s,11} +  \gamma_{120}w_{u,1,s,12} +  \gamma_{116}w_{u,1,s,13} +  \gamma_{125}w_{u,1,s,14} = 0,\\  
&\gamma_{131}w_{u,1,s,1} + \gamma_{135}w_{u,1,s,2} +  \gamma_{69}w_{u,1,s,3} +  a_1w_{u,1,s,4, 77}w_{u,1,s,5} +  a_2w_{u,1,s,6}\\
&\quad +   \gamma_{141}w_{u,1,s,7} +  \gamma_{101}w_{u,1,s,8} +  a_3w_{u,1,s,9}+  \gamma_{\{110, 112, 117, 124\}}w_{u,1,s,11}\\
&\quad  +  \gamma_{\{112, 124\}}w_{u,1,s,10} +  \gamma_{121}w_{u,1,s,12} +  \gamma_{117}w_{u,1,s,13} +  \gamma_{124}w_{u,1,s,14} = 0,\\  
&\gamma_{132}w_{u,1,s,1} + \gamma_{136}w_{u,1,s,2} +  \gamma_{70}w_{u,1,s,3} + a_4w_{u,1,s,4} +  \gamma_{80}w_{u,1,s,5} +  a_5w_{u,1,s,6}\\
&\quad +  \gamma_{142}w_{u,1,s,7} +  \gamma_{102}w_{u,1,s,8} +  a_6w_{u,1,s,9} +  \gamma_{\{111, 113, 118, 126\}}w_{u,1,s,11}\\
&\quad  +  \gamma_{\{113, 126\}}w_{u,1,s,10}+  \gamma_{122}w_{u,1,s,12} +  \gamma_{118}w_{u,1,s,13} +  \gamma_{126}w_{u,1,s,14} = 0,  
\end{align*}
where
\begin{align*}
&a_1 = \begin{cases}  \gamma_{\{81, 152\}}, &s = 3,\\  \gamma_{81}, &s \geqslant 4, \end{cases}\ \
a_4 = \begin{cases}   \gamma_{\{114, 128, 137, 144, 146, 147, 148, 152\}}, &s = 3,\\   \gamma_{137}, &s \geqslant 4, \end{cases}
\end{align*}
\begin{align*}
&a_2 = \begin{cases}  \gamma_{\{84, 153\}}, &s = 3,\\  \gamma_{84}, &s \geqslant 4, \end{cases}\ \
a_5 = \begin{cases}   \gamma_{\{114, 127, 138, 145, 146, 149, 153\}}, &s = 3,\\  \gamma_{138}, &s \geqslant 4, \end{cases}\\
&a_3 = \begin{cases}  \gamma_{\{103, 154\}}, &s = 3,\\  \gamma_{103}, &s \geqslant 4, \end{cases}\ \
a_6 = \begin{cases}  \gamma_{\{127, 128, 143, 147, 150, 151, 154\}}, &s = 3,\\  \gamma_{143}, &s \geqslant 4. \end{cases}
\end{align*}
These equalities imply
\begin{equation}\begin{cases}
a_i = 0, i = 1, 2, \ldots , 6,\\
\gamma_j = 0,\ j = 67, \ldots, 70, 77, 78, 79, 80,\\  82, \ldots , 86,
100,101,102, 104, \ldots , 113,\\ 115, \ldots , 126, 129, \ldots , 136, 139, \ldots, 142.
\end{cases}\tag{\ref{8.7.3}.3}
\end{equation}

With the aid of (\ref{8.7.3}.2) and (\ref{8.7.3}.3), the homomorphism $h$ sends (\ref{8.7.3}.1) to
\begin{align*}
&a_7w_{u,1,s,1} + a_8w_{u,1,s,2} +  \gamma_{144}w_{u,1,s,3} +  \gamma_{145}w_{u,1,s,4} +  \gamma_{146}w_{u,1,s,5}\\
&\quad +  \gamma_{114}w_{u,1,s,6} +  a_9w_{u,1,s,7} +  \gamma_{147}w_{u,1,s,8} +\gamma_{127}w_{u,1,s,9} +  \gamma_{148}w_{u,1,s,10}\\
&\quad +   \gamma_{149}w_{u,1,s,11} +  \gamma_{150}w_{u,1,s,12} +  \gamma_{151}w_{u,1,s,13} +  \gamma_{128}w_{u,1,s,14} = 0, 
\end{align*}
where
$$a_7 = \begin{cases}  \gamma_{137}, &s=3,\\   \gamma_{152}, &s \geqslant 4,\end{cases} \quad
a_8 = \begin{cases}  \gamma_{138}, &s=3,\\   \gamma_{153}, &s \geqslant 4,\end{cases} \quad
a_9 = \begin{cases}  \gamma_{143}, &s=3,\\   \gamma_{154}, &s \geqslant 4.\end{cases}
$$
From the above equalities, it implies
\begin{equation}
a_7 = a_8= a_9 = 0,  \gamma_j = 0,\ j = 114, 127, 128, 144, \ldots , 151.\tag{\ref{8.7.3}.4}
\end{equation}

Combining (\ref{8.7.3}.2), (\ref{8.7.3}.3) and (\ref{8.7.3}.4), we get $\gamma_j = 0$ for $1 \leqslant j \leqslant 154$. The proposition is proved.
\end{proof}

\subsection{The case $s \geqslant 2, t \geqslant 2$ and $u=1$}\label{8.8}\ 

\medskip
According to Kameko \cite{ka}, for $s\geqslant 2, t \geqslant 2$, the dimension of the space $(\mathbb F_2\underset{\mathcal A}\otimes P_3)_{2^{s+t+1}+2^{s+t} + 2^s -3}$  is 14 with a basis given by the classes $w_{1,t,s,j}, 1\leqslant j \leqslant 14,$ which are determined as follows:

\smallskip
\centerline{\begin{tabular}{l}
$1.\  [2^{s} - 1,2^{s+t} - 1,2^{s+t+1} - 1],$\cr 
$2.\  [2^{s} - 1,2^{s+t+1} - 1,2^{s+t} - 1],$\cr 
$3.\  [2^{s+t} - 1,2^{s} - 1,2^{s+t+1} - 1],$\cr 
$4.\  [2^{s+t} - 1,2^{s+t+1} - 1,2^{s} - 1],$\cr 
$5.\  [2^{s+t+1} - 1,2^{s} - 1,2^{s+t} - 1],$\cr 
$6.\  [2^{s+t+1} - 1,2^{s+t} - 1,2^{s} - 1],$\cr 
$7.\  [2^{s+1} - 1,2^{s+t}-2^{s} - 1,2^{s+t+1} - 1],$\cr 
$8.\  [2^{s+1} - 1,2^{s+t+1} - 1,2^{s+t}-2^{s} - 1],$\cr 
$9.\  [2^{s+t+1} - 1,2^{s+1} - 1,2^{s+t}-2^{s} - 1],$\cr 
$10.\  [2^{s+1} - 1,2^{s+t} - 1,2^{s+t+1}-2^{s} - 1],$\cr 
$11.\  [2^{s+1} - 1,2^{s+t+1}-2^{s} - 1,2^{s+t} - 1],$\cr 
$12.\  [2^{s+t} - 1,2^{s+1} - 1,2^{s+t+1}-2^{s} - 1],$\cr 
\end{tabular}}
\centerline{\begin{tabular}{l}
$13.\  [2^{s+2} - 1,2^{s+t+1}-2^{s+1} - 1,2^{s+t}-2^{s} - 1],$\cr 
$14.\  [2^{s+2}-1, 2^{s+2}-1,2^{s+2}+2^{s}-1], \ \text{for } t=2,$\cr 
$14.\  [2^{s+2} - 1,2^{s+t}-2^{s+1} - 1,2^{s+t+1}-2^{s} - 1],\ \text{for } t>2.$\cr
\end{tabular}}

\smallskip
So, we easily obtain

\begin{props}\label{8.8.2} For any positive integer $s, t \geqslant 2$, we have
$$\dim (\mathbb F_2\underset{\mathcal A}\otimes Q_4)_{2^{s+t+1}+ 2^{s+t} + 2^s -3} = 56.$$
\end{props}

Now, we determine $(\mathbb F_2\underset{\mathcal A}\otimes R_4)_{2^{s+t+1}+ 2^{s+t} + 2^s -3}$. We have

\begin{thms}\label{dlc8.8} $(\mathbb F_2\underset{\mathcal A}\otimes R_4)_{2^{s+t+1}+ 2^{s+t} + 2^s -3}$ is  an $\mathbb F_2$-vector space of dimension 154 with a basis consisting of all the classes represented by the monomials $a_{1,t,s,j}, 1 \leqslant j \leqslant 154$, which are determined as follows:

\medskip
For $s \geqslant 2$ and $t \geqslant 2$,

\medskip
\centerline{\begin{tabular}{l}
$1.\ (1,2^{s} - 2,2^{s+t} - 1,2^{s+t+1} - 1),$\cr 
$2.\ (1,2^{s} - 2,2^{s+t+1} - 1,2^{s+t} - 1),$\cr 
$3.\ (1,2^{s+t} - 1,2^{s} - 2,2^{s+t+1} - 1),$\cr 
$4.\ (1,2^{s+t} - 1,2^{s+t+1} - 1,2^{s} - 2),$\cr 
$5.\ (1,2^{s+t+1} - 1,2^{s} - 2,2^{s+t} - 1),$\cr 
$6.\ (1,2^{s+t+1} - 1,2^{s+t} - 1,2^{s} - 2),$\cr 
$7.\ (2^{s+t} - 1,1,2^{s} - 2,2^{s+t+1} - 1),$\cr 
$8.\ (2^{s+t} - 1,1,2^{s+t+1} - 1,2^{s} - 2),$\cr 
$9.\ (2^{s+t} - 1,2^{s+t+1} - 1,1,2^{s} - 2),$\cr 
$10.\ (2^{s+t+1} - 1,1,2^{s} - 2,2^{s+t} - 1),$\cr 
$11.\ (2^{s+t+1} - 1,1,2^{s+t} - 1,2^{s} - 2),$\cr 
$12.\ (2^{s+t+1} - 1,2^{s+t} - 1,1,2^{s} - 2),$\cr 
$13.\ (1,2^{s} - 1,2^{s+t} - 2,2^{s+t+1} - 1),$\cr 
$14.\ (1,2^{s} - 1,2^{s+t+1} - 1,2^{s+t} - 2),$\cr 
$15.\ (1,2^{s+t} - 2,2^{s} - 1,2^{s+t+1} - 1),$\cr 
$16.\ (1,2^{s+t} - 2,2^{s+t+1} - 1,2^{s} - 1),$\cr 
$17.\ (1,2^{s+t+1} - 1,2^{s} - 1,2^{s+t} - 2),$\cr 
$18.\ (1,2^{s+t+1} - 1,2^{s+t} - 2,2^{s} - 1),$\cr 
$19.\ (2^{s} - 1,1,2^{s+t} - 2,2^{s+t+1} - 1),$\cr 
$20.\ (2^{s} - 1,1,2^{s+t+1} - 1,2^{s+t} - 2),$\cr 
$21.\ (2^{s} - 1,2^{s+t+1} - 1,1,2^{s+t} - 2),$\cr 
$22.\ (2^{s+t+1} - 1,1,2^{s} - 1,2^{s+t} - 2),$\cr 
$23.\ (2^{s+t+1} - 1,1,2^{s+t} - 2,2^{s} - 1),$\cr 
$24.\ (2^{s+t+1} - 1,2^{s} - 1,1,2^{s+t} - 2),$\cr 
$25.\ (1,2^{s} - 1,2^{s+t} - 1,2^{s+t+1} - 2),$\cr 
$26.\ (1,2^{s} - 1,2^{s+t+1} - 2,2^{s+t} - 1),$\cr 
$27.\ (1,2^{s+t} - 1,2^{s} - 1,2^{s+t+1} - 2),$\cr 
$28.\ (1,2^{s+t} - 1,2^{s+t+1} - 2,2^{s} - 1),$\cr 
$29.\ (1,2^{s+t+1} - 2,2^{s} - 1,2^{s+t} - 1),$\cr 
$30.\ (1,2^{s+t+1} - 2,2^{s+t} - 1,2^{s} - 1),$\cr 
\end{tabular}}
\centerline{\begin{tabular}{l}
$31.\ (2^{s} - 1,1,2^{s+t} - 1,2^{s+t+1} - 2),$\cr 
$32.\ (2^{s} - 1,1,2^{s+t+1} - 2,2^{s+t} - 1),$\cr 
$33.\ (2^{s} - 1,2^{s+t} - 1,1,2^{s+t+1} - 2),$\cr 
$34.\ (2^{s+t} - 1,1,2^{s} - 1,2^{s+t+1} - 2),$\cr 
$35.\ (2^{s+t} - 1,1,2^{s+t+1} - 2,2^{s} - 1),$\cr 
$36.\ (2^{s+t} - 1,2^{s} - 1,1,2^{s+t+1} - 2),$\cr 
$37.\ (1,2^{s+1} - 2,2^{s+t}-2^{s} - 1,2^{s+t+1} - 1),$\cr 
$38.\ (1,2^{s+1} - 2,2^{s+t+1} - 1,2^{s+t}-2^{s} - 1),$\cr 
$39.\ (1,2^{s+t+1} - 1,2^{s+1} - 2,2^{s+t}-2^{s} - 1),$\cr 
$40.\ (2^{s+t+1} - 1,1,2^{s+1} - 2,2^{s+t}-2^{s} - 1),$\cr 
$41.\ (1,2^{s+1} - 1,2^{s+t}-2^{s} - 2,2^{s+t+1} - 1),$\cr 
$42.\ (1,2^{s+1} - 1,2^{s+t+1} - 1,2^{s+t}-2^{s} - 2),$\cr 
$43.\ (1,2^{s+t+1} - 1,2^{s+1} - 1,2^{s+t}-2^{s} - 2),$\cr 
$44.\ (2^{s+1} - 1,1,2^{s+t}-2^{s} - 2,2^{s+t+1} - 1),$\cr 
$45.\ (2^{s+1} - 1,1,2^{s+t+1} - 1,2^{s+t}-2^{s} - 2),$\cr 
$46.\ (2^{s+1} - 1,2^{s+t+1} - 1,1,2^{s+t}-2^{s} - 2),$\cr 
$47.\ (2^{s+t+1} - 1,1,2^{s+1} - 1,2^{s+t}-2^{s} - 2),$\cr 
$48.\ (2^{s+t+1} - 1,2^{s+1} - 1,1,2^{s+t}-2^{s} - 2),$\cr 
$49.\ (1,2^{s+1} - 1,2^{s+t}-2^{s} - 1,2^{s+t+1} - 2),$\cr 
$50.\ (1,2^{s+1} - 1,2^{s+t+1} - 2,2^{s+t}-2^{s} - 1),$\cr 
$51.\ (1,2^{s+t+1} - 2,2^{s+1} - 1,2^{s+t}-2^{s} - 1),$\cr 
$52.\ (2^{s+1} - 1,1,2^{s+t}-2^{s} - 1,2^{s+t+1} - 2),$\cr 
$53.\ (2^{s+1} - 1,1,2^{s+t+1} - 2,2^{s+t}-2^{s} - 1),$\cr 
$54.\ (2^{s+1} - 1,2^{s+t}-2^{s} - 1,1,2^{s+t+1} - 2),$\cr 
$55.\ (1,2^{s+1} - 2,2^{s+t} - 1,2^{s+t+1}-2^{s} - 1),$\cr 
$56.\ (1,2^{s+1} - 2,2^{s+t+1}-2^{s} - 1,2^{s+t} - 1),$\cr 
$57.\ (1,2^{s+t} - 1,2^{s+1} - 2,2^{s+t+1}-2^{s} - 1),$\cr 
$58.\ (2^{s+t} - 1,1,2^{s+1} - 2,2^{s+t+1}-2^{s} - 1),$\cr 
$59.\ (1,2^{s+1} - 1,2^{s+t} - 2,2^{s+t+1}-2^{s} - 1),$\cr 
$60.\ (1,2^{s+1} - 1,2^{s+t+1}-2^{s} - 1,2^{s+t} - 2),$\cr 
$61.\ (1,2^{s+t} - 2,2^{s+1} - 1,2^{s+t+1}-2^{s} - 1),$\cr 
$62.\ (2^{s+1} - 1,1,2^{s+t} - 2,2^{s+t+1}-2^{s} - 1),$\cr 
$63.\ (2^{s+1} - 1,1,2^{s+t+1}-2^{s} - 1,2^{s+t} - 2),$\cr 
$64.\ (2^{s+1} - 1,2^{s+t+1}-2^{s} - 1,1,2^{s+t} - 2),$\cr 
$65.\ (1,2^{s+1} - 1,2^{s+t} - 1,2^{s+t+1}-2^{s} - 2),$\cr 
$66.\ (1,2^{s+1} - 1,2^{s+t+1}-2^{s} - 2,2^{s+t} - 1),$\cr 
$67.\ (1,2^{s+t} - 1,2^{s+1} - 1,2^{s+t+1}-2^{s} - 2),$\cr 
$68.\ (2^{s+1} - 1,1,2^{s+t} - 1,2^{s+t+1}-2^{s} - 2),$\cr 
$69.\ (2^{s+1} - 1,1,2^{s+t+1}-2^{s} - 2,2^{s+t} - 1),$\cr 
$70.\ (2^{s+1} - 1,2^{s+t} - 1,1,2^{s+t+1}-2^{s} - 2),$\cr 
$71.\ (2^{s+t} - 1,1,2^{s+1} - 1,2^{s+t+1}-2^{s} - 2),$\cr 
$72.\ (2^{s+t} - 1,2^{s+1} - 1,1,2^{s+t+1}-2^{s} - 2),$\cr 
$73.\ (3,2^{s+t} - 3,2^{s} - 2,2^{s+t+1} - 1),$\cr 
$74.\ (3,2^{s+t} - 3,2^{s+t+1} - 1,2^{s} - 2),$\cr 
$75.\ (3,2^{s+t+1} - 1,2^{s+t} - 3,2^{s} - 2),$\cr 
\end{tabular}}
\centerline{\begin{tabular}{l}
$76.\ (2^{s+t+1} - 1,3,2^{s+t} - 3,2^{s} - 2),$\cr 
$77.\ (3,2^{s+t} - 1,2^{s+t+1} - 3,2^{s} - 2),$\cr 
$78.\ (3,2^{s+t+1} - 3,2^{s} - 2,2^{s+t} - 1),$\cr 
$79.\ (3,2^{s+t+1} - 3,2^{s+t} - 1,2^{s} - 2),$\cr 
$80.\ (2^{s+t} - 1,3,2^{s+t+1} - 3,2^{s} - 2),$\cr 
$81.\ (3,2^{s} - 1,2^{s+t} - 3,2^{s+t+1} - 2),$\cr 
$82.\ (3,2^{s+t} - 3,2^{s} - 1,2^{s+t+1} - 2),$\cr 
$83.\ (3,2^{s+t} - 3,2^{s+t+1} - 2,2^{s} - 1),$\cr 
$84.\ (3,2^{s} - 1,2^{s+t+1} - 3,2^{s+t} - 2),$\cr 
$85.\ (3,2^{s+t+1} - 3,2^{s} - 1,2^{s+t} - 2),$\cr 
$86.\ (3,2^{s+t+1} - 3,2^{s+t} - 2,2^{s} - 1),$\cr 
$87.\ (1,2^{s+2} - 2,2^{s+t+1}-2^{s+1} - 1,2^{s+t}-2^{s} - 1),$\cr 
$88.\ (1,2^{s+2} - 1,2^{s+t+1}-2^{s+1} - 1,2^{s+t}-2^{s} - 2),$\cr 
$89.\ (2^{s+2} - 1,1,2^{s+t+1}-2^{s+1} - 1,2^{s+t}-2^{s} - 2),$\cr 
$90.\ (2^{s+2} - 1,2^{s+t+1}-2^{s+1} - 1,1,2^{s+t}-2^{s} - 2),$\cr 
$91.\ (1,2^{s+2} - 1,2^{s+t+1}-2^{s+1} - 2,2^{s+t}-2^{s} - 1),$\cr 
$92.\ (2^{s+2} - 1,1,2^{s+t+1}-2^{s+1} - 2,2^{s+t}-2^{s} - 1),$\cr 
$93.\ (3,2^{s+1} - 3,2^{s+t}-2^{s} - 2,2^{s+t+1} - 1),$\cr 
$94.\ (3,2^{s+1} - 3,2^{s+t+1} - 1,2^{s+t}-2^{s} - 2),$\cr 
$95.\ (3,2^{s+t+1} - 1,2^{s+1} - 3,2^{s+t}-2^{s} - 2),$\cr 
$96.\ (2^{s+t+1} - 1,3,2^{s+1} - 3,2^{s+t}-2^{s} - 2),$\cr 
$97.\ (3,2^{s+1} - 3,2^{s+t}-2^{s} - 1,2^{s+t+1} - 2),$\cr 
$98.\ (3,2^{s+1} - 3,2^{s+t+1} - 2,2^{s+t}-2^{s} - 1),$\cr 
$99.\ (3,2^{s+1} - 3,2^{s+t} - 2,2^{s+t+1}-2^{s} - 1),$\cr 
$100.\ (3,2^{s+1} - 3,2^{s+t+1}-2^{s} - 1,2^{s+t} - 2),$\cr 
$101.\ (3,2^{s+t+1} - 3,2^{s+1} - 2,2^{s+t}-2^{s} - 1),$\cr 
$102.\ (3,2^{s+1} - 3,2^{s+t} - 1,2^{s+t+1}-2^{s} - 2),$\cr 
$103.\ (3,2^{s+1} - 3,2^{s+t+1}-2^{s} - 2,2^{s+t} - 1),$\cr 
$104.\ (3,2^{s+t} - 1,2^{s+1} - 3,2^{s+t+1}-2^{s} - 2),$\cr 
$105.\ (2^{s+t} - 1,3,2^{s+1} - 3,2^{s+t+1}-2^{s} - 2),$\cr 
$106.\ (3,2^{s+1} - 1,2^{s+t}-2^{s} - 3,2^{s+t+1} - 2),$\cr 
$107.\ (2^{s+1} - 1,3,2^{s+t}-2^{s} - 3,2^{s+t+1} - 2),$\cr 
$108.\ (3,2^{s+1} - 1,2^{s+t+1} - 3,2^{s+t}-2^{s} - 2),$\cr 
$109.\ (3,2^{s+t+1} - 3,2^{s+1} - 1,2^{s+t}-2^{s} - 2),$\cr 
$110.\ (2^{s+1} - 1,3,2^{s+t+1} - 3,2^{s+t}-2^{s} - 2),$\cr 
$111.\ (3,2^{s+t} - 3,2^{s+1} - 2,2^{s+t+1}-2^{s} - 1),$\cr 
$112.\ (3,2^{s+1} - 1,2^{s+t} - 3,2^{s+t+1}-2^{s} - 2),$\cr 
$113.\ (3,2^{s+t} - 3,2^{s+1} - 1,2^{s+t+1}-2^{s} - 2),$\cr 
$114.\ (2^{s+1} - 1,3,2^{s+t} - 3,2^{s+t+1}-2^{s} - 2),$\cr 
$115.\ (3,2^{s+1} - 1,2^{s+t+1}-2^{s} - 3,2^{s+t} - 2),$\cr 
$116.\ (2^{s+1} - 1,3,2^{s+t+1}-2^{s} - 3,2^{s+t} - 2),$\cr 
$117.\ (7,2^{s+t} - 5,2^{s+t+1} - 3,2^{s} - 2),$\cr 
$118.\ (7,2^{s+t+1} - 5,2^{s+t} - 3,2^{s} - 2),$\cr 
$119.\ (3,2^{s+2} - 3,2^{s+t+1}-2^{s+1} - 1,2^{s+t}-2^{s} - 2),$\cr 
$120.\ (3,2^{s+2} - 3,2^{s+t+1}-2^{s+1} - 2,2^{s+t}-2^{s} - 1),$\cr 
\end{tabular}}
\centerline{\begin{tabular}{l}
$121.\ (3,2^{s+2} - 1,2^{s+t+1}-2^{s+1} - 3,2^{s+t}-2^{s} - 2),$\cr 
$122.\ (2^{s+2} - 1,3,2^{s+t+1}-2^{s+1} - 3,2^{s+t}-2^{s} - 2),$\cr 
$123.\ (7,2^{s+t+1} - 5,2^{s+1} - 3,2^{s+t}-2^{s} - 2),$\cr 
$124.\ (7,2^{s+t} - 5,2^{s+1} - 3,2^{s+t+1}-2^{s} - 2),$\cr 
$125.\ (7,2^{s+2} - 5,2^{s+t+1}-2^{s+1} - 3,2^{s+t}-2^{s} - 2).$\cr
\end{tabular}}

\medskip
For $s =  2$ and $t\geqslant 2$,

\medskip
\centerline{\begin{tabular}{ll}
$126.\  (3,3,2^{t+2} - 4,2^{t+3} - 1),$\cr 
$127.\  (3,3,2^{t+3} - 1,2^{t+2} - 4),$\cr 
$128.\  (3,2^{t+3} - 1,3,2^{t+2} - 4),$\cr 
$129.\  (2^{t+3}-1,3,3,2^{t+2} - 4),$\cr 
$130.\  (3,3,2^{t+2} - 1,2^{t+3} - 4),$\cr 
$131.\  (3,3,2^{t+3} - 4,2^{t+2} - 1),$\cr 
$132.\  (3,2^{t+2} - 1,3,2^{t+3} - 4),$\cr 
$133.\  (2^{t+2}-1,3,3,2^{t+3} - 4),$\cr 
$134.\  (3,7,2^{t+2} - 5,2^{t+3} - 4),$\cr 
$135.\  (7,3,2^{t+2} - 5,2^{t+3} - 4),$\cr 
$136.\  (7,2^{t+2} - 5,3,2^{t+3} - 4),$\cr 
$137.\  (3,7,2^{t+3} - 5,2^{t+2} - 4),$\cr 
$138.\  (7,3,2^{t+3} - 5,2^{t+2} - 4),$\cr 
$139.\  (7,2^{t+3} - 5,3,2^{t+2} - 4),$\cr 
$140.\  (7,7,2^{t+2} - 8,2^{t+3} - 5),$\cr 
$141.\  (7,7,2^{t+2} - 7,2^{t+3} - 6),$\cr 
$142.\  (7,7,2^{t+3} - 7,2^{t+2} - 6),$\cr 
$143.\  (7,7,2^{t+3} - 8,2^{t+2} - 5).$\cr
\end{tabular}}

\medskip
For $s \geqslant 3$ and $t \geqslant 2$,

\medskip
\centerline{\begin{tabular}{l}
$126.\  (3,2^{s} - 3,2^{s+t} - 2,2^{s+t+1} - 1),$\cr 
$127.\  (3,2^{s} - 3,2^{s+t+1} - 1,2^{s+t} - 2),$\cr 
$128.\  (3,2^{s+t+1} - 1,2^{s} - 3,2^{s+t} - 2),$\cr 
$129.\  (2^{s+t+1} - 1,3,2^{s} - 3,2^{s+t} - 2),$\cr 
$130.\  (3,2^{s} - 3,2^{s+t} - 1,2^{s+t+1} - 2),$\cr 
$131.\  (3,2^{s} - 3,2^{s+t+1} - 2,2^{s+t} - 1),$\cr 
$132.\  (3,2^{s+t} - 1,2^{s} - 3,2^{s+t+1} - 2),$\cr 
$133.\  (2^{s+t} - 1,3,2^{s} - 3,2^{s+t+1} - 2),$\cr 
$134.\  (2^{s} - 1,3,2^{s+t} - 3,2^{s+t+1} - 2),$\cr 
$135.\  (2^{s} - 1,3,2^{s+t+1} - 3,2^{s+t} - 2),$\cr 
$136.\  (7,2^{s+t} - 5,2^{s} - 3,2^{s+t+1} - 2),$\cr 
$137.\  (7,2^{s+t+1} - 5,2^{s} - 3,2^{s+t} - 2),$\cr 
$138.\  (7,2^{s+1} - 5,2^{s+t}-2^{s} - 3,2^{s+t+1} - 2),$\cr 
$139.\  (7,2^{s+1} - 5,2^{s+t+1} - 3,2^{s+t}-2^{s} - 2),$\cr 
$140.\  (7,2^{s+1} - 5,2^{s+t} - 3,2^{s+t+1}-2^{s} - 2),$\cr 
$141.\  (7,2^{s+1} - 5,2^{s+t+1}-2^{s} - 3,2^{s+t} - 2).$\cr
\end{tabular}}

\medskip
For $s = 3$ and $t\geqslant 2$,

\medskip
\centerline{\begin{tabular}{l}
$142.\  (7,7,2^{t+3} - 7,2^{t+4} - 2),$\cr 
$143.\  (7,7,2^{t+4} - 7,2^{t+3} - 2).$\cr 
\end{tabular}}

\medskip
For $s \geqslant 4$ and $t \geqslant 2$,

\medskip
\centerline{\begin{tabular}{l}
$142.\  (7,2^{s} - 5,2^{s+t} - 3,2^{s+t+1} - 2),$\cr 
$143.\  (7,2^{s} - 5,2^{s+t+1} - 3,2^{s+t} - 2).$\cr
\end{tabular}}

\medskip
For $s \geqslant 2$ and $t=2$,

\medskip
\centerline{\begin{tabular}{l}
$144.\  (1,2^{s+2} - 2,2^{s+2} - 1,2^{s+2} + 2^{s}- 1),$\cr 
$145.\  (1,2^{s+2} - 1,2^{s+2} - 2,2^{s+2} + 2^{s}- 1),$\cr 
$146.\  (2^{s+2}-1,1,2^{s+2} - 2,2^{s+2} + 2^{s}- 1),$\cr 
$147.\  (1,2^{s+2} - 1,2^{s+2} - 1,2^{s+2} + 2^{s}- 2),$\cr 
$148.\  (2^{s+2}-1,1,2^{s+2} - 1,2^{s+2} + 2^{s}- 2),$\cr 
$149.\  (2^{s+2}-1,2^{s+2} - 1,1,2^{s+2} + 2^{s}- 2),$\cr 
$150.\  (3,2^{s+2} - 3,2^{s+2} - 1,2^{s+2} + 2^{s}- 2),$\cr 
$151.\  (3,2^{s+2} - 3,2^{s+2} - 2,2^{s+2} + 2^{s}- 1),$\cr 
$152.\  (3,2^{s+2} - 1,2^{s+2} - 3,2^{s+2} + 2^{s}- 2),$\cr 
$153.\  (2^{s+2}-1,3,2^{s+2} - 3,2^{s+2} + 2^{s}- 2),$\cr 
$154.\  (7,2^{s+2} - 5,2^{s+2} - 3,2^{s+2} + 2^{s}- 2).$\cr
\end{tabular}}

\medskip
For $s \geqslant 2$ and $t \geqslant 3$,

\medskip
\centerline{\begin{tabular}{l}
$144.\  (1,2^{s+2} - 2,2^{s+t}-2^{s+1} - 1,2^{s+t+1}-2^{s} - 1),$\cr 
$145.\  (1,2^{s+2} - 1,2^{s+t}-2^{s+1} - 2,2^{s+t+1}-2^{s} - 1),$\cr 
$146.\  (2^{s+2} - 1,1,2^{s+t}-2^{s+1} - 2,2^{s+t+1}-2^{s} - 1),$\cr 
$147.\  (1,2^{s+2} - 1,2^{s+t}-2^{s+1} - 1,2^{s+t+1}-2^{s} - 2),$\cr 
$148.\  (2^{s+2} - 1,1,2^{s+t}-2^{s+1} - 1,2^{s+t+1}-2^{s} - 2),$\cr 
$149.\  (2^{s+2} - 1,2^{s+t}-2^{s+1} - 1,1,2^{s+t+1}-2^{s} - 2),$\cr 
$150.\  (3,2^{s+2} - 3,2^{s+t}-2^{s+1} - 2,2^{s+t+1}-2^{s} - 1),$\cr 
$151.\  (3,2^{s+2} - 3,2^{s+t}-2^{s+1} - 1,2^{s+t+1}-2^{s} - 2),$\cr 
$152.\  (3,2^{s+2} - 1,2^{s+t}-2^{s+1} - 3,2^{s+t+1}-2^{s} - 2),$\cr 
$153.\  (2^{s+2} - 1,3,2^{s+t}-2^{s+1} - 3,2^{s+t+1}-2^{s} - 2),$\cr 
$154.\  (7,2^{s+2} - 5,2^{s+t}-2^{s+1} - 3,2^{s+t+1}-2^{s} - 2).$\cr
\end{tabular}}
\end{thms}

We prove the theorem by proving the following.

\begin{props}\label{mdc8.8} The $\mathbb F_2$-vector space $(\mathbb F_2\underset {\mathcal A}\otimes R_4)_{2^{s+t+1}+ 2^{s+t}+2^s -3}$ is generated by the  elements listed in Theorem \ref{dlc8.8}.
\end{props}

We need the following for the proof of this proposition.

\begin{lems}\label{8.8.1} The following matrices are strictly inadmissible
 $$\begin{pmatrix} 1&1&1&0\\ 1&1&0&1\\ 1&1&0&0\\ 1&1&0&0\\ 0&0&1&0\end{pmatrix} \quad \begin{pmatrix} 1&1&1&0\\ 1&1&1&0\\ 1&1&0&0\\ 0&0&1&1\\ 0&0&1&0\end{pmatrix} \quad \begin{pmatrix} 1&1&1&0\\ 1&1&1&0\\ 0&1&1&0\\ 0&1&1&0\\ 0&0&0&1\end{pmatrix} \quad \begin{pmatrix} 1&1&1&0\\ 1&1&1&0\\ 1&0&1&0\\ 1&0&1&0\\ 0&0&0&1\end{pmatrix} $$    
$$\begin{pmatrix} 1&1&1&0\\ 1&1&1&0\\ 1&1&0&0\\ 1&1&0&0\\ 0&0&0&1\end{pmatrix} \quad \begin{pmatrix} 1&1&0&1\\ 1&1&0&1\\ 1&1&0&0\\ 1&1&0&0\\ 0&0&1&0\end{pmatrix} \quad \begin{pmatrix} 1&1&1&0\\ 1&1&1&0\\ 1&1&0&0\\ 0&0&1&1\\ 0&0&0&1\end{pmatrix} \quad \begin{pmatrix} 1&1&1&0\\ 1&1&0&1\\ 1&0&1&0\\ 1&0&0&1\\ 0&1&0&0\end{pmatrix} $$    
$$\begin{pmatrix} 1&1&1&0\\ 1&1&0&1\\ 1&1&0&0\\ 0&1&0&1\\ 0&0&1&0\end{pmatrix} \quad \begin{pmatrix} 1&1&1&0\\ 1&1&0&1\\ 1&1&0&0\\ 1&0&0&1\\ 0&0&1&0\end{pmatrix} \quad \begin{pmatrix} 1&1&1&0\\ 1&1&1&0\\ 1&0&1&0\\ 0&1&1&0\\ 0&0&0&1\end{pmatrix} \quad \begin{pmatrix} 1&1&1&0\\ 1&1&1&0\\ 1&1&0&0\\ 0&1&1&0\\ 0&0&0&1\end{pmatrix} $$    
$$\begin{pmatrix} 1&1&0&1\\ 1&1&0&1\\ 1&1&0&0\\ 0&1&0&1\\ 0&0&1&0\end{pmatrix} \quad \begin{pmatrix} 1&1&1&0\\ 1&1&1&0\\ 1&1&0&0\\ 1&0&1&0\\ 0&0&0&1\end{pmatrix} \quad \begin{pmatrix} 1&1&0&1\\ 1&1&0&1\\ 1&1&0&0\\ 1&0&0&1\\ 0&0&1&0\end{pmatrix} . $$
\end{lems}

\begin{proof} The monomials corresponding to the above matrices respectively are  
\begin{align*}
&(15,15,17,2), (7,7,27,8),  (3,15,15,16), (15,3,15,16), (15,15,3,16),\\
&(15,15,16,3), (7,7,11,24), (15,19,5,10), (7,15,17,10), (15,7,17,10),\\ 
&(7,11,15,16), (7,15,11,16), (7,15,16,11), (15,7,11,16), (15,7,16,11).     
\end{align*}
 We prove the lemma for the matrices associated with the monomials
\begin{align*}
&(15,15,17,2), (7,7,27,8),  (3,15,15,16), (15,15,16,3), (7,7,11,24),\\ 
&(15,19,5,10), (15,7,17,10), (7,11,15,16), (15,7,11,16), (15,7,16,11).     
\end{align*}
 By a direct computation, we have
\begin{align*}
&(15,15,17,2) = Sq^1(15,15,15,3) + Sq^2(15,15,15,2)\\ 
&\quad+Sq^4\big((15,13,15,2)+ (15,12,15,3) \big)+Sq^8\big((9,15,15,2)\\ 
&\quad+ (11,13,15,2) + (8,15,15,3) + (11,12,15,3) \big) +  (9,23,15,2)\\ 
&\quad +(9,15,23,2) + (11,21,15,2) + (11,13,23,2) + (15,13,19,2)\\ 
&\quad + (8,23,15,3) + (8,15,23,3) + (15,12,19,3) + (15,15,16,3)\\ 
&\quad + (11,20,15,3) + (11,12,23,3)\quad \text{mod  }\mathcal L_4(3;3;2;2;1),\\
&(7,7,27,8) = Sq^1\big((7,7,27,7) + (7,5,29,7)  + (7,7,21,13)  \\ 
&\quad+ (7,7,15,19)  + (7,5,15,21)  \big)+ Sq^2\big((7,6,27,7)  \\ 
&\quad+ (7,7,26,7)  + (7,7,22,11)  + (7,3,30,7)  + (7,7,19,14) \\ 
&\quad + (7,7,15,18)  + (7,6,15,19)  + (7,3,23,14)  \big)\\ 
&\quad+Sq^4\big((4,7,27,7)+ (5,6,27,7)  + (5,7,26,7)  + (11,5,22,7)  
\end{align*}
\begin{align*}
&\quad+ (5,7,22,11)  + (5,3,30,7) + (5,7,19,14)   + (11,5,15,14) \\ 
&\quad + (5,7,15,18)  + (4,7,15,19)  + (5,6,15,19)  +  (5,3,23,14)   \big)\\ 
&\quad+Sq^8\big((7,5,22,7)+ (7,5,15,14) \big) +  (4,11,27,7) \\ 
&\quad+(4,7,27,11) + (5,10,27,7) + (5,6,27,11) + (7,6,27,9)\\ 
&\quad + (7,7,26,9) + (5,11,26,7)  + (7,5,26,11) + (5,11,22,11)\\ 
&\quad   + (7,7,24,11)  + (5,3,30,11)  + (7,3,30,9) + (5,11,19,14)\\ 
&\quad  + (5,11,15,18)  + (4,11,15,19) + (5,10,15,19)  \\ 
&\quad   + (7,3,25,14) + (5,3,27,14)   \quad \text{mod  }\mathcal L_4(3;3;2;2;1),\\
&(3,15,15,16) = Sq^1(3,15,15,15) + Sq^2(2,15,15,15)\\ 
&\quad+Sq^4\big((3,12,15,15)+ (3,15,12,15) \big)+Sq^8(3,11,12,15)\\ 
&\quad+  (2,17,15,15) +(2,15,17,15) + (2,15,15,17)\\ 
&\quad+ (3,12,19,15) + (3,12,15,19)  + (3,15,12,19)\\ 
&\quad + (3,11,20,15) + (3,11,12,23)\ \text{mod  }\mathcal L_4(3;3;2;2;1),\\
&(7,7,11,24) = Sq^1\big((7,7,7,27) + (7,5,7,29) + (7,7,13,21)\\ 
&\quad  + (7,7,19,15)+ (7,5,21,15) \big) + Sq^2\big( (7,7,11,22)  +(7,7,7,26)\\ 
&\quad+ (7,6,7,27) + (7,3,7,30) + (7,7,14,19) + (7,7,18,15)\\ 
&\quad+ (7,6,19,15)+ (7,3,22,15)\big) +Sq^4\big((5,7,11,22) + (11,5,7,22)\\ 
&\quad + (5,7,7,26) + (4,7,7,27) + (5,6,7,27) + (5,3,7,30)\\ 
&\quad + (5,7,14,19) + (11,5,14,15) + (5,7,18,15)+ (4,7,19,15)\\ 
&\quad + (5,6,19,15)+ (5,3,22,15)  \big)+Sq^8\big((7,5,7,22) + (7,5,14,15)\big)\\ 
&\quad + (5,11,11,22)  + (7,5,11,26)  + (7,7,9,26)+ (5,11,7,26)\\ 
&\quad + (7,7,8,27) + (4,11,7,27) + (4,7,11,27) + (7,6,9,27)\\
&\quad  + (5,10,7,27)+ (5,6,11,27) + (7,3,9,30) + (5,3,11,30)\\
&\quad+ (5,11,14,19) + (5,11,18,15)+ (4,11,19,15) + (5,10,19,15) \\ 
&\quad + (7,3,24,15) + (5,3,26,15) \quad \text{mod  }\mathcal L_4(3;3;2;2;1),\\
&(15,15,16,3) = Sq^1\big((15,15,13,5) + (15,15,11,7) + (15,15,9,9)\\ 
&\quad + (15,15,7,11) + (15,15,5,13)\big) + Sq^2\big((15,15,14,3) + (15,15,11,6)\\ 
&\quad + (15,15,10,7) + (15,15,7,10) + (15,15,6,11) + (15,15,3,14)\big)\\
&\quad +Sq^4\big((15,13,14,3) +(15,13,11,6)+(15,12,11,7) + (15,13,10,7)\\ 
&\quad + (15,13,7,10)  + (15,12,7,11) + (15,13,6,11) + (15,13,3,14)\big)\\
&\quad+Sq^8\big((9,15,14,3) +(11,13,14,3) + (9,15,11,6) + (11,13,11,6)\\ 
&\quad + (8,15,11,7) + (11,12,11,7) + (9,15,10,7) + (11,13,10,7) \\ 
&\quad+ (9,15,7,10) + (11,13,7,10) + (8,15,7,11) + (11,12,7,11)
\end{align*}
\begin{align*}
&\quad + (9,15,6,11) + (11,13,6,11) + (9,15,3,14) + (11,13,3,14)\big)\\ 
&\quad+  (9,23,14,3) +(9,15,22,3) + (15,13,18,3) + (11,21,14,3)\\ 
&\quad + (11,13,22,3) + (9,23,11,6) + (9,15,19,6) + (11,21,11,6)\\ 
&\quad + (11,13,19,6) + (8,23,11,7) + (8,15,19,7) + (11,20,11,7)\\ 
&\quad + (11,12,19,7) + (9,23,10,7) + (9,15,18,7) + (11,21,10,7)\\ 
&\quad + (11,13,18,7) + (9,23,7,10) + (9,15,7,18) + (11,21,7,10)\\ 
&\quad + (11,13,7,18) + (8,23,7,11) + (8,15,7,19) + (11,20,7,11)\\ 
&\quad + (11,12,7,19) + (9,23,6,11) + (9,15,6,19) + (11,21,6,11)\\ 
&\quad + (11,13,6,19) + (15,15,3,16) + (9,23,3,14) + (9,15,3,22)\\ 
&\quad + (15,13,3,18) + (11,21,3,14) + (11,13,3,22)\quad \text{mod  }\mathcal L_4(3;3;2;2;1),\\
&(15,19,5,10) = Sq^1(15,15,7,11)  + Sq^2(15,15,7,10) +Sq^4\big((15,15,5,10)\\ 
&\quad+ (15,15,4,11)  \big)+Sq^8\big((11,15,5,10) + (9,15,7,10) \\ 
&\quad+ (8,15,7,11)+ (11,15,4,11)\big)+  (11,23,5,10) +(11,15,5,18)\\ 
&\quad + (9,23,7,10)+ (9,15,7,18) + (15,17,7,10) + (8,23,7,11) \\ 
&\quad+ (8,15,7,19) + (15,16,7,11)+ (11,23,4,11)\\ 
&\quad  + (11,15,4,19)+ (15,19,4,11)\quad \text{mod  }\mathcal L_4(3;3;2;2;1),\\
&(15,7,17,10) = Sq^1(15,7,15,11) + Sq^2(15,7,15,10)+Sq^4\big((15,5,15,10)\\ 
&\quad+ (15,4,15,11)\big) +Sq^8\big((9,7,15,10) + (8,7,15,11)+(11,5,15,10)\\ 
&\quad+(11,4,15,11)\big)+(11,5,23,10) +(11,5,15,18) + (15,5,19,10)\\ 
&\quad + (9,7,23,10) + (9,7,15,18) + (11,4,23,11)\\ 
&\quad + (11,4,15,19) + (15,4,19,11) + (8,7,23,11) \\ 
&\quad+ (8,7,15,19) + (15,7,16,11)\quad \text{mod  }\mathcal L_4(3;3;2;2;1),\\
&(7,11,15,16) = Sq^1\big((7,11,15,15) + (7,7,19,15)+ (7,7,11,23)\big) \\ 
&\quad+ Sq^2(7,10,15,15)+Sq^4\big((4,11,15,15)+ (5,10,15,15)\\ 
&\quad + (11,7,12,15)+ (4,7,19,15)+ (4,7,11,23)\big)+Sq^8(7,7,12,15)\\ 
&\quad+(4,11,19,15) + (4,11,15,19)+ (7,10,17,15) + (7,10,15,17)\\ 
&\quad + (5,10,19,15) + (5,10,15,19) + (7,11,12,19) + (7,8,19,15) \\ 
&\quad+ (4,11,19,15)+ (7,8,11,23)+ (7,7,11,24)\\ 
&\quad+ (4,11,11,23)+ (4,7,11,27)\quad \text{mod  }\mathcal L_4(3;3;2;2;1),\\
&(15,7,11,16) = Sq^1(15,7,11,15) + Sq^2(15,7,10,15)+Sq^4\big((15,4,11,15)\\ 
&\quad+ (15,5,10,15)\big)+Sq^8\big((8,7,11,15) + (11,4,11,15) + (9,7,10,15)\\ 
&\quad + (11,5,10,15)\big)+(8,7,19,15)  + (8,7,11,23)+ (15,4,11,19)\\ 
&\quad + (11,4,19,15) + (11,4,11,23) + (9,7,18,15) 
\end{align*}
\begin{align*}
&\quad+ (9,7,10,23) + (15,5,10,19) + (11,5,18,15)\\ 
&\quad+ (11,5,10,23)+ (15,7,10,17)\quad \text{mod  }\mathcal L_4(3;3;2;2;1),\\
&(15,7,16,11) = Sq^1(15,7,13,13) + Sq^2\big((15,7,14,11) + (15,7,11,14)\big)\\ 
&\quad+Sq^4\big((15,5,14,11)+ (15,5,11,14)\big)+Sq^8\big((9,7,14,11)\\ 
&\quad + (11,5,14,11) + (9,7,11,14)  + (11,5,11,14)\big)+(9,7,22,11)\\ 
&\quad + (9,7,14,19)+ (11,5,22,11) + (11,5,14,19) + (15,5,18,11)\\ 
&\quad + (15,7,11,16) + (9,7,19,14) + (9,7,11,22) + (15,5,11,18)\\ 
&\quad+ (11,5,19,14)+ (11,5,11,22)\quad \text{mod  }\mathcal L_4(3;3;2;2;1).
\end{align*}
The lemma is proved.
\end{proof}

The results in Section \ref{7}, the lemmas in Sections \ref{3}, \ref{5}, \ref{6}, Lemmas \ref{8.3.1}, \ref{8.3.2}, \ref{8.8.1} and Theorem \ref{2.4} imply Proposition \ref{mdc8.8} 

\medskip
Next, we prove that the classes listed in Theorem \ref{dlc8.8} are linearly independent.

\begin{props}\label{8.8.3}  The elements $[a_{1,2,2,j}], 1 \leqslant j \leqslant 154,$  are linearly independent in $(\mathbb F_2\underset {\mathcal A}\otimes R_4)_{49}$.
\end{props}

\begin{proof} Suppose that there is a linear relation
\begin{equation}\sum_{j=1}^{154}\gamma_j[a_{1,2,2,j}] = 0, \tag {\ref{8.8.3}.1}
\end{equation}
with $\gamma_j \in \mathbb F_2$.

Apply the homomorphisms $f_1, f_2, \ldots, f_6$ to the relation (\ref{8.8.3}.1) and we obtain
\begin{align*}
&\gamma_{1}[3,15,31]  +  \gamma_{2}[3,31,15]  +    \gamma_{15}[15,3,31]  +   \gamma_{16}[15,31,3]\\
&\quad  +   \gamma_{29}[31,3,15]  +   \gamma_{30}[31,15,3]  +   \gamma_{37}[7,11,31]  +   \gamma_{38}[7,31,11]\\
&\quad  +   \gamma_{51}[31,7,11]  +   \gamma_{55}[7,15,27]  +   \gamma_{56}[7,27,15]\\
&\quad  +   \gamma_{61}[15,7,27]  +   \gamma_{87}[15,23,11]  +   \gamma_{144}[15,15,19]  = 0,\\   
&\gamma_{3}[3,15,31]  +  \gamma_{5}[3,31,15]  +   \gamma_{\{13, 126\}}[15,3,31]  +   \gamma_{18}[15,31,3]\\
&\quad  +   \gamma_{\{26, 131\}}[31,3,15]  +   \gamma_{28}[31,15,3]  +   \gamma_{41}[7,11,31]  +   \gamma_{39}[7,31,11]\\
&\quad  +   \gamma_{\{50, 143\}}[31,7,11]  +   \gamma_{\{57, 66\}}[7,15,27]  +   \gamma_{66}[7,27,15]\\
&\quad  +   \gamma_{\{59, 66,140\}}[15,7,27]  +   \gamma_{91}[15,23,11]  +   \gamma_{\{66,91,145\}}[15,15,19]  = 0,\\   
&\gamma_{4}[3,15,31]  +   \gamma_{6}[3,31,15]  +  \gamma_{\{14, 127\}}[15,3,31]  +   \gamma_{\{17, 128\}}[15,31,3]\\
&\quad  +    \gamma_{\{25, 130\}}[31,3,15]  +   \gamma_{\{27, 132\}}[31,15,3]  +   \gamma_{42}[7,11,31] +   \gamma_{43}[7,31,11]\\
&\quad   +   \gamma_{\{49,134\}}[31,7,11]  +   \gamma_{\{65,88,147\}}[7,15,27]  +   \gamma_{\{65, 147\}}[7,27,15]\\
&\quad  +   \gamma_{\{60,65,137\}}[15,7,27]  +   \gamma_{67}[15,23,11]  +   \gamma_{\{65, 67\}}[15,15,19]  = 0,
\end{align*}
\begin{align*}
&\gamma_{\{19,73,93,126\}}[3,15,31]  +   \gamma_{\{32,78,103,131\}}[3,31,15]  +   \gamma_{7}[15,3,31]\\
&\quad  +   \gamma_{35}[15,31,3]  +   \gamma_{10}[31,3,15]  +   \gamma_{23}[31,15,3]  +   \gamma_{44}[7,11,31]\\
&\quad  +   \gamma_{\{53, 143\}}[7,31,11]  +   \gamma_{40}[31,7,11]  +   \gamma_{\{62, 140\}}[7,15,27]\\
&\quad  +   \gamma_{69}[7,27,15]  +   \gamma_{58}[15,7,27]  +   \gamma_{92}[15,23,11]  +   \gamma_{146}[15,15,19]  = 0,\\
&\gamma_{\{20,74,94,127\}}[3,15,31]  +  \gamma_{\{31, 79, 102, 130, 150\}}[3,31,15]  +   \gamma_{8}[15,3,31]\\
&\quad  +   \gamma_{\{34, 133\}}[15,31,3]  +  \gamma_{11}[31,3,15]  +   \gamma_{\{22, 129\}}[31,15,3]  +    \gamma_{45}[7,11,31]\\
&\quad  +   \gamma_{\{52, 135\}}[7,31,11]  +   \gamma_{\{63, 138, 148\}}[7,15,27]  +   \gamma_{\{68, 148\}}[7,27,15]\\
&\quad  +   \gamma_{47}[31,7,11] +   \gamma_{89}[15,7,27]  +   \gamma_{71}[15,23,11]  +   \gamma_{148}[15,15,19]  = 0,\\   
&\gamma_{\{33, 77, 104, 121, 132, 152\}}[3,15,31]  +   \gamma_{\{21, 75, 95, 128\}}[3,31,15]\\
&\quad  +   \gamma_{\{36, 80, 105, 122, 133, 153\}}[15,3,31]  +   \gamma_{9}[15,31,3]  +   \gamma_{\{24, 76, 96, 129\}}[31,3,15]\\
&\quad  +   \gamma_{12}[31,15,3]  +   \gamma_{\{54, 117, 124, 125, 136, 154\}}[7,11,31]  +   \gamma_{46}[7,31,11]\\
&\quad  +   \gamma_{48}[31,7,11]  +   \gamma_{70}[7,15,27]  +   \gamma_{\{64, 118, 123, 139\}}[7,27,15]\\
&\quad  +   \gamma_{72}[15,7,27]  +   \gamma_{90}[15,23,11]  +   \gamma_{149}[15,15,19]  =0.  
\end{align*}
Computing from these equalities, we obtain
\begin{equation}\begin{cases}
\gamma_j = 0, \ j = 1, \ldots, 12, 15, 16, 18, 23, 28, 29, \\
\hskip2cm 30, 35, 37, \ldots , 48, 51, 55, 56, 57, 58, 61,\\ 
\hskip2cm65, \ldots , 72, 87, \ldots , 92, 144, \ldots , 149,\\
\gamma_{\{13, 126\}} =   
\gamma_{\{26, 131\}} =    
\gamma_{\{50, 143\}} =   
\gamma_{\{59, 140\}} =   
\gamma_{\{14, 127\}} = 0,\\  
\gamma_{\{17, 128\}} =  
\gamma_{\{25, 130\}} =   
\gamma_{\{27, 132\}} =   
\gamma_{\{60, 137\}} =  
\gamma_{\{53, 143\}} = 0,\\  
\gamma_{\{49, 134\}} =   
\gamma_{\{19, 73, 93, 126\}} =   
\gamma_{\{32, 78, 103, 131\}} =   
\gamma_{\{62, 140\}} = 0,\\  
\gamma_{\{20, 74, 94, 127\}} = 
\gamma_{\{31, 79, 102, 130, 150\}} =   
\gamma_{\{34, 133\}} =   
\gamma_{\{22, 129\}} = 0,\\  
\gamma_{\{52, 135\}} =   
\gamma_{\{63, 138\}} =  
\gamma_{\{21, 75, 95, 128\}} =   
\gamma_{\{24, 76, 96, 129\}} = 0,\\  
\gamma_{\{33, 77, 104, 121, 132, 152\}} =   
\gamma_{\{36, 80, 105, 122, 133, 153\}} = 0,\\  
\gamma_{\{54, 117, 124, 125, 136, 154\}} =   
\gamma_{\{64, 118, 123, 139\}} = 0.     
\end{cases}\tag{\ref{8.8.3}.2}
\end{equation}

With the aid of (\ref{8.8.3}.2), the homomorphisms $g_1, g_2$ send (\ref{8.8.3}.1) to
\begin{align*}
&\gamma_{19}[3,15,31] +  \gamma_{32}[3,31,15] +   \gamma_{73}[15,3,31] +  \gamma_{83}[15,31,3]\\
&\quad +  \gamma_{78}[31,3,15] +  \gamma_{86}[31,15,3] +  \gamma_{93}[7,11,31] +  \gamma_{\{50, 98\}}[7,31,11]\\
&\quad +  \gamma_{101}[31,7,11] +  \gamma_{\{59, 99\}}[7,15,27] +  \gamma_{103}[7,27,15]\\
&\quad +  \gamma_{111}[15,7,27] +  \gamma_{120}[15,23,11] +  \gamma_{151}[15,15,19] = 0,\\  
&\gamma_{20}[3,15,31] +  \gamma_{31}[3,31,15] +  \gamma_{74}[15,3,31] +  \gamma_{\{82, 136\}}[15,31,3]\\
&\quad +  \gamma_{79}[31,3,15] +  \gamma_{\{85, 139\}}[31,15,3] +  \gamma_{94}[7,11,31] +  \gamma_{109}[31,7,11]\\
&\quad +  \gamma_{97}[7,31,11] +  \gamma_{\{100, 150\}}[7,15,27] +  \gamma_{119}[15,7,27]\\
&\quad +  \gamma_{\{102, 150\}}[7,27,15] +  \gamma_{113}[15,23,11] +  \gamma_{150}[15,15,19] =0.  
\end{align*}

These equalities imply
\begin{equation}\begin{cases}
\gamma_j = 0,\ j = 19, 20, 31, 32, 73, 74, 78, 79, 83, 86, 93, 94, 97,\\
\hskip1cm 100, 101, 102, 103, 109, 111, 113, 119, 120, 150, 151,\\
\gamma_{\{50, 98\}} =  
\gamma_{\{59, 99\}} =  
\gamma_{\{82, 136\}} =  
\gamma_{\{85, 139\}} =  0. 
\end{cases}\tag{\ref{8.8.3}.3}
\end{equation}

With the aid of (\ref{8.8.3}.2) and (\ref{8.8.3}.3), the homomorphisms $g_3, g_4$ send (\ref{8.8.3}.1) to
\begin{align*}
&\gamma_{21}[3,15,31] + \gamma_{33}[3,31,15] + \gamma_{75}[15,3,31] +  a_1[15,31,3]\\
&\quad +   \gamma_{77}[31,3,15] +  a_2[31,15,3] +  \gamma_{95}[7,11,31] +  a_3[7,31,11]\\
&\quad +  \gamma_{\{108, 117, 142\}}[31,7,11] +  \gamma_{\{64, 121, 123, 139, 152\}}[7,15,27]\\
&\quad +  \gamma_{\{104, 121, 152\}}[7,27,15] +  \gamma_{\{118, 121\}}[15,7,27]\\
&\quad +  a_4[15,23,11] +  \gamma_{\{50, 115, 121, 125, 142, 152\}}[15,15,19] = 0,\\  
&\gamma_{24}[3,15,31] + \gamma_{36}[3,31,15] +  \gamma_{76}[15,3,31] +  a_5[15,31,3]\\
&\quad +  \gamma_{80}[31,3,15] +  a_6[31,15,3] +  \gamma_{96}[7,11,31] +  a_7[7,31,11]\\
&\quad +  a_8[31,7,11] +  \gamma_{\{122, 153\}}[7,15,27] +  \gamma_{\{105, 122, 153\}}[7,27,15]\\
&\quad +  \gamma_{122}[15,7,27] +  a_9[15,23,11] +  a_{10}[15,15,19] = 0,     
\end{align*}
where
\begin{align*}
a_1 &= \gamma_{\{22, 24, 34, 36, 52, 76, 81, 96, 105, 107, 114, 153\}},\
a_2 = \gamma_{\{22, 24, 63, 76, 80, 84, 96, 110, 116, 122\}},\\
a_3 &= \gamma_{\{54, 59, 106, 124, 136, 141\}},\
a_4 = \gamma_{\{50, 112, 115, 125, 141, 142, 154\}},\\
a_5 &= \gamma_{\{13, 17, 21, 25, 27, 33, 49, 75, 81, 82, 85, 95, 99, 104, 106, 112, 152\}},\\
a_6 &= \gamma_{\{14, 17, 21, 26, 60, 75, 77, 84, 85, 95, 98, 108, 115, 121\}},\\
a_7 &= \gamma_{\{54, 59, 64, 107, 136, 139, 141\}},\
a_8 = \gamma_{\{110, 117, 118, 123, 125, 142\}},\\
a_9 &= \gamma_{\{50, 64, 114, 116, 118, 123, 124, 139, 141, 142\}},\
a_{10} = \gamma_{\{50, 64, 116, 122, 139, 142, 153\}}.
\end{align*}
From the above equalities, we obtain
\begin{equation}\begin{cases}
a_i = 0, \ i = 1,2, \ldots, 10,\ \ 
\gamma_j = 0,\ j = 21, 24, 33, 36,\\ 
\hskip2.5cm75, 76, 77, 80, 95, 96, 105, 122, 153,\\
\gamma_{\{108, 117, 142\}} = 
\gamma_{\{64, 121, 123, 139, 152\}} =    
\gamma_{\{104, 121, 152\}} = 0,\\  
\gamma_{\{118, 121\}} =   
\gamma_{\{50, 115,121, 125, 142, 152\}} = 0.   
\end{cases}\tag{\ref{8.8.3}.4}
\end{equation}
With the aid of (\ref{8.8.3}.2), (\ref{8.8.3}.3) and (\ref{8.8.3}.4), the homomorphism $h$ sends (\ref{8.8.3}.1) to
\begin{align*}
&\gamma_{\{59, 81, 107, 114\}}[3,15,31] + \gamma_{\{50, 84, 110, 116\}}[3,31,15] +   \gamma_{54}[15,3,31]\\
&\quad +  \gamma_{117}[15,31,3] +  \gamma_{64}[31,3,15] +  \gamma_{118}[31,15,3] +  a_{11}[7,11,31]\\
&\quad +  \gamma_{\{108, 110, 117, 142\}}[7,31,11] +  \gamma_{123}[31,7,11] +  a_{12}[7,15,27]\\
&\quad +  a_{13}[7,27,15] +  \gamma_{\{64, 85, 104, 123, 124\}}[15,7,27]\\
&\quad +  \gamma_{125}[15,23,11] +  \gamma_{\{64, 85, 118, 123, 152, 154\}}[15,15,19] = 0,    
\end{align*}
where
\begin{align*}
a_{11} &= \gamma_{\{54, 59, 82, 106, 107, 124, 141, 155\}},\\
a_{12} &= \gamma_{\{64, 85, 112, 114, 123, 141, 154, 155\}},\ \
a_{13} = \gamma_{\{50, 64, 85, 115, 116, 123, 125, 142\}}.
\end{align*}
From the above equalities, it implies
\begin{equation}\begin{cases}
a_{11}= a_{12} = a_{13}=0,\  \gamma_j = 0,\ j = 54, 64, 117, 118, 123, 125,\\
\gamma_{\{59, 81, 107, 114\}} =   
\gamma_{\{50, 84, 110, 116\}} = 0,\\   
\gamma_{\{108, 110, 142\}} =   
\gamma_{\{85, 104, 124\}} =  
\gamma_{\{85, 152, 154\}} = 0,\\  
\end{cases}\tag{\ref{8.8.3}.5}
\end{equation}

Combining (\ref{8.8.3}.2), (\ref{8.8.3}.3), (\ref{8.8.3}.4) and (\ref{8.8.3}.5), we get 
$\gamma_j = 0$ for any $j$.
The proposition is proved.
\end{proof}

\begin{props}\label{8.8.4} For $t \geqslant 3$, the elements $[a_{1,t,2,j}], 1 \leqslant j \leqslant 154,$  are linearly independent in $(\mathbb F_2\underset {\mathcal A}\otimes R_4)_{2^{t+3}+2^{t+2}+1}$.
\end{props}

\begin{proof} Suppose that there is a linear relation
\begin{equation}\sum_{j=1}^{154}\gamma_j[a_{1,t,2,j}] = 0, \tag {\ref{8.8.4}.1}
\end{equation}
with $\gamma_j \in \mathbb F_2$.

Applying the homomorphisms $f_1, f_2, \ldots, f_6$ to the relation (\ref{8.8.4}.1), we obtain
\begin{align*}
&\gamma_{1}w_{1,t,2,1} +  \gamma_{2}w_{1,t,2,2} +   \gamma_{15}w_{1,t,2,3} +  \gamma_{16}w_{1,t,2,4} +  \gamma_{29}w_{1,t,2,5}\\
&\quad +  \gamma_{30}w_{1,t,2,6} +  \gamma_{37}w_{1,t,2,7} +  \gamma_{38}w_{1,t,2,8} +  \gamma_{51}w_{1,t,2,9} +  \gamma_{55}w_{1,t,2,10}\\
&\quad +  \gamma_{56}w_{1,t,2,11} +  \gamma_{61}w_{1,t,2,12} +  \gamma_{87}w_{1,t,2,13} +  \gamma_{144}w_{1,t,2,14} = 0,\\  
&\gamma_{3}w_{1,t,2,1} + \gamma_{5}w_{1,t,2,2} +  \gamma_{\{13, 126\}}w_{1,t,2,3} +  \gamma_{18}w_{1,t,2,4} +  \gamma_{\{26, 131\}}w_{1,t,2,5}\\
&\quad +  \gamma_{28}w_{1,t,2,6} +  \gamma_{41}w_{1,t,2,7} +  \gamma_{39}w_{1,t,2,8} +  \gamma_{\{50, 143\}}w_{1,t,2,9} +  \gamma_{57}w_{1,t,2,10}\\
&\quad +  \gamma_{66}w_{1,t,2,11} +  \gamma_{\{59, 140\}                                                                                                                                                                            }w_{1,t,2,12} +  \gamma_{91}w_{1,t,2,13} +  \gamma_{\{66, 91, 145\}}w_{1,t,2,14} = 0,\\  
&\gamma_{4}w_{1,t,2,1} + \gamma_{6}w_{1,t,2,2} + \gamma_{\{14, 127\}}w_{1,t,2,3} +  \gamma_{\{17, 128\}}w_{1,t,2,4}\\
&\quad +   \gamma_{\{25, 130\}}w_{1,t,2,5} +  \gamma_{\{27, 132\}}w_{1,t,2,6} +  \gamma_{42}w_{1,t,2,7} +  \gamma_{43}w_{1,t,2,8}\\
&\quad +  \gamma_{\{49, 134\}}w_{1,t,2,9} +  \gamma_{67}w_{1,t,2,10} +  \gamma_{65}w_{1,t,2,11} +  \gamma_{\{60, 137\}}w_{1,t,2,12}\\
&\quad +  \gamma_{\{67, 147\}}w_{1,t,2,13} +  \gamma_{\{65, 88, 147\}}w_{1,t,2,14} = 0,\\  
&\gamma_{\{19, 73, 93, 126\}}w_{1,t,2,1} +  \gamma_{\{32, 78, 103, 131\}}w_{1,t,2,2} +  \gamma_{7}w_{1,t,2,3}\\
&\quad +  \gamma_{35}w_{1,t,2,4} +  \gamma_{10}w_{1,t,2,5} +  \gamma_{23}w_{1,t,2,6} +  \gamma_{44}w_{1,t,2,7}\\
&\quad +  \gamma_{\{53, 143\}}w_{1,t,2,8} +  \gamma_{40}w_{1,t,2,9} +  \gamma_{\{62, 140\}}w_{1,t,2,10} +  \gamma_{69}w_{1,t,2,11}\\
&\quad +  \gamma_{58}w_{1,t,2,12} +  \gamma_{92}w_{1,t,2,13} +  \gamma_{146}w_{1,t,2,14} = 0,\\  
&\gamma_{\{20, 74, 94, 127\}}w_{1,t,2,1} + \gamma_{\{31, 79, 102, 130\}}w_{1,t,2,2} +  \gamma_{8}w_{1,t,2,3}\\
&\quad +  \gamma_{\{34, 133\}}w_{1,t,2,4} + \gamma_{11}w_{1,t,2,5} +  \gamma_{\{22, 129\}}w_{1,t,2,6} +   \gamma_{45}w_{1,t,2,7}\\
&\quad +  \gamma_{\{52, 135\}}w_{1,t,2,8} +  \gamma_{47}w_{1,t,2,9} +  \gamma_{\{63, 138\}}w_{1,t,2,10} +  \gamma_{68}w_{1,t,2,11}\\
&\quad +  \gamma_{71}w_{1,t,2,12} +  \gamma_{\{71, 148\}}w_{1,t,2,13} +  \gamma_{\{71, 89\}}w_{1,t,2,14} = 0,
\end{align*}
\begin{align*}
&\gamma_{\{33, 77, 104, 132\}}w_{1,t,2,1} + \gamma_{\{21, 75, 95, 128\}}w_{1,t,2,2} +  \gamma_{\{36, 80, 105, 133\}}w_{1,t,2,3}\\
&\quad +  \gamma_{9}w_{1,t,2,4} +  \gamma_{\{24, 76, 96, 129\}}w_{1,t,2,5} +  \gamma_{12}w_{1,t,2,6} +  \gamma_{\{54, 117, 124, 136\}}w_{1,t,2,7}\\
&\quad +  \gamma_{46}w_{1,t,2,8} +  \gamma_{48}w_{1,t,2,9} +  \gamma_{70}w_{1,t,2,10} +  \gamma_{\{64, 118, 123, 139\}}w_{1,t,2,11}\\
&\quad +  \gamma_{72}w_{1,t,2,12} +  \gamma_{90}w_{1,t,2,13} +  \gamma_{149}w_{1,t,2,14} = 0.  
\end{align*}
Computing from these equalities gives
\begin{equation}\begin{cases}
\gamma_j = 0, \ j = 1, \ldots, 12, 15, 16, 18, 23, 28, 29, \\
\hskip 2cm 30, 35, 37, \ldots , 48, 51, 55, 56, 57, 58, 61, \\
\hskip2cm 65, \ldots , 72, 87, \ldots, 92, 144, \ldots, 149,\\ 
\gamma_{\{13, 126\}} =   
\gamma_{\{26, 131\}} =    
\gamma_{\{50, 143\}} =   
\gamma_{\{59, 140\}} =   
\gamma_{\{14, 127\}} =  0,\\  
\gamma_{\{17, 128\}} = 
\gamma_{\{25, 130\}} = 
\gamma_{\{27, 132\}} = 
\gamma_{\{49, 134\}} = \  
\gamma_{\{60, 137\}} = 0,\\  
\gamma_{\{19, 73, 93, 126\}} =  
\gamma_{\{32, 78, 103, 131\}} =  
\gamma_{\{53, 143\}} =   
\gamma_{\{62, 140\}} = 0,\\  
\gamma_{\{20, 74, 94, 127\}} =   
\gamma_{\{31, 79, 102, 130\}} =   
\gamma_{\{34, 133\}} =   
\gamma_{\{22, 129\}} = 0,\\  
\gamma_{\{52, 135\}} =  
\gamma_{\{63, 138\}} =   
\gamma_{\{33, 77, 104, 132\}} =   
\gamma_{\{21, 75, 95, 128\}} = 0,\\  
\gamma_{\{36, 80, 105, 133\}} =   
\gamma_{\{24, 76, 96, 129\}} = 0,\\  
\gamma_{\{54, 117, 124, 136\}} = 
\gamma_{\{64, 118, 123, 139\}} = 0.  
\end{cases}\tag{\ref{8.8.4}.2}
\end{equation}

With the aid of (\ref{8.8.4}.2), the homomorphisms $g_1, g_2$ send (\ref{8.8.4}.1) to
\begin{align*}
&\gamma_{19}w_{1,t,2,1} +  \gamma_{32}w_{1,t,2,2} +   \gamma_{73}w_{1,t,2,3} +  \gamma_{83}w_{1,t,2,4} +  \gamma_{78}w_{1,t,2,5}\\
&\quad +  \gamma_{86}w_{1,t,2,6} +  \gamma_{93}w_{1,t,2,7} +  \gamma_{\{50, 98\}}w_{1,t,2,8} +  \gamma_{101}w_{1,t,2,9} +  \gamma_{\{59, 99\}}w_{1,t,2,10}\\
&\quad +  \gamma_{103}w_{1,t,2,11} +  \gamma_{111}w_{1,t,2,12} +  \gamma_{120}w_{1,t,2,13} +  \gamma_{150}w_{1,t,2,14} = 0,\\  
&\gamma_{20}w_{1,t,2,1} +  \gamma_{31}w_{1,t,2,2} +  \gamma_{74}w_{1,t,2,3} +  \gamma_{\{82, 136\}}w_{1,t,2,4} +  \gamma_{79}w_{1,t,2,5}\\
&\quad +  \gamma_{\{85, 139\}}w_{1,t,2,6} +  \gamma_{94}w_{1,t,2,7} +  \gamma_{97}w_{1,t,2,8} +  \gamma_{109}w_{1,t,2,9} +  \gamma_{100}w_{1,t,2,10}\\
&\quad +  \gamma_{102}w_{1,t,2,11} +  \gamma_{113}w_{1,t,2,12} +  \gamma_{\{113, 151\}}w_{1,t,2,13} +  \gamma_{\{113, 119\}}w_{1,t,2,14} = 0.  
\end{align*}

These equalities imply
\begin{equation}\begin{cases}
\gamma_j = 0,\ j = 
19, 20, 31, 32, 73, 74, 78, 79, 83, 86, 93, 94,\\
\hskip1cm 97, 100, 101, 102, 103, 109, 111, 113, 119, 120, 150, 151,\\
\gamma_{\{50, 98\}} =   \gamma_{\{59, 99\}} =    \gamma_{\{82, 136\}} =   \gamma_{\{85, 139\}} = 0.
\end{cases}\tag{\ref{8.8.4}.3}
\end{equation}

With the aid of (\ref{8.8.4}.2) and (\ref{8.8.4}.3), the homomorphisms $g_3, g_4$ send (\ref{8.8.4}.1) to
\begin{align*}
&\gamma_{21}w_{1,t,2,1} +  \gamma_{33}w_{1,t,2,2} + \gamma_{75}w_{1,t,2,3} + a_1w_{1,t,2,4} +   \gamma_{77}w_{1,t,2,5} +  a_2w_{1,t,2,6}\\
&\quad +  \gamma_{95}w_{1,t,2,7} +  a_3w_{1,t,2,8} +  \gamma_{\{108, 117, 142\}}w_{1,t,2,9} +  a_4w_{1,t,2,10} +  \gamma_{104}w_{1,t,2,11}\\
&\quad +  \gamma_{\{112, 118, 141\}}w_{1,t,2,12} +  a_5w_{1,t,2,13} +  \gamma_{\{112, 121, 141\}}w_{1,t,2,14} = 0,\\  
&\gamma_{24}w_{1,t,2,1} + \gamma_{36}w_{1,t,2,2} +  \gamma_{76}w_{1,t,2,3} +  a_6w_{1,t,2,4} +  \gamma_{80}w_{1,t,2,5} +  a_7w_{1,t,2,6}\\
&\quad +  \gamma_{96}w_{1,t,2,7} +  a_8w_{1,t,2,8} +  a_9w_{1,t,2,9} +  \gamma_{\{50, 64, 116, 139, 142\}}w_{1,t,2,10}\\
&\quad +  \gamma_{105}w_{1,t,2,11} +  a_{10}w_{1,t,2,12} + a_{11}w_{1,t,2,13} +  a_{12}w_{1,t,2,14} = 0,  
\end{align*}
where
\begin{align*}
a_1 &=  \gamma_{\{22, 24, 34, 36, 52, 76, 81, 96, 105, 107, 114, 153\}},\
a_2 = \gamma_{\{22, 24, 63, 76, 80, 84, 96, 110, 116, 122\}},\\
a_3 &= \gamma_{\{54, 59, 106, 124, 136, 141, 154\}},\ 
a_4 = \gamma_{\{50, 64, 115, 123, 125, 139, 142\}},\\
a_5 &= \gamma_{\{50, 112, 115, 121, 125, 141, 142, 152\}},\
a_8 = \gamma_{\{54, 59, 64, 107, 136, 139, 141\}},\\
a_6 &= \gamma_{\{13, 17, 21, 25, 27, 33, 49, 75, 81, 82, 85, 95, 99, 104, 106, 112, 152\}},\\
a_7 &= \gamma_{\{14, 17, 21, 26, 60, 75, 77, 84, 85, 95, 98, 108, 115, 121\}},\
a_9 = \gamma_{\{110, 117, 118, 123, 125, 142\}}\\
a_{11} &=  \gamma_{\{50, 64, 114, 116, 118, 122, 123, 124, 139, 141, 142, 153, 154
\}},\\
a_{10} &= \gamma_{\{114, 118, 123, 124, 141, 154\}},\
a_{12} = \gamma_{\{114, 118, 122, 123, 124, 141, 154\}}.
\end{align*}
From the above equalities, we obtain
\begin{equation}\begin{cases}
a_i = 0, \ i = 1,2, \ldots, 12,\ \ 
\gamma_j = 0,\ j = 21, 24, 33, 36,\\ 
\hskip 3cm 75, 76, 77, 80, 95, 96, 104, 105,\\
 \gamma_{\{108, 117, 142\}} =   
\gamma_{\{112, 118, 141\}} =  0,\\
\gamma_{\{112, 121, 141\}} =   
\gamma_{\{50, 64, 116, 139, 142\}} = 0.       
\end{cases}\tag{\ref{8.8.4}.4}
\end{equation}
With the aid of (\ref{8.8.4}.2), (\ref{8.8.4}.3) and (\ref{8.8.4}.4), the homomorphism $h$ sends (\ref{8.8.4}.1) to
\begin{align*}
&\gamma_{\{59, 81, 107, 114, 153\}}w_{1,t,2,1} +  \gamma_{\{50, 84, 110, 116, 122\}}w_{1,t,2,2} +  \gamma_{54}w_{1,t,2,3} +  \gamma_{117}w_{1,t,2,4}\\
&\quad  + \gamma_{64}w_{1,t,2,5} +  \gamma_{118}w_{1,t,2,6} +   a_{13}w_{1,t,2,7} +  \gamma_{\{108, 110, 117, 122, 142\}}w_{1,t,2,8}\\
&\quad +  \gamma_{123}w_{1,t,2,9} +  \gamma_{\{112, 114, 118, 141, 153\}}w_{1,t,2,10} + a_{14}w_{1,t,2,11} +  \gamma_{124}w_{1,t,2,12}\\
&\quad +  \gamma_{\{118, 121, 122, 125\}}w_{1,t,2,13} +  \gamma_{\{64, 85, 123, 152, 153, 154\}}w_{1,t,2,14} = 0,   
\end{align*}
where
$$a_{13} = \gamma_{\{54, 59, 82, 106, 107, 124, 141, 153, 154\}},\
a_{14} =  \gamma_{\{50, 64, 85, 115, 116, 122, 123, 125, 142\}}.
$$
From the above equalities, it implies
\begin{equation}\begin{cases}
a_{13}= a_{14} = 0,\  \gamma_j = 0,\ j = 54, 64, 117, 118, 123, 124,\\
\gamma_{\{59, 81, 107, 114, 153\}} =   
\gamma_{\{50, 84, 110, 116, 122\}} = 
\gamma_{\{108, 110, 122, 142\}} = 0,\\  
\gamma_{\{112, 114, 141, 153\}} =  
\gamma_{\{121, 122, 125\}} =   
\gamma_{\{85, 152, 153, 154\}} = 0.     
\end{cases}\tag{\ref{8.8.4}.5}
\end{equation}

Combining (\ref{8.8.4}.2), (\ref{8.8.4}.3), (\ref{8.8.4}.4) and (\ref{8.8.4}.5), we get $\gamma_j = 0$ for all $1\leqslant j \leqslant 154$. The proposition is proved.
\end{proof}

\begin{props}\label{8.8.5} For $s \geqslant 3$, the elements $[a_{1,2,s,j}], 1 \leqslant j \leqslant 154,$  are linearly independent in $(\mathbb F_2\underset {\mathcal A}\otimes R_4)_{2^{s+3}+2^{s+2}+2^s - 3}$.
\end{props}

\begin{proof} Suppose that there is a linear relation
\begin{equation}\sum_{j=1}^{154}\gamma_j[a_{1,2,s,j}] = 0, \tag {\ref{8.8.5}.1}
\end{equation}
with $\gamma_j \in \mathbb F_2$.

Apply the homomorphisms $f_1, f_2, \ldots, f_6$ to the relation (\ref{8.8.5}.1) and we have
\begin{align*}
&\gamma_{1}w_{1,2,s,1} +  \gamma_{2}w_{1,2,s,2} +   \gamma_{15}w_{1,2,s,3} +  \gamma_{16}w_{1,2,s,4} +  \gamma_{29}w_{1,2,s,5}\\
&\quad +  \gamma_{30}w_{1,2,s,6} +  \gamma_{37}w_{1,2,s,7} +  \gamma_{38}w_{1,2,s,8} +  \gamma_{51}w_{1,2,s,9} +  \gamma_{55}w_{1,2,s,10}\\
&\quad +  \gamma_{56}w_{1,2,s,11} +  \gamma_{61}w_{1,2,s,12} +  \gamma_{87}w_{1,2,s,13} +  \gamma_{144}w_{1,2,s,14} = 0,\\  
&\gamma_{3}w_{1,2,s,1} + \gamma_{5}w_{1,2,s,2} +  \gamma_{13}w_{1,2,s,3} +  \gamma_{18}w_{1,2,s,4} +  \gamma_{26}w_{1,2,s,5}\\
&\quad +  \gamma_{28}w_{1,2,s,6} +  \gamma_{41}w_{1,2,s,7} +  \gamma_{39}w_{1,2,s,8} +  \gamma_{50}w_{1,2,s,9} +  \gamma_{\{57, 66\}}w_{1,2,s,10}\\
&\quad +  \gamma_{66}w_{1,2,s,11} +  \gamma_{\{59, 66\}}w_{1,2,s,12} +  \gamma_{91}w_{1,2,s,13} +  \gamma_{\{66, 91, 145\}}w_{1,2,s,14} = 0,\\  
&\gamma_{4}w_{1,2,s,1} +  \gamma_{6}w_{1,2,s,2} + \gamma_{14}w_{1,2,s,3} +  \gamma_{17}w_{1,2,s,4} +   \gamma_{25}w_{1,2,s,5}\\
&\quad +  \gamma_{27}w_{1,2,s,6} +  \gamma_{42}w_{1,2,s,7} +  \gamma_{43}w_{1,2,s,8} +  \gamma_{49}w_{1,2,s,9} +  \gamma_{\{65, 88, 147\}}w_{1,2,s,10}\\
&\quad +  \gamma_{\{65, 147\}}w_{1,2,s,11} +  \gamma_{\{60, 65\}}w_{1,2,s,12} +  \gamma_{67}w_{1,2,s,13} +  \gamma_{\{65, 67\}}w_{1,2,s,14} = 0,\\  
&\gamma_{19}w_{1,2,s,1} + \gamma_{32}w_{1,2,s,2} +  \gamma_{7}w_{1,2,s,3} +  \gamma_{35}w_{1,2,s,4} +  \gamma_{10}w_{1,2,s,5}\\
&\quad +  \gamma_{23}w_{1,2,s,6} +  \gamma_{44}w_{1,2,s,7} +  \gamma_{53}w_{1,2,s,8} +  \gamma_{40}w_{1,2,s,9} +  \gamma_{62}w_{1,2,s,10}\\
&\quad +  \gamma_{69}w_{1,2,s,11} +  \gamma_{58}w_{1,2,s,12} +  \gamma_{92}w_{1,2,s,13} +  \gamma_{146}w_{1,2,s,14} = 0,\\   
&\gamma_{20}w_{1,2,s,1} + \gamma_{31}w_{1,2,s,2} +  \gamma_{8}w_{1,2,s,3} +  \gamma_{34}w_{1,2,s,4} + \gamma_{11}w_{1,2,s,5}\\
&\quad +  \gamma_{22}w_{1,2,s,6} +   \gamma_{45}w_{1,2,s,7} +  \gamma_{52}w_{1,2,s,8} +  \gamma_{47}w_{1,2,s,9} +  \gamma_{\{63, 148\}}w_{1,2,s,10}\\
&\quad +  \gamma_{\{68, 148\}}w_{1,2,s,11} +  \gamma_{89}w_{1,2,s,12} +  \gamma_{71}w_{1,2,s,13} +  \gamma_{148}w_{1,2,s,14} =0,\\  
&\gamma_{33}w_{1,2,s,1} +  \gamma_{21}w_{1,2,s,2} +  \gamma_{36}w_{1,2,s,3} +  \gamma_{9}w_{1,2,s,4} +  \gamma_{24}w_{1,2,s,5}\\
&\quad +  \gamma_{12}w_{1,2,s,6} +  \gamma_{54}w_{1,2,s,7} +  \gamma_{46}w_{1,2,s,8} +  \gamma_{48}w_{1,2,s,9} +  \gamma_{70}w_{1,2,s,10}\\
&\quad +  \gamma_{64}w_{1,2,s,11} +  \gamma_{72}w_{1,2,s,12} +  \gamma_{90}w_{1,2,s,13} +  \gamma_{149}w_{1,2,s,14} = 0.  
\end{align*}
Computing from these equalities gives
\begin{equation}
\gamma_j = 0, \ j = 1, \ldots, 72, 87, \ldots , 92, 144, \ldots , 149.
\tag{\ref{8.8.5}.2}
\end{equation}

With the aid of (\ref{8.8.5}.2), the homomorphisms $g_1, g_2, g_3, g_4$ send (\ref{8.8.5}.1) to
\begin{align*}
&\gamma_{126}w_{1,2,s,1} +  \gamma_{131}w_{1,2,s,2} +   \gamma_{73}w_{1,2,s,3} +  \gamma_{83}w_{1,2,s,4} +  \gamma_{78}w_{1,2,s,5}\\
&\quad +  \gamma_{86}w_{1,2,s,6} +  \gamma_{93}w_{1,2,s,7} +  \gamma_{98}w_{1,2,s,8} +  \gamma_{101}w_{1,2,s,9} +  \gamma_{99}w_{1,2,s,10}\\
&\quad +  \gamma_{103}w_{1,2,s,11} +  \gamma_{111}w_{1,2,s,12} +  \gamma_{120}w_{1,2,s,13} +  \gamma_{151}w_{1,2,s,14} = 0,\\  
&\gamma_{127}w_{1,2,s,1} +  \gamma_{130}w_{1,2,s,2} +  \gamma_{74}w_{1,2,s,3} +  \gamma_{82}w_{1,2,s,4} +  \gamma_{79}w_{1,2,s,5}\\
&\quad +  \gamma_{85}w_{1,2,s,6} +  \gamma_{94}w_{1,2,s,7} +  \gamma_{97}w_{1,2,s,8} +  \gamma_{109}w_{1,2,s,9} +  \gamma_{\{100, 150\}}w_{1,2,s,10}\\
&\quad +  \gamma_{\{102, 150\}}w_{1,2,s,11} +  \gamma_{119}w_{1,2,s,12} +  \gamma_{113}w_{1,2,s,13} +  \gamma_{150}w_{1,2,s,14} = 0,\\  
&\gamma_{128}w_{1,2,s,1} + \gamma_{132}w_{1,2,s,2} + \gamma_{75}w_{1,2,s,3} +  a_1w_{1,2,s,4} +   \gamma_{77}w_{1,2,s,5} +  a_2w_{1,2,s,6}\\
&\quad +  \gamma_{95}w_{1,2,s,7} +  \gamma_{106}w_{1,2,s,8} +  \gamma_{108}w_{1,2,s,9} +  \gamma_{\{121, 152\}}w_{1,2,s,10} +  \gamma_{121}w_{1,2,s,12}\\
&\quad +  \gamma_{\{104, 121, 152\}}w_{1,2,s,11} +  \gamma_{\{112, 115\}}w_{1,2,s,13} +  \gamma_{\{115, 121, 152\}}w_{1,2,s,14} = 0,
\end{align*}
\begin{align*}
&\gamma_{129}w_{1,2,s,1} + \gamma_{133}w_{1,2,s,2} +  \gamma_{76}w_{1,2,s,3} +  a_3w_{1,2,s,4} +  \gamma_{80}w_{1,2,s,5} +  a_6w_{1,2,s,6}\\
&\quad +  \gamma_{96}w_{1,2,s,7} +  \gamma_{107}w_{1,2,s,8} +  \gamma_{110}w_{1,2,s,9} +  \gamma_{\{122, 153\}}w_{1,2,s,10} +  \gamma_{122}w_{1,2,s,12}\\
&\quad +  \gamma_{\{105, 122, 153\}}w_{1,2,s,11} +  \gamma_{\{114, 116\}}w_{1,2,s,13} +  \gamma_{\{116, 122, 153\}}w_{1,2,s,14} = 0,
\end{align*}
where
\begin{align*}
a_1 &= \begin{cases} \gamma_{\{81, 142\}}, &s= 3,\\   \gamma_{81}, & s \geqslant 4, \end{cases}\ \
a_2 = \begin{cases} \gamma_{\{84, 143\}}, &s= 3,\\   \gamma_{84}, & s \geqslant 4, \end{cases}\\
a_3 &= \begin{cases} \gamma_{\{118, 123, 124, 134, 136, 137, 138, 140, 142, 154\}}, &s= 3,\\   \gamma_{134}, & s \geqslant 4, \end{cases}\\
a_4 &= \begin{cases} \gamma_{\{117, 118, 123, 125, 135, 137, 139, 141, 143\}}, &s= 3,\\   \gamma_{135}, & s \geqslant 4. \end{cases}
\end{align*}

These equalities imply
\begin{equation}\begin{cases}
\gamma_j = 0,\ j = 73, \ldots , 80, 82, 83, 85, 86, 93, \ldots , 116,\\ 
\hskip 1cm119, \ldots , 122, 126, \ldots , 133, 150, \ldots, 153,\\
a_1 = a_2 = a_3 = a_4 = 0.
\end{cases}\tag{\ref{8.8.5}.3}
\end{equation}

With the aid of (\ref{8.8.5}.2) and  (\ref{8.8.5}.3), the homomorphism $h$ sends (\ref{8.8.5}.1) to
\begin{align*}
&a_5w_{1,2,s,1} +  a_6w_{1,2,s,2} +  \gamma_{136}w_{1,2,s,3} +  \gamma_{117}w_{1,2,s,4} + \gamma_{137}w_{1,2,s,5}\\
&\quad +  \gamma_{118}w_{1,2,s,6} +   \gamma_{138}w_{1,2,s,7} +  \gamma_{139}w_{1,2,s,8} +  \gamma_{123}w_{1,2,s,9} +  \gamma_{140}w_{1,2,s,10}\\
&\quad +  \gamma_{141}w_{1,2,s,11} +  \gamma_{124}w_{1,2,s,12} +  \gamma_{125}w_{1,2,s,13} +  \gamma_{154}w_{1,2,s,14} = 0,
\end{align*}
where
$$a_5 = \begin{cases}  \gamma_{134}, &s = 3,\\   \gamma_{142}, &s \geqslant 4,\end{cases} \ \
a_6 = \begin{cases}  \gamma_{135}, &s = 3,\\   \gamma_{143}, &s \geqslant 4. \end{cases} $$
From the above equalities, it implies
\begin{equation}\begin{cases}
a_5= a_6 = 0,\ \  \gamma_j = 0,\ j = 117, 118, 123,\\ 124, 125, 136, 137, 138, 139, 140, 141, 154.
\end{cases}\tag{\ref{8.8.5}.4}
\end{equation}

Combining (\ref{8.8.5}.2), (\ref{8.8.5}.3) and (\ref{8.8.5}.4), we get $\gamma_j = 0$ for all $1\leqslant j \leqslant 154$. The proposition is proved.
\end{proof}

\begin{props}\label{8.8.6} For $s\geqslant 3$ and $t \geqslant 3$, the elements $[a_{1,t,s,j}], 1 \leqslant j \leqslant 154,$  are linearly independent in $(\mathbb F_2\underset {\mathcal A}\otimes R_4)_{2^{s+ t+1}+2^{s+ t}+2^s -3}$.
\end{props}

\begin{proof} Suppose that there is a linear relation
\begin{equation}\sum_{j=1}^{154}\gamma_j[a_{1,t,s,j}] = 0, \tag {\ref{8.8.6}.1}
\end{equation}
with $\gamma_j \in \mathbb F_2$.

Applying the homomorphisms $f_1, f_2, \ldots, f_6$ to the relation (\ref{8.8.6}.1), we obtain
\begin{align*}
&\gamma_{1}w_{1,t,s,1} +  \gamma_{2}w_{1,t,s,2} +   \gamma_{15}w_{1,t,s,3} +  \gamma_{16}w_{1,t,s,4} +  \gamma_{29}w_{1,t,s,5}\\
&\quad +  \gamma_{30}w_{1,t,s,6} +  \gamma_{37}w_{1,t,s,7} +  \gamma_{38}w_{1,t,s,8} +  \gamma_{51}w_{1,t,s,9} +  \gamma_{55}w_{1,t,s,10}\\
&\quad +  \gamma_{56}w_{1,t,s,11} +  \gamma_{61}w_{1,t,s,12} +  \gamma_{87}w_{1,t,s,13} +  \gamma_{144}w_{1,t,s,14} = 0,\\  
&\gamma_{3}w_{1,t,s,1} + \gamma_{5}w_{1,t,s,2} +  \gamma_{13}w_{1,t,s,3} +  \gamma_{18}w_{1,t,s,4} +  \gamma_{26}w_{1,t,s,5}\\
&\quad +  \gamma_{28}w_{1,t,s,6} +  \gamma_{41}w_{1,t,s,7} +  \gamma_{39}w_{1,t,s,8} +  \gamma_{50}w_{1,t,s,9} +  \gamma_{57}w_{1,t,s,10}\\
&\quad +  \gamma_{66}w_{1,t,s,11} +  \gamma_{59}w_{1,t,s,12} +  \gamma_{91}w_{1,t,s,13} +  \gamma_{\{66, 91, 145\}}w_{1,t,s,14} = 0,\\  
&\gamma_{4}w_{1,t,s,1} +  \gamma_{6}w_{1,t,s,2} + \gamma_{14}w_{1,t,s,3} +  \gamma_{17}w_{1,t,s,4} +   \gamma_{25}w_{1,t,s,5}\\
&\quad +  \gamma_{27}w_{1,t,s,6} +  \gamma_{42}w_{1,t,s,7} +  \gamma_{43}w_{1,t,s,8} +  \gamma_{49}w_{1,t,s,9} +  \gamma_{67}w_{1,t,s,10}\\
&\quad +  \gamma_{65}w_{1,t,s,11} +  \gamma_{60}w_{1,t,s,12} +  \gamma_{\{67, 147\}}w_{1,t,s,13} +  \gamma_{\{65, 88, 147\}}w_{1,t,s,14} = 0,\\  
&\gamma_{19}w_{1,t,s,1} + \gamma_{32}w_{1,t,s,2} +  \gamma_{7}w_{1,t,s,3} +  \gamma_{35}w_{1,t,s,4} +  \gamma_{10}w_{1,t,s,5}\\
&\quad +  \gamma_{23}w_{1,t,s,6} +  \gamma_{44}w_{1,t,s,7} +  \gamma_{53}w_{1,t,s,8} +  \gamma_{40}w_{1,t,s,9} +  \gamma_{62}w_{1,t,s,10}\\
&\quad +  \gamma_{69}w_{1,t,s,11} +  \gamma_{58}w_{1,t,s,12} +  \gamma_{92}w_{1,t,s,13} +  \gamma_{146}w_{1,t,s,14} = 0,\\  
&\gamma_{20}w_{1,t,s,1} + \gamma_{31}w_{1,t,s,2} +  \gamma_{8}w_{1,t,s,3} +  \gamma_{34}w_{1,t,s,4} + \gamma_{11}w_{1,t,s,5}\\
&\quad +  \gamma_{22}w_{1,t,s,6} +   \gamma_{45}w_{1,t,s,7} +  \gamma_{52}w_{1,t,s,8} +  \gamma_{47}w_{1,t,s,9} +  \gamma_{63}w_{1,t,s,10}\\
&\quad +  \gamma_{68}w_{1,t,s,11} +  \gamma_{71}w_{1,t,s,12} +  \gamma_{\{71, 148\}}w_{1,t,s,13} +  \gamma_{\{71, 89\}}w_{1,t,s,14} = 0,\\  
&\gamma_{33}w_{1,t,s,1} +  \gamma_{21}w_{1,t,s,2} +  \gamma_{36}w_{1,t,s,3} +  \gamma_{9}w_{1,t,s,4} +  \gamma_{24}w_{1,t,s,5}\\
&\quad +  \gamma_{12}w_{1,t,s,6} +  \gamma_{54}w_{1,t,s,7} +  \gamma_{46}w_{1,t,s,8} +  \gamma_{48}w_{1,t,s,9} +  \gamma_{70}w_{1,t,s,10}\\
&\quad +  \gamma_{64}w_{1,t,s,11} +  \gamma_{72}w_{1,t,s,12} +  \gamma_{90}w_{1,t,s,13} +  \gamma_{149}w_{1,t,s,14} = 0.
\end{align*}
Computing from these equalities gives
\begin{equation}
\gamma_j = 0, \ j = 1, \ldots, 72, 87, \ldots , 92, 144, \ldots , 149.
\tag{\ref{8.8.6}.2}
\end{equation}

With the aid of (\ref{8.8.6}.2), the homomorphisms $g_1, g_2, g_3, g_4$ send (\ref{8.8.6}.1) to
\begin{align*}
&\gamma_{126}w_{1,t,s,1} +  \gamma_{131}w_{1,t,s,2} +   \gamma_{73}w_{1,t,s,3} +  \gamma_{83}w_{1,t,s,4} +  \gamma_{78}w_{1,t,s,5}\\
&\quad +  \gamma_{86}w_{1,t,s,6} +  \gamma_{93}w_{1,t,s,7} +  \gamma_{98}w_{1,t,s,8} +  \gamma_{101}w_{1,t,s,9} +  \gamma_{99}w_{1,t,s,10}\\
&\quad +  \gamma_{103}w_{1,t,s,11} +  \gamma_{111}w_{1,t,s,12} +  \gamma_{120}w_{1,t,s,13} +  \gamma_{150}w_{1,t,s,14} = 0,\\  
&\gamma_{127}w_{1,t,s,1} + \gamma_{130}w_{1,t,s,2} +  \gamma_{74}w_{1,t,s,3} +  \gamma_{82}w_{1,t,s,4} +  \gamma_{79}w_{1,t,s,5}\\
&\quad +  \gamma_{85}w_{1,t,s,6} +  \gamma_{94}w_{1,t,s,7} +  \gamma_{97}w_{1,t,s,8} +  \gamma_{109}w_{1,t,s,9} +  \gamma_{100}w_{1,t,s,10}\\
&\quad +  \gamma_{102}w_{1,t,s,11} +  \gamma_{113}w_{1,t,s,12} +  \gamma_{\{113, 151\}}w_{1,t,s,13} +  \gamma_{\{113, 119\}}w_{1,t,s,14} = 0,\\  
&\gamma_{128}w_{1,t,s,1} + \gamma_{132}w_{1,t,s,2} + \gamma_{75}w_{1,t,s,3} + a_1w_{1,t,s,4} +   \gamma_{77}w_{1,t,s,5} +  a_2w_{1,t,s,6}\\
&\quad +  \gamma_{95}w_{1,t,s,7} +  \gamma_{106}w_{1,t,s,8} +  \gamma_{108}w_{1,t,s,9} +  \gamma_{115}w_{1,t,s,10} +  \gamma_{104}w_{1,t,s,11}\\
&\quad +  \gamma_{112}w_{1,t,s,12} +  \gamma_{\{112, 115, 121, 152\}}w_{1,t,s,13} +  \gamma_{\{112, 121\}}w_{1,t,s,14} = 0,
\end{align*}
\begin{align*}
&\gamma_{129}w_{1,t,s,1} + \gamma_{133}w_{1,t,s,2} +  \gamma_{76}w_{1,t,s,3} +  a_3w_{1,t,s,4} +  \gamma_{80}w_{1,t,s,5} +  a_4w_{1,t,s,6}\\
&\quad +  \gamma_{96}w_{1,t,s,7} +  \gamma_{107}w_{1,t,s,8} +  \gamma_{110}w_{1,t,s,9} +  \gamma_{116}w_{1,t,s,10} +  \gamma_{105}w_{1,t,s,11}\\
&\quad +  \gamma_{114}w_{1,t,s,12} +  \gamma_{\{114, 116, 122, 153\}}w_{1,t,s,13} +  \gamma_{\{114, 122\}}w_{1,t,s,14} = 0,  
\end{align*}
where
\begin{align*}
a_1 &= \begin{cases}  \gamma_{\{81, 142\}}, &s = 3,\\      \gamma_{81}, & s\geqslant 4, \end{cases}\ \ 
a_2 = \begin{cases}  \gamma_{\{84, 143\}}, &s = 3,\\     \gamma_{84}, & s\geqslant 4, \end{cases}\\
a_3 &= \begin{cases}  \gamma_{\{118, 123, 124, 134, 136, 137, 138, 140, 142, 154\}}, &s = 3,\\     \gamma_{134}, & s\geqslant 4, \end{cases}\\
a_4 &= \begin{cases}  \gamma_{\{117,118,123,125,135,137,139,141,143\}}, &s = 3,\\     \gamma_{135}, & s\geqslant 4. \end{cases}
\end{align*}

These equalities imply
\begin{equation}\begin{cases}
\gamma_j = 0,\ j = 73, \ldots , 83, 85, 86, 93, \ldots, 116,\\ 
119, \ldots , 122, 126, \ldots , 133, 150, 151, 152, 153,\\
a_1 = a_2 = a_3 = a_4 = 0.
\end{cases}\tag{\ref{8.8.6}.3}
\end{equation}

With the aid of (\ref{8.8.6}.2) and (\ref{8.8.6}.3), the homomorphisms $h$ sends (\ref{8.8.6}.1) to
\begin{align*}
&a_5w_{1,t,s,1} +  a_6w_{1,t,s,2} +  \gamma_{136}w_{1,t,s,3} +  \gamma_{117}w_{1,t,s,4} + \gamma_{137}w_{1,t,s,5}\\
&\quad +  \gamma_{118}w_{1,t,s,6} +   \gamma_{138}w_{1,t,s,7} +  \gamma_{139}w_{1,t,s,8} +  \gamma_{123}w_{1,t,s,9} +  \gamma_{140}w_{1,t,s,10}\\
&\quad +  \gamma_{141}w_{1,t,s,11} +  \gamma_{124}w_{1,t,s,12} +  \gamma_{125}w_{1,t,s,13} +  \gamma_{154}w_{1,t,s,14} = 0, 
\end{align*}
where
$$a_5 = \begin{cases}  \gamma_{134}, &s = 3,\\      \gamma_{142}, & s\geqslant 4, \end{cases}\ \ 
a_6 = \begin{cases}  \gamma_{135}, &s = 3,\\      \gamma_{143}, & s\geqslant 4. \end{cases} 
$$
From the above equalities, we obtain
\begin{equation}\begin{cases}
a_5 = a_6 = 0,\ \ 
\gamma_j = 0,\ j = 117, 118, 123,\\ 124, 125, 136, 137, 138, 139, 140, 141, 154. 
\end{cases}\tag{\ref{8.8.6}.4}
\end{equation}

Combining (\ref{8.8.6}.2), (\ref{8.8.6}.3) and (\ref{8.8.6}.4), we get $\gamma_j = 0$ for all $1\leqslant j \leqslant 154$. The proposition is proved.
\end{proof}

\subsection{The case $s \geqslant 2, t \geqslant 2$ and $u\geqslant 2$}\label{8.9}\ 

\medskip
According to Kameko \cite{ka}, for $s\geqslant 2, t\geqslant 2, u\geqslant 2$, the dimension of the space $(\mathbb F_2\underset{\mathcal A}\otimes P_3)_{2^{s+t+u}+2^{s+t} + 2^s -3}$ is  21 with a basis given by the following classes: 

\smallskip
\centerline{\begin{tabular}{l}
$w_{u,t,s,1} = [2^{s} - 1,2^{s+t} - 1,2^{s+t+u}- 1],$\cr 
$w_{u,t,s,2} = [2^{s} - 1,2^{s+t+u}- 1,2^{s+t} - 1],$\cr 
$w_{u,t,s,3} = [2^{s+t} - 1,2^{s} - 1,2^{s+t+u}- 1],$\cr 
$w_{u,t,s,4} = [2^{s+t} - 1,2^{s+t+u}- 1,2^{s} - 1],$\cr 
\end{tabular}}
\centerline{\begin{tabular}{l}
$w_{u,t,s,5} = [2^{s+t+u} - 1,2^{s} - 1,2^{s+t} - 1],$\cr 
$w_{u,t,s,6} = [2^{s+t+u} - 1,2^{s+t} - 1,2^{s} - 1],$\cr 
$w_{u,t,s,7} = [2^{s} - 1,2^{s+t+1} - 1,2^{s+t+u}- 2^{s+t}- 1],$\cr 
$w_{u,t,s,8} = [2^{s+t+1} - 1,2^{s} - 1,2^{s+t+u}- 2^{s+t}- 1],$\cr 
$w_{u,t,s,9} = [2^{s+t+1} - 1,2^{s+t+u}- 2^{s+t}- 1,2^{s} - 1],$\cr 
$w_{u,t,s,10} = [2^{s+1} - 1,2^{s+t}-2^{s} - 1,2^{s+t+u}- 1],$\cr 
$w_{u,t,s,11} = [2^{s+1} - 1,2^{s+t+u}- 1,2^{s+t}-2^{s} - 1],$\cr 
$w_{u,t,s,12} = [2^{s+t+u} - 1,2^{s+1} - 1,2^{s+t}-2^{s} - 1],$\cr 
$w_{u,t,s,13} = [2^{s+1} - 1,2^{s+t} - 1,2^{s+t+u}- 2^{s}- 1],$\cr 
$w_{u,t,s,14} = [2^{s+1} - 1,2^{s+t+u}- 2^{s}- 1,2^{s+t} - 1],$\cr 
$w_{u,t,s,15} = [2^{s+t} - 1,2^{s+1} - 1,2^{s+t+u}- 2^{s}- 1],$\cr 
$w_{u,t,s,16} = [2^{s+1} - 1,2^{s+t+1}-2^{s} - 1,2^{s+t+u}- 2^{s+t}- 1],$\cr 
$w_{u,t,s,17} = [2^{s+1} - 1,2^{s+t+1} - 1,2^{s+t+u}- 2^{s+t}- 2^{s} - 1],$\cr 
$w_{u,t,s,18} = [2^{s+t+1} - 1,2^{s+1} - 1,2^{s+t+u}- 2^{s+t}- 2^{s} - 1],$\cr 
$w_{u,t,s,19} = [2^{s+2} - 1,2^{s+t+1}-2^{s+1} - 1,2^{s+t+u}- 2^{s+t}- 2^{s} - 1],$\cr 
$w_{u,t,s,20} = [2^{s+2} - 1,2^{s+t+u}- 2^{s+1}- 1,2^{s+t}-2^{s} - 1],$\cr 
$w_{u,2,s,21} = [2^{s+2} - 1,2^{s+2} - 1,2^{s+u+2}- 2^{s+1} - 2^{s}- 1],$ for $t=2$,\cr
$w_{u,t,s,21} = [2^{s+2} - 1,2^{s+t}-2^{s+1} - 1,2^{s+t+u}- 2^{s}- 1,],$ for $t \geqslant 3$.\cr
\end{tabular}}

\medskip
So, we easily obtain

\begin{props}\label{8.9.2} For any positive integer $s, t,u \geqslant 2$, we have
$$\dim (\mathbb F_2\underset{\mathcal A}\otimes Q_4)_{2^{s+t+u}+ 2^{s+t} + 2^s -3} = 84.$$
\end{props}

Next, we determine $(\mathbb F_2\underset{\mathcal A}\otimes R_4)_{2^{s+t+u}+ 2^{s+t} + 2^s -3}$. We have

\begin{thms}\label{dlc8.9} For $s, t, u \geqslant 2$, $(\mathbb F_2\underset{\mathcal A}\otimes R_4)_{2^{s+t+u}+ 2^{s+t} + 2^s -3}$ is  an $\mathbb F_2$-vector space of dimension 231 with a basis consisting of all the classes represented by the monomials $a_{u,t,s,j}, 1\leqslant j \leqslant 231$, which are determined as follows:

\smallskip
For $s, t, u \geqslant 2$

\medskip
\centerline{\begin{tabular}{l}
$1.\ (1,2^{s} - 2,2^{s+t} - 1,2^{s+t+u}- 1),$\cr 
$2.\ (1,2^{s} - 2,2^{s+t+u}- 1,2^{s+t} - 1),$\cr 
$3.\ (1,2^{s+t} - 1,2^{s} - 2,2^{s+t+u}- 1),$\cr 
$4.\ (1,2^{s+t} - 1,2^{s+t+u}- 1,2^{s} - 2),$\cr 
$5.\ (1,2^{s+t+u}- 1,2^{s} - 2,2^{s+t} - 1),$\cr
$6.\ (1,2^{s+t+u}- 1,2^{s+t} - 1,2^{s} - 2),$\cr 
$7.\ (2^{s+t} - 1,1,2^{s} - 2,2^{s+t+u}- 1),$\cr
$8.\ (2^{s+t} - 1,1,2^{s+t+u}- 1,2^{s} - 2),$\cr 
$9.\ (2^{s+t} - 1,2^{s+t+u}- 1,1,2^{s} - 2),$\cr
$10.\ (2^{s+t+u} - 1,1,2^{s} - 2,2^{s+t} - 1),$\cr 
$11.\ (2^{s+t+u} - 1,1,2^{s+t} - 1,2^{s} - 2),$\cr
$12.\ (2^{s+t+u} - 1,2^{s+t} - 1,1,2^{s} - 2),$\cr 
$13.\ (1,2^{s} - 1,2^{s+t} - 2,2^{s+t+u}- 1),$\cr
$14.\ (1,2^{s} - 1,2^{s+t+u}- 1,2^{s+t} - 2),$\cr 
\end{tabular}}
\centerline{\begin{tabular}{l}
$15.\ (1,2^{s+t} - 2,2^{s} - 1,2^{s+t+u}- 1),$\cr
$16.\ (1,2^{s+t} - 2,2^{s+t+u}- 1,2^{s} - 1),$\cr 
$17.\ (1,2^{s+t+u}- 1,2^{s} - 1,2^{s+t} - 2),$\cr
$18.\ (1,2^{s+t+u}- 1,2^{s+t} - 2,2^{s} - 1),$\cr 
$19.\ (2^{s} - 1,1,2^{s+t} - 2,2^{s+t+u}- 1),$\cr
$20.\ (2^{s} - 1,1,2^{s+t+u}- 1,2^{s+t} - 2),$\cr 
$21.\ (2^{s} - 1,2^{s+t+u}- 1,1,2^{s+t} - 2),$\cr
$22.\ (2^{s+t+u} - 1,1,2^{s} - 1,2^{s+t} - 2),$\cr 
$23.\ (2^{s+t+u} - 1,1,2^{s+t} - 2,2^{s} - 1),$\cr
$24.\ (2^{s+t+u} - 1,2^{s} - 1,1,2^{s+t} - 2),$\cr 
$25.\ (1,2^{s} - 1,2^{s+t} - 1,2^{s+t+u}- 2),$\cr
$26.\ (1,2^{s} - 1,2^{s+t+u}- 2,2^{s+t} - 1),$\cr 
$27.\ (1,2^{s+t} - 1,2^{s} - 1,2^{s+t+u}- 2),$\cr
$28.\ (1,2^{s+t} - 1,2^{s+t+u}- 2,2^{s} - 1),$\cr 
$29.\ (1,2^{s+t+u}- 2,2^{s} - 1,2^{s+t} - 1),$\cr
$30.\ (1,2^{s+t+u}- 2,2^{s+t} - 1,2^{s} - 1),$\cr 
$31.\ (2^{s} - 1,1,2^{s+t} - 1,2^{s+t+u}- 2),$\cr
$32.\ (2^{s} - 1,1,2^{s+t+u}- 2,2^{s+t} - 1),$\cr 
$33.\ (2^{s} - 1,2^{s+t} - 1,1,2^{s+t+u}- 2),$\cr
$34.\ (2^{s+t} - 1,1,2^{s} - 1,2^{s+t+u}- 2),$\cr 
$35.\ (2^{s+t} - 1,1,2^{s+t+u}- 2,2^{s} - 1),$\cr
$36.\ (2^{s+t} - 1,2^{s} - 1,1,2^{s+t+u}- 2),$\cr
$37.\ (1,2^{s} - 2,2^{s+t+1} - 1,2^{s+t+u}- 2^{s+t}- 1),$\cr 
$38.\ (1,2^{s+t+1} - 1,2^{s} - 2,2^{s+t+u}- 2^{s+t}- 1),$\cr 
$39.\ (1,2^{s+t+1} - 1,2^{s+t+u}- 2^{s+t}- 1,2^{s} - 2),$\cr 
$40.\ (2^{s+t+1} - 1,1,2^{s} - 2,2^{s+t+u}- 2^{s+t}- 1),$\cr 
$41.\ (2^{s+t+1} - 1,1,2^{s+t+u}- 2^{s+t}- 1,2^{s} - 2),$\cr 
$42.\ (2^{s+t+1} - 1,2^{s+t+u}- 2^{s+t}- 1,1,2^{s} - 2),$\cr 
$43.\ (1,2^{s} - 1,2^{s+t+1} - 2,2^{s+t+u}- 2^{s+t}- 1),$\cr 
$44.\ (1,2^{s+t+1} - 2,2^{s} - 1,2^{s+t+u}- 2^{s+t}- 1),$\cr 
$45.\ (1,2^{s+t+1} - 2,2^{s+t+u}- 2^{s+t}- 1,2^{s} - 1),$\cr 
$46.\ (2^{s} - 1,1,2^{s+t+1} - 2,2^{s+t+u}- 2^{s+t}- 1),$\cr 
$47.\ (1,2^{s} - 1,2^{s+t+1} - 1,2^{s+t+u}- 2^{s+t}- 2),$\cr 
$48.\ (1,2^{s+t+1} - 1,2^{s} - 1,2^{s+t+u}- 2^{s+t}- 2),$\cr 
$49.\ (1,2^{s+t+1} - 1,2^{s+t+u}- 2^{s+t}- 2,2^{s} - 1),$\cr 
$50.\ (2^{s} - 1,1,2^{s+t+1} - 1,2^{s+t+u}- 2^{s+t}- 2),$\cr 
$51.\ (2^{s} - 1,2^{s+t+1} - 1,1,2^{s+t+u}- 2^{s+t}- 2),$\cr 
$52.\ (2^{s+t+1} - 1,1,2^{s} - 1,2^{s+t+u}- 2^{s+t}- 2),$\cr 
$53.\ (2^{s+t+1} - 1,1,2^{s+t+u}- 2^{s+t}- 2,2^{s} - 1),$\cr 
$54.\ (2^{s+t+1} - 1,2^{s} - 1,1,2^{s+t+u}- 2^{s+t}- 2),$\cr 
$55.\ (1,2^{s+1} - 2,2^{s+t}-2^{s} - 1,2^{s+t+u}- 1),$\cr 
$56.\ (1,2^{s+1} - 2,2^{s+t+u}- 1,2^{s+t}-2^{s} - 1),$\cr 
$57.\ (1,2^{s+t+u}- 1,2^{s+1} - 2,2^{s+t}-2^{s} - 1),$\cr 
$58.\ (2^{s+t+u} - 1,1,2^{s+1} - 2,2^{s+t}-2^{s} - 1),$\cr 
$59.\ (1,2^{s+1} - 1,2^{s+t}-2^{s} - 2,2^{s+t+u}- 1),$\cr 
\end{tabular}}
\centerline{\begin{tabular}{l}
$60.\ (1,2^{s+1} - 1,2^{s+t+u}- 1,2^{s+t}-2^{s} - 2),$\cr 
$61.\ (1,2^{s+t+u}- 1,2^{s+1} - 1,2^{s+t}-2^{s} - 2),$\cr 
$62.\ (2^{s+1} - 1,1,2^{s+t}-2^{s} - 2,2^{s+t+u}- 1),$\cr 
$63.\ (2^{s+1} - 1,1,2^{s+t+u}- 1,2^{s+t}-2^{s} - 2),$\cr 
$64.\ (2^{s+1} - 1,2^{s+t+u}- 1,1,2^{s+t}-2^{s} - 2),$\cr 
$65.\ (2^{s+t+u} - 1,1,2^{s+1} - 1,2^{s+t}-2^{s} - 2),$\cr 
$66.\ (2^{s+t+u} - 1,2^{s+1} - 1,1,2^{s+t}-2^{s} - 2),$\cr 
$67.\ (1,2^{s+1} - 1,2^{s+t}-2^{s} - 1,2^{s+t+u}- 2),$\cr 
$68.\ (1,2^{s+1} - 1,2^{s+t+u}- 2,2^{s+t}-2^{s} - 1),$\cr 
$69.\ (1,2^{s+t+u}- 2,2^{s+1} - 1,2^{s+t}-2^{s} - 1),$\cr 
$70.\ (2^{s+1} - 1,1,2^{s+t}-2^{s} - 1,2^{s+t+u}- 2),$\cr 
$71.\ (2^{s+1} - 1,1,2^{s+t+u}- 2,2^{s+t}-2^{s} - 1),$\cr 
$72.\ (2^{s+1} - 1,2^{s+t}-2^{s} - 1,1,2^{s+t+u}- 2),$\cr 
$73.\ (1,2^{s+1} - 2,2^{s+t} - 1,2^{s+t+u}- 2^{s}- 1),$\cr 
$74.\ (1,2^{s+1} - 2,2^{s+t+u}- 2^{s}- 1,2^{s+t} - 1),$\cr 
$75.\ (1,2^{s+t} - 1,2^{s+1} - 2,2^{s+t+u}- 2^{s}- 1),$\cr 
$76.\ (2^{s+t} - 1,1,2^{s+1} - 2,2^{s+t+u}- 2^{s}- 1),$\cr 
$77.\ (1,2^{s+1} - 1,2^{s+t} - 2,2^{s+t+u}- 2^{s}- 1),$\cr 
$78.\ (1,2^{s+1} - 1,2^{s+t+u}- 2^{s}- 1,2^{s+t} - 2),$\cr 
$79.\ (1,2^{s+t} - 2,2^{s+1} - 1,2^{s+t+u}- 2^{s}- 1),$\cr 
$80.\ (2^{s+1} - 1,1,2^{s+t} - 2,2^{s+t+u}- 2^{s}- 1),$\cr 
$81.\ (2^{s+1} - 1,1,2^{s+t+u}- 2^{s}- 1,2^{s+t} - 2),$\cr 
$82.\ (2^{s+1} - 1,2^{s+t+u}- 2^{s}- 1,1,2^{s+t} - 2),$\cr 
$83.\ (1,2^{s+1} - 1,2^{s+t} - 1,2^{s+t+u}- 2^{s}- 2),$\cr 
$84.\ (1,2^{s+1} - 1,2^{s+t+u}- 2^{s}- 2,2^{s+t} - 1),$\cr 
$85.\ (1,2^{s+t} - 1,2^{s+1} - 1,2^{s+t+u}- 2^{s}- 2),$\cr 
$86.\ (2^{s+1} - 1,1,2^{s+t} - 1,2^{s+t+u}- 2^{s}- 2),$\cr 
$87.\ (2^{s+1} - 1,1,2^{s+t+u}- 2^{s}- 2,2^{s+t} - 1),$\cr 
$88.\ (2^{s+1} - 1,2^{s+t} - 1,1,2^{s+t+u}- 2^{s}- 2),$\cr 
$89.\ (2^{s+t} - 1,1,2^{s+1} - 1,2^{s+t+u}- 2^{s}- 2),$\cr 
$90.\ (2^{s+t} - 1,2^{s+1} - 1,1,2^{s+t+u}- 2^{s}- 2),$\cr 
$91.\ (3,2^{s+t} - 3,2^{s} - 2,2^{s+t+u}- 1),$\cr 
$92.\ (3,2^{s+t} - 3,2^{s+t+u}- 1,2^{s} - 2),$\cr 
$93.\ (3,2^{s+t+u}- 1,2^{s+t} - 3,2^{s} - 2),$\cr 
$94.\ (2^{s+t+u} - 1,3,2^{s+t} - 3,2^{s} - 2),$\cr 
$95.\ (3,2^{s+t} - 1,2^{s+t+u}- 3,2^{s} - 2),$\cr 
$96.\ (3,2^{s+t+u}- 3,2^{s} - 2,2^{s+t} - 1),$\cr 
$97.\ (3,2^{s+t+u}- 3,2^{s+t} - 1,2^{s} - 2),$\cr 
$98.\ (2^{s+t} - 1,3,2^{s+t+u}- 3,2^{s} - 2),$\cr
$99.\ (1,2^{s+1} - 2,2^{s+t+1}-2^{s} - 1,2^{s+t+u}- 2^{s+t}- 1),$\cr 
$100.\ (1,2^{s+1} - 2,2^{s+t+1} - 1,2^{s+t+u}- 2^{s+t}- 2^{s} - 1),$\cr 
$101.\ (1,2^{s+t+1} - 1,2^{s+1} - 2,2^{s+t+u}- 2^{s+t}- 2^{s} - 1),$\cr 
$102.\ (2^{s+t+1} - 1,1,2^{s+1} - 2,2^{s+t+u}- 2^{s+t}- 2^{s} - 1),$\cr 
$103.\ (1,2^{s+1} - 1,2^{s+t+1}-2^{s} - 2,2^{s+t+u}- 2^{s+t}- 1),$\cr 
$104.\ (2^{s+1} - 1,1,2^{s+t+1}-2^{s} - 2,2^{s+t+u}- 2^{s+t}- 1),$\cr 
\end{tabular}}
\centerline{\begin{tabular}{l}
$105.\ (1,2^{s+1} - 1,2^{s+t+1}-2^{s} - 1,2^{s+t+u}- 2^{s+t}- 2),$\cr 
$106.\ (2^{s+1} - 1,1,2^{s+t+1}-2^{s} - 1,2^{s+t+u}- 2^{s+t}- 2),$\cr 
$107.\ (2^{s+1} - 1,2^{s+t+1}-2^{s} - 1,1,2^{s+t+u}- 2^{s+t}- 2),$\cr
$108.\ (3,2^{s} - 1,2^{s+t} - 3,2^{s+t+u}- 2),$\cr 
$109.\ (3,2^{s+t} - 3,2^{s} - 1,2^{s+t+u}- 2),$\cr 
$110.\ (3,2^{s+t} - 3,2^{s+t+u}- 2,2^{s} - 1),$\cr 
$111.\ (3,2^{s} - 1,2^{s+t+u}- 3,2^{s+t} - 2),$\cr 
$112.\ (3,2^{s+t+u}- 3,2^{s} - 1,2^{s+t} - 2),$\cr 
$113.\ (3,2^{s+t+u}- 3,2^{s+t} - 2,2^{s} - 1),$\cr 
$114.\ (1,2^{s+1} - 1,2^{s+t+1} - 2,2^{s+t+u}- 2^{s+t}- 2^{s} - 1),$\cr 
$115.\ (1,2^{s+t+1} - 2,2^{s+1} - 1,2^{s+t+u}- 2^{s+t}- 2^{s} - 1),$\cr 
$116.\ (2^{s+1} - 1,1,2^{s+t+1} - 2,2^{s+t+u}- 2^{s+t}- 2^{s} - 1),$\cr 
$117.\ (1,2^{s+1} - 1,2^{s+t+1} - 1,2^{s+t+u}- 2^{s+t}- 2^{s} - 2),$\cr 
$118.\ (1,2^{s+t+1} - 1,2^{s+1} - 1,2^{s+t+u}- 2^{s+t}- 2^{s} - 2),$\cr 
$119.\ (2^{s+1} - 1,1,2^{s+t+1} - 1,2^{s+t+u}- 2^{s+t}- 2^{s} - 2),$\cr 
$120.\ (2^{s+1} - 1,2^{s+t+1} - 1,1,2^{s+t+u}- 2^{s+t}- 2^{s} - 2),$\cr 
$121.\ (2^{s+t+1} - 1,1,2^{s+1} - 1,2^{s+t+u}- 2^{s+t}- 2^{s} - 2),$\cr 
$122.\ (2^{s+t+1} - 1,2^{s+1} - 1,1,2^{s+t+u}- 2^{s+t}- 2^{s} - 2),$\cr 
$123.\ (1,2^{s+2} - 2,2^{s+t+u}- 2^{s+1}- 1,2^{s+t}-2^{s} - 1),$\cr 
$124.\ (1,2^{s+2} - 1,2^{s+t+u}- 2^{s+1}- 1,2^{s+t}-2^{s} - 2),$\cr 
$125.\ (2^{s+2} - 1,1,2^{s+t+u}- 2^{s+1}- 1,2^{s+t}-2^{s} - 2),$\cr 
$126.\ (2^{s+2} - 1,2^{s+t+u}- 2^{s+1}- 1,1,2^{s+t}-2^{s} - 2),$\cr 
$127.\ (1,2^{s+2} - 1,2^{s+t+u}- 2^{s+1}- 2,2^{s+t}-2^{s} - 1),$\cr 
$128.\ (2^{s+2} - 1,1,2^{s+t+u}- 2^{s+1}- 2,2^{s+t}-2^{s} - 1),$\cr 
$129.\ (3,2^{s+t+1} - 3,2^{s} - 2,2^{s+t+u}- 2^{s+t}- 1),$\cr 
$130.\ (3,2^{s+t+1} - 3,2^{s+t+u}- 2^{s+t}- 1,2^{s} - 2),$\cr 
$131.\ (3,2^{s+t+1} - 1,2^{s+t+u}- 2^{s+t}- 3,2^{s} - 2),$\cr 
$132.\ (2^{s+t+1} - 1,3,2^{s+t+u}- 2^{s+t}- 3,2^{s} - 2),$\cr 
$133.\ (3,2^{s} - 1,2^{s+t+1} - 3,2^{s+t+u}- 2^{s+t}- 2),$\cr 
$134.\ (3,2^{s+t+1} - 3,2^{s} - 1,2^{s+t+u}- 2^{s+t}- 2),$\cr 
$135.\ (3,2^{s+t+1} - 3,2^{s+t+u}- 2^{s+t}- 2,2^{s} - 1),$\cr 
$136.\ (3,2^{s+1} - 3,2^{s+t}-2^{s} - 2,2^{s+t+u}- 1),$\cr 
$137.\ (3,2^{s+1} - 3,2^{s+t+u}- 1,2^{s+t}-2^{s} - 2),$\cr 
$138.\ (3,2^{s+t+u}- 1,2^{s+1} - 3,2^{s+t}-2^{s} - 2),$\cr 
$139.\ (2^{s+t+u} - 1,3,2^{s+1} - 3,2^{s+t}-2^{s} - 2),$\cr 
$140.\ (3,2^{s+1} - 3,2^{s+t}-2^{s} - 1,2^{s+t+u}- 2),$\cr 
$141.\ (3,2^{s+1} - 3,2^{s+t+u}- 2,2^{s+t}-2^{s} - 1),$\cr 
$142.\ (3,2^{s+t+u}- 3,2^{s+1} - 2,2^{s+t}-2^{s} - 1),$\cr 
$143.\ (1,2^{s+2} - 2,2^{s+t+1}-2^{s+1} - 1,2^{s+t+u}- 2^{s+t}- 2^{s} - 1),$\cr 
$144.\ (1,2^{s+2} - 1,2^{s+t+1}-2^{s+1} - 2,2^{s+t+u}- 2^{s+t}- 2^{s} - 1),$\cr 
$145.\ (2^{s+2} - 1,1,2^{s+t+1}-2^{s+1} - 2,2^{s+t+u}- 2^{s+t}- 2^{s} - 1),$\cr 
$146.\ (1,2^{s+2} - 1,2^{s+t+1}-2^{s+1} - 1,2^{s+t+u}- 2^{s+t}- 2^{s} - 2),$\cr 
$147.\ (2^{s+2} - 1,1,2^{s+t+1}-2^{s+1} - 1,2^{s+t+u}- 2^{s+t}- 2^{s} - 2),$\cr 
$148.\ (2^{s+2} - 1,2^{s+t+1}-2^{s+1} - 1,1,2^{s+t+u}- 2^{s+t}- 2^{s} - 2),$\cr $149.\ (3,2^{s+1} - 1,2^{s+t}-2^{s} - 3,2^{s+t+u}- 2),$\cr 
\end{tabular}}
\centerline{\begin{tabular}{l}
$150.\ (2^{s+1} - 1,3,2^{s+t}-2^{s} - 3,2^{s+t+u}- 2),$\cr 
$151.\ (3,2^{s+1} - 1,2^{s+t+u}- 3,2^{s+t}-2^{s} - 2),$\cr 
$152.\ (3,2^{s+t+u}- 3,2^{s+1} - 1,2^{s+t}-2^{s} - 2),$\cr 
$153.\ (2^{s+1} - 1,3,2^{s+t+u}- 3,2^{s+t}-2^{s} - 2),$\cr 
$154.\ (3,2^{s+1} - 3,2^{s+t} - 2,2^{s+t+u}- 2^{s}- 1),$\cr 
$155.\ (3,2^{s+1} - 3,2^{s+t+u}- 2^{s}- 1,2^{s+t} - 2),$\cr 
$156.\ (3,2^{s+1} - 3,2^{s+t} - 1,2^{s+t+u}- 2^{s}- 2),$\cr 
$157.\ (3,2^{s+1} - 3,2^{s+t+u}- 2^{s}- 2,2^{s+t} - 1),$\cr 
$158.\ (3,2^{s+t} - 1,2^{s+1} - 3,2^{s+t+u}- 2^{s}- 2),$\cr 
$159.\ (2^{s+t} - 1,3,2^{s+1} - 3,2^{s+t+u}- 2^{s}- 2),$\cr 
$160.\ (3,2^{s+t} - 3,2^{s+1} - 2,2^{s+t+u}- 2^{s}- 1),$\cr 
$161.\ (3,2^{s+1} - 1,2^{s+t} - 3,2^{s+t+u}- 2^{s}- 2),$\cr 
$162.\ (3,2^{s+t} - 3,2^{s+1} - 1,2^{s+t+u}- 2^{s}- 2),$\cr 
$163.\ (2^{s+1} - 1,3,2^{s+t} - 3,2^{s+t+u}- 2^{s}- 2),$\cr 
$164.\ (3,2^{s+1} - 1,2^{s+t+u}- 2^{s}- 3,2^{s+t} - 2),$\cr 
$165.\ (2^{s+1} - 1,3,2^{s+t+u}- 2^{s}- 3,2^{s+t} - 2),$\cr 
$166.\ (7,2^{s+t} - 5,2^{s+t+u}- 3,2^{s} - 2),$\cr 
$167.\ (7,2^{s+t+u}- 5,2^{s+t} - 3,2^{s} - 2),$\cr 
$168.\ (3,2^{s+1} - 3,2^{s+t+1}-2^{s} - 2,2^{s+t+u}- 2^{s+t}- 1),$\cr 
$169.\ (3,2^{s+1} - 3,2^{s+t+1}-2^{s} - 1,2^{s+t+u}- 2^{s+t}- 2),$\cr 
$170.\ (3,2^{s+1} - 3,2^{s+t+1} - 2,2^{s+t+u}- 2^{s+t}- 2^{s} - 1),$\cr 
$171.\ (3,2^{s+1} - 3,2^{s+t+1} - 1,2^{s+t+u}- 2^{s+t}- 2^{s} - 2),$\cr 
$172.\ (3,2^{s+t+1} - 1,2^{s+1} - 3,2^{s+t+u}- 2^{s+t}- 2^{s} - 2),$\cr 
$173.\ (2^{s+t+1} - 1,3,2^{s+1} - 3,2^{s+t+u}- 2^{s+t}- 2^{s} - 2),$\cr 
$174.\ (3,2^{s+t+1} - 3,2^{s+1} - 2,2^{s+t+u}- 2^{s+t}- 2^{s} - 1),$\cr 
$175.\ (3,2^{s+1} - 1,2^{s+t+1}-2^{s} - 3,2^{s+t+u}- 2^{s+t}- 2),$\cr 
$176.\ (2^{s+1} - 1,3,2^{s+t+1}-2^{s} - 3,2^{s+t+u}- 2^{s+t}- 2),$\cr 
$177.\ (3,2^{s+1} - 1,2^{s+t+1} - 3,2^{s+t+u}- 2^{s+t}- 2^{s} - 2),$\cr 
$178.\ (3,2^{s+t+1} - 3,2^{s+1} - 1,2^{s+t+u}- 2^{s+t}- 2^{s} - 2),$\cr 
$179.\ (2^{s+1} - 1,3,2^{s+t+1} - 3,2^{s+t+u}- 2^{s+t}- 2^{s} - 2),$\cr 
$180.\ (7,2^{s+t+1} - 5,2^{s+t+u}- 2^{s+t}- 3,2^{s} - 2),$\cr 
$181.\ (3,2^{s+2} - 3,2^{s+t+u}- 2^{s+1}- 1,2^{s+t}-2^{s} - 2),$\cr 
$182.\ (3,2^{s+2} - 3,2^{s+t+u}- 2^{s+1}- 2,2^{s+t}-2^{s} - 1),$\cr 
$183.\ (3,2^{s+2} - 1,2^{s+t+u}- 2^{s+1}- 3,2^{s+t}-2^{s} - 2),$\cr 
$184.\ (2^{s+2} - 1,3,2^{s+t+u}- 2^{s+1}- 3,2^{s+t}-2^{s} - 2),$\cr 
$185.\ (7,2^{s+t+u}- 5,2^{s+1} - 3,2^{s+t}-2^{s} - 2),$\cr 
$186.\ (7,2^{s+t} - 5,2^{s+1} - 3,2^{s+t+u}- 2^{s}- 2),$\cr 
$187.\ (3,2^{s+2} - 3,2^{s+t+1}-2^{s+1} - 2,2^{s+t+u}- 2^{s+t}- 2^{s} - 1),$\cr 
$188.\ (3,2^{s+2} - 3,2^{s+t+1}-2^{s+1} - 1,2^{s+t+u}- 2^{s+t}- 2^{s} - 2),$\cr 
$189.\ (3,2^{s+2} - 1,2^{s+t+1}-2^{s+1} - 3,2^{s+t+u}- 2^{s+t}- 2^{s} - 2),$\cr 
$190.\ (2^{s+2} - 1,3,2^{s+t+1}-2^{s+1} - 3,2^{s+t+u}- 2^{s+t}- 2^{s} - 2),$\cr 
$191.\ (7,2^{s+t+1} - 5,2^{s+1} - 3,2^{s+t+u}- 2^{s+t}- 2^{s} - 2),$\cr 
$192.\ (7,2^{s+2} - 5,2^{s+t+u}- 2^{s+1}- 3,2^{s+t}-2^{s} - 2),$\cr 
$193.\ (7,2^{s+2} - 5,2^{s+t+1}-2^{s+1} - 3,2^{s+t+u}- 2^{s+t}- 2^{s} - 2).$\cr
\end{tabular}}

\smallskip
For $ s=2, t, u \geqslant 2$,

\medskip
\centerline{\begin{tabular}{ll}
$194.\  (3,3,2^{t+2} - 4,2^{t+u+2} - 1),$&\cr 
$195.\  (3,3,2^{t+u+2} - 1,2^{t+2} - 4),$&\cr 
$196.\  (3,2^{t+u+2} - 1,3,2^{t+2} - 4),$&\cr
$197.\  (2^{t+u+2} - 1,3,3,2^{t+2} - 4),$&\cr 
$198.\  (3,3,2^{t+2} - 1,2^{t+u+2} - 4),$&\cr 
$199.\  (3,3,2^{t+u+2} - 4,2^{t+2} - 1),$&\cr 
$200.\  (3,2^{t+2} - 1,3,2^{t+u+2} - 4),$&\cr 
$201.\  (2^{t+2} - 1,3,3,2^{t+u+2} - 4),$&\cr 
$202.\  (3,7,2^{t+2} - 5,2^{t+u+2} - 4),$&\cr 
$203.\  (7,3,2^{t+2} - 5,2^{t+u+2} - 4),$&\cr 
$204.\  (3,3,2^{t+3} - 4,2^{t+u+2} - 2^{t+2} - 1),$&\cr
 $205.\  (3,3,2^{t+3} - 1,2^{t+u+2} - 2^{t+2} - 4),$&\cr 
$206.\  (3,2^{t+3} - 1,3,2^{t+u+2} - 2^{t+2} - 4),$& \cr
$207.\  (2^{t+3} - 1,3,3,2^{t+u+2} - 2^{t+2} - 4),$&\cr 
$208.\  (3,7,2^{t+u+2} - 5,2^{t+2} - 4),$&\cr 
$209.\  (7,2^{t+2} - 5,3,2^{t+u+2} - 4),$&\cr 
$210.\  (7,2^{t+u+2} - 5,3,2^{t+2} - 4),$& \cr
$211.\  (7,2^{t+3} - 5,3,2^{t+u+2} - 2^{t+2} - 4),$& \cr 
$212.\  (7,3,2^{t+3} - 5,2^{t+u+2} - 2^{t+2} - 4),$&\cr 
$213.\  (7,7,2^{t+2} - 8,2^{t+u+2} - 5),$&\cr 
$214.\  (7,7,2^{t+2} - 7,2^{t+u+2} - 6),$&\cr 
$215.\  (7,7,2^{t+u+2} - 7,2^{t+2} - 6),$&\cr 
$216.\  (7,7,2^{t+u+2} - 8,2^{t+2} - 5),$&\cr
$217.\  (7,7,2^{t+3} - 8,2^{t+u+2} - 2^{t+2} - 5),$& \cr 
$218.\  (7,7,2^{t+3} - 7,2^{t+u+2} - 2^{t+2} - 6),$&\cr 
$219.\  (3,7,2^{t+3} - 5,2^{t+u+2} - 2^{t+2} - 4),$&\cr 
$220.\  (7,3,2^{t+u+2} - 5,2^{t+2} - 4).$& \cr
\end{tabular}}

\medskip
For $s\geqslant 3, t, u \geqslant 2$,

\medskip
\centerline{\begin{tabular}{l}
$194.\ (3,2^{s} - 3,2^{s+t} - 2,2^{s+t+u}- 1),$\cr 
$195.\ (3,2^{s} - 3,2^{s+t+u}- 1,2^{s+t} - 2),$\cr 
$196.\ (3,2^{s+t+u}- 1,2^{s} - 3,2^{s+t} - 2),$\cr 
$197.\ (2^{s+t+u} - 1,3,2^{s} - 3,2^{s+t} - 2),$\cr 
$198.\ (3,2^{s} - 3,2^{s+t} - 1,2^{s+t+u}- 2),$\cr 
$199.\ (3,2^{s} - 3,2^{s+t+u}- 2,2^{s+t} - 1),$\cr 
$200.\ (3,2^{s+t} - 1,2^{s} - 3,2^{s+t+u}- 2),$\cr 
$201.\ (2^{s+t} - 1,3,2^{s} - 3,2^{s+t+u}- 2),$\cr 
$202.\ (2^{s} - 1,3,2^{s+t} - 3,2^{s+t+u}- 2),$\cr 
$203.\ (2^{s} - 1,3,2^{s+t+u}- 3,2^{s+t} - 2),$\cr 
$204.\ (3,2^{s} - 3,2^{s+t+1} - 2,2^{s+t+u}- 2^{s+t}- 1),$\cr 
$205.\ (3,2^{s} - 3,2^{s+t+1} - 1,2^{s+t+u}- 2^{s+t}- 2),$\cr 
$206.\ (3,2^{s+t+1} - 1,2^{s} - 3,2^{s+t+u}- 2^{s+t}- 2),$\cr 
$207.\ (2^{s+t+1} - 1,3,2^{s} - 3,2^{s+t+u}- 2^{s+t}- 2),$\cr 
\end{tabular}}
\centerline{\begin{tabular}{l}
$208.\ (2^{s} - 1,3,2^{s+t+1} - 3,2^{s+t+u}- 2^{s+t}- 2),$\cr 
$209.\ (7,2^{s+t} - 5,2^{s} - 3,2^{s+t+u}- 2),$\cr 
$210.\ (7,2^{s+t+u}- 5,2^{s} - 3,2^{s+t} - 2),$\cr 
$211.\ (7,2^{s+t+1} - 5,2^{s} - 3,2^{s+t+u}- 2^{s+t}- 2),$\cr 
$212.\ (7,2^{s+1} - 5,2^{s+t}-2^{s} - 3,2^{s+t+u}- 2),$\cr 
$213.\ (7,2^{s+1} - 5,2^{s+t+u}- 3,2^{s+t}-2^{s} - 2),$\cr 
$214.\ (7,2^{s+1} - 5,2^{s+t} - 3,2^{s+t+u}- 2^{s}- 2),$\cr 
$215.\ (7,2^{s+1} - 5,2^{s+t+u}- 2^{s}- 3,2^{s+t} - 2),$\cr 
$216.\ (7,2^{s+1} - 5,2^{s+t+1}-2^{s} - 3,2^{s+t+u}- 2^{s+t}- 2),$\cr 
$217.\ (7,2^{s+1} - 5,2^{s+t+1} - 3,2^{s+t+u}- 2^{s+t}- 2^{s} - 2).$\cr
\end{tabular}}

\medskip
For $s=3, t, u\geqslant 2$,

\medskip
\centerline{\begin{tabular}{l}
$218.\  (7,7,2^{t+3} - 7,2^{t+u+3} - 2),$\cr 
$219.\  (7,7,2^{t+u+3} - 7,2^{t+3} - 2),$\cr 
$220.\  (7,7,2^{t+4} - 7,2^{t+u+3}- 2^{t+3} - 2).$\cr 
\end{tabular}}

\medskip
For $s \geqslant 4, t, u\geqslant 2$,

\medskip
\centerline{\begin{tabular}{l}
$218.\ (7,2^{s} - 5,2^{s+t} - 3,2^{s+t+u}- 2),$\cr 
$219.\ (7,2^{s} - 5,2^{s+t+u}- 3,2^{s+t} - 2),$\cr 
$220.\ (7,2^{s} - 5,2^{s+t+1} - 3,2^{s+t+u}- 2^{s+t}- 2).$\cr 
\end{tabular}}

\medskip
For $s \geqslant 2, t=2, u\geqslant 2$,

\medskip
\centerline{\begin{tabular}{l}
$221.\  (1,2^{s+2} - 2,2^{s+2} - 1,2^{s+u+2} - 2^{s+1} - 2^{s} - 1),$\cr 
$222.\  (1,2^{s+2} - 1,2^{s+2} - 2,2^{s+u+2} - 2^{s+1} - 2^{s} - 1),$\cr 
$223.\  (2^{s+2} - 1,1,2^{s+2} - 2,2^{s+u+2} - 2^{s+1} - 2^{s} - 1),$\cr 
$224.\  (1,2^{s+2} - 1,2^{s+2} - 1,2^{s+u+2} - 2^{s+1} - 2^{s} - 2),$\cr 
$225.\  (2^{s+2} - 1,1,2^{s+2} - 1,2^{s+u+2} - 2^{s+1} - 2^{s} - 2),$\cr 
$226.\  (2^{s+2} - 1,2^{s+2} - 1,1,2^{s+u+2} - 2^{s+1} - 2^{s} - 2),$\cr 
$227.\  (3,2^{s+2} - 3,2^{s+2} - 2,2^{s+u+2} - 2^{s+1} - 2^{s} - 1),$\cr 
$228.\  (3,2^{s+2} - 3,2^{s+2} - 1,2^{s+u+2} - 2^{s+1} - 2^{s} - 2),$\cr 
$229.\  (3,2^{s+2} - 1,2^{s+2} - 3,2^{s+u+2} - 2^{s+1} - 2^{s} - 2),$\cr 
$230.\  (2^{s+2} - 1,3,2^{s+2} - 3,2^{s+u+2} - 2^{s+1} - 2^{s} - 2),$\cr 
$231.\  (7,2^{s+2} - 5,2^{s+2} - 3,2^{s+u+2} - 2^{s+1} - 2^{s} - 2).$\cr
\end{tabular}}

\medskip
For $s \geqslant 2, t\geqslant 3, u \geqslant 2$,

\medskip
\centerline{\begin{tabular}{l}
$221.\ (1,2^{s+2} - 2,2^{s+t}-2^{s+1} - 1,2^{s+t+u}- 2^{s}- 1),$\cr 
$222.\ (1,2^{s+2} - 1,2^{s+t}-2^{s+1} - 2,2^{s+t+u}- 2^{s}- 1),$\cr 
$223.\ (2^{s+2} - 1,1,2^{s+t}-2^{s+1} - 2,2^{s+t+u}- 2^{s}- 1),$\cr 
$224.\ (1,2^{s+2} - 1,2^{s+t}-2^{s+1} - 1,2^{s+t+u}- 2^{s}- 2),$\cr 
$225.\ (2^{s+2} - 1,1,2^{s+t}-2^{s+1} - 1,2^{s+t+u}- 2^{s}- 2),$\cr 
$226.\ (2^{s+2} - 1,2^{s+t}-2^{s+1} - 1,1,2^{s+t+u}- 2^{s}- 2),$\cr 
$227.\ (3,2^{s+2} - 3,2^{s+t}-2^{s+1} - 2,2^{s+t+u}- 2^{s}- 1),$\cr 
$228.\ (3,2^{s+2} - 3,2^{s+t}-2^{s+1} - 1,2^{s+t+u}- 2^{s}- 2),$\cr 
$229.\ (3,2^{s+2} - 1,2^{s+t}-2^{s+1} - 3,2^{s+t+u}- 2^{s}- 2),$\cr 
$230.\ (2^{s+2} - 1,3,2^{s+t}-2^{s+1} - 3,2^{s+t+u}- 2^{s}- 2),$\cr 
$231.\ (7,2^{s+2} - 5,2^{s+t}-2^{s+1} - 3,2^{s+t+u}- 2^{s}- 2).$\cr
\end{tabular}}
\end{thms}

We prove this theorem by combining some propositions.

\begin{props}\label{mdc8.9} For $n= {2^{s+t+u}+ 2^{s+t}+2^s -3}$ with $s,t,u \geqslant 2$, the $\mathbb F_2$-vector space $(\mathbb F_2\underset {\mathcal A}\otimes R_4)_n$ is generated by the  elements listed in Theorem \ref{dlc8.9}.
\end{props}

We need the following lemma.

\begin{lems}\label{8.9.1} The following matrix is strictly inadmissible
$$\begin{pmatrix} 1&1&1&0\\ 1&1&1&0\\ 1&1&0&0\\ 1&1&0&0\\ 0&0&1&0\\ 0&0&0&1\end{pmatrix} .$$
\end{lems}

\begin{proof} The monomials corresponding to the above matrix is (15,15,19,32). By a direct calculation, we get
\begin{align*}
&(15,15,19,32) = Sq^1(15,15,19,31) + Sq^2(15,15,18,31)\\ 
&\quad +Sq^4\big((4,23,19,31) + (15,12,19,31) + (7,20,19,31) + (5,23,18,31) \\ 
&\quad+ (15,13,18,31) +(7,21,18,31)\big)+Sq^8\big((8,15,19,31) + (4,27,11,31) \\ 
&\quad+ (11,12,19,31) + (7,24,11,31) + (9,15,18,31) + (5,27,10,31) \\ 
&\quad+ (11,13,18,31) + (7,25,10,31)\big)  + (8,15,19,39) + (4,35,11,31)\\ 
&\quad + (4,27,11,39) + (15,12,19,35) + (11,12,19,39) + (7,32,11,31)\\ 
&\quad + (7,24,11,39) + (15,15,18,33) + (9,15,18,39)\\ 
&\quad+ (5,35,10,31) + (5,27,10,39) + (15,13,18,35) + (11,13,18,39) \\ 
&\quad+ (7,33,10,31) + (7,25,10,39)\quad \text{mod  }\mathcal L_4(3;3;2;2;1;1).
\end{align*}
  
The lemma is proved.
\end{proof}

Using the results in Section \ref{7}, the lemmas in Sections \ref{3}, \ref{5}, \ref{6}, Lemmas \ref{8.3.1}, \ref{8.3.2}, \ref{8.4.1}, \ref{8.8.1}, \ref{8.9.1} and Theorem \ref{2.4},  we get Proposition \ref{mdc8.9}.

\medskip
Now, we prove that the classes listed in Theorem \ref{dlc8.9} are linearly independent. 

\begin{props}\label{8.9.3} For $u \geqslant 2$, the elements $[a_{u,2,2,j}], 1 \leqslant j \leqslant 231,$  are linearly independent in $(\mathbb F_2\underset {\mathcal A}\otimes R_4)_{2^{u+4}+17}$.
\end{props}

\begin{proof} Suppose that there is a linear relation
\begin{equation}\sum_{j=1}^{231}\gamma_j[a_{u,2,2,j}] = 0, \tag {\ref{8.9.3}.1}
\end{equation}
with $\gamma_j \in \mathbb F_2$.

Applying the homomorphisms $f_1, f_2, \ldots, f_6$ to the relation (\ref{8.9.3}.1), we obtain
\begin{align*}
&\gamma_{1}w_{u,2,2,1} + \gamma_{2}w_{u,2,2,2} +   \gamma_{15}w_{u,2,2,3} +  \gamma_{16}w_{u,2,2,4} +  \gamma_{29}w_{u,2,2,5}\\
&\ +  \gamma_{30}w_{u,2,2,6}  +  \gamma_{37}w_{u,2,2,7} +  \gamma_{44}w_{u,2,2,8} +  \gamma_{45}w_{u,2,2,9}\\
&\ +  \gamma_{55}w_{u,2,2,10} +  \gamma_{56}w_{u,2,2,11} +  \gamma_{69}w_{u,2,2,12} +  \gamma_{73}w_{u,2,2,13}\\
&\ +  \gamma_{74}w_{u,2,2,14} +  \gamma_{79}w_{u,2,2,15} +  \gamma_{99}w_{u,2,2,16} +  \gamma_{100}w_{u,2,2,17}\\
&\ +  \gamma_{115}w_{u,2,2,18} +  \gamma_{143}w_{u,2,2,19} +  \gamma_{123}w_{u,2,2,20} +  \gamma_{221}w_{u,2,2,21} =0,\\  
&\gamma_{3}w_{u,2,2,1} +  \gamma_{5}w_{u,2,2,2} +  \gamma_{\{13, 194\}}w_{u,2,2,3} +  \gamma_{18}w_{u,2,2,4} +  \gamma_{\{26, 199\}}w_{u,2,2,5}\\
&\ +  \gamma_{28}w_{u,2,2,6} +  \gamma_{38}w_{u,2,2,7} +  \gamma_{\{43, 204\}}w_{u,2,2,8} +  \gamma_{49}w_{u,2,2,9} +  \gamma_{59}w_{u,2,2,10}\\
&\ +  \gamma_{57}w_{u,2,2,11} + \gamma_{\{68, 216\}}w_{u,2,2,12} +  \gamma_{\{75, 84, 103\}}w_{u,2,2,13} +   \gamma_{84}w_{u,2,2,14}\\
&\ +  \gamma_{\{77, 84, 103, 213\}}w_{u,2,2,15} +  \gamma_{103}w_{u,2,2,16} +  \gamma_{101}w_{u,2,2,17} +  \gamma_{\{114, 217\}}w_{u,2,2,18}\\
&\ +  \gamma_{144}w_{u,2,2,19} +  \gamma_{127}w_{u,2,2,20} +  \gamma_{\{84, 103, 127, 144, 222\}}w_{u,2,2,21} = 0,\\  
&\gamma_{4}w_{u,2,2,1} +  \gamma_{6}w_{u,2,2,2} +  \gamma_{\{14, 195\}}w_{u,2,2,3} +  \gamma_{\{17, 196\}}w_{u,2,2,4} +  \gamma_{\{25, 198\}}w_{u,2,2,5}\\
&\ +  \gamma_{\{27, 200\}}w_{u,2,2,6} +  \gamma_{39}w_{u,2,2,7} +  \gamma_{\{47, 205\}}w_{u,2,2,8} +  \gamma_{\{48, 206\}}w_{u,2,2,9}\\
&\ +  \gamma_{60}w_{u,2,2,10} +  \gamma_{61}w_{u,2,2,11} +  \gamma_{\{67, 202\}}w_{u,2,2,12}\\
&\ +  \gamma_{\{83, 117, 118, 124, 224\}}w_{u,2,2,13} +  \gamma_{\{83, 224\}}w_{u,2,2,14} +  \gamma_{\{78, 83, 117, 208\}}w_{u,2,2,15}\\
&\ +  \gamma_{117}w_{u,2,2,16} +  \gamma_{118}w_{u,2,2,17} +  \gamma_{\{105, 219\}}w_{u,2,2,18} +  \gamma_{146}w_{u,2,2,19}\\
&\ +  \gamma_{\{85, 118\}}w_{u,2,2,20} +  \gamma_{\{83, 85, 117, 118, 146\}}w_{u,2,2,21} =0,\\  
&\gamma_{\{19, 91, 136, 194\}}w_{u,2,2,1} + \gamma_{\{32, 96, 157, 199\}}w_{u,2,2,2} +  \gamma_{7}w_{u,2,2,3} +   \gamma_{35}w_{u,2,2,4}\\
&\ +  \gamma_{10}w_{u,2,2,5} +  \gamma_{23}w_{u,2,2,6} +  \gamma_{\{46, 129, 168, 204\}}w_{u,2,2,7} +  \gamma_{40}w_{u,2,2,8}\\
&\ +  \gamma_{53}w_{u,2,2,9} +  \gamma_{62}w_{u,2,2,10} +  \gamma_{\{71, 216\}}w_{u,2,2,11} +  \gamma_{58}w_{u,2,2,12} +  \gamma_{87}w_{u,2,2,14}\\
&\  +  \gamma_{\{80, 213\}}w_{u,2,2,13}+  \gamma_{76}w_{u,2,2,15} +  \gamma_{104}w_{u,2,2,16} +  \gamma_{\{116, 217\}}w_{u,2,2,17}\\
&\ +  \gamma_{110}w_{u,2,2,18} +  \gamma_{145}w_{u,2,2,19} +  \gamma_{128}w_{u,2,2,20} +  \gamma_{223}w_{u,2,2,21} = 0,\\  
&\gamma_{\{20, 92, 137, 195\}}w_{u,2,2,1} +  \gamma_{\{31, 97, 156, 198, 228\}}w_{u,2,2,2} +  \gamma_{8}w_{u,2,2,3}\\
&\ +  \gamma_{\{34, 201\}}w_{u,2,2,4} +  \gamma_{11}w_{u,2,2,5} +  \gamma_{\{22, 197\}}w_{u,2,2,6} +  \gamma_{\{50, 130, 171, 205\}}w_{u,2,2,7}\\
&\ +  \gamma_{41}w_{u,2,2,8} +  \gamma_{\{52, 207\}}w_{u,2,2,9} +  \gamma_{63}w_{u,2,2,10} +  \gamma_{\{70, 203\}}w_{u,2,2,11}\\
&\ +  \gamma_{65}w_{u,2,2,12} + \gamma_{\{81, 220, 225\}}w_{u,2,2,13} +  \gamma_{\{86, 225\}}w_{u,2,2,14}\\
&\ +   \gamma_{\{121, 125\}}w_{u,2,2,15} +  \gamma_{119}w_{u,2,2,16} +  \gamma_{\{106, 212\}}w_{u,2,2,17} +  \gamma_{121}w_{u,2,2,18}\\
&\ +  \gamma_{147}w_{u,2,2,19} +  \gamma_{\{89, 121\}}w_{u,2,2,20} +  \gamma_{225}w_{u,2,2,21} = 0,\\  
&\gamma_{\{33, 95, 158, 183, 189, 200, 229\}}w_{u,2,2,1} + \gamma_{\{21, 93, 138, 196\}}w_{u,2,2,2}\\
&\ +  \gamma_{\{36, 98, 159, 184, 190, 201, 230\}}w_{u,2,2,3} +  \gamma_{9}w_{u,2,2,4} +  \gamma_{\{24, 94, 139, 197\}}w_{u,2,2,5}\\
&\ +  \gamma_{12}w_{u,2,2,6} +  \gamma_{\{51, 131, 172, 206\}}w_{u,2,2,7} +  \gamma_{\{54, 132, 173, 207\}}w_{u,2,2,8}\\
&\ +  \gamma_{42}w_{u,2,2,9} +  \gamma_{\{72, 166, 186, 192, 193, 209, 231\}}w_{u,2,2,10} +  \gamma_{64}w_{u,2,2,11}\\
&\ +  \gamma_{66}w_{u,2,2,12} +  \gamma_{88}w_{u,2,2,13} +  \gamma_{\{82, 167, 185, 210\}}w_{u,2,2,14} +  \gamma_{90}w_{u,2,2,15}\\
&\ +  \gamma_{\{107, 180, 191, 211\}}w_{u,2,2,16} +  \gamma_{120}w_{u,2,2,17} +  \gamma_{122}w_{u,2,2,18}\\
&\ +  \gamma_{148}w_{u,2,2,19} +  \gamma_{126}w_{u,2,2,20} +  \gamma_{226}w_{u,2,2,21} = 0.  
\end{align*}

Computing directly from the above equalities, we obtain
\begin{equation}\begin{cases}
\gamma_j = 0,\ j = 1, \ldots , 12, 15, 16, 18, 23, 28, 29, 30, 35,\\
 37, \ldots , 42, 44, 45, 49, 53, 55, \ldots , 66, 69, 73, 74, 75,\\ 76, 79,
 83, \ldots , 90, 99, 100, 101, 103, 104, 110, 115,\\ 117, \ldots , 128,
 143, 144, \ldots , 148, 221, \ldots , 226,\\
\gamma_{\{13, 194\}} =   
\gamma_{\{26, 199\}} =    
\gamma_{\{43, 204\}} =   
\gamma_{\{68, 216\}} = 
\gamma_{\{77, 213\}} = 0,\\  
\gamma_{\{114, 217\}} =   
\gamma_{\{14, 195\}} =   
\gamma_{\{17, 196\}} =   
\gamma_{\{25, 198\}} =   
\gamma_{\{27, 200\}} = 0,\\  
\gamma_{\{47, 205\}} =   
\gamma_{\{48, 206\}} =   
\gamma_{\{78, 208\}} =   
\gamma_{\{105, 219\}} =   
\gamma_{\{71, 216\}} = 0,\\  
\gamma_{\{67, 202\}} =   
\gamma_{\{19, 91, 136, 194\}} =  
\gamma_{\{80, 213\}} =   
\gamma_{\{32, 96, 157, 199\}} =   0,\\
\gamma_{\{46, 129, 168, 204\}} =   
\gamma_{\{116, 217\}} =  
\gamma_{\{34, 201\}} =   
\gamma_{\{20, 92, 137, 195\}} =  0,\\
\gamma_{\{22, 197\}} =   
\gamma_{\{31, 97, 156, 198, 228\}} =  
\gamma_{\{52, 207\}} =   
\gamma_{\{70, 203\}} = 0,\\  
\gamma_{\{50, 130, 171, 205\}} =   
\gamma_{\{81, 220\}} =   
\gamma_{\{33, 95, 158, 183, 189, 200, 229\}} = 0,\\  
\gamma_{\{106, 212\}} =   
\gamma_{\{21, 93, 138, 196\}} =  
\gamma_{\{36, 98, 159, 184, 190, 201, 230\}} = 0,\\  
\gamma_{\{24, 94, 139, 197\}} =    
\gamma_{\{51, 131, 172, 206\}} =   
\gamma_{\{54, 132, 173, 207\}} = 0,\\  
\gamma_{\{72, 166, 186, 192, 193, 209, 231\}} = 0,\\  
\gamma_{\{82, 167, 185, 210\}} =   
\gamma_{\{107, 180, 191, 211\}} = 0.
\end{cases}\tag{\ref{8.9.3}.2}
\end{equation}

With the aid of (\ref{8.9.3}.2), the homomorphisms $g_1, g_2$ send (\ref{8.9.3}.1) to
\begin{align*}
&\gamma_{19}w_{u,2,2,1} +  \gamma_{32}w_{u,2,2,2} +   \gamma_{91}w_{u,2,2,3} +  \gamma_{109}w_{u,2,2,4} +  \gamma_{96}w_{u,2,2,5}\\
&\  +  \gamma_{113}w_{u,2,2,6} +  \gamma_{46}w_{u,2,2,7} +  \gamma_{129}w_{u,2,2,8} +  \gamma_{135}w_{u,2,2,9} +  \gamma_{136}w_{u,2,2,10}\\
&\ +  \gamma_{\{68, 141\}}w_{u,2,2,11} +  \gamma_{142}w_{u,2,2,12} +  \gamma_{\{77, 154\}}w_{u,2,2,13} +  \gamma_{157}w_{u,2,2,14}\\
&\ +  \gamma_{160}w_{u,2,2,15} +  \gamma_{168}w_{u,2,2,16} +  \gamma_{\{114, 170\}}w_{u,2,2,17} +  \gamma_{174}w_{u,2,2,18}\\
&\ +  \gamma_{187}w_{u,2,2,19} +  \gamma_{182}w_{u,2,2,20} +  \gamma_{227}w_{u,2,2,21} = 0,\\  
&\gamma_{20}w_{u,2,2,1} + \gamma_{31}w_{u,2,2,2} +  \gamma_{92}w_{u,2,2,3} +  \gamma_{\{108, 209\}}w_{u,2,2,4} +  \gamma_{97}w_{u,2,2,5}\\
&\ +  \gamma_{\{112, 210\}}w_{u,2,2,6} +  \gamma_{50}w_{u,2,2,7} +  \gamma_{130}w_{u,2,2,8} +  \gamma_{\{134, 211\}}w_{u,2,2,9}\\
&\ +  \gamma_{137}w_{u,2,2,10} +  \gamma_{140}w_{u,2,2,11} + \gamma_{152}w_{u,2,2,12}+  \gamma_{\{155, 228\}}w_{u,2,2,13}\\
&\  +   \gamma_{\{156, 228\}}w_{u,2,2,14} +  \gamma_{\{178, 181\}}w_{u,2,2,15} +  \gamma_{171}w_{u,2,2,16} +  \gamma_{169}w_{u,2,2,17}\\
&\ +  \gamma_{178}w_{u,2,2,18} +  \gamma_{188}w_{u,2,2,19} +  \gamma_{\{162, 178\}}w_{u,2,2,20} +  \gamma_{228}w_{u,2,2,21} = 0.  
\end{align*}

These equalities imply
\begin{equation}\begin{cases}
\gamma_j = 0,\ j = 19, 20, 31, 32, 46, 50, 91, 92, 96, 97, 109, 113,\\ 129,
 130, 135, 136, 137, 140, 142, 152, 155, 156, 157, 160,\\ 162, 168, 169, 171, 174, 178, 181, 182, 187, 188, 227, 228,\\
\gamma_{\{68, 141\}} =   
\gamma_{\{77, 154\}} =   
\gamma_{\{114, 170\}} =   
\gamma_{\{108, 209\}} = 0,\\  
\gamma_{\{112, 210\}} =   
\gamma_{\{134, 211\}} =  0. 
\end{cases}\tag{\ref{8.9.3}.3}
\end{equation}

With the aid of (\ref{8.9.3}.2) and (\ref{8.9.3}.3), the homomorphisms $g_3, g_4$ send (\ref{8.9.3}.1) to
\begin{align*}
&\gamma_{21}w_{u,2,2,1} + \gamma_{33}w_{u,2,2,2} +  \gamma_{93}w_{u,2,2,3} +  a_1w_{u,2,2,4} +  \gamma_{95}w_{u,2,2,5}\\
&\ +  a_2w_{u,2,2,6} +  \gamma_{51}w_{u,2,2,7} +  \gamma_{131}w_{u,2,2,8} +  a_3w_{u,2,2,9} +  \gamma_{138}w_{u,2,2,10}\\
&\ +  a_4w_{u,2,2,11} +  \gamma_{\{151, 166, 215\}}w_{u,2,2,12} +  a_5w_{u,2,2,13} +  \gamma_{172}w_{u,2,2,16}\\
&\ +  \gamma_{\{158, 183, 189, 229\}}w_{u,2,2,14} +  \gamma_{\{167, 177, 183, 218\}}w_{u,2,2,15}  +  a_6w_{u,2,2,17}\\
&\ +  \gamma_{\{177, 180, 218\}}w_{u,2,2,18} +  \gamma_{189}w_{u,2,2,19} +  a_7w_{u,2,2,20} +  a_8w_{u,2,2,21} = 0,\\  
&\gamma_{24}w_{u,2,2,1} + \gamma_{36}w_{u,2,2,2} +  \gamma_{94}w_{u,2,2,3} +   a_9w_{u,2,2,4} +  \gamma_{98}w_{u,2,2,5}\\
&\ +  a_{10}w_{u,2,2,6} +  \gamma_{54}w_{u,2,2,7} +  \gamma_{132}w_{u,2,2,8} +  a_{11}w_{u,2,2,9} +  \gamma_{139}w_{u,2,2,10}\\
&\ +  a_{12}w_{u,2,2,11} +  a_{13}w_{u,2,2,12} +  a_{14}w_{u,2,2,13} +  \gamma_{\{159, 184, 190, 230\}}w_{u,2,2,14}\\
&\ +  a_{15}w_{u,2,2,15} +  \gamma_{173}w_{u,2,2,16} +  \gamma_{\{107, 114, 176, 211, 218\}}w_{u,2,2,17} +  \gamma_{190}w_{u,2,2,19}\\
&\ +  \gamma_{\{179, 180, 191, 193, 218\}}w_{u,2,2,18}  +  a_{16}w_{u,2,2,20} +  a_{17}w_{u,2,2,21} = 0,             
\end{align*}
where
\begin{align*}
a_1&= \gamma_{\{22, 24, 34, 36, 52, 54, 70, 94, 102, 139, 150, 159, 163, 173, 230,\}},\\
a_2 &= \gamma_{\{22, 24, 81, 94, 98, 111, 132, 139, 153, 165, 184\}},\\
a_3 &= \gamma_{\{52, 54, 106, 132, 133, 173, 176, 179, 190\}},\ \
a_4 = \gamma_{\{72, 77, 149, 186, 209, 214\}},\\
a_5 &= \gamma_{\{82, 114, 175, 183, 185, 189, 193, 210, 218, 229\}},\\
a_6 &= \gamma_{\{107, 114, 175, 191, 193, 211, 218\}},\\
a_7 &= \gamma_{\{68, 114, 161, 164, 175, 177, 192, 193, 214, 215, 231\}},\\
a_8 &= \gamma_{\{68, 114, 164, 175, 183, 189, 192, 193, 215, 218, 229\}},\\
a_9 &= \gamma_{\{13, 17, 21, 25, 27, 33, 48, 51, 67, 93, 102, 108, 112, 134, 138, 149, 154, 158, 161, 172, 229\}},\\
a_{10} &= \gamma_{\{14, 17, 21, 26, 78, 93, 95, 111, 112, 131, 138, 141, 151, 164, 183\}},\\
a_{11} &= \gamma_{\{43, 47, 48, 51, 105, 131, 133, 134, 170, 172, 175, 177, 189\}},\\
a_{12} &= \gamma_{\{72, 77, 82, 107, 150, 209, 210, 211, 214\}},\ \
a_{13} = \gamma_{\{153, 166, 167, 180, 185, 192, 215\}},\\
a_{14} &= \gamma_{\{107, 114, 176, 184, 190, 211, 218, 230\}},\ \
a_{15} = \gamma_{\{179, 180, 184, 191, 193, 218\}},\\
a_{16} &= \gamma_{\{68, 82, 107, 114, 163, 165, 167, 176, 179, 180, 185, 186, 193, 210, 211, 214, 215, 231\}},\\
a_{17} &= \gamma_{\{68, 82, 107, 114, 165, 176, 184, 190, 210, 211, 215, 218, 230\}}.
\end{align*}
From the above equalities, we obtain
\begin{equation}\begin{cases}
a_i = 0, \ i = 1,2, \ldots, 17,\\ 
\gamma_j = 0,\ j = 21, 24, 33, 36, 51, 54, 93, 94,\\ 
95, 98, 131, 132, 138, 139, 172, 173, 189, 190,\\ 
\gamma_{\{151, 166, 215\}} =  
\gamma_{\{158, 183, 229\}} = 0,\\   
\gamma_{\{167, 177, 183, 218\}} =   
\gamma_{\{177, 180, 218\}} = 0.
\end{cases}\tag{\ref{8.9.3}.4}
\end{equation}
With the aid of (\ref{8.9.3}.2), (\ref{8.9.3}.3) and (\ref{8.9.3}.4), the homomorphism $h$ sends (\ref{8.9.3}.1) to
\begin{align*}
&a_{18}w_{u,2,2,1} +  \gamma_{\{68, 111, 153, 165, 184\}}w_{u,2,2,2} +   \gamma_{72}w_{u,2,2,3} +  \gamma_{166}w_{u,2,2,4}\\
&\ +  \gamma_{82}w_{u,2,2,5} +  \gamma_{167}w_{u,2,2,6} +  \gamma_{\{114, 133, 176, 179\}}w_{u,2,2,7} +  \gamma_{107}w_{u,2,2,8}\\
&\ +  \gamma_{180}w_{u,2,2,9} +  a_{19}w_{u,2,2,10} +  \gamma_{\{151, 153, 166, 184, 215\}}w_{u,2,2,11} +  \gamma_{185}w_{u,2,2,12}\\
&\ +  a_{20}w_{u,2,2,13} +  a_{21}w_{u,2,2,14} +  a_{22}w_{u,2,2,15} +  a_{23}w_{u,2,2,16}\\
&\ +  \gamma_{\{177, 179, 180, 218\}}w_{u,2,2,17} +  \gamma_{191}w_{u,2,2,18} +  \gamma_{193}w_{u,2,2,19}\\
&\ +  \gamma_{\{167, 180, 183, 184, 192\}}w_{u,2,2,20} +  a_{24}w_{u,2,2,21} = 0, 
\end{align*}
where
\begin{align*}
a_{18} &= \gamma_{\{77, 102, 150, 159, 163, 230\}},\\
a_{19} &= \gamma_{\{72, 77, 108,  149, 150, 159, 186,  214\}},\\
a_{20} &= \gamma_{\{82, 107, 112, 134, 159, 161, 163, 180, 185, 191, 214, 231\}},\\
a_{21} &= \gamma_{\{68, 82, 112, 164, 165, 184, 185, 192, 215\}},\\
a_{22} &= \gamma_{\{82, 107, 112, 134, 158, 159, 185, 186, 191\}},\\
a_{23} &= \gamma_{\{107, 114, 134, 175, 176, 191, 193, 218\}},\\
a_{24} &= \gamma_{\{82, 107, 112, 134, 167, 180, 185, 191, 229, 230\}}.
\end{align*}
From the above equalities, it implies
\begin{equation}\begin{cases}
a_i = 0,\ i = 18, 19, \ldots , 24,\\
\gamma_j = 0,\ j = 72, 82, 107, 166, 167, 180, 185, 191, 193,\\
\gamma_{\{68, 111, 153, 165, 184\}} = 0,\\  
\gamma_{\{114, 133, 176, 179\}} = 0,\\   
\gamma_{\{151, 153, 184, 215\}} =  0,\\ 
\gamma_{\{177, 179, 218\}} =   
\gamma_{\{183, 184, 192\}} =  0.             
\end{cases}\tag{\ref{8.9.3}.5}
\end{equation}

Combining (\ref{8.9.3}.2), (\ref{8.9.3}.3), (\ref{8.9.3}.4) and (\ref{8.9.3}.5), we get $\gamma_j = 0$ for any $j$.  The proposition is proved.
\end{proof}

\begin{props}\label{8.9.4} For $u \geqslant 2$ and $t \geqslant 3$, the elements $[a_{u,t,2,j}], 1 \leqslant j \leqslant 231,$  are linearly independent in $(\mathbb F_2\underset {\mathcal A}\otimes R_4)_{2^{t+u+2}+2^{t+2} +1}$.
\end{props}

\begin{proof} Suppose that there is a linear relation
\begin{equation}\sum_{j=1}^{231}\gamma_j[a_{u,t,2,j}] = 0, \tag {\ref{8.9.4}.1}
\end{equation}
with $\gamma_j \in \mathbb F_2$.

Applying the homomorphisms $f_1, f_2, \ldots, f_6$ to the relation (\ref{8.9.4}.1), we get
\begin{align*}
&\gamma_{1}w_{u,t,2,1} + \gamma_{2}w_{u,t,2,2} +   \gamma_{15}w_{u,t,2,3} +  \gamma_{16}w_{u,t,2,4} +  \gamma_{29}w_{u,t,2,5}\\
&\ +  \gamma_{30}w_{u,t,2,6} +  \gamma_{37}w_{u,t,2,7} +  \gamma_{44}w_{u,t,2,8} +  \gamma_{45}w_{u,t,2,9} +  \gamma_{55}w_{u,t,2,10}\\
&\ +  \gamma_{56}w_{u,t,2,11} +  \gamma_{69}w_{u,t,2,12} +  \gamma_{73}w_{u,t,2,13} +  \gamma_{74}w_{u,t,2,14}\\
&\ +  \gamma_{79}w_{u,t,2,15} +  \gamma_{99}w_{u,t,2,16} +  \gamma_{100}w_{u,t,2,17} +  \gamma_{115}w_{u,t,2,18}\\
&\ +  \gamma_{143}w_{u,t,2,19} +  \gamma_{123}w_{u,t,2,20} +  \gamma_{221}w_{u,t,2,21} = 0,\\  
&\gamma_{3}w_{u,t,2,1} + \gamma_{5}w_{u,t,2,2} +  \gamma_{\{13, 194\}}w_{u,t,2,3} +  \gamma_{18}w_{u,t,2,4} +  \gamma_{\{26, 199\}}w_{u,t,2,5}\\
&\ +  \gamma_{28}w_{u,t,2,6} +  \gamma_{38}w_{u,t,2,7} +  \gamma_{\{43, 204\}}w_{u,t,2,8} +  \gamma_{49}w_{u,t,2,9} +  \gamma_{59}w_{u,t,2,10}\\
&\ +  \gamma_{57}w_{u,t,2,11} +
\gamma_{\{68, 216\}}w_{u,t,2,12} +  \gamma_{75}w_{u,t,2,13} +   \gamma_{84}w_{u,t,2,14}\\
&\ +  \gamma_{\{77, 213\}}w_{u,t,2,15} +  \gamma_{103}w_{u,t,2,16} +  \gamma_{101}w_{u,t,2,17} +  \gamma_{\{114, 217\}}w_{u,t,2,18}\\
&\ +  \gamma_{144}w_{u,t,2,19} +  \gamma_{127}w_{u,t,2,20} +  \gamma_{\{84, 103, 127, 144, 222\}}w_{u,t,2,21} = 0,\\  
&\gamma_{4}w_{u,t,2,1} + \gamma_{6}w_{u,t,2,2} +  \gamma_{\{14, 195\}}w_{u,t,2,3} +  \gamma_{\{17, 196\}}w_{u,t,2,4}\\
&\ +  \gamma_{\{25, 198\}}w_{u,t,2,5} +  \gamma_{\{27, 200\}}w_{u,t,2,6} +  \gamma_{39}w_{u,t,2,7} +  \gamma_{\{47, 205\}}w_{u,t,2,8}\\
&\ +  \gamma_{\{48, 206\}}w_{u,t,2,9} +  \gamma_{60}w_{u,t,2,10} +  \gamma_{61}w_{u,t,2,11} +  \gamma_{\{67, 202\}}w_{u,t,2,12}\\
&\ +  \gamma_{85}w_{u,t,2,13} +  \gamma_{83}w_{u,t,2,14} +  \gamma_{\{78, 208\}}w_{u,t,2,15} +  \gamma_{117}w_{u,t,2,16}\\
&\ +  \gamma_{118}w_{u,t,2,17} +  \gamma_{\{105, 219\}}w_{u,t,2,18} +  \gamma_{146}w_{u,t,2,19}\\
&\ +  \gamma_{\{85, 118, 224\}}w_{u,t,2,20} +  \gamma_{\{83, 117, 124, 146, 224\}}w_{u,t,2,21} = 0,\\  
&\gamma_{\{19, 91, 136, 194\}}w_{u,t,2,1} + \gamma_{\{32, 96, 157, 199\}}w_{u,t,2,2} +  \gamma_{7}w_{u,t,2,3} +   \gamma_{35}w_{u,t,2,4}\\
&\ +  \gamma_{10}w_{u,t,2,5} +  \gamma_{23}w_{u,t,2,6} +  \gamma_{\{46, 129, 168, 204\}}w_{u,t,2,7} +  \gamma_{40}w_{u,t,2,8}\\
&\ +  \gamma_{53}w_{u,t,2,9} +  \gamma_{62}w_{u,t,2,10} +  \gamma_{\{71, 216\}}w_{u,t,2,11} +  \gamma_{58}w_{u,t,2,12} +  \gamma_{87}w_{u,t,2,14} \\
&\ +  \gamma_{\{80, 213\}}w_{u,t,2,13} +  \gamma_{76}w_{u,t,2,15} +  \gamma_{104}w_{u,t,2,16} +  \gamma_{\{116, 217\}}w_{u,t,2,17}\\
&\ +  \gamma_{102}w_{u,t,2,18} +  \gamma_{145}w_{u,t,2,19} +  \gamma_{128}w_{u,t,2,20} +  \gamma_{223}w_{u,t,2,21} = 0,\\  
&\gamma_{\{20, 92, 137, 195\}}w_{u,t,2,1} +  \gamma_{\{31, 97, 156, 198\}}w_{u,t,2,2} +  \gamma_{8}w_{u,t,2,3} +  \gamma_{\{34, 201\}}w_{u,t,2,4}\\
&\ +  \gamma_{11}w_{u,t,2,5} +  \gamma_{\{22, 197\}}w_{u,t,2,6} +  \gamma_{\{50, 130, 171, 205\}}w_{u,t,2,7} +  \gamma_{41}w_{u,t,2,8}\\
&\ +  \gamma_{\{52, 207\}}w_{u,t,2,9} +  \gamma_{63}w_{u,t,2,10} +  \gamma_{\{70, 203\}}w_{u,t,2,11} +  \gamma_{65}w_{u,t,2,12}\\
&\ + \gamma_{\{81, 220\}}w_{u,t,2,13} +  \gamma_{86}w_{u,t,2,14} +   \gamma_{89}w_{u,t,2,15} +  \gamma_{119}w_{u,t,2,16}\\
&\ +  \gamma_{\{106, 212\}}w_{u,t,2,17} +  \gamma_{121}w_{u,t,2,18} +  \gamma_{147}w_{u,t,2,19}\\
&\ +  \gamma_{\{89, 121, 225\}}w_{u,t,2,20} +  \gamma_{\{89, 121, 125\}}w_{u,t,2,21} = 0,\\  
&\gamma_{\{33, 95, 158, 200\}}w_{u,t,2,1} + \gamma_{\{21, 93, 138, 196\}}w_{u,t,2,2} +  \gamma_{\{36, 98, 159, 201\}}w_{u,t,2,3}\\
&\ +  \gamma_{9}w_{u,t,2,4} +  \gamma_{\{24, 94, 139, 197\}}w_{u,t,2,5} +  \gamma_{12}w_{u,t,2,6} +  \gamma_{\{51, 131, 172, 206\}}w_{u,t,2,7}\\
&\ +  \gamma_{\{54, 132, 173, 207\}}w_{u,t,2,8} +  \gamma_{42}w_{u,t,2,9} +  \gamma_{\{72, 166, 186, 209\}}w_{u,t,2,10} +  \gamma_{64}w_{u,t,2,11}\\
&\ +  \gamma_{66}w_{u,t,2,12} +  \gamma_{88}w_{u,t,2,13} +  \gamma_{\{82, 167, 185, 210\}}w_{u,t,2,14} +  \gamma_{90}w_{u,t,2,15}\\
&\ +  \gamma_{\{107, 180, 191, 211\}}w_{u,t,2,16} +  \gamma_{120}w_{u,t,2,17} +  \gamma_{122}w_{u,t,2,18}\\
&\ +  \gamma_{148}w_{u,t,2,19} +  \gamma_{126}w_{u,t,2,20} +  \gamma_{226}w_{u,t,2,21} = 0.   
\end{align*}

Computing directly from the above equalities, we obtain
\begin{equation}\begin{cases}
\gamma_j = 0,\ 1, \ldots , 12, 15, 16, 18, 23, 28, 29, 30, 35,\\
 37, \ldots, 42, 44, 45, 49, 53, 55, \ldots , 66, 69, 73,\\
 74, 75, 76, 79, 83, \ldots , 90, 99, \ldots , 104, 115,\\
 117, \ldots , 128, 143, 144, \ldots , 148, 221, \ldots , 226,\\
\gamma_{\{13, 194\}} =   
\gamma_{\{26, 199\}} =    
\gamma_{\{43, 204\}} =  
 \gamma_{\{68, 216\}} =  
\gamma_{\{77, 213\}} = 0,\\  
\gamma_{\{114, 217\}} =  
\gamma_{\{14, 195\}} =   
\gamma_{\{17, 196\}} =   
\gamma_{\{25, 198\}} = 
\gamma_{\{27, 200\}} = 0,\\  
\gamma_{\{47, 205\}} =   
\gamma_{\{48, 206\}} =   
\gamma_{\{67, 202\}} =   
\gamma_{\{78, 208\}} =  
\gamma_{\{71, 216\}} = 0,\\  
\gamma_{\{105, 219\}} =   
\gamma_{\{19, 91, 136, 194\}} =  
\gamma_{\{32, 96, 157, 199\}} =   
\gamma_{\{80, 213\}} = 0,\\  
\gamma_{\{46, 129, 168, 204\}} = 
\gamma_{\{116, 217\}} =   
\gamma_{\{20, 92, 137, 195\}} = 
\gamma_{\{34, 201\}} = 0,\\  
\gamma_{\{31, 97, 156, 198\}} =   
\gamma_{\{22, 197\}} =   
\gamma_{\{50, 130, 171, 205\}} =   
\gamma_{\{52, 207\}} = 0,\\  
\gamma_{\{70, 203\}} =   
\gamma_{\{81, 220\}} =   
\gamma_{\{106, 212\}} =   
\gamma_{\{33, 95, 158, 200\}} = 0,\\  
\gamma_{\{21, 93, 138, 196\}} = 
\gamma_{\{36, 98, 159, 201\}} =  
\gamma_{\{24, 94, 139, 197\}} = 0,\\   
\gamma_{\{51, 131, 172, 206\}} =   
\gamma_{\{54, 132, 173, 207\}} =  
\gamma_{\{72, 166, 186, 209\}} = 0,\\  
\gamma_{\{82, 167, 185, 210\}} =   
\gamma_{\{107, 180, 191, 211\}} = 0. 
\end{cases}\tag{\ref{8.9.4}.4}
\end{equation}

With the aid of (\ref{8.9.4}.2), the homomorphisms $g_1, g_2$ send (\ref{8.9.4}.1) to
\begin{align*}
&\gamma_{19}w_{u,t,2,1} + \gamma_{32}w_{u,t,2,2} +   \gamma_{91}w_{u,t,2,3} +  \gamma_{110}w_{u,t,2,4} +  \gamma_{96}w_{u,t,2,5}\\
&\ +  \gamma_{113}w_{u,t,2,6} +  \gamma_{46}w_{u,t,2,7} +  \gamma_{129}w_{u,t,2,8} +  \gamma_{135}w_{u,t,2,9} +  \gamma_{136}w_{u,t,2,10}\\
&\ +  \gamma_{68, 141}w_{u,t,2,11} +  \gamma_{142}w_{u,t,2,12} +  \gamma_{\{77, 154\}}w_{u,t,2,13} +  \gamma_{157}w_{u,t,2,14}\\
&\ +  \gamma_{160}w_{u,t,2,15} +  \gamma_{168}w_{u,t,2,16} +  \gamma_{\{114, 170\}}w_{u,t,2,17} +  \gamma_{174}w_{u,t,2,18}\\
&\ +  \gamma_{187}w_{u,t,2,19} +  \gamma_{182}w_{u,t,2,20} +  \gamma_{227}w_{u,t,2,21} = 0,\\  
&\gamma_{20}w_{u,t,2,1} +  \gamma_{31}w_{u,t,2,2} +  \gamma_{92}w_{u,t,2,3} +  \gamma_{\{109, 209\}}w_{u,t,2,4} +  \gamma_{97}w_{u,t,2,5}\\
&\ +  \gamma_{\{112, 210\}}w_{u,t,2,6} +  \gamma_{50}w_{u,t,2,7} +  \gamma_{130}w_{u,t,2,8} +  \gamma_{\{134, 211\}}w_{u,t,2,9}\\
&\ +  \gamma_{137}w_{u,t,2,10} +  \gamma_{140}w_{u,t,2,11} + \gamma_{152}w_{u,t,2,12} +  \gamma_{155}w_{u,t,2,13} +   \gamma_{156}w_{u,t,2,14}\\
&\ +  \gamma_{162}w_{u,t,2,15} +  \gamma_{171}w_{u,t,2,16} +  \gamma_{169}w_{u,t,2,17} +  \gamma_{178}w_{u,t,2,18}\\
&\ +  \gamma_{188}w_{u,t,2,19} +  \gamma_{\{162, 178, 228\}}w_{u,t,2,20} +  \gamma_{\{162, 178, 181\}}w_{u,t,2,21} = 0.  
\end{align*}

These equalities imply
\begin{equation}\begin{cases}
\gamma_j = 0,\ j = 19, 20, 31, 32, 46, 50, 91, 92, 96, 97, 110, 113, \\
129, 130, 135, 136, 137, 140, 142, 152, 155, 156, 157, 160,\\
162, 168, 169, 171, 174, 178, 181, 182, 187, 188, 227, 228,\\
\gamma_{\{68, 141\}} =   
\gamma_{\{77, 154\}} =    
\gamma_{\{114, 170\}} = 0,\\  
\gamma_{\{109, 209\}} =   
\gamma_{\{112, 210\}} =   
\gamma_{\{134, 211\}} = 0.  
\end{cases}\tag{\ref{8.9.4}.3}
\end{equation}

With the aid of (\ref{8.9.4}.2) and (\ref{8.9.4}.3), the homomorphisms $g_3, g_4$ send (\ref{8.9.4}.1) to
\begin{align*}
&\gamma_{21}w_{u,t,2,1} +  \gamma_{33}w_{u,t,2,2} +  \gamma_{93}w_{u,t,2,3} +  a_1w_{u,t,2,4} +  \gamma_{95}w_{u,t,2,5}\\
&\ +  a_2w_{u,t,2,6} +  \gamma_{51}w_{u,t,2,7} +  \gamma_{131}w_{u,t,2,8} +  a_3w_{u,t,2,9} +  \gamma_{138}w_{u,t,2,10}\\
&\ +  a_4w_{u,t,2,11} +  \gamma_{\{151, 166, 215\}}w_{u,t,2,12} +  a_5w_{u,t,2,13} +  \gamma_{158}w_{u,t,2,14}\\
&\ +  \gamma_{\{161, 167, 214\}}w_{u,t,2,15} +  \gamma_{172}w_{u,t,2,16} +  a_6w_{u,t,2,17} +  \gamma_{\{177, 180, 218\}}w_{u,t,2,18}\\
&\ +  \gamma_{189}w_{u,t,2,19} +  a_7w_{u,t,2,20} +  \gamma_{\{161, 177, 183, 214, 218\}}w_{u,t,2,21} = 0,\\  
&\gamma_{24}w_{u,t,2,1} + \gamma_{36}w_{u,t,2,2} +  \gamma_{94}w_{u,t,2,3} +   a_8w_{u,t,2,4} +  \gamma_{98}w_{u,t,2,5}\\
&\ + a_9w_{u,t,2,6} +  \gamma_{54}w_{u,t,2,7} +  \gamma_{132}w_{u,t,2,8} +  a_{10}w_{u,t,2,9} +  \gamma_{139}w_{u,t,2,10}\\
&\ +  a_{11}w_{u,t,2,11} +  a_{12}w_{u,t,2,12} +  \gamma_{\{68, 82, 165, 210, 215\}}w_{u,t,2,13} +  \gamma_{159}w_{u,t,2,14}\\
&\ +  a_{13}w_{u,t,2,15} +  \gamma_{173}w_{u,t,2,16} +  \gamma_{\{107, 114, 176, 211, 218\}}w_{u,t,2,17}\\
&\ +  \gamma_{\{179, 180, 191, 193, 218\}}w_{u,t,2,18} +  \gamma_{190}w_{u,t,2,19} +  a_{14}w_{u,t,2,20} +  a_{15}w_{u,t,2,21} = 0,  
\end{align*}
where
\begin{align*}
a_1 &= \gamma_{\{22, 24, 34, 36, 52, 54, 70, 94, 108, 139, 150, 159, 163, 173, 230\}},\\
a_2 &= \gamma_{\{22, 24, 81, 94, 98, 111, 132, 139, 153, 165, 184\}},\ \
a_3 = \gamma_{\{52, 54, 106, 132, 133, 173, 176, 179, 190\}},\\
a_4 &= \gamma_{\{72, 77, 149, 186, 209, 214, 231\}},\ \
a_5 = \gamma_{\{68, 82, 164, 185, 192, 210, 215\}},\\
a_6 &= \gamma_{\{107, 114, 175, 191, 193, 211, 218\}},\\
a_7 &= \gamma_{\{68, 114, 161, 164, 175, 177, 183, 189, 192, 193, 214, 215, 229, \}},\\
a_8 &= \gamma_{\{13, 17, 21, 25, 27, 33, 48, 51, 67, 93, 108, 109, 112, 134, 138, 149, 154, 158, 161, 172, 229\}},\\
a_9 &=  \gamma_{\{14, 17, 21, 26, 78, 93, 95, 111, 112, 131, 138, 141, 151, 164, 183\}},\\
a_{10} &= \gamma_{\{43, 47, 48, 51, 105, 131, 133, 134, 170, 172, 175, 177, 189\}},\\
a_{11} &= \gamma_{\{72, 77, 82, 107, 150, 209, 210, 211, 214\}},\\
a_{12} &= \gamma_{\{153, 166, 167, 180, 185, 192, 215\}},\ \
a_{13} = \gamma_{\{163, 167, 185, 186, 191, 214, 231\}},\\
a_{14} &= \gamma_{\{68, 82, 107, 114, 163, 165, 167, 176, 179, 180, 184, 185, 186, 190, 193, 210, 211, 214, 215, 230, 231\}},\\
a_{15} &= \gamma_{\{163, 167, 179, 180, 184, 185, 186, 193, 214, 218, 231\}}.
\end{align*}
From the above equalities, we obtain
\begin{equation}\begin{cases}
a_i = 0, \ i = 1,2, \ldots, 15,\\ 
\gamma_j = 0,\ j = 21, 24, 33, 36, 51, 54, 93, 94, 95, 98,\\ 
131, 132, 138, 139, 158, 159, 172, 173, 189, 190,\\
\gamma_{\{151, 166, 215\}} =  
\gamma_{\{161, 167, 214\}} =   
\gamma_{\{161, 177, 183, 214, 218\}} = 0,\\  
\gamma_{\{161, 167, 214\}} =   
\gamma_{\{177, 180, 218\}} =  
\gamma_{\{68, 82, 165, 210, 215\}} = 0,\\  
\gamma_{\{107, 114, 176, 211, 218\}} =   
\gamma_{\{179, 180, 191, 193, 218\}} = 0.  
\end{cases}\tag{\ref{8.9.4}.4}
\end{equation}
With the aid of (\ref{8.9.4}.2), (\ref{8.9.4}.3) and (\ref{8.9.4}.4), the homomorphism $h$ sends (\ref{8.9.4}.1) to
\begin{align*}
&\gamma_{\{77, 108, 150, 163, 230\}}w_{u,t,2,1} +  \gamma_{\{68, 111, 153, 165, 184\}}w_{u,t,2,2} +  \gamma_{72}w_{u,t,2,3}\\
&\ +  \gamma_{166}w_{u,t,2,4} +  \gamma_{82}w_{u,t,2,5} +  \gamma_{167}w_{u,t,2,6} +  \gamma_{\{114, 133, 176, 179\}}w_{u,t,2,7}\\
&\ +  \gamma_{107}w_{u,t,2,8} +  \gamma_{180}w_{u,t,2,9} +  a_{16}w_{u,t,2,10} +  \gamma_{\{151, 153, 166, 184, 215\}}w_{u,t,2,11}\\
&\ +  \gamma_{185}w_{u,t,2,12} + \gamma_{\{161, 163, 167, 214, 230\}}w_{u,t,2,13} +  a_{17}w_{u,t,2,14} +   \gamma_{186}w_{u,t,2,15}\\
&\ +  a_{18}w_{u,t,2,16} +  \gamma_{\{177, 179, 180, 218\}}w_{u,t,2,17} +  \gamma_{191}w_{u,t,2,18}\\
&\ +  \gamma_{193}w_{u,t,2,19} +  \gamma_{\{167, 180, 183, 184, 192\}}w_{u,t,2,20} +  a_{19}w_{u,t,2,21} = 0,    
\end{align*}
where
\begin{align*}
a_{16} &=  \gamma_{\{72, 77, 109, 149, 150, 186, 214, 230, 231\}},\ \
a_{17} =  \gamma_{\{68, 82, 112, 164, 165, 184, 185, 192, 215\}},\\
a_{18} &=  \gamma_{\{107, 114, 134, 175, 176, 191, 193, 218\}},\ \
a_{19} =  \gamma_{82, 107, 112, 134, 185, 191, 229, 230, 231}.
\end{align*}
From the above equalities, it implies
\begin{equation}\begin{cases}
a_i = 0,\ i = 16, 17, 18, 19,\\
\gamma_j = 0,\ j = 72, 82, 107, 166, 167, 180, 185, 191, 193,\\
\gamma_{\{77, 108, 150, 163, 230\}} =   
\gamma_{\{68, 111, 153, 165, 184\}} = 0,\\  
 \gamma_{\{114, 133, 176, 179\}} =   
\gamma_{\{151, 153, 184, 215\}} = 0,\\  
\gamma_{\{161, 163, 214, 230\}} =   
\gamma_{\{177, 179, 218\}} =   
\gamma_{\{183, 184, 192\}} = 0. 
\end{cases}\tag{\ref{8.9.4}.5}
\end{equation}

Combining (\ref{8.9.4}.2), (\ref{8.9.4}.3), (\ref{8.9.4}.4) and (\ref{8.9.4}.5), we get $\gamma_j = 0$ for all $j$. The proposition is proved.
\end{proof}

\begin{props}\label{8.9.5} For $s \geqslant 3$ and $u \geqslant 2$, the elements $[a_{u,2,s,j}], 1 \leqslant j \leqslant 231,$  are linearly independent in $(\mathbb F_2\underset {\mathcal A}\otimes R_4)_{2^{s+u+2}+2^{s+2}+2^s -3}$.
\end{props}

\begin{proof} Suppose that there is a linear relation
\begin{equation}\sum_{j=1}^{231}\gamma_j[a_{u,2,s,j}] = 0, \tag {\ref{8.9.5}.1}
\end{equation}
with $\gamma_j \in \mathbb F_2$.

Apply the homomorphisms $f_1, f_2, \ldots, f_6$ to the relation (\ref{8.9.5}.1) and we obtain
\begin{align*}
&\gamma_{1}w_{u,2,s,1} +  \gamma_{2}w_{u,2,s,2} +   \gamma_{15}w_{u,2,s,3} +  \gamma_{16}w_{u,2,s,4} +  \gamma_{29}w_{u,2,s,5}\\
&\ +  \gamma_{30}w_{u,2,s,6} +  \gamma_{37}w_{u,2,s,7} +  \gamma_{44}w_{u,2,s,8} +  \gamma_{45}w_{u,2,s,9}\\
&\ +  \gamma_{55}w_{u,2,s,10} +  \gamma_{56}w_{u,2,s,11} +  \gamma_{69}w_{u,2,s,12} +  \gamma_{73}w_{u,2,s,13}\\
&\ +  \gamma_{74}w_{u,2,s,14} +  \gamma_{79}w_{u,2,s,15} +  \gamma_{99}w_{u,2,s,16} +  \gamma_{100}w_{u,2,s,17}\\
&\ +  \gamma_{115}w_{u,2,s,18} +  \gamma_{143}w_{u,2,s,19} +  \gamma_{123}w_{u,2,s,20} +  \gamma_{221}w_{u,2,s,21} = 0,\\  
\end{align*}
\begin{align*}
&\gamma_{3}w_{u,2,s,1} + \gamma_{5}w_{u,2,s,2} +  \gamma_{13}w_{u,2,s,3} +  \gamma_{18}w_{u,2,s,4} +  \gamma_{26}w_{u,2,s,5}\\
&\ +  \gamma_{28}w_{u,2,s,6} +  \gamma_{38}w_{u,2,s,7} +  \gamma_{43}w_{u,2,s,8} +  \gamma_{49}w_{u,2,s,9} +  \gamma_{59}w_{u,2,s,10}\\
&\ +  \gamma_{57}w_{u,2,s,11} + \gamma_{68}w_{u,2,s,12} +  \gamma_{\{75, 84, 103\}}w_{u,2,s,13} +   \gamma_{84}w_{u,2,s,14}\\
&\ +  \gamma_{\{77, 84, 103\}}w_{u,2,s,15} +  \gamma_{103}w_{u,2,s,16} +  \gamma_{101}w_{u,2,s,17} +  \gamma_{114}w_{u,2,s,18}\\
&\ +  \gamma_{144}w_{u,2,s,19} +  \gamma_{127}w_{u,2,s,20} +  \gamma_{\{84, 103, 127, 144, 222\}}w_{u,2,s,21} = 0,\\  
&\gamma_{4}w_{u,2,s,1} + \gamma_{6}w_{u,2,s,2} +  \gamma_{14}w_{u,2,s,3} +  \gamma_{17}w_{u,2,s,4} +  \gamma_{25}w_{u,2,s,5}\\
&\ +  \gamma_{27}w_{u,2,s,6} +  \gamma_{39}w_{u,2,s,7} +  \gamma_{47}w_{u,2,s,8} +  \gamma_{48}w_{u,2,s,9} +  \gamma_{60}w_{u,2,s,10}\\
&\ +  \gamma_{61}w_{u,2,s,11} +  \gamma_{67}w_{u,2,s,12} +  \gamma_{\{83, 117, 118, 124, 224\}}w_{u,2,s,13} +  \gamma_{\{83, 224\}}w_{u,2,s,14}\\
&\ +  \gamma_{\{78, 83, 117\}}w_{u,2,s,15} +  \gamma_{117}w_{u,2,s,16} +  \gamma_{118}w_{u,2,s,17} +  \gamma_{105}w_{u,2,s,18}\\
&\ +  \gamma_{146}w_{u,2,s,19} +  \gamma_{\{85, 118\}}w_{u,2,s,20} +  \gamma_{\{83, 85, 117, 118, 146\}}w_{u,2,s,21} = 0,\\  
&\gamma_{19}w_{u,2,s,1} + \gamma_{32}w_{u,2,s,2} +  \gamma_{7}w_{u,2,s,3} +   \gamma_{35}w_{u,2,s,4} +  \gamma_{10}w_{u,2,s,5}\\
&\ +  \gamma_{23}w_{u,2,s,6} +  \gamma_{46}w_{u,2,s,7} +  \gamma_{40}w_{u,2,s,8} +  \gamma_{53}w_{u,2,s,9}\\
&\ +  \gamma_{62}w_{u,2,s,10} +  \gamma_{71}w_{u,2,s,11} +  \gamma_{58}w_{u,2,s,12} +  \gamma_{80}w_{u,2,s,13}\\
&\ +  \gamma_{87}w_{u,2,s,14} +  \gamma_{76}w_{u,2,s,15} +  \gamma_{104}w_{u,2,s,16} +  \gamma_{116}w_{u,2,s,17}\\
&\ +  \gamma_{102}w_{u,2,s,18} +  \gamma_{145}w_{u,2,s,19} +  \gamma_{128}w_{u,2,s,20} +  \gamma_{223}w_{u,2,s,21} = 0,\\  
&\gamma_{20}w_{u,2,s,1} + \gamma_{31}w_{u,2,s,2} +  \gamma_{8}w_{u,2,s,3} +  \gamma_{34}w_{u,2,s,4} +  \gamma_{11}w_{u,2,s,5}\\
&\ +  \gamma_{22}w_{u,2,s,6} +  \gamma_{50}w_{u,2,s,7} +  \gamma_{41}w_{u,2,s,8} +  \gamma_{52}w_{u,2,s,9}\\
&\ +  \gamma_{63}w_{u,2,s,10} +  \gamma_{70}w_{u,2,s,11} +  \gamma_{65}w_{u,2,s,12} +\gamma_{\{81, 225\}}w_{u,2,s,13}\\
&\ +  \gamma_{\{86, 225\}}w_{u,2,s,14} +   \gamma_{\{121, 125\}}w_{u,2,s,15} +  \gamma_{119}w_{u,2,s,16} +  \gamma_{106}w_{u,2,s,17}\\
&\ +  \gamma_{121}w_{u,2,s,18} +  \gamma_{147}w_{u,2,s,19} +  \gamma_{\{89, 121\}}w_{u,2,s,20} +  \gamma_{225}w_{u,2,s,21} = 0,\\  
&\gamma_{33}w_{u,2,s,1} + \gamma_{21}w_{u,2,s,2} +  \gamma_{36}w_{u,2,s,3} +  \gamma_{9}w_{u,2,s,4} +  \gamma_{24}w_{u,2,s,5}\\
&\ +  \gamma_{12}w_{u,2,s,6} +  \gamma_{51}w_{u,2,s,7} +  \gamma_{54}w_{u,2,s,8} +  \gamma_{42}w_{u,2,s,9}\\
&\ +  \gamma_{72}w_{u,2,s,10} +  \gamma_{64}w_{u,2,s,11} +  \gamma_{66}w_{u,2,s,12} +  \gamma_{88}w_{u,2,s,13}\\
&\ +  \gamma_{82}w_{u,2,s,14} +  \gamma_{90}w_{u,2,s,15} +  \gamma_{107}w_{u,2,s,16} +  \gamma_{120}w_{u,2,s,17}\\
&\ +  \gamma_{122}w_{u,2,s,18} +  \gamma_{148}w_{u,2,s,19} +  \gamma_{126}w_{u,2,s,20} +  \gamma_{226}w_{u,2,s,21} = 0.  
\end{align*}

Computing directly from the above equalities, we obtain
\begin{align*}\tag{\ref{8.9.5}.4}
\gamma_j = 0,\ j &= 1, \ldots , 90, 99, \ldots , 107,\\
& 114, \ldots ,128,143, 144, \ldots, 148, 221\ldots, 226.
\end{align*}

With the aid of (\ref{8.9.5}.2), the homomorphisms $g_1, g_2, g_3, g_4$ send (\ref{8.9.5}.1) to
\begin{align*}
&\gamma_{194}w_{u,2,s,1} +  \gamma_{199}w_{u,2,s,2} +   \gamma_{91}w_{u,2,s,3} +  \gamma_{110}w_{u,2,s,4} +  \gamma_{96}w_{u,2,s,5}\\
&\ +  \gamma_{113}w_{u,2,s,6} +  \gamma_{204}w_{u,2,s,7} +  \gamma_{129}w_{u,2,s,8} +  \gamma_{135}w_{u,2,s,9}\\
&\ +  \gamma_{136}w_{u,2,s,10} +  \gamma_{141}w_{u,2,s,11} +  \gamma_{142}w_{u,2,s,12} +  \gamma_{154}w_{u,2,s,13}\\
&\ +  \gamma_{157}w_{u,2,s,14} +  \gamma_{160}w_{u,2,s,15} +  \gamma_{168}w_{u,2,s,16} +  \gamma_{170}w_{u,2,s,17}\\
&\ +  \gamma_{174}w_{u,2,s,18} +  \gamma_{187}w_{u,2,s,19} +  \gamma_{182}w_{u,2,s,20} +  \gamma_{227}w_{u,2,s,21} = 0,  
\end{align*}
\begin{align*}
&\gamma_{195}w_{u,2,s,1} +  \gamma_{198}w_{u,2,s,2} +  \gamma_{92}w_{u,2,s,3} +  \gamma_{109}w_{u,2,s,4} +  \gamma_{97}w_{u,2,s,5}\\
&\ +  \gamma_{112}w_{u,2,s,6} +  \gamma_{205}w_{u,2,s,7} +  \gamma_{130}w_{u,2,s,8} +  \gamma_{134}w_{u,2,s,9} +  \gamma_{137}w_{u,2,s,10}\\
&\ +  \gamma_{140}w_{u,2,s,11} + \gamma_{152}w_{u,2,s,12} +  \gamma_{\{155, 228\}}w_{u,2,s,13} +   \gamma_{\{156, 228\}}w_{u,2,s,14}\\
&\ +  \gamma_{\{178, 181\}}w_{u,2,s,15} +  \gamma_{171}w_{u,2,s,16} +  \gamma_{169}w_{u,2,s,17} +  \gamma_{178}w_{u,2,s,18}\\
&\ +  \gamma_{188}w_{u,2,s,19} +  \gamma_{\{162, 178\}}w_{u,2,s,20} +  \gamma_{228}w_{u,2,s,21} = 0,\\  
&\gamma_{196}w_{u,2,s,1} + \gamma_{200}w_{u,2,s,2} +  \gamma_{93}w_{u,2,s,3} +  a_1w_{u,2,s,4} +  \gamma_{95}w_{u,2,s,5} +  a_2w_{u,2,s,6}\\
&\ +  \gamma_{206}w_{u,2,s,7} +  \gamma_{131}w_{u,2,s,8} +  a_3w_{u,2,s,9} +  \gamma_{138}w_{u,2,s,10} +  \gamma_{149}w_{u,2,s,11}\\
&\ +  \gamma_{151}w_{u,2,s,12} +  \gamma_{\{175, 183, 189, 229\}}w_{u,2,s,13} +  \gamma_{\{158, 183, 189, 229\}}w_{u,2,s,14}\\
&\  +  \gamma_{\{177, 183\}}w_{u,2,s,15} +  \gamma_{172}w_{u,2,s,16} +  \gamma_{175}w_{u,2,s,17} +  \gamma_{177}w_{u,2,s,18}\\
&\ +  \gamma_{189}w_{u,2,s,19} +  \gamma_{\{161, 164, 175, 177\}}w_{u,2,s,20} +  \gamma_{\{164, 175, 183, 189, 229\}}w_{u,2,s,21} = 0,\\  
&\gamma_{197}w_{u,2,s,1} + \gamma_{201}w_{u,2,s,2} +  \gamma_{94}w_{u,2,s,3} +   a_4w_{u,2,s,4}+  \gamma_{98}w_{u,2,s,5} +  a_5w_{u,2,s,6}\\
&\ +  \gamma_{207}w_{u,2,s,7} +  \gamma_{132}w_{u,2,s,8} +  a_6w_{u,2,s,9} +  \gamma_{139}w_{u,2,s,10} +  \gamma_{150}w_{u,2,s,11}\\
&\ +  \gamma_{153}w_{u,2,s,12} +  \gamma_{\{176, 184, 190, 230\}}w_{u,2,s,13} +  \gamma_{\{159, 184, 190, 230\}}w_{u,2,s,14}\\
&\ +  \gamma_{\{179, 184\}}w_{u,2,s,15} +  \gamma_{173}w_{u,2,s,16} +  \gamma_{176}w_{u,2,s,17} +  \gamma_{179}w_{u,2,s,18}\\
&\ +  \gamma_{190}w_{u,2,s,19} +  \gamma_{\{163, 165, 176, 179\}}w_{u,2,s,20} +  \gamma_{\{165, 176, 184, 190, 230\}}w_{u,2,s,21} = 0,  
\end{align*}
where
\begin{align*}
a_1 &= \begin{cases}  \gamma_{\{108, 218\}}, &s=3,\\    \gamma_{108}, &s \geqslant 4, \end{cases} \ \ 
a_2 = \begin{cases}  \gamma_{\{111, 219\}}, &s=3,\\    \gamma_{111}, &s \geqslant 4, \end{cases} \\\
a_3 &= \begin{cases}  \gamma_{\{133, 220\}}, &s=3,\\    \gamma_{133}, &s \geqslant 4, \end{cases} \\
a_4 &= \begin{cases}  \gamma_{\{167, 185, 186, 191, 202, 209, 210, 211, 212, 214, 218, 231\}}, &s=3,\\    \gamma_{202}, &s \geqslant 4, \end{cases} \\
a_5 &= \begin{cases}    \gamma_{\{166, 167, 180, 185, 192, 203, 210, 213, 215, 219\}}, &s=3,\\    \gamma_{203}, &s \geqslant 4, \end{cases} \\
a_6 &= \begin{cases}  \gamma_{\{180, 191, 193, 208, 211, 216, 217, 220\}}, &s=3,\\    \gamma_{208}, &s \geqslant 4. \end{cases} \\
\end{align*}
These equalities imply
\begin{equation}\begin{cases}
a_i = 0, \ i = 1,2,3,4,5,6,\\
\gamma_j = 0, \ j = 91, \ldots , 98, 109, 110, 112, 113, 129, 130, 131,\\ 
132, 134, \ldots , 141, 142, 149, 150, \ldots , 165, 168, \ldots , 179,\\ 181, 182, 183, 184, 187, 188, 189, 190, 194, \ldots, 201,\\
204, 205, 206, 207, 227, 228, 229, 230.
\end{cases}\tag{\ref{8.9.5}.3}
\end{equation}

With the aid of (\ref{8.9.5}.2) and (\ref{8.9.5}.3), the homomorphism $h$ sends (\ref{8.9.5}.1) to
\begin{align*}
&a_7w_{u,2,s,1} +  a_8w_{u,2,s,2} +  \gamma_{209}w_{u,2,s,3} +  \gamma_{166}w_{u,2,s,4} +  \gamma_{210}w_{u,2,s,5}\\
&\ +  \gamma_{167}w_{u,2,s,6} +  a_9w_{u,2,s,7} +  \gamma_{211}w_{u,2,s,8} +  \gamma_{180}w_{u,2,s,9}\\
&\ +  \gamma_{212}w_{u,2,s,10} +  \gamma_{213}w_{u,2,s,11} +  \gamma_{185}w_{u,2,s,12} + \gamma_{214}w_{u,2,s,13}\\
&\ +  \gamma_{215}w_{u,2,s,14} +   \gamma_{186}w_{u,2,s,15} +  \gamma_{216}w_{u,2,s,16} +  \gamma_{217}w_{u,2,s,17}\\
&\ +  \gamma_{191}w_{u,2,s,18} +  \gamma_{193}w_{u,2,s,19} +  \gamma_{192}w_{u,2,s,20} +  \gamma_{231}w_{u,2,s,21} = 0, 
\end{align*}
where
$$a_7 = \begin{cases}  \gamma_{202}, &s=3,\\   \gamma_{218}, &s \geqslant 4, \end{cases} \
a_8 = \begin{cases}  \gamma_{203}, &s=3,\\   \gamma_{219}, &s \geqslant 4, \end{cases} \
a_9 = \begin{cases}  \gamma_{208}, &s=3,\\   \gamma_{220}, &s \geqslant 4. \end{cases} 
$$
From the above equalities, it implies
\begin{equation}\begin{cases}
a_7 = a_8 = a_9 = 0, \ 
\gamma_j = 0,\ j = 166, 167, 180, 185, 186, 191,\\ 192, 193, 209, 210, 211, 212, 213, 214, 215, 216, 217, 231.  
\end{cases}\tag{\ref{8.9.5}.4}
\end{equation}

Combining (\ref{8.9.5}.2), (\ref{8.9.5}.3) and (\ref{8.9.5}.4), we get $\gamma_j = 0$ for all $j$. The proposition is proved.
\end{proof}

\begin{props}\label{8.9.6} For $s \geqslant 3$, $t \geqslant 3$ and $u \geqslant 2$, the elements $[a_{u,t,s,j}]$, $1 \leqslant j \leqslant 231,$  are linearly independent in $(\mathbb F_2\underset {\mathcal A}\otimes R_4)_{2^{s+t+u}+2^{s+t}+2^s -3}$.
\end{props}

\begin{proof} Suppose that there is a linear relation
\begin{equation}\sum_{j=1}^{231}\gamma_j[a_{u,t,s,j}] = 0, \tag {\ref{8.9.6}.1}
\end{equation}
with $\gamma_j \in \mathbb F_2$.

Apply the homomorphisms $f_1, f_2, \ldots, f_6$ to the relation (\ref{8.9.6}.1) and we obtain

\begin{align*}
&\gamma_{1}w_{u,t,s,1} +  \gamma_{2}w_{u,t,s,2} +   \gamma_{15}w_{u,t,s,3} +  \gamma_{16}w_{u,t,s,4} +  \gamma_{29}w_{u,t,s,5}\\
&\ +  \gamma_{30}w_{u,t,s,6} +  \gamma_{37}w_{u,t,s,7} +  \gamma_{44}w_{u,t,s,8} +  \gamma_{45}w_{u,t,s,9}\\
&\ +  \gamma_{55}w_{u,t,s,10} +  \gamma_{56}w_{u,t,s,11} +  \gamma_{69}w_{u,t,s,12} +  \gamma_{73}w_{u,t,s,13}\\
&\ +  \gamma_{74}w_{u,t,s,14} +  \gamma_{79}w_{u,t,s,15} +  \gamma_{99}w_{u,t,s,16} +  \gamma_{100}w_{u,t,s,17}\\
&\ +  \gamma_{115}w_{u,t,s,18} +  \gamma_{143}w_{u,t,s,19} +  \gamma_{123}w_{u,t,s,20} +  \gamma_{221}w_{u,t,s,21} = 0,\\  
&\gamma_{3}w_{u,t,s,1} + \gamma_{5}w_{u,t,s,2} +  \gamma_{13}w_{u,t,s,3} +  \gamma_{18}w_{u,t,s,4} +  \gamma_{26}w_{u,t,s,5}\\
&\ +  \gamma_{28}w_{u,t,s,6} +  \gamma_{38}w_{u,t,s,7} +  \gamma_{43}w_{u,t,s,8} +  \gamma_{49}w_{u,t,s,9}+  \gamma_{59}w_{u,t,s,10}\\
&\  +  \gamma_{57}w_{u,t,s,11} + \gamma_{68}w_{u,t,s,12} +  \gamma_{75}w_{u,t,s,13}+   \gamma_{84}w_{u,t,s,14}\\
&\  +  \gamma_{77}w_{u,t,s,15} +  \gamma_{103}w_{u,t,s,16} +  \gamma_{101}w_{u,t,s,17} +  \gamma_{114}w_{u,t,s,18}\\
&\ +  \gamma_{144}w_{u,t,s,19} +  \gamma_{127}w_{u,t,s,20} +  \gamma_{\{84, 103, 127, 144, 222\}}w_{u,t,s,21} = 0,\\  
\end{align*}
\begin{align*}
&\gamma_{4}w_{u,t,s,1} + \gamma_{6}w_{u,t,s,2} +  \gamma_{14}w_{u,t,s,3} +  \gamma_{17}w_{u,t,s,4} +  \gamma_{25}w_{u,t,s,5}\\
&\ +  \gamma_{27}w_{u,t,s,6} +  \gamma_{39}w_{u,t,s,7} +  \gamma_{47}w_{u,t,s,8} +  \gamma_{48}w_{u,t,s,9} +  \gamma_{60}w_{u,t,s,10}\\
&\ +  \gamma_{61}w_{u,t,s,11} +  \gamma_{67}w_{u,t,s,12} +  \gamma_{85}w_{u,t,s,13} +  \gamma_{83}w_{u,t,s,14}\\
&\ +  \gamma_{78}w_{u,t,s,15} +  \gamma_{117}w_{u,t,s,16} +  \gamma_{118}w_{u,t,s,17} +  \gamma_{105}w_{u,t,s,18}\\
&\ +  \gamma_{146}w_{u,t,s,19} +  \gamma_{\{85, 118, 224\}}w_{u,t,s,20} +  \gamma_{\{83, 117, 124, 146, 224\}}w_{u,t,s,21} = 0,\\  
&\gamma_{19}w_{u,t,s,1} + \gamma_{32}w_{u,t,s,2} +  \gamma_{7}w_{u,t,s,3} +   \gamma_{35}w_{u,t,s,4} +  \gamma_{10}w_{u,t,s,5}\\
&\ +  \gamma_{23}w_{u,t,s,6} +  \gamma_{46}w_{u,t,s,7} +  \gamma_{40}w_{u,t,s,8} +  \gamma_{53}w_{u,t,s,9} +  \gamma_{62}w_{u,t,s,10}\\
&\ +  \gamma_{71}w_{u,t,s,11} +  \gamma_{58}w_{u,t,s,12} +  \gamma_{80}w_{u,t,s,13} +  \gamma_{87}w_{u,t,s,14}\\
&\ +  \gamma_{76}w_{u,t,s,15} +  \gamma_{104}w_{u,t,s,16} +  \gamma_{116}w_{u,t,s,17} +  \gamma_{102}w_{u,t,s,18}\\
&\ +  \gamma_{145}w_{u,t,s,19} +  \gamma_{128}w_{u,t,s,20} +  \gamma_{223}w_{u,t,s,21} = 0,\\  
&\gamma_{20}w_{u,t,s,1} +  \gamma_{31}w_{u,t,s,2} +  \gamma_{8}w_{u,t,s,3} +  \gamma_{34}w_{u,t,s,4} +  \gamma_{11}w_{u,t,s,5}\\
&\ +  \gamma_{22}w_{u,t,s,6} +  \gamma_{50}w_{u,t,s,7} +  \gamma_{41}w_{u,t,s,8} +  \gamma_{52}w_{u,t,s,9} +  \gamma_{63}w_{u,t,s,10}\\
&\ +  \gamma_{70}w_{u,t,s,11} +  \gamma_{65}w_{u,t,s,12} + \gamma_{81}w_{u,t,s,13} +  \gamma_{86}w_{u,t,s,14}\\
&\ +   \gamma_{89}w_{u,t,s,15} +  \gamma_{119}w_{u,t,s,16} +  \gamma_{106}w_{u,t,s,17}+  \gamma_{121}w_{u,t,s,18}\\
&\  +  \gamma_{147}w_{u,t,s,19} +  \gamma_{\{89, 121, 225\}}w_{u,t,s,20} +  \gamma_{\{89, 121, 125\}}w_{u,t,s,21} = 0,\\  
&\gamma_{33}w_{u,t,s,1} + \gamma_{21}w_{u,t,s,2} +  \gamma_{36}w_{u,t,s,3} +  \gamma_{9}w_{u,t,s,4} +  \gamma_{24}w_{u,t,s,5}\\
&\ +  \gamma_{12}w_{u,t,s,6} +  \gamma_{51}w_{u,t,s,7} +  \gamma_{54}w_{u,t,s,8} +  \gamma_{42}w_{u,t,s,9}\\
&\ +  \gamma_{72}w_{u,t,s,10} +  \gamma_{64}w_{u,t,s,11} +  \gamma_{66}w_{u,t,s,12} +  \gamma_{88}w_{u,t,s,13}\\
&\ +  \gamma_{82}w_{u,t,s,14} +  \gamma_{90}w_{u,t,s,15} +  \gamma_{107}w_{u,t,s,16} +  \gamma_{120}w_{u,t,s,17}\\
&\ +  \gamma_{122}w_{u,t,s,18} +  \gamma_{148}w_{u,t,s,19} +  \gamma_{126}w_{u,t,s,20} +  \gamma_{226}w_{u,t,s,21} = 0.  
\end{align*}

Computing directly from the above equalities, we obtain
\begin{align*}\tag{\ref{8.9.6}.2}
\gamma_j = 0,\ j &= 1, \ldots , 90, 99, \ldots , 107,\\
& 114, \ldots ,128, 143, 144, \ldots, 148, 221\ldots, 226.
\end{align*}

With the aid of (\ref{8.9.6}.2), the homomorphisms $g_1, g_2, g_3, g_4$ send (\ref{8.9.6}.1) to
\begin{align*}
&\gamma_{194}w_{u,t,s,1} +  \gamma_{199}w_{u,t,s,2} +   \gamma_{91}w_{u,t,s,3} +  \gamma_{110}w_{u,t,s,4} +  \gamma_{96}w_{u,t,s,5}\\
&\ +  \gamma_{113}w_{u,t,s,6} +  \gamma_{204}w_{u,t,s,7} +  \gamma_{129}w_{u,t,s,8} +  \gamma_{135}w_{u,t,s,9}\\
&\ +  \gamma_{136}w_{u,t,s,10} +  \gamma_{141}w_{u,t,s,11} +  \gamma_{142}w_{u,t,s,12} +  \gamma_{154}w_{u,t,s,13}\\
&\ +  \gamma_{157}w_{u,t,s,14} +  \gamma_{160}w_{u,t,s,15} +  \gamma_{168}w_{u,t,s,16} +  \gamma_{170}w_{u,t,s,17}\\
&\ +  \gamma_{174}w_{u,t,s,18} +  \gamma_{187}w_{u,t,s,19} +  \gamma_{182}w_{u,t,s,20} +  \gamma_{227}w_{u,t,s,21} = 0,\\  &\gamma_{195}w_{u,t,s,1} +  \gamma_{198}w_{u,t,s,2} +  \gamma_{92}w_{u,t,s,3} +  \gamma_{109}w_{u,t,s,4} +  \gamma_{97}w_{u,t,s,5}\\
&\ +  \gamma_{112}w_{u,t,s,6} +  \gamma_{205}w_{u,t,s,7} +  \gamma_{130}w_{u,t,s,8} +  \gamma_{134}w_{u,t,s,9} +  \gamma_{137}w_{u,t,s,10}\\
&\ +  \gamma_{140}w_{u,t,s,11} + \gamma_{152}w_{u,t,s,12} +  \gamma_{155}w_{u,t,s,13} +   \gamma_{156}w_{u,t,s,14}\\
&\ +  \gamma_{162}w_{u,t,s,15} +  \gamma_{171}w_{u,t,s,16} +  \gamma_{169}w_{u,t,s,17} +  \gamma_{178}w_{u,t,s,18}\\
&\ +  \gamma_{188}w_{u,t,s,19} +  \gamma_{\{162, 178, 228\}}w_{u,t,s,20} +  \gamma_{\{162, 178, 181\}}w_{u,t,s,21} = 0,  
\end{align*}
\begin{align*}
&\gamma_{196}w_{u,t,s,1} + \gamma_{200}w_{u,t,s,2} +  \gamma_{93}w_{u,t,s,3} +  a_1w_{u,t,s,4} +  \gamma_{95}w_{u,t,s,5} +  a_2w_{u,t,s,6}\\
&\ +  \gamma_{206}w_{u,t,s,7} +  \gamma_{131}w_{u,t,s,8} +  a_3w_{u,t,s,9} +  \gamma_{138}w_{u,t,s,10} +  \gamma_{149}w_{u,t,s,11}\\
&\ +  \gamma_{151}w_{u,t,s,12} +  \gamma_{164}w_{u,t,s,13} +  \gamma_{158}w_{u,t,s,14} +  \gamma_{161}w_{u,t,s,15}\\
&\ +  \gamma_{172}w_{u,t,s,16} +  \gamma_{175}w_{u,t,s,17} +  \gamma_{177}w_{u,t,s,18} +  \gamma_{189}w_{u,t,s,19}\\
&\ +  \gamma_{\{161, 164, 175, 177, 183, 189, 229\}}w_{u,t,s,20} +  \gamma_{\{161, 177, 183\}}w_{u,t,s,21} = 0,\\  
&\gamma_{197}w_{u,t,s,1} + \gamma_{201}w_{u,t,s,2} +  \gamma_{94}w_{u,t,s,3} +  a_4w_{u,t,s,4} +  \gamma_{98}w_{u,t,s,5} +  a_5w_{u,t,s,6}\\
&\ +  \gamma_{207}w_{u,t,s,7} +  \gamma_{132}w_{u,t,s,8} +  a_6w_{u,t,s,9} +  \gamma_{139}w_{u,t,s,10} +  \gamma_{150}w_{u,t,s,11}\\
&\ +  \gamma_{153}w_{u,t,s,12} +  \gamma_{165}w_{u,t,s,13} +  \gamma_{159}w_{u,t,s,14} +  \gamma_{163}w_{u,t,s,15}\\
&\ +  \gamma_{173}w_{u,t,s,16} +  \gamma_{176}w_{u,t,s,17} +  \gamma_{179}w_{u,t,s,18} +  \gamma_{190}w_{u,t,s,19}\\
&\ +  \gamma_{\{163, 165, 176, 179, 184, 190, 230\}}w_{u,t,s,20} +  \gamma_{\{163, 179, 184\}}w_{u,t,s,21} = 0, 
\end{align*}
where
\begin{align*}
a_1 &= \begin{cases}  \gamma_{\{108, 218\}}, &s=3,\\    \gamma_{108}, &s \geqslant 4, \end{cases} \ \ 
a_2 = \begin{cases}  \gamma_{\{111, 219\}}, &s=3,\\    \gamma_{111}, &s \geqslant 4, \end{cases} \\
a_3 &= \begin{cases}  \gamma_{\{133, 220\}}, &s=3,\\    \gamma_{133}, &s \geqslant 4, \end{cases} \\
a_4 &= \begin{cases}  \gamma_{\{167,185,186,191,202,209,210,211,212,214, 218, 231\}}, &s=3,\\     \gamma_{202}, &s \geqslant 4, \end{cases} \\
a_5 &= \begin{cases}    \gamma_{\{166, 167, 180, 185, 192, 203, 210, 213, 215, 219\}}, &s=3,\\    \gamma_{203}, &s \geqslant 4, \end{cases} \\
a_6 &= \begin{cases}  \gamma_{\{180, 191, 193, 208, 211, 216, 217, 220\}}, &s=3,\\    \gamma_{208}, &s \geqslant 4. \end{cases} 
\end{align*}

These equalities imply
\begin{equation}\begin{cases}
a_i = i, \ i = 1,2,3,4,5,6,\
\gamma_j = 0, \ j = 91, \ldots , 98,\\ 109, 110, 112, 113, 129, \ldots , 132,
134, \ldots , 141, 142,\\
149, \ldots , 165, 168, \ldots , 179, 181, \ldots , 184, 187, \ldots , 190,\\ 194, \ldots , 201, 204, \ldots , 207, 227, \ldots , 230.
\end{cases}\tag{\ref{8.9.6}.3}
\end{equation}

With the aid of (\ref{8.9.6}.2) and (\ref{8.9.6}.3), the homomorphism $h$ sends (\ref{8.9.6}.1) to
\begin{align*}
&a_7w_{u,t,s,1} + a_8w_{u,t,s,2} +  \gamma_{209}w_{u,t,s,3} +  \gamma_{166}w_{u,t,s,4} +  \gamma_{210}w_{u,t,s,5}\\
&\ +  \gamma_{167}w_{u,t,s,6} +  a_9w_{u,t,s,7} +  \gamma_{211}w_{u,t,s,8} +  \gamma_{180}w_{u,t,s,9}\\
&\ +  \gamma_{212}w_{u,t,s,10} +  \gamma_{213}w_{u,t,s,11} +  \gamma_{185}w_{u,t,s,12} + \gamma_{214}w_{u,t,s,13}\\
&\ +  \gamma_{215}w_{u,t,s,14} +   \gamma_{186}w_{u,t,s,15} +  \gamma_{216}w_{u,t,s,16} +  \gamma_{217}w_{u,t,s,17}\\
&\ +  \gamma_{191}w_{u,t,s,18} +  \gamma_{193}w_{u,t,s,19} +  \gamma_{192}w_{u,t,s,20} +  \gamma_{231}w_{u,t,s,21} = 0, 
\end{align*}
where
$$a_7 = \begin{cases}  \gamma_{202}, &s=3,\\   \gamma_{218}, &s \geqslant 4, \end{cases} \
a_8 = \begin{cases}  \gamma_{203}, &s=3,\\   \gamma_{219}, &s \geqslant 4, \end{cases} \
a_9 = \begin{cases}  \gamma_{208}, &s=3,\\   \gamma_{220}, &s \geqslant 4. \end{cases} 
$$

From the above equalities, it implies
\begin{equation}\begin{cases}
a_7 = a_8 = a_9 = 0,\
\gamma_j = 0,\ j = 166, 167, 180,\\ 185, 186, 191, 192, 193, 209, \ldots , 217, 231.  
\end{cases}\tag{\ref{8.9.6}.4}
\end{equation}

Combining (\ref{8.9.6}.2), (\ref{8.9.6}.3) and (\ref{8.9.6}.4), we get $\gamma_j = 0$ for all $j$. The proposition is proved.
\end{proof}

\section{Final Remark}\label{9}

1. It is well known that the $\tau$-sequence associated with an admissible monomial in $P_3$ is decreasing. However, this is not true in $P_4$. The admissible monomials in $P_4$ with non-decreasing $\tau$-sequences are:
\begin{align*}
&(1,2,2,2), (3,1,2,2), (1,3,2,2), (1,2,3,2), (1,2,2,3),\\
&(1,3,6,6), (3,1,6,6), (3,5,2,6), (3,5,6,2).
\end{align*}

2. Consider the $s\times 4$-matrices $A_{s,i}$, $1\leqslant i  \leqslant \mu(s)$, where $\mu(1)=4, \mu(2) = 10, \mu(3) = 14, \mu(s) = 15$ for $s\geqslant 4$,  which are determined as follows:

For $s \geqslant 1$,

$$A_{s,1} = \begin{pmatrix} 0&1&1&1\\ 0&1&1&1\\ \vdots &\vdots &\vdots &\vdots \\ 0&1&1&1\end{pmatrix} \quad A_{s,2} = \begin{pmatrix} 1&0&1&1\\ 1&0&1&1\\ \vdots &\vdots &\vdots &\vdots \\ 1&0&1&1\end{pmatrix} $$ 
$$A_{s,3} = \begin{pmatrix} 1&1&0&1\\ 1&1&0&1\\ \vdots &\vdots &\vdots &\vdots \\ 1&1&0&1\end{pmatrix} \quad A_{s,4} = \begin{pmatrix} 1&1&1&0\\ 1&1&1&0\\ \vdots &\vdots &\vdots &\vdots \\ 1&1&1&0\end{pmatrix}. $$

For $s \geqslant 2$,
$$A_{s,5} = \begin{pmatrix} 1&0&1&1\\ 0&1&1&1\\ \vdots &\vdots &\vdots &\vdots \\ 0&1&1&1\end{pmatrix} \quad A_{s,6} = \begin{pmatrix} 1&1&0&1\\ 0&1&1&1\\ \vdots &\vdots &\vdots &\vdots \\ 0&1&1&1\end{pmatrix} \quad A_{s,7} = \begin{pmatrix} 1&1&1&0\\ 0&1&1&1\\ \vdots &\vdots &\vdots &\vdots \\ 0&1&1&1\end{pmatrix} $$  $$A_{s,8} = \begin{pmatrix} 1&1&0&1\\ 1&0&1&1\\ \vdots &\vdots &\vdots &\vdots \\ 1&0&1&1\end{pmatrix} \quad A_{s,9} = \begin{pmatrix} 1&1&1&0\\ 1&0&1&1\\ \vdots &\vdots &\vdots &\vdots \\ 1&0&1&1\end{pmatrix} \quad A_{s,10} = \begin{pmatrix} 1&1&1&0\\ 1&1&0&1\\ \vdots &\vdots &\vdots &\vdots \\ 1&1&0&1\end{pmatrix} .$$

For $s \geqslant 3$,
$$A_{s,11} = \begin{pmatrix} 1&1&0&1\\ 1&0&1&1\\ 0&1&1&1\\ \vdots &\vdots &\vdots &\vdots \\ 0&1&1&1\end{pmatrix} \quad A_{s,12} = \begin{pmatrix} 1&1&1&0\\ 1&0&1&1\\ 0&1&1&1\\ \vdots &\vdots &\vdots &\vdots \\ 0&1&1&1\end{pmatrix} $$ 
$$ A_{s,13} = \begin{pmatrix} 1&1&1&0\\ 1&1&0&1\\ 0&1&1&1\\ \vdots &\vdots &\vdots &\vdots \\ 0&1&1&1\end{pmatrix} \quad A_{s,14} = \begin{pmatrix} 1&1&1&0\\ 1&1&0&1\\ 1&0&1&1\\ \vdots &\vdots &\vdots &\vdots \\ 1&0&1&1\end{pmatrix} .$$

For $s \geqslant 4$,
$$A_{s,15} = \begin{pmatrix} 1&1&1&0\\ 1&1&0&1\\ 1&0&1&1\\ 0&1&1&1\\ \vdots &\vdots &\vdots &\vdots \\ 0&1&1&1\end{pmatrix} .$$

3. Consider the $t \times 4$-matrices $B_{t,j}$, $1 \leqslant j \leqslant \rho(t),$ where  $\rho(1) = 6, \rho(2) = 20, \rho(t) = 35$ for $t \geqslant 3$, which are determined as follows:

For $t \geqslant 1$,
$$B_{t,1} = \begin{pmatrix} 0&0&1&1\\ 0&0&1&1\\ \vdots &\vdots &\vdots &\vdots \\ 0&0&1&1\end{pmatrix} \quad B_{t,2} = \begin{pmatrix} 0&1&0&1\\ 0&1&0&1\\ \vdots &\vdots &\vdots &\vdots \\ 0&1&0&1\end{pmatrix} \quad B_{t,3} = \begin{pmatrix} 0&1&1&0\\ 0&1&1&0\\ \vdots &\vdots &\vdots &\vdots \\ 0&1&1&0\end{pmatrix} $$  
$$B_{t,4} = \begin{pmatrix} 1&0&0&1\\ 1&0&0&1\\ \vdots &\vdots &\vdots &\vdots \\ 1&0&0&1\end{pmatrix} \quad B_{t,5} = \begin{pmatrix} 1&0&1&0\\ 1&0&1&0\\ \vdots &\vdots &\vdots &\vdots \\ 1&0&1&0\end{pmatrix} \quad B_{t,6} = \begin{pmatrix} 1&1&0&0\\ 1&1&0&0\\ \vdots &\vdots &\vdots &\vdots \\ 1&1&0&0\end{pmatrix}. $$

For $t \geqslant 2$,
$$B_{t,7} = \begin{pmatrix} 0&1&0&1\\ 0&0&1&1\\ \vdots &\vdots &\vdots &\vdots \\ 0&0&1&1\end{pmatrix} \quad B_{t,8} = \begin{pmatrix} 0&1&1&0\\ 0&0&1&1\\ \vdots &\vdots &\vdots &\vdots \\ 0&0&1&1\end{pmatrix} \quad B_{t,9} = \begin{pmatrix} 0&1&1&0\\ 0&1&0&1\\ \vdots &\vdots &\vdots &\vdots \\ 0&1&0&1\end{pmatrix} $$  
$$B_{t,10} = \begin{pmatrix} 1&0&0&1\\ 0&0&1&1\\ \vdots &\vdots &\vdots &\vdots \\ 0&0&1&1\end{pmatrix} \quad B_{t,11} = \begin{pmatrix} 1&0&1&0\\ 0&0&1&1\\ \vdots &\vdots &\vdots &\vdots \\ 0&0&1&1\end{pmatrix} \quad B_{t,12} = \begin{pmatrix} 1&0&0&1\\ 0&1&0&1\\ \vdots &\vdots &\vdots &\vdots \\ 0&1&0&1\end{pmatrix} $$  
$$B_{t,13} = \begin{pmatrix} 1&0&1&0\\ 0&1&1&0\\ \vdots &\vdots &\vdots &\vdots \\ 0&1&1&0\end{pmatrix} \quad B_{t,14} = \begin{pmatrix} 1&1&0&0\\ 0&1&0&1\\ \vdots &\vdots &\vdots &\vdots \\ 0&1&0&1\end{pmatrix} \quad B_{t,15} = \begin{pmatrix} 1&1&0&0\\ 0&1&1&0\\ \vdots &\vdots &\vdots &\vdots \\ 0&1&1&0\end{pmatrix} $$  
$$B_{t,16} = \begin{pmatrix} 1&0&1&0\\ 1&0&0&1\\ \vdots &\vdots &\vdots &\vdots \\ 1&0&0&1\end{pmatrix} \quad B_{t,17} = \begin{pmatrix} 1&1&0&0\\ 1&0&0&1\\ \vdots &\vdots &\vdots &\vdots \\ 1&0&0&1\end{pmatrix} \quad B_{t,18} = \begin{pmatrix} 1&1&0&0\\ 1&0&1&0\\ \vdots &\vdots &\vdots &\vdots \\ 1&0&1&0\end{pmatrix}. $$
$$B_{t,19} = \begin{pmatrix} 1&1&0&0\\ 0&0&1&1\\  \vdots &\vdots &\vdots &\vdots \\ 0&0&1&1\end{pmatrix} \quad B_{t,20} = \begin{pmatrix} 1&0&1&0\\ 0&1&0&1\\  \vdots &\vdots &\vdots &\vdots \\ 0&1&0&1\end{pmatrix}. $$

For $t \geqslant 3$,
$$B_{t,21} = \begin{pmatrix} 0&1&1&0\\ 0&1&0&1\\ 0&0&1&1\\ \vdots &\vdots &\vdots &\vdots \\ 0&0&1&1\end{pmatrix} \quad B_{t,22} = \begin{pmatrix} 1&0&1&0\\ 1&0&0&1\\ 0&0&1&1\\ \vdots &\vdots &\vdots &\vdots \\ 0&0&1&1\end{pmatrix} \quad B_{t,23} = \begin{pmatrix} 1&1&0&0\\ 1&0&0&1\\ 0&1&0&1\\ \vdots &\vdots &\vdots &\vdots \\ 0&1&0&1\end{pmatrix} $$  
$$B_{t,24} = \begin{pmatrix} 1&1&0&0\\ 1&0&1&0\\ 0&1&1&0\\ \vdots &\vdots &\vdots &\vdots \\ 0&1&1&0\end{pmatrix} \quad  B_{t,25} = \begin{pmatrix} 1&0&0&1\\ 0&1&0&1\\ 0&0&1&1\\ \vdots &\vdots &\vdots &\vdots \\ 0&0&1&1\end{pmatrix} \quad B_{t,26} = \begin{pmatrix} 1&0&1&0\\ 0&1&1&0\\ 0&0&1&1\\ \vdots &\vdots &\vdots &\vdots \\ 0&0&1&1\end{pmatrix}$$  
$$ B_{t,27} = \begin{pmatrix} 1&1&0&0\\ 0&1&1&0\\ 0&1&0&1\\ \vdots &\vdots &\vdots &\vdots \\ 0&1&0&1\end{pmatrix}\quad  B_{t,28} = \begin{pmatrix} 1&1&0&0\\ 1&0&1&0\\ 1&0&0&1\\ \vdots &\vdots &\vdots &\vdots \\ 1&0&0&1\end{pmatrix} \quad B_{t,29} = \begin{pmatrix} 1&0&1&0\\ 0&1&0&1\\ 0&0&1&1\\ \vdots &\vdots &\vdots &\vdots \\ 0&0&1&1\end{pmatrix}$$  
$$ B_{t,30} = \begin{pmatrix} 1&1&0&0\\ 0&1&0&1\\ 0&0&1&1\\ \vdots &\vdots &\vdots &\vdots \\ 0&0&1&1\end{pmatrix}  \quad B_{t,31} = \begin{pmatrix} 1&1&0&0\\ 0&1&1&0\\ 0&0&1&1\\ \vdots &\vdots &\vdots &\vdots \\ 0&0&1&1\end{pmatrix} \quad B_{t,32} = \begin{pmatrix} 1&1&0&0\\ 1&0&0&1\\ 0&0&1&1\\ \vdots &\vdots &\vdots &\vdots \\ 0&0&1&1\end{pmatrix}$$  
$$B_{t,33} = \begin{pmatrix} 1&1&0&0\\ 1&0&1&0\\ 0&0&1&1\\ \vdots &\vdots &\vdots &\vdots \\ 0&0&1&1\end{pmatrix} \quad  B_{t,34} = \begin{pmatrix} 1&1&0&0\\ 1&0&1&0\\ 0&1&0&1\\ \vdots &\vdots &\vdots &\vdots \\ 0&1&0&1\end{pmatrix} . $$

For $t=3$,
$$B_{3,35} = \begin{pmatrix} 1&1&0&0\\ 1&1&0&0\\ 0&0&1&1\end{pmatrix}.$$

For $t \geqslant 4$,
$$B_{t,35} = \begin{pmatrix} 1&1&0&0\\ 1&0&1&0\\ 0&1&0&1\\ 0&0&1&1\\ \vdots &\vdots &\vdots &\vdots \\ 0&0&1&1\end{pmatrix}. $$

4. Consider the $u \times 4$-matrices $C_{u,r}$,  $ 1 \leqslant r \leqslant \mu(u)$,  which are determined as follows:

For $ u \geqslant 1$,
$$C_{u,1} = \begin{pmatrix} 0&0&0&1\\ 0&0&0&1\\ \vdots &\vdots &\vdots &\vdots \\ 0&0&0&1\end{pmatrix} \quad C_{u,2} = \begin{pmatrix} 0&0&1&0\\ 0&0&1&0\\ \vdots &\vdots &\vdots &\vdots \\ 0&0&1&0\end{pmatrix} $$ 
$$ C_{u,3} = \begin{pmatrix} 0&1&0&0\\ 0&1&0&0\\ \vdots &\vdots &\vdots &\vdots \\ 0&1&0&0\end{pmatrix} \quad C_{u,4} = \begin{pmatrix} 1&0&0&0\\ 1&0&0&0\\ \vdots &\vdots &\vdots &\vdots \\ 1&0&0&0\end{pmatrix} .$$

For $u \geqslant 2$,
$$C_{u,5} = \begin{pmatrix} 0&0&1&0\\ 0&0&0&1\\ \vdots &\vdots &\vdots &\vdots \\ 0&0&0&1\end{pmatrix} \quad C_{u,6} = \begin{pmatrix} 0&1&0&0\\ 0&0&0&1\\ \vdots &\vdots &\vdots &\vdots \\ 0&0&0&1\end{pmatrix} \quad C_{u,7} = \begin{pmatrix} 0&1&0&0\\ 0&0&1&0\\ \vdots &\vdots &\vdots &\vdots \\ 0&0&1&0\end{pmatrix} $$  
$$C_{u,8} = \begin{pmatrix} 1&0&0&0\\ 0&0&0&1\\ \vdots &\vdots &\vdots &\vdots \\ 0&0&0&1\end{pmatrix} \quad C_{u,9} = \begin{pmatrix} 1&0&0&0\\ 0&0&1&0\\ \vdots &\vdots &\vdots &\vdots \\ 0&0&1&0\end{pmatrix} \quad C_{u,10} = \begin{pmatrix} 1&0&0&0\\ 0&1&0&0\\ \vdots &\vdots &\vdots &\vdots \\ 0&1&0&0\end{pmatrix} .$$

For $ u \geqslant 3$,
$$C_{u,11} = \begin{pmatrix} 0&1&0&0\\ 0&0&1&0\\ 0&0&0&1\\ \vdots &\vdots &\vdots &\vdots \\ 0&0&0&1\end{pmatrix} \quad C_{u,12} = \begin{pmatrix} 1&0&0&0\\ 0&0&1&0\\ 0&0&0&1\\ \vdots &\vdots &\vdots &\vdots \\ 0&0&0&1\end{pmatrix} $$ 
$$C_{u,13} = \begin{pmatrix} 1&0&0&0\\ 0&1&0&0\\ 0&0&0&1\\ \vdots &\vdots &\vdots &\vdots \\ 0&0&0&1\end{pmatrix} \quad C_{u,14} = \begin{pmatrix} 1&0&0&0\\ 0&1&0&0\\ 0&0&1&0\\ \vdots &\vdots &\vdots &\vdots \\ 0&0&1&0\end{pmatrix} .$$

For $u \geqslant 4$,
$$C_{u,15} = \begin{pmatrix} 1&0&0&0\\ 0&1&0&0\\ 0&0&1&0\\ 0&0&0&1\\ \vdots &\vdots &\vdots &\vdots \\ 0&0&0&1\end{pmatrix} .$$

5. Let $A=(a_{ij}), B = (b_{ij})$ be $p\times k, q\times k$ matrices, respectively. Following Kameko \cite{ka}, we define the $(p+q)\times k$-matrix $A*B= (c_{ij})$ by
$$ c_{ij} = \begin{cases} a_{ij}, &\text{if } i \leqslant p,\\
b_{i-p,j}, & \text{if } i > p.\end{cases}
$$

Suppose $x$ is an admissible monomial in $P_4$ with $\tau(x)$  decreasing  and $\tau_1(x) < 4$. 

5.1. If $\deg(x) = 2^{s+1} - 3$ then the matrix associated with $x$ is of the form $C_{1,j}$ or $A_{s-1,i}*C_{1,j}$ for some $i, j$.

5.2. If $\deg(x) = 2^{s+1} - 2$ then the matrix associated with $x$ is of the form $B_{s,j}$ for some $j$.

5.3. If $\deg(x) = 2^{s+1} - 1$ then the matrix associated with $x$ is of the form $C_{s+1,r}$ for some $r$ or $A_{1,i}*B_{s-1,j}$ for some $i, j$.

5.4. If $\deg(x) = 2^{s+2}+2^{s+1} - 3$ then the matrix associated with $x$ is of the form $A_{s,i}*C_{2,r}$ for some $i, r,$ or $A_{s+1,i}$ for some $i, j$.

5.5. If $\deg(x) = 2^{s+t+1}+2^{s+1} - 3$ for $t \geqslant 2$, then the matrix associated with $x$ is of the form $A_{s,i}*C_{t+1,r}$ for some $i, r,$ or $A_{s+1,i}*B_{t-1,j}$ for some $i, j$.

5.6. If $\deg(x) = 2^{s+t}+2^{s} - 2$ then the matrix associated with $x$ is of the form $B_{s,j}*C_{t,r}$ for some $j, r$.

5.7. If $\deg(x) = 2^{s+t+u}+2^{s+t}+2^s - 3$ then the matrix associated with $x$ is of the form $A_{s,i}*B_{t,j}*C_{u,r}$ for some $i,j, r$.

\vfill\eject

{}

\medskip\noindent
{Department of Mathematics, Quy Nh\ohorn n University, 

\noindent
170 An D\uhorn \ohorn ng V\uhorn \ohorn ng, Quy Nh\ohorn n, B\`inh \DD \d inh, Vi\^{\d e}t Nam.}

\medskip\noindent
E-mail: nguyensum@qnu.edu.vn
\end{document}